%% file: thesis.tex
\input{meta/preamble}

\input{meta/math}

\input{meta/book-meta}

\begin{document}

\frontmatter

\input{front/coverpage-centered}

\input{front/dedication}

\input{front/copyright}

\input{front/version-history}

\input{chap/acknowledgments}

\input{front/abstract}

\tableofcontents 

\cleardoublepage
\mainmatter
\pagestyle{plain}

\input{chap/overview}

\part{General Background}
\label{part: gen-background}

\input{chap/prelim3}

\input{chap/formal-lang}
\input{chap/fsa}

\input{chap/pda}

\input{chap/research-intro}

\part{Specific Background}
\label{part: specific-background}
	
\input{chap/hyperbolic}
\input{chap/semi-wreath}

\input{chap/Reidemeister-Schreier-method}

\input{chap/transducers}

\part{Results}
\label{part: results}
	
\input{chap/closure-finite-index}

\input{chap/closure-extension}
\input{chap/fg-positive-cones}

\input{chap/crossing-w-z}

\part{Appendix}
\label{part: appendix-dump}
	
\input{chap/old-lang-convex}

\input{chap/fg-positive-cones-code}

\part{Conclusion(e)s} 
\label{part: conclusion}
\input{back/conclusions}

\backmatter

\input{bibliography/bib-page}{}

\end{document}

%% file: meta/preamble.tex
\documentclass[nobib]{tufte-book} %

\ifdefined\soulregister
  \soulregister\MakeUppercase{1}%
  \soulregister\MakeLowercase{1}%
\fi

\hypersetup{colorlinks} %

\usepackage[utf8]{inputenc} 
\usepackage{fontspec}

\usepackage{newunicodechar}

\newunicodechar{Š}{\v{S}}
\newunicodechar{ć}{\'{c}}

\newunicodechar{ž}{\v{z}}    %
\newunicodechar{–}{--}       %
\newunicodechar{’}{'}        %

\usepackage[
    backend=biber,
    style=alphabetic,
  ]{biblatex}
\let\oldcite\cite
\renewcommand{\cite}[2][]{\oldcite[#1]{#2}\sidenote{\fullcite{#2}}}

\usepackage{lipsum} %

\usepackage{booktabs} %

\usepackage{graphicx} %
\graphicspath{{graphics/}} %
\setkeys{Gin}{width=\linewidth,totalheight=\textheight,keepaspectratio} %

\usepackage{fancyvrb} %
\fvset{fontsize=\normalsize} %

 \renewcommand\allcapsspacing[1]{{\addfontfeature{LetterSpace=15}#1}}
 \renewcommand\smallcapsspacing[1]{{\addfontfeature{LetterSpace=10}#1}}

\usepackage{xspace} %

\newcommand{\monthyear}{\ifcase\month\or January\or February\or March\or April\or May\or June\or July\or August\or September\or October\or November\or December\fi\space\number\year} %

\usepackage{makeidx} %
\makeindex %

\makeatletter
\newcommand{\HUGE}{\@setfontsize\Huge{36pt}{40pt}} %
\makeatother

\newenvironment{abstract}[1][Abstract]
  {{\normalfont\LARGE\itshape #1\par\bigskip}}
  {\par}

%% file: meta/math.tex
\PassOptionsToPackage{hyphens}{url}\usepackage{hyperref}
\usepackage{booktabs} %
\usepackage{tabulary} %
\usepackage{threeparttable}
\usepackage{nameref}

\usepackage{amsmath, amssymb,amsfonts,amsthm,mathtools,enumerate}
\usepackage{hyphenat}
\usepackage{import}
\usepackage{graphicx,wrapfig}
\usepackage{svg}
\usepackage{tikz}
\usetikzlibrary{automata, positioning}
\usepackage{listings}
\lstdefinelanguage{GAP}{%
  morekeywords={%
    Assert,Info,IsBound,QUIT,%
    TryNextMethod,Unbind,and,break,%
    continue,do,elif,%
    else,end,false,fi,for,%
s    function,if,in,local,%
    mod,not,od,or,%
    quit,rec,repeat,return,%
    then,true,until,while%
  },%
  sensitive,%
  morecomment=[l]\#,%
  morestring=[b]',%
  numbers=left,                    
}[keywords,comments,strings]

\usepackage[T1]{fontenc}
\usepackage[variablett]{lmodern}
\usepackage{xcolor}
\usepackage{xurl}

\newtheorem{thm}{Theorem}[chapter]
\numberwithin{thm}{subsection}

\newtheorem{cor}{Corollary}[thm]{}
\newtheorem{obs}[thm]{Observation}
\newtheorem{lem}[thm]{Lemma} 
\newtheorem{prop}[thm]{Proposition}

\theoremstyle{definition}
\newtheorem{rmk}[thm]{Remark}
\newtheorem{defn}[thm]{Definition}
\newtheorem{ex}[thm]{Example}
\newtheorem{non-ex}[thm]{Non-example}
\newtheorem{exc}[thm]{Exercise}
\newtheorem{open-q}[thm]{Open Question}

\numberwithin{subcase}{case}

\newcommand{\Sunik}{{\v S}uni{\'c} }

\newcommand{\bA}{\mathbb{A}}
\newcommand{\bB}{\mathbb{B}}
\newcommand{\bG}{\mathbb{G}}
\newcommand{\bH}{\mathbb{H}}
\newcommand{\bQ}{\mathbb{Q}}
\newcommand{\bM}{\mathbb{M}}
\newcommand{\bN}{\mathbb{N}}
\newcommand{\bR}{\mathbb{R}}

\newcommand{\bT}{\mathbb{T}}
\newcommand{\bZ}{\mathbb{Z}}

\newcommand{\cA}{\mathcal A}
\newcommand{\cC}{\mathcal C}

\newcommand{\cG}{\mathcal G}
\newcommand{\cH}{\mathcal H}
\newcommand{\cL}{\mathcal L}
\newcommand{\cM}{\mathcal M}

\newcommand{\cP}{\mathcal P}
\newcommand{\cR}{\mathcal R}
\newcommand{\cS}{\mathcal S}

\newcommand{\inv}{^{-1}}
\newcommand{\rel}{_{\text{rel}}}
\DeclareMathOperator{\LO}{LO}

\DeclareMathOperator{\Aut}{Aut}
\DeclareMathOperator{\GL}{GL}
\DeclareMathOperator{\SL}{SL}
\DeclareMathOperator{\Homeo}{Homeo}
\DeclareMathOperator{\rk}{rk}

\newcommand{\xrightarrowdbl}[2][]{%
  \xrightarrow[#1]{#2}\mathrel{\mkern-14mu}\rightarrow
}

\DeclareMathOperator{\Stab}{Stab}

\newcommand{\Reg}{\mathsf{Reg}}

\newcommand{\onecounter}{\mathsf{1C}}

\newcommand{\Regex}{\mathsf{Regex}}

\DeclareMathOperator{\dist}{d}

\DeclareMathOperator{\BS}{BS}
\DeclareMathOperator{\Ab}{Ab}
\DeclareMathOperator{\rip}{rip}
\DeclareMathOperator{\atanTwo}{atan2}

\newcommand{\ovl}{\overline}
\newcommand{\udb}{\underbrace}

\newcommand{\isom}{\cong}
\newcommand{\imp}{\Rightarrow}

\newcommand{\real}{\mathbb{R}}

\newcommand{\nat}{\mathbb{N}}

\newcommand{\union}{\cup}
\newcommand{\intersect}{\cap}

\newcommand{\Y}{(X^\$ \times X^\$)^*}
\newcommand{\Z}{({X^\$})^*}
\newcommand{\X}{{X^\$}}

\newcommand{\nats}{\mathbb{N}}

\newcommand{\vphi}{\varphi}
\renewcommand{\L}{\mathcal{L}}

\newcommand{\padL}{L^\$}

\newcommand{\PSL}{\text{PSL}}

\newcommand{\Rec}{\text{Rec}}
\newcommand{\Rat}{\text{Rat}}

\newcommand{\derives}{{\stackrel {*}{\Rightarrow }}}

\makeindex

%% file: meta/book-meta.tex
\title[A toolbox for left-orders of low complexity]{} %

\author{Hang Lu Su} %

\publisher{Universidad Autónoma de Madrid} %

%% file: front/coverpage-centered.tex
\begin{titlepage} %
\begin{fullwidth} %

\begingroup
\setlength\parindent{0cm}

\begin{center}
\includegraphics[width=7.090918cm,height=1.36906cm]{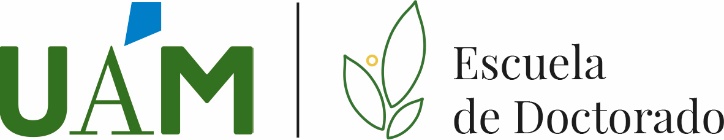}%
\hspace{1cm}%
\includegraphics[width=7.090918cm,height=1.46906cm]{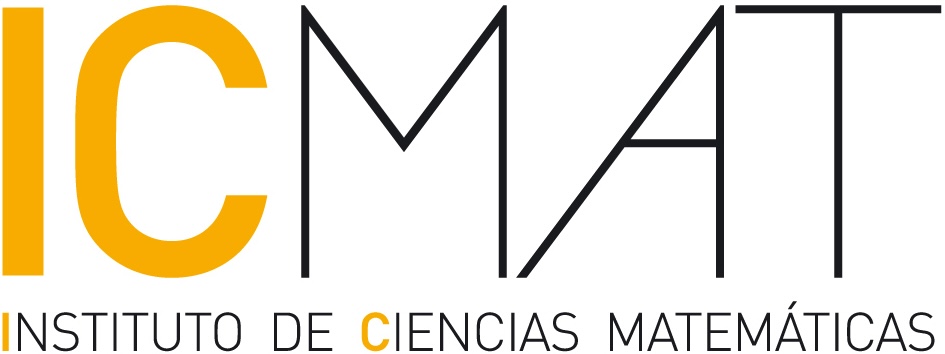}%
\end{center}

\vspace{2cm}

\vspace{2cm}

\begin{center}
{\HUGE
    A toolbox for left-orders of low complexity}
\end{center}

\vspace{1cm}

\begin{flushright}
    \Huge
    Hang Lu Su
\end{flushright}

\vspace{1cm}

\begin{multicols}{2}\huge
Programas de Doctorado en Matem\'aticas

\columnbreak

\vspace*{0.5cm}
\begin{flushright}
	UAM-ICMAT
\end{flushright}
\end{multicols}

{\huge
Dirección:

\vspace{0.5cm}
\hspace{1cm} Yago Antol\'in
}

\vspace{1cm}
\begin{flushright}\LARGE
    \textbf{Madrid, 2025}
\end{flushright}

\vfill

\includegraphics[width=2.7917in,height=0.539in]{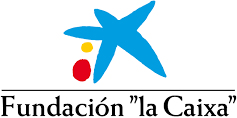}%
\hfill
\includegraphics[width=2.7917in,height=0.539in]{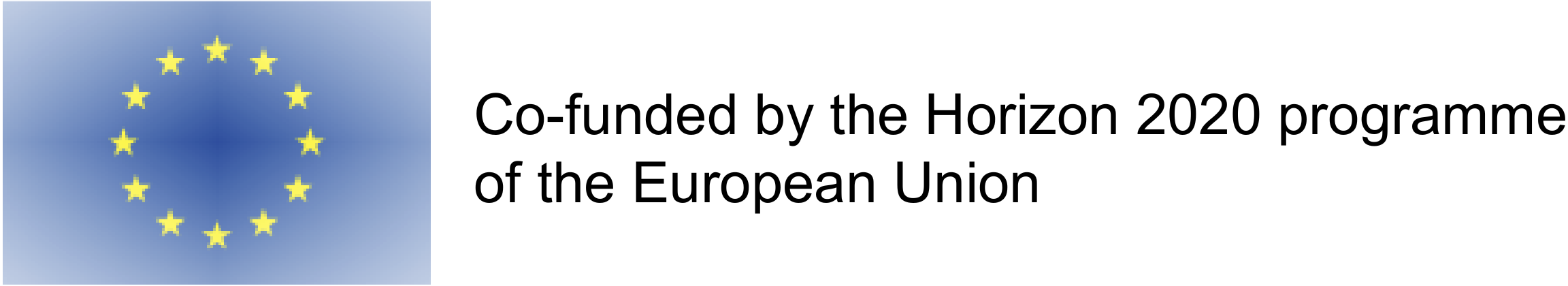}%

\endgroup
\end{fullwidth}
\end{titlepage}

%% file: front/dedication.tex
\cleardoublepage
\thispagestyle{empty}
~\vfill
\begin{doublespace}
\noindent\fontsize{18}{22}\selectfont\itshape
\nohyphenation

To Zhilu.
\end{doublespace}
\vfill
\vfill

\cleardoublepage

%% file: front/copyright.tex
\newpage
\begin{fullwidth}
~\vfill
\thispagestyle{empty}
\setlength{\parindent}{0pt}
\setlength{\parskip}{\baselineskip}
Copyright \copyright\ \the\year\ \thanklessauthor

\par\smallcaps{Published by \thanklesspublisher}

\par\textit{Version 1.1, \monthyear}
\end{fullwidth}

%% file: front/version-history.tex
\cleardoublepage
\chapter*{Version history}
  \thispagestyle{plain}
  \setlength{\parindent}{0pt}
  \setlength{\parskip}{0.75\baselineskip}

This thesis was originally submitted and archived as a doctoral dissertation at the Universidad Autónoma de Madrid in October 2025. The archived thesis is the institutional version of record. This PDF is an updatable arXiv version, and should be interpreted as a subsequent refinement.

\bigskip

  {\large\smallcaps{v1 \quad Version 1.1 \qquad (2025-12-04)}}\par
	\textit{Integrated Zoran Šunić's feedback.}\par

 {\textbf{Sections affected.} Secs \ref{sec: research-P-Z2} and \ref{sec: lang-conv}.} \par

\textbf{Summary of changes:}
 	\begin{itemize}
 		\item The proof of Proposition \ref{prop: Zsq-irrational-not-CF}\sidenote{in relation to how $\bZ^2$ doesn't have context-free positive cones when the slope is irrational} has been significantly corrected. It was plainly wrong before, as I was using the irrational slope in the proof instead of the induced angle while thinking of the angle; this has been fixed through the use of atan2.
 		\item Theorem \ref{thm: lang-convex-closure} and its proof\sidenote{in relation to how language-convex subgroups inherit regular positive cones from their overgroups} have been changed in order to rigorously meet the definition of a positive cone language for the subgroup. The problem before was that the subgroup regular positive cone language was in the alphabet of the overgroup, not the subgroup, contradicting Definition \ref{def: PCL}.
 		\item Many minor typos, misattributions and inaccuracies have been fixed.
 	\end{itemize}

\textbf{Overall impact.}
	No significant changes to the statement of the results, but overall accuracy and legibility has been improved.

%% file: chap/acknowledgments.tex
\cleardoublepage
\thispagestyle{empty}
\pagestyle{empty}
\chapter*{Acknowledgments}

First and foremost, I would like to express my deepest gratitude to my supervisor Yago Antol\'in for his exceptionally generous and patient mentorship and for making me more mathematically rigorous (sometimes in spite of myself). Your unwavering dedication to my success has had a profound impact on my academic journey. Thank you for your availability to provide detailed feedback any time I needed it, for your ability to explain complex ideas with clarity, and for your understanding whenever I needed a break. I draw inspiration from your integrity and dedication, both mathematically and personally.

I would also like to thank the group theory group at the ICMAT. I truly believe it is a phenomenal place that fosters its students, and I count myself lucky to have been able to have a place there. I would like to thank Andrei Jaikin in particular for helping me with the administrative side of things as my academic tutor.

I am also grateful to my thesis committee members, Zoran \v{S}uni\'c, Alejandra Garrido, Dominik Francoeur, Leo Margolis, and Mar\'ia Cumplido Cabello for seeing me through this last part of my journey. I would like to especially thank Dominik for his careful reading of my thesis and the insightful feedback he provided.

Beyond my immediate academic institutions, I would like to thank everyone outside of them who has contributed to my growth as a mathematician, whether through advice, organizing events, or other forms of support. In particular, I would like to thank Crist\'obal Rivas for giving me interesting problems to solve throughout graduate school. I extend my gratitude to the group at UPV/EHU Bilbao for often hosting those of us from Madrid. In particular, I would like to thank Gustavo A. Fernández Alcober and Ilya Kazachkov for fostering an exceptionally warm environment and Urban Jezernik and Primož Moravec for their GAP seminar. I would like to thank the group at the University of Toronto, and particularly Kasra Rafi, who made me feel like I had a second home to return to in Canada. Lastly, I would like to thank the group at Heriot-Watt and Laura Ciabanu for hosting me during a very nice research trip in Edinburgh.

This journey would not have been possible without the many friends and collaborators I have met along the way. In particular, I would like to thank Thomas Ng, Turbo Ho, Arman Darbinyan, Alex Evetts, Aurora Marks and Iván Chércoles Cuesta for the great mathematical discussions and sometimes mentorship, and them along with everybody else that I spent time with in the mathematical community for the nice memories.

I would also like to thank everyone who inspired me to embark on this journey in the first place, which I will now do in reverse chronological order. Thank you to everyone I met at the Recurse Center for giving me such a fun experience of what learning could be like, from which I have drawn ever since. I would like to thank my undergraduate research project advisors Prakash Panangaden, Tim Hsu, Moon Duchin and Ayla P. S\'anchez. I draw inspiration from each of you and think about you often. I would like to thank my undergraduate teachers Robert Brandenberger, Henri Darmon and Dani Wise for particularly inspirational lectures and challenging homework that resonated with me mathematically, and Axel W. Hundemer for his remarkable kindness towards the undergraduate students. I would like to thank my c\'egep teachers J\'er\'emie Vinet for nurturing my curiosity about physics with such kindness, patience, humour and inspiration and Tamara Zakon for being a role model of what a female mathematician could be, even before I knew I wanted to become one. From \'ecole secondaire, I would like to thank the late Gilles Roy for recognising me as a budding scientist before I even realised I liked science.

I also want to deeply thank everyone who has been there for me in my personal life during a time when I grew enormously as a person. I don't think this is the place to share details of my personal journey, so I will just say that you know who you are, what you've done, and (I hope) how much that meant to me.

Finally, I would like to acknowledge anyone I may have unintentionally overlooked who has supported me in any way along this path. Thank you all for your contributions.

\pagebreak

This work would not have been possible without the financial support of ``la Caixa" Foundation (ID 100010434) with fellowship code LCF/BQ/IN17/11620066, from the European Union's Horizon 2020 research and innovation programme under the Marie Sk\l{}odowska-Curie grant agreement No. 713673, and from the Severo Ochoa Programme for Centres of Excellence in R\&D (SEV-20150554). I am truly appreciative of the resources and opportunities provided.

\pagebreak

%% file: front/abstract.tex
\begin{fullwidth} %
\thispagestyle{empty} %
\setlength{\parindent}{0pt} %

\begin{abstract}[Resumen]
Esta tesis explora cómo los conceptos de la teoría de lenguajes formales pueden emplearse para estudiar grupos ordenables a la izquierda. Se analizan los lenguajes formados por sus conos positivos y se muestra cómo las familias abstractas de lenguajes (FAL) de la jerarquía de Chomsky (en particular, los lenguajes regulares y los lenguajes libres de contexto) interactúan con construcciones fundamentales de la teoría de grupos bajo subgrupos, extensiones, generación finita y la formación de productos directos con \(\mathbb{Z}\). Estas investigaciones ofrecen nuevas perspectivas sobre la relación entre la decidibilidad y la geometría en la teoría de grupos.

\vspace{1em}

En la tesis se incluyen algunos resultados que pueden suponer mejoras respecto a la literatura existente. Se obtiene una clasificación de la complejidad de los conos positivos de $\mathbb{Z}^2$, una demostración más constructiva para hallar lenguajes regulares de conos positivos de subgrupos *language-convex* en comparación con un resultado de Su (2020), una construcción de una infinidad numerable de conos positivos regulares de $\BS(1,q)$ para $q \geq -1$, todos automórficos entre sí y que amplían un resultado de Antolín, Rivas y Su (2022), y una construcción de conos positivos con conjunto generador finito para grupos de la forma $F_{2n} \times \mathbb{Z}$ que extiende un resultado de Malicet, Mann, Rivas y Triestino (2019).
\end{abstract}
\vspace{1em}

\noindent\textbf{Palabras clave:} Grupos ordenables por la izquierda, lenguajes formales, conos positivos finitamente generados, conos positivos regulares, conos positivos de un contador.

\vspace{2em}

\begin{abstract}
    This thesis explores how concepts of formal language theory can be used to study left-orderable groups. It analyses the languages formed by their positive cones and demonstrates how the abstract families of languages (AFLs) in the Chomsky hierarchy (in particular regular and context-free languages) interact with core group-theoretic constructions under subgroups, extensions, finite generation and taking direct products with $\mathbb{Z}$. These investigations yield new insights into the interplay between decidability and geometry in group theory.

\vspace{1em}

    Some results which may be improvements to the existing literature are included in the thesis. There is a classification of the complexity of positive cones of $\mathbb{Z}^2$, a more constructive proof on finding regular positive cone languages of language-convex subgroups compared to a result of Su (2020), a construction of countably infinite many regular positive cones of $\BS(1,q)$ for $q \geq -1$ which are all automorphic to each other extending a result of Antol\'in, Rivas, and Su (2022), and a construction of positive cones with finite generating set for groups of the form $F_{2n} \times \mathbb{Z}$ extending a result of Malicet, Mann, Rivas, and Triestino (2019).
\end{abstract}

\vspace{1em}

\noindent\textbf{Keywords:} Left-orderable groups, formal languages, finitely generated positive cones, regular positive cones, one-counter positive cones.

\end{fullwidth}

%% file: chap/overview.tex
\setcounter{chapter}{-1}
\chapter{Overview}

\section{Problem summary}
A group $G$ is \emph{left-orderable} if there exists a strict total order $\prec$ on the elements of $G$ which is invariant under left-multiplication,
$$ g \prec h \iff fg \prec fh, \quad \forall g,h,f \in G.$$ 

In other words, no matter what element $f$ is multiplied on the left, the order of $g$ and $h$ remains the same. Equivalently, one can specify a left-order by choosing a positive cone $P \subseteq G$, defined as the set of ``positive'' elements: $$P = \{g \in G \mid 1 \prec g\}.$$ The positive cone $P$ uniquely determines the order $\prec$ (since $g \prec h$ if and only if $g^{-1}h \in P$).

In this thesis, we consider finitely generated groups, which means that every group element can be written as a word over some finite generating set. If $$G = \langle X \mid R \rangle$$ is a presentation of $G$, any element $g \in G$ can be represented as a word $$g = x_1 x_2 \cdots x_n$$ with each $x_i$ in the generating set $X$ or its inverse. This allows us to talk about formal languages (sets of words over the alphabet $X$ or its inverse) that correspond to subsets of the group. 

In particular, a positive cone $P$ can be described by a language $L \subseteq X^*$ such that the evaluation map $\pi: X^* \to G$ sends $L$ onto $P$ (i.e. $\pi(L) = P$).

A central question motivating this work is: can we find a left-order $\prec$ on a given left-orderable group $G$ such that membership in its positive cone is ``easily decidable'' algorithmically? In other words, is there a systematic way to write the elements of $G$ (choose a generating set and representatives) and a positive cone language $L$ for that order, so that given any word representing an element, one can easily determine whether that element lies in $P$ (and hence whether $g \prec 1$ or $1 \prec g$)? We interpret ``easily decidable'' through the lens of formal language complexity: we seek positive cones that belong to low complexity language classes where membership can be decided by simple automata.

To formalise the notion of complexity, we use the Chomsky hierarchy of formal languages. This hierarchy classifies languages by the computational power needed to decide membership, ranging from very restrictive models (like finite automata) to very powerful ones (like Turing machines). At the lower levels of this hierarchy are regular languages, which can be decided by a finite state automaton (a simple machine with limited memory), and context-free languages, which can be decided by a pushdown automaton (a machine with a stack). A particularly tractable subclass of context-free languages is the class of one-counter languages – these are languages recognisable by a pushdown automaton that uses only one stack symbol (essentially, it has a single counter for memory). 

In this thesis, we focus on finitely generated left-orderable groups whose positive cone languages are either regular or one-counter. These classes represent cases where the order decision problem is relatively low in complexity.

\section{Structure of the thesis}
The thesis is structured into four parts, each divided into several chapters. A summary of the main results of can be found in Section \ref{sec: conclusions}. 

\subsection{Part \ref{part: gen-background}: General Background}
	The purpose is to introduce the fundamental concepts and examples needed to understand the later results. This part provides the necessary background in group orderability and formal language theory.	
	\begin{itemize} 
			\item Chapter \ref{chap: LO} - \nameref{chap: LO}.  Provides background on left-orderable groups. We introduce the main examples of left-orderable groups that are central to this thesis and define key concepts (like left-orders and positive cones) used throughout. 				
				\begin{itemize}
					\item We present the Klein bottle group $K_2$ (and its finite-index subgroup $\mathbb{Z}^2$) as a running example to motivate questions about the computational complexity of positive cones.
				\end{itemize}
			\item Chapter \ref{chap: informal-lang} - \nameref{chap: informal-lang}. Offers a quick overview of the Chomsky hierarchy of formal languages and the corresponding automata. We review the characteristics of each class (regular, context-free, etc.) and discuss how these language classes relate to decision problems. This sets up the computational perspective for studying positive cones.
			\item Chapter \ref{chap: fsa} - \nameref{chap: fsa}. Focuses on regular languages, the class of languages decidable by finite state automata. 
				\begin{itemize}
					\item In Section \ref{sec: fsa-reg-P-geom}, we give a visual proof of a result of Alonso, Antol\'in, Brum and Rivas \cite{AlonsoAntolinBrumRivas2022} that says that positive cones of non-abelian free groups are not coarsely connected (Lemma \ref{lem: regular-implies-coarsely-connected}). This example illustrates how geometric properties of groups can impose limitations on the complexity of their positive cones.
				\end{itemize}
			\item Chapter \ref{chap: pushdown-automata} - \nameref{chap: pushdown-automata}. Discusses context-free languages, which are the languages accepted by pushdown automata. In particular, we will discuss one-counter languages, which are context-free languages accepted by pushdown automata with one stack symbol. 
			\begin{itemize}
					\item In Section \ref{sec: research-P-Z2} there is a result about the classification of the positive cone complexity of $\mathbb{Z}^2$ which may be new to the literature (Theorem \ref{thm: Zsq-P-classified}).
				\end{itemize}
			\item Chapter \ref{chap: research-intro} - \nameref{chap: research-intro}. Bridges to research. It synthesizes the concepts from the previous chapters and places them in the broader context of geometric group theory. This chapter serves as a gentle introduction to the research questions and prepares the reader for the new results in Part \ref{part: results}. It connects the abstract concepts of orders and formal languages with their applications in modern group theory.
	\end{itemize}

\subsection{Part \ref{part: specific-background}: Specific Background}
	The purpose is to provide specialised background topics that are used in particular proofs or results later in the thesis. These chapters are more narrowly focused and can be consulted on a need-to-know basis. (Readers primarily interested in the new results may skip this part on a first reading and refer back to it when necessary.)
	\begin{itemize}
			\item Chapter \ref{chap: hyperbolic} - \nameref{chap: hyperbolic}. Reviews hyperbolicity and coarseness in geometric group theory. This chapter complements Chapter \ref{chap: closure-finite-index} (a results chapter) which discusses quasi-convexity and acylindrically hyperbolic groups. 
			\item Chapter \ref{chap: semi-wreath} - \nameref{chap: semi-wreath}. Reviews of semi-direct products and wreath products. It is a reference chapter for Chapter \ref{chap: closure-extension} (a results chapter) which discusses results on group extensions. It is also referenced in Chapter \ref{chap: LO} when discussing group extensions.
			\item Chapter \ref{chap: RS} - \nameref{chap: RS}.  Introduces the Reidemeister–Schreier method for finding presentations of finite-index subgroups of finitely generated groups, along with pictures and a running example. It is a companion chapter to Chapter \ref{chap: fg} (a results chapter), where we use the method and its underlying ideas in constructions of finitely generated positive cones.
			\item Chapter \ref{chap: transducers} - \nameref{chap: transducers}. Gives an overview of transducers and how they are equivalent to finite state automata through Nivat's theorem. This chapter supports Chapter \ref{chap: cross-Z} (a results chapter), where transducers are employed to construct positive cones for direct products of certain groups with one-counter positive cones with $\bZ$. Readers unfamiliar with transducers can refer to this chapter when encountering them in the results.
	\end{itemize}
	
\subsection{Part \ref{part: results}: Results}
\label{sec: conclusions}
This part presents the core research contributions of the thesis. It contains the original results obtained during the author's doctoral program, either solo or in collaboration. It focuses on the results found in \cite{Su2020} and \cite{AntolinRivasSu2021}, with some improvements which have been found during the writing of this thesis, specifically Proposition \ref{prop: lang-conv-easy}, Proposition \ref{prop: BS-reg-conj} and Theorem \ref{thm: FnxZ-fg-P}. 
\begin{itemize}
		\item Chapter \ref{chap: closure-finite-index} - \nameref{chap: closure-finite-index}. Focuses on the closure properties of positive cone languages under taking \emph{language-convex} subgroups (see Definition \ref{defn: lang-convex}). 
			\begin{itemize}
				\item Sections \ref{sec: lang-conv-fi} and \ref{sec: lang-conv-order-convex}: we show that subgroups of finite index and finitely generated subgroups of right semi-direct products with $\bZ$ where the left-order is lexicographic with leading factor $\bZ$ inherit a regular positive cone from their overgroups. 
					\begin{itemize}
						\item Finitely generated positive cones are not passed down to finite index subgroups, as we discuss in Chapter \ref{chap: LO} with the running example of the Klein bottle group $K_2$ and $\bZ^2$, posing the natural question of what computational positive cone properties can be passed down. 
					\end{itemize}
				\item Section \ref{sec: acyl-hyp-grp}: we show by contradiction that a quasi-geodesic positive cone language of a finitely generated acylindrically hyperbolic group cannot be regular. 
					\begin{itemize}
						\item Calegari showed in 2003 that no fundamental group of a hyperbolic $3$-manifold has a regular geodesic positive cone \cite{Calegari2003}. 
						\item In 2017 Hermiller and Sunic showed that no free products admits a regular positive cone \cite{HermillerSunic2017NoPC}. 
						\item A related 2022 result is that of Alonso, Antol\'in, Brum and Rivas \cite{AlonsoAntolinBrumRivas2022} says that positive cones of non-abelian free groups are not coarsely connected. We discuss this in Section \ref{sec: fsa-reg-P-geom}, Lemma \ref{lem: regular-implies-coarsely-connected}. 
					\end{itemize}
			\end{itemize}
		\item Chapter \ref{chap: closure-extension} - \nameref{chap: closure-extension}. Focuses on the closure properties of positive cone languages under taking extensions.  
			\begin{itemize}
				\item Sections \ref{sec: virtually-polycyclic} and \ref{sec: clos-ext-wreath}: we show that regular positive cones are passed to their extensions and wreath products. 
				\item Section \ref{sec: virtually-polycyclic}: we show that left-orderable virtually polycyclic groups admit regular positive cones.
				\item Section \ref{sec: BS}: we construct one-counter positive cones for solvable Baumslag-Solitar groups and classify when they have regular positive cones. 
					\begin{itemize}
						\item In a sense, solvable Baumslag-Solitar groups are close to polycyclic groups. However, the result about regularity of  left-orders on polycyclic groups cannot be promoted to the case of all solvable groups by a result to Darbinyan \cite{Darbinyan2020}. This partially answers the natural follow-up question of when a solvable group has a regular positive cone. 
					\end{itemize}
				\item Section \ref{sec: all left-orders are regular}: we classify which groups only admit positive cones which are regular. 
			\end{itemize}
		\item Chapter \ref{chap: fg} - \nameref{chap: fg}. Focuses constructing finitely generated positive cones for certain infinite families of groups.			
			\begin{itemize}
				\item Section \ref{sec: fg-family}: for every natural number $k \geq 3$, we construct an infinite family of groups admitting a positive cone of rank $k$. 
					\begin{itemize}
						\item A 2011 paper of Navas \cite{Navas2011} constructs an infinite family of groups given by the presentation $\Gamma_n = \langle a, b \mid ba^nba^{-1} \rangle$ for $n \geq 1$, which have positive cones of rank $2$, and poses the question of finding infinite families of group as the above as an open question. 
					\end{itemize}
				\item Section \ref{sec: fg-F_nxZ}: we show that $F_n \times \bZ$ has a finitely generated positive cone if and only if $n$ is even. 
					\begin{itemize}
						\item A 2018 result of Clay, Mann and Rivas shows that $F_2 \times \mathbb{Z}$ has an isolated left-order \cite{ClayMannRivas2018}. A 2019 result of Malicet, Mann, Rivas and Triestino then expands on this result by showing that $F_n \times \mathbb{Z}$ has an isolated left-order if and only if $n$ is even \cite{MalicetMannRivasTriestino2019}. Our theorem is a strengthening of these results as a finitely generated positive cone implies an isolated order. 
					\end{itemize}
			\end{itemize}
		\item Chapter \ref{chap: cross-Z} - \nameref{chap: cross-Z}. Focuses on when, if $G$ has a one-counter positive cone, we can obtain a regular positive cone for $G \times \bZ$, effectively lowering the positive cone complexity by taking the direct product with $\bZ$ acting as a stack.  
		\begin{itemize}
			\item Section \ref{sec: crossz-transducer}: we explain how, if $G$ is an amalgamated free product of groups admitting regular positive cones, then $G$ admits a one-counter positive cone through ordering quasi-morphisms (a 2013 construction of \Sunik \cite{Sunic2013}). 
				\begin{itemize}
					\item When the amalgamation is trivial, this result is optimal in the view of a 2017 result of Hermiller and \u{S}uni\'{c} \cite{HermillerSunic2017NoPC} that says that free products do not admit regular positive cones.  Moreover, Alonso, Antol\'in, Brum and Rivas proved in 2022 that certain free products with amalgamation also do not admit regular positive cones \cite[Theorem 1.6]{AlonsoAntolinBrumRivas2022}. 
				\end{itemize}
			\item Section \ref{sec: crossz-stack-embedding}: $G \times \bZ$ admits a regular positive cone. 
			\item Section \ref{sec: crossz-applications}: we apply this result to when $G$ is a free group and when $G$ is an Artin group whose defining graph is a tree using a 1999 result of Hermiller and Meier \cite{HermillerMeier1999}. Moreover, we show that right-angled Artin groups based on a connected graph with no induced subgraph isomorphic to $C_4$ (the cycle with $4$ edges) or $L_3$ (the line with 3 edges) satisfy the conditions of the result above by a 1982 result of Droms \cite{Droms1982}. 
			\end{itemize}
	\end{itemize}

\subsection{Part \ref{part: appendix-dump}: Appendix}
	It contains content  which did not fit in Part \ref{part: results}. 
		\begin{itemize}
			\item Chapter \ref{chap: old-lang-convex} - \nameref{chap: old-lang-convex}. It is an outtake of Chapter \ref{chap: closure-finite-index}, whose exposition is more reflective of the original Theorem 1.1 as presented in \cite{Su2020}. Through the writing of what is now the outtake, a better (as in simpler and more constructive) version of the original Theorem 1.1 was found and is now included in Chapter \ref{chap: closure-finite-index}. Nevertheless, we included this outtake in case it is useful for the reader who is looking for a more one-to-one exposition of the original Theorem 1.1. 
			\item Chapter \ref{chap: fg-code} - \nameref{chap: fg-code}. Contains GAP code that was used in the numerical experiments presented in Chapter \ref{chap: fg} for reference. This is provided for transparency and for readers who may want to verify or further explore the examples computationally. 
		\end{itemize}

\subsection{Part \ref{part: conclusion}: Conclusion(e)s}
	A bilingual conclusion is given to satisfy the formatting requirements of the graduate school. 

\section{A personal note about the exposition}

When I decided to do a PhD in mathematics, I was excited not only to discover new theorems but also to understand, on a human level, the process of their discovery. In writing this thesis, I wanted to convey not only the results of my doctoral work but the process and intuitions that led to these results. (I believe Chapter \ref{chap: fg} most effectively reflects this approach.)

This thesis represents both my journey into mathematics and my attempt to share that journey with others. I was inspired to study mathematics by humanist values, which are beautifully conveyed in the essay ``Mathematics for Human Flourishing'' by Francis Su.\sidenote{Available at \url{https://mathyawp.wordpress.com/2017/01/08/mathematics-for-human-flourishing/} or at The American Mathematical Monthly,  Volume 124, 2017 - Issue 6.} As such, although I anticipate only a handful of people will ever read this thesis, my imaginary audience was the (struggling) undergraduate reader that I was, in the hopes that this thesis will be an understandable, and perhaps even joyful, read to any budding mathematician, no matter their background. I have tried to make it as accessible and self-contained as the limits of my time and energy could permit, and this came at the cost of succinctness. 

To the experts in my field: there are many chapters and sections that can be skipped completely, and I have tried to be exhaustive in my proofs, hopefully making it possible for the thesis to be read relatively quickly in spite of its deceptive length. I suggest starting with Part \ref{part: results} and going back as needed.

%% file: chap/prelim3.tex
\chapter{Left-orderable groups}
\label{chap: left-orderable-groups}\label{chap: LO}

In this chapter, we introduce left-orderable groups and discuss their closure properties. We work out many examples of left-orderable groups that will feature later in our thesis to build intuition for later chapters. 

\section{What is a left-orderable group?}

\begin{defn}[left-orderable group]
A group $G$ is \emph{left-orderable}\index{left-order} if there exists a strict total order $\prec$ on the elements of $G$ which is invariant under left-multiplication,
$$ g \prec h \iff fg \prec fh, \qquad \forall g,h,f \in G.$$
\end{defn}

The following three examples with $\bZ, \bZ^2$ and the Klein bottle group $K_2$ will be running throughout this chapter and will be central to the themes of this thesis.  

\begin{ex}[$\bZ$]\label{ex: LO-Z}
The group of integers under the addition operation $(\mathbb{Z}, +)$ has a natural left-order $\prec$ given by $$\dotsm \prec -1 \prec 0 \prec 1 \prec \dotsm.$$
More succinctly, $x \prec x + 1$ for all $x \in \bZ$. The order is strict and total. Indeed, by induction, all pairs of distinct integers $x \neq y$ in $\bZ$ can be written as either $y = x+n$ or $y = x-n$ for some $n \in \bN$. In the first case, $x \prec y$, and in the second case, $x \succ y$. Moreover, if $x \prec y$ and $y \prec z$ for some $z \in \bZ$, then by the same reasoning, $z = y + m$ for some $m \in \bN$ and $z = x + n + m$, so $x \prec z$. 
\end{ex}

We can extend this ordering to $\bZ \times \bZ$ lexicographically. 

\begin{defn}[lexicographic order on Cartesian products]\label{defn: LO-lex-cartesian}
	Let $S = S_1 \times \dots \times S_n$ be a Cartesian product of orderable sets $(S_i, \prec_i)$ for $1 \leq i \leq n$. A \emph{lexicographic order} on $S$ is an order $\prec$ on $S$ given by $$(s_1,\dots, s_n) \prec (t_1, \dots, t_n)$$ if and only if

$$\exists i \in \{1,\dots,n\} \text{ such that } (s_1 = t_1), \dots, (s_{i-1} = t_{i-1}) \text{ and } (s_i \prec_i t_i).$$
\end{defn}

\begin{ex}[$\bZ^2$]\label{ex: LO-Z-square}
The group of pairs of integers under coordinate-wise addition ($\mathbb{Z}^2$) has a natural lexicographic left-order $\prec$. Let $\prec'$ be the left-order on $\bZ$ of Example \ref{ex: LO-Z}, then let $\prec$ be an order on $\bZ^2$ given by $(x,y) \prec (x', y')$ if and only if either $x \prec' x'$, or $x = x'$ and $y \prec' y'$ also as integers in $\bZ$. 

The order $\prec$ is total since $\prec'$ is total. Furthermore, it is invariant since addition in $\bZ^2$ is coordinate-wise. In other words, $(z,w) + (x,y) = (z + x, w + y) \prec (z,w) + (x', y') = (z + x', w + y')$, but by definition this is true if and only if either $z + x \prec' z + x' \iff x \prec' x'$ (by invariance of $\prec'$) or $z + x = z + x' \iff x = x'$ (since we are in $\bZ$) and $w + y \prec' w + y' \iff y \prec' y'$ (again by invariance of $\prec'$). We conclude that the conditions for $(x,y) \prec (x',y')$ and $(z,w) + (x,y) \prec (x',y')$ are equivalent, giving us (left)-invariance. 
\end{ex}

\begin{figure}[h]
\centering
{
\includegraphics[width = \textwidth]{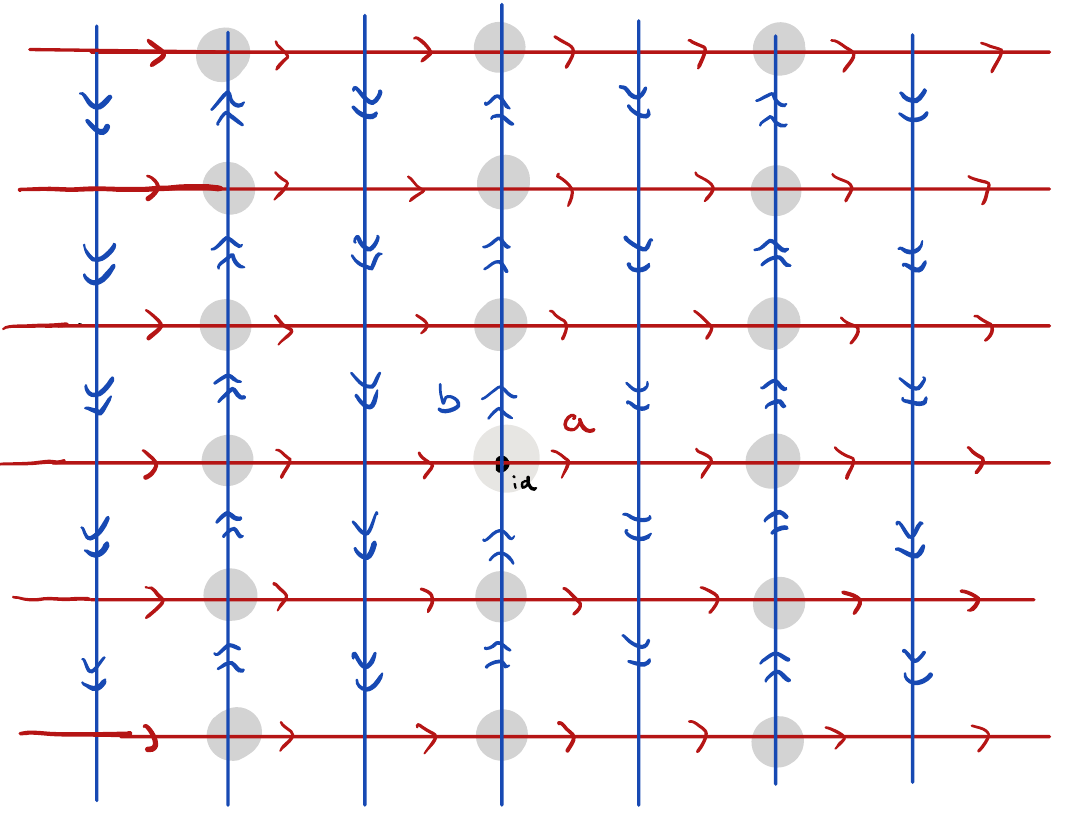}
}
\caption{The Cayley graph of $K_2$ under presentation $\langle a, b \mid bab = a\rangle$. The group has $\langle a^2, b\rangle \cong \bZ^2$ as a subgroup of index $2$, and the elements belonging to $\bZ^2$ are marked by the grey dots. One can visually derive the relation $a^2b = ba^2$ by observing the pattern of grey dots (the $\mathbb{Z}^2$ subgroup) in the graph.
}
\label{fig: cay-K2}
\end{figure}

Our next example is the Klein bottle group, which has presentation $$K_2 = \langle a, b \mid bab = a \rangle,$$ whose Cayley graph is illustrated in Figure \ref{fig: cay-K2}. The group $K_2$ is very similar to $\mathbb{Z}^2$ but with a ``twist'' introduced by the conjugation action of $a$ on $b$ (as specified by the relation $bab = a \iff aba\inv = b\inv$). In fact, $K_2$ contains a subgroup isomorphic to $\mathbb{Z}^2$ of index $2$. One can deduce this either by observing the Cayley graph or by a direct computation. Indeed, using the rewritten relation $ba = ab\inv$, we obtain that $ba^2 = (ba)a = a(b\inv a) = a(ab) = a^2b$, and conclude that $a^2$ and $b$ commute. The subgroup $\langle a^2, b \rangle$ has the presentation $\langle a^2, b \mid a^2b = ba^2 \rangle$, which is isomorphic to $\mathbb{Z}^2$. It is straightforward to check that this subgroup has index $2$ in $K_2$ by comparing presentations.

We would like to define a lexicographic left-order on $K_2$ similarly to how we defined it for $\bZ^2$ using the given presentation. To do so, let us use the natural lexicographic ordering on words, similar to the one we find in the dictionary. 

\begin{defn}[word, word length and language]\label{defn: word}
A \emph{word} $w$ is a finite sequence of elements, called \emph{characters}, \emph{letters} or \emph{symbols}, selected from a finite set $X$, called an \emph{alphabet}. The \emph{length} of a word $|w| = n$ is the length of its sequence $w = x_1 \dots x_n$. 
\end{defn}

\begin{defn}[language, free monoid and empty word]\label{def: LO-lang}\index{free monoid, empty word}
A set of words over the same alphabet $X$ is called a \emph{language}, usually denoted by $L$. 

Let $X$ be a finite set. We define $X^n$ where $n \geq 0$ to be the set of words of length $n$ with alphabet $X$. We denote by $X^* = \bigcup_{n=0}^\infty X^n$ to be the \emph{free monoid} over $X$, where $X^0$ is thought of as the word formed over zero characters. $X^0$ is often denoted by $\epsilon$ and is called the \emph{empty word} or \emph{empty letter} and always maps to the identity in a group.

If $L$ is a language with alphabet $X$, then $L \subseteq X^*$.  
\end{defn}

\begin{defn}[lexicographic order on languages]
	Given a language $L$, over an alphabet $X$, a lexicographic order $\prec$ is an order that is constructed in the following way. 
	\begin{itemize}
		\item We choose a total strict order $<$ for all the letters $x \in X$. 
		\item We add a padding symbol $\$$ such that $\$ < x$ for all $x \in X$. 
		\item Viewing all words $w = x_1 \dots x_\ell, v = y_1 \dots y_m$ in $L$ as an element in $X^*$ padded at the end with $\$$’s if necessary (so that $w$ and $v$ have the same length in the comparison), we declare $w \prec v$ if they satisfy the conditions of Definition \ref{defn: LO-lex-cartesian}. 
			
			That is, to compare two words $u,v \in L$, create an injective map $\iota: X^* \to (X \sqcup \{\$\})^*$ such that $u = x_1 \dots x_n \mapsto (x_1, \dots, x_n, \$, \dots \$)$ and $v = (y_1, \dots, y_m, \$, \dots, \$)$ where the lengths the padding for $u$ and $v$ are given by $\max(|u|,|v|) - |u|$ and $\max(|u|, |v|) - |v|$ respectively.
	\end{itemize}
\end{defn}
That is, this defines an order
$$
u \prec v \iff \begin{cases}
& u \text{ is a prefix of } v \\
& u = w x u', \quad v = w y v' \text{ with } x < y.
\end{cases}$$
for $x,y \in X$. 

	For groups, a word will refer to a finite sequence of generators that we choose to represent a group element. Formally, for a presentation $\langle X \mid R \rangle$, we have an evaluation map $\pi: (X \cup X\inv)^* \to G$, and a lexicographic order for $G$ is on a language $L \subseteq (X \cup X\inv)^*$ such that $\pi(L) = G$. In this thesis, we will often assume that $X$ is symmetric (that is $X = X\inv$) even if we omit the inverses in the presentation for convenience. 
	
	For example, let $\bZ = \langle x \rangle$. If $X = \{x,x\inv\}$, then there is no meaningful lexicographic left-order on $X^*$ as either $x\inv x \prec x x\inv$ or vice-versa depending on the choice of $x > x\inv$ or $x < x\inv$, but both words represent the identity. For this reason, it is important to restrict the language $L$ to a subset of $X$, usually given by a normal form. For example, the set of reduced words in $X^*$ is a normal form the elements of $\bZ$. In that case, the reduced words $x^n, x^{-n}$ evaluate to their length in $\bZ$ for any $n \in \bN$ and the lexicographic order $\prec$ extending $x > x\inv$ corresponds to the left-order introduced in Example \ref{ex: LO-Z}.

\begin{ex}[$K_2$]\label{ex: LO-K2}
An alternative way of describing the left-order of $\bZ^2$ of Example \ref{ex: LO-Z-square} is by assigning presentation 
$$\bZ^2 = \langle a^2, b \mid [a^2,b] \rangle$$
and the language of normal forms $L = \{a^m b^n\}$ such that 
$$a^m b^n \prec a^{p} b^{q} \iff m < p \text{ or } (m = p, n < q).$$ 

Using this idea of normal forms, the group $K_2$ admits a lexicographic left-order compatible with that of $\bZ^2$. First, observe that every $g \in K_2$ can be written as a normal form as $g = a^m b^n$ for some $m,n \in \bZ$. Indeed, by playing with the relation, we can observe it can be rewritten as $ba = ab\inv$, $ba\inv = a\inv b\inv$, $b\inv a = ab$, and $b\inv a\inv = a\inv b$. This implies that every time there is a $b$-generator before an $a$-generator regardless of their respective signs, we can reorder the string to have an $a$-generator followed by a $b$-generator, changing the signs as necessary. Therefore, we can define a left-order $\prec$ on $K_2$ by declaring $a^m b^n \prec a^p b^q \iff m < p$ or $m = p$, $n < q$ where $<$ is a left-order in $\bZ$. 

It is clear why $\prec$ is a total order. To show left-invariance, let us use the fact that $a^2$ is in the center of the group. For any $n,m \in \bZ$, we have that $b^n a^m = a^m b^n$ if $m$ is even. If $m$ is odd, write $m = 2k + 1$ for some $k \in \bZ$. Then, we have that $b^n a^{2k + 1} = (b^n a^{2k}) a = a^{2k} (b^n a)$. Now $b^n a = a b^{-n}$ by using the relation $ba = ab\inv$ or $b\inv a = ab$ repeatedly. We conclude that $b^n a^m = a^{2k + 1} b^{-n} = a^m b^{-n}$ when $m$ is odd.

Let us see what happens when we left-multiply an element by $a^\alpha b^\beta$. We get $(a^\alpha b^\beta) (a^m b^n) = a^\alpha (b^\beta a^m) b^n = a^{\alpha} (a^m b^{(-1)^m \beta}) b^n$. Therefore, 
\begin{align*}
	a^\alpha b^\beta \cdot a^m b^n &= a^{\alpha + m} b^{(-1)^m \beta + n} \\
	a^\alpha b^\beta \cdot a^p b^q &= a^{\alpha + p} b^{(-1)^p \beta + q}
\end{align*}
Therefore, 
\begin{align*}
	& a^\alpha b^\beta \cdot a^m b^n \prec a^\alpha b^\beta \cdot a^p b^q \\
	&\iff \alpha + m < \alpha + p \quad \text{ or } \quad m = p, (-1)^m \beta + n < (-1)^p \beta + q \\
	&\iff m < p \text{ or } n < q \\
	&\iff a^m b^n \prec a^p b^q.
\end{align*}
\end{ex}

\begin{rmk}
	A language description of a left-order naturally arose when trying to describe the left-orders of $\bZ^2$ compared to $K_2$, and they were found to be relatively similar to work with computationally. This thesis is concerned with formal ways of qualifying this similarity computationally, via the study of formal languages associated to left-orders, which we will define in the subsequent introductory chapters. 
	
	As a contrast, we will find out later in this chapter that these left-orders are completely different topologically, so let's keep in mind this similarity. 
\end{rmk}

A common feature of the three previous groups is that they are torsion-free. This will be the case for all non-trivial left-orderable groups. 

\begin{non-ex}\label{non-ex: finite-not-LO}
Any group with an element of finite order cannot be left-orderable unless it is trivial. 
\begin{proof}
Indeed, assume that $G$ a non-trivial group and admits the left-order $\prec$. Let $g \in G$ be the element of finite order with $1_G \not= g$ and $n$ be such that $g^n = 1_G$. Suppose without loss of generality that $1_G \prec g$. Then we have that 
$$1_G \prec g \prec \dots \prec g^n \prec 1_G.$$
Since $\prec$ is a strict order, $1_G \prec 1_G$ is a contradiction.
\end{proof}
\end{non-ex}

\begin{cor}
Non-trivial left-orderable groups are infinite torsion-free.
\end{cor}

The class of left-orderable groups is rather large. Here is a list of examples of left-orderable groups. 

\begin{itemize}
\item $\mathbb{Z}$, $\mathbb{Z}^n$ and poly-$\bZ$ groups (see Example \ref{ex: LO-polyZ}), 
\item the lamplighter group $\bZ \wr \bZ$ (see Example \ref{ex: LO-bZ-wr-bZ})
, 
\item free groups (see Section \ref{sec: LO-Fn}),
\item the Klein bottle group and other surface groups (see Example \ref{ex: LO-surface-groups}),
\item braid groups, in particular $B_3$ (see Example \ref{ex: LO-B3}),
\item torsion-free one-relator groups, in particular solvable Baumslag-Solitar groups (see Examples \ref{ex: LO-BS-ext} and \ref{ex: LO-dyn-BS}),
\item RAAGs (see Example \ref{ex: LO-RAAGs}),
\item subgroups of $\Homeo^+(\bR)$ (see Theorem \ref{thm: LO-dynamic-realization}),
\item locally indicable groups (see Section \ref{sec: LO-locally-indicable}).  
\end{itemize}

Apart from locally indicable groups, each of these classes of groups will be relevant to our thesis. As such, we will provide left-orders for them in this introduction. 

For an in-depth overview of left-orderable groups, we recommend \cite{Clay2016}, which was also our main source for this chapter. 

Before we work out more examples, let's go over some basic properties of left-orderable groups. 

\section{Non-uniqueness of left-orders}

For any given left-orderable group that is not trivial (i.e. not just the identity element), left-orders on the group are not unique. 

\begin{lem}[reverse order]\label{lem: opp-ord}
If a group is left-orderable with order $\prec$, then there exists another left-order $\prec'$ such that 
$$g \prec' h \iff h \prec g.$$
\begin{proof}
The reverse order $\prec'$ is strict and total since $\prec$ is strict and total. To check for left-invariance, $fg \prec' fh \iff fh \prec fg \iff h \prec g \iff g \prec' h$, thus $\prec'$ is left-invariant as well. 
\end{proof}
\end{lem}

\begin{ex}[$\bZ$]\label{ex: LO-Z-opp-ord}
From Example \ref{ex: LO-Z} we had $x \prec x + 1$ for all $x \in \bZ$ as a left-order. The direction of the inequality was arbitrary as we can check that the reverse order $x + 1 \prec' x$ for all $x \in \bZ$ also defines a left-order. 
$$ \dots \prec' 1 \prec' 0 \prec' -1 \prec' \dots$$

Note that the set $\{\prec, \prec'\}$ is the collection of all left-orders on $\bZ$. In other words, there are only two left-orders on $\mathbb{Z}$. Indeed, for any left-order on $\mathbb{Z}$, consider the comparison of $0$ and $1$. Totality forces either $0 \prec 1$ or $1 \prec 0$. These choices yield exactly the orders $\prec$ and $\prec'$ described above. Moreover, as shown in Example 1.2, specifying $0 \prec 1$ entails $x \prec x+1$ for all $x$ (by left-invariance), and specifying $1 \prec 0$ similarly entails $x+1 \prec x$ for all $x$. Thus no other left-orders exist on $\mathbb{Z}$ 
\end{ex}

\section{Convexity in left-ordered groups}
Defining a left-order automatically defines a notion of convexity that is compatible with its usual geometrical counterpart in Euclidean geometry. 

\begin{defn}
Let $(G,\prec)$ be a left-ordered group.
A subgroup $H\leqslant G$ is called {\it $\prec$-convex} if for all $h_1,h_2\in H$ and $g\in G$ satisfying $h_1\prec g \prec h_2$ we have that $g\in H$.
\end{defn}

Intuitively, no new group element can fit between two elements of $H$ without already lying in $H$. 

\begin{ex}[$\bZ \leq \bZ^2$]\label{ex: convexity-Z-Zsq}
	Let $\bZ^2 = \{(x,y) \mid x,y \in \bZ\}$ with the lexicographic order $\prec$ as in Example \ref{ex: LO-Z-square}. Then the subgroup $\{(0,y) \mid y  \in \bZ\} \cong \bZ$ is $\prec$-convex, since 
	\begin{align*}
		& (0 , y_1) \prec (x,y) \prec (0, y_2)
	\end{align*}
	implies that $x = 0$ as otherwise we would have the contradiction that $0 < x$ and $x < 0$. 
\end{ex}

\section{Positive cones}\label{sec: pos-cone}
Up to now, to specify a left-order we essentially had to specify the comparison between every pair of elements. Equivalently (and more conveniently), it turns out that knowing just the set of elements greater than the identity uniquely determines the order.

\begin{defn}[semigroup]\index{semigroup}
	A semigroup $(S, \cdot)$ is a set together with an associative multiplication $\cdot: S \times S \to S$.
\end{defn}
\begin{defn}[finitely generated subsemigroup]
	For a semigroup $S$, if $s_i \in S$, for $i = 1,\dots, n$, then $\langle s_1, \dots, s_n \rangle^+$ represents the subsemigroup generated by the elements $s_i$, that is, the set of all words containing $s_1, \dots, s_n$ as letters but not containing any $s_i\inv$ for any $i \in \{1,\dots,n\}$.  
\end{defn}

\begin{defn}[positive cone]\index{positive cone}
A set $P \subset G$ is a \emph{positive cone} for $G$ if 
\begin{enumerate}
\item it is closed under semigroup operation $$PP \subseteq P,$$  
\item the positive cone $P$ respects the trichotomy property $$G = P \sqcup P^{-1} \sqcup \{1_G\}$$ where the union is disjoint.
\end{enumerate}
\end{defn}

One way to think about positive cones is to think of them as capturing the notion of additive positivity: $P$ represents the positive elements, $P\inv$ are the negative elements and $\{1_G\}$ is the neutral element. The semigroup closure rule represents how adding two positive elements results in another positive element, and trichotomy ensures how every element is either positive, neutral, or negative. 

\begin{ex}[$\bZ$]\label{ex: LO-additive-positivity}\label{ex: LO-P-Z}\label{ex: P-Z}
Let $\prec$ be a left-order on $\bZ$. Then, $$P_\prec = \{x \in \bZ \mid x \succ 0\}$$ defines a positive cone.  Indeed, for any $x,y \in P_\prec$, since $x \succ 0$ and $y \succ 0$, thus $x + y \succ 0 + y \succ 0 + 0 = 0$ (the identity element in $(\bZ, +)$) -- satisfying semigroup closure. 

Furthermore, every integer is either positive, zero, or negative with respect to $\prec$. Indeed, for any $x\neq 0$, either $x \succ 0$ or $x^{-1} = -x \succ 0$ (meaning $x \prec 0$), and $0$ is of course neither positive nor negative. Thus the trichotomy property holds.
\end{ex}

\begin{lem}[left-order and positive cone correspondence]\label{lem: pos-corr}
A group $G$ is left-orderable if and only if $G$ admits a positive cone. In particular,
\begin{enumerate}
\item if $\prec$ is a left-order on $G$, then $P_{\prec} = \{g \in G \mid 1 \prec g \}$ is the positive cone associated with $\prec$,
\item if $P$ is a positive cone on $G$, then $P$ defines the left-order $\prec_P$ by $g \prec_P h \iff g\inv h \in P$.
\end{enumerate}
\end{lem}

\begin{proof}
To show that $P_\prec$ is a positive cone, consider that $1_G \prec g \iff g\inv  \prec g\inv g \iff g\inv \prec 1_G$. Therefore, if $g \in P_\prec$, then $g\inv \in P_\prec\inv$. This gives us $G = P_\prec \sqcup P_\prec^{-1} \sqcup \{1_G\}$. Moreover, if both $1_G \prec g$ and $1_G \prec h$ then we can left-multiply the second equation by $g$, leading to $1_G \prec g \prec gh$ so $P_\prec P_\prec  \subseteq P_\prec$.

Conversely, assume that $P$ is a positive cone. To show that $\prec_P$ is a left-order, consider that 
$fg \prec_P fh \iff (fg)\inv fh \in P$ but $(fg)\inv fh = g\inv f\inv fh = g\inv h \in P \iff g \prec_P h$. The left-order $\prec_P$ is a total order because due to trichotomy $g\inv h$ is either in $P, P\inv$ or $\{1_G\}$, in which case $g=h$. 
\end{proof}

\begin{figure}[h]{
\centering
\includegraphics[width = \textwidth]{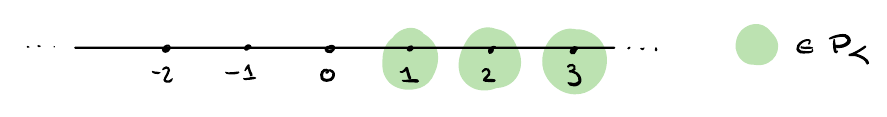}
}
\caption{The positive cone $P_\prec$ for $\mathbb{Z}$ (as in Example \ref{ex: LO-P-Z}) shown as the set of green points on the integer line.}
\label{fig: P-Z}
\end{figure}

\begin{rmk}\label{rmk: LO-P-sym}
	Note that the definition of a positive cone $P$ is completely symmetric with that of its negative cone $P\inv$. That is, if $P$ is a positive cone, so is $P\inv$.
	
	Moreover, if $P_\prec = \{g \in G \mid g \succ 1\}$ corresponds all the positive element under $\prec$, then 
	\begin{align*}
		P\inv_\prec 
		&= \{g\inv \in G \mid g \succ 1-\} \\
		&= \{g \in G \mid g \prec 1\} \\
		&= P_{\prec'} := \{g \in G \mid g \prec' 1\}
	\end{align*}
	where $\prec'$ corresponds to the reverse order of Lemma \ref{lem: opp-ord}. In other words, taking the reverse order swaps the roles of $P$ and $P^{-1}$.
\end{rmk}

\begin{ex}\label{LO-P-Z}
Let $\prec, \prec'$ be the left-orders of $\bZ$ as in Example \ref{ex: LO-Z-opp-ord}. Let $P_\prec$ be as in Example \ref{ex: LO-additive-positivity} and let $$P_{\prec'} = \{x \in \bZ \mid x \succ' 0\} = \{x \in \bZ \mid x \prec 0\} = P_\prec\inv.$$ The positive cone corresponding to the opposite order $\prec'$ is precisely $P_\prec \inv$. 
	
Furthermore, by specifying $x \succ 0$ or $x \prec 0$, we define either $P_\prec$ or $P_\prec\inv$ as a positive cone for $\bZ$. This in turns defines the only two left-orders $\prec, \prec'$ for $\bZ$ as in Example \ref{ex: LO-Z-opp-ord}.
\end{ex}

\begin{ex}[$\bZ^2$]\label{ex: P-Zsq}\label{ex: LO-P-Zsq}
Thinking of left-orders as positive cones can simplify identifying left-orders. Recall that in Example \ref{ex: LO-Z-square}, we had $\bZ^2 = \langle a, b \mid [a,b] \rangle$ and a left-order $\prec$ defined by $a^m b^n \prec a^p b^q$ if $m < p$ or $m = p$ and $n < q$, which corresponds to a lexicographic positive cone we denote as 
\begin{align*}
	P_{\pm \infty, <, >} &:= \{a^m b^n \mid m > 0 \text{ or } m = 0, n > 0 \} \\
	&= \{(x,y) \in \bZ^2 \mid x > 0 \text{ or } x = 0, y > 0 \}.
\end{align*}
(The notation will be explained shortly.)

\begin{figure}[h]{
\includegraphics{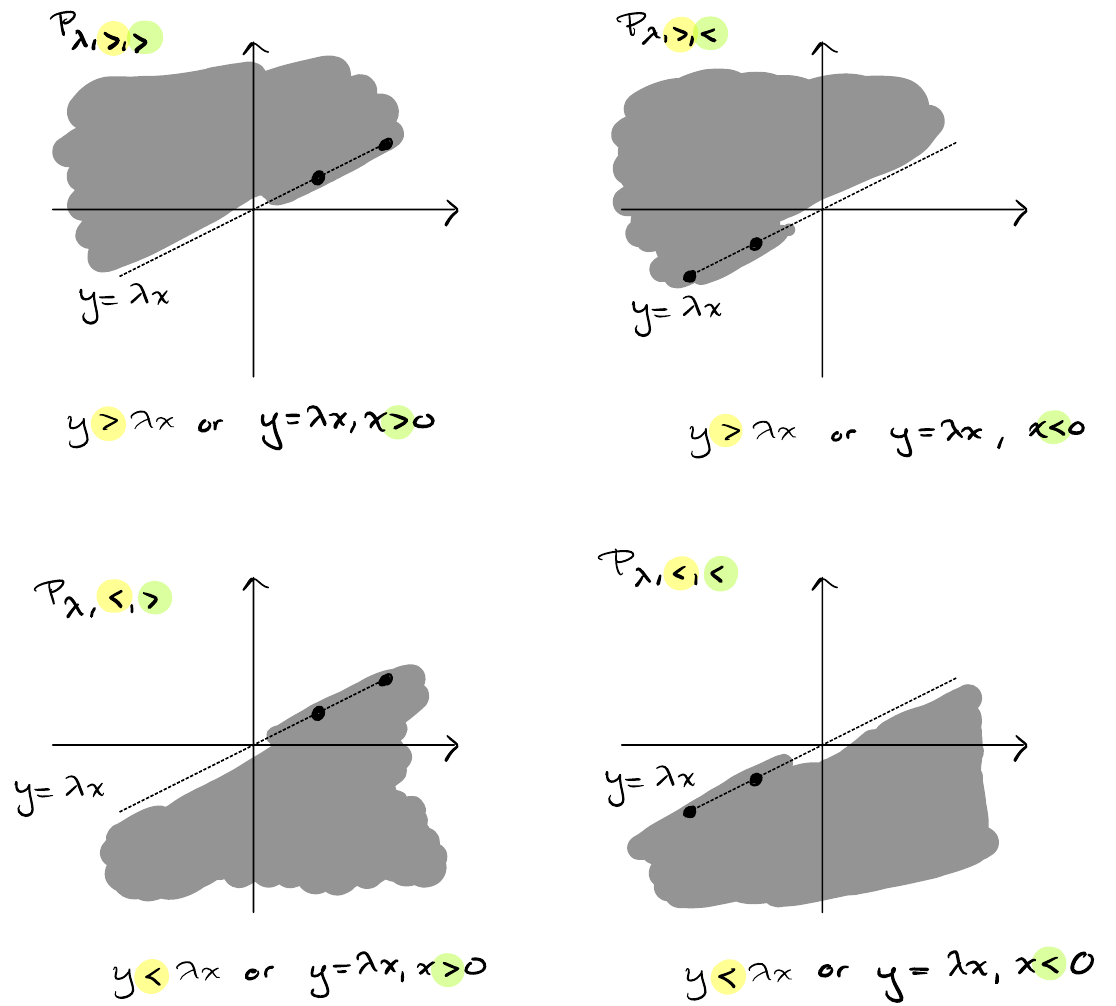}
}
\caption{Four positive cones on $\bZ^2$ associated with the slope $\lambda \in \bR$. The points in $\bZ^2$ are denoted by $(x,y)$, and the highlighted points in black are in the positive cone $P_i$.  If $\lambda$ is irrational, then $P_{\lambda, \diamond_1, \diamond_2} = P_{\lambda, \diamond_1, \diamond_2'}$ for the two possible choices of $\diamond_2, \diamond_2'$ since there are no points on the line $y = \lambda x$ where $(x,y) \in \bZ^2$. 

Note that when the dividing line corresponds to $x = 0$, we denote the associated positive cones with $\lambda = \infty$, as we think of it as the line such that $y = \infty x$.}
\label{fig: LO-Z-square}
\end{figure}

By choosing different sides of the inequalities above, we can get different positive cones as illustrated in Figure \ref{fig: LO-Z-square}. These positive cones all look like half-planes with an additional boundary component. Generalising our geometrical observation, we will show that $\bZ^2$ has uncountably many positive cones. 

Define $$P_{\lambda, \diamond_1, \diamond_2} := \{(x,y) \in \bZ^2 \mid y \diamond_1 \lambda x \text{ or } y = \lambda x, x \diamond_2 0 \}$$ where $\lambda \in \bR$, $\diamond_1, \diamond_2 \in \{<,>\}$. The first inequality $\diamond_1$ defines which $\bZ^2$ half-plane determined by the line $y = \lambda x$ is included, and $\diamond_2$ defines which side of the line $y=\lambda x$ is included from the origin (though excluding the origin itself). An increase in $\lambda$ corresponds to a counter-clockwise rotation of the slope, up to $\lambda = \infty$ (which is the same as $\lambda = -\infty$, and thus we write $\lambda = \pm \infty$). That is, for $P_{\pm \infty, <, >}$, the signs of $\diamond_1, \diamond_2$ are defined according to the convention of $|\lambda| < \infty$. 

Proving that the $P_{\lambda, \diamond_1, \diamond_2}$ are positive cones is fairly straightforward. Assume that $\lambda \not= \pm \infty$ since we have already shown the case for $\lambda = \pm \infty$. Fix $\lambda, \diamond_1, \diamond_2$ and let $P := P_{\lambda, \diamond_1, \diamond_2}$. Then, it is clear that 
$$\bZ^2 = P \sqcup \{(0,0)\} \sqcup P\inv,$$
since $(0, 0) \not\in P$ and if $(x,y) \in P$ then either $y \diamond_1 x$ meaning $-y \not \diamond_1 -x$, or $y = \lambda x$ and $x \diamond_2 0$, meaning that $-y = \lambda(-x)$ and $-x \not \diamond_2 0$. 

To show the semigroup closure of $P$, let $(x,y), (x',y') \in P$. Then, we have three cases. First, suppose that $y \diamond_1 \lambda x$ and $y' \diamond_1 \lambda x'$. Then, $(y + y') \diamond_1 \lambda (x + x')$. Second, suppose that $y \diamond_1 x$ and $y' = \lambda x', x' \diamond_2 0$. Then, $(y+y') \diamond_1 (x + x')$. Third, suppose that $y = \lambda x, x \diamond_2 0$ and $y' = \lambda x', x' \diamond_2 0$. Then, $(y+y') = \lambda (x+x'), (x+x') \diamond_2 0$. The three cases put together show that $(x,y) + (x',y') \in P$.

Next, let $$\cP_\lambda := \{P_{\lambda, \diamond_1, \diamond_2} \mid \diamond_1, \diamond_2 \in \{<,>\}\}$$ be the family of positive cones with slope $\lambda$. We will show that each $\lambda \in \bR$ defines a unique collection of positive cones, that is, for $\lambda \not= \lambda'$, $\cP_\lambda \cap \cP_{\lambda'} = \emptyset$. More precisely, $P_{\lambda, \diamond_1^\lambda, \diamond_2^\lambda} \not= P_{\lambda', \diamond_1^{\lambda'}, \diamond_2^{\lambda'}}$ for any $\diamond_1^\lambda, \diamond_2^{\lambda}, \diamond_1^{\lambda'}, \diamond_2^{\lambda'} \in \{<, >\}$. 

Suppose without loss of generality that $\lambda < \lambda'$, and denote $P_\lambda := P_{\lambda, \diamond_1^{\lambda}, \diamond_2^\lambda}$, $P_\lambda' := P_{\lambda', \diamond_1^{\lambda'}, \diamond_2^{\lambda'}}$. There are two cases. First, suppose that $\diamond_1^\lambda = \diamond_1^{\lambda'}$. Then, choose $(x, y) \in \bZ^2$ satisfying
$$\lambda x < y < \lambda' x.$$
Then, 
$$(x,y) \in \begin{cases}
	P_\lambda \text{ and } \not\in P_{\lambda'} & \diamond_1 = > \\
	P\inv_\lambda \text{ and } \not\in P\inv_{\lambda'} & \diamond_1 = < \end{cases},
$$
proving that $P_\lambda \not= P_{\lambda'}$ for this case. 

Second, suppose that $\diamond_1^{\lambda} \not= \diamond_1^{\lambda'}$. Then, choose $(x,y) \in \bZ^2$ satisfying 
$$y < \lambda x < \lambda' x.$$
Then, 
$$(x,y) \in \begin{cases}
	P_\lambda \text{ and } \not\in P_{\lambda'} & \diamond_1 = < \\
	P\inv_\lambda \text{ and } \not\in P\inv_{\lambda'} & \diamond_1 = >
\end{cases},$$
again showing that $P_\lambda \not= P_{\lambda'}$.

We note that if $\lambda \in \bQ \cup \{\pm \infty\}$, then $\cP_\lambda$ has four different positive cones. When $\lambda \in \bR - \bQ$, $P_{\lambda, \diamond_1, >} = P_{\lambda, \diamond_1, <}$ for $\diamond_1 \in \{<,>\}$ since there are no elements $(x,y) \in \bZ^2$ that satisfy $y = \lambda x$.  Since there is at least one positive cones associated with each $\cP_\lambda$ for $\lambda \in \bR$, there are uncountably many positive cones for $\bZ^2$. 

The work of Teh \cite{Teh1961} shows that this is a full characterisation of every positive cone of $\bZ^2$.\sidenote{I was not able to access the original manuscript, but my intuitive justification is that this is because for any vectors $u,v$ belonging to a positive cone $P$, it is clear that $u+v$ and $nv$ for $n \in \nats$ are both in $P$, and since $\bZ^2 = P \sqcup P\inv \sqcup \{(0,0)\}$, it must be always the case that $P$ is shaped like a half-plane. } 
\end{ex}

\begin{figure}[h]
\centering
{
\includegraphics[width = \textwidth]{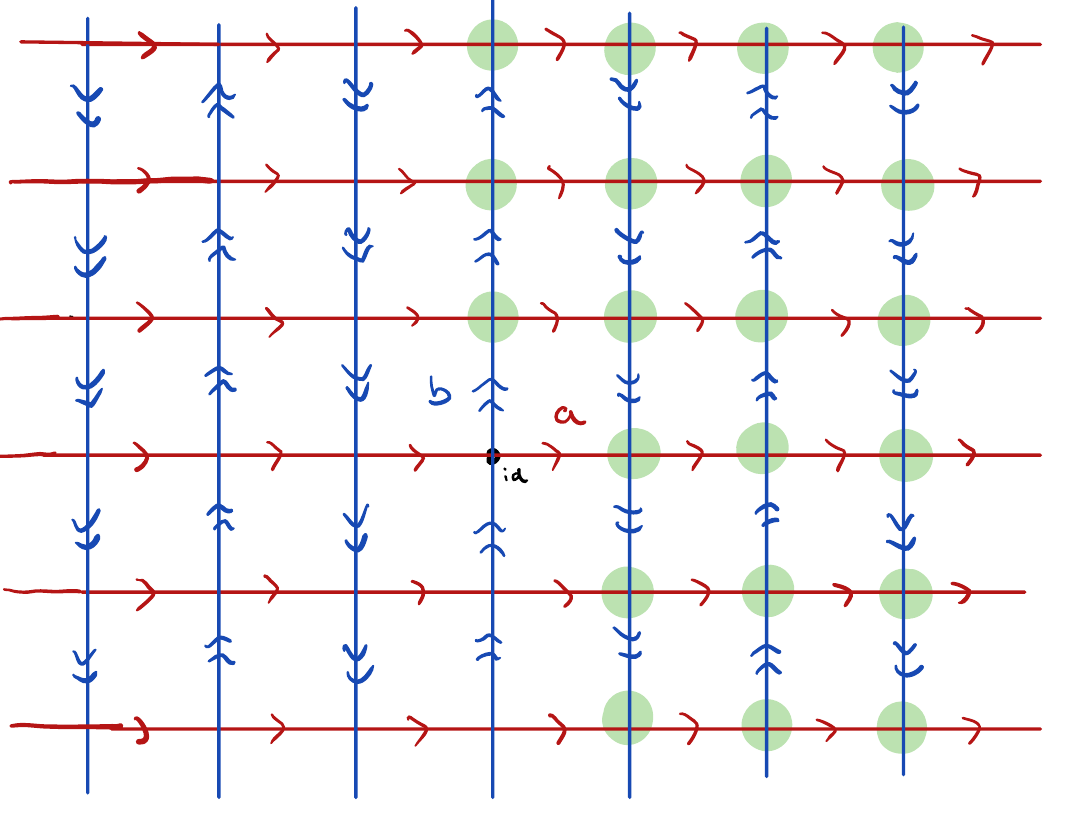}
}
\caption{A Cayley graph for the Klein bottle group given by $K_2$. The positive elements reside in the right-hand side of the graph, and depend on having a positive number of $a$'s.  

This figure is referred to in Section \ref{sec: fg-Gamma_n}.

}
\label{fig: P-K2}
\setfloatalignment{b}
\end{figure}

\begin{ex}[$K_2$]\label{ex: P-K2}\label{ex: LO-P-K2}
Recall that in Example \ref{ex: LO-K2} we had a lexicographical left-order $\prec$ on the Klein bottle group $K_2$ defined by writing every elements as $a^m b^n$ and defining $a^m b^n \prec a^p b^q$ by $m<p$ or $m = p$ and $n < q$ for some left-order $<$ in $\bZ$. 

Similarly to $\bZ^2$, the positive cone associated with $\prec$ is given by $P_\prec := \{a^m b^n \mid m > 0 \text{ or } m = 0, n > 0\}$ and is illustrated in Figure \ref{fig: P-K2}. 

An interesting fact is that this positive cone can be described as a finitely generated semigroup $P_\prec = \langle a, b \rangle^+$. One way to see this is through Figure \ref{fig: P-K2}: the relation $bab=a$ causes the $b$-edges in the Cayley graph of $K_2$ to alternate direction every time we pass an $a$-edge. As a result, the semigroup $\langle a, b\rangle^+$ (all words in $a,b$ with no inverses) in $K_2$ fills roughly half of the Cayley graph (excluding the identity). By contrast, in the Cayley graph of $\mathbb{Z}^2 = \langle a,b \mid [a,b]\rangle$, the semigroup $\langle a,b\rangle^+$ would only fill a quarter of the graph (since there $a$ and $b$ commute).

A formal way to see this is as follows. Suppose that $a^m b^n \in P_\prec$. Assume that $m > 0$ and $n < 0$. Then recall the equation $a b^n = b^{-n} a$ for $m,n \in \bZ$ found in Example \ref{ex: LO-K2}. Since $m > 0$, we can use this relation at least once to rewrite $a^m b^n$ as $a^{m-1}b^{-n} a \in \langle a, b \rangle^+$. For the other cases, if $m < 0$, then $a^m b^n \not\in P_\prec$, and if $m = 0$, then $n > 0$ so $b^n \in \langle a, b \rangle^+$. We have shown that $P_\prec \subseteq \langle a, b \rangle^+$.

To conclude that $\langle a, b \rangle^+ \subseteq P_\prec$, observe that $a, b \in P_\prec$ by definition, and use the semigroup closure property of $P_\prec$.

There are four positive cones for $K_2$ in total. Indeed, by choosing $a^{\epsilon_a}, b^{\epsilon_b} \in P$, for $\epsilon_a, \epsilon_b = \pm 1$, the four possible positive cones given by $P = \langle a^{\epsilon_a}, b^{\epsilon_b}\rangle^+$ are uniquely determined (by the choice of signs) and each looks like a half-plane.

\end{ex}

\begin{ex}[$B_3$]\label{ex: LO-B3}
The braid groups are meant to capture the algebraic properties of braids on $n$ strings, and have standard presentation 
$$B_n = \langle \sigma_1, \dots, \sigma_{n-1} \mid  \sigma_i \sigma_j = \sigma_j \sigma_i \text{ if } |i - j| > 1, \quad \sigma_i \sigma_j \sigma_i = \sigma_j \sigma_i \sigma_j \text{ if } |i-j| = 1 \rangle,$$
where each $\sigma_i$ represents a crossing between two adjacent strings (see \cite{Glasscock2012} for a quick introduction). 

The first left-order on $B_n$ was discovered by Dehornoy \cite{Dehornoy1999}. The \emph{Dehornoy ordering} of $B_n$ is given by a positive cone $P$, defined as follows. A word $w$ is \emph{$i$-positive} (resp. \emph{$i$-negative}) when it contains at least one occurrence of $\sigma_i$, no occurrence of $\sigma_{1}, \dots, \sigma_{i=1}$ and every occurrence of $\sigma_i$ has a positive (resp. negative) exponent. The positive cone $P$ is the set of all elements which can be written as an $i$-positive word for some $1 \leq i \leq n$.  Note that if an element is not $i$-positive $i$-negative for any $i$, we can use the relation to rewrite it to obtain such properties. For example, the word $\sigma_1^2 \sigma_2\inv \sigma_1\inv$ is neither $i$-positive nor $i$-negative, but can be rewritten as $\sigma_2\inv \sigma_1\inv \sigma_2^2$ which is $1$-negative (see \cite[Chapter 7]{Clay2016} for a more in-depth overview). It is not trivial to see that such a left-order is well-defined, and there were efforts to find more straightforward orderings. 

In 2010, Navas discovered a finitely generated positive cone for $B_3$ similar to the one for the Klein bottle group \cite{Navas2010}. By rewriting the Klein bottle group as $$K_2 = \Gamma_1 = \langle a, b \mid bab = a \rangle,$$ and the braid group $B_3$, using the isomorphism $a = \sigma_1 \sigma_2, b = \sigma_2\inv$, as $$\Gamma_2 = \langle a, b \mid ba^2b = a\rangle,$$ we can generalise $K_2$ and $B_3$ to be part of the same family $$\Gamma_n = \langle a,b \mid ba^nb = a \rangle.$$ Navas showed that this family of groups all have finitely generated positive cone given by $$P_n = \langle a, b \rangle^+.$$
	
	In Chapter \ref{chap: fg}, we will discuss $\Gamma_n$ more in depth and explain some of our own results about finitely generated positive cones. 
\end{ex}

\subsection{Relative positive cones}\label{sec: rel-ord}
We first introduce a generalization on positive cones, which we call \emph{relative positive cones}. They are the positive cone counterparts convex subgroups, as we will see. 

\begin{defn}
Let $G$ be a group and $H$ be a subgroup of $G$. A set $P\rel \subseteq G$ is a \emph{positive cone relative to} $H$ if it is a subsemigroup and $G=P\rel \sqcup H \sqcup {P\rel}\inv$ where the union is disjoint.
\end{defn}

The term relative positive cone can be justified using the following lemma. 

\begin{lem}
If $P\rel$ is a positive cone relative to $H$, then we can define a total $G$-invariant order $\prec'$ on $G/H$ by setting $g_1 H \prec' g_2 H \iff g_1 \inv g_2 \in P\rel$. 
\end{lem}

\begin{proof}
Suppose that $P\rel$ is a positive cone relative to $H$. We start by showing that $P\rel H \subseteq P\rel$.
Suppose not. Then $P\rel H$ will have a non-trivial intersection with ${P\rel}\inv \cup H$ since $G=P\rel \sqcup H \sqcup {P\rel}\inv$. We will show this leads to a contradiction. 

Since ${P\rel}\inv$ and $H$ are disjoint, we may suppose first that ${P\rel} H \cap {P\rel}\inv \not= \emptyset$, and that $g$ is an element in this intersection. Then $g$ can be decomposed as $g = p_1 h = p_2\inv$ where $h\in H$ and $p_1, p_2 \in {P\rel}$. This means that $h = p_1\inv p_2\inv$, so $H$ is not disjoint from ${P\rel}\inv$, a contradiction. 

Suppose second that ${P\rel} H \cap H$ is non-empty. Then, there exists an element $g \in G$ such that $g = ph_1 = h_2$ where $p \in {P\rel}$ and $h_1, h_2 \in H$. This means that $p = h_2 h_1\inv \in P\rel\inv$ thanks to its semigroup closure property. Therefore, ${P\rel}$ and $H$ are not disjoint, a contradiction. This completes the proof that $P\rel H \subseteq P\rel$.

Observe that $P\rel$ is also a positive cone relative to $H$. Following the reasoning above but replacing $P\rel$ by $P\rel\inv$, we obtain that $P\rel\inv H \subseteq P\rel \inv$. By inverting on both sides, we obtain that $HP\rel \subseteq P\rel$.  

Now that we have shown that $H {P\rel}, {P\rel} H \subseteq {P\rel}$, and thus that $H {P\rel} H \subseteq {P\rel}$, we state that if ${P\rel}$ is a positive cone relative to $H$, then we can define a total $G$-invariant order on $G/H$ by setting $g_1H \prec' g_2H\Leftrightarrow g_1^{-1}g_2\in {P\rel}$. 
This is well-defined, since if we pick different coset representatitives such that $g_1'=g_1h$ and $g_2'=g_2h'$ then $g_1^{-1}g_2\in {P\rel}$ implies that $h^{-1}g_1^{-1}g_2 h'\in {P\rel}$. 

The fact that $\prec'$ is a total $G$-left-invariant order on $G/H$ follows straightforwardly from $gg_1H \prec' gg_2H \iff g_1\inv g\inv g g_2 \in {P\rel} \iff g_1\inv g_2 \in {P\rel} \iff g_1H \prec' g_2H$. 
\end{proof}

\begin{lem}\label{lem: ord-convex-rel-cone}
If $H$ is $\prec$-convex for a left-order $\prec$ on G, then $P\rel = \{g \in G \mid H \prec g\}$ is a positive cone relative to $H$.
\end{lem}
\begin{proof}
First, we observe that for all $g\in G- H$ either $g\prec H$ or $H\prec g$. Indeed, if $g \not\in H$ and all $h\in H$, we have that either $g \prec h$ or $h \prec g$ as otherwise we would have that $h \prec g \prec h'$ for $h, h' \in H$, which by $\prec$-convexity of  $H$ implies that $g \in H$. 

Now, if $H \prec g$, then $g\inv \prec 1\in H$. 
As $g^{-1}\notin H$ this means  that $g \inv \prec  H$. 
This shows that $G = P \sqcup P\inv \sqcup H$. 
Finally, to show that $P$ is a semigroup, notice that since $1 \in H$, we have that $1 \prec g_1$ and $1 \prec g_2$ for $g_1, g_2 \in P$. 
This implies that $H \prec g_1 \prec g_1 g_2$. 
\end{proof}

\begin{ex}
Take $G = \bZ^2 = \{(x,y) \mid x,y \in \bZ\}$, $H = \bZ = \{(0,y) \mid y \in \bZ\}$ with the standard lexicographic order as in Example \ref{ex: convexity-Z-Zsq}. Then, 
$$P\rel = \{(x,y) \in \bZ^2 \mid x > 0\}$$
is a positive cone relative to $\bZ$. We can easily see that 
$$\bZ^2 = \{(x,y) \in \bZ^2 \mid x > 0\} \sqcup \{(0,y) \mid y \in \bZ\} \sqcup \{(x,y) \in \bZ^2 \mid x < 0\}.$$
Furthermore, we can set a $G$-invariant order $\prec'$ on $$\bZ^2 / \bZ = \{(x,0) \mid x \in \bZ\} \cong \bZ$$ by $$(x,\bZ) \prec' (x',\bZ) \iff (x',0) - (x,0) \in P\rel \iff x'-x > 0.$$ 
\end{ex}

The concept of relative positive cone will help us construct the required positive cones for the closure results we will later prove in this thesis. 

\section{Right-orders and bi-orders}

The ``left'' in left-order refers to its invariance under left-multiplication. There is a corresponding definition of \emph{right-order}\index{right-order} which is given by invariance under right-multiplication. 

\begin{lem}[]\label{lem: LO-RO-correspondence}
Left-orders are in one-to-one correspondence with right-orders. Indeed, suppose that $(G,\prec)$ is left-orderable. Then we may define a corresponding right-order $\prec'$ by 
$$ g \prec' h \iff g\inv \prec h\inv.$$
\begin{proof}
The left-order $\prec'$ is strict and total, and $gf \prec' hf \iff f\inv g\inv \prec f\inv h\inv \iff g\inv \prec h\inv \iff g \prec' h$.
\end{proof}
\end{lem}

An order that is both left- and right-orderable is called a \emph{bi-order}.

\begin{defn}
	A group $G$ is \emph{bi-orderable} if there exists a strict total order $\prec$ on the elements of $G$ which is invariant under left-multiplication and right multiplication, that is, 
$$ g \prec h \iff fgf' \prec fhf', \qquad \forall g,h,f,f' \in G.$$
\end{defn}

\begin{rmk}
	Note that a positive cone defines a left $\prec$ and a right order $\prec'$ (but not necessarily a bi-order). Indeed, given a positive cone $P$, we get $$g \prec h \iff f \prec fh \iff g\inv f\inv f h \in P$$ and $$g \prec' h \iff gf \prec' hf \iff hff\inv g\inv \in P.$$ 
\end{rmk}

To define a bi-order, a positive cone needs the extra property that it closed under conjugation. 

\begin{lem}\label{lem: P-bi-ord}
	Let $G$ be a left-orderable group with positive cone $P$. Then, $P$ defines a bi-order $\prec$ by $$g \prec h \iff g\inv h \in P$$ 
	if and only if $$fPf\inv \subseteq P$$ for all $f \in G$.  
\end{lem}
\begin{proof}
	$(\implies)$ Suppose that $P$ defines a bi-order $\prec$. Then, $$1 \in P \iff 1 \prec g \iff 1 \prec fgf\inv \iff fgf\inv \in P.$$
		
	$(\impliedby)$ Let $P$ be a positive cone for $G$ with the property that $fPf\inv \subseteq P$ for all $f \in G$. We already know that $P$ defines a left-order $\prec$ and we need to show that $\prec$ is also a right-order. That is, we need to show that for all $f' \in G$, 
	$$g \prec h \iff gf' \prec hf' \iff (gf')\inv(hf') = f'^{-1} g\inv h f \in P.$$
	Setting $f = f^{-1'}$, we have that $fPf\inv \subseteq P$, which satisfies the necessary condition of the proof.  
\end{proof}

\begin{ex}[$\bZ, \bZ^2$]
Any left-order on a free abelian group is a bi-order, since the left-multiplication operation is abelian and hence left-invariance extends to right-invariance.
\end{ex}

\begin{non-ex}[$K_2$]
Recall the Klein bottle group $K_2$ of Example \ref{ex: LO-K2}. We claim that no left-order on this group can be a bi-order. Indeed, recall that we had the relation $ab = b\inv a$. Without loss of generality, suppose that $\prec$ is a bi-order such that $1 \prec b$ (replacing $\prec$ by its ``opposite order'' as in Lemma \ref{lem: opp-ord} as necessary). Then left-multiplying by $a$ and by left-invariance, $a \prec ab = b\inv a$. By multiplying on the right by $a\inv$ and using right-invariance, this is equivalent to $1 \prec b\inv$, which is equivalent, multiplying both sides by $b$, to $b \prec 1$ which contradicts our assumption.
\end{non-ex}

Since every left-order induces a right-order and vice-versa, we will only mention left-orders (or positive cones) and bi-orders. Moreover, we will mostly talk about left-orders in this thesis. 

\section{Closure properties of left-orders}
\subsection{Maximality}

\begin{lem}\label{lem: lo-max-subset}
	If $P, Q$ are positive cones such $P \subseteq Q$, then $P = Q$. 
\end{lem}
\begin{proof}
	Suppose that there is some element $q$ in $Q-P$. Then, since $q$ is not in $P$, it must be by the trichotomy property that $q\inv \in P$. But then $P$ is not contained in $Q$, a contradiction.
\end{proof}

\subsection{Closure under isomorphisms}

\begin{lem}\label{lem: lo-clos-isom}
Let $(G, \prec)$ be a left-orderable group and $\phi: G \to H$ be an isomorphism. Then $H$ is a left-orderable group with a left-order $\prec'$ given by $$h_1 \prec' h_2 \iff \phi\inv(h_1) \prec \phi\inv(h_2).$$ for all $h_1, h_2 \in H$.  
\end{lem}
\begin{proof}
The order $\prec'$ is strict and total since $\prec$ is strict and total. To prove left-invariance, let $f \in H$. Then, 
\begin{align*}
fh_1 \prec' fh_2 &\iff \phi\inv(fh_1) \prec \phi\inv(fh_2) \\
&\iff \phi\inv(f)\phi\inv(h_1) \prec \phi\inv(f)\phi\inv(h_2) \text{ since $\phi$ is an isomorphism.} \\
&\iff \phi\inv(h_1) \prec \phi\inv(h_2) \text{ since $\prec$ is a left-order.} \\
&\iff h_1 \prec' h_2 \text{ by definition of $\prec'$.}
\end{align*}
\end{proof}

We can apply Lemma \ref{lem: lo-clos-isom} to automorphisms and define equivalence classes of left-orders.

\begin{lem}
Let $G$ be a left-orderable group with left-order $\prec$ and $\Aut(G)$ be the automorphism group acting on $G$. Then the relation $\prec \sim <$ if and only if $$\exists \phi \in \Aut(G) \text{ such that } g \prec h \iff \phi(g) < \phi(h) \text{ for all } g,h \in G$$ defines an equivalence class.
\end{lem}
\begin{proof}
The relation is reflexive since $\prec \sim \prec$ by the identity automorphism. It is symmetric since if $\prec \sim <$, then there exists an automorphism $\phi$ such that $g \prec h \iff \phi(g) < \phi(h)$. Then, $\phi\inv(g) \prec \phi\inv(h) \iff \phi(\phi\inv(g)) < \phi(\phi\inv(h)) \iff g < h$, and therefore $< \sim \prec$ by the $\phi\inv$ automorphism. Finally, the relation is transitive since if $\prec \sim <$ by the $\phi$ automorphism and $< \sim <'$ by the $\psi$ automorphism, then $g \prec h \iff \phi(g) < \phi(h) \iff \psi(\phi(g)) <' \psi(\phi(h))$. Therefore, $\prec \sim <'$ by the $\psi \circ \phi$ automorphism. 
\end{proof}

\begin{ex}[$\bZ$]
The group $\bZ$ has two left-orders and only one equivalence class of left-orders. Indeed, let $\prec, \prec'$ be the left-orders defined by $x \prec x + 1, x + 1 \prec' x$ for all $x \in \bZ$ as in Example \ref{ex: LO-Z-opp-ord}. Let $\phi: \bZ \to \bZ$ be the automorphism sending $x \mapsto -x$ for all $x \in \bZ$. Let $\prec'_\phi$ be the left-order defined as $g \prec h \iff \phi(g) \prec'_\phi \phi(h)$. Then $x \prec x + 1 \iff -x \prec'_\phi -x - 1$. Now adding $2x + 1$ on both sides of the left order, we get that $x + 1 \prec'_\phi x$ for all $x \in \bZ$. Therefore $\prec' = \prec'_\phi$ showing that $\prec, \prec'$ are in the same equivalence class of left-orders.
\end{ex}

\begin{ex}[$\bZ^2$]\label{ex: LO-aut-Zsq}

The automorphism group of $\bZ^2$ is known to be 
$\GL(2, \bZ) = \{
\bigl( \begin{smallmatrix} 
p & r \\ q & s
\end{smallmatrix} \bigr) 
\mid p, r, q, s \in \bZ, |ps - rq| = 1\}$ the group of invertible matrices over $\bZ$ with determinant $\pm 1$.\sidenote{
	A direct way to see this is by taking the standard generating set $\hat x = (1,0), \hat y = (0,1)$. Then any endomorphism $\phi: \bZ^2 \to \bZ^2$ is determined by $\phi(\hat x), \phi(\hat y)$ and can be realized by multiplying by a matrix $M$ whose columns are $\phi(\hat x), \phi(\hat y)$. The endomorphism is bijective if and only if the determinant of $M$ is $\pm 1$, for otherwise the determinant of either $M$ or $M\inv$ is a fraction and thus the matrices cannot be obtained from $\phi(\hat x), \phi(\hat y) \in \bZ^2$. (Thank you Yago Antol\'in for showing me this!)
	
	A second way to see this, is by using the Dehn-Nielser-Baer theorem and identify $\Aut(\bZ^2)$ with the mapping class group of a torus, which is known to be $\SL(2,\bZ)$.  (Thank you Thomas Ng for showing me this!)
}

Recall the positive cones $P_{\lambda, \diamond_1, \diamond_2}$ of $\bZ^2$ as defined in Example \ref{ex: LO-P-Zsq}. 

$$P_{\lambda, \diamond_1, \diamond_2} = \{(x,y) \in \bZ^2 \mid y \diamond_1 \lambda x \text{ or } y = \lambda x, x \diamond_2 0 \}.$$

We claim that for rational slopes $\lambda \in \bQ \cup \{\infty\}$ and orientations $\diamond_1, \diamond_2 \in \{<,>\}$, the associated positive cones $P_{\lambda,\diamond_1, \diamond_2}$ are equivalent under automorphism class $\GL_2(\bZ)$. Indeed, for $\lambda \in \bQ \cup \{\infty\}$, $P_{\lambda, \diamond_1, \diamond_2}$ can be uniquely identified with a defining parallelogram $\hat p$ of area $1$ whose sides are given by basis vectors $\hat x, \hat y \in \bZ^2$ as illustrated in Figure \ref{fig: P-Z-square-parallelogram}. 

\begin{figure}[h]{
\includegraphics{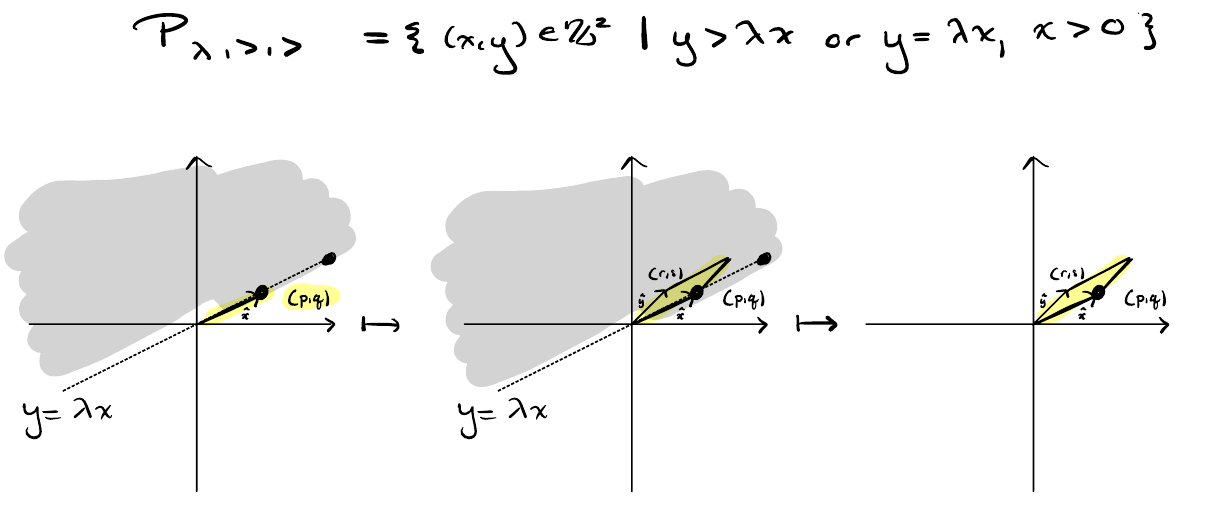}
}
\caption{Illustrating how for each $(p/q, \diamond_1, \diamond_2)$ there is only one choice of basis $\hat x, \hat y$ such that the resulting parallelogram belongs to $P_{p/q, \diamond_1, \diamond_2}$.}
\label{fig: P-Z-square-parallelogram}
\setfloatalignment{b}
\end{figure}

Indeed, if $\lambda = \pm \infty$, then 
\begin{align*}
	&\hat x = \begin{cases}
		(0,1) & \diamond_2 = > \\
		-(0,1) & \diamond_2 = < 
	\end{cases}, \\
	&\hat y = \begin{cases}
	(1,0) & \diamond_1 = < \\
	-(1,0) & \diamond_1 = >
	\end{cases}.
\end{align*}

Otherwise, if $\lambda$ is given by some irreducible fraction $\lambda = b/a$, then 
\begin{align*}
	&\hat x := (p,q) = \begin{cases}
	(a,b) & (a,b) \in P_{b/a, \diamond_1, \diamond_2} \\
	-(a,b) & (a,b) \in P\inv_{b/a, \diamond_1, \diamond_2}
	\end{cases},\\
	&\hat y := (r,s) = (c,d) \text{ such that } |ad-bc| = 1 \text{ and } (c,d) \in P_{b/a, \diamond_1, \diamond_2}.
\end{align*}
Essentially, there are only two choices for $\hat x, \hat y$ given by $\pm (a,b), \pm (c,d)$ respectively depending on the signs of $\diamond_1, \diamond_2$. 

Define $f$ as the map $$f(\lambda, \diamond_1, \diamond_2) = \hat p = (\hat x, \hat y),$$
for $\lambda \in \bQ \cup \{\pm \infty\}$, and $\diamond_1, \diamond_2 \in \{<,>\}$.  We will show that this map is a bijection. 

Suppose first for contradiction that the map is not injective, and that for some $(\lambda, \diamond_1, \diamond_2) \not= (\lambda', \diamond_1', \diamond_2')$, we have that
$$(\hat x, \hat y) = f(\lambda, \diamond_1, \diamond_2) = f(\lambda', \diamond_1', \diamond_2') = (\hat x', \hat y').$$
Then, if $\lambda \not= \lambda'$, we have that $\hat x \not= \hat x'$, so it must be that $\lambda = \lambda'$. Now suppose that $\lambda = \lambda' \not= \pm \infty$, and that either $\diamond_1 \not= \diamond_1'$ and $\diamond_2 \not= \diamond_2'$. Since $|ps-rq| \not= 0$, it must be that $\hat x$ and $\hat y$ are not parallel, and furthermore since $q/p = \lambda$ it must be that $\hat x = (p,q)$ satisfies $q = \lambda p, p \diamond_2 0$ and $\hat y = (r,s)$ satisfies $s \diamond_1 \lambda r$. Since we have that $\hat x' = \hat x$ and $\hat y' = \hat y$, we either have simultaneously that $s > \lambda r$ and $s < \lambda r$ if $\diamond_1 \not= \diamond_1'$ or $p > 0$ and $p < 0$ for $\diamond_2 \not= \diamond_2'$. In the $\lambda = \lambda' = \pm \infty$ case, it is clear that each value of $(\diamond_1, \diamond_2)$ defines a unique $(\hat x, \hat y)$ respectively. 

To show surjectivity, we must show that for any $\hat p = (\hat x, \hat y)$ where $\hat x = (p,q), \hat y = (r,s)$ satisfying that $q/p$ is irreducible and $|ps-rq| = 1$, there exists some $\lambda \in \bQ \cup \{\pm \infty\}$ and $\diamond_1, \diamond_2 \in \{<,>\}$ such that $f(\lambda, \diamond_1, \diamond_2) = \hat p$. The case where $\hat x = \pm (0,1), \hat y = \pm (1,0)$ is clear, so let us assume that $q/p \in \bQ$. 

Let $\lambda = q/p$, and 
$$\diamond_1 = \begin{cases}
	>, & s > \lambda r \\
	<, & s < \lambda r
\end{cases},$$
$$\diamond_2 = \begin{cases}
	>, & p > 0 \\
	<, & p < 0
\end{cases}.$$
Then, $$f(\lambda, \diamond_1, \diamond_2) = ((p,q), (r,s)) = (\hat x, \hat y)$$ by construction.

Since $\hat p$ forms a parallelogram of area one, the matrix whose columns correspond to $\hat x, \hat y$ is in $\GL(2,\bZ)$. Since $\GL(2,\bZ)$ acts transitively on itself, it acts transitively on all parallelogram with integer coordinates of area $1$, and therefore on all positive cones with rational slopes in $\bZ^2$. One way to see this is to write 
$$P_{0, >, >} = \{\gamma (1,0) + \gamma_1(0,1) \mid \gamma \in \bZ, \gamma_1 > 0\} \cup \{\gamma_2(1,0) \mid \gamma_2 > 0\},$$

Then, there exists $\rho \in \GL(2, \bZ)$ such that $$\phi(P_{0, >, >}) = P_{\lambda, \diamond_1, \diamond_2} = \{\gamma (p,q) + \gamma_1(r,s) \mid \gamma \in \bZ, \gamma_1 > 0\} \cup \{\gamma_2(p,q) \mid \gamma_2 > 0\}$$ where $\phi = \bigl( \begin{smallmatrix} 
p & r \\ q & s
\end{smallmatrix} \bigr)$ as given by $f(\lambda, \diamond_1, \diamond_2) = (\hat x, \hat y) = ((p,q),(r,s))$. Indeed, $\phi$ sends $(1,0)$ to $\hat x$ and $(0,1)$ to $\hat y$. (We know $\phi$ is bijective because $\phi$ is an automorphism.)

We conclude by saying that any positive cone with rational or infinite slope in $\bZ^2$ are automorphic.

\end{ex}

\begin{ex}[$K_2$] 
	As seen in Example \ref{ex: LO-P-K2}, the Klein bottle group has only four positive cones which are all isomorphic under the automorphisms 
\begin{align*}
	\phi_a(a) = a\inv, \qquad \phi_a(b) = b \\
	\phi_b(a) = a, \qquad \phi_b(b) = b\inv \\
	\phi_a \phi_b(a) = a\inv, \qquad \phi_a(a) \phi_b(b) = b\inv. 
\end{align*}
	
\end{ex}

\subsection{Closure under subgroup}
\begin{lem}\label{lem: clos-subgroup}\index{left-order!closure!subgroup}
If $(G,\prec)$ is a left-orderable group and $H$ is a subgroup of $G$, then $(H, \prec)$ is also a left-orderable group. 
\end{lem}

\begin{proof}
Since $\prec$ is a total strict order that is left-invariant on $G$, $\prec$ is a total strict order that is left-invariant on $H$.
\end{proof}

\begin{ex}[$\bZ$]
The group of integers has $n\bZ$ as subgroup for any $n \in \bN$. If $\prec$ is a left-order on $\bZ$, then clearly, any $(n\bZ, \prec)$ is also a left-orderable group. 
\end{ex}

\begin{ex}[$\mathbb{Z}^2 \leq K_2$]
	Since $\bZ^2$ is a subgroup of $K_2$, it inherits its four positive cones. It is interesting to note how many more left-orders arise on $\bZ^2$ outside of those for $K_2$! %
\end{ex}

\subsection{Closure under group extension}
\label{sec: LO-clos-ext}

We are ready to study left-orders on group extensions. Suppose that $Q$ and $N$ are groups and that $G$ is an extension of $Q$ by $N$. That is, there is a short exact sequence $$1\to N \to G \stackrel{f}{\to} Q \to 1.$$
Fixing a right inverse  $s$ of $f$  (i.e. a transversal of $f$), with $s(1_Q)=1_G$, we obtain the set bijection 
$$G\to N \times Q, \qquad g\mapsto ( g \cdot s(f(g))^{-1}, f(g)).$$ 
We will use this bijection to construct  natural lexicographic left-orders on extension $G$. 

\begin{defn}[Lexicographic order on direct products of totally ordered sets]\label{defn: lexicographic-leading}
Let $(N,\prec_N)$ and $(Q,\prec_Q)$ be two totally ordered sets. \begin{itemize}
    \item  The {\it lexicographic order $\prec_{lex}$ on $N\times Q$ with leading factor $Q$} is given by $(n,q)\prec_{lex} (n',q')$ if and only if  $q\prec_Q q'$ or $q=q'$ and $n\prec_N n'$.
    \item The {\it lexicographic order $\prec_{lex}$ on $N\times Q$ with leading factor $N$} is given by $(n,q)\prec_{lex} (n',q')$ if and only if $n\prec_N n'$ or $n=n'$ and $q\prec_Q q'$.

\end{itemize}
\end{defn}

\begin{lem}
\label{lem: quotient-leads}
\label{lem: LO-clos-ext}
Let $f\colon G\to Q$ be a group epimorpishm  with kernel $N$.
Let $P_Q$ and $P_N$ be positive cones on $Q$ and $N$ respectively.
Then  $P_{\prec_{lex}}\coloneqq f^{-1}(P_{Q})\cup  P_{N}$ is a positive cone of $G$.
The left-order on $G$ induced by $P_{\prec_{lex}}$ is called \emph{ lexicographic led by the quotient $Q$}.
\end{lem}
\begin{proof}
We want show that $f^{-1}(P_Q)$ is a positive cone relative to $N$, then use Lemma \ref{lem: lang-from-Prel} to finish the proof. 

Observe that the $f^{-1}(P_Q)$ is a subsemigroup as it is the pre-image of a semi-group under a homomorphism. Indeed, if $r,s \in f\inv(P_Q)$, then $f(rs) = f(r)f(s) \in P_Q$, so $rs \in f\inv(P_Q)$ as well. 

As for the trichotomy property, we use that $Q=P^{-1}_Q\sqcup \{1_Q\}\sqcup P_Q$, and apply the pre-image to get $G= f^{-1}(P_Q)^{-1}\sqcup N \sqcup f^{-1}(P_Q)$. This shows that $f^{-1}(P_Q)$ is a positive cone relative to $N$. The result then follows from Lemma \ref{lem: lang-from-Prel}.
\end{proof}
\begin{cor}\label{cor: P-ext-bi-ord}
Let $G$, $Q$, $N$, $f$, $P_{\prec_{lex}}$, $P_Q$ and $P_N$ be as in the statement of the previous Lemma \ref{lem: LO-clos-ext}. If $Q$ and $N$ are bi-ordered with positive cones $P_Q, P_N$ respectively, then $P_{\prec_{lex}}$ defines a bi-order if and only if the conjugation action of $G$ on $N$ preserves the positive cone of $N$. That is, if and only if $$hP_Nh\inv \subseteq P_N$$ for all $h \in G$. 
\end{cor}
\begin{proof}
	Indeed, let $P := P_{\prec_{lex}}$. By Lemma \ref{lem: P-bi-ord}, $P$ defines a bi-order if and only if for all $h \in G$, $hPh\inv \subseteq P$. 
	
	First observe that since $Q$ is assumed to be bi-orderable, we have for all $h \in G$ that 
	\begin{align*}
		 f(h)P_Qf(h)\inv &\subseteq P_Q \\
		\implies f\inv(f(h)P_Qf(h)) &\subseteq f\inv(P_Q).
	\end{align*}
	In particular, since $hf\inv(P_Q)h\inv \subseteq f\inv(f(h)P_Qf(h))$, we have that 
	$$hf\inv(P_Q)h\inv \subseteq f\inv(P_Q).$$
	
	($\impliedby$) Suppose that $hP_Nh\inv \subseteq P_N$ for all $h \in G$. Then, $$hPh\inv = h f\inv(P_Q) h\inv \cup hP_Nh\inv \subseteq f\inv(P_Q) \cup P_N = P.$$ 
	
	($\implies$) Suppose that for all $h \in G$, $hPh\inv \subseteq P$. That is, for all $g \in P$, $hgh\inv \in P$. Let $n \in P_N$. We want to show that $hnh\inv \in P_N$. 
	
	Suppose not, and that there exists some $n \in N$, $h \in G$ such that $hnh\inv \not\in P_N$. Then, $hnh\inv \in f\inv(P_Q)$ since $hnh\inv \in P - P_N$. That is, 
	\begin{align*}
		 n \in h\inv f\inv(P_Q)h &\subseteq f\inv(P_Q), \\
		\implies f(n) = 1_Q &\in P_Q,
	\end{align*}
	since by assumption $N$ is the kernel of $f$ in $G$. This is our desired contradiction. 
	
\end{proof}

In general, if $G$ is a $Q$-by-$N$ extension of left-orderable groups, the lexicographic order on $N\times Q$ with leading factor $N$ is not $G$-left-invariant. However, in the special case that a positive cone for $N$ is invariant under conjugation by elements of the quotient $Q$, it is. The following lemma will be used to produce lexicographic left-orders on extensions where the kernel leads. We encourage the reader who needs a refresher on inner semidirect product to consult Chapter \ref{chap: semi-wreath} as needed. 

\begin{lem}\label{lem: semidirect} \label{lem: P-clos-ext-N-leads}
Let $G$ be an inner semidirect product of the subgroups $N \unlhd G$ and $Q\leqslant G$.
Let $P_N$ and $P_Q$ be positive cones of $N$ and $Q$ respectively, and assume that $qP_Nq^{-1} = P_N$ for all $q\in Q$.
Then, the lexicographic order on the underlying set $N\times Q$  with leading factor $N$ is $G$-left-invariant. In particular, it is a left-order on $G$.
\end{lem}
\begin{proof}
Since $G$ is an (inner) semidirect product of $N$ and $Q$, we have that $N\cap Q = \{1_G\}$ and $NQ=G$.
There is a natural bijection between $N\times Q$ to $G=NQ$, and under this bijection, the subset of elements of $N\times Q$ that are lexicographical greater to $(1_N,1_Q)$ corresponds to the subset $P\coloneqq P_N Q \cup P_Q$ of $G$.
Note that $P^{-1}= Q^{-1} P_N^{-1} \cup P_Q^{-1}$ and since $qP_Nq^{-1} = P_N$ for all $q\in Q$, we get that $P^{-1}= P_N^{-1} Q \cup P_Q^{-1}$. 
Therefore $G= P \sqcup P^{-1} \sqcup \{1_G\}$. 

To show that $P$ is a subsemigroup, let $nq, n'q'\in P$, with $n,n'\in N$ and $q,q'\in Q$. 
Recall that $n$ and $n'$ either are trivial or belong to $P_N$.
Then $nq n'q'= n (q n' q^{-1}) q q'$ and we see that if  $n \neq 1_G$ or $n' \neq 1_G$  then $ n (q n' q^{-1})\in P_N$ and $nq n'q'\in P_NQ\subseteq P$.
 If $n = 1_G = n'$, then $q,q'\in P_Q$ and $nq n'q'=qq'\in P_Q \subseteq P$.
\end{proof}

\begin{ex}[$\bZ$, $\bZ^2$]\label{ex: Zsq-lex-ord}
Since $\bZ$ is left-orderable, $\bZ^2$ is left-orderable by viewing the group as a $\bZ$-by-$\bZ$ extension, and this order can be made to coincide with that of Example \ref{ex: LO-K2}. Indeed, let  $\bZ^2 = A \times B = \langle a \rangle \times \langle b \rangle$. Then $1 \to A \xhookrightarrow{\iota} \bZ^2 \xrightarrowdbl[]{f} B \to 1$ is a short exact sequence. By abuse of notation, we identify $\iota(A)$ with $A$.

Every element of $\bZ^2$ can be written in the form $a^m b^n$. Suppose that the left-order on $A$ is $\prec_A$ such that $1_A \prec_A a$, and similarly define $\prec_B$ such that $1_B \prec b$. Then $a^m b^n \prec a^p b^q \iff m < p$ or $m = p$ and $n < q$. By varying the underlying left-orders for $A$ and $B$, we may obtain three more left-orders which are lexicographic with leading factor $A$. Moreover, since $A$ and $B$ are interchangeable as quotient and kernel, we may obtain four more lexicographic left-orders by having $B$ be the quotient. 
\end{ex}

\begin{defn}
	A group $G$ is \emph{poly}-X if there is a subnormal series 
$$\{1\} = G_0 \triangleleft G_1 \triangleleft \dots G_n = G$$
such that the quotients $G_i/G_{i-1}$ have property $X$.   
\end{defn}

\begin{ex}[Poly-$\bZ$ groups]\label{ex: LO-polyZ}
A group $G$ is poly-$\bZ$ if $$G_i/G_{i-1} \cong \bZ$$ for $i=1, \dots, n$. 

By Lemma \ref{lem: LO-clos-ext} and by induction on the length of the subnormal series, we see that poly-$\bZ$ groups are lexicographically left-orderable. We will discuss left-orderable virtually poly-$\bZ$ groups in Chapter \ref{chap: closure-extension}.  
\end{ex}

\begin{ex}[$\BS(1,q)$]\label{ex: LO-BS-ext}
	For $q\neq 0$ the solvable Baumslag-Solitar groups are defined by the presentation
$$\BS(1,q)= \langle a, b \,|\, aba^{-1} = b^{q}\rangle,$$
and for $q = 0$, $\BS(1, 0)\cong \mathbb{Z}$. It is well-known that 
$$BS(1,q)\cong \bZ \ltimes \bZ[1/q].$$
\end{ex}

We will go over the proof of this and provide more details in Chapter \ref{chap: closure-extension}. 

\begin{ex}[Semi-direct products and wreath products]
	Semi-direct products and wreath products of left-orderable groups inherit a left-order as extensions. We go into the details of their definitions in Chapter \ref{chap: semi-wreath} and construct an explicit left-order for the lamplighter group $\bZ \wr \bZ$ there. Then, in Chapter \ref{chap: closure-extension} we will go over the computational complexity of some left-orders for semi-direct products and wreath products. 
\end{ex}

With the closure properties in hand, let us explore some more involved examples of left-orderable groups. 

\section{Free groups}\label{sec: LO-Fn}
\subsection{Magnus embedding}

\begin{ex}[Free groups with Magnus embedding]\label{ex: LO-Magnus}
	A well-known method for ordering free groups is given by the \emph{Magnus embedding}. Given a (possibly infinitely generated) free group $F = \langle x_1, x_2 \dots, \rangle$, we define the group ring $\Lambda = \bZ[[X_1, X_2 \dots]]$ of formal power series of non-commuting variables $X_i$, for $1 \leq i \leq n$, and define the embedding as $\mu: F \to \Lambda$ as $$\mu(x_i) = 1 + X_i, \quad \mu(x_i \inv) = 1 - X_i + X_i^2 - X_i^3 + \dots$$
	that is, we send $x_i$ to the $1 + X_i$, and $x_i\inv$ to the Taylor expansion of $\frac{1}{1 + X_i}$, known in this case as the \emph{Magnus expansion}. Now, since every term in $\Lambda$ is given by positive words in $X_i$, we can define a well-order $\prec'$ on the words by ordering them in terms of degree, and then lexicographically. Then, for $g, h \in F$, we say that $g \prec h$ if and only if the first non-trivial term (with respect to $\prec'$) of $\mu(h) - \mu(g)$ has positive coefficient. For example, $x_i \succ 1$ because the first such term is $X_i$ with coefficient $1$ and $x_i\inv \prec 1$ because the first such term is $X_i$ with coefficient $-1$. 
	
	Note that this defines a bi-order. Indeed, let $g,h,f \in F$ and suppose that $g \prec h$. Write $\mu(g) = 1 + U, \mu(h) = 1 + V, \mu(f) = 1 + W$ for $U,V,W$ of degree greater or equal to $1$. Extend $\prec'$ naturally on $\Lambda$ viewed as as the direct sum $\Lambda = \oplus_{i=1}^\infty \bZ$, where each $\bZ$ is the coefficient space for each term. Then, 
	\begin{align*}
		&g \prec h \iff \mu(g) \prec' \mu(h) \iff (1+U) \prec' (1+V) \iff (V - U) \succ' 0 \\
		&fg \prec fh \iff \mu(f)\mu(g) \prec' \mu(f)\mu(h) \iff (1+W)(1+U) \prec' (1+W)(1+V) \\
		&gf \prec hf \iff \mu(g)\mu(f) \prec' \mu(h)\mu(f) \iff (1+U)(1+W) \prec' (1+V)(1+W)
	\end{align*}
	where 
	\begin{align*}
 		& (1+W)(1+U) = 1 + W + U + WU \\
 		& (1+U)(1+W) = 1 + U + W + UW \\
 		& (1+W)(1+V) = 1 + W + V + WV \\
 		& (1+V)(1+W) = 1 + V + W + VW
 	\end{align*}
	Since the degrees of the polynomials in $\Lambda$ are non-negative, 
	$$\deg(U) \leq \deg(WU), \deg(UW)$$
	and $$\deg(V) \leq \deg(WV), \deg(VW),$$ we have that $U \prec' WU$ and $W, V \prec' WV$ and so $$(1+W)(1+U) = 1 + W + U + WU \prec' 1 + W + V + WV \iff V - U \succ' 0,$$ so $$g \prec h \iff V-U \succ' 0 \iff fg \prec fh  \iff gf \prec hf.$$	
\end{ex}

Computing with the Magnus expansions can get complicated. %
We next look at a left-order that is more straightforward to compute defined by an \emph{ordering quasi-morphism}. 

\subsection{Ordering quasi-morphism}
\begin{ex}[$F_n$ with ordering quasi-morphism]\label{ex: LO-Fn-ordqm}
The following is due to \Sunik \cite{Sunic2013}. 

Let $F_n = *_{i=1} G_i$ with $G_i = \langle x_i \rangle$, with positive cones given by $P_i = \langle x_i \rangle^+$, and with the index ordering $<_I$ given by $G_i < G_j \iff i < j$. 
		
	\begin{itemize}
		\item Write every $g \in F_n$ in normal form as $g = g_1 \dots g_\ell$ where $g_i \in G_{i_j}$ and $g_{i}, g_{i+1}$ always lie in different factors, $G_{i_j} \not= G_{i_{j+1}}$ for $1 \leq i \leq \ell$. 
		\item We define a syllable as \emph{positive} (resp. \emph{negative}) if it is in one of the $P_i$'s (resp $P_i\inv)$ for $1 \leq i \leq n$.  
		\item We define an \emph{index jump} (resp. \emph{index drop}) as an occurence when $g_{i} g_{i+1}$ is such that $g_i \in G_{i_j}$, $g_{i+1} \in G_{i_{j+1}}$ such that $G_{i_j} <_I G_{i_{j+1}} $(resp $G_{i_j} >_I G_{i_{j+1}}$). 
	\end{itemize}
	
	Define the map $\tau: F_n \to \bZ$ by 
	\begin{align*}
		\tau(g )&:= \sharp(\text{positive syllables in $g$})- \sharp(\text{negative syllables in $g$})\\
		&+ \sharp(\text{index jumps in $g$})- \sharp(\text{index drops in $g$}).
	\end{align*}
	Such a map is called an \emph{ordering quasi-morphism}, and defines a positive cone $P := \{g \in F_n \mid \tau(g) > 0\}$. For example, let $F_2 = \langle a, b \rangle$ with positive cones $P_A = \langle a \rangle^+$ and $P_B = \langle b \rangle^+$. Then, the word $ab\inv$ has a positive syllable followed by a negative syllable as well as an index jump. Since $\tau(ab\inv) = 1$, $ab\inv \in P$. 
	
	We will prove this and discuss the computational complexity of positive cones defined by ordering quasi-morphisms in Chapter \ref{chap: cross-Z}. 
\end{ex}

\section{Right-angled Artin groups}
\begin{defn}[Artin groups and defining graphs]\label{def: LO-Artin-groups}
Let $\Gamma$ be a finite simplicial graph with edges labeled by integers $n \geq 2$. 
We associate to $\Gamma$ a group $A(\Gamma)$ whose presentation has generators corresponding to the vertices of $\Gamma$ and the relations are of the form
$$\underbrace{aba\dots}_{n \text{ letters }} = \underbrace{bab\dots}_{n \text{ letters }}$$
 where $\{a,b\}$ is an edge of $\Gamma$ labeled with $n$. 
 The graph $\Gamma$ is called the \emph{defining graph} of the \emph{Artin group} $A(\Gamma)$.
\end{defn}

\begin{defn}\label{def: LO-RAAG}
A \emph{right-angled Artin group} is an Artin group whose defining graph only has edges with label 2. We often abbreviate them as \emph{RAAGs}.
\end{defn}

\begin{figure}[h]{
\includegraphics{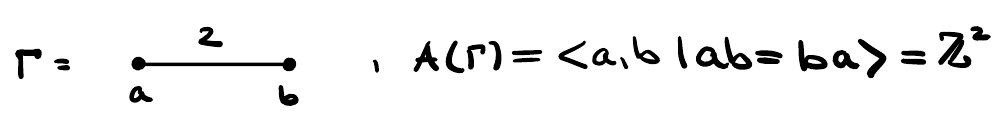}
}
\caption{An example of an Artin group that is also a RAAG. The graph $\Gamma$ is given by a straight line with endpoints $a,b$. This means that $A(\Gamma)$ has generators $a,b$ and relation $ab=ba$ of length two on both sides. This group is isomorphic to $\bZ^2$.}
\label{fig: LO-Artin}
\end{figure}

\begin{ex}[RAAGs]\label{ex: LO-RAAGs}
	One way to see the left-orderability of RAAGs is as follows. It was shown by Hermiller and \Sunik in 2007 that RAAGs with finite clique number $\text{clq}(\Gamma)$ and finite chromatic number $\text{chr}(\Gamma)$ are poly-free \cite{HermillerSunic2007}. More specifically, if $n$ is the length of the polyfree subnormal central series, then 
	$$\text{clq}(\Gamma) \leq n \leq \text{chr}(\Gamma).$$ 
	Hence, RAAGs are left-orderable by extension. 
\end{ex}

In Chapter \ref{chap: cross-Z}, we will give discuss the computational complexity of some left-orders for certain Artin groups and RAAGs. For an overview on Artin groups, we recommend \cite{McCammond2017}. 

\section{Surface groups}

The following proof and figures are taken from \cite[Theorem 3.11]{Clay2016} without much change.

\begin{ex}[Surface groups (except $\pi_1(P^2)$)]\label{ex: LO-surface-groups}
	We will restrict ourselves to proving that the fundamental group of closed surfaces is bi-orderable, as if the surface is not closed, its fundamental group is a free group which we have shown to be bi-orderable. We will reduce the proof for closed surfaces except for the sphere, tori, and Klein bottle to that of proving that $3P^2$, the connected sum of three projective planes, is left-orderable. 
	
	To start with, it is easy to see why the fundamental group of the projective plane is not left-orderable, as $\pi_1(P^2) = \bZ/2\bZ$, which is a torsion group. 
	
	Now, for the 2-sphere $S^2$, $\pi_1(S^2) = 1$. We have already looked at the fundamental group of the torus $T^2$ which is given by $\bZ^2$, and that of the Klein bottle $K^2 \isom P^2 \# P^2$,\sidenote{See these StackOverflow answers for an explanation as to why $K^2 \isom P^2 \# P^2$. \url{https://math.stackexchange.com/questions/1039819/connected-sum-of-projective-plane-cong-klein-bottle}} which is given by the Klein bottle group $K_2$ of Example \ref{ex: LO-K2}.  
	
	\begin{figure}[h]
	\centering
		{\includegraphics[width = 0.5\textwidth]{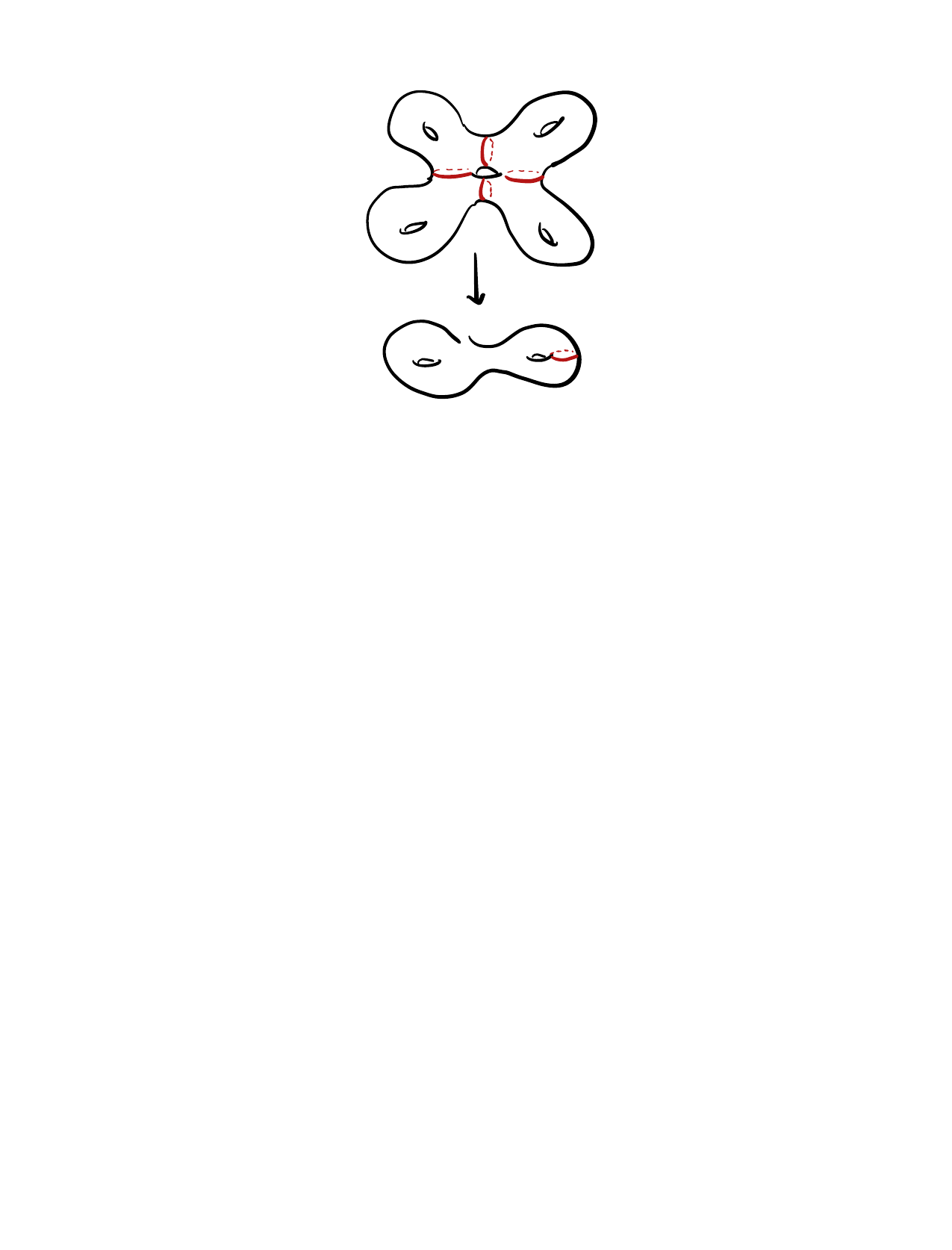}
		}
		\caption{A covering of $2T^2$ by $5T^2$. A deck transformation of order $4$ is given by rotating the four handles by an angle of $\pi/2$.}
		\label{fig: LO-surface-deck}
	\setfloatalignment{b}
	\end{figure}
	
	Let $kT^2$ denote the connected sum of $k$ tori. Then, we can use the fact that for $k \geq 2$, there is a covering map from $kT^2 \to 2T^2$ of order $k-1$ given by rotating the handles, and therefore $\pi(kT^2) \leq \pi_1(2T^2)$ (see Figure \ref{fig: LO-surface-deck}). 
	
	Let $kP^2$ denote the connected sum of $k$ projective planes. We use the fact that there is an oriented double cover $(k-1)T^2 \to kP^2$ for $k \geq 0$ (where $0T^2 \isom S^2$), such that $\pi_1((k-1)P^2) \leq \pi_1(kP^2)$. 
	
	Finally, we use the fact that for $k \geq 3$, there is a covering map from $kP^3 \to 3P^2$ such that $\pi_1(kP^3) \leq \pi_1(3P^2)$ to reduce the bi-orderability of every surface group to that of $\pi_1(3P^2)$. 
	
	\begin{figure}[h]
	\centering
		{
		\includegraphics[width = 0.75\textwidth]{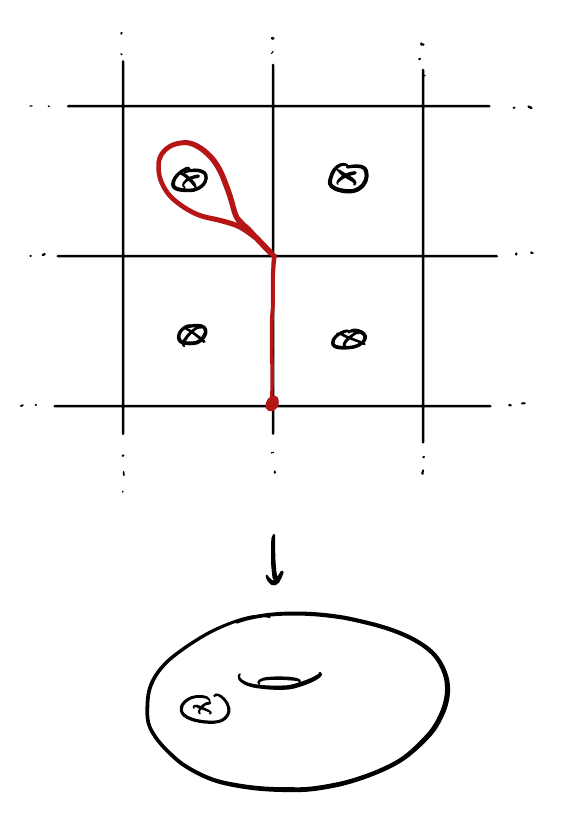}
		}
		\caption{The lift of $T^2 \# P^2$ into the universal cover $\widetilde{T^2 \# P^2}$, which has as fundamental group $F_\infty$, the free group on countably many generators. Each generator is given by a loop around a cross-cap (in red).}
		\label{fig: LO-3P2-cover}
	\end{figure}
	
	To show bi-orderability of $3P^2$, we use the fact that it $3P^2 \cong T^2 \# P^2$, that is, a torus with a cross-cap.\sidenote{See this REU paper by Justin Huang for an excellent diagram. \url{https://www.math.uchicago.edu/~may/VIGRE/VIGRE2008/REUPapers/Huang.pdf}} We lift $T^2 \# P^2$ to its universal cover to obtain the following sequence 
	
	$$1 \to \pi_1(\widetilde{T^2 \# P^2}) \to \pi_1(T^2 \# P^2) \to \bZ^2 \to 1.$$
	
	Now, $\widetilde{T^2 \# P^2}$ is a grid of cross-caps (see Figure \ref{fig: LO-3P2-cover}), so its fundamental group is the free group on a countable number of generators, $\pi_1(\widetilde{T^2 \# P^2}) \cong F_\infty$ which is bi-orderable under the Magnus embedding (see Example \ref{ex: LO-Magnus}), and $\bZ^2$ is obviously bi-orderable. Finally, $\pi_1(T^2 \# P^2)$ acts on its covering space $\widetilde{T^2 \# P^2}$ by covering translation, sending each generator of $\pi_1(\widetilde{T^2 \# P^2})$ to a conjugate of itself. Now, the positive cone of $\pi_1(\widetilde{T^2 \# P^2}) \cong F_\infty$ is preserved under conjugation by $F_\infty$ because it is a bi-order (Lemma \ref{lem: P-bi-ord}). 
		
	By Corollary \ref{cor: P-ext-bi-ord}, this means that $\pi_1(T^2 \# P^2)$ is bi-orderable. 
\end{ex}

In Chapter \ref{chap: closure-finite-index}, we will prove a result pertaining to the computational complexity of left-orders for acylindrically hyperbolic groups, which are a generalisation of hyperbolic groups. 

\section{Dynamics on the real line}
A rich class of examples of left-orderable groups are groups acting by orientation-preserving homeomorphisms on $\bR$. For countable left-orderable groups, such an action is always possible to construct.

\begin{prop}\label{prop: LO-dynamical realisation}
	For a left-orderable group $(G, \prec)$, and an order-preserving embedding $t: (G, \prec) \to (\bR,<)$, we may construct an embedding $\rho: G \to \Homeo^+(\bR)$ from $G$ to the set of orientation-preserving homomorphisms on $(\bR,<)$ of $(G, \prec)$ with respect to $t$ as follows. 
	
	Since $G$ is countable, write $G = \{1_g = g_0, g_1, \dots, g_i, \dots \}$, and define $G_i = \{g_0, \dots, g_i\}$. Then, we can define $t: G \to \bR$ inductively as $t(g_0) := 0$, 
	$$t(g_{i+1}) := \begin{cases}
		\max_<(t(G_i)) + 1 & g_{i+1} \succ \max_\prec(G_i) \\
		\frac{t(g_m) + t(g_n)}{2} &  g_m := \max_\prec \{g \in G_i \mid g \prec g_{i+1}\}, \quad g_n := \min_\prec \{g \in G_i \mid g \succ g_{i+1}\}\\
		\min_<(t(G_i)) - 1 & g_{i+1} \prec \min_\prec(G_i)
	\end{cases}$$
	
	\begin{enumerate}
		\item If $x \in t(G)$, then $x = t(h)$ for some $h \in G$, and set $$\rho(g)(t(h)) = t(gh).$$ 
		\item If $x \in \overline{t(G)}$, the closure of $t(G)$, then define $\rho(g)$ such that it is continuous of $\overline{t(G)}$, i.e. as the limit of converging sequences. 
		\item If $x \in (t(g),t(h))$, then there exists a gap $h,k$ such that $x \in (t(h), t(k))$ and we can write $$x = \alpha t(h) + (1-\alpha)t(k)$$ for some $\alpha \in (0,1)$ and define $$\rho(g)(x) = \alpha t(gh) + (1-\alpha) t(gk).$$
		\item If $x < \inf(t(G))$ or $x > \sup(t(G))$, then define $$\rho(g)(x) = t(gx) = x.$$ 
	\end{enumerate}
\end{prop}
\begin{proof}
	To see that $\rho$ is injective, observe that since $\rho(g)(x) = t(gx)$, for all $g \in G$, $x \in \bR$, $\rho(g)$ is uniquely defined by its action on $t(G)$, which is in one-to-one correspondence with its action on $G$ since by order-preserving property, $t(g) \not= t(h)$ if $g \not= h$. %

	To see that $\rho(g) \in \Homeo^+(\bR)$ for every $g \in G$, start by observing that if $t(f) < t(h)$, then $\rho(g)(t(f)) = t(gf) < \rho(g)(t(h)) = t(gh)$ because $t$ was defined to be order-preserving, and hence, $t(f) < t(h) \iff f \prec h \iff gf \prec gh \iff t(gf) < t(gh)$. Since the action on $\rho(g)$ is extended on $\bR$ via continuity on $\overline{t(G)}$ and affine mapping, it follows that $\rho(g)$ is orientation-preserving on $\bR$. 
\end{proof}

\begin{defn}
	The map $\rho$ as defined in Proposition \ref{prop: LO-dynamical realisation} is called a \emph{dynamical realisation} of $(G, \prec)$. 
\end{defn}

Note that we do not refer to the order-preserving embedding $t$ in our definition of dynamical realisation. This is because any two dynamical realisations are conjugate.

\begin{lem}\label{lem: LO-dyn-top-conjugate}
	Any two dynamical realisations of $(G, \prec)$ are topologically conjugate. That is, given two embeddings $t,t': (G, \prec) \to \bR$ as in Proposition \ref{prop: LO-dynamical realisation} and corresponding dynamical realisations $\rho, \rho': G \to \Homeo^+(\bR)$, there exists a homeomorphism $\varphi: \bR \to \bR$ such that 
	$$\rho'(g)(x) = \varphi \circ \rho(g) \circ \varphi\inv (x)$$
	for all $g \in G, x \in \bR$. 
\end{lem}

\begin{proof}
	Define $\varphi: t(G) \to t'(G)$, 
	$$\varphi(t(g)) = t'(g).$$
	Then, for any $h \in G$, $$[\rho'(g) \circ \varphi](t(h)) = \rho'(g)(t'(h)) = t'(gh) = \varphi(t(gh)) = [\varphi \circ \rho(g)](t(h)).$$
	Since $\varphi$ is clearly bijective on $t(G)$, we can invert $\varphi$ to obtain $$\rho'(g)(t(h)) = \varphi \circ \rho(g) \circ \varphi\inv (t(h)).$$
	
	To extend $\varphi$ into $\bR$, observe that for any gap $(t(g),t(h))$, that is, any $g \prec h$ such that there is no $f \in G$ with $g \prec f \prec h$, we can pick an orientation-preserving homeomorphism $\psi$ mapping to $(t'(g), t'(h))$ and define $\varphi := \psi$ on that interval. 
\end{proof}

Moreover, subgroups of $\Homeo^+(\bR)$ are always left-orderable. 
\begin{lem}\label{lem: LO-homeo}
	Let $G \leq \Homeo^+(\bR)$. Then, $G$ is left-orderable. 
\end{lem}
\begin{proof}
	Since $G \leq \Homeo^+(\bR)$, we can order group elements by their action on a countable dense set $X = \{x_1, \dots, x_n, \dots \}$ (for example $X = \bQ$). Indeed, we can define $g \prec h \iff g(x_i) < h(x_i)$ where $i$ is the minimal index such that $g(x_i) \not= h(x_i)$. 
	
	We first show that the order is left-invariant. If $g \prec h$, then $g(x_i) < h(x_i)$ for some $i \in \bN$ and $g(x_j) = h(x_j)$ for all $j < i$. Therefore, if $f \in G$, we have that $f(g(x_j)) = fh(x_j)$ for all $j < i$, and $f(g(x_i)) < f(h(x_i))$ by orientation-preserving property, so $i$ is the minimal index such that $fg(x_i) \not= fh(x_i)$ as well for $fg, fh$. Therefore, $g \prec h \implies fg \prec fh$ for any $f \in G$. To obtain the reverse implication, we use $fg \prec fh \implies g = f\inv fg \prec f\inv fh = h$. 
			
	We next show that $\prec$ is a strict total order. It is clear that $g \not\prec g$ and $g \prec h \implies h \not\prec g$ and thus antireflexivity and asymmetry are satisfied. To show transitivity, suppose that $g \prec h$ and $h \prec f$, and let $i,j$ be the minimal indices such that $g(x_i) \not= h(x_i)$, $h(x_j) \not= f(x_j)$. That is, $g(x_\alpha) = h(x_\alpha)$ for $\alpha < i$ and $h(x_\alpha) = f(x_\alpha)$ for $\alpha < j$. 
	
	If $i = j$, then $g(x_i) < h(x_i) < f(x_i)$, so $g \prec f$. 
	
	Suppose that $i < j$. Then, $g(x_\alpha) = h(x_\alpha) = f(x_\alpha)$ for $\alpha < i$, and $g(x_i) < h(x_i) = f(x_i)$, showing that $i$ is the minimal index such that $g(x_i) \not= f(x_i)$, and thus $g \prec f$. 
	
	Suppose that $i > j$. Then, $g(x_\alpha) = h(x_\alpha) = f(x_\alpha)$ for $\alpha < j$ and $g(x_j) = h(x_j) < f(x_j)$ showing that $j$ is the minimal index such that $g(x_j) \not= f(x_j)$ and thus $g \prec f$. 
	
	Finally, to show that $\prec$ is total, suppose that for $g \not= h$, there does not exist a minimal index such that $g(x_i) \not= h(x_i)$. Then, $g(x) = h(x)$ for all $x \in X$. Then, it must be that there is some $r \in \bR$ such that $g(r) \not= h(r)$. However, since both $g,h$ are assumed to be continuous, we have that there exists a converging sequence $\lim_{n \to \infty} x_n = r$ with $x_n \in X$ such that $g(r) = \lim_{n \to \infty} g(x_n) = \lim_{n \to \infty} h(x_n) = h(r)$, contradicting our assumption. Thus, there must be a minimal index $i$ such that $g(x_i) < h(x_i)$ or $g(x_i) > h(x_i)$, and hence either $g \prec h$ or $g \succ h$. 
\end{proof}

As a corollary, we get the following statemnet. 

\begin{thm}\label{thm: LO-dynamic-realization}
Let $G$ be a countable group. Then $G$ is left-orderable if and only if $G$ is isomorphic to a subgroup of $\Homeo^+(\bR)$, the group of orientation-preserving homeomorphisms on $\bR$. 
\end{thm}
\begin{proof}
	The ``if'' direction is given by Proposition \ref{prop: LO-dynamical realisation} whereas the ``only if'' direction is given by Lemma \ref{lem: LO-homeo}. 
\end{proof}

\begin{ex}[$\BS(1,q), \quad q \geq 1$]\label{ex: LO-dyn-BS}
	The solvable Baumslag-Solitar groups given by presentation 
	$$\BS(1,q) = \langle a, b \,|\, aba^{-1} = b^{q}\rangle,$$
	
	have an embedding map $\rho: \BS(1,a) \to \Homeo^+(\bR)$ given by $$\rho(a)(x) = qx, \quad \rho(b)(x) = x + 1$$
	for $q > 0$. (If $q < 0$, then notice that this map is no longer orientation-preserving. Also note that $\rho$ is not necessarily a dynamical realisation.) 
		
Since $\BS(1,a)$ is countable and $\rho$ maps every group element to an affine action on $\bR$, it is left-orderable by Lemma \ref{lem:  LO-homeo}. This gives a different left-order than the one obtained by extension in Example \ref{ex: LO-BS-ext}. In Chapter \ref{chap: closure-extension} we will see that the two types of left-orders have different formal language complexities.
\end{ex}

For more details on the dynamical approach to left-orders, we recommend \cite{Navas2010}.  

As we have seen familiarised ourself with some examples of left-orders, we remark that given a group $G$, the set of left-orders can be quite small (such as in $\bZ, K_2$) or uncountably large (such as in $\bZ^2$). One way to keep track of the left-orders is via the space of left-orders and its topology. 

\section{The space of left-orders}
As introduced by Sikora \cite{Sikora2016}, the space of left-orders for a group $G$, denoted $\LO(G)$, is the set of all possible left-orders on $G$ equipped with an inherited product topology $\tau$ which we will define below. We review most topological terms used as sidenotes in case it is of help to the reader. 

Recall from Section \ref{sec: pos-cone} that left-orders are in bijection with positive cones, which are subsets of $G - \{1_G\}$. We can think of the set of all left-orders $\LO(G)$ as the set of all positive cones, which is a subset of the power set of $G - \{1_G\}$, denoted $\cP({G - \{1_G\}})$. 

On one hand, we can view $\cP({G - \{1_G\}})$ as the set of all functions from ${G - \{1_G\}}$ to a set of two elements, say $\{0, 1\}$. We identify a subset $S \in \cP({G - \{1_G\}})$ with its corresponding indicator function $\chi_S: (G - \{1_G\}) \to \{0,1\}$ where
$$\chi_S(g) = \begin{cases}
	1 & g \in S \\
	0 & g \not\in S.
\end{cases}$$

On the other hand, we can also view $\cP({G - \{1_G\}})$ as the product space $$X = \prod_{g \in G - \{1_G\}} \{0, 1\},$$ where the coordinates are indexed by $g \in G - \{1_G\}$. Each subset $S \subseteq G - \{1_G\}$ is a point in this space, with $g$th coordinate $\chi_S(g)$. In other words, $\prod_{G - \{1_G\}} \chi_S(g)$ is the point corresponding to $S$ in the product space $X$. We make this association implicit and may write $S \in X$.

The set $\{0,1\}$ has the discrete topology\sidenote{In the discrete topology, every singleton is an open set.} and the product has the resulting product topology\sidenote{If $X = \prod_{i \in I} X_i$, and $\pi_i$ are canonical projections maps from $X \to X_i$, then the product topology is the coarsest topology for which the maps $\pi_i$ are continuous. In other words, the pre-images under $\pi_i$ of open sets in $X_i$ must be open in $X$.}. Note that by Tychonoff's theorem, $X$ is compact.\sidenote{The theorem says that the product space of any family of compact spaces is compact under the product topology. Finite discrete spaces are compact, since any open cover will have a finite subcover.}

The open sets in $X$ are of the form $$U = \prod_{g \in G - \{1_G\}} U_g$$ where $U_g \not= \{0,1\}$ for finitely many $g$. 
Viewed as a topological space, the topology on $\LO(G) \subset \cP(G - \{1_G\})$ is the subspace topology\sidenote{The open sets in the subspace topology are the open sets in the original topology intersected with the subspace $Y$. If $\tau$ is the original topology, then the subspace topology $\tau' = \{U \cap Y \mid U \in \tau\}$.} on $X$. Let $$V = U \cap \LO(G)$$ be an open set in this subspace topology, where $U = \prod_{g \in G - \{1_G\}} U_g$ is an open set in $X$. Suppose that $P \in V$. Then $U$ must be non-empty, meaning that each factor $U_g \cap \LO(G)$ is non empty. If $U_g = \{1\}$, then $\chi_P(g) = 1$, meaning that $g \in P$. If $U_g = \{0\}$, then  $\chi_P(g) = 0$ meaning that $g \notin P$, and that $g\inv \in P$ since $g \not= 1_G$. If $U_g = \{0,1\}$, there is no restriction on $P$ with respect to containing $g$.

\subsection{The space of left-orders versus the product space}

Let us give an example of a space of left-orders which is much ``smaller'' than the product space where it comes from. 

\begin{ex}[$\bZ$]\label{ex: Z-LO}
We saw in Example \ref{ex: LO-Z-opp-ord} that $\bZ$ has exactly two orders $\prec$ and $<$, uniquely defined by $0 \prec 1$ and $0 > 1$. Therefore, while the power set $\cP(\bZ - \{0\})$ and the associated product space $$\prod_{g \in \bZ - \{0\}} \{0,1\}_g$$ is infinite, $\LO(\bZ)$ has actually only two points corresponding to $\prec$ and  $<$.
\end{ex}

The following two observations capture some of the topological implications of restricting $X$ to $\LO(G)$, where every point must respect the properties of positive cones.

First, we observe that there exists non-empty open sets in $X$ which do not result in non-empty open sets in $\LO(G)$. Let us give an example.

\begin{lem}\label{lem: U-non-empty-V-empty}
Fix $g \in G$, and let $$U' = \{S \in X \mid g, g\inv \in S\}$$ and $$U'' = \{S \in X \mid g, g\inv \not\in S\}.$$ Then $U'$ and $U''$ are open sets such that $U' \cap \LO(G)$ and $U'' \cap \LO(G)$ are empty.
\end{lem}
\begin{proof}
First, by definition of $U'$ we have that for any $h \in G$, $U'_h = U'_{h\inv} = \{1\}$ for $h=g$ and $U'_h = \{0,1\}$ otherwise, so it is open. If we assume that there exists a positive cone $P \in U' \cap \LO(G)$, then both $g, g\inv \in P$, a contradiction. Similarly, if $P \in U''$, then $P$ neither contains $g$ nor $g\inv$, a contradiction. 
\end{proof}

Second, we observe that for a non-trivial open set $V \in \LO(G)$, there are many ``underlying open sets'' $U$ such that $U \cap \LO(G) = V$. Hence, the concept of an underlying open set is not well-defined.

\begin{obs}
\label{lem: U-non-unique}
\label{obs: U-non-unique}
Let $V \in \LO(G)$ be a non-trivial\sidenote{Not equal to the whole space or the empty set.} such that $V = U \cap \LO(G)$ for some open set $U \subseteq X$. Then, there exists an open set $W \not= U$ such that $V = W \cap \LO(G)$.  
\end{obs}
\begin{proof}
Since $V \not= \LO(G)$, $U\not= X$. Therefore, there is some $g \not= 1_G$ such that $U_g \not= \{0,1\}$. If $U_g = \{0\}$, take $U' := \{S \in X \mid g, g\inv \in S\}$, and if $U_g = \{1\}$, take $U' = \{S \in X \mid g, g\inv \not\in S\}$, and $W = U \cup U'$. By Lemma \ref{lem: U-non-empty-V-empty}, $W$ is an open set. Using the same lemma, $W \cap \LO(G) = (U \cup U') \cap \LO(G) = (U \cap \LO(G)) \cup (U' \cap \LO(G)) = V \cup \emptyset = V$.
\end{proof}

However, we can refine Observation \ref{obs: U-non-unique} into the observation that for any an open set in $V \in \tau$ we can extract a (not necessarily unique) antisymmetric\sidenote{In the sense that $F \cap F\inv = \emptyset$.} finite set $F$ of elements which every positive cone in $V$ must contain. This observation will be useful later. 

\begin{lem}\label{lem: U-by-pos}
For a non-trivial open set $V$ of $\LO(G)$, there is a finite antisymmetric set $F$ such that $V = \{P \in \LO(G) \mid f \in P, \quad \forall f \in F\}$.
\end{lem}
\begin{proof}
Let $U = \prod_{g \in G - \{1_G\}} U_g$ be an open set in $X$ such that $V = U \cap \LO(G)$. Let $S := \{g \in G - \{1_g\} \mid U_g \not= \{0,1\} \}$, $S^+ := \{g \in S \mid U_g = 1\}$, and $S^- := \{g \in S \mid U_g = 0\}$. Let $F = S^+ \cup \{g\inv \mid g \in S^-\}$. Finally, let $U' = \prod_{g \in G - \{1_G\}} U'_g$ such that $U'_g = \{1\}$ for $g \in F$ and $U'_g = \{0,1\}$ otherwise.

Let $V' = U' \cap \LO(G)$. To show that $V \subseteq V'$, let $P \in V$. Then $\chi_P(g) = 1$ for $g \in S^+$ and $\chi_P(g) = 0$ for $g \in S^-$. Therefore, $\chi_P(g\inv) = 1$ for $g \in S^-$. This gives us that $\chi_P(g) = 1$ for $g \in F$. Therefore, $P \in V'$. 

Similarly, to show that $V' \subseteq V$, let $P \in V'$. Then $\chi_P(g) = 1$ for $g \in F$. In particular, $\chi_P(g\inv) = 0$ for $g \in S^-$, and $\chi_P(g)$ for $g \in S^+$, therefore $P \in V$. We have shown that $V' = V$. 

Now suppose that $F$ is not antisymmetric and that there is $s$ such that both $s$ and $s\inv$ are in $F$. Then since $V$ is non-empty, there exists $P \in V$ such that $P$ contain both $s$ and $s\inv$, a contradiction. 

\end{proof}

\subsection{Topological properties of the space of left-orders}

At the end of this section, we will conclude that the non-existence of isolated points in $\LO(G)$ completely defines $\LO(G)$ topologically when $G$ is countable. One source of isolated points are finitely generated positive cones. 

\begin{lem}\label{lem: fg-cone-isolated}
If $G$ is a left-orderable group which has a finitely generated positive cone $P$, then $P$ is an isolated point\sidenote{An isolated point is a point for which there exists an open neighbourhood which only contains that point.} in $\LO(G)$.
\end{lem}
\begin{proof}
Suppose that $P$ is finitely generated, and let $P = \langle S \rangle^+$ be the semigroup generated by $S$, where $S = \{g_1, \dots, g_n\}$. Let $V = \prod_{g \in G - \{1_G\}} U_g \cap \LO(G)$ be an open set such that $U_{g} = \{1\}$ if $g \in S$, and $U_g = \{0,1\}$ otherwise. The open set $V$ is non-empty since it contains $P$. Suppose that $Q \in V$. Then $Q$ is a semigroup containing $\{g_1, \dots, g_n\}$, and therefore containing $\langle g_1, \dots, g_n \rangle^+ = P$. Therefore, $Q \supseteq P$. Since $Q$ is also a positive cone, we have by maximality (Lemma \ref{lem: lo-max-subset}) that $Q = P$. 

\end{proof}

\begin{ex}[$\bZ$, $K_2$]\label{ex: Z-fg-pos-cone}\label{ex: LO-K2-fg}
The orders of $\bZ$ $\prec$ and ${\prec'}$, defined by $0 \prec 1$, and $0 > 1$ of Example \ref{ex: LO-Z-opp-ord} are isolated orders. Indeed, they correspond to finitely generated positive cones $P_\prec = \langle 1 \rangle^+$ and $P_{\prec'} = \langle -1 \rangle^+$ respectively.

A similar argument can be made for the isolated positive cones of $K_2$ in Example \ref{ex: LO-P-K2}. 
\end{ex}

The following is a well-known fact about the properties of $\LO(G)$.

\begin{prop}\label{prop: LO-top}
For a left-orderable group $G$, $\LO(G)$ is closed in $X$, compact, and totally disconnected.
\end{prop}

Let us prove each property one at a time. We will have Lemma \ref{lem: LO-closed}, \ref{lem: LO-compact}, and \ref{lem: LO-tot-disc} prove Proposition \ref{prop: LO-top}. 

\begin{lem}\label{lem: LO-closed}\index{space of left-orders!closed}
For a left-orderable group $G$, $\LO(G)$ is closed\sidenote{We will use the definition that a set is closed if it contains its limit points.} in $X$.
\end{lem}

\begin{proof}
Recall that in the product topology $X$, a sequence $P_n \in X$ converges to $P \in X$ if and only if its associated sequence of indicator functions $\chi_{P_n}$ converges pointwise to $\chi_P$. Assume that $P_n$ is a sequence of positive cones in $\LO(G)$ which converges to $P$. We want to show that $P$ is a positive cone, hence $P \in \LO(G)$. 

To show that $P$ is a semigroup, suppose that $g,h \in P$. Then $\chi_P(g) = \chi_P(h) = 1$, so there exists $N \in \bN$ such that if $n \geq N$, $\chi_{P_n}(g) = \chi_{P_n}(h) = 1$. Moreover, $\chi_{P_n}(gh) = 1$ for all such $n \geq N$ since $P_n$ are positive cones. Therefore, $\lim_{n \to \infty} \chi_{P_n}(gh) = 1 = \chi_P(gh)$, so $gh \in P$. 

Similarly, to show that $P$ respects trichotomy, assume that $g \in P$ for some $g \not= 1_G$. Then for some $N \in \bN$, $\chi_{P_n}(g) = 1$ if $n \geq N$. Since $P_n$ is a positive cone, $\chi_{P_n}(g\inv) = 0$ for $n \geq N$, so $\lim_{n \to \infty} \chi_{P_n}(g\inv) = \chi_P(g\inv) = 0$, hence $g\inv \not \in P$, but $g \in P\inv$ since $g \in P$. If $g = 1_G$, then $1_G$ is not a coordinate in $X$, and hence $1_G \not \in P$. Hence $G = P \sqcup P\inv \sqcup \{1_G\}$.
\end{proof}

\begin{lem}\label{lem: LO-compact}
For a left-orderable group $G$, $\LO(G)$ is compact\sidenote{We will use the definition that very open cover has a finite subcover.}.
\end{lem}
\begin{proof}
Notice that for $\{0, 1\}$ with the discrete topology, the set of open sets is finite, so it is trivially true that every open cover has a finite subcover. Therefore, $\{0, 1\}$ is compact. By Tychonoff's theorem\sidenote{Any collection of compact topological spaces is compact with respect to the product topology.} the product space $X = \prod_{g \in G - \{1_G\}} \{0, 1\}_g$ is compact. Recall that every closed subspace of a compact space is also compact. The subspace $\LO(G) \subset X$ is closed by Lemma \ref{lem: LO-closed}, and therefore compact.
\end{proof}

\begin{lem}\label{lem: LO-tot-disc}
For a left-orderable group $G$, $\LO(G)$ is totally disconnected\sidenote{A space $X$ is totally disconnected if for every pair of distinct point $P, P'$ there exists a pair of disjoint open sets $U,V$ such that $X = U \sqcup V$ and $P \in U$, $P' \in V$.}.
\end{lem}
\begin{proof}
Let us first show that $X$ is totally disconnected. If two points $P, P' \in X$ are not equal, then there must be some $g \in G$ such that $g \in P, g \not\in P'$. Let $\pi_g: X \to \{0,1\}_g$ be the projection map onto the $gth$ coordinate. Then by definition of the product topology, the preimages $\pi_g\inv(1)$, $\pi_g\inv(0)$ are disjoint open sets. Moreover, $\pi_g\inv(1) \sqcup \pi_g\inv(0) = X$ and $P \in \pi_g\inv(1)$, $P' \in \pi_g\inv(0)$. Therefore, $X$ satisfies the requirement of being totally disconnected. Since subspaces of totally disconnected spaces are totally disconnected, $\LO(G)$ is totally disconnected.
\end{proof}

Let us now discuss countable left-orderable groups, whose space of left-orders have additional desirable properties.

\begin{prop}
For a countable left-orderable group $G$, $\LO(G)$ is metrizable.
\end{prop}\label{prop: LO-met}\label{prop: LO-top-countable-metrisable}

We will prove this proposition with Lemma \ref{lem: LO-metric} and Lemma \ref{lem: LO-met-top}. Since $G$ is countable, we can list its elements of $G-\{1_G\}$ as $G - \{1_G\} = \{g_0, \dots, g_i, \dots \}$. We define a metric between two positive cones $P$ and $Q$ as 
$$d(P,Q) = 2^{-n}, \qquad n := \min\{i \mid \chi_P(g_i) \not= \chi_Q(g_i)\}.$$

\begin{lem} The function $d: \cP(G - \{1_G\}) \to [0,1]$ is a metric. 
\end{lem}\label{lem: LO-metric}
\begin{proof}
First, $d(P,P) = 2^{-n}$ where $n := \min\{i \mid \chi_P(g_i) \not= \chi_P(g_i)\}$. Since the minimum of an empty set is $\infty$, we have $d(P,P) = \lim_{n \to \infty} 2^{-n} = 0$. Second, it is clear that $d(P,Q) = d(Q,P)$ by the definition of $n$. Third, we need to show $d$ satisfies the triangle inequality. Let $d(P,Q) = 2^{-n}$ and $d(Q,R) = 2^{-m}$. We want to bound $d(P,R)$. If $G_i := \{g_0, \dots, g_i\}$, then $\chi_P(g) = \chi_Q(g)$ for $g \in G_{n-1}$ and $\chi_Q(g) = \chi_R(g)$ for $g \in G_{m-1}$. Therefore, $\chi_P(g) = \chi_Q(g)$ on $G_{n-1} \cap G_{m-1}$, meaning that $d(P, R) \leq \max\{2^{-n}, 2^{-m}\} \leq d(P,Q) + d(Q,R)$.
\end{proof}

Recall that $\tau$ is the topology on $\LO(G)$ as defined earlier in the chapter.
\begin{lem} Let $(\LO(G), d)$ be a metric space where $\LO(G)$ is viewed as a set. The topology $\tau'$ induced by the metric $d$ on $\LO(G)$ coincides with $\tau$.
\end{lem}\label{lem: LO-met-top}\label{lem: LO-top-met}
\begin{proof}
Recall that the basis\sidenote{A family $\mathcal{B}$ of open subsets such that every open set is equal to a union of some subfamily of $\mathcal{B}$.} in $\tau'$ is the family of balls $B_P(2^{-n})$ of radius $2^{-n}$ centered at some positive cone $P \in \LO(G)$. In other words, $B_P(2^{-n}) = \{Q \in \LO(G) \mid \chi_Q(g) = \chi_P(g), \quad g \in G_{n-1}\}$.

Let us first show that $\tau' \subseteq \tau$. Let $V \in \tau$. Then $V = \prod_{g \in G - \{1_G\}} U_g \cap \LO(G)$ where $U_g \not= \{0,1\}$ for a finite collection $g \in S_U$. Denote $S_U := \{g_{k_1}, \dots, g_{k_m}\}$. Let $m := \max\{k_1, \dots, k_m\}$. Then if $Q \in B_P(2^{-m})$, $Q \supset S_U$. Therefore, $B_P(2^{-m}) \subseteq U$, implying $\tau' \subseteq \tau$.

To show that $\tau \subseteq \tau'$, fix a basis element of $B_P(2^{-n}) \in \tau'$. Let $S = P \cap G_{n-1}$, and let $U \in X$ be such that $\pi_g(U) = \{1\}, \quad g \in S$ and $\pi_g(U) = \{0,1\}, \quad g \not \in S$. Let $V = U \cap \LO(G)$. Then $V \in \tau$ and for every $Q \in V$, $Q \in B_P(2^{-n})$ so $V \subseteq B_P(2^{-n})$ and $\tau \subseteq \tau'$. We have shown $\tau = \tau'$.
\end{proof}

\begin{prop}
For a countable left-orderable group $G$, if $\LO(G)$ has no isolated points, then $\LO(G)$ is homeomorphic to the Cantor set.
\end{prop}

\begin{proof}
Using Brouwer's theorem \cite{Brouwer1910}, a topological space is homeomorphic to the Cantor set if and only if it is non-empty, perfect\sidenote{A topological space is perfect if it is closed and has no isolated points.}, compact, totally disconnected, and metrizable. Since $G$ is left-orderable, $\LO(G)$ has at least one point and thus is non-empty. In Proposition \ref{prop: LO-top}, we have shown that $\LO(G)$ is closed, compact and totally disconnected. By assumption $\LO(G)$ has no isolated point, so it is perfect. In Proposition \ref{prop: LO-met} we have shown that $\LO(G)$ is metrizable if $G$ is countable. Thus $\LO(G)$ satisfies the requirements for being homeomorphic to the Cantor set.
\end{proof}

\begin{figure}[h]{
\includegraphics{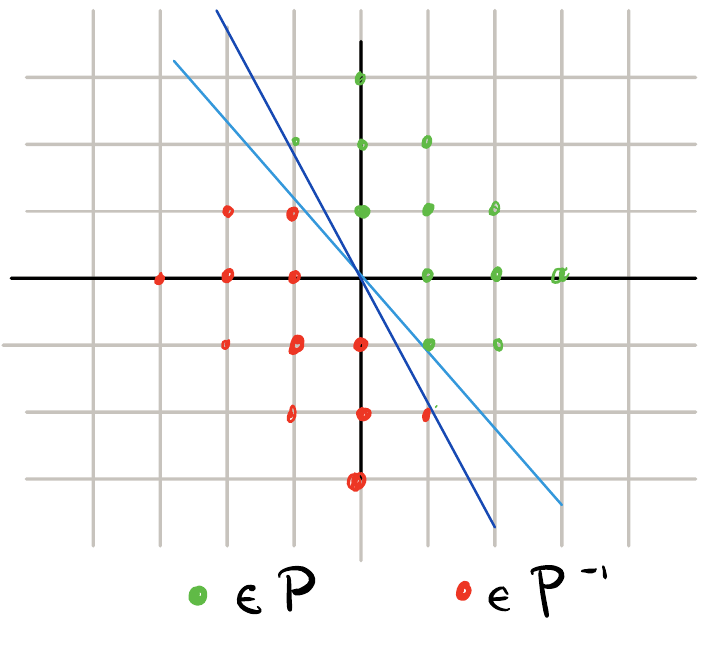}
}
\caption{A finite antisymmetric set $F \subset \bZ^2$ (in green) determining a neighbourhood of $\LO(\bZ^2)$ containing $P_{\lambda}$. The points in $-F$ are marked in red for clarity. The slopes $\lambda, \lambda'$ (in blue) partition $\bR^2$ such that one side contains $F$ and the other $-F$. Thus, $F \subset P_{\lambda}, P_{\lambda'}$.}
\label{fig: P-Zsq-nbhd}
\end{figure}

\begin{ex}[$\bZ^2$]\label{ex: LO-Zsq-no-isolated-point}
The space $\LO(\bZ^2)$ has no isolated points. recall the positive cones $$P_{\lambda, \diamond_1, \diamond_2} := \{(x,y) \in \bZ^2 \mid y \diamond_1 \lambda x \text{ or } y = \lambda x, x \diamond_2 0 \}$$ of Example \ref{ex: LO-P-Zsq} parametrised by $\lambda \in \bR, \diamond_1, \diamond_2 \in \{<,>\}$. Let $\lambda, \diamond_1, \diamond_2$, be fixed, and write $P_\lambda := P_{\lambda, \diamond_1, \diamond_2}$ for short. 

Let $V$ be a non-empty open neighborhood of $\LO(\bZ^2)$ containing $P_\lambda$. Let $F = \{(x_i, y_i) \in \bZ^2 \mid 1 \leq i \leq n\}$ be a finite set determining $V$ as in Lemma \ref{lem: U-by-pos}. We will show that there is $\lambda' \not= \lambda$ such that $P_{\lambda'}$ contains $F$, and thus $P_{\lambda'} \in V$. 

If we picture $F$ in $\bZ^2 \hookrightarrow \bR^2$ as in Figure \ref{fig: P-Zsq-nbhd} we see that $F$ being antisymmetric means that any positive cone $P_\lambda$ containing $F$ must divide the $\bR^2$ plane such that one half of it contains $F$ and the other $-F$. It is intuitive to see that by rotating the plane slightly to have slope $\lambda'$ sufficiently close to $\lambda$, the positive cone $P_{\lambda'}$ will also contain $F$ and thus be in the neighbourhood $V$, as illustrated in Figure \ref{fig: P-Zsq-nbhd}.

Let us work out the bounds on $\lambda'$. Let's look at the case where $(\diamond_1, \diamond_2) = (>,>)$. Then, since $F \subset P_\lambda$, we have that $(x_i, y_i)$ satisfy

$$y_i > \lambda x_i \text{ or } y_i = \lambda x_i, \quad x_i > 0$$
for $1 \leq i \leq n$. 

Suppose now w.l.o.g. that 
$$y_i > \lambda x_i, \qquad 1 \leq i < \ell$$ with 
$$	x_i > 0, \qquad 1 \leq i < j $$
$$	x_i = 0, \qquad j \leq i < k $$
$$	x_i < 0, \qquad k \leq i < \ell$$
and
$$y_i = \lambda x_i, \quad x_i > 0, \qquad  \ell \leq i \leq n.$$ 

Then, 
\begin{equation}\label{eq: LO-1}
	\lambda < \frac{y_i}{x_i}, \qquad 1 \leq i < j,
\end{equation}

\begin{equation}\label{eq: LO-2}
	y_i > 0, \qquad j \leq i < k,
\end{equation}

\begin{equation}\label{eq: LO-3}
	\lambda > \frac{y_i}{x_i}, \qquad k \leq i < \ell,
\end{equation}
and
\begin{equation}\label{eq: LO-4}
	\lambda = \frac{y_i}{x_i}, \qquad \ell \leq i \leq n.
\end{equation}

By selecting $\lambda' \in (\min_{k \leq i < \ell} \frac{y_i}{x_i}, \max_{1 \leq i < j} \frac{y_i}{x_i})$, we satisfy Condition \ref{eq: LO-1}-\ref{eq: LO-3}. By choosing $\lambda' > \lambda$, we have that $(x_i, y_i) \in P_{\lambda'}$ for $\ell \leq i \leq n$ without needing to have $\lambda' = \lambda = \frac{y_i}{x_i}$. Putting it all together, we can choose some $$\lambda' \in (\lambda, \max_{1 \leq i < j} \frac{y_i}{x_i})$$ such that $F \subset P_{\lambda'}$ and thus $P_{\lambda'} \in V$.\sidenote{Note that it is clear that $\lambda \not= \max_{1 \leq i < j} \frac{y_i}{x_i}$ by Condition \ref{eq: LO-1}.} 

The other cases of $\diamond_1, \diamond_2$ are similar. 
\end{ex}

\begin{cor}
	No positive cone in $\bZ^2$ is finitely generated. 
\end{cor}

\begin{rmk}
	In the light of Example \ref{ex: LO-K2-fg},
	it is interesting to remark how the left-orders of $K_2$ and $\bZ^2$ are so topologically distinct despite $\bZ^2$ being a subgroup of index $2$ in $K_2$. Moreover, by Example \ref{ex: LO-aut-Zsq}, the inherited subgroups orders of $\bZ^2 \leq K_2$ are automorphic orders corresponding to $\cP_\lambda$ when $\lambda \in \bQ$ and should in, in some sense, computationally similar despite not being finitely generated. 
	
	In the next chapters, we will give a notion of computational similarity that is shared between the positive cones of $\bZ^2$ and $K_2$. 
\end{rmk}

Finally, our last example of left-orderable groups concern locally indicable groups, whose left-orders can be obtained non-explicitly by taking limit points in the space of left-orders. 

\section{Locally indicable groups}\label{sec: LO-locally-indicable}

\begin{defn}[Locally indicable] 
A group $G$ is \emph{locally indicable} if every non-trivial, finitely generated subgroup has a non-trivial homomorphism onto $\mathbb{Z}$.
\end{defn}

We can show that locally indicable groups are left-orderable by using a ``local-to-global'' chain of choices. By choosing an increasingly large sequence of semi-groups obtained by the ``local'' homomorphisms to $\bZ$ we can create a``global'' semi-group of maximal size respecting the trichotomy property, which will be a positive cone. 

\begin{lem}\label{lem: LO-locally-ind-iff}
	A group $G$ is left-orderable if and only if for every finite set $\{g_1, \dots, g_n\}$ of $G$ which does not contain the identity, there exists exponent-signs $\epsilon_i = \pm 1$ for $i = 1, \dots, n$ such that $1 \not\in \langle g_1^{\epsilon_1}, \dots, g_n^{\epsilon_n} \rangle^+$. 
\end{lem}
\begin{proof}
	$(\Rightarrow)$
	Let $P$ be the positive cone given by the left-order on $G$. Then, for $1 \leq i \leq n$, either $g_i$ or $g_i\inv$ is in $P$. Let 
	$$\epsilon_i = \begin{cases}
		1 & g_i \in P \\
		-1 & g_i\inv \in P
	\end{cases}.$$
	Then, since $P$ is a semigroup, we have that $P \supseteq \langle g_1^{\epsilon_1}, \dots, g_n^{\epsilon_n} \rangle^+$. Moreover, by the trichotomy property, $1 \not\in P$, which is what we wanted to show. 
	
	$(\Leftarrow)$ Define $P_n = \langle g_1^{\epsilon_1}, \dots, g_n^{\epsilon_n} \rangle^+$. Then, define $P := \lim_{n \to \infty} P_n$, where the $P_n$'s are viewed as elements of $\cP(G -\{1_G\})$. This space is compact, so the limit $P$ exists. Since by construction $P$ is a semigroup with the trichonomy property, $P$ is a positive cone for $G$. 
\end{proof}

\begin{ex}[Locally indicable groups]
	We will show that locally indicable groups are left-orderable by satisfying the condition of Lemma \ref{lem: LO-locally-ind-iff}. 
	
	Let $S_1 = \{g_1, \dots, g_n\}$ be a finite set of $G$ not containing the identity, and $G_1 = \langle S_1 \rangle$ be the subgroup generated by $S_1$. Then, there exists a non-trivial surjective homomorphism $\phi_1: G_1 \to \bZ$. That is, there exists at least one generator $g \in S_1$ such that $\phi_1(g) \not= 0$. Define 
	$$S'_j = \{g_i \in S_j \mid \phi_j(g_i) \not= 0\}$$ and  
	$$S_{j+1} := \{g_j \in S_j \mid \phi_j(g_i) = 0\}.$$
	Let $G_{j+1} := \langle S_{j+1} \rangle$ and let $\phi_{j+1}: G_{j+1} \to \bZ$ be another non-trivial homomorphism from a subgroup of $G$ to $\bZ$. 
	Then, for every $g_i \in S'_i$, we can define $\epsilon^j_i$ as 
	$$\epsilon^j_i := \begin{cases}
		1 & \phi_j(g_i) > 0 \\
		-1 & \phi_j(g_i) < 0
	\end{cases}$$
	for $1 \leq i,j \leq n$. 
	Observe that $\epsilon^j_i$ is only defined once over $1 \leq j \leq n$, as if $g_i \in S'_j$, then naturally $g_i \in \not\in S'_{j+1}$. Thus, we can drop the superscript and write $\epsilon_i := \epsilon^j_i$ for $g_i \in S'_j$. Moreover, define 
	$$\phi(g_i) := \min_j \phi_j(g), \quad \phi_j(g) \not= 0,$$
	and notice again that this is defined over one such $j$. 
	
	Thus, by definition, $\phi(g^\epsilon_i(g_i)) \geq 0$ for any $g_i \in S_1$. That is, it is positive when $g_i \in S'_j$, or trivial when $g_i \not\in S'_j$. 
	
	Now, we claim that $1 \not\in \langle g_1^{\epsilon_1}, \dots, {g_n^{\epsilon_n}} \rangle^+$. Indeed, assume otherwise. Then, $1$ can be written as a product of the $g_i^{\epsilon_i}$'s, and thus $\phi(1) \geq 0$ for all $1 \leq j \leq n$, and $\phi(1) > 0$ whenever $g_i^{\epsilon_i}$ appears in the product for some $j$ such that $g_i \in S'_j$, contradiction the assumption that $\phi_j$ is a homomorphism. 
\end{ex}

\begin{rmk}
	In the proof of the lemma above (\ref{lem: LO-locally-ind-iff}), no explicit left-order is given as the proof relies on the existence of a limit point to obtain the left-order. Thus, although locally indicable groups are known to be left-orderable, an explicit left-order is not necessarily known. This makes studying their left-orders with respect to a formal language difficult. In this thesis, we will not study groups that are only known to be left-orderable by local indicability. 
\end{rmk}

%% file: chap/formal-lang.tex
\chapter{Formal languages, informally.}\label{chap: informal-lang}

In this chapter, we will introduce formal languages under the lens of decidability. This chapter is meant to be introductory and elaborate on the broader meaning of the formal languages framework. We will then introduce regular and context-free languages formally in the subsequent chapters.

We have already seen the definitions of words, languages and monoids at the beginning of Chapter \ref{chap: LO} and how these concepts arise naturally when studying left-orders of finitely generated groups. In Example \ref{ex: LO-K2}, we have devised a natural way to write each element of the Klein bottle group $K_2$ in terms of their generators in order to decide whether $g \prec h$ for some left-order $\prec$. The framework of formal languages allow us to classify the computational complexity of such a decision.

Let us be more discuss what we mean by making a ``decision'' and how to classify its computational complexity. 

\section{Turing machines}
A Turing machine (Figure \ref{fig: turing-machine}) is often thought of as an abstract machine modelling a computer. However, its relative simplicity compared to the modern computer makes it a great abstract model for proving which problems are ``solvable'' (or \emph{decidable}) by a computer and which are not. 

More formally, a Turing machine is a structure which takes as input a \emph{single} word over a set alphabet $X$ and returns as output, potentially in infinite time, a \emph{decision} which is either ``yes'' or ``no''. A Turing machine has the following components. 

\begin{itemize}
\item A finite set of states. 
\item An infinite tape divided into cells which stores the input, and subsequently, the transformed input. 
\item A current state and a current cell.
\item A set of pre-defined rules on how to do the computation, called the \emph{transition function}. 
\end{itemize}

\begin{figure}[h]{
\includegraphics{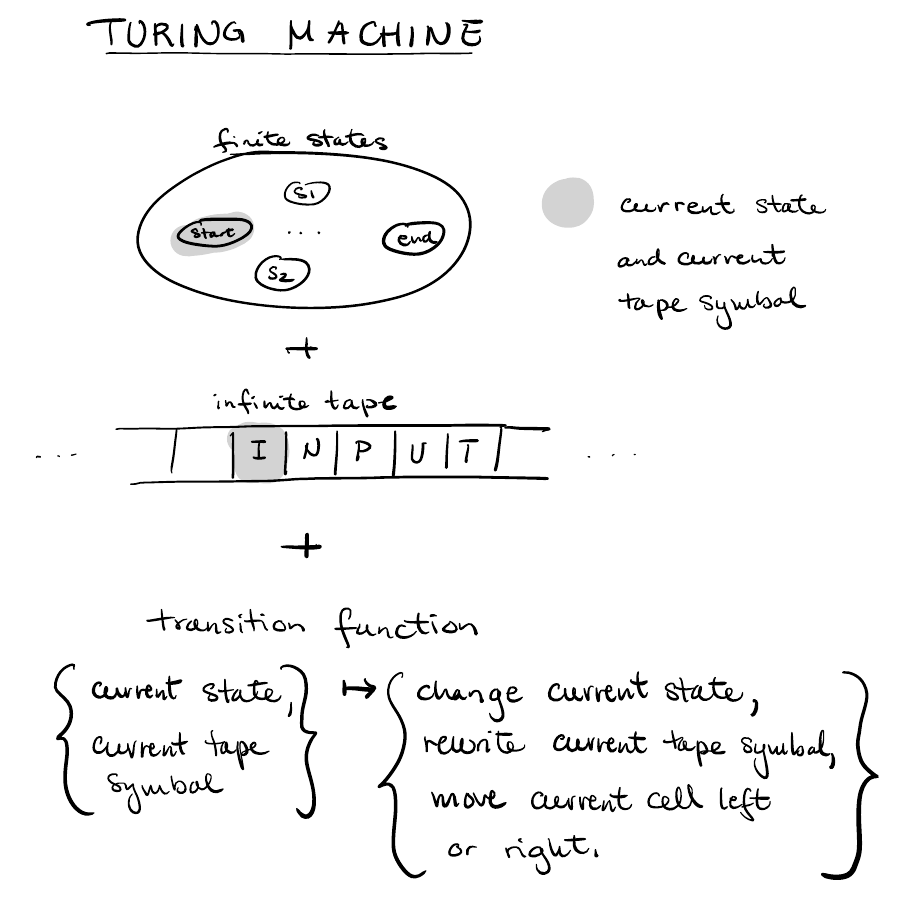}
}
\caption[][\baselineskip]{A Turing machine. The Turing machine starts with the start state as current state (here denoted as ``start'') and the first letter of the input word (here denoted ``INPUT''). The current state and current tape symbol are fed to the transition function which dictates how the Turing machine is to process them.}
\label{fig: turing-machine}
\end{figure}

The finite set of states can be thought of as a finite memory storage telling the Turing machine which subroutine it is executing. There is always a \emph{start state} indicating the start subroutine and \emph{accept} or \emph{final} states. The start state and accept states are not necessarily different. 

The infinite tape is made out of an infinite list of cells which store a letter of the tape alphabet $\Gamma$, an alphabet distinct from the input alphabet $X$ but containing $X$. The infinite tape is assumed to store the input word $w \in X^*$ already at the start state of the Turing machine, where all the letters of the input reside sequentially in their own cells as in Figure \ref{fig: turing-machine}. The rest of the cells initially store a letter designating that they are blank, called the \emph{blank symbol}. From then on, the tape will be modified according to the rules encoded in the transition function. 

At each \emph{move} (or step of the Turing machine), the transition function takes as inputs the current state and current tape symbol in the current cell and,according to those inputs, either
\begin{enumerate}
\item changes the current state (it can be the same current state as before),
\item writes a new symbol unto the current cell (it can be the same symbol that was already on the tape),
\item moves the tape head left or right from its original position, changing the current cell. 
\end{enumerate}
The Turing machine then repeats for the new current state and current cell and \emph{halts} if there are no more rules that can be applied according to the transition function. 

A Turing machine does not always halt, depending on the input word. If the Turing machine halts and the current state at halting is the accept state, then we say that the word is \emph{accepted} and the decision is ``yes''. If the Turing machine halts without accepting, then the word is \emph{rejected} and the decision is ``no''. If the Turing machine does not halt, then no decision is made.

There exists a specific Turing machine, equipped with specific finite states and transition function, that models how a computer works. On the other hand, it is hypothesised that every algorithm can be represented by a Turing machine (again equipped with a specific finite control and a specific transition function). This hypothesis is called the Church-Turing thesis. A formal proof for the thesis does not exist because the notion of algorithm is ill-defined.

\section{The Chomsky hierarchy}

When thinking in terms of decidability, a Turing machine is can be viewed as an ``algorithm'' which may not necessarily halt if the input word is not accepted. The set of words accepted by such a Turing machines form language which, by definition, has the property that its membership problem can be decided by a Turing machine. Languages which are accepted by Turing machines are classified as \emph{recursively enumerable} and deemed \emph{semi-decidable}.\sidenote{To re-iterate, we are defining decidability on the languages accepted by abstract machines, not the abstract machines themselves!} 

Turing machines which always halt regardless of input are formally called \emph{algorithms}. This matches an alternative definition of algorithm which the reader may already be familiar with, which is a finite set of instructions which for finite input returns an output in finite time. The set of languages accepted by algorithms are called \emph{decidable}. 

If words accepted by algorithms are a starting point for decidability, then to be of lower complexity means to be accepted by a less sophisticated algorithm, or equivalently a less sophisticated abstract machine. One way to gradually lower the complexity of a Turing machine is to impose specific restrictions on its tape until there is no tape, as illustrated by the table below. We will discuss the definitions of finite state automaton and pushdown automaton at length in the next chapters. Here, we only wish to give the reader a first look at the technical terms and the high-level overview.

\begin{center}
\begin{table*}
\begin{tabular}{l|l|l}
Automaton type & Tape restriction & Class of accepted language \\
\bottomrule 
Turing machine & Full infinite tape; no restrictions. & Recursively enumerable \\
Linear bound automata & Length of tape$^\dagger$ is $k$ times the length of input$^{\ddagger}$. & Context-sensitive \\
Pushdown automata & Stack capable of running in parallel$^\ast$ \ & Context-free \\
Finite state automata & No tape & Regular \\
\end{tabular}
\caption[][2\baselineskip]{

$^\dagger$: The length of the tape here refers to the number of cells it has. 

$^{\ddagger}$: for some $k \geq 0$. 

$^\ast$: see Chapter \ref{chap: pushdown-automata} for a definition.}
\end{table*}
\end{center}

The hierarchy on the table is a containment hierarchy called the \emph{Chomsky hierarchy}. The Chomsky hierarchy was originally developed by Noam Chomsky in the 1950s to provide a formal framework for describing natural language syntax as products of generative grammars.\cite{Chomsky1956} We highlight the containment structure in Figure \ref{fig: chomsky}.

\begin{figure}[h]{
\includegraphics{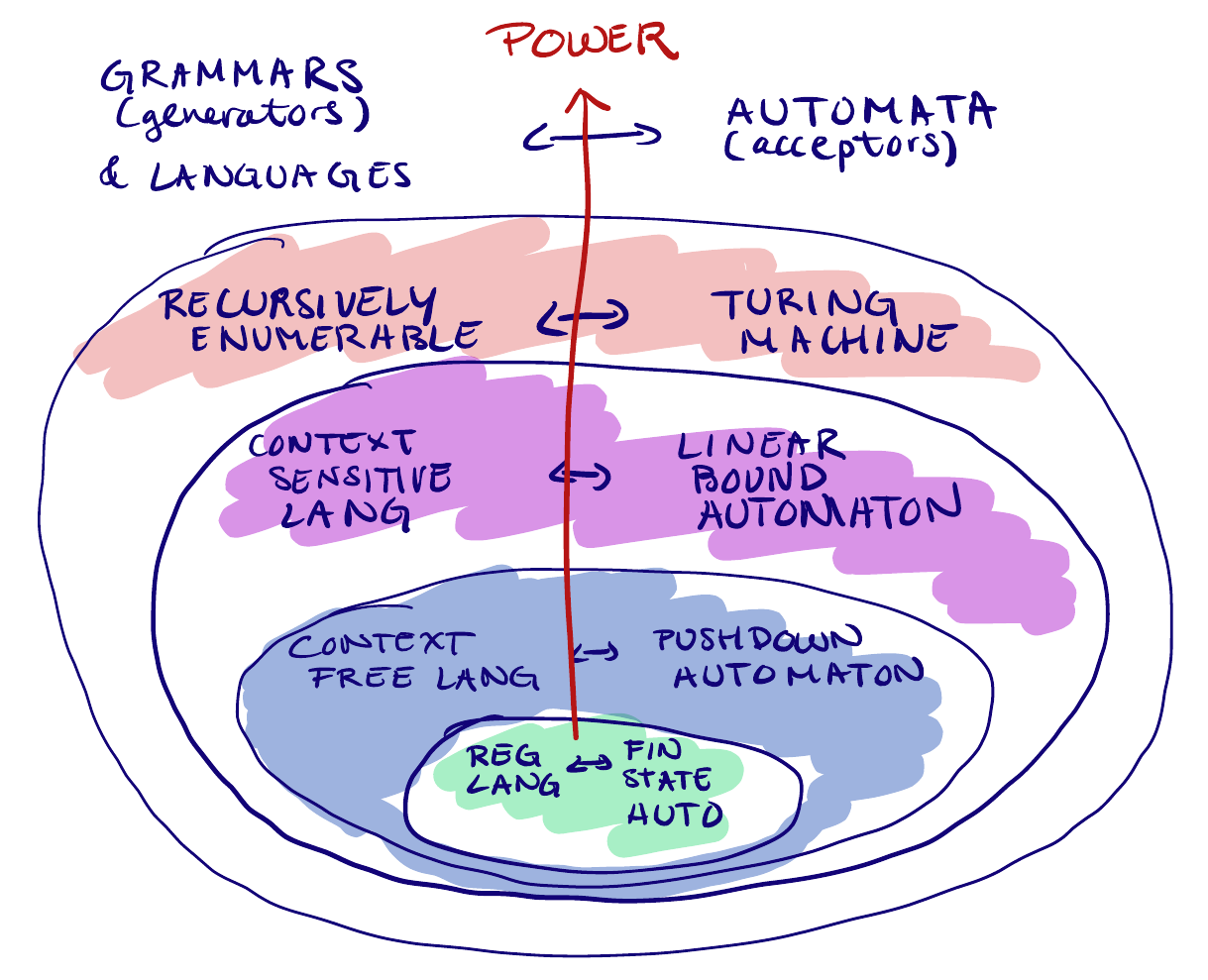}
}
\caption{We illustrate the Chomsky hierarchy, which is a containment hierarchy.
}
\label{fig: chomsky}
\setfloatalignment{b}
\end{figure}

\begin{rmk}[The halting problem]
A famous problem in decision theory is the halting problem: is it possible to determine if an abstract machine will halt or not? The problem is known to be unsolvable for Turing machines.\sidenote{See \url{https://en.wikipedia.org/wiki/Halting_problem}.} However, it is solvable for linear bounded automata (and by containment hierarchy, everything of lower complexity than linear bound automata). Vaguely, this is because automata with bounded memory have a finite number of configurations (current state, current tape symbol, word on the tape). Suppose a linear bound automaton has $N$ configurations. By pigeonhole principle, all the configurations will be visited after $N$ steps and halt unless it gets stuck in an infinite loop.\sidenote{Source: \url{https://cs.stackexchange.com/questions/22925/why-is-the-halting-problem-decidable-for-lba}.}

In conclusion, the infinite nature of the Turing machine tape is what causes the halting problem to be unsolvable. 
\end{rmk}  

Officially, a \emph{formal language} is a set of words over an alphabet, which is the same as how we have defined a language in Chapter \ref{chap: LO} Definition \ref{def: LO-lang}. The term ``formal'' is used to differentiate formal languages from the colloquial definition of a language, such as the English language. In practice, we often refer to formal languages in the context of decidability, such as when we are referring to the languages in the Chomsky hierarchy. A more precise term for languages in the Chomsky hierarchy is \emph{abstract family of languages} or AFLs,\sidenote{Since ``abstract family of languages'' is essentially meaningless, AFL is preferred.} which we discuss below. 

\section{Closure properties}\label{sec: closure-formal-lang}

The classification of complexity using tape restriction comes in very handy in terms of closure properties for family of languages in the Chomsky hierarchy. 

Every family of languages listed in the Chomsky hierarchy is closed under changing the underlying alphabet without introducing the empty word (\emph{homomorphism}) and reversing the change of the underlying alphabet (known as \emph{inverse homomorphism}). This means that the decidability class of a language is an inherently computational property independent of the chosen alphabet.\sidenote{This is especially striking in practice, since our computers famously run on binary yet I am writing this thesis using LaTeX! In other words, somewhere under the hood there is a homomorphism between the LaTeX code I am writing and its binary equivalence. Changing the alphabet does not change the ability of a computer.} 

\begin{defn}[monoid homomorphism]
Given two alphabets $X$ and $Y$ and a function $f : X \to Y^*$ there is a unique \emph{monoid homomorphism} $h: X^* \to Y^*$ extending $f$, namely, the map sending $w \in X^*$, $w = x_1 \dots x_n$ to $h(w) = f(x_1)\dots f(x_n)$, where each $f(x_i)$ is a word in $Y^*$. 
\end{defn}
\begin{rmk}
	A homomorphism introduces the empty word if there exists $x \in X$ such that $f(x) = \epsilon$, where `$\epsilon$' is the empty string. A homomorphism which is allowed to introduce the empty word is considered \emph{arbitrary}. 
\end{rmk}

An inverse homomorphism is the pre-image map of a homomorphism.

\begin{defn}[Inverse monoid homomorphism]
Given a monoid homomorphism $h: X^* \to Y^*$, the \emph{inverse monoid homomorphism} $h\inv$ is the map such that for any $L \subseteq Y^*$, we have $h\inv(L) = \{ w \in X^* \mid h(w) \in L\}$. 
\end{defn}

Moreover, the family of languages in the Chomsky hierarchy have a few additional closure properties, which we define. 

\begin{defn}[Union]
	The \emph{union} of two languages $L_1, L_2 \subseteq X^*$ is given by
	$$L_1 \cup L_2 := \{w \in X^* \mid w \in L_1 \cup L_2\}.$$
\end{defn}

\begin{defn}[Intersection]
	The \emph{intersection} of two languages $L_1, L_2 \subseteq X^*$ is given by
	$$L_1 \cap L_2 := \{w \in X^* \mid w \in L_1 \cap L_2\}.$$
\end{defn}

\begin{defn}[Concatenation]
	The \emph{union} of two languages $L_1, L_2 \subseteq X^*$ is given by 
	$$L_1 L_2 := \{w \in X^* \mid w = uv, \quad u \in L_1, v \in L_2 \}.$$
\end{defn}

\begin{defn}[Kleene star]
The \emph{Kleene star of a language} $L$ is denoted $L^*$ and is given by $$L^* := \bigcup_{n=0}^\infty L^n,$$ where $L^0 = \{\epsilon\}$ denotes the empty word, and $L^n$ denotes the set of concatenations of $n$ words belonging to $L$. 
\end{defn}

\begin{defn}[Kleene plus]
The \emph{Kleene plus of a language} $L$ is denoted $L^+$ and is given by $$L^+:= L^* - L^0 = \bigcup_{n=1}^\infty L^n.$$
\end{defn}

\begin{defn}[AFLs]
A family of language is called an \emph{AFL} if it is closed under homomorphisms that do not introduce the empty word, inverse homomorphisms, unions, concatenations, Kleene star, and taking intersections with regular languages. 

A family of languages is called a \emph{full AFL} if it is an AFL that is closed under arbitrary homomorphisms (including the ones which introduce the empty word). 
\end{defn}

All languages in the Chomsky hierarchy are AFLs, but only the regular, context-free, and recursively enumerable languages are full AFLs (leaving out context-sensitive languages). 

Finally, a useful property of AFLs in the Chomsky hierarchy is that they are closed under reversal. This is not true of AFLs in general. 

\begin{defn}(Reversal)
	The reversal of a language $L$ is given by 
	$$L^R := \{x_n \dots x_1 \mid x_1 \dots x_n \in L\}.$$
\end{defn}

To AFLs in the Chomsky hierarchy, the classification by restrictions on the tape has a lot to do with its nice closure. We will roughly explain in Table \ref{tab: closure-tape} the reason behind why each AFL operation does not increase the complexity of the automaton. Note that our explanations assume that we can run our automata in parallel, which is known as \emph{non-determinism}. 

\subsection{Non-determinism}

Running things in parallel is formally referred to as \emph{non-determinism} in automata theory. The notion arises naturally when discussion closure properties of AFLs. For example, an automaton $\bA$ accepting $L_1 \cup L_2$ can be viewed as running the automaton $\bA_1$ accepting $L_1$ in parallel with the automaton $\bA_2$ accepting $L_2$. By definition of AFLs, this cannot increase the complexity of the automaton accepting $L_1 \cup L_2$ from the maximal complexity of $\{\bA_1, \bA_2\}$. 

Thus, passing from a deterministic automaton to a non-deterministic one should not change the complexity of an automaton. Table \ref{tab: tape-parallel} lists why this is the case for each complexity class. 

\begin{center}
\begin{table*}\label{tab: tape-parallel}
\begin{tabulary}{\textwidth}{p{0.25\textwidth}|p{0.75\textwidth}}
Automata type & Why non-determinism does not increase complexity\\
\bottomrule 
Finite state automata & Non-deterministic FSAs can be made deterministic by changing its set of states $S$ to its power set $\cP(S)$ and updating the transition function accordingly. Details found in Chapter \ref{chap: fsa} Section \ref{sec: fsa-det-non-det}.\\
\midrule
Pushdown automata & PDAs are assumed to be capable of non-determinism by default. Deterministic PDAs are in a strictly lower complexity category than non-deterministic PDAs, with a different stack that cannot handle parallelism.\\
\midrule
Linear bound automata and Turing machines & Multiple tapes can be simulated with a single tape where the number of cells used for the single tape is the maximum over the number of cells used for the multiple tapes. \\
\end{tabulary}
\end{table*}
\end{center}

\begin{center}
\begin{table*}\label{tab: closure-tape}
\begin{tabulary}{\textwidth}{p{0.25\textwidth}|p{0.75\textwidth}}
Closure property & Why complexity does not increase\\
\bottomrule 
Homomorphism $h$ which does not introduce the empty word and inverse homomorphism $h\inv$ & Change causes a linear change to the input, which is handled by the transition function taking $h(x)$ as input instead of $x$ (and vice-versa for $h\inv$) and the tape scaling linearly to memorise the input. Note that for linear-bound automata, if the homomorphism is arbitrary, then it is possible to lose the property that the tape is linearly bound by the length of the input, as the empty word has length $0$. \\
\midrule
Union & Run two automata in parallel on the same input word and accept if word is accepted in at least one. \\
\midrule
Concatenation & Run first automaton on first word and second automaton on second word, where all the different possible separations between the two words are considered in parallel.\\
\midrule
Intersection with regular languages & Run automaton and finite state automaton in parallel on same word and accept if word is accepted in both. \\
\midrule
Kleene star & Modify automaton to accept empty string. Then, use the same procedure as concatenation. \\
\midrule 
Reversal & Proceed with all the operations on the finite states / stack / tape in reverse. Doing so does not increase the memory requirement of the original automaton. 
\end{tabulary}
\end{table*}
\end{center}

To learn about the multiple tape to single tape equivalence and non-determinism in Turing machines, we refer to \cite[Chapter 8.4]{HopcroftMotwaniUllman2007}. The information for linear-bound automata can be inferred from the information for Turing machines by bookkeeping that the number of available cells is still linearly bound by the input length.  

In Chapter \ref{chap: fsa} and Chapter \ref{chap: pushdown-automata}, we will explore finite state automata and pushdown automata and their closure properties in more detail. We will make heavy use of non-determinism in the form of $\epsilon$-transitions to aid in our treatment. 

%% file: chap/fsa.tex
\chapter{Finite state automata}\label{chap: fsa}

In this chapter, we introduce finite state automata first informally, then formally along with their closure properties. At the end of the chapter, we use what we learn to prove a recent result in the literature about positive cones admitting a regular language representation. 

\section{Information definition}

If you take a Turing machine and keep the finite state part, you get a \emph{finite state automaton} (FSA). The FSA starts at one state, then goes to another based on the input that is read, and so on. This will be encoded, as we will see below, as directed labeled edges. As with Turing machines, a finite state automaton has a set of inputs which are \emph{accepted}. This sets of input forms a language, and we call all such languages accepted by finite state automata \emph{regular languages}.

We encode finite state automata as directed graphs. States are vertices, and going from state $s_a$ to state $s_b$ when reading input $x$ is encoded as a directed edge from $s_a$ to $s_b$ with label $x$. The labels on the edge come from an alphabet $X$, and paths form words over $X$ by collecting the edge labels.
\begin{figure}[h]
\centering
\includegraphics[width=0.5\textwidth]{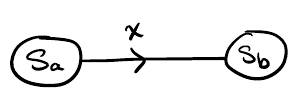}
\caption{Part of a finite state automaton. }
\label{fig: fsa-a-to-b}
\end{figure}

Some states are identified as special and called \emph{start states} and \emph{accept states}. Start states are graphically denoted by a vertex with an incoming arrow that has no source. Accept states are graphically denoted by concentric circles.

\begin{figure}[h]
\centering
\includegraphics[width=0.75\textwidth]{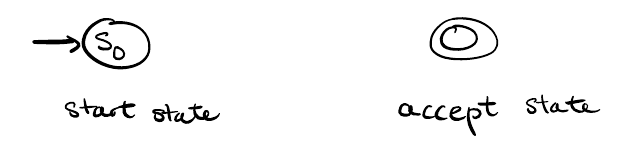}
\caption{Graphical notation for start and accept states.}
\label{fig: fsa-start-accept}
\end{figure}

The accepted language is the collection of words which arise from the edge labelss of paths from the start state to an accept state. 

The vertices of a finite state automaton informally act as a proxy for memory.

\begin{figure}[h]{\includegraphics{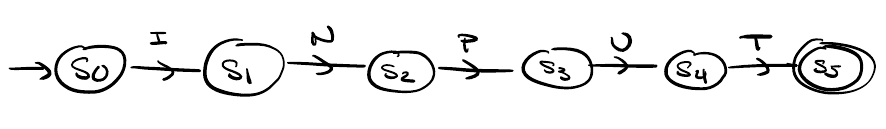}}
\caption{A finite state automaton accepting only the word "INPUT". 
}
\label{fig: fsa-ex-input}
\end{figure}

\begin{ex}\label{ex: fsa-input}
The finite state automaton in Figure \ref{fig: fsa-ex-input} accepts only the word "INPUT". The states memorize the inputted prefixes, since the only words which leads us from the start state $s_0$ to:
\begin{itemize}
	\item $s_1$ is "I"
	\item $s_2$ is "IN"
	\item $s_3$ is "INP"
\end{itemize} and so on.
\end{ex}

The following FSA is a bit more sophisticated.

\begin{figure}[h]
\centering
\includegraphics[width=0.75\textwidth]{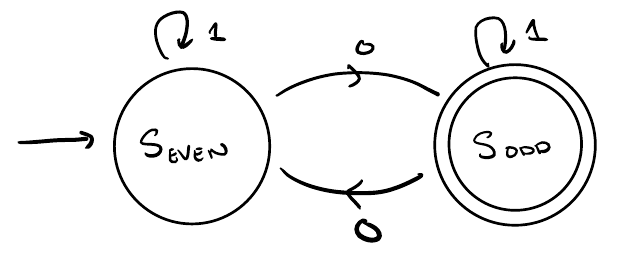}
\caption{A finite state automaton accepting all binary strings with an odd number of zeroes}
\label{fig: fsa-ex-even-odd}
\end{figure}

\begin{ex}\label{ex: fsa-even-odd}
The finite state automaton in Figure \ref{fig: fsa-ex-even-odd} accepts all binary strings with an odd number of zeros. The vertices act as memory for the parity of the number of zeros encountered.
\end{ex}

It is worth noting that a vertex can have two outgoing edges with the same label, and multiple accept states. In that case, a word $w$ can form two different paths. As long as one of the path is from the start state to one of the accept states, $w$ is accepted. Let's look at a slight modification of the previous automaton. 

\begin{figure}[h]
\centering
\includegraphics[width=0.75\textwidth]{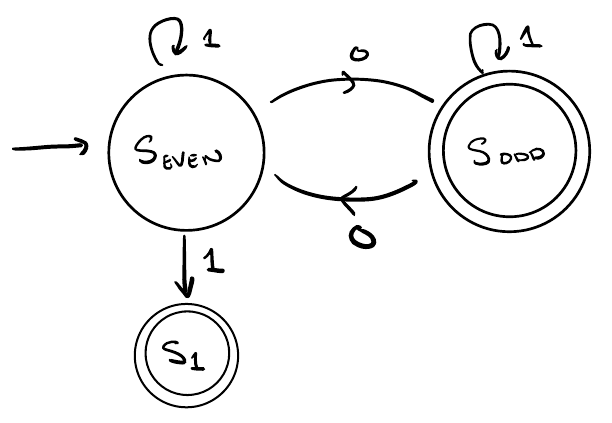}
\caption{A finite state automaton accepting all binary strings with an odd number of zeroes and the string ``1''.}
\label{fig: fsa-ex-even-odd-mod}
\end{figure}

\begin{ex}\label{ex: fsa-even-odd-mod}
The finite state automaton in Figure \ref{fig: fsa-ex-even-odd-mod} is a modified version of Figure \ref{fig: fsa-ex-even-odd} to which we've added to the $s_{\text{even}}$ state an outgoing edge with label `1' to the state $s_1$. This automaton accepts all binary strings with an odd number of zeros \emph{and} the string ``1''. Two of the vertices act as memory for the number of zeros encountered and the new vertex acts as memory for seeing the string ``1''. 
\end{ex}

We are ready to look at a more formal definition of a finite state automaton, which is basically a set of objects holding all the structure that we have previously defined.

\section{Formal definition}

\begin{defn}\label{defn: fsa}
	A \emph{finite state automaton} (FSA) is a quintuple $$\bA = (S, X, \delta, s_0, A),$$ where 
	\begin{itemize}
		\item $S$ is a finite set called the \emph{state set},
		\item $X$ is a finite alphabet for the input words,
		\item $\delta : S \times X \to \cP(S)$ is a \emph{transition function} taking one state to a set of other states,
		\item $A \subseteq S$ is a set of states called the \emph{accept states} (or \emph{final states}),
		\item $s_0 \in S$ is the \emph{initial state}.
	\end{itemize}
	The function $\delta$ extends recursively to $\delta : S \times X^* \to \cP(S)$ by setting $$\delta(s, wx) = \delta(\delta(s,w),x)$$ where $w \in X^*, x \in X$, and $s \in S$. The \emph{accepted language} by the automaton is the set of words $$\cL(\bA) := \{w \in X^* \mid \exists a \in \delta(s_0, w), a \in A\}.$$
\end{defn}

That is to say, a word $w$ is in the accepted language of $\bA$ if and only if one of the paths that $w$ induces goes from start state $s_0$ to one of the accept state $a \in A$. The fact that $w$ may induce another path which does not end at an accept state, or another path which ends at one, is not further taken into account in the language once $w$ has been included.

\begin{defn}
A language $L$ is {\it regular} if there is a finite state automaton which accepts it. That is, 
$$L= \cL(\bA) =\{w\in X^* \mid  \exists a \in \delta(s_0,w) \text{ with } a\in A\}.$$
\end{defn}

\begin{ex}[Example \ref{ex: fsa-input} continued]
The finite state automaton of Figure \ref{fig: fsa-ex-input} can be described with the following structure. The state set $S = \{s_0, \dots, s_5\}$. The alphabet $X$ is the standard English alphabet. The edge labelss are encoded in the transition function
\begin{align*}
\delta(s_0, `I') &= \{s_1\} \\
\delta(s_1, `N') &= \{s_2\}, \\
\delta(s_2, `P') &= \{s_3\}, \\
\delta(s_4, `U') &= \{s_4\}, \\
\delta(s_4, `T') &= \{s_5\}.
\end{align*} 
There is one accept state $A = \{s_5\}$, and the start state is already labeled (by abuse of notation) as $s_0$.
\end{ex}

\begin{ex}[Example \ref{ex: fsa-even-odd} continued]
The finite state automaton in Figure \ref{fig: fsa-ex-even-odd} can be described as follows. The state set $S = \{s_{\text{even}}, s_{\text{odd}}\}$. The alphabet is binary thus $X = \{0,1\}$. The transition function is given by 
\begin{align*}
&\delta(s_{\text{even}}, 0) = \{s_{\text{odd}}\}, 
&\delta(s_{\text{even}}, 1) = \{s_{\text{even}}\}, \\
&\delta(s_{\text{odd}}, 0) = \{s_{\text{even}}\}, 
&\delta(s_{\text{odd}}, 1) = \{s_{\text{odd}}\}. 
\end{align*}
\end{ex}
There is one accept state, $A = \{s_{\text{odd}}\}$ and the start state is $s_0 = s_{\text{even}}$.

\section{Uniqueness}

\begin{rmk}[Non-uniqueness of finite state automata]\label{rmk: fsa-non-unique}
The above definition says that for every regular language, there exists a finite state automaton which accepts that language. However, regular languages and finite state automata are \emph{not} in bijection; given a regular language $L$, there can be more than one finite state automaton which accepts $L$.
\end{rmk}

\begin{figure}[h]{
\includegraphics{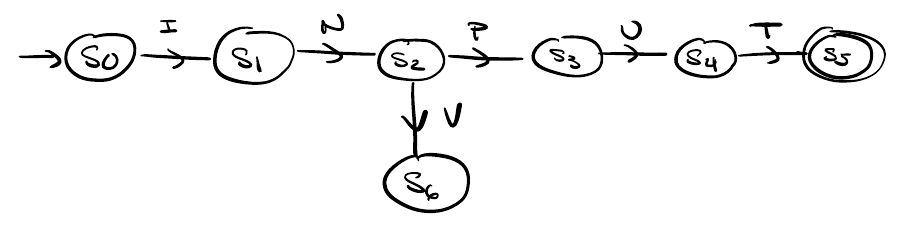}}
\caption{A finite state automaton accepting the word ``INPUT'' with state $s_6$ and label `V' which does not affect the accepted language.}
\label{fig: fsa-input-v}
\end{figure}

\begin{ex}
The finite state automaton of Figure \ref{fig: fsa-input-v} has a transition from $s_2$ to $s_6$ with label `V'. However, the letter `V' does not appear in the only accepted word, which is still `INPUT'. 
\end{ex}

At this point, the reader may wonder if there exists a unique representation of a finite state automaton. The answer to this question is yes, up to renaming states, if the finite state automaton is \emph{deterministic}. 

\section{Determinism}
A finite state automaton is \emph{deterministic} if for every state, there is only one defined outgoing state for every edge label and no edge label is given by the empty word. 

One way in which a finite state automaton is not deterministic is if the empty word $\epsilon$ is part of the alphabet of a finite state automaton. An edge with label $\epsilon$ from state $s_a$ to state $s_b$ means that you may, without reading any input, go from state $s_a$ to state $s_b$. The transition is denoted like any other, $\delta(s_a, \epsilon) = s_b$, and we often call such transitions \emph{$\epsilon$-transitions}. 

\begin{figure}[h]{
\includegraphics{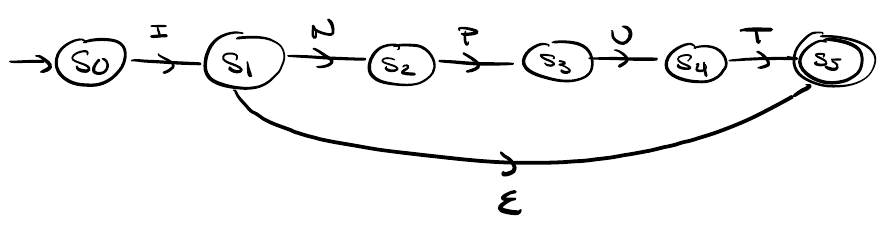}}
\caption{A finite state automaton accepting ``INPUT'' and ``I''.}
\label{fig: fsa-ex-input-mod}
\end{figure}

Another way to have a non-deterministic automaton is to have one state transition to multiple states. That is from a state $s_1$, we have that for some input $x$ the transition function gives us $\delta(s_1, x) = \{s_2, s_3\}$. Note that this equivalent to having an $\epsilon$-transition since for $s_1$ going to $s_2 = s_a$ with $\epsilon$-transition to $s_3 = s_b$, we also get $\delta(s_1, x) = \{s_2, s_3\}$ and vice-versa. 

In a deterministic finite state automaton, any word has only one associated path. That is, the path each word takes is unique. On the other hand, in a non-deterministic automaton, a word has many paths and is accepted if one of those paths leads to an accept state.

\begin{ex}\label{ex: fsa-ex-input-mod}
The finite state automaton in Figure \ref{fig: fsa-ex-input-mod} is a modified version of Figure \ref{fig: fsa-ex-input} to which we've added an $\epsilon$-transition from $s_1$ to $s_5$. As a result, both ``INPUT'' and ``I'' are accepted. We roughly think of adding this $\epsilon$-transition as adding some ``parallelism'' in our automaton because after reading the input ``I'' we are both at $s_1$ and $s_5$. 
\end{ex}

More concretely, non-determinism is akin to having an embedded backtracking algorithm: each of the multiple paths a word induces is a candidate solution, and paths are eliminated if they do not lead to an accept state.\sidenote{In practice, non-determinism can be computationally costly. However, our preoccupation with the Chomsky hierarchy is with the degree of decidability, and not computational costs associated with time and space.}

To recap, let us define deterministic automata formally. 

\begin{defn}
A \emph{deterministic finite state automaton} (DFA) is a FSA such that there are no $\epsilon$-transitions and such that the transition function $\delta: S \times X \to \cP(S)$ maps only to singletons, that is, for each vertex in the graph there is only one outgoing arrow with the same label. Equivalently, the transition function can be viewed as a function $\delta: S \times X \to S$. 
\end{defn}

\begin{defn}
	A \emph{non-deterministic finite state automaton} (NFA) is a FSA that is not a DFA. 
\end{defn}

\subsection{Uniqueness of DFAs}

\begin{rmk}[Remark \ref{rmk: fsa-non-unique} continued]\label{rmk: fsa-non-unique2}
If we have a deterministic finite state automaton, then the automaton can be reduced to a unique finite state automaton (up to renaming states) with a minimal number of state. There are three types of changes which can be done in a DFA $\bA = (S, X, \delta, s_0, A)$ without changing its accept language $\cL(\bA)$. 
\begin{itemize}
	\item Eliminating unreachable states. A state $s \in S$ is \emph{unreachable} if $$\nexists w \in X^* \text{ with } \delta(s_0, w) = s.$$
	\item Eliminating dead states (also known as fail states). A state $s \in S$ is \emph{dead} if it does not lead to an accept state in $A$. That is, $$\nexists w \in X^* \text{ with } \delta(s, w) \in A.$$
	\item Merging non-distinguishable states. Two states $s_1$ and $s_2$ are \emph{non-distinguishable} if they cannot be distinguished from one another by any input string. That is $$\nexists w \in X^* \text{ with } \delta(s_1, w) \not= \delta(s_2, w).$$  
\end{itemize}

Such a process is called \emph{DFA minimization} and can be formalised via the Myhill-Nerode theorem. Details of this process can be found in \cite[Section 4.4.3]{HopcroftMotwaniUllman2007}. 

Note that eliminating dead states turns $\delta$ into a partial function. Due to this, it is not considered allowed according to other sources. 

\end{rmk}

\section{Equivalence of determinism and non-determinism}\label{sec: fsa-det-non-det}
There is an algorithm to convert non-deterministic finite state automata to deterministic state automata by passing from a set of states $S$ to its power set $\cP(S)$. For example, instead of having two arrows with the same label $x$ pointing from $s_1$ to both $s_2$ and $s_3$, we simply have an arrow pointing from $\{s_1\}$ to $\{s_2, s_3\}$, as illustrated in Figure \ref{fig: fsa-nfa-dfa}

\begin{figure}[h]{
\includegraphics{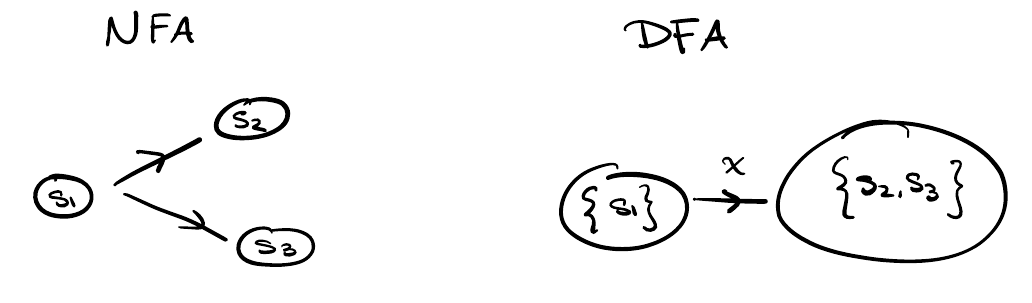}}
\caption{Passing from NFA to DFA. The transition $\delta_{\text{DFA}}: S \times X \to S$, $\delta_{\text{DFA}}(s_1, x) = \{s_2, s_3\}$ changes to $\delta_{\text{NFA}}: \cP(S) \times X \to \cP(S)$, $\delta_{\text{NFA}}(\{s_1\}) = \{s_2, s_3\}$.} 
\label{fig: fsa-nfa-dfa}
\end{figure}

The main idea is that by denoting the multiple outgoing states into subsets, the new deterministic transition function is inducing every possible path in the old finite state automaton. The details of this process can be found in \cite[Section 2.3.5]{HopcroftMotwaniUllman2007}.  %

\begin{rmk}[Remark \ref{rmk: fsa-non-unique2} continued] Given a non-deterministic finite state automaton which is to be minimized, the procedure would be 
\begin{enumerate}
	\item first convert it into a deterministic finite state automaton, 
	\item then apply the DFA minimization algorithm. 
\end{enumerate}
\end{rmk}

Because deterministic finite state automata and non-deterministic finite state automata are essentially equivalent, we will often refer to a non-deterministic FSA as simply a FSA without care as to whether it is deterministic or not. 

\section{Finite state automata as graphs}

Finite state automata are directed graphs with edge labels, identified accept vertices and an identified start vertex. The accepted words are viewed as arising from the edge labels of paths from the start state to an accept state. This is very natural to observe, and is made clear in much of the literature on the subject where it is discussed as a casual fact. It is made formal in \cite{GallierQuaintance2025}, which in turns takes reference from \cite{Eilenberg1974}. Note that the following is a non-standard treatment on the subject. For a more standard treatment, see the textbook by Ullman \cite{HopcroftMotwaniUllman2007}.

\begin{defn}\label{defn: fsa-graph-functor}
Define the \emph{graph functor} $\cG$ going from the category of finite state automata to the category of directed graphs equipped with an alphabet, a start vertex and accept vertices. $\cG$ takes as input a finite state automaton and return as output a directed graph. For a finite state automaton $\bA = (S, X, \delta, s_0, A)$, $$\cG(\bA) = (V, E, X, s_0, A)$$ is defined as follows. 
\begin{itemize}
	\item The vertex set is given by $V = S$. 
	\item The directed labeled edge-set is given by $$E = \{(s_a, s_b, x) \mid \exists x \in X \cup \{\epsilon\}, \delta(x, s_a) = s_b \}.$$
	\item The alphabet $X$ is mapped to itself. 
	\item The start state maps to the corresponding start vertex.
	\item The accept states map to the corresponding accept vertices. 
\end{itemize}
\end{defn}

The graph functor is a natural functor since a morphism of finite state automata translates to morphism of graph (see \cite{GallierQuaintance2025}). The unfortunate result is that a properly functored automaton is still a five-tuple as a graph. However, we will often be able to work with only $(V,E)$, the set of edges and vertices, forgetting the rest of the structure and accessing the language of graph theory. %

We now define the inverse operation to $\cG\inv$. 

\begin{defn}\label{defn: fsa-inv-graph-functor}
Define the \emph{contravariant graph functor} $\cG\inv$ from the category of graphs with labeled edges equipped with an alphabet, a set of accept vertices, and a special start vertex to the category of finite state automata. 
\end{defn}

\section{Closure properties of regular languages}

We will use the graph functor notation to prove the closure properties of regular languages. 

\subsection{AFL closure properties}

We recall here the discussion of Section \ref{sec: closure-formal-lang} concerning closure properties formal languages. In order to be both intuitive and efficient in our proofs, we will view our FSAs as graphs and prove statements about their paths. For a more classical treatment using induction on the transition functions, see \cite[4.2]{HopcroftMotwaniUllman2007}. 

\begin{thm}\label{thm: reg-lang-closure}
 Regular languages are full AFL, that is, they are closed under union, concatenation, Kleene star, intersection (with a regular language), homomorphism and inverse homomorphism. 
\end{thm}

For the following proofs, let $\cG$ be the graph functor and $\cG\inv$ its contravariant functor as defined in Definition \ref{defn: fsa-graph-functor} and Definition \ref{defn: fsa-inv-graph-functor}. 

\begin{lem}[Closure under union]
\label{lem: fsa-union}
Let $L_1$ and $L_2$ be two regular languages. Then $L_1 \cup L_2$ is a regular language. 
\end{lem}
\begin{proof}
We refer to Figure \ref{fig: fsa-union}. Let $\bA_1$, $\bA_2$ be finite state automata accepting $L_1, L_2$ respectively, and let \begin{align*}
	&\cG(\bA_1) = (V_1, E_1, X_1, s_1, A_1), \\
	& \cG(\bA_2) = (V_2, E_2, X_2, s_2, A_2).
\end{align*}

Create a new graph $G$ with labeled edges $$G = (V_1 \cup V_2 \cup \{s_0\}, E_1 \cup E_2 \cup \{(s_0, s_1, \epsilon) \cup (s_0, s_2, \epsilon)\}).$$ Let $A = A_1 \cup A_2$. Then, the set of paths from $s_0$ to $A$ is the edge $(s_0, s_1, \epsilon)$ concatenated with the set of paths from $s_1$ to $A$, plus the edge $(s_0, s_2, \epsilon)$ concatenated with the set of paths from $s_2$ to $A$. However, the paths starting from $s_1$ can only end in $A_1$ if they end in $A$, and the paths starting at $s_2$ can only end in $A_2$ if they end in $A$. 

This means that the language accepted by $$\cG\inv(G, X_1 \cup X_2, s_0, A),$$ our automaton induced by our new graph, is exactly $$\epsilon L_1 \cup \epsilon L_2 = L_1 \cup L_2.$$
\end{proof}

\begin{figure}[h]{
\includegraphics{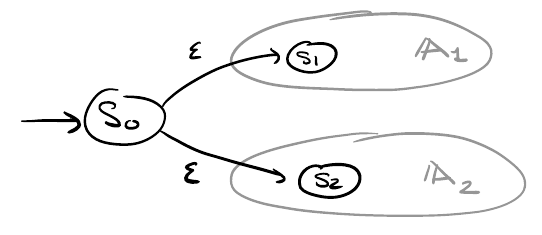}}
\caption{Illustration of creating a finite state automaton which accepts the union of two regular languages.}
\label{fig: fsa-union}
\end{figure}

\begin{lem}[Closure under concatenation]
Let $L_1$ and $L_2$ be two regular languages. Then $L_1 L_2 = \{w_1 w_2 \mid w_1 \in L_1, w_2 \in L_2\}$ is a regular language. 
\end{lem}
\begin{proof}
We refer to Figure \ref{fig: fsa-concat}. Let $\bA_1, \bA_2$ be finite state automata accepting $L_1$ and $L_2$ and let \begin{align*}
	&\cG(\bA_1) = (V_1, E_1, X_1, s_1, A_1), \\
	& \cG(\bA_2) = (V_2, E_2, X_2, s_2, A_2).
\end{align*} Let $$G = (V_1 \cup V_2, E_1 \cup E_2 \cup \{(a_1, s_2, \epsilon) \mid a_1 \in A_1\}).$$ Let $s_0 = s_1$ and $A = A_2$. Then every path from $s_0$ to $A$ is a path starting at $s_1$ going to $A_1$, passing by the edge $(a_1, s_2, \epsilon)$ with $a_1 \in A_1$, and ending at $A_2$. Therefore, the language accepted by $$\cG\inv(G, X_1 \cup X_2, s_0, A)$$ is $$L_1 \epsilon L_2 = L_1 L_2.$$
\end{proof}

\begin{figure}[h]{
\includegraphics{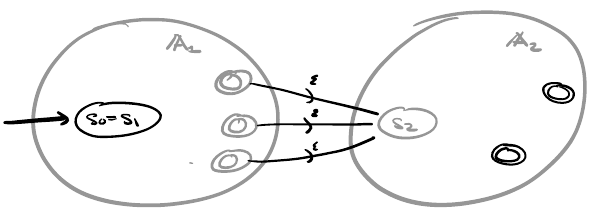}}
\caption{Illustration of creating a finite state automaton which accepts the concatenation of two regular languages.}
\label{fig: fsa-concat}
\end{figure}

\begin{lem}[Closure under Kleene star]
If $L$ is a regular language, then $L^*$ is also a regular language.
\end{lem}
\begin{proof}
\begin{figure}[h]{
\includegraphics[width=\textwidth]{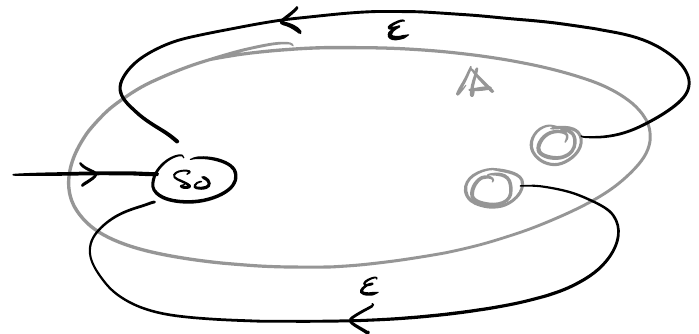}}
\caption{Illustration of creating a finite state automaton which accepts a regular language taken to any positive power.}
\label{fig: fsa-npow}
\end{figure}
We refer to Figure \ref{fig: fsa-npow} for the first part of the proof. Let $\bA$ be a finite state automaton accepting $L$ and let $$\cG(\bA) = (V,E, X, s_0, A)$$ be the associated graph. Define $$G = (V, E \cup \{(a, s_0, \epsilon) \mid a \in A\}).$$ Then, every path from $s_0$ to $A$ is either a path from $s_0$ to $A$ in the original graph $\cG(\bA)$, or a path passing through a new edge of the form $(a, s_0, \epsilon)$, in which case from $s_0$ it must form another path in $\cG(\bA)$. Therefore, $$\cG\inv(G, X, s_0, A)$$ accepts $$\bigcup_{n=1}^\infty L^n.$$

\begin{figure}[h]
\centering
\includegraphics[width=0.2\textwidth]{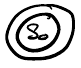}
\caption{Finite state automaton accepting the empty word only.}
\label{fig: fsa-epsilon}
\end{figure}

We refer to Figure \ref{fig: fsa-epsilon} for the second part of the proof. By that figure, it is easy to see that $\{\epsilon\}$ is a regular language. To complete our proof, we use Lemma \ref{lem: fsa-union} to state that $\bigcup_{n=1}^\infty L^n \cup \{\epsilon\}$ is a regular language.
\end{proof}

\begin{cor}[Closure under Kleene plus]
	Regular languages are closed under Kleene plus because $$L^+ = LL^*$$ 
\end{cor}

We will use the tensor product of graphs for our next result. 

\begin{defn}[Tensor product of graphs]
	Let $G_1$ and $G_2$ be graphs. The tensor product of graphs 
	$$G = G_1 \times G_2 = (V,E)$$
such that 
\begin{itemize}
	\item the vertex set is the cartesian product of the vertices of $G_1, G_2$ respectively, $V = V(G_1) \times V(G_2)$,
	\item $(ac,bd)$ is an edge in $E$ with label $x$ if and only if $(a,b)$ and $(c,d)$ are edges with label $x$ in $G_1$ and $G_2$ respectively. 
\end{itemize}
We illustrate this in Figure \ref{fig: fsa-graph-product}. 
\end{defn}

\begin{figure}[h]
\centering
\includegraphics[width=\textwidth]{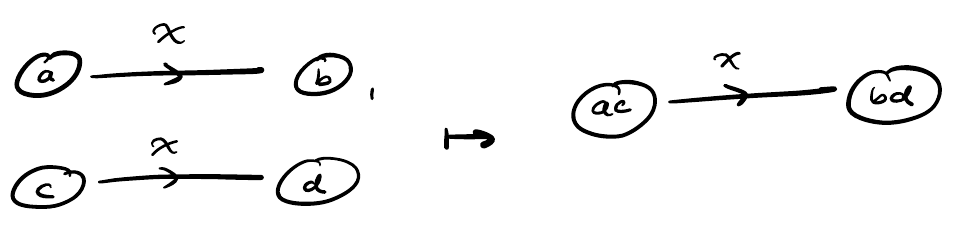}
\caption{Illustrating the edge $(ac,bd,x)$ of the graph tensor product $G = G_1 \times G_2$ arising from $(a,b,x)$ and $(c,d,x)$ from $G_1$ and $G_2$ respectively.}
\label{fig: fsa-graph-product}
\end{figure}

\begin{lem}[Closure under intersection with regular languages]
Let $L_1$ and $L_2$ be two regular languages. Then $L_1 \cap L_2$ is a regular language. 
\end{lem}

\begin{proof}
Let $\bA_1, \bA_2$ be the automata accepting $L_1$ and $L_2$. The idea behind the proof is that we think of an input word $w$ ``runs in parallel'' in both automata $\bA_1$ and $\bA_2$ and is accepted if and only if it is accepted by both $\bA_1$ and $\bA_2$. Let us formalize this. 

Let \begin{align*}
	&\cG(\bA_1) = (V_1, E_1, X_1, s_1, A_1), \\
	& \cG(\bA_2) = (V_2, E_2, X_2, s_2, A_2)
\end{align*}
and let 
$$G_1 = (V_1, E_1), \quad G_2 = (V_2, E_2).$$ 
Take the graph tensor product 
$$G = G_1 \times G_2 = (V,E).$$ 

 Take the start vertices in the graph product to be $s_0 = s_1s_2$, and accept vertices to be $A = \{a_1a_2 \mid a_1 \in A_1, a_2 \in A_2\}$. Then we claim that the language accepted by $$\bA = \cG\inv(V, E, X_1 \cap X_2, s_0, A)$$ is $$L_1 \cap L_2.$$ Indeed, observe that the set of paths from $s_0$ to $A$ in $G$ arise from a path whose pre-image under the graph product is a path from $s_1$ to $A_1$ and a path from $s_2$ to $A_2$. 
\end{proof}

\begin{figure}[h]
\centering
\includegraphics[width=\textwidth]{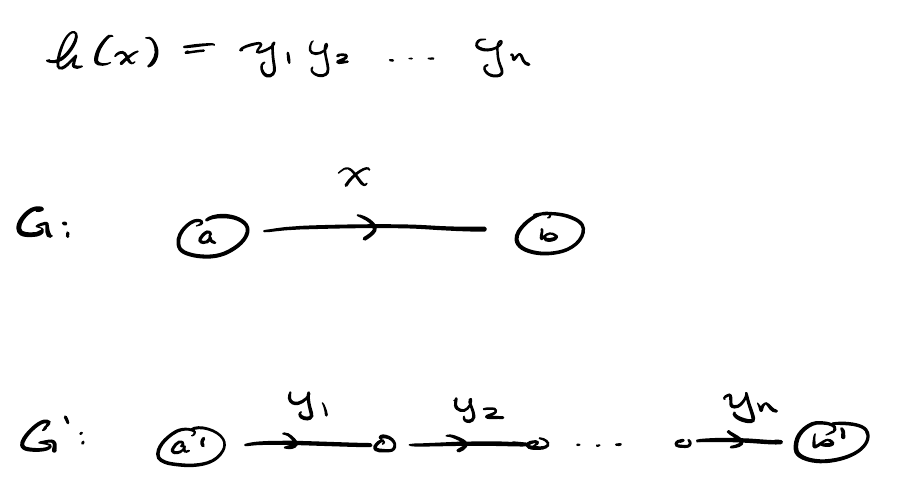}
\caption{Illustrating the path $p_{h(x)}$ in $G'$ induced by the homomorphism $h: X^* \to Y^*$ sending $x \mapsto y_1 \dots y_n$.}
\label{fig: fsa-hom}
\end{figure}

\begin{lem}[Closure under homomorphism]
Let $L$ be a regular language over $X$ and $h: X^* \to Y^*$ be a homomorphism. Then $h(L)$ is a regular language.
\end{lem}
\begin{proof}
We refer to Figure \ref{fig: fsa-hom}. Let $\bA$ be an automaton accepting $L$ and let $$G = \cG(\bA) = (V, E, X, s_0, A).$$ Let $G'$ be a copy of $G$ where every $x$-labeled edge $e$ is replaced with a path $p_{h(x)}$ with edge labelss forming the word $h(x) \in Y^*$, such that $p_{h(x)}$ has the same endpoints as $e$. Then every path inducing a word $w$ in $G$ is replaced by a path inducing $h(w)$ in $G'$. Therefore, $$\cG\inv(G', Y, s_0, A)$$ a finite state automaton accepting $$h(L).$$
\end{proof}

\begin{lem}[Closure under inverse homomorphism]
Let $L$ be a regular language over $X$ and let $h\inv$ be an inverse homomorphism. Then $h\inv(L)$ is a regular language.
\end{lem}
\begin{proof}
Let $\bA = (S, Y, \delta, s_0, A)$ be a finite state automaton accepting $L$. Let $$\bB = (S, X, \gamma, s_0, A)$$ be a finite state automaton such that $\gamma(s, x) = \delta(s, h(x))$. We argue that $\bB$ accepts $h\inv(L)$. Let $w$ such that $w \in h\inv(L)$. Then this is true if and only if $h(w) \in L \iff \delta(s_0, h(w)) \in A$, but $\delta(s_0, h(w)) = \gamma(s_0, w)$, so $\gamma(s_0,w) \in A$ also.
\end{proof}

This concludes the proof of Theorem \ref{thm: reg-lang-closure} on the closure properties of regular languages, which are shared by all classes pf formal languages listed in the Chomsky hierarchy. We mention an additional property which is particular to the class regular languages. 

\subsection{Closure under reversal}
As stated in Chapter \ref{chap: informal-lang}, all AFLs in the Chomsky hierarchy are closed under reversal. In the case of finite state automata, this is by observing that a FSA with the direction of its arrows reversed, (where a new start state is added for good measure) is still a FSA. 

\begin{lem}
	Let $L$ be a regular language. Then, the reversal of $L$, 
	$$L^R := \{x_n \dots x_1 \mid x_1 \dots x_n \in L\}$$ 
	is also a regular language. 
\end{lem}
\begin{proof}
	\begin{figure}[h]
	\centering
	\includegraphics[width=\textwidth]{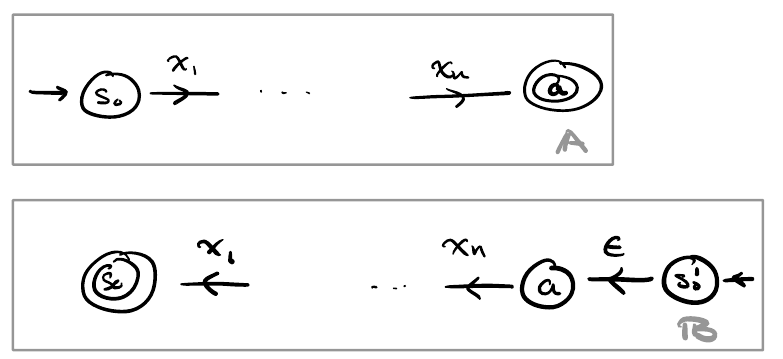}
	\caption{We illustrate reversing a path from $s_0$ to $a \in A$ in $\bA$ to a path from $s'_0$ to $s_0$ in $\bB$.}
	\label{fig: fsa-reversal}
	\end{figure}
	We refer to Figure \ref{fig: fsa-reversal}.
	Let $\bA = (S, X, \delta, s_0, A)$ and $$\cG(\bA) = (V, E, X, s_0, A)$$ be its associated graph. We want to reverse every path from $s_0$ to $A$ to a path from $A$ to $s_0$, keeping the same labels along the way. 
	
	 To do so, define $$G = (V \sqcup \{s'_0\}, E', X, s'_0, \{s_0\})$$ such that
		$$E' = \{(s_b, s_a, x) \mid (s_a, s_b, x) \in E\} \cup \{(s'_0, a, \epsilon) \mid a \in A\}$$
		The direction of the arrows in $G$ are reversed compared to those in $\cG(\bA)$, there is a new start state $s'_0$ creating an $\epsilon$-transition to all the accept states $A$ of $\bA$ (such that in case $|A| > 1$, we still end up with a single start state in $G$), and the new accept state is given by the old start state $s_0$. 
		
	Since every path from $s_0$ to $A$ in $\cG(\bA)$ is in one-to-one correspondence with a path from $A$ to $s_0$ in $G$, and thus in one-to-one correspondence with a path from $s'_0$ passing through $A$ going to $s_0$, and the arrows from $s'_0$ to $A$ are marked by the $\epsilon$-symbol, it is clear that $L^R$ is the language accepted by $\bB := \cG\inv(G)$. 
\end{proof}

\subsection{Closure under difference}
\begin{lem}[Closure under difference]
\label{lem: reg-closed-under-complement}
	Let $L$ and $M$ be regular languages over $X$. Then, 
	$$L - M$$
	is also a regular language.
\end{lem}
\begin{proof}
	Let $\bA_1$ and $\bA_2$ be finite state automata accepting $L$ and $M$ respectively, and w.l.og. assume that $\bA_2$ is deterministic and that its transition function $\delta_2$ is a full function (that is, we did not remove the dead states). 
	Let 
	\begin{align*}
		& G_1 = \cG(\bA_1) = (V_1, E_1, X, s^1_0, A_1) \\
		& G_2 = \cG(\bA_2) = (V_2, E_2, X, s^2_0, A_2).
	\end{align*}
	 
	Let 
	$$G = (V_1 \times V_2, E_1 \times E_2, X, s_0 = s^1_0 \times s^2_0, A_1 \times (V_2 - A_2).$$
	That is, the accepted paths are precisely the ones which start at $s_0^1 \times s_0^2$ and end in $A_1 \times (V_2 - A_2)$. They have for tensor product preimage the paths which are accepted in $G_1$ but not $G_2$ (note this only works when $\bA_2$ is deterministic $\delta_2$ being is a full function). Thus, $\cG\inv(G)$ is our desired automaton.  
\end{proof}

\begin{cor}[Closure under complement]
	Regular languages are closed under complements because $$L^c = X^* - L.$$
\end{cor}

\section{Regular expressions}
There is a straightforward way to write down languages accepted by finite state automata, known as \emph{regular expressions}. Regular expressions encode the closure properties of regular languages in their grammar, making them an efficient mean of expression as they can be expressed quite succinctly. 

\begin{defn}[Regular expressions]\label{defn: fsa-reg-expr}
	Let $X$ be a finite alphabet. The family of \emph{regular expressions} with alphabet $X$, $\Regex(X)$, is the least family of languages over $X$ that contains its \emph{atoms} and their closure under certain operations. More precisely, if $\cR := \Regex(X)$ for short, then 
	\begin{itemize}
	\item the empty language $\emptyset \in \cR$,
	\item the empty string language $\{\epsilon\} \in \cR$,
	\item for all $x \in X$, the language of literal characters $\{x\} \in \cR$. 
	\end{itemize}
	Moreover, $\cR$ must be closed under the following operations.  
	\begin{itemize}
		\item (Concatenation) For all $S,T \in \cR$, $ST \in \cR$.
 		\item (Union) For all $S, T \in \cR$, $S \cup T \in \cR$.  
		\item (Kleene star) For all $R \in \cR$, $R^* := \bigcup_{n=0}^\infty R^n \in \cR$. 
	\end{itemize}
\end{defn}
In other words, regular expressions are precisely the languages which you can construct using your starting alphabet letters as building blocks and iterating over them using the operations described. 

\begin{rmk}[Notation]
	For practical and historical reasons, it is common practice to omit the set brackets when talking about regular expressions. Moreover, union is also known as \emph{alternation} and the `|' character can be used in place of the union. Parentheses are used to specify order of operation. 
	
	For example, 
	\begin{itemize}
		\item $\{x\} \to x$,
		\item $\{x\} \cup \{y\}  \to x | y$,
		\item $(\{x\} \cup \{y\})^* \to (x | y)^*$. 
	\end{itemize}  
	This makes it easy to read and write regular expression in command-line. The cost of confusing languages consisting of a single characters and the characters themselves is not much as it is usually clear from context which type of object we are referring to. 
\end{rmk}

\begin{thm}[Kleene]\label{thm: fsa-Kleene}
	Let $\Reg(X)$ be the family of languages accepted by NFA over alphabet $X$, and $\Regex(X)$ be the family of regular expressions over $X$. Then
	$$\Reg(X) = \Regex(X).$$
\end{thm}

It is straightforward to see why $\Regex(X) \subseteq \Reg(X)$, the family of regular languages over $X$. Indeed, Figure \ref{fig: fsa-atoms} illustrate the finite state automata accepting the atomic languages of Definition \ref{defn: fsa-reg-expr}. Moreover, we have already done the work to show that $\Reg(X)$ is also closed under concatenation, union, and Kleene star. 

\begin{figure}[h]
\centering
\includegraphics[width=\textwidth]{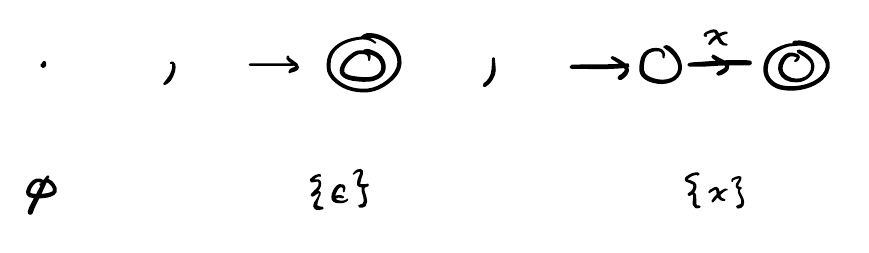}
\caption{Finite state automata accepting the empty language, the empty string, and a literal character respectively.}
\label{fig: fsa-atoms}
\end{figure}

Showing $\Reg(X) \subseteq \Regex(X)$  
from what we have so far is a bit more involved, but the idea is essentially that we can turn a finite state automaton into a regular expression by eliminating the states one by one and collapsing the labels of the incoming and outgoing arrows into regular expressions. The proof below is adapted from \cite{Viswanathan2013}. 

We start by converting finite state automata into a \emph{generalised non-deterministic finite state automata}, whose arrows are labeled by regular expressions. 

\begin{defn}
	A \emph{generalized non-deterministic finite state automaton} (GNFA) $\bG$ is given by a quintuple $\bG = (S, X, \rho, s_0, s_F)$ where
	\begin{itemize}
		\item $S$ is the set of states,
		\item $X$ is the finite state alphabet,
		\item $s_0 \in S$ is the initial state,
		\item $s_F \in S - \{s_0\}$ is a (single) final state, 
		\item $\rho : (S - \{s_F\}) \times (S - \{s_0\}) \to \Regex(X)$ is a map from a pair of states which are not the final and initial states respectively to the set of regular expressions over $X$. 
	\end{itemize}
	We say that $\bG$ accepts $w$ is there exists $x_1 \dots x_t \in X^*$ and states $r_0 \dots r_t$ such that 
	\begin{itemize}
		\item $w = x_1 s_2 \dots x_t$,
		\item $r_0 = s_0, r_t = s_F$, 
		\item $x_i \in \rho(r_{i-1}, r_i)$ for each $i \in \{1, \dots, t\}$. 
	\end{itemize}
\end{defn}

Essentially, a GNFA adds two features to a FSA. 
\begin{enumerate}
	\item There is a single finite state $s_F$ (and a single start state $s_0$). 
	\item The transition labels admit a regular expression given by $\rho$ instead of a single letter $x \in X$. 
\end{enumerate}

These two features will allow us to collapse a GFSA into a regular expression by eliminating one state at a time. 

\begin{figure}[h]
\centering
\includegraphics[width=\textwidth]{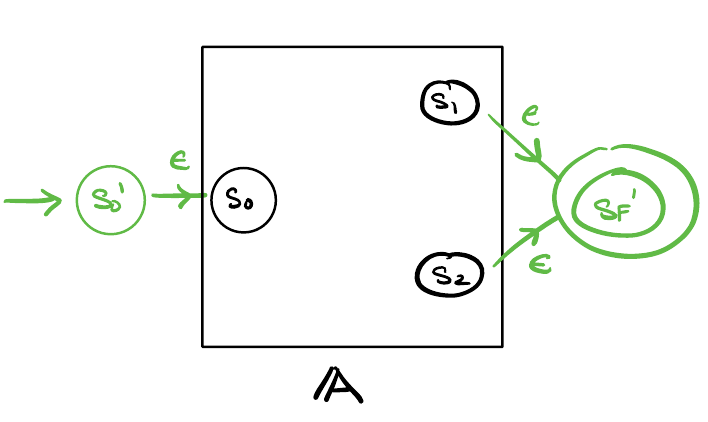}
\caption{Converting a finite state automaton $\bA$ (in black) into a GNFA (with added features in green) by adding a new start state $s_0'$, a new final state $s_F'$ and the appropriate $\epsilon$-transitions.}
\label{fig: fsa-gnfa}
\end{figure}

\begin{lem}\label{lem: fsa-gnfa}
	Any FSA can be converted into a GNFA accepting the same language. 
\end{lem}
\begin{proof}
	For any FSA $\bA = (S, X, \delta, s_0, A)$, it is always possible to convert $\bA$ into an GNSA $\bG = (S', X, \rho, s_0', s_F')$ by adding a start state and a final state using $\epsilon$-transitions such that the start state $s_0'$ and final state $s_F'$ are totally new and the regular expressions are given by the $\delta$-labels. That is, 
	\begin{itemize}
		\item $S' = S \sqcup \{s_0\}$, 
		\item $s_F' \cap A = \emptyset$, 
		\item $\rho(s_1, s_2) = \begin{cases}
			\epsilon & s_1 = s_0', s_2 = s_0 \\
			\epsilon & s_1 \in F, s_2 = s_F' \\
			\bigcup_{ \{x \mid \delta(s_1, x) = s_2 \}} x & \text{ otherwise.}
		\end{cases}$
	\end{itemize}
		This is illustrated in Figure \ref{fig: fsa-gnfa}. 
\end{proof}

Once we have our new GNFA $\bG$ from $\bA$ the goal is to stepwise delete all states that are not the new start or the new end states such that the only arrow between the start and end states of $\bG$ are labelled with the regular expression accepted by $\bA$.  

Let us formally define what we mean by deleting a state. 

\begin{defn}
	Given a GNFA $\bG = (S, X, \rho, s_0, s_F)$ with $|S| > 2$ and a chosen state $s \in S - \{s_0, s_F\}$, define GNFA $\rip(\bG, s) = (S', X, q_0, q_F, \rho')$ as 
	\begin{itemize}
		\item $S' = S - \{s\}$, 
		\item For any $(s_1, s_2) \in S' - \{s_F\} \times S' - \{s_0\}$ (possibly $s_1 = s_2$), let $$\rho'(s_1, s_2) = (R_1 R_2^* R_3) \cup R_4,$$ where $R_1 = \rho(s_1, s), R_2 = \rho(s,s), R_3 = \rho(s, s_2), R_4 = \rho(s_1,s_2)$. By definition of regular expressions $\rho'(s_1, s_2)$ is a regular expression if $R_1, R_2, R_3$ and $R_4$ are regular expressions. 
	\end{itemize}
\end{defn}

\begin{prop}\label{prop: fsa-rip}
	Let $\bG = (S, X, s_0, s_F, \rho)$ be a GNFA, and $s \in S - \{s_0, s_F\}$. Then, $$\cL(\bG) = \cL(\rip(\bG, s)).$$ 
\end{prop}
\begin{proof}
	Let $\bG' := \rip(\bG, s)$. 

	($\cL(\bG) \subseteq \cL(\bG')$). Since $w \in \cL(\bG)$, we can write $w = x_1 \dots x_t \in X^*$ such that $w$ passes through the state-path $p = \{s_0 = p_0, p_1, \dots, p_t = s_F\}$ in $\cG$ and $x_i \in \rho(s_{i-1}, s_i)$. 
	
	Define $q = \{s_0 = q_0, \dots, q_d = s_F\}$ to be the sequence obtained from deleting all occurrences of $s$ in $p$, and define 
	$$y_j = x_{\sigma(j-1) + 1} \dots x_{\sigma(j)}$$
	for $j = 0, \dots, d$, where $\sigma: [0,d] \to [0,t]$, 
	$$\sigma(j) = \begin{cases}
		0 & j = 0, \\
		i & 0 < \sigma(j-1) \leq t, \text{ where } i = \min_{i > \sigma(j-1)}(q_i \not= s).
		\end{cases}$$
	That is, we want to re-index $w$ via the $y_j$'s as to ``absorb'' the $s$-labels when when $q_i = s$, making $y_j$ a composition of multiple $x_i$-labels if necessary. \begin{itemize}
			\item If $\sigma(j) = \sigma(j-1) + 1$, then we are not ``absorbing'' any $s$-labels, and $y_j = x_{\sigma(j)}$. Setting $i := \sigma(j)$, we have $y_j = x_i$, and $q_{j-1} = p_{i-1}$, and $q_j = p_i$, such that $$y_j = x_i \in \rho(p_{i-1}, p_i) \subseteq \rho'(p_{i-1}, p_i) = \rho'(q_{i-1},q_i).$$ 
			\item On the other hand, if $\sigma(j) > \sigma(j-1)$, then $q_{\sigma(j-1)+1} = s$ and we are need to ``absorb'' those indices. In that case, $$y_j = x_{\sigma(j-1) + 1} x_{\sigma(j-1) + 2} \dots x_{\sigma(j)},$$ and $q_{j-1} = p_{\sigma(j - 1)}$ and $q_j = p_{\sigma(j)}$. Then, \begin{align*} y_j =  x_{\sigma(j-1) + 1} \dots x_{\sigma(j)} & \in \rho(p_{\sigma(j-1)}, p_{\sigma(j-1) + 1}) \rho(p_{\sigma(j-1) + 1}, p_{\sigma(j-1) + 2}) \dots \rho(p_{\sigma(j)-1}, p_{\sigma(j)}) \\ 
				&= \rho(p_{\sigma(j-1)}, s) \rho(s, s)^{i- (\sigma(j-1) + 1)}\rho(p_{\sigma(j)-1}, p_{\sigma(j)}) \\
				&\subseteq \rho(p_{\sigma(j-1)}, s) \rho(s, s)^* \rho(p_{\sigma(j)-1}, p_{\sigma(j)}) \\
				&= \rho(q_{j-1}, s) \rho(s, s)^* \rho(q_{j-1}, q_j) \\
				&\subseteq \rho'(q_{j-1}, q_j),
				\end{align*}
				by construction of $\rho'$. 
 		\end{itemize}
 	By construction, $w = x_1 \dots x_t = y_1 \dots y_d \in \cL(\bG')$, finishing the first part of the proof. 
 	
 	($\cL(\bG') \subseteq \cL(\bG)$). Since $w \in \cL(\bG')$, write $w = y_1 \dots y_d$ such that $w$ passes through the state path $q = \{s_0 = q_0, q_1, \dots, q_d = s_F\}$. For each $j \in [1, d]$, 
 	$$y_j \in \rho'(q_{j-1},q_j) = \rho(q_{j-1}, s)\rho(s,s)^*\rho(s, q_j) \cup \rho(q_{j-1}, q_j)$$
 	by definition of $\rho'$. The rest of this proof is essentially the reverse of the proof above, where we define $\sigma: [0, d] \to [0,t]$ to reassign the ``absorbed'' $s$-labels of $y_j$'s to the $x_i$'s. More precisely, 
 	$$\sigma(j) = \begin{cases}
 		0 & j = 0 \\
 		\sigma(j-1) + 1 & y_j \in \rho(q_{j-1}, q_j) \\
 		\sigma(j-1) + u + 2  & \text{otherwise}, \quad u := \min_u(y_j \in \rho(q_{j-1},s) \rho(s,s)^u \rho(s,q_j)). 
 	\end{cases}$$
 	Let $t := \sigma(d)$. For $i \in \{0, \dots, t\}$, define $p_i$ as follows. 
 	$$p_i = \begin{cases}
		q_j &\text{ if there exists $j$ such that $i = \sigma(j)$}, \\
		s & \text{ otherwise}. 
	\end{cases}$$
	\begin{itemize}
		\item If $i = \sigma(j)$ and $i-1 = \sigma(j-1)$, then we can write $y_i = x_i$. 
		\item For $\sigma(j-1) < i -1 < i \leq \sigma(j)$ for some $j \in [1,d] $, we have that $y_j \in \rho(q_{j-1}, s)\rho(s,s)^u \rho(s,q_j)$, where $u = \sigma(j) - \sigma(j-1) - 2$. Therefore, we can write $y_j = x_{\sigma(j-1) + 1} \dots x_{\sigma(j)}$, such that $$x_{\sigma(j-1)+1} \in \rho(q_{j-1},s), \qquad x_{\sigma(j)} \in \rho(s,q_j), \qquad x_{\sigma(j-1)}\dots x_{\sigma(j)-1} \in \rho(s,s).$$
	\end{itemize}
	This concludes the proof that $w = y_1 \dots y_d = x_1 \dots x_t \in \cL(\bG)$. 
	Thus, $\cL(\bG) = \cL(\bG')$ as claimed. 
\end{proof}

We are ready to prove Kleene's theorem. 

\begin{proof}[Proof of Theorem \ref{thm: fsa-Kleene}]
	We only need to show that $\Reg(X) \subseteq \Regex(X)$. Let $L \in \Reg(X)$, that is, $L$ is accepted by NFA $\bA$ with $k$ states. Then, by lemma \ref{lem: fsa-gnfa}, there exists GFA $\bG$ accepting $L$ with $k + 2$ states, $S = \{s_0, s_1, \dots, s_k, s_F\}$. 
	
	Let $\bG_0 := \bG$, $$\bG_{i+1} := \rip(\bG_i, s_{i+1})$$ for $i = 0, \dots, k-1$. Then, $\bG_k$ has states $\{s_0, s_F\}$, and by Proposition \ref{prop: fsa-rip}, $\cL(\bG) = \cL(\bG_k)$. 
	
	Thus, $$L = \cL(\bG) = \cL(\bG_k) = \rho(s_0, s_F) \in \Regex(X).$$ 
\end{proof}

\section{Languages which are not regular}
\subsection{Intuition}
Given the structural limitations of finite state automata compared to Turing machines, it is reasonable to expect that many formal languages are not regular. 

But how do we prove that a language is not regular? Intuitively, one of the most salient aspects of finite state automata is that they have finitely many states, which we think of as having finite memory. 

An example of a language which requires infinite memory is $$L = \{0^n 1^n \mid n \in \bN\}.$$ To decide whether a binary string belongs to this language, an automaton must first keep track of the number of zeros encountered, then check that the number of ones encountered is the same as that number. Note that for any \emph{fixed} arbitrary number of zeros matching ones, this is possible to decide with a finite state automaton. 

\begin{ex}
The language $L = \{0^n 1^n\}$ is regular for any $n \in \bN$ by Figure \ref{fig: fsa-0n-1n}. 
\end{ex}

\begin{figure}[h]{
\includegraphics{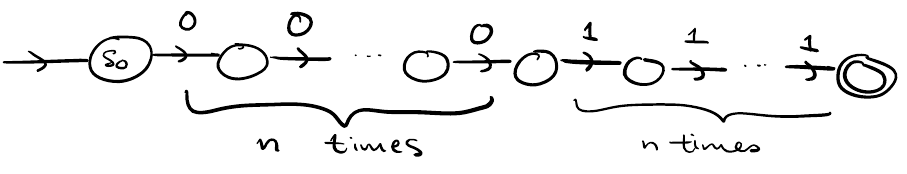}}
\caption{Illustration of creating a finite state automaton which accepts the union of two regular languages}
\label{fig: fsa-0n-1n}
\end{figure}

\begin{ex}
The language $L = \{0^n 1^n \mid n \leq N\}$ for some $N \in \bN$ is regular by taking the finite union of regular languages $L = \bigcup_{n=1}^N \{0^n1^n\}$.
\end{ex}

Intuitively is the fact that $N$ needs to grow to infinity that makes $L = \{0^n 1^n \mid n \in \bN \}$ not a regular language. 

\subsection{Pumping lemma}

Our intuition about finite memory can be formalised into a version of the pidgeonhole principe for finite state automata, called the pumping lemma.

\begin{figure}[h]{
\includegraphics{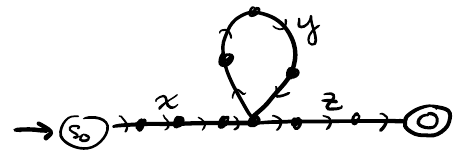}}
\caption{Part of a finite state automaton accepting the word $w = xyz$ as in the pumping lemma. Only the start and accept states are represented as circles. Moreover, only the states which are visited by $w$ are illustrated.}
\label{fig: fsa-pumping}
\end{figure}

The key idea of the pumping lemma is that on a graph with finitely many vertices, any path which visits more than the total number of vertices must repeat a vertex and therefore contain a loop. If we functor the  graph into a finite state automaton, this means that any sufficiently long word arises from a path containing a loop. We refer to Figure \ref{fig: fsa-pumping} for the following. Let $w = xyz$ be a sufficiently long  word such that $x$ is the subword before the loop, $y$ is the loop subword, and $z$ is the subword after the loop. Then $xy^kz$ has the same start and end points as $xyz$, as we are only repeating the $y$-part $k$-times. Therefore, if $xyz$ is accepted, so is $xy^kz$. The ``pumping'' in the pumping lemma refers to the process of passing from $xyz$ to $xy^kz$.

\begin{defn}
Let $w \in X^*$ be a word. We define the \emph{word length} to be the number of non-empty characters needed to spell $w$. That is, $|w| := n$ such that $x_1 \dots x_n = w, \quad x_i \in X - \{\epsilon\}, 1 \leq i \leq n$.  
\end{defn}

\begin{lem}[The pumping lemma for regular languages]\label{lem: fsa-pumping}
Let $L$ be a regular language. There exists a constant $n$ depending on $L$ such that for every word $w \in L$ such that $|w| \geq n$, we can break $w$ into $w = xyz$ where 
\begin{enumerate}
	\item $y$ is not the empty word,
	\item $|xy| \leq n$,
	\item for all $k \geq 0$, the word $xy^kz$ is also in $L$. 
	\end{enumerate}
\end{lem}

\begin{figure}[h]{
\includegraphics{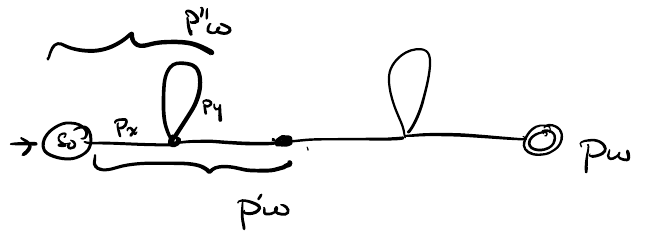}}
\caption{Illustrating the paths in the proof of the pumping lemma.}
\label{fig: fsa-pumping-pf}
\end{figure}

\begin{proof}
We refer to Figure \ref{fig: fsa-pumping-pf}. Let $\cG$ be the graph functor as in Definition \ref{defn: fsa-graph-functor}.  
Let $\bA$ be a finite state automaton accepting $L$ with $n$ states. Let $\cG(\bA) = (V,E,X, s_0, A)$ the associated graph. Then if $w \in L$ such that $|w| \geq n$, then $w$ induces a path $p_w$ of length $\geq n$. Let $p'_w$ be the initial segment of path $p_w$ of length $n$. Then $p'_w$ must pass through $n+1$ vertices of $V$. By the pigeonhole principle, this means that at least one vertex must be visited at least twice inside $p'_w$. Let $p''_w$ be the initial segment of $p'_w$ which starts at $s_0$ and concludes a loop around some vertex $v \in V$. Let $xy$ be the word induced by the edge labels of $p''_w$. Then $|xy| = |p''_w| \leq |p'_w| = n$. Suppose that $p''_w = p_x p_y$, where $p_x$ is the acyclic path from $s_0$ to the cycle $p_y$. Let $p_z$ be such that $p_x p_y p_z = p_w$. Then since $p_w$ is a path from  start vertex to an accept vertex, and $p_y$ is a cycle, $p_x p_y^k p_z$ for any $k \in \bN$ is also a path from start vertex to accept vertex since it has the same endpoints as $p_w$. Therefore, $xy^kz$ is accepted by $\bA$. 
\end{proof}

Let us do an example using the pumping lemma to illustrate how it can be used. In practice, the details of the decomposition of a word $w = xyz$ are unknown, but we can deduce them using the information we have about the language. Once this is found, we ``pump'' $w$ to obtain a contradiction about the language. 

\begin{ex}\label{ex: fsa-pumping}
Assume for contradiction that $$L = \{0^m 1^m \mid m \in \bN\}$$ is a regular language. Take $n$ as in the statement of the pumping lemma, and select a specific $m$ such that $m > n$. Let $w = 0^m 1^m$. Then $|w| = 2m > n$, and therefore we can break $w$ into $w = xyz$ where $y$ is non-empty, $|xy| \leq n$ and $xy^kz \in L$ for any $k \geq 0$. 

Although we do not know exactly what $x,y$ and $z$ are, since we have that $|xy| \leq n$ and $m > n$, this necessarily means that $xy = 0^p$ for some $p \leq n$. Therefore, $y = 0^q$ for some $q \leq n$. Moreover, $q > 1$ since $y$ is non-empty. 

We know that $w = xyz \in L$ and that $xy^2z \in L$ by the pumping lemma. This leads conclusion that $0^{m + q}1^m$ must be in $L$, for $1 < q \leq n$. We have a contradiction.  
\end{ex}

The intuition behind the example we just did hinges on the fact that a finite state automaton cannot possibly remember by its finite states how many $0$'s we have seen, since there may be an arbitrarily large number of them before we get to read $1$'s. 

\section{Geometrical interpretation of regular positive cones}\label{sec: fsa-reg-P-geom}
When I first learned of all the material above, I was confused and thought I was still far away from being able to do meaningful mathematical progress. I would like to reassure the discouraged reader that they have come very far already by presenting next a beautiful geometrical result about groups which have positive cones which can be represented by a regular language. 

That is, if $P$ is a positive cone in $G = \langle X \mid R \rangle$ with evaluation map $\pi: X^* \to G$ and $L \subseteq X^*$ is a regular language such that $\pi(L) = P$, then $P$ is a positive cone represented by a regular language $L$, which we also call a \emph{regular positive cone}.\sidenote{Let us consider this a preview: we will repeat this material and develop on it later in Chapter \ref{chap: research-intro}.}

A priori, having a regular positive cone seems like a purely computational property. This is not the case. Alonso, Antol\'in, Brum and Rivas in 2020 showed that if $P$ is a subset with a regular language representation, then $P$ is coarsely connected.\sidenote{See Chapter \ref{chap: hyperbolic} for an overview on coarsenes.}

\begin{defn}\label{defn: coarsely connected}
Let $(M,\dist)$ be a metric space.  A subset $Y\subseteq M$ is {\it coarsely connected } if there is $R>0$ such that $\{p\in M\mid \dist(p,Y)\leq R\}$, the $R$-neighborhood of $Y$, is connected.
 \end{defn} 

\begin{defn} A set is $P \subseteq G$ is \emph{coarsely connected} if it is connected in the Cayley graph up to some $R$-neighbourhood, for $R \geq 0$. 
\end{defn}
 
\begin{lem}\label{lem: regular-implies-coarsely-connected}
Let $G$ be a finitely generated group. If $P$ is a regular positive cone of $G$, then $P$ and $P^{-1}$ are coarsely connected subsets of the Cayley graph of $G$. \cite{AlonsoAntolinBrumRivas2022}.
\end{lem}

\begin{proof}	Let $L$ be a regular language such that there exists an evaluation map $\pi: X^* \to G$ with $\pi(L) = P$. Let $\bA$ be a finite state automaton which accepts $L$, and let $\Gamma(G,X)$ be the Cayley graph of $G$ with respect to generating set $X$. Let $w \in L$ such that $w = x_1\dots x_n$ and let $w_i = x_1 \dots x_i$ be the prefixes of $w$. 
	
	We will first show that for each pair of vertices corresponding to consecutive prefixes, $\pi(w_i), \pi(w_{i+1})$ for $i = 1, \dots, n-1$, there exists a path of length $\leq R$ in the Cayley graph with endpoints in $P$ connecting the two vertices and use this to then show that $P$ is coarsely connected. 
	
	Let $S$ be the set of states for the automaton $\bA$ accepting $L$. Let $w_i$ be a prefix such that $i \in 1, \dots, n-1$. Then $w_i$ starts at a start state in $\bA$ and ends at some state $s \in S$. Set $u_i = x_{i+1} \dots x_n$ if $|x_{i+1} \dots x_n| \leq |S|$ and otherwise take $u_i$ to be a word corresponding to the acyclic path made by $x_{i+1} \dots x_n$ in $\bA$. Since such a path does not go through the same state twice, forcibly $|u_i| \leq |S|$. In either case, we have $w_i u_i \in L$ with $|u_i| \leq S$. This is illustrated in Figure \ref{fig: coarsely-connected-1}.
			
\begin{figure}[h]{\includegraphics{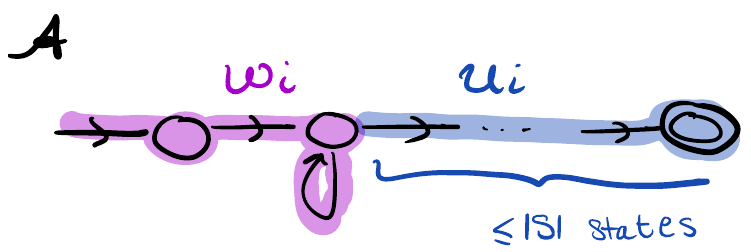}}
\caption{Illustrating how the prefix $w_i$ (with induced path highlighted in purple) and suffix $u_i$ (with induced path highlighted in blue) would look in the automaton $\bA$. The concatenation $w_iu_i$ is an accepted word. The suffix $u_i$ can always be chosen to have less than the number of states in $\bA$.}
\label{fig: coarsely-connected-1}
\end{figure}

Consider now two consecutive prefixes, $w_i, w_{i+1}$ with $i \in 1,\dots, n-1$ and look at their induced paths, now in the Cayley graph instead of the automaton. First observe that since $w_i u_i \in L$, $\pi(w_i u_i) \in P$. Therefore, starting from the vertex $\pi(w_i u_i)$, the path induced by $u_i\inv x_i u_{i+1}$ is a path from $P$ to $P$ ending at $\pi(w_i x_i u_{i+1})$, and $|u_i\inv x_i u_{i+1}| \leq 2|S| +1$. This is illustrated in Figure \ref{fig: coarsely-connected-2}.

\begin{figure}[h]{\includegraphics{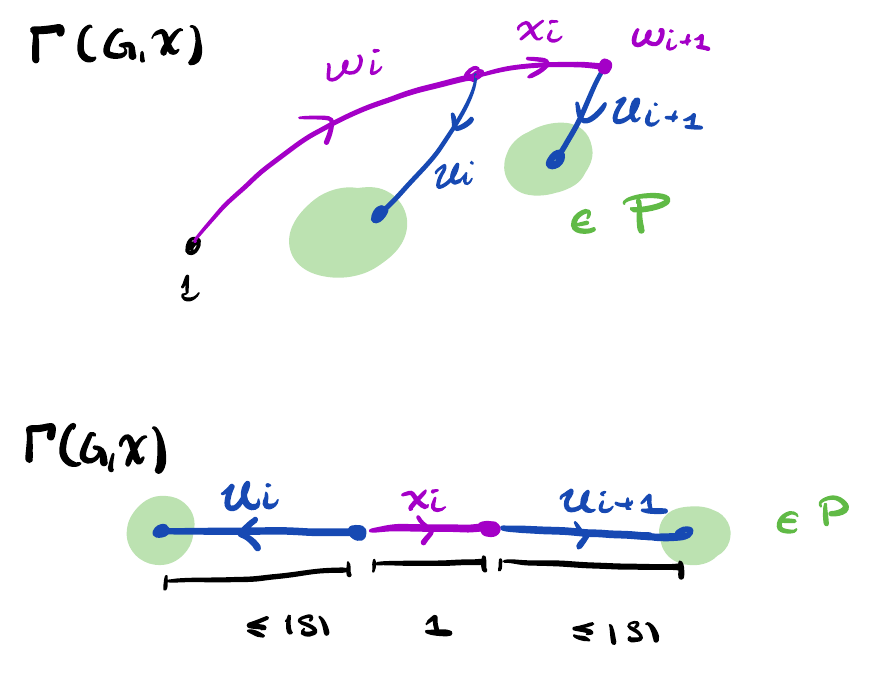}}
\caption{The drawing on top illustrates how $w_i$ (in purple), and $u_i$ (in blue) look as induced paths in the Cayley graph. Highlighted in green are the vertices which belong to the positive cone $P$. The one on the bottom illustrates a path from $P$ to $P$ with bounded length that we will use in the sequel of our proof.}
\label{fig: coarsely-connected-2}
\end{figure}	

Finally, take two elements $p,q \in P$. Our goal is to show that $p$ and $q$ can be connected in $\Gamma(G,X)$ by using paths from $P$ to $P$ of bounded length $R$, as this would show that the ball of radius $R$ around $P$, $B_R(P)$ is connected. Since $\pi(L) = P$, there exists $w, v$ such that $\pi(w) = p$ and $\pi(v) = q$. Set $v = y_1 \dots y_m$, with prefixes $v_i = y_1\dots y_i$ and $t_i$ such that $v_i t_i \in L$ for $i \in 1, \dots, m-1$ and $|t_i| \leq |S|$.  Then, starting from $p$ and going to the $\pi(w_1u_1)$, there exists paths induced by elements of the form $u_{i+1}\inv x_{i+1}\inv u_i$ for $i \in n-1, \dots, 1$ of length $\leq 2|S| +1$. Going from $\pi(w_1 u_1)$ to $\pi(v_1 t_1)$ is done by using the path induced by $u_1\inv x_1\inv y_1 t_1$ which is of length $\leq 2|S| + 2$. Then starting from $\pi(v_1t_1)$ and going to $q$, there exists paths induced by elements of the form $t_i\inv y_i t_{i+1}$ of length $\leq 2|S| + 1$. This is illustrated in Figure \ref{fig: coarsely-connected-3}. Thus, setting $R = 2|S| + 2$ gives us that $P$ is coarsely connected with coefficient $R$. 

Finally, notice that $P$ is regular if and only if $P\inv$ is. 

\begin{figure}[h]{\includegraphics{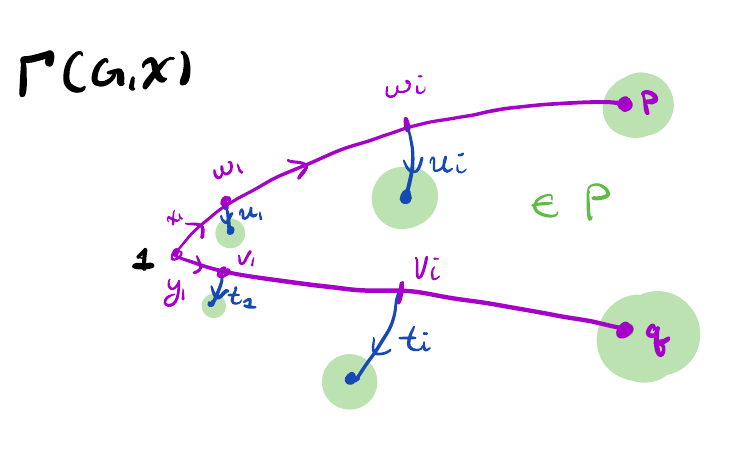}}
\caption{Illustrating how $p$ and $q$ can be connected to one another using subpaths with endpoints in $P$ of length $\leq 2|S| + 1$. Highlighted in green are the elements which belong to $P$.}
\label{fig: coarsely-connected-3}
\end{figure}		
\end{proof}

\begin{rmk}
	Note that the claim of this lemma is already apparent, albeit not explicitly stated, in the proof of Lemma 3 in \cite{HermillerSunic2017NoPC}. Namely, within the proof of Lemma 3, it is shown that if $L$ is a regular language for a positive cone and $w \in L$, then for every prefix $w'$ of $w$, the element represented by $w'$ is within bounded distance from some positive element (the bound being the number of the states of the automaton accepting $L$).
\end{rmk}

We will later see in Proposition \ref{prop: pcl-comp-indep-gen-set} that given a fixed positive cone $P$, its language complexity is independent choice of finite generating set, so the statement above is true for any Cayley graph.

%% file: chap/pda.tex
\chapter{Pushdown automata}\label{chap: pushdown-automata}

In this chapter, we introduce pushdown automata first informally, then formally along with their closure properties. At the end of the chapter, we use what we learned to prove a result concerning positive cone complexity which may be new to the literature. 

\section{Informal definition}

Pushdown automata sit one complexity level higher than finite state automata in the Chomsky hierarchy. A pushdown automaton is essentially a finite state automaton upgraded with a stack capable of non-determinism. A stack is an abstract data structure which may store an arbitrarily large number of symbols, with the some caveats on accessing that infinite memory, which we will use Figure \ref{fig: pda-stack} to illustrate. 

\begin{figure}[h]{{}
\includegraphics[width=\textwidth]{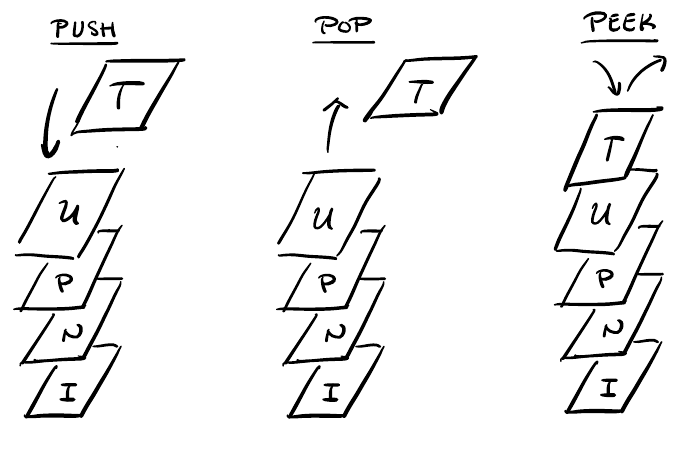}}
\caption{Pictorial representation of how the \emph{stack input}, here the string ``INPUT'', is stored with the first letter at the bottom of the stack: storing a new symbol is called \emph{pushing}, discarding the top symbol is \emph{popping}, pushing and popping the same symbol (thought of as taking the symbol out of the stack to look at it then putting it back) is called \emph{peeking}.}
\label{fig: pda-stack}
\end{figure}

A stack stores its information as follows. 
\begin{itemize}
	\item The input symbols of a stack are stacked as a pile (such as a deck of cards).
	\item The most recently inputted symbol goes on top of the pile (this is called \emph{pushing} a new symbol.
	\item We may only access the symbol at the top of the pile, which we then discard (called \emph{popping}).
\end{itemize}

Combined with the action of storing the a new symbol at the top of pile, we may obtain two additional operations. 
\begin{itemize}
	\item First, we may replace the top symbol with a different symbol by popping the symbol at the top of the stack then pushing a new symbol.
	\item Second, we may simply look at the symbol at the top of the stack (called \emph{peeking}) by discarding the symbol at the top of the stack, then adding back that symbol.
\end{itemize}
  Since the number of symbols stored in this way can be arbitrary large, we often say that pushdown automata have infinite memory (unlike finite state automata).\sidenote{Note that Turing machines also have infinite memory in the sense that the number of symbols stored can be arbitrarily, countably, large. The reason why Turing machines are more general than pushdown automata is because unlike with pushdown automata, the rules under which the tape storing the memory operates are a lot less restrictive since we may move the tape left or right. In other words, we do not only have access to one end of the tape. We encourage the reader unfamiliar with these concepts to go back to Chapter \ref{chap: informal-lang} to compare pushdown automata and Turing machines.}

When a stack is added to a finite state automaton, it forms a pushdown automaton, which we illustrate in Figure \ref{fig: pda-arrow}. More specifically, the new transition function's input now includes by what's on top of the stack, and has as output either pushing a new symbol or popping the one on top of the stack. In other words, the stack is equipped with a \emph{stack alphabet}, and the transition function's domain is now over the set of states, the finite state automaton alphabet, and the stack alphabet. 

\begin{figure}[h]{
\includegraphics[width=\textwidth]{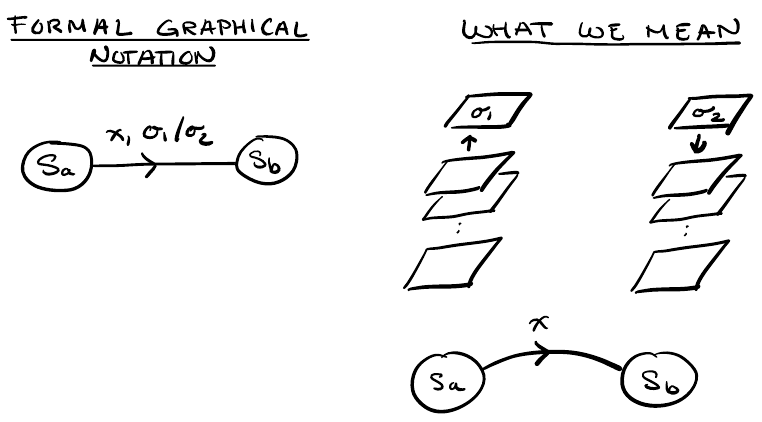}}
\caption[][4\baselineskip]{Part of a pushdown automaton. The labels for the input $x$ and the stack symbols are separated by a comma. The pop stack symbol and the pushed stack symbol are separated by a slash.}
\label{fig: pda-arrow}
\end{figure}

To encode this new information graphically, we add extra labels to the directed edges encoding the transition function. In Figure \ref{fig: pda-arrow}, we have a transition from state $s_a$ to $s_b$ on input $x$ if the top of the stack is the stack symbol $\sigma_1$. On going to $s_b$, $\sigma_1$ is replaced by $\sigma_2$ as the symbol on top of the stack. In computer science parlance, $\sigma_1$ is ``popped'' out of the stack, and $\sigma_2$ is ``pushed'' into the stack. The symbols $\sigma_1, \sigma_2$ belong in a new alphabet associated to the stack. 

\section{Formal definition}

Let us define this new structure formally. 

\begin{defn}
A {\it (non-deterministic) pushdown automaton} (PDA) is a 7-tuple $$\bA=(S,X, \Sigma, \delta,s_0, \Sigma_0,A),$$ where 
\begin{itemize}
	\item $S$ is a finite set whose elements are called {\it states},
	\item $A$ is a subset of $S$ whose states are called {\it accepting states},
	\item $s_0\in S$ is a distinguished element called {\it initial state},
	\item $X$ is a finite set called the {\it input alphabet},
	\item $\Sigma$ is a finite set called the {\it stack alphabet},
	\item  $\Sigma_0$ is a distinguished symbol in $\Sigma$ called the {\it stack start symbol},
	\item and $\delta$ is  a (non-deterministic) transition function
\end{itemize}$$\delta\colon  S \times (X\sqcup \{\epsilon\}) \times \Sigma \to  \cP(S \times \Sigma^*).$$
\end{defn}

The intended meaning of the transition function is that the transition function takes as input the current state $s_1$, an input symbol $x$, and the current top symbol on the stack $\sigma_1$, and outputs a set of possible outcomes which are the new current state, and the new top of the stack. The new top of the stack replaces the old top of the stack symbol. For example, $$\delta(s_1, x, \sigma) \ni (s_2, u)$$ means that if we are in state $s_1$ with input $x$ and top of the stack $\sigma$, then we transition to state $s_2$ with top of the stack $u \in \Sigma^*$. 

There are many interpretations for $u$. 

\begin{itemize}
	\item If $u = \epsilon$ it means that we have popped (deleted) the top symbol $\sigma$ from the stack.
	\item  If $u = \sigma$, it means we have peeked at the stack.
	\item  If $u$ is a word of length greater than $1$ which starts with $\sigma$, it means that we have pushed stack symbols on top of $\sigma$. 
	\item Finally, if $u$ is a word of length greater or equal to $1$ which does not start with $\sigma$, it means that we have replaced $\sigma$ with the contents of $u$, where the letters are stacked such that the first letter of $u$ replaces $\sigma$ on top of the stack, and the subsequent letters are placed on top of $\sigma$ such that the very last letter of $u$ is on top of the stack at the end of the move.  
\end{itemize}

\begin{defn}
A language $L$ is {\it context-free} if there is a (non-deterministic) pushdown automaton which accepts it. I.e, there exists a PDA $\bA = (S,X, \Sigma, \delta,s_0, \Sigma_0,A)$ such that 
$$L=\cL(\bA):=\{w\in X^* \mid  \exists (s,u) \in \delta(s_0,w, \Sigma_0) \text{ with } s\in A\}.$$
\end{defn}

\begin{figure}[h]{
\includegraphics{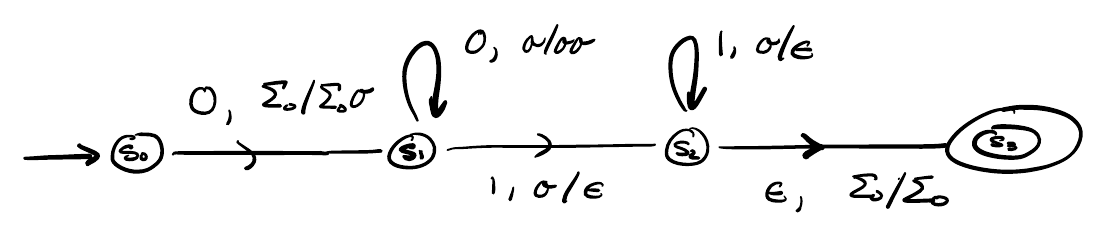}}
\caption{An example of a pushdown automaton accepting the language $\{0^n 1^n \mid n \in \bN\}$.}
\label{fig: pda-0n-1n}
\end{figure}

\begin{ex}\label{ex: pda-0n-1n}
We state without  formal proof that Figure \ref{fig: pda-0n-1n} illustrates a pushdown automaton accepting the language $$L = \{0^n 1^n \mid n \geq 1\},$$ which we have shown in the previous chapter not to be a regular language.

Intuitively, the stack has only the start symbol $\Sigma_0$ at the beginning, and $\sigma$ is added on top of the stack each time $0$ is read. Then, each time $1$ is read, the top stack symbol $\sigma$ gets replaced by the empty word $\epsilon$ (in other words, $\sigma$ is ``popped'' from the stack), until the stack empties to have only the start stack symbol at the top again again. This way, the only words accepted are of the form $0^n 1^n$ for $n \geq 1$.  

We will prove this formally later in Example \ref{ex: pda-0n-1n-pf} and provide its associated context-free grammar in Example \ref{ex: pda-0n-1n-CFG}.  
\end{ex}

\subsection{Non-determinism as default}

Note that the pushdown automaton is assumed to be non-deterministic by default as the image of the transition function is a set of a state and stack word tuples. This means that the stack is assumed to be capable of simultaneously being in many different configurations at once, with the automaton accepting words which induce at least one path which leads to an accept state.\sidenote[][]{The way I think about this concretely (say, if it were to be simulated on a modern computer using the finite state and stack data structures and without resorting to parallelism) is that the stack follows one path at once and assumes the configuration related to that path. However, if the path does not lead to an accept state, then the stack backtracks to the next possible path leading to an accept state and so on. This way of thinking helps me go around the fact that I do not need multiple stacks to keep track of all the possible paths for each words, as it would appear I do when I take the notion of the stack being in simultaneously different configurations at once too literally. The need for the backtracking, however, is what makes this pushdown automaton structure non-deterministic. It is known that there is no way of converting a non-deterministic pushdown automaton into a deterministic one.}

To make a pushdown automaton deterministic, we need to restrict the number of possible outcomes given a state, input symbol and stack symbol combination to at most one. Example \ref{ex: pda-0n-1n} is such an example of a deterministic pushdown automaton.

Let's make precise this idea and define formally what it means to be a deterministic pushdown automaton. 

\begin{defn}
A pushdown automaton $\bA = (S, X, \Sigma, \delta, \Sigma_0, A)$ is \emph{deterministic} if it satisfies both of the following conditions: 

\begin{enumerate}
\item For any state $s \in S$, any input symbol $x \in X \sqcup \{\epsilon\}$, and any stack symbol on top of the stack $\sigma \in \Sigma \cup \{\epsilon\}$, the cardinality of the output subset is $|\delta(s, x, \sigma)| \leq 1$, i.e. that there is only one possible new state and new top of the stack. 
\item For any state $s \in S$ and stack symbol $\sigma \in \Sigma$, if the set of $\epsilon$-transitions $\delta(s, \epsilon, \sigma) \not= \emptyset$, then the set of non-$\epsilon$-transitions $\delta(s, x, \sigma) = \emptyset$ for every $x \in X$. 
\end{enumerate}
\end{defn}
Condition 2 here serves to formally rule out the possibility of having two different transitions, one an $\epsilon$-transition and one non-$\epsilon$-transition, for the same state and stack configuration as to really ensure that there is truly only one possible outcome per input for the transition function.

We have really taken the time to clarify what it means for a pushdown automaton to be non-deterministic because the class of languages accepted by deterministic pushdown automata is a strict subset of the class of languages accepted by pushdown automata. This is important to remember when thinking about context-free languages. 

Let us finish this section with an example of a context-free language which is \emph{not} accepted by a deterministic finite state automaton. 

\begin{figure}[h]{
\includegraphics{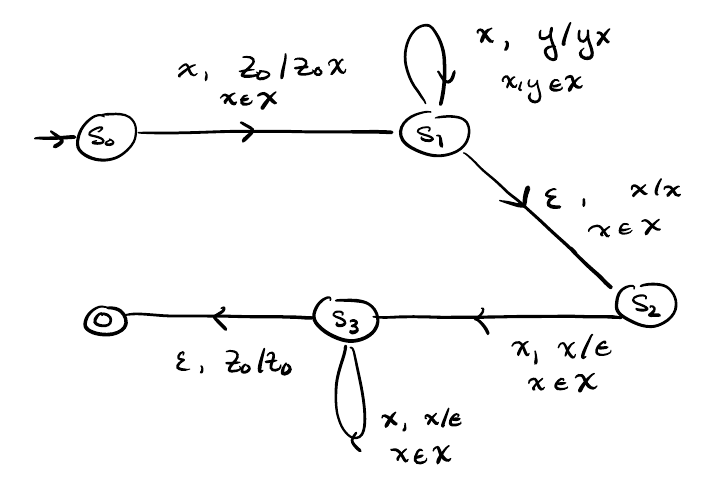}}
\caption{A pushdown automaton accepting non-trivial palindromes. The arrows have been compressed to one for each character $x \in X$ in the interest of space. Each arrow is in fact many arrows leading to the same state.}
\label{fig: pda-palindrome}
\end{figure}

\begin{ex}[Palindromes]
\label{ex: pda-palin}
Let $X$ be an alphabet, and $\cdot^R: X^* \to X^*$ be the reversal map $(x_1 \dots x_n)^R = x_n \dots x_1$ where $x_i \in X$ for $1 \leq i \leq n$. Then the language $$L = \{w w^R \mid w \in (X^* - \{\epsilon\}) \}$$ of non-trivial palindromes is accepted by a pushdown automaton with stack alphabet $\Sigma = X \cup \{\Sigma_0\}$. 

We state without formal proof that Figure \ref{fig: pda-palindrome} is the pushdown automaton accepting $L$. Let us give an intuitive explanation instead. At the start, each character $x \in X$ is pushed into the stack. When the automaton detects the ``middle'' of the palindrome it goes from $s_1$ to $s_2$, which has an $\epsilon$-transition. From $s_2$ on, the automaton expects to encounter, one-by-one in reverse order the letters it has encountered before. This is done by expecting to encounter the same letter as the letter on top of the stack, and pushing out that letter when that happens. The automaton proceeds as follows until the stack is empty, which is represented by accepting when the start stack symbol is at the top of the stack. 

Note that the automaton cannot truly detect when $w^R$ starts. By having a $\epsilon$-transition to separate $w$ from $w^R$, we are saying that after each character of $w$, we ``guess'' that $w^R$ is starting via the $\epsilon$-transition. Assuming that the input word is indeed a palindrome, only one ``guess'' will be correct, but the $\epsilon$-transition means that every guess is taken and the correct one is accepted.

\end{ex}

\begin{rmk} Example \ref{ex: pda-palin} illustrates why making the distinction between non-deterministic pushdown automata (also referred as simply pushdown automata) and deterministic pushdown automata is necessary. A deterministic pushdown automata does not have the stack capability to ``guess'' the middle a word, since each transition may only lead to a unique outcome. As stated in Chapter \ref{chap: informal-lang}, the class of languages accepted by deterministic pushdown automata is somewhere in between regular and context-free. 	
\end{rmk}

\begin{rmk}[Proving non-determinism]\label{rmk: pda-non-det}
I do not have a formal proof as to why the language of palindromes is a non-deterministic context-free language, although it appears to be a well-known fact. One possible way would be to use the pumping lemma for deterministic context-free languages (see \cite{Yu1989}). Another possible way of obtaining proving that a language is non-deterministic context-free is using Ogden's lemma\sidenote{See \url{https://en.wikipedia.org/wiki/Ogden\%27s_lemma}.} to show that a language is inherently ambiguous\sidenote{A concept that pertains to the grammar point of view of formal languages.}, which implies it cannot be deterministic context-free. %
\end{rmk}

\section{Acceptance by empty stack}

We will now introduce a neat concept which will facilitate working with pushdown automata, both conceptually and for computational purposes. 

\begin{lem}[Acceptance by empty stack]\label{lem: pda-empty} Let $L$ be a context-free language. Then $L$ can be accepted by a pushdown automaton that has the property that for every accepted word $w \in L$, $w$ leaves the stack empty.\end{lem}

\begin{figure}[h]{
\includegraphics{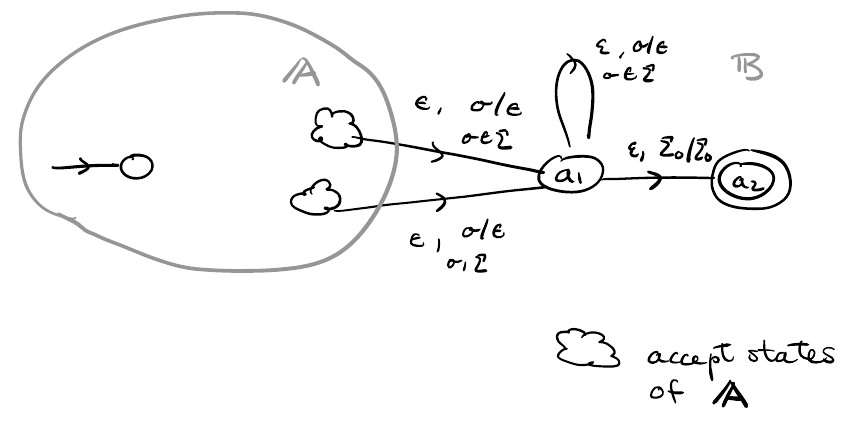}}
\caption{Two pushdown automata are represented. $\bA$ is a pushdown automaton without additional assumptions, and $\bB$ is a pushdown automaton deriving from $\bA$ which accepts the same language as $\bA$, but such that the stack is empty once each word is accepted. The arrows have been condensed for brevity, where $\sigma \in \Sigma$ refers to any stack symbol.}
\label{fig: pda-empty-stack}
\end{figure}

\begin{proof}
We refer to Figure \ref{fig: pda-empty-stack} for illustration. Let $\bA=(S,X, \Sigma, \delta,s_0, \Sigma_0,A)$ be any pushdown automaton which accepts $L$. We are going to create a pushdown automaton $\bB = (S', X, \Sigma, \delta', s_0, \Sigma_0, A')$ which accepts the same words as $\bA$, but by empty stack. To do so, we will add two extra states $a_1$ and $a_2$ to $\bB$ which transitions all the accept states of $\bA$ and empties the stack while transitioning, reaching $a_2$ when the stack is empty. 

Formally, the new states are $S' = S \cup \{a_1, a_2\}$ where the new accept state is $A' = \{a_2\}$, and the new transition function works as $\delta'(s, x, \sigma) = \delta(s, x, \sigma)$ for all $s \in S$, $x \in X$ and $\sigma \in \Sigma$. In other words, it stays unchanged on the states in the previous automaton. The change in the transition happens as follows: for all $\sigma \in \Sigma$, $\delta'(a, \epsilon, \sigma) = (a_1, \epsilon)$, $\delta'(a_1, \epsilon, \sigma) = (a_1, \epsilon), \delta'(a_1, \epsilon, \Sigma_0) = (a_2, \Sigma_0)$. 

Since the new transitions are only $\epsilon$-transitions and leads every accepted word in $\bA$ (and only those) to the accept state to $\bB$ regardless of the stack content, $\bA$ and $\bB$ accept the same words. However, the stack is empty when a word is accepted by $\bB$. For a more formal treatment, the details can be found in \cite[Section 6.2.3]{HopcroftMotwaniUllman2007}.
\end{proof}

\begin{rmk}
The standard method for formally proving statements about pushdown automata is via induction on the length of the \emph{instantaneous transitions} which is essentially a tool that permits us to see the state and stack after each transition, which we introduce in the next section. 
\end{rmk}

\section{Instantaneous description}
	Unlike finite state automata where it is relatively straightforward to picture the state of the machinery at each step (since all we have to remember is the current state, for which there are finitely many possibilities by definition, and the remaining input), a PDA a more involved to keep track of since the stack can be grow infinitely long. For this reason, we often represent the configuration of a PDA using its \emph{instantaneous description} (ID), given by a triple $(s,w,\gamma)$ which keeps track of three key features: 
	\begin{enumerate}
		\item $s$ is the state, 
		\item $w$ is the remaining input,
		\item $\gamma$ is the stack content. 
	\end{enumerate}
	Note that by convention, the top of the stack is at the left-hand-side of $\gamma$. 
	
	\begin{defn}(Turnstile notation)
		If $\bA$ is a PDA where $\bA = (S, \Sigma, X, \delta, s_0, \Sigma_0, A)$, and $\delta(s,x,\gamma) \ni (s', \alpha)$, then we write 
		$$(s,xw, \gamma \beta) \vdash_\bA (s', w, \alpha \beta).$$ 
		
		That is, in one step the configuration of $\bA$ goes from state $s$, with remaining input $xw$ and stack content $\gamma \beta$ to state $s'$ with remaining input $w$ and stack content $\alpha \beta$. 
		
		If $\bA$ is understood implicitly, we may drop the subscript and write $\vdash$. Moreover, we use the notation $\vdash^*$ when representing a change in configuration in zero or more moves. 
	\end{defn}
	
	In general, the instantaneous description is used as the main inductive tool when proving that a PDA accepts a certain language. 
	
	\begin{ex}[Proof of Example \ref{ex: pda-0n-1n}]\label{ex: pda-0n-1n-pf}
		Take the PDA $\bA$ of Example \ref{ex: pda-0n-1n} as illustrated in Figure \ref{fig: pda-0n-1n}, and $L = \{0^n 1^n \mid n \geq 1\}$. We will show using IDs that $L = \cL(\bA)$. 
		
		Let $X = \{0,1\}$. 
		
		($L \subseteq \cL(\bA$). First observe that for any $w \in X^*$, 
		$$(s_0, 0^n w, \Sigma_0) \vdash^* (s_1, w, \sigma^n \Sigma_0).$$
		Indeed, the base case $n = 0$ is given by the arrow going from $s_0$ to $s_1$, and the induction is given by the arrow from $s_1$ to $s_1$. 
		
		Similarly, $$(s_1, 1^n, \sigma^n \Sigma_0) \vdash^* (s_2, \epsilon, \Sigma_0).$$
		
		Finally, for any $v \in X^*$, $$(s_2, v, \Sigma_0) \vdash (s_3, v, \Sigma_0).$$
		
		This shows that $$(s_0, 0^n 1^n, \Sigma_0) \vdash^* (s_1, 1^n, \sigma^n \Sigma_0) \vdash^* (s_2, \epsilon, \Sigma_0) \vdash (s_3, \epsilon, \Sigma_0),$$
		where $w = 1^n$, and $v = \epsilon$. This shows that $0^n 1^n$ is accepted by $\bA$ for any $n \geq 1$. 
		
		($\cL(\bA) \subseteq L$). For the other direction, we need to show that for any $w$ such that $(s_0, w, \Sigma_0) \vdash^* (s_3, \epsilon, \Sigma_0)$ (assuming acceptance by empty stack), we have that $w = 0^n1^n$ for some $n \geq 1$. Hence, $L \subseteq \cL(\bA)$. 
		
		First observe that if $$(s_0, w_1 w_2 w_3, \Sigma_0) \vdash^* (s_1, w_2 w_3, \gamma \Sigma_0),$$ then $w_1 = 0^n$ and $\gamma = \sigma^n$ for some $n \geq 1$ as that is the only possible path allowing this. 
		
		Similarly, observe that if $n \geq 1$ and $$(s_1, w_2 w_3, \sigma_n \beta \Sigma_0) \vdash^* (s_2, w_3, \beta \Sigma_0),$$
		then $w_2 = 1^n$ as that is the only possible path allowing this. 
		
		Finally, observe that if $$(s_2, w_3, \beta \Sigma_0) \vdash^* (s_3, \epsilon, \Sigma_0)$$
		then $w_3 = \epsilon$ and $\beta = \epsilon$ as that is the only possible path allowing this. 
		
		Putting it all together, we must have that $w = w_1 w_2 w_3 = 0^n 1^n$ for some $n \geq 1$. Hence, $\cL(\bA) \subseteq L$. 
		
		This finishes the proof that $L = \cL(\bA)$. 
\end{ex}

\section{One-counter languages}

Another important family of sub-context-free languages apart from deterministic context-free languages are one-counter languages, which we define below.

\begin{defn}
A language $\cL$ is {\it one-counter} if there is a non-deterministic pushdown automaton with $|\Sigma|=2$ (where one of the symbol is the start stack symbol usually denoted by $\Sigma_0$) and such that $\cL=|\bA|$.
\end{defn}

The term ``one-counter'' can be intuitively thought of as referring to the fact that with access to only one stack symbol to push or pop (excluding the start symbol), the stack of a pushdown automaton effectively acts as a counter. 

Example \ref{ex: pda-0n-1n} is an example of a one-counter automaton. For a string of the form $0^n 1^n$, the stack is used to count of number of zeros and make sure the number of $1$'s is equal to the number of zeros. Therefore, $\{0^n1^n \mid n \in \bN \}$ is an example of a one-counter language.

Next, we introduce a one-counter PDA that will come up often when describing left-orders in our thesis. 

\begin{figure}[h]{\includegraphics{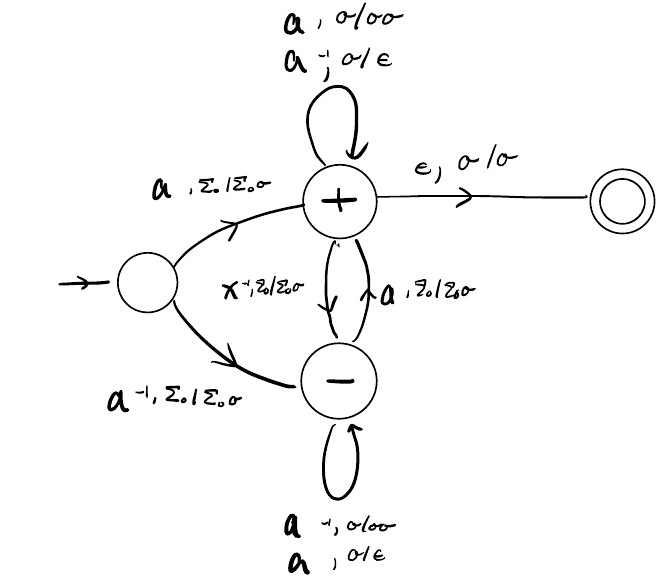}}
\caption{A pushdown automaton accepting words in $\{a, a\inv\}^*$ with positive exponent sum. The automaton, which is one-counter, essentially stores the exponent sum as follows: the states $\{+, -\}$ represent the sign of the exponent sum, while the stack stores an integer given by the number of $\sigma$ symbols in it. For example, while in $+$ state, reading an $a$ input increments the stack and reading $a\inv$ decrements it; the response is inverted when in the $-$ state. If the exponent sum is positive after reading the input, the automaton should be in state $+$ with stack symbol $\sigma$ at the top of the stack. If so, the input word is accepted.}
\label{fig: pda-int-counter}
\end{figure}

\begin{ex}\label{ex: pda-int-counter}
	Figure \ref{fig: pda-int-counter} shows a pushdown automaton accepting the language $L$ of words in $\{a, a\inv\}^*$ with positive exponent sum. A pumping argument (see Lemma \ref{lem: fsa-pumping}) can be used to show that $L$ is not regular. Indeed, suppose that $L$ is regular and $n$ is the associated pumping constant. Let $w = a^{-m} a^\ell$ where $m = n$ and $\ell > m$. Then $w$ is an accepted word which can, by Pumping Lemma, be separated into $w = xyz$ where $|xy| \leq n$ and $y$ is non-empty. By construction, $xy = a^{-|xy|}$, and in particular $y = a^{-|y|}$. Therefore $xy^kz$ must be accepted for $k > \ell \geq 0$ by Pumping Lemma, leading to a contradiction. Intuitively, much like Example \ref{ex: fsa-pumping} this proof exploited the fact that finite state automata cannot keep track of an arbitrary large integer. 
\end{ex}

In general, stack alphabets can be arbitrarily large in cardinality. Example \ref{ex: pda-palin} is an example of an automaton whose stack alphabet $\Sigma$ is a function of the input alphabet $X$, $|\Sigma| = |X| + 1$.\sidenote{I do not have the tools to prove that the language of palindromes is not one-counter for $|X| \geq 2$. This would require showing that \emph{no} pushdown automaton accepting this language has a stack alphabet of size greater than $2$, and I do not know of any theorem which could be used to do such a thing. However, the reasoning for why that would be true is quite intuitive: each letter in the first half of each input must be saved in the stack memory to be later compared with those of the second half. Indeed, since the length of the first half of the input word is arbitrary, we cannot save the letters in a finite state automaton since it would require unbounded memory.}

We mentioned the class of one-counter languages because we will study left-orders of regular and one-counter complexity specifically; context-free languages in general will actually feature very little on our treatment of left-orderable groups.

\begin{rmk}
Two-counter automata are equivalent to Turing machines with the caveat that the inputs and outputs must be ``properly encoded''.\sidenote{See \url{https://en.wikipedia.org/wiki/Counter_machine\#Two-counter_machines_are_Turing_equivalent_.28with_a_caveat.29}.} 
\end{rmk}

We end the section with Figure \ref{fig: cf-sub-venn} placing one-counter languages and deterministic context-free languages in a partial Chomsky hierarchy. 

\begin{figure}[h]{
\includegraphics{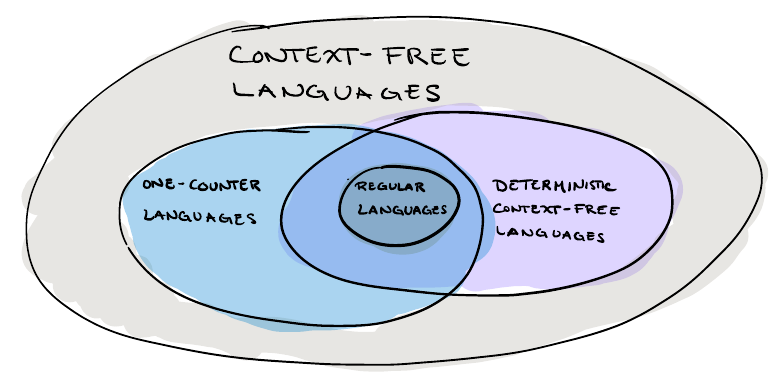}}
\caption{A Venn diagram illustrating the complexity classes of one-counter languages and deterministic context-free languages: both are language classes containing the class of regular languages and contained in the class of context-free languages.}
\label{fig: cf-sub-venn}
\setfloatalignment{b}
\end{figure}

\section{Context-free grammars}
Just like regular expressions are to finite state automata, the languages accepted by pushdown automata can be written as context-free grammars. 

\begin{defn}[Context-free grammar]
	A \emph{context-free grammar} (CFG) is defined as a four-tuple $$\bG = (V, X, R, S),$$ where 
	\begin{itemize}
		\item $V$ is a finite set of \emph{variables} or \emph{non-terminals},
		\item $X$ is a finite set of \emph{terminals}, disjoint from $V$, which will make up the alphabet of the language generated by the CFG, 
		\item $R \subseteq V \times (V \sqcup X)^*$ is a set of relations called the \emph{rewrite rules} or \emph{production rules}, often written using arrows, that is $$v \to w,$$ where $v$ is a variable in $V$ and $w$ a string in $(V \sqcup X)^*$. 
		\item $S \in V$ is the start symbol that are used in the production rules.  
	\end{itemize} 
	
	For any strings $u, w \in (V \cup X)^*$, we say that $u$ \emph{yields} $v$ if, written as $$u \Rightarrow w$$
	if there exists $(\alpha, \beta) \in R$ and intermediate words $u', u'' \in (V \cup X)^*$, such that 
	$$u = u' \alpha u'', \qquad w = u' \beta u'',$$
	that is, $w$ is the result of applying production rule $(\alpha, \beta)$ to $u$. 
	
	For repetitive rule application, we say that $u$ \emph{yields} or \emph{derives} and write 
	$$u \derives w  $$ if there exists an integer $n$ such that 
	$$u = u_1 \Rightarrow u_2 \Rightarrow \dots \Rightarrow u_n = w.$$	
\end{defn}
	
\begin{defn}[Sentential form and sentence]
For $\bG = (V,X,R,S)$ a CFG, we call a word in $(V \sqcup X)^*$ derived from start symbol $S$ a \emph{sentential form}, and a sentential form consisting only of terminal symbols in $X^*$ a \emph{sentence}.
\end{defn}

\begin{defn}[Leftmost and rightmost derivation]
We call a derivation in a CFG \emph{leftmost} (resp. \emph{rightmost}) if we choose the leftmost (resp. rightmost) variable as the one being expanded at each step. 
\end{defn} 

\begin{defn}[Derived language] Suppose that $\bG = (V, X, R, S)$ is a context-free grammar. The \emph{language derived by $\bG$} is given by 
	$$\cL(\bG) = \{w \in X^* \mid S {\stackrel {*}{\Rightarrow }} w\},$$
	that is, the $\cL(\bG)$ is given by the set of all strings with only terminal symbols derivable from start symbols, also known as the sentences of $\bG$. 
\end{defn}

\begin{ex}[CFG of Example \ref{ex: pda-0n-1n}]\label{ex: pda-0n-1n-CFG}
	The CFG generating $$L = \{0^n 1^n \mid n \in \nats\}$$ is given by 
	$\bG = (V, X, R, S)$ where 
	\begin{itemize}
		\item $V = \{S\}$,
		\item $X = \{0,1\}$,
		\item $R = \{S \to 0S1, \quad S \to \epsilon\}.$
	\end{itemize} 
	For each $n \geq 0$, the word $0^n 1^n$ is derived by $n$ applications of the derivation rule. 
\end{ex}

\begin{thm}
	A language $L$ is accepted by a pushdown automaton $\bA$ if and only if there exists a context-free grammar $\bG$ that derives it, i.e.
	$$L = \cL(\bA) = \cL(\bG).$$
\end{thm}
The proof is available in \cite[Section 6.3]{HopcroftMotwaniUllman2007}.

\section{Pushdown automata as graphs}
Similarly to finite state automata, pushdown automata are often viewed as labeled graphs. To make this intuition formal, let us invoke the graph functor. 

\begin{defn}\label{defn: pda-graph-functor}
Define the \emph{graph functor} $\cG$ going from the category of pushdown automata to the category of directed graphs equipped with an alphabet, a stack alphabet, a start vertex and accept vertices. For a pushdown automaton $\bA = (S, X, \Sigma, \delta, s_0, \Sigma_0, A)$, $\cG(\bA) = (V, E, X, \Sigma, s_0, \Sigma_0, A)$ is defined as follows. 
The vertex set is given by $V = S$, and the directed labeled edge-set is given by $E = \{(s_a, s_b, x, \sigma_1/\sigma_2) \mid \exists x \in X \cup \{\epsilon\}, \sigma_1, \sigma_2 \in \Sigma \cup \{\epsilon\}$ such that $\delta(s_a, x, \sigma_1) = (s_b, \sigma_2) \}$. The alphabet $X$ is mapped to itself, so is $\Sigma$, the start state maps to the corresponding start vertex, the start symbol to the corresponding stack symbol, and the accept states map to the corresponding accept vertices. 
\end{defn}
Under this definition and using Lemma \ref{lem: pda-empty}, we have that an accepted word in $\bA$ corresponds to an accepted word in $\bB$ by empty stack, which corresponds to a path $\cG(\bB)$ starting from the the start vertex with the empty stack symbol, and ends in an accepted vertex with empty stack. 

\begin{defn}\label{defn: pda-inv-graph-functor}
Define the \emph{contravariant graph functor} $\cG\inv$ from the category of graphs with labeled edges equipped with an alphabet, a set of accept vertices, and a special start vertex to the category of finite state automata. 
\end{defn}

\section{Closure properties of pushdown automata}

We now use the graph functor notation to prove the closure properties of regular languages, which will be very useful for the results contained in this thesis. We recall here the discussion of Section \ref{sec: closure-formal-lang} concerning closure properties formal languages. For a different treatment on closure properties of regular languages, see \cite[7.3]{HopcroftMotwaniUllman2007}. 

\subsection{AFLs closure properties}

\begin{thm}
Context-free (resp. one-counter) languages are full AFL\sidenote{abstract family of languages}, that is, they are closed under union, concatenation, Kleene star, intersection with a regular language, homomorphism and inverse homomorphism. 
\end{thm}

The proofs of these statements are very reminiscent to the ones in the last section, but let us do them anyway for practice. 

\begin{lem}[Closure under union] If $L_1$ and $L_2$ are context-free (resp. one-counter) languages, then $L_1 \cup L_2$ is context-free (resp. one-counter). \end{lem}
\begin{proof}

\begin{figure}[h]{
\includegraphics{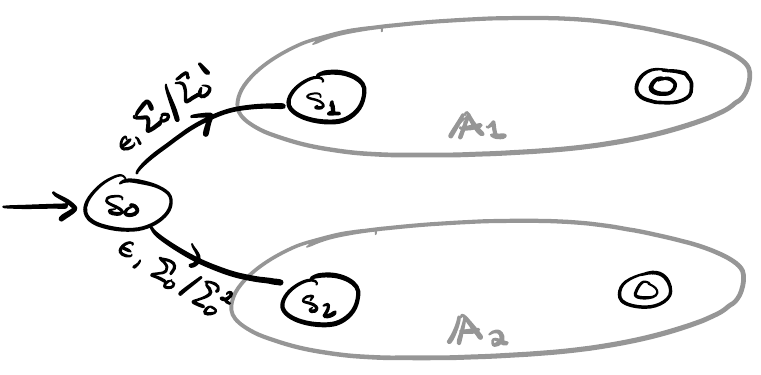}}
\caption{A pushdown automaton which accepts the union of two context-free languages.}
\label{fig: pda-union}
\end{figure}

We refer to Figure \ref{fig: pda-union}. Let $\bA_1$ and $\bA_2$ be pushdown automata which accept $L_1$ and $L_2$ respectively. By Lemma \ref{lem: pda-empty} we may assume that they accept by empty stack. Then, let 
\begin{align*}
	& \cG(\bA_1) = (V_1, E_1, X_1, \Sigma_1, s_1, \Sigma_0^1, A_1) \\
	& \cG(\bA_2) = (V_2, E_2, X_2, \Sigma_2, s_2, \Sigma_0^2, A_2)
\end{align*}
be the graphs of $\bA_1$ and $\bA_2$ respectively. 

We will do the context-free case first. Create a new empty stack symbol $\Sigma_0$ and a new graph with vertices $$V = V_1 \cup V_2 \cup \{s_0\}$$ and edges 
$$E = E_1 \cup E_2 \cup \{(s_0, s_1, \epsilon, \Sigma_0/\Sigma_0^1), (s_0, s_2, \epsilon, \Sigma_0/\Sigma_0^2)\}.$$ 
Then all the paths in $G$ from start vertices to final vertices $$A = A_1 \cup A_2$$ with empty stack arise from words in $$L_1 \cup L_2.$$ Then $$\cG\inv(V, E, X_1 \cup X_2, \Sigma = (\{\Sigma_0\} \cup \Sigma_1 \cup \Sigma_2), s_0, A)$$ is a pushdown automaton accepting $L_1 \cup L_2$. 

For the one-counter case, we modify $\bA_2$ such that $\Sigma = \Sigma_2 = \Sigma_1 = \{\Sigma_0, \sigma\}$ such that $\Sigma_0 = \Sigma_0^1 = \Sigma_0^2$ and where $\sigma$ acts as the counter. As a result, $\cG\inv(V, E, X_1 \cup X_2, \Sigma, s_0, A)$ will be a one-counter automaton. 

\end{proof}

\begin{lem}[Closure under concatenation]
Let $L_1$ and $L_2$ be two context-free languages (resp. one-counter). Then $L_1 L_2 = \{w_1 w_2 \mid w_1 \in L_1, w_2 \in L_2\}$ is a context-free (resp. one-counter) language. 
\end{lem}
\begin{proof}

\begin{figure}[h]{
\includegraphics{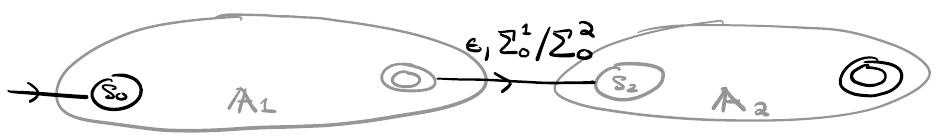}}
\caption{A pushdown automaton which accepts the concatenation of two context-free languages.}
\label{fig: pda-concat}
\end{figure}

We refer to Figure \ref{fig: fsa-concat}. Let $\bA_1, \bA_2$ be pushdown automata accepting $L_1$ and $L_2$ by empty stack respectively (which is allowed by Lemma \ref{lem: pda-empty}), with empty stack symbol $\Sigma_0$ and let $\cG(\bA_1) = (V_1, E_1, X_1, \Sigma_1, s_1, \Sigma_0, A_1), \cG(\bA_2) = (V_2, E_2, X_2, \Sigma_2, s_2, \Sigma_0, A_2)$. Let $$V = V_1 \cup V_2$$ and $$E= E_1 \cup E_2 \cup \{(a_1, s_2, \epsilon, \Sigma_0^1/\Sigma_0^2) \mid a_1 \in A_1\}).$$ Let $s_0 = s_1$ and $A = A_2$. Then every path from $s_0$ to $A$ is a path starting at $s_1$ going to $A_1$, passing by the edge $(a_1, s_2, \Sigma_0^1/\Sigma_0^2)$ with $a_1 \in A_1$ which swaps the empty stack symbol of $\bA_1$ (since the $\bA_1$ accepts by empty stack), which the empty stack symbol of $\bA_2$, and ending at $A_2$. Therefore, the language accepted by $\cG\inv(V, E, X_1 \cup X_2, \Sigma_1 \cup \Sigma_2, s_0,  A)$ is $L_1 \varepsilon L_2 = L_1 L_2$. 

Similarly to the proof of the last lemma, we approach the one-counter case as follows. We modify $\bA_2$ such that $\Sigma = \Sigma_2 = \Sigma_1 = \{\Sigma_0, \sigma\}$ such that $\Sigma_0 = \Sigma_0^1 = \Sigma_0^2$ and where $\sigma$ acts as the counter. As a result, $\cG\inv(V, E, X_1 \cup X_2, \Sigma, s_0, A)$ will be a one-counter automaton. 
\end{proof}

\begin{lem}[Closure under Kleene star]
Let $L$ be a context-free (resp. one-counter) language. Then $L^*$ is a context-free (resp. one-counter) language. 
\end{lem}
\begin{proof}

\begin{figure}[h]{
\includegraphics[width=0.5\textwidth]{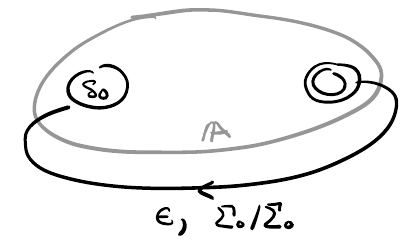}}
\caption{A pushdown automaton which accepts any power of a context-free language.}
\label{fig: pda-npow}
\end{figure}

We refer to Figure \ref{fig: pda-npow}. Let us start with the context-free case. Let $\bA$ be a finite state automaton accepting $L$ and let $\cG(\bA) = (V,E, X, s_0, A)$ be the associated graph. Define $$G = (V, E \cup \{(a, s_0, \varepsilon) \mid a \in A\}).$$ Then every path from $s_0$ to $A$ is either a path from $s_0$ to $A$ in the original graph $\cG(\bA)$, or a path passing through a new edge of the form $(a, s_0, \varepsilon)$, in which case from $s_0$ it must form another path in $\cG(\bA)$. Therefore, $\cG\inv(G, X, s_0, A)$ accepts the language $\bigcup_{n=1}^\infty L^n$, and $L^* = \bigcup_{n=1}^\infty L^n \cup \{\epsilon\}$ is a context-free language. 

For the one-counter case, notice that since the stack alphabet for the new pushdown automaton is $\Sigma$, $L^*$ is one-counter if and only if $L$ is one-counter.
\end{proof}

\begin{lem}[Closure under intersection with regular languages]
Let $L$ be a context-free (resp. one-counter) language, and $R$ be a regular language. Then $L \cap R$ is a context-free (resp. one-counter) language. 
\end{lem}
\begin{proof}

\begin{figure}[h]{
\includegraphics{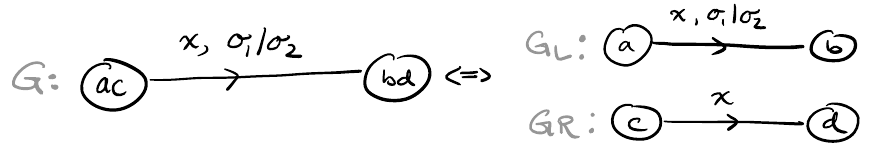}}
\caption{Part of a pushdown automaton which accepts the intersection of a context-free language with a regular one.}
\label{fig: pda-intersect}
\end{figure}

Let $\bA_L, \bA_R$ be the automata accepting $L$ and $R$ respectively. We think of an input word $w$ ``runs in parallel'' in both automata and is accepted if and only if it is accepted by both $\bA_L$ and $\bA_R$. Let us formalize this. 

Let $\cG(\bA_L) = (V_L, E_L, X_L, \Sigma,s_L, \Sigma_0, A_L), \cG(\bA_R) = (V_R, E_R, X_R, s_R, A_R)$, and let $G_L = (V_L, E_L), G_R = (V_R, E_R)$. Take the graph tensor product $G = G_L \times G_R = (V,E)$, that is, the vertex set is $V = V_L \times V_R$ and $(ac,bd, x, \sigma_1/\sigma_2)$ is an edge in $E$ if and only if $(a,b,x,\sigma_1/\sigma_2)$ and $(c,d,x)$ are edges in $G_L$ and $G_R$ respectively.

 Take the start vertices in the graph product to be $s_0 = s_Ls_R$, and accept vertices to be $A = \{a_La_R \mid a_L \in A_L, a_R \in A_R\}$. Let $X = X_L \cap X_R$. Then we claim that the language accepted by $\bA = \cG\inv(V, E, X, \Sigma, s_0, \Sigma_0, A)$ is $L_1 \cap L_2$, by empty stack. Indeed, observe that the set of paths from $s_0$ to $A$ in $G$ arise from a path whose pre-image under the graph product is a path from $s_L$ to $A_L$ with empty stack and a path from $s_R$ to $A_R$. 

To observe that the statement holds for one-counter languages, observe that the stack alphabet $\Sigma$ is the same as the starting stack alphabet.
\end{proof}

\begin{lem}[Closure under homomorphism]
Let $L$ be a context-free (resp. one-counter) language over $X$ and $h: X^* \to Y^*$ be a homomorphism. Then $h(L)$ is a context-free (resp. one-counter) language.
\end{lem}
\begin{proof}
\begin{figure}[h]{
\includegraphics{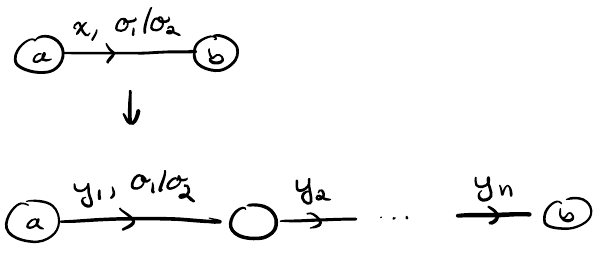}}
\caption{Part of a pushdown automaton which accepts the homomorphism of a context-free language.}
\label{fig: pda-hom}
\end{figure}
Let $\bA$ be an pushdown automaton accepting $L$ and let $G = \cG(\bA) = (V, E, X, \Sigma, s_0, \Sigma_0, A)$. Let $G'$ be a copy of $G$ where every $(x, \sigma_1/\sigma_2)$-labeled edge $e$ is replaced with a path $p_{h(x)}$ where $h(x) = y_1 \dots y_n$ such that the first edge in the path has label $(y_1, \sigma_1/\sigma_2)$, and the rest have label $(y_i, \sigma_2/\sigma_2)$ for $2 \leq i \leq n$. Then every path inducing a word $w$ in $G$ is replaced by a path inducing $h(w)$ in $G'$ such that the stack transitions remain the same for $x$ in $G$ and $h(x)$ in $G'$. Therefore, $\cG\inv(G', Y, \Sigma, s_0, \Sigma_0, A)$ is pushdown automaton accepting $h(L)$. 

To observe that the statement holds for one-counter languages, observe that the stack alphabet $\Sigma$ is the same as the starting stack alphabet.
\end{proof}

\begin{lem}[Closure under inverse homomorphism]
Let $L$ be a context-free (resp. one-counter) language over $X$ and let $h\inv$ be an inverse homomorphism. Then $h\inv(L)$ is a context-free (resp. one-counter) language.
\end{lem}
\begin{proof} 
Let $\bA = (S, Y, \Sigma, \delta, s_0, \Sigma_0, A)$ be a finite state automaton accepting $L$. Let $\bB = (S, X, \Sigma, \gamma, s_0, \Sigma_0, A)$ be a finite state automaton such that $$\gamma(s, x, \sigma) = \delta(s, h(x), \sigma).$$ We argue that $\bB$ accepts $h\inv(L)$. Let $w$ such that $w \in h\inv(L)$. Then this is true if and only if $h(w) \in L \iff \delta(s_0, h(w), \Sigma_0) \in A$, but $\delta(s_0, h(w), \Sigma_0) = \gamma(s_0, w, \Sigma_0)$, so $\gamma(s_0,w, \Sigma_0) \in A$ also.

To observe that the statement holds for one-counter languages, observe that the stack alphabet $\Sigma$ is the same as the starting stack alphabet.
\end{proof}

\subsection{Closure under reversal}

\begin{lem}[Closure under reversal]
Let $L \subseteq X^*$ be a context-free language over $X$. Then, the reversal $L^R$ is also context-free.
\end{lem}
\begin{proof}

	From a starting context-free grammar $\bG = (V, X, R, S)$, construct the reversal CFG $\bH = (V,X,R',S)$ given by 
	$$R' = \{ \alpha \to \beta^R \mid (\alpha, \beta) \in R\}.$$
	
	The goal is to show that $L^R = \cL(\bG)^R = \cL(\bH)$. 
		
	Let $\cG$ and $\cH$ be the collection of all sentential forms obtained by $\bG$ and $\bH$ respectively, and let $G_n, H_n$ be the collection of sentential forms obtainable by derivation of length $n$ in $\bG$ and $\bH$ respectively. Showing that $\cG^R = \cH$ completes the proof since $\cL(\bG)$ and $\cL(\bH)$ are all the sentential forms which are sentences generated by $\bG$ and $\bH$ respectively, and $\bG$ and $\bH$ share the same terminal alphabet $X$. Since $\cG = \bigcup_{n=0}^\infty G_n$ and $\cH = \bigcup_{n=0}^\infty H_n$, it suffices to show that $G_n^R = H_n$ for all $n \geq 0$. We will show this by induction. 
	
	It is clear that $G_0^R = H_0$ since $\bG$ and $\bH$ have the same set of variables and terminals. 
		
	Now, suppose that $G_{n-1}^R = H_{n-1}$. We need to show that $G_n^R = H_n$. 
	
	($G_n^R \subseteq H_n$). Let $w \in G_n$ be a sentential form derived from $v \in G_{n-1}$, that is, 
	$$v = u_1 \alpha u_2 \quad \text{ and } w = u_1 \beta u_2,$$
	for some variable $\alpha$ and $(\alpha, \beta) \in R$. 
	
	Observe that 
	$$v^R = u_2^R \alpha u_1^R \quad \text{ and } w^R = u_2^R \beta^R u_1^R.$$
	
	Since $v \in G_{n-1}$, $v^R \in G_{n-1}^R$ and $G_{n-1}^R = H_{n-1}$ by hypothesis. Therefore, $w^R \in H_n$ by application of the $\alpha \to \beta^R$ production rule in $R'$. This shows that $G_n^R \subseteq H_n$.
	
	($G_n^R \supseteq H_n$). The other direction is similar. If $w \in H_n$, then there exists $v \in H_{n-1}$ such that 
	$$v = u_1 \alpha u_2 \quad \text{ and } w = u_1 \beta^R u_2,$$
	for some $(\alpha, \beta) \in R$. Then, since $v \in H_{n-1}$ and by hypothesis $H_{n-1} = G_{n-1}^R$, we have that $v \in G_{n-1}^R$. That is, $v^R \in G_{n-1}$. 
	
	Now, $$v^R = u_2^R \alpha u_1^R \Rightarrow_\bG u_2^R \beta u_1^R = w^R$$
	by application of the $\alpha \to \beta$ rule in $R$. In other words, $w^R \in G_n$, and $w \in G_n^R$ as wanted. 
	
This finishes showing that $\cG^R = \cH$, and completes the proof. 
\end{proof}

\begin{rmk}
	Note that this does not show that one-counter languages are closed under reversal. For this reason, we will mostly construct one-counter languages by constructing one-counter PDAs in this thesis. 
\end{rmk}

\section{Languages which are not context-free}

\subsection{Intuition}\label{subsec: intuition-lang-not-context-free}

Context-free languages are languages where repetition can be ``balanced'' by pushing and popping in the stack. It can be helpful to think of only pushdown automata with empty stack for that reason, and this is a mental shortcut that does not lose us any information about context-free languages by Lemma \ref{lem: pda-empty}. We have seen that the language $\{0^n 1^n \mid n \geq 1\}$ is context-free. So is $\{0^n 1^{2n} \mid n \geq 1\}$ as a homomorphic image of sending $1 \mapsto 1^2$, or more concretely, every time we push a stack symbol to read $0$, we pop a stack symbol to read two $1$'s. Another language that preserves stack symmetry is $\{0^m 1^n 2^n 3^m \mid m \geq 1\}$. Intuitively, it is because the $1^n 2^n$ is ``nested'' between $0^m$ and $3^m$, as illustrated in Figure \ref{fig: pda-nested}. Finally, this stack symmetry can be across a large stack alphabet, such as for the language of palindromes $\{ww^R \mid w \in X^*\}$, as (partially) shown previously in Figure \ref{fig: pda-palindrome}.

\begin{figure}[h]{
\includegraphics[width=\textwidth]{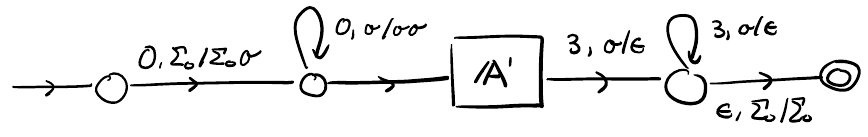}}
\caption{A pushdown automaton which accepts which accepts $\{0^m 1^n 2^n 3^m \mid m \geq 1\}$. The part labeled with $\bA'$ is given a pushdown automaton similar to Figure \ref{fig: pda-0n-1n} but with the arrows with the label $0$ swapped the label $1$ and the arrows with the labels $1$ swapped to the label $2$.}
\label{fig: pda-nested}
\end{figure}

Languages that are \emph{not} context-free lack this stack symmetry between pushing and popping. A classic example is the language $\{0^n 1^n 2^n \mid n \geq 1\}$, for which there is no obvious stack symmetry. Another example is $\{0^m 1^n 2^m 3^n \mid n,m \geq 1\}$ for the same reason: recognizing $c^m$ requires to push out the $m$ counter, however, that would push out the $n$ counter first which is nested within the $m$ counter, so there would be no way of recognizing $d^n$. Similarly, while $\{ww^R \mid w \in X^*\}$ is a context-free language, $\{ww \mid w \in X^*\}$ is not because the push-pop stack symmetry has been lost. 

This ``balanced'' stack intuition has a formal  counterpart known as the Chomsky–Sch\"utzenberger representation theorem.\sidenote{See \url{https://en.wikipedia.org/wiki/Chomsky\%E2\%80\%93Sch\%C3\%BCtzenberger_representation_theorem}.} For our purposes, the theorem tells us that we may generate all context-free languages by starting with the Dyck language $D_2$ of all string of matching brackets over two pairs of brackets, say the pairs $\{[,]\}$ and $\{(,)\}$.\sidenote{For example ``[]()'' and ``([()])[]'' are strings of matching brackets, whereas ``[)['' is not.} Then every context-free language can be obtained from $D_2$ using homomorphisms, inverse homomorphisms and intersection with regular languages.\sidenote{Source \url{https://cs.stackexchange.com/questions/30026/constructing-all-context-free-languages-from-a-set-of-base-languages-and-closure}.} The balancing of the stack we have observed are derived from the balancing given by the matching pairs of brackets in the Dyck language. 

\subsection{The pumping lemma for context-free languages}

As with finite state automata, the only lemma I know how to use to prove that a language is not context-free is the pumping lemma for context-free languages. The statement somewhat matches that of the pumping lemma for regular languages (Lemma \ref{lem: fsa-pumping}). 

\begin{thm}[The pumping lemma for context-free languages] Let $L$ be a context-free language. There exists some integer $n \geq 1$ such that for every $w \in L$ with $|w| > n$, there exists a decomposition $w = uvxyz$ such that 
\begin{enumerate}
\item $|vxy| \leq n$, that is, the middle portion of $w$ is not too long. 
\item $vy \not= \epsilon$, meaning at least one of the pieces to be ``pumped'' is not empty,
\item $uv^kxy^kz \in L$ for all $k \geq 0$, which is the pumping condition. 
\end{enumerate}
\end{thm}
The proofs of this statement typically use the \emph{context-free grammar}, an alternative but equivalent way of describing context-free languages. Such a proof can be found for example in \cite[Theorem 7.18]{HopcroftMotwaniUllman2007}. An immediate question is then: what is the equivalent proof using the pushdown automata? What is the idea behind the proof? 

A user going by a3nm asked the same very question on StackOverflow in 2011, and after not being able to find a satisfactory answer, came up with their own rigorous proof and uploaded it on ArXiv \cite{AmarilliJeanmougin2013}. After a cursory glance at the paper, I believe the answer to our question is that the ``loop'' which may be pumped is, perhaps as expected, at the level of paths and the stack. Since by Lemma \ref{lem: pda-empty} and Section \ref{subsec: intuition-lang-not-context-free}, we may expect an ``inherent push-pop stack symmetry'' in the loop, pumping on the loop is often a way to break the symmetry and show that a language is \emph{not} context-free.\sidenote[][-0.75cm]{While writing this chapter, I particularly enjoyed about learning this tidbit because it is confirmation that one can truly contribute  fundamental understanding to a field by asking a very basic question late in the game. This question is one I would have been embarrassed to voice out loud because the equivalence between context-free grammars and pushdown automata is ``well-known'', yet having the definitive concrete answer is still important and valuable.} %

Let us move onto working out an example of applying the pumping lemma for context-free languages. As with the previous example with the pumping lemma for regular languages, the details of the decomposition of $w = uvxyz$ are unknown and we deduce them using information we have about the language we want to show is not context-free.

\begin{ex}[Using the pumping lemma for context-free languages]\label{ex: pumping-lem-context-free}
The language $L = \{0^n 1^n 2^n \mid n \geq 1\}$ is not context-free. Indeed, let $w = 0^n 1^n 2^n$ for $n$ as given in the statement of the pumping lemma, and decompose $w = uvxyz$. Then if $|vxy| \leq n$, this means that $vxy$ cannot contain both $0's$ and $2's$. Now ``pump'' $v$ and $y$ for $k = 0$ such that $uxz \in L$ by the pumping lemma. Since $vy$ did not contain an equal number of $0's$ and $2$'s (as it cannot contain both) and $vy$ is non-empty by the assumptions of the lemma (meaning it must contain one of the either $0$'s or $2$'s), $uxz$ cannot contain an equal amount of $0$'s, $1$'s and $2$'s. This contradicts that $uxz$ is in $L$. The full details can be found in \cite[Example 7.19]{HopcroftMotwaniUllman2007}. 
\end{ex}
{}
\section{Non-closure properties of context-free languages}\label{sec: non-closure-context-free}
Context-free languages are notably \emph{not} closed under intersections with other context-free languages or under complement. Let us show this by a counter example.

We have seen in our graphical notation that intersecting with another language involves some kind of tensor product of paths, or parallelism. This kind of property would require another stack to keep up with if we were to simultaneously simulate two pushdown automata together. 
\begin{ex}[Context-free languages not closed under intersections]
Let $L_1 = \{0^m 1^m 2^n \mid m, n \geq 0\}$ and $L_2 = \{0^m 1^n 2^n \mid m, n \geq 0\}$. $L_1$ is context-free because $\{0^m 1^m \mid m \geq 0\}$ is context-free, $\{2^n \mid n \geq 0\}$ is a regular language (and thus context-free) and the concatenation of two context-free languages is context-free. The proof is similar for $L_2$. However, $L_1 \cap L_2 = \{0^n 1^n 2^n\}$ which we have seen is not context-free by Example \ref{ex: pumping-lem-context-free}. 
\end{ex}

Now note the following set theory identity $L_1 \cap L_2 = (L_1^c \cup L_2^c)^c$, where the $\cdot^c$ superscript stands for complement. Suppose that $L_1$ and $L_2$ are context-free. If context-free languages were closed under complement, then the right-hand side of our equality would be context-free, and context-free languages would be closed under intersection, leading to a contradiction. Therefore, context-free languages cannot be closed under complement.

\section{The positive cones of $\bZ^2$}\label{sec: research-P-Z2}

Similarly to how we've ended Chapter \ref{chap: fsa} with a recent geometrical result on positive cones that basically utilises the Pumping Lemma, let us end this one by showing the reader a geometrical result concerning context-free languages whose intuition is mostly based on what we have so far. This will permit us to classify all positive cones of $\bZ^2$ into two categories. 

Recall from Example \ref{ex: P-Zsq} that the positive cones of $\bZ^2 = \langle a, b \mid [a,b]\rangle$ are shaped like half-planes given by 
$$P_{\lambda, \diamond_1, \diamond_2} = \{(x,y) \in \bZ^2 \mid y \diamond_1 \lambda x \text{ or } y = \lambda x, x \diamond_2 0 \}$$
for $\lambda \in \bR \cup \{\pm \infty\}$, where $\diamond_1, \diamond_2 \in \{<,>\}$.

The following result is new to this thesis.\sidenote{To my knowledge.}

\begin{thm}\label{thm: Zsq-P-classified}
	If $\lambda \in \bQ \cup \{\pm \infty\}$, then $P_{\lambda, \diamond_1, \diamond_2}$ is a positive cone which can be represented by a regular language. Otherwise, if $\lambda \in \bR - \bQ$, then $P_{\lambda, \diamond_1, \diamond_2}$ cannot be represented by a context-free language. 
\end{thm}

We will prove this in two parts, given by Propositions \ref{prop: Zsq-rational-reg} and \ref{prop: Zsq-irrational-not-CF}. 

\begin{prop}\label{prop: Zsq-rational-reg}
	If $\lambda \in \bQ \cup \{\pm \infty\}$, then $P_{\lambda, \diamond_1, \diamond_2}$ is a regular positive cone.
\end{prop}
\begin{proof}
 Let $\cP_{\bQ \cup \{\pm \infty\}}$ be the collection of all positive cones $P_{\lambda, \diamond_1, \diamond_2}$ with $\lambda \in \bQ \cup \{\pm \infty\}$ and $\diamond_1, \diamond_2 \in \{<,>\}$. Recall from Example \ref{ex: LO-aut-Zsq} that $\GL_2(\bZ)$ acts transitively on $\cP_{\bQ \cup \{\pm \infty\}}$ and sends the basis $\{(1,0), (0,1)\}$ to defining parallelograms of area $1$ given by $\{(p,q), (r,s)\}$, where $p,q,r,s \in \bZ$ and $|ps - rq| = 1$. 

It is clear that $P_{0, >, >}$ is a regular positive cone as it is given by the language $$L = \{a^m b^n \mid n > 0 \text{ or } n =0, m > 0\},$$ where we implicitly identify $a$ with the basis element $(1,0)$ and $b$ with $(0,1)$. Let $X = \{a,b,a\inv,b\inv\}$ be our assumed alphabet. Then, the monoid homomorphism $h: X^* \to X^*$ sending 
$$h(a) = a^p b^q, \qquad h(b) = a^r b^s$$
send the positive language $L$ for $P_{0, >, >}$ to a positive cone language for $P_{\lambda, \diamond_1, \diamond_2}$, since the basis elements $(1,0)$ is sent to $(p,q)$ and the basis element $(0,1)$ is sent to $(r,s)$. Thus, all the positive cones in $\cP_{\bQ \cup \{\pm \infty\}}$ are regular. 
\end{proof}

One way to understand this result intuitively is that a half-plane encodes about the same information as a rational number, and rational numbers can be encoded a repeating decimal expression, which can be approximated to arbitrary precision by a regular expression.

Indeed, if $\lambda \in \bQ$, then its periodic decimal expression is given by
$$\lambda = a_i a_{i-1} \dots a_1.b_1\dots b_j\overline{c_1 \dots c_k},$$
where $i,j,k \in \bN$ and the $a_{i'}, b_{j'}, c_{k'} \in \{0,\dots, 9\}$ for $i' \in \{1,i\}, j' \in \{1,j\}, k' \in \{1,k\}$. Then, the rational expression
$$\lambda = a_i a_{i-1} \dots a_1.b_1\dots b_j({c_1 \dots c_k})^*$$
can approximate this rational number to arbitrary precision. 

 On the other hand, irrational numbers are precisely the numbers that cannot be encoded as an eventually periodic decimal expression. One would expect that it is not possible to encode its information via a regular expression, nor by the symmetry of a balanced stack. This is the intuition behind the following result. 

\begin{prop}\label{prop: Zsq-irrational-not-CF}
	Let $\lambda \in \bR - \bQ$ and $P_\lambda = \{(x,y) \in \bZ^2 \mid y > \lambda x\}$ and $P'_\lambda = \{(x,y) \in \bZ^2 \mid y < \lambda x\}$. Then $P_\lambda$ and $P'_\lambda$ cannot be represented by a context-free language. 
\end{prop}

To show this, we will use Parikh's theorem, which is roughly a visual version of the pumping lemma (\cite{Shallit2008} is a reference for more advanced topics in automata theory). Let us set some definitions. 

\begin{defn}[Parikh map]
	Let $X = \{x_1, \dots, x_n\}$ be a finite alphabet. For $i = 1, \dots, n$, let the map $\#_{x_i} : X^* \to \nats$ send a word to the number of times the letter $x_i$ appears in that word. 	We define the \emph{Parikh map} as $\psi: X^* \to \nats^n$, 
	$$\psi(w) = (\#_{x_1}(w), \dots, \#_{x_n}(w)).$$
\end{defn}

The Parikh map sends a word $w$ to a corresponding vector encoding the number of time each letter appears in $w$. 

\begin{defn}[Linear]
A set $S$ is called \emph{linear} if it can be given by the finite linear sum  
$$S = u_0 + \nats u_1 + \dots + \nats u_n,$$ where $u_i \in \nats^m$ for $i \in \{0, \dots, n\}$, where $n,m$ are some fixed constants.

A set is \emph{semilinear} if it can be given by the finite union of linear sets. That is, 
$$S = \bigcup_{i=1}^k S_i = \bigcup_{i=1}^k u^i_0 + \sum_{j=1}^{n_j} \nats u^i_j.$$
\end{defn}

\begin{thm}[Parikh's lemma]
	Let $X = \{x_1, \dots, x_n\}$ be a finite alphabet, $L \in X^*$ be a context-free language, and $\psi: X^* \to \nats^n$ be the Parikh map. Then, $\psi(L)$ is semi-linear.
\end{thm}
The proof of this theorem is outside the scope of this thesis. However, there has been many proofs of Parikh's theorem over the years in order to simplify its exposition, such as fairly recent one in \cite{Rubtsov2022}.

The proof of Proposition \ref{prop: Zsq-irrational-not-CF} is not so different from a fairly standard exercise involving Parikh's lemma.\sidenote{See \url{https://cs.stackexchange.com/questions/105836/language-involving-irrational-number-is-not-a-cfl}.}%

\begin{proof}[Proof of Proposition \ref{prop: Zsq-irrational-not-CF}]
	Let $X = \{a, a\inv, b, b\inv\}$ be an alphabet for $\bZ^2$ such that $\pi(a) = (1,0)$ and $\pi(b) = (0,1)$. Assume that there exists $L \subseteq X^*$ a context-free language evaluating to $P_\lambda$. Let $\psi$ be the Parikh map. Then $\psi(L)$ must be semi-linear. We will show that this leads to a contradiction. %
	
	Let us start by writing each $u \in \bN^4$ as 
	$$u = (x^+(u), x^-(u), y^+(u), y^-(u)),$$ and define 
	$$\vec{u} = (x^+(u) + x^-(u), y^+(u) - y^-(u)),$$
	corresponding to the realization of $u$ as a vector in $\bZ^2$. 
	
	Let $$\hat{\lambda} = \frac{(1,\lambda)}{\|(1,\lambda)\|}, \qquad \hat{\lambda}^\perp = \frac{(\lambda, 1)}{\|(\lambda, 1)\|},$$ be a unit vector with slope $\lambda$ and its perpendicular belonging to $P_\lambda$ respectively. 
	
	For an element $\vec{v} \in \bZ$, define 
	$$x(\vec{v}) := \vec{v} \cdot \hat{\lambda}, \qquad y(\vec{v}) := \vec{v} \cdot \hat{\lambda}^\perp.$$
	
	Define $\theta: (\bZ^2 - \{\vec{0}\}) \to (-\pi, \pi]$ as 
	$$\theta(\vec{v}) = \atanTwo(y(\vec{v}),x(\vec{v})),$$ 
	that is, $\theta(\vec{v})$ is the signed angle between the vector $\vec{v}$ and $\hat{\lambda}$.

	Since $\lambda$ is irrational, it is clear that for all non-zero $u \in \bN^4$, we have 
	$$|\theta(\vec{u})| > 0,$$
	since the coordinates of $\vec{u}$ are rational. 
	
	However, by density of the rationals, there exists a sequence of rationals $q_n > \lambda$ such that $q_n \to \lambda$ as $n \to \infty$. As $\theta$ is continuous on its domain, there exists a sequence $(1,q_n)$ such that $\theta(1,q_n) \to 0$ as $n \to \infty$. 
	
	Let $(x_n, y_n) \in \bZ^2$ be a sequence such that $\frac{y_n}{x_n} = q_n$ with $x_n > 0$. Then, $(x_n, y_n) \in P_\lambda$ since $$q_n = \frac{y_n}{x_n} > \lambda \iff y_n > \lambda x_n.$$
	Let every $w_n$ be a representative of $(x_n, y_n) \in P_\lambda$ in $L$. Then,
	$$\lim_{n \to \infty}  \theta(\vec{\psi}(w_n))) = 0 \text{ from above}.$$
	
	That is, $$\inf \theta((\vec{\psi}(L))) = 0.$$
	
	On the other hand, $$\psi(L) = \bigcup_{i=1}^k u_0^i +  \sum_{j=1}^{n_j} \nats u^i_j$$

	where $u^i_j \in \nats^4$ for $i = 1,\dots, k$ and $j = 1,\dots, n_j$ by assumption of semi-linearity. To derive a contradiction, we will show that applying this infimum to the right hand side of $\vec{\psi}$ yields an angle $> 0$. 
	
	Assume for now (we will show this later), that 
	\begin{enumerate}
		\item $\theta(n \vec{v}) = \theta(\vec{v})$ for all $n \in \bN$, and 
		\item $\min\{ \theta(\vec{v}), \theta (\vec{w}) \} \leq \theta(\vec{v} + \vec{w}) \leq \max \{ \theta(\vec{v}), \theta(\vec{w}) \}$ when $\vec v$ and $\vec w$ belong to the same half-plane defined by a slope of $\lambda$. That is, when $y(\vec{v})$ and $y(\vec{w})$ are of the same sign. 
	\end{enumerate}	
	
	Then, 
	\begin{align*}
		0 &= \inf \theta((\vec{\psi}(L))) \\
		&= \min_{i=1,\dots,k} \theta \left(\vec u_0^i + \sum_{j=1}^{n_j} \nats \vec u^i_j \right) \\
		&\geq \min_{i=1,\dots,k}  \min_{j=1, \dots, n_j} \{ \theta \left(\vec u_0^i \right), \theta \left( \nats \vec u^i_j \right) \} \\
		&= \min_{i=1,\dots,k}  \min_{j=1, \dots, n_j} \{ \theta \left(\vec u_0^i \right), \theta \left( \vec u^i_j \right) \} \\
		&> 0,
	\end{align*}
	providing our desired contradiction.

	Finally, let us show Property 1 and 2 for $\theta$. It is clear why Property 1 holds. As for Property 2.
		
	Indeed, let $\vec{s}(t) = \vec{v} + t\vec{w}$. Define $$\Theta(t) := \theta(\vec{s}(t)),$$ 
	such that $\Theta(0) = \theta(\vec{v})$ and $\lim_{t \to \infty} \Theta(t) = \theta(\vec{w})$. 	Assume w.l.o.g that both $\vec v, \vec w$ are in $P_\lambda$, that is, both $y(\vec v), y(\vec w) \geq 0$. Then, $y(\vec s(t) = y(\vec v) + t y(\vec w) \geq 0$ for $t \geq 0$, such that $\Theta(t) \in [0, \pi]$ and is thus continuous for $t \geq 0$. 
	
	We will show that $\Theta$ is monotone on the interval $t \in [0, \infty)$ by taking its derivative with respect to $t$. 
	
	Let $X(t) = x(\vec s(t))$, $Y(t) = y(\vec s(t))$, and $R(t) = \|\vec s\| = \sqrt{X^2(t) + Y^2(t)}$.

	\begin{align*}
		\Theta'(t) &= \frac{d}{dt} \theta(\vec s(t)) \\
		&= \frac{d}{dt} \atanTwo(Y(t), X(t)) \\
		&= \atanTwo_y Y' + \atanTwo_x X' \\
		&= \frac{XY' - YX'}{R^2} \\
		&= - \frac{\|\vec v\| \|\vec w\|}{\|\vec s\|^2} \sin(\theta(\vec v) - \theta(\vec w))
	\end{align*}

	Thus, on $t \in [0, \infty)$ 
	\begin{align*}
		\Theta' & \geq 0 & \sin(\theta(\vec v) - \theta(\vec w)) \leq 0 \implies \theta(\vec v) \leq \theta(\vec w) \\
		\Theta' & \leq 0 & \sin(\theta(\vec v) - \theta(\vec w)) \geq 0 \implies \theta(\vec w) \leq \theta(\vec v).
	\end{align*}

	Since $\Theta(0) = \theta(\vec{v})$ and $\lim_{t \to \infty} \Theta(t) = \theta(\vec{w})$, we have that for $t = 1$, 
	$$\min\{ \theta(\vec{v}), \theta (\vec{w}) \} \leq \theta(\vec{v} + \vec{w}) \leq \max \{ \theta(\vec{v}), \theta(\vec{w}) \}$$
	as claimed. 
		
 The case for $P'_\lambda$ is essentially the same. 
\end{proof}

Proposition \ref{prop: Zsq-rational-reg} together with Proposition \ref{prop: Zsq-irrational-not-CF} show Theorem \ref{thm: Zsq-P-classified}.

%% file: chap/research-intro.tex
\chapter{Left-orders and formal languages}\label{chap: research-intro}

The goal of this chapter is to recap everything we have seen so far, explicitly stating how we will study left-orders via formal languages.\sidenote{At the risk of some repetition.} We will also provide additional context for this area of study by linking our area of interests to the study of the Word Problem and automatic groups.\sidenote{Moreover, my hope with this chapter is that this offers some justification or some clarity to our choice of definitions and of the research direction I will be take for granted when presenting the results of my research. These notions at first may seem arbitrary, because they \emph{are} ultimately very much arbitrary and inspired from these past concepts without, of course, the future knowledge that our own theory will be as successful as those of the past.} Finally, we will prove some basic properties of languages which represent positive cones, setting up the terrain for the presentation of our research results. 

\section{The setup}

We start by reiterating our research problem in full. Let $G$ be a finitely generated left-orderable group with presentation $$G = \langle X \mid R \rangle.$$ For computational purposes, group elements are encoded as words in $X^*$, where $X$ as an alphabet will be assumed to be symmetric even if not written explicitly in the presentation, that is, the alphabet will always be of the form $$X = X\inv.$$ Associated to the generating set $X$ is an \emph{evaluation map} $\pi: X^* \to G$ which evaluates word in $X^*$ as elements of $G$. A \emph{word} $$w = x_1 \dots x_n$$ with $x_i \in X$ for $i \in \{1, \dots, n\}$ is evaluated as the product of its letters in $X$ viewed as group generators for $G$, $$\pi(w) = \pi(x_1) \dots \pi(x_n).$$ For succinctness, we may denote the generating set for $G$ as a tuple $(X, \pi)$. 

We are interested in studying left-orderable groups in terms of formal languages. That is, we want to find a left-order $\prec$ and an algorithm that takes as input two elements $g_1, g_2 \in G$ and decides whether $g_1 \prec g_2$. It is natural as this stage to encode $g$ and $h$ as words $w_g$, $w_h \in X^*$ and the problem naturally becomes finding an algorithmic solution to whether $$\pi(w_g) \prec \pi(w_h),$$ as in Example \ref{ex: LO-K2} in Chapter \ref{chap: LO}. 

Recall from Chapter \ref{chap: left-orderable-groups} that a left-order on a group $G$ is equivalent to a positive cone, a subset of the elements of $G$. That is, whether $g_1 \prec g_2$ is equivalent to deciding whether $1 \prec g_1\inv g_2$, i.e. whether $$g_1\inv g_2 \in P := \{g \in G \mid g \succ 1\}.$$ Translating this in terms of formal language means that an equivalent problem is finding a formal language $L$ which evaluates to $P$, $$\pi(L) = P.$$ 
\begin{defn}[Positive cone language]\label{def: PCL}\index{positive cone language}
Let $G$ be a finitely generated group with presentation $\langle X \mid R \rangle$. We say that $L \subset X^*$ is a \emph{positive cone language} (PCL) if there exists a positive cone $P$ of $G$ such that $\pi(L) = P$. \end{defn}

In addition, we want the membership problem in $L$ to be ``easily'' decidable according to the Chomsky hierarchy. For the scope of this thesis, this will generally mean regular or one-counter. Concretely, this will take the form of comparing elements in a left-order by remembering a finite number of things and keeping track of the sign of an integer sum.\sidenote{See Example \ref{ex: pda-int-counter} for a pushdown automaton which keeps track of an integer sum and accepts if the sum is positive.}

\begin{defn}[$\cC$-positive cones]
	Let $\cC$ be an AFL family in the Chomsky hierarchy. We say that $G$ admits $\cC$-left-orders or $\cC$-positive cones if it admits a positive cone $P$ such that there exists a positive cone language $L \in \cC$ for $P$. 
\end{defn}

We will denote the family of regular languages by $\Reg$ and the family of one-counter languages by $\onecounter$. Thus, we are looking for $\Reg$- and $\onecounter$-left-orders. 

\begin{ex}\label{ex: prelim-Z-reg}
	Let $X = \{x, x\inv\}$ and $\bZ = \langle x \rangle$.\sidenote{Notice that $X$ is assumed to be symmetric even if it is not written as such in the presentation of $\bZ$, which only has generator $x$. For the sake of brevity, we will often only write the positive generators for the presentation in the sequel, but still assume by abuse of notation that the language alphabet $X$ is symmetric.} Consider the positive cone $P = \{x \in \bZ \mid x \succ 0\}$ of Example \ref{ex: LO-P-Z}. The regular expression $L = x^+ = \{x^n \mid n > 0 \}$ is a positive language for $P$. Thus, we say that $\bZ$ admits $\Reg$-positive cones.
	
	Intuitively, this ``easy'' to check because we only need to verify that the input characters are always positive (that is, we only need to remember to always allow $x$-character inputs and never $x\inv$-character inputs).  
\end{ex}

\begin{ex}\label{ex: prelim-Z-1C}
	Consider $X, \bZ$ and $P$ as before in the previous example (\ref{ex: prelim-Z-reg}), and let $\pi: X^* \to \bZ$ be the evaluation map. Then, the preimage language $L := \pi\inv(P)$ is also a positive cone language for $P$. We have shown in Example \ref{ex: pda-int-counter} that this language is one-counter without being regular. Thus, $\bZ$ admits $\onecounter$-left-orders which are not in $\Reg$. 
	
	Intuitively, this preimage positive cone is ``easy'' to check because all we need to do is keep track of the sign of the integer sum of the exponent of the input string as we read it. 
\end{ex}

Note that Example \ref{ex: prelim-Z-1C} shows that while taking the full pre-image of a positive cone under $\pi$ always gives a positive cone language, it might not be of minimal language complexity.

Moreover, as we have seen in Section \ref{sec: research-P-Z2} with the positive cones of $\bZ^2$, groups can contain positive cones of degrees of decidability. This was first observed with the positive cones of solvable Baumslag-Solitar groups $\BS(1,q)$ in \cite{AntolinRivasSu2021}, which we review in Chapter \ref{chap: closure-extension}. Thus, some care must be taken when selecting positive cones of a group.

A strategy we will often employ to find positive cone languages of low complexity is to construct first an explicit left-order for $G$ that looks intuitively ``easy'' to understand\sidenote{such as only having to keep track of a finite number of things, such as the positivity of a $x$ as in Example \ref{ex: prelim-Z-reg}, or keeping track of an integer count, such as in Example \ref{ex: prelim-Z-1C}}, construct the corresponding positive cone language of normal forms and check that there is automaton of low complexity generating it, then use a lower bound on the complexity of positive cones in $G$ to show that $L$ is indeed of minimal complexity. 

\begin{ex}\label{ex: prelim-K2}
	In Chapter \ref{chap: LO} Example \ref{ex: LO-K2}, we saw that the Klein bottle group $K_2 = \langle a,b \mid bab= a\rangle$ had a subgroup of index $2$ isomorphic to $\mathbb{Z}^2$ given by $\bZ^2 = \langle a^2, b \rangle$, and that elements could be lexicographically left-ordered by writing them in the normal form $$a^m b^n, \quad m,n \in \bZ.$$
	We saw in Example \ref{ex: LO-P-K2} that it induced the positive cone $$P = \{a^m b^n \mid m > 0 \text{ or } m = 0, n > 0\},$$ %
	which can be represented by the regular expression $$L = a^+(b \mid b^{-1})^* \cup b^+.$$  
	
	We observe that here, once the elements are written in normal form, the finite memory of the corresponding FSA was used to check that the input is indeed always in lexicographic order, and that the sign of the input character strings are appropriate relative to their place in the lexicographic order. Arguably, the bulk of the computational work is done prior to verifying the left-order. Thus, we should think of left-orders of low complexity as left-orders that are easy to \emph{verify}, but not necessarily put in a form that is easily verifiable.\sidenote{A problem which captures this ``solving vs verifying'' nuance that the reader may already be familiar with is the famous P = NP problem. Essentially, it asks if a problem which can be verifiable in polynomial time (NP) can also be solved in polynomial time (P). It is widely believed that this is not the case.} In the pre-image left-order case, the computational work of evaluating the element is offloaded to the verifying step, hence increasing the complexity as seen in Example \ref{ex: prelim-Z-1C} versus Example \ref{ex: prelim-Z-reg}. 
\end{ex}

Next, let us look at some properties of positive cones languages we will be using throughout our thesis. 

\section{Closure properties of positive cone languages}

To meaningfully talk about positive cone complexity, it is imperative that such a complexity must be independent of the underlying generating set for the group containing the positive cone. We first show it is indeed the case.

\begin{prop}\label{prop: pcl-comp-indep-gen-set}
The complexity of a positive cone language is independent of generating set in the sense that if $\cC$ is a class of languages closed under homomoprhism and inverse homomorphism, and if $(X, \pi_X)$ and $(Y, \pi_Y)$ are two generating set for $G$ with a positive cone $P$, then there is a positive cone language $L_X \subset X^*$ for $P$ of complexity $\cC$ if and only if there exists a positive cone language $L_Y \subset Y^*$ for $P$ of complexity $\cC$. 
\end{prop}

Proposition \ref{prop: pcl-comp-indep-gen-set} is just a corollary of the following well-known result. 

\begin{lem}\label{lem: independence gen set}
Let $\mathcal{C}$ be a class of languages closed under homomorphism and inverse homomorphism. Let $(X,\pi_X)$ and $(Y,\pi_Y)$ be two finite generating sets of a group $G$. Let $S\subseteq G$ be any subset.

There exists a language $L_X\subseteq X^*$ in $\cC$ such that $\pi_X(L_X)=S$ if and only if there is a language $L_Y\subseteq Y^*$ in $\cC$ such that $\pi_Y(L_Y) = S$.
\end{lem}

\begin{proof}
For the $(\imp)$ direction, let $X = \{x_1, \dots, x_n\}$.
Since $X$ and $Y$ are both generating sets of $G$, we have that for each $x_i \in X$, there exists a word $w_i \in Y^*$ such that
$\pi_X(x_i) = \pi_Y(w_i(x_i))$.
Let $h\colon X^* \to Y^*$ be the monoid homomorphism such that $x_i \mapsto w_i$. 
It follows from the definition that for each $v\in X^*$, $\pi_X(v)=\pi_Y(h(v))$.

Let $L_X \subseteq X^*$ be a language in $\cC$ such that $\pi_X(L_X)=S$.
Define $L_Y = h(L_X)$. Then $S = \pi_X(L_X) = \pi_Y(h(L_X)) = \pi_Y(L_Y)$. Moreover, $L_Y$ belongs to $\cC$, since the $\cC$ is closed under homomorphism. 
 
For the $(\impliedby)$ direction, we may interchange $X$ and $Y$ in the previous proof. 
\end{proof}

To show Proposition \ref{prop: pcl-comp-indep-gen-set}, we can substitute $S$ in the statement of Lemma \ref{lem: independence gen set} with any positive cone $P$ in a left-orderable group $G$. Any complexity $\cC$ in the Chomsky hierarchy is closed under homomorphisms, thus satisfying the condition of the statement of Lemma \ref{lem: independence gen set}.

Moreover, we have seen in Chapter \ref{chap: LO}, Remark \ref{rmk: LO-P-sym} that any positive cone is interchangeable with its negative cone. The hope is that if we have a $\cC$-positive cone $P$, then $P\inv$ is also a $\cC$-positive cone. This is indeed the case for any $\cC$ in the Chomsky hierarchy thanks to its closure under reversal.\sidenote{See Chapter \ref{chap: informal-lang}, Section \ref{sec: closure-formal-lang} for a brief discussion of this.} 

\begin{lem}\label{lem: negative cone in cC}
Let $(X,\pi)$ be a finite generating set of $G$ and $P\subset G$.
Let $\cC$ be a class of languages closed by homomorphisms, inverse homomorphisms and reversal. Then
\begin{enumerate}
\item [(i)]
there is a language $L\in \cC$ such that $\pi(L)=P$ if and only 
there is a language $L'\in \cC$ such that $\pi(L')=P^{-1}$, 

\item[(ii)]  
 $\pi^{-1}(P)\subseteq X^*$ is in the class $\cC$ if and only if 
 $\pi^{-1}(P^{-1})\subseteq X^*$ is in the class $\cC$.
\end{enumerate}
\end{lem}
\begin{proof}
Since $\cC$ is closed by homomorphisms and inverse homorphisms, Lemma \ref{lem: independence gen set} implies that the properties claimed in (i) and (ii) are independent of the generating set. 
So we will assume that $X\subseteq G$ is a finite generating set closed under taking inverses.

Let $f\colon X^* \to X^*$ be the map sending $x_1 \dots x_n \mapsto x_n^{-1} \dots x_1^{-1}$. 
Then $f$ is a composition of the homomorphism map sending $x \mapsto x^{-1}$ and the reversal map.
Moreover $f^2 = id \colon  X^* \to X^*$. 
In particular, $L\subseteq X^*$ is in $\cC$ if and only if $f(L)$ is in $\cC$.\\

Note that we have that $P^{-1}=\pi(f(L))$ and so (i) follows. 
Also $\pi^{-1}(P^{-1})=f^{-1}(\pi^{-1}(P))$ and (ii) follows.
\end{proof}

The theory of formal languages and left-orderable groups is relatively new. To give the problem more context, let us introduce two computational problems in geometric group theory with more antecedent, and use them to compare them to the problem we are studying.

\section{The context of the Word Problem}\label{sec: WP}

The Word Problem, originally posed by Max Dehn in 1911, is one of the oft-cited problems in geometric group theory. Let $G$ be a finitely generated group with finite generating set $X$ and evaluation map $\pi: X^* \to G$. The Word Problem is about algorithmically deciding whether two words $w$ and $v$ in $X^*$ represent the same element in $G$, i.e. whether $\pi(w) = \pi(v)$. For simplicity, we often pose the equivalent question of deciding whether $\pi(v)\inv \pi(w) = 1_G$, or simply whether a word $w$ is such that $$\pi(w) = 1_G.$$ 

The formulation of the Word Problem can then be posed in terms of formal languages. We want to determine the automaton of lowest complexity which can decide the membership of a word in $\pi\inv(1_G)$. The formal language approach to the Word Problem has yielded some successful results. 

In the 1950s, Novikov and Boone independently showed that $\pi\inv(G)$ is not a decidable language for some $G$ and $X$, and further work focused on which $G$ and $X$ have decidable Word Problem (when $\pi\inv(G)$ is a decidable language). This new question lead to a partial classification of the correspondence between the algebraic structure of a group and the degree of decidability of its Word Problem. A few relevant results are given in Table \ref{tab: wp-history}, where all results assume that the Word Problem is on $G$, and that $G$ is a finitely generated group. 
\begin{center}
\begin{table*}\label{tab: wp-history}
\begin{tabulary}{\textwidth}{p{0.3\textwidth} p{0.1\textwidth} p{0.3\textwidth} p{0.3\textwidth}}
Word Problem complexity & Implication &  Conditions on group $G$ & Due to \\
\bottomrule 
Regular & $\iff$ & Finite &  A. V. Anisimov (1971) \\
\midrule
Context-free $\iff$  
deterministic context-free & $\iff$ & Virtually free & Combined results of J. Stallings (1971), A. V. Anisimov (1972), D. E. Muller and P. E. Schupp (1983), and M. J. Dunwoody, D. E. Muller and P. E. Schupp (1985) \\ 
\midrule
One-counter $\iff$ 
deterministic one-counter & $\iff$ & Virtually cyclic & T. Herbst (1991)\\ 
 \\
\midrule
Deterministic context-sensitive & $\impliedby$ & Finitely generated subgroup of an automatic group & M. Shapiro (1994) \\
\midrule 
Decidable & $\iff$ & There exists a finitely generated simple group $H$ and a finitely presented group $K$ such that $G \leq H \leq K$. & W. W. Boone and G. Higman (1974), generalized by R. J. Thompson (1980)
\end{tabulary}
\end{table*}
\end{center}

The statements of Table \ref{tab: wp-history} are taken from \cite{StewartThomas1999}.

\section{The Co-Word Problem and the many Positive Word Problems}

Whether $\pi(w) \not= 1_G$ is known as the \emph{Co-Word Problem}, with the associated language $$L = \{w \in X^* \mid \pi(w) \not= 1_G\} = X^* - \pi\inv(1_G),$$ which is simply the complement of the Word Problem. 

In our research, we ask whether $\pi(w) \succ 1_G$ for a given left-order $\prec$. We name this problem a \emph{Positive Word Problem} because $w$ maps to an element in a positive cone. 
Unlike the notion of equality, positivity relative to a left-ordering is far from unique in any left-orderable group, as we have explored in Chapter \ref{chap: left-orderable-groups}.

However, if $\pi(w) \succ 1_G$ for any left-order $\prec$, then $\pi(w) \not= 1_G$. Solving any Positive Word Problem completely, meaning knowing the entire preimage of a positive cone $P_\prec$, $\pi\inv(P_\prec)$, implies solving the Co-Word Problem, because 
\begin{equation}\label{eqn: prelim-cC-coword}
	X^* - \pi\inv(1_G) = \pi\inv(P_\prec) \sqcup \pi\inv(P_\prec\inv)
\end{equation}
for any positive cone $P_\prec$ (since $G = 1_G \sqcup P_\prec \sqcup P_\prec\inv)$. We note that by Lemma \ref{lem: negative cone in cC}, $\pi\inv(P_\prec)$ is of class $\cC$ if and only $\pi\inv(P\inv_\prec)$, and their union will also be of class $\cC$, meaning that the Positive Word Problem and the Co-Word problem are of the same complexity under the Chomsky hierarchy. Note that this does not imply that the Word Problem is of complexity $\cC$ since AFLs are not generally closed under complementation,\sidenote{in Chapter \ref{chap: pushdown-automata}, Section \ref{sec: non-closure-context-free} we have seen that context-free languages are not closed under complementation} with the exception of regular languages (see Lemma \ref{lem: reg-closed-under-complement}).

Our work does not actually concern itself with solving the Positive Word Problem completely, but rather finding a normal form language $L$ inside $\pi\inv(P_\prec)$ for some $\prec$ such that it is generally easier to decide membership in $L$, as illustrated by the contrast of Example \ref{ex: prelim-Z-reg} versus Example \ref{ex: prelim-Z-1C}, and discussed at the end of Example \ref{ex: prelim-K2}. 

In fact, there isn't anything meaningful we can say about preimage languages of positive cones which are regular. 

\begin{prop}\label{prop: regular-preimage}
A finitely generated group $G$ has a regular preimage left-order $\pi(P_\prec)$ if and only if it is trivial. \cite{AntolinRivasSu2021}
\end{prop}
\begin{proof}
Let $G$ be a finitely generated group with generating set $X$ and evaluation map $\pi$. If $G$ is trivial, the positive cone is the empty set and the preimage language is empty, which is a regular language ``accepted'' by a finite state automaton with no accept state.

Suppose that $\pi^{-1}(P)$ is a regular language. Then the complement $(\pi\inv(1_G))^c$ is also regular by the discussion under Equation \ref{eqn: prelim-cC-coword}, and by closure of regular languages under complement (Lemma \ref{lem: reg-closed-under-complement}), $\pi\inv(1_G)$ is also a regular language. Then, by Anisimov's theorem $G$ must be finite. Finally, a finite left-orderable group must be trivial by Non-Example \ref{non-ex: finite-not-LO}.
\end{proof}

\section{The context of automatic groups}
As emphasised in the last section, our framework for studying left-orderable groups follows the tradition geometric group theory of capturing group structure using normal forms. An important example comes from the theory of automatic groups. ``The basic idea behind automatic groups is to extend the notion of multiplication table from finite to discrete infinite groups.'' (See \cite{Choffrut2002} for an overview of the theory). Automatic groups are defined by two key properties. First, the group elements are encoded as a regular language $L \subseteq X^*$ where $L$ is a language of normal forms. Second, we can ``check, by means of a finite state automaton, whether two words in a given presentation represent the same element of not, and whether or not the elements they represent differ by right multiplication by a single generator. The totality of these data - the generators and the automata - constitute an automatic structure for a group.''\cite{Epstein1992}. 
More precisely, the formal definition of an automatic group is given as follows. 
\begin{defn}\label{defn: aut-group}
Let $G$ be a group. Then $(X,L)$ is an \emph{automatic structure} on $G$ if
\begin{enumerate}
\item $L \subseteq X^*$ is recognized by a finite state automaton and $\pi(L) = G$.
\item There are finite state automata $\bA_x$ for each $x \in X \cup \{\epsilon\}$ (where $\epsilon$ is the empty word) which accept pairs of words $(w_g, w_h)$ if and only if $$\pi(w_gx) = \pi(w_h)$$ and $w_g, w_h \in L$. 
\end{enumerate}
\end{defn}

It is worth noting despite only requiring $L$ to surject to $G$ instead of bijecting as one could expect from a normal form, this is not truly introducing a misnomer since if a group is automatic there exists a regular language of bijective normal forms $L'$ representing the elements of $G$ such that $(X,L')$ is an automatic structure. (See for example \cite[Theorem 5]{Choffrut2002} for a proof.) However, it is not the case that this property is always retained in extensions of automatic groups.

The key properties of automatic groups are as follows (according to \cite{Choffrut2002}). 
\begin{enumerate}[i)]
\item The notion of automatic is intrinsic and does not depend on the generating set. 
\item The family of automatic groups is closed under several operations such as ``direct sums, finite extensions, finite index subgroups, free products and some particular amalgamated free products'' \cite{CannizzoKhoussainovKharlampovich2014}. 
\item The Word Problem is solvable in quadratic time. 
\item Many natural groups are automatic, such as ``hyperbolic groups, braid groups, mapping class groups, Coxeter groups Artin groups of large types'' \cite{CannizzoKhoussainovKharlampovich2014}.
\end{enumerate}

Automatic groups were introduced by Thurston, Cannon, Gilman Epsein and Holt with the motivation of ``understanding the fundamental group of compact $3$-manifolds and to approach their natural geometric structures via the geometry and complexity of the optimal normal forms; and to make them tractable for computing'' \cite{CannizzoKhoussainovKharlampovich2014}. Indeed, ``the fundamental group of a $3$-manifold $M$ is automatic if and only if none of the factors in the prime decomposition of $M$ is a closed manifold modelled on one of  the geometries of \emph{Nil}\sidenote{The class of differentiable manifolds which have a transitive nilpotent group of diffeomorphism acting on it.} or \emph{Sol}\sidenote{The class of homogeneous spaces of connected solvable Lie groups.}'' \cite{BridsonGilman1996}.

Later, the definition of automatic groups was extended in different ways. Notably, one such extension captures the entire subclass of fundamental groups of compact geometrizable $3$-manifolds\sidenote{The geometrization conjecture, fully proven in 2006, says that $3$-manifolds can be decomposed into pieces that have one of eight types of geometric structures. I found this page useful as a quick overview: \url{https://math.stackexchange.com/questions/3413924/what-does-it-mean-for-a-manifold-to-be-geometrizable}.} by defining $L$ in the automatic structure as being a bijective indexed language \cite{BridsonGilman1996}, which is a degree of decidability between context-free and context-sensitive. Another example of an extension was done to broaden the definition of automatic while focusing on preserving their nice computational properties \cite{CannizzoKhoussainovKharlampovich2014}. 

In some sense, a direction of the research presented here is to enhance structure automatic structures for left-orderable group such that once we have a language of normal forms solving the Word Problem, we can select from it a positive cone language on which we easily decide membership. 

 We have the following four properties for left-orderable groups with positive cones of complexity $\cC$ which are somewhat analogous to the automatic group properties.

\begin{enumerate}[i)]
\item The complexity of a positive cone is intrinsic and does not depend on generating sets. 
\item The family of groups $\cC$ positive cone where $\cC$ is an AFL class (including regular and one-counter) is closed under direct sums, finite extensions, and wreath products. Moreover, the family of groups with regular positive cones is closed under finite index subgroups (see Chapter \ref{chap: closure-finite-index}), and the free products of left-orderable groups with regular positive cones has one-counter positive cone (see Chapter \ref{chap: cross-Z}). 
\item If $w$ is in a positive cone language of complexity $\cC$, then we know that it is not equal to the identity using an automaton of complexity $\cC$. This is related to the Co-Word Problem. 
\item Some natural groups are already known to be left-orderable with a regular or one-counter positive cone, such as free abelian groups, free groups, left-orderable virtually poly-$\bZ$ groups (see Chapter \ref{chap: closure-extension}, certain solvable Baumslag-Solitar groups (see Chapter \ref{chap: closure-extension}, and Artin groups whose defining graphs are trees (see Chapter \ref{chap: cross-Z}). 
\end{enumerate} %

Arguably the weakest comparison point is Property iii). While automatic groups have a Word Problem which can be solved in quadratic time, we get a vague relation to the Co-Word Problem for positive cone languages, and the ``normal forms'' or the surjection from a positive cone language $L$ mapping to $P$ does not necessarily extend to a bijection. Moreover, automatic structures come with the extra property that their language $L$ is quasigeodesic\sidenote{A language is quasi-geodesic if there exists $(\lambda, k)$ such that every word $w$ in the language induces a path $\hat w: [1, n] \to G$ in the Cayley graph such that the length of the path $|\hat w|$ satisfies the inequality $\frac{1}{\lambda} n - k \leq |\hat w| \leq \lambda n + k$} (see \cite[Theorem 3.3.4]{Epstein1992}). Having a group for we can solve the Word Problem (such as an automatic group) and selecting our positive cone language as an easily decidable subset of the normal forms given by the Word Problem overcomes this difficulty.

On the other hand, we do get some properties for free with positive cone languages that are not easily obtainable in automatic groups. An open question in automatic group theory is whether $G$ being automatic implies that it is bi-automatic. That is if $L$ is an automatic structure with respect to $X$ for $G$, is $$L\inv = \{x_n\inv \dots x_1\inv \mid x_1 \dots x_n \in L\}$$ with respect to $X$ also an automatic structure for $G$? Lemma \ref{lem: negative cone in cC} that in  left-orderable groups, we get that if $L$ is a positive cone language of complexity $\cC$ for $P$, then $L\inv$ is also a positive cone language of complexity $\cC$ for $P\inv$ for any $\cC$ in the Chomsky hierarchy. 

Although very much at its infancy, we hope to offer a theory of left-orderable groups which admit left-orders with an attached computational structure that is interesting to the reader, and perhaps replicate some of the success of the theory of automatic groups.

The next part of the thesis concerns some specific background needed to understand particular chapters. We encourage the reader to go straight to the results in Part \ref{part: results}, and come back to Part \ref{part: specific-background} as needed.

%% file: chap/hyperbolic.tex
\chapter{Hyperbolic groups}\label{chap: hyperbolic}
In this chapter, we will introduce the topic of hyperbolic groups, a class of groups often discussed in geometric group theory. The goal is to give a small survey of the topic with just enough information for the undergraduate reader to follow our main results in Chapter \ref{chap: closure-finite-index}. We recommend the textbook of Bridson and Haefliger \cite[Part III, H]{BridsonHaefliger1999} for a more complete reference. 

\section{Hyperbolicity}
A hyperbolic space is a space of negative curvature, i.e. a space where two parallel lines diverge. Popular examples are the saddle-shaped hyperbolic manifold (Figure \ref{fig: hyp-saddle}), and the Poincar\'e disk (Figure \ref{fig: hyp-poincare-disk-geod}). 

\begin{figure}[h]{
\includegraphics{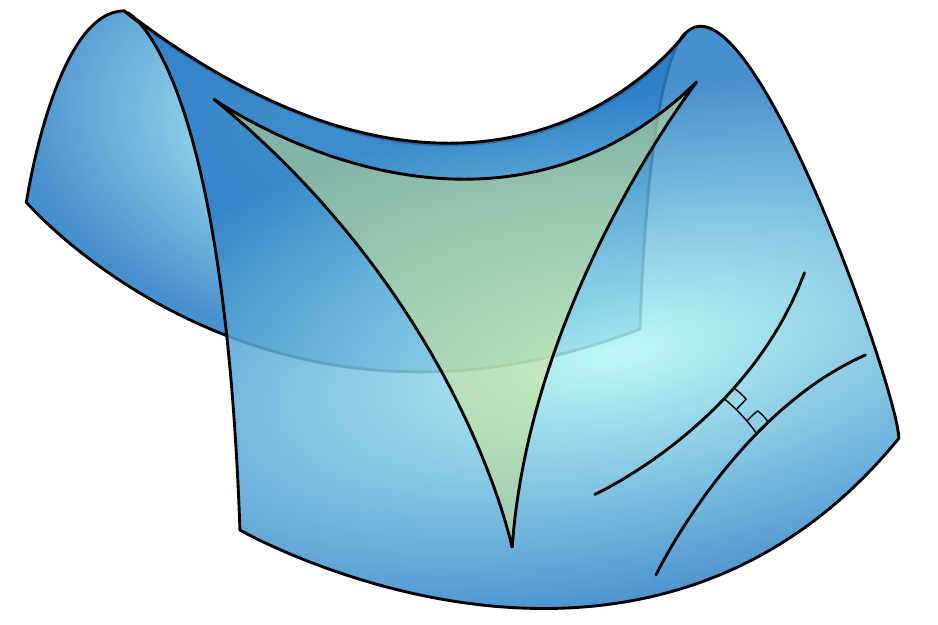}
}
\caption{A triangle immersed in a saddle-shape plane (a hyperbolic paraboloid), along with two diverging ultra-parallel lines. Source: Wikipedia, author unknown. \url{https://en.wikipedia.org/wiki/Hyperbolic_geometry\#/media/File:Hyperbolic_triangle.svg}}
\label{fig: hyp-saddle}
\end{figure}

\begin{figure}[h]{
\includegraphics{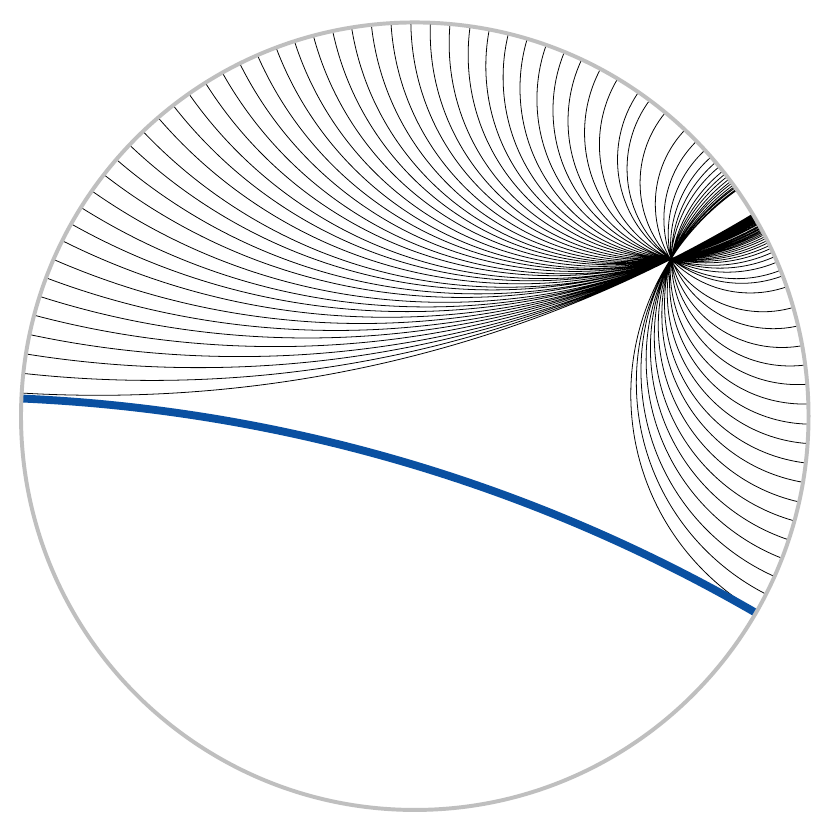}
}
\caption{Poincaré disk with hyperbolic parallel lines. A geodesic triangle is formed. Source: Wikipedia, figure by Trevorgoodchild. \url{https://en.wikipedia.org/wiki/Poincar\'e_disk_model\#/media/File:Poincare_disc_hyperbolic_parallel_lines.svg}}
\label{fig: hyp-poincare-disk-geod}
\end{figure}

One way to capture this sort of space is by the thinness of its triangles. A $\delta$-hyperbolic space, attributed to Gromov, is a space where every edge of a triangle is contained in the fixed neighbourhood of size $\delta$ of the other two. 

\begin{figure}[h]{
\includegraphics{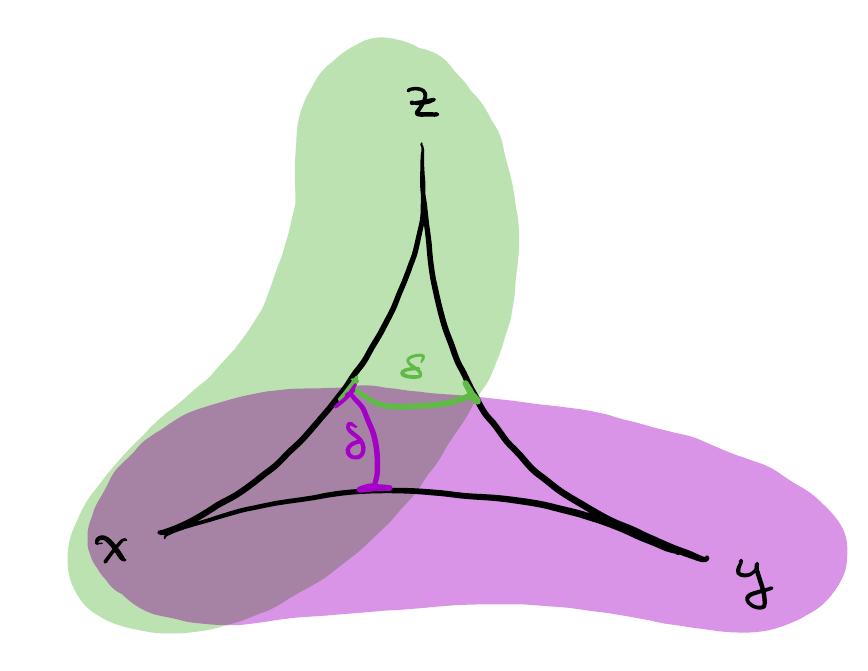}
}
\caption{A $\delta$-hyperbolic triangle formed by points $x,y,z$. The $\delta$-neighborhood of geodesic segments $[x,y]$ (in purple) and $[x,z]$ contains the geodesic segment $[y,z]$.}
\label{fig: hyp-delta-triangle}
\end{figure}

Formally, we use the following definitions to define what we mean by a triangle, and what we mean by it being $\delta$-thin. 

\begin{defn}	[Geodesic space]
	A \emph{geodesic space} is a metric space $X$ such that for all points $x, y \in X$, we have a path $\gamma$ from the interval $I = [0, d(x,y)]$ to the space $X$ with endpoints $x,y$ such that $d(\gamma(a), \gamma(b)) = \vert a - b \vert$ for all $a,b \in I$ (thus, $\gamma$ is an an isometric embedding of the interval $I$, see Figure \ref{fig: hyp-geodesic-embedding}).
\end{defn}

\begin{figure}[h]{
\includegraphics{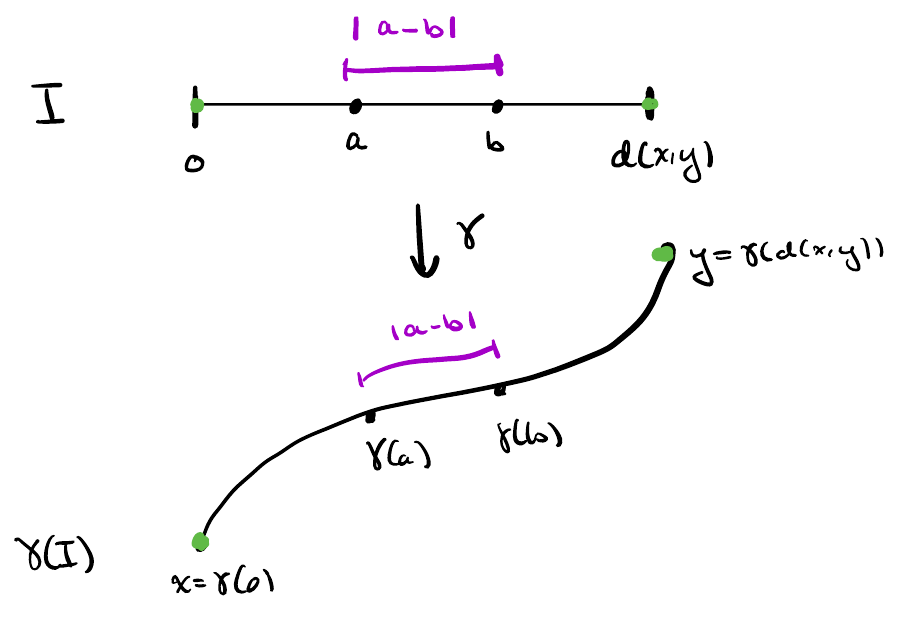}
}
\caption{Illustration of a geodesic embedding $\gamma: I \to X$ connecting $x,y \in X$. The geodesic $\gamma$ embeds every subinterval $[a,b] \subseteq I$ isometrically.}
\label{fig: hyp-geodesic-embedding}
\end{figure}

Let $N_\delta([x,y])$ be the $\delta$-neighbourhood of $[x,y]$, that is, the union of all the $\delta$ balls centered in $[x,y]$, $N_\delta([x,y]) := \bigcup_{m \in [x,y]} B_\delta(m)$.

\begin{defn}[$\delta$-hyperbolic space]
	A geodesic metric space $X$ is \emph{$\delta$-hyperbolic} if there exists a $\delta \geq 0$ such that for every geodesic triangle with vertices $x,y,z$ and edges formed by geodesics $[x,y], [y,z], [z,x]$ and all points $m$ in one of the edge, say $m \in [x,y]$ we have that $m \in N_\delta([y,z]) \cup N_\delta([x,z])$ and similarly for $m \in [y,z], m \in [x,z]$. 
\end{defn}
With this, we can verify that Figure \ref{fig: hyp-delta-triangle} satisfies the definition of a $\delta$-thin triangle. Let us do this for the Poincar\'e disk $\mathbb{D} = \{z \in \mathbb{C} : \vert z \vert < 1\}$, which is an open disk with negative curvature and metric defined by arc length $ds^2 = \frac{4(dx^2 + dy^2)}{(1 - (x^2 + y^2))^2}$.

A geodesic triangle are illustrated in Figure \ref{fig: hyp-poincare-disk-geod}. One could compute that $\mathbb{D}$ is $\delta$-hyperbolic with $\delta = \log(\sqrt{2} + 1)$.\sidenote{One could do this by starting with a triangle with all vertices on the boundary and then use the fact that $\text{PSL}_2(\mathbb{R})$ acts transitively on any triples in $\mathbb{D}$ to conclude the $\delta$-thinness of all the other triangles, then use the fact that all ideal triangles have the same area on $\mathbb{D}$.}

This notion of $\delta$-hyperbolicity generalizes the Riemannian notion strictly, meaning for any complete Riemannian manifold with negative curvature bounded by $K$, there exists a $\delta = \delta(K)$ for which the manifold is $\delta$-hyperbolic. 

One way in which the $\delta$-hyperbolic abstraction is particularly useful for our purposes is that we can apply this notion of hyperbolicity to objects which do not have a differential structure, such as graphs.

\subsection{Thin triangles in graphs}
The property of $\delta$-thinness hinges on the space and the metric. To emphasise this, let's do an example with a graph.  

\begin{ex}
Take graph $\Gamma$, given in Figure \ref{fig: hyp-graph-triangle}.  

\begin{figure}[h]{
\includegraphics{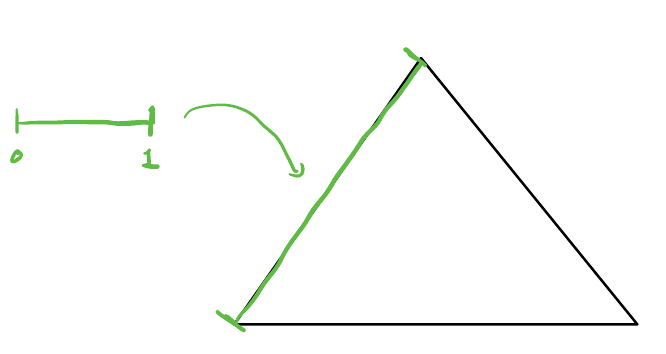}
}
\caption{
The graph $\Gamma$ can be realised as a geodesic space by embedding the interval $[0,1]$ on each of the three sides of the triangle.}
\label{fig: hyp-graph-triangle}
\end{figure}

We say that every edge has length $1$ by embedding the real interval $[0,1]$
 on each edge. It is straightforward to see that the space is geodesic.  
 
What is the minimal $\delta$ for $G$ to be hyperbolic? We need to remember that the space here is a graph, composed of edges and vertices: the inside of the triangle is not in our space. The only way to go from one point to another is passing through the edges until we hit another one. This is illustrated in Figure \ref{fig: hyp-graph-triangle-2}. 

\begin{figure}[h]{
\includegraphics{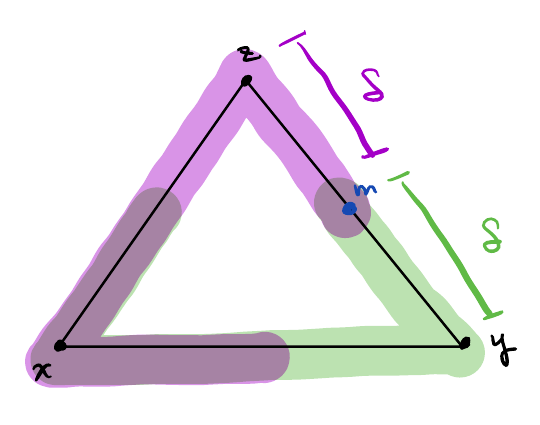}
}
\caption{
The graph $\Gamma$ can be realised as a geodesic space by embedding the interval $[0,1]$ on each of the three sides of the triangle.}
\label{fig: hyp-graph-triangle-2}
\end{figure}

To answer our question, any point on say $[x,y]$ is at most $1/2$ in distance away from a vertex $x,y$ or $z$, which is sure to be an element of $[y,z]$ or $[z,x]$. Thus $\delta = 1/2 + \epsilon$ (plus an $\epsilon > 0$ to contain the midpoint of each edge) in this example. Although this triangle does not actually look thin in the sense of Figure \ref{fig: hyp-delta-triangle}, it is important to remember that a ``thin'' triangle in a hyperbolic setting is only thin with respect to its given metric. 
\end{ex}

Let us look at two more examples. 

\begin{ex}
	Let the graph $\Gamma$ be a square as shown in Figure \ref{fig: hyp-graph-square}. We illustrate what happens when we take a triangle $[x,w,z]$ with neighborhood $\delta = 1 + \epsilon$ for $\epsilon > 0$.  

\begin{figure}[h]{
\includegraphics{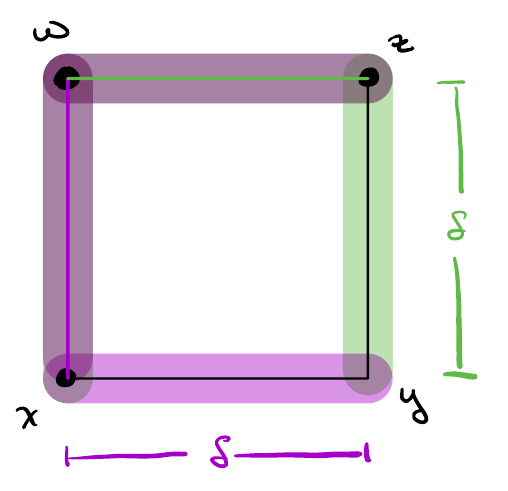}
}
\caption{
A square with vertices $w,x,y,z$. The $\delta$-neighborhoods with $\delta = 1 + \epsilon$ of sides $[x,w]$ (in purple) and $[w,z]$ (in green) contain the side $[x,z]$, passing by midpoint $y$.}
\label{fig: hyp-graph-square}
\end{figure} 
\end{ex}

\begin{ex}
Let $\Gamma$ be a tree, as illustrated in Figure \ref{fig: hyp-graph-tree}. In our illustration, we pick three labelled points $x,y,z$ and look at what happens when we form a triangle on the tree. Since there is only one path from $m$ to the vertices of the triangle, the graph is $\delta$-hyperbolic with $\delta = 0$. 

\begin{figure}[h]{
\includegraphics{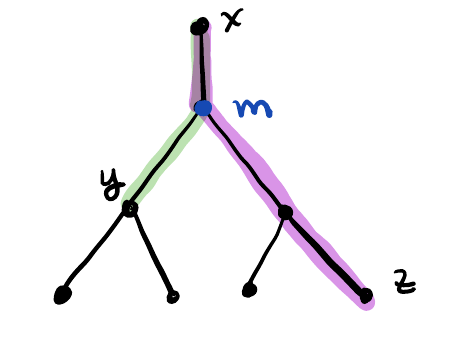}
}
\caption{
A tree with labelled vertices $x,y,z$. The union of the sides $[x,y]$ (in green) and $[x,z]$ (in purple) completely contains the side $[y,z]$. Thus, the tree is $0$-hyperbolic. This is because there is only one path from $m$ to $x,y$ or $z$.}
\label{fig: hyp-graph-tree}
\end{figure} 
\end{ex}

We leave the following statements as observations or exercises to the reader. 
\begin{exc}\label{exc: hyp-finite-graphs} All finite graphs $\Gamma$ are $\delta$-hyperbolic for $\delta = \text{diam}(\Gamma)$, the diameter of $\Gamma$. 
\end{exc}

\begin{exc} All graphs that are trees are $0$-hyperbolic. 
\end{exc}

\section{Coarseness}
Hyperbolic spaces are nice to study because they have nice large scale geometric properties. More precisely, their properties are invariant up to quasi-isometry, which is also known as the property of notion of being \emph{coarse}. 

\begin{defn}[Quasi-isometry]\label{defn: quasi-isometry}
	Let $(M_1, d_1)$, $(M_2, d_2)$ be two metric spaces, and $f: (M_1, d_1) \to (M_2, d_2)$. Then $f$ is called a \emph{quasi-isometry} if there exists constants $\lambda \geq 1, \epsilon \geq 0, \Delta \geq 0$ such that 
	\begin{enumerate}
		\item For every $a,b \in M_1$, $$\frac{1}{\lambda} d_1(a,b) - \epsilon \leq d_2(f(a),f(b)) \leq \lambda d_1(a,b) + \epsilon.$$
		\item For every $c \in M_2$, there exists a point $a \in M_1$ such that $$d(c,f(a)) \leq \Delta.$$
	\end{enumerate}
	If only condition $1$ is met, then $f$ is called a \emph{quasi-isometric embedding}. 
\end{defn}

\begin{lem}
	Suppose that metric space $(M_1, d_1)$ is quasi-isometric to $(M_2,d_2)$ by quasi-isometric constants $(\lambda,\epsilon)$ as in Definition \ref{defn: quasi-isometry}. If $(M_1, d_1)$ is $\delta$-hyperbolic, then $(M_2, d_2)$ is $\delta'$-hyperbolic where $\delta' = \delta'(\delta, \lambda, \epsilon)$. 
\end{lem}
A proof is available in \cite[Chapter III.H, Theorem 1.9]{BridsonHaefliger1999}. 

This notion of studying hyperbolic spaces up to some coarseness will apply to geodesics as well, which are relaxed to \emph{quasi-geodesics}.

\begin{defn}[Quasi-geodesic]
	Let $(M,d)$ be a metric space. For constants $\lambda\ge 1$ and $\varepsilon\ge 0$, a map
	\[
	q:I\to M \qquad (I\subset\mathbb{R}\ \text{an interval})
	\]
	is a \emph{$(\lambda,\varepsilon)$–quasi-geodesic} if for all $s,t\in I$,
	\[
	\frac{1}{\lambda}\,|s-t|-\varepsilon \;\le\; d\big(q(s),q(t)\big) \;\le\; \lambda\,|s-t|+\varepsilon.
	\]
	Equivalently, $q$ is a $(\lambda,\varepsilon)$–quasi-isometric embedding of $I$ (with the usual metric) into $M$.
	
	If $I=[a,b]$, $[0,\infty)$, or $\mathbb{R}$, one speaks of a quasi-geodesic \emph{segment}, \emph{ray}, or \emph{line}, respectively.
\end{defn}

One of the salient facts of hyperbolic spaces is that geodesics and quasi-geodesics sharing endpoints never stray too far from one another. Note that this is absolutely not true in flat Euclidean space!\sidenote{Take a rectangle in $\mathbb{R}^2$ with sides of length $a,b$, and a diagonal of length $c = \sqrt{a^2 + b^2}$. We know from trigonometry that $\frac{1}{\sqrt{2}}(a+b) \leq c \leq \sqrt{2}(a+b)$, so the sides connecting the endpoints of the diagonal together form two quasi-geodesics of length $a+b$. Since the diagonal can be made arbitrarily large, so can the distance between the sides and the diagonal.} Formally, the notion of ``never straying far'' is captured by the Hausdorff distance\sidenote{Intuitively, the Hausdorff distance measures how far two subsets are from one another.  It is the longest distance you can be forced to travel by an adversary who chooses a point in one of the two sets, from where you then must travel to the other set. For two subsets $S,R$ and a metric space with metric $d$, it is defined as $d_H(S,R) := \max \left\{ \sup_{s \in S} d(s,R), \quad \sup_{r \in R} d(S, r) \right\}$.} and the stability of quasi-geodesics is captured by the Morse Lemma.

\begin{lem}[Morse Lemma]
Let $\gamma$ be a quasi-geodesic with parameters $(\lambda, \epsilon)$ and endpoint $p,q$. Then $\gamma$ is $R(\lambda, \epsilon)$ bounded away from geodesic $[p,q]$ in Hausdorff distance.
\end{lem}
A proof is available in \cite[Chapter III. H]{BridsonHaefliger1999}.

\section{Hyperbolic groups}
For our purposes, $\delta$-hyperbolic spaces will always be Cayley graphs of infinite hyperbolic groups. 

Recall that the Cayley graph $\Gamma(G,X)$ is the graph obtained from a group $G$ with presentation $G = \langle X \mid R \rangle$. The vertices of the graph are the elements of $G$ i.e. $V(\Gamma) = \{g \mid g \in G\}$, and a directed edge $(g_1, g_2)$ with label $x$ exists if $\exists x \in X : g_1 x = g_2$. 

We inherit the graph metric as above, but since paths in this graph can be labelled by the directed edges they go through which spell \emph{words} with characters in $X$, ($p = x_1 x_2 \dots x_n, \quad x_i \in X$), we call this metric on the Cayley graph the \emph{word metric}. 

The notions of geodesic and quasi-geodesic naturally extends through this metric. 

\begin{defn}[Geodesic and quasi-geodesic word]\label{geodesic-words}
A \emph{geodesic path} in a Cayley graph $\Gamma$ is a path $p$ connecting two vertices $v_1, v_2 \in V(\Gamma)$ such that $p$ has the shortest length amongst all paths from $v_1$ to $v_2$ in $\Gamma$. A word $w \in X^*$ is \emph{geodesic} if $w \in X^*$ and the induced path $p_w$ is geodesic. If $g$ is a group element such that $\bar w = g$, then $w$ is a \emph{geodesic representative} of $g$. 

Similarly, a $(\lambda, \epsilon)$-\emph{quasi-geodesic word} $w$ is a word of length $n$ which induces a $(\lambda,\epsilon)$-quasi-geodesic path $\gamma: [0,n] \to \Gamma$. 
\end{defn}

A priori, it seems like the choice of generating set may influence the property of a Cayley graph to be hyperbolic. This is when a key property to studying $\delta$-hyperbolic space comes in, that is, its invariance under quasi-isometry.

\begin{lem}\label{lem: hyperbolic-gen-set-change-qi}
	For a fixed group $G$ and finite generating sets $X,Y$, the Cayley graphs $\Gamma_X(G, X)$, $\Gamma_Y(G, Y)$ are quasi-isometric to one-another.
\end{lem}
\begin{proof}
	Indeed, since any generator $y \in Y$ corresponds to a finite word $y = x_1 \dots x_n = w_X(y)$ in $X^*$, we can define $N := \max_{y \in Y} |w_X(y)| \geq 1$ to be the maximal length in $X$ of the generators in $Y$ when translated as words in $X$. This means that a geodesic word from $a$ to $b$ in $Y$ can be written as a word in $X$ of maximal length $N d_Y(a,b)$. This gives us the inequality $d_X(a,b) \leq N d_Y(a,b)$, since any geodesic in $X$ connecting $a$ to $b$ will be of length less that the right-hand side. Similarly, there exists a constant $M \geq 1$ such that $d_Y(a,b) \leq M d_Y(a,b)$. Let $\lambda := \max\{M,N\} \geq 1$. Putting everything together, we get 

	$$\frac{1}{\lambda} d_X(a,b) \leq d_Y(a,b) \leq \lambda d_X(a,b).$$

	Let $f: \Gamma_X \to \Gamma_Y$ be the map which passes from one Cayley graph to another. This map is the identity on the group elements and is a quasi-isometry by the above. 
\end{proof}

\begin{ex}
	The Cayley graph of the group $G = \langle g \mid g^3 \rangle$ of order 3 is illustrated in Figure \ref{fig: hyp-graph-triangle-Cayley}. By the previous, this Cayley graph is $\delta$-hyperbolic. Thus our group is $\delta$-hyperbolic with $\delta > 1/2$. 
\end{ex}

\begin{figure}[h]{
\includegraphics{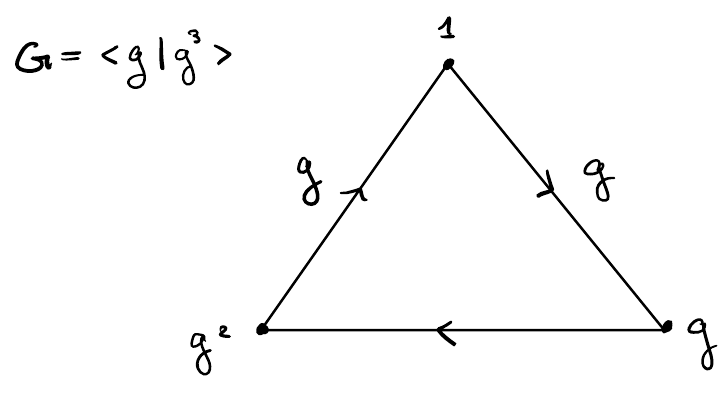}
}
\caption{Illustrating the Cayley graph for the group $G$. The underlying unlabelled graph is the graph of Figure \ref{fig: hyp-graph-triangle-2}.
}
\label{fig: hyp-graph-triangle-Cayley}
\end{figure} 

\begin{ex}[Finite groups] All finite groups are $\delta$-hyperbolic. This is because you can take $\delta$ to be the diameter, as done in the Exercise \ref{exc: hyp-finite-graphs}. 	
\end{ex}

\begin{ex}[Free groups]
All free groups $F_n = \langle x_1 \dots x_n \rangle$ are $0$-hyperbolic, since all trees are $0$-hyperbolic. 
\end{ex}

Before we explore more interesting examples, let us introduce an important result is that sometimes called the ``fundamental observation in geometric group theory'' \cite{delaHarpe2000}. We will use this lemma for the next examples. 

\begin{lem}[\v{S}varc-Milnor lemma]
	If $G$ acts properly\sidenote{An action of $G$ on $X$ is \emph{properly discontinuous} if for every compact subset $K \subset X$ there are finitely many $g \in G$ such that $gK \cap K \not= \emptyset$.} and cocompactly\sidenote{The action $G$ on $X$ is cocompact if $X / G$ is compact.} via isometries on a geodesic space $X$, then there is a quasi-isometry $g \mapsto g x $ for any basepoint $x \in X$. 
\end{lem}

A proof for this lemma can be found in \cite[Part I, Proposition 8.19]{BridsonHaefliger1999}. Basically, the lemma says that, up to large scale geometry, studying a group and its geometrical action\sidenote{A properly discontinuous cocompact isometric action of a group $G$ on a proper geodesic metric space $X$ is called a \emph{geometric action}.} on a space are the same thing. This observation is very fundamental in geometric group theory because it justifies the way we approach studying groups via studying spaces on which they act geometrically. 

Next, the following example is written in mind for the reader who already has some background in algebraic topology. For more details, we refer to \cite{Hatcher2002} for the definitions and \cite[Part 3, Chapter 9, Section 3]{OfficeHours2017} for a friendly introduction to the concepts.

\begin{ex}[Surface groups of genus $g \geq 2$] 
The classification of surfaces says that any closed, connected, orientable surface is homeomorphic to either a sphere or a connected sum of tori.\sidenote{See for example \url{https://www3.nd.edu/~andyp/notes/ClassificationSurfaces.pdf}.} 

Any surface $\Sigma_g$ of genus $g \geq 2$ can be given by universal cover $\mathbb{H}^2$ given by the Poincar\'e disk, and a quotient by a marked polygon of $4g$ sides.

Suppose $g = 2$. The quotient gives rise to a tessellation in the universal cover $\mathbb{H}^2$, and elements of the fundamental groups $\pi_1(\Sigma_2, x_0)$ lift to paths between the elements of $\tilde x_0$, the lift of $x_0$ in the $\mathbb{H}^2$, which form dual curves to the tessellation. Labelling the edges of the paths by the elements of the fundamental group they are lifting from, and the vertices by the resulting group elements of right multiplying by the labels, the dual curves form the Cayley graph of $\pi_1(\Sigma_2, x_0)$. This is illustrated in Figure \ref{fig: hyp-graph-surface}. 

\begin{figure}[h]{
\includegraphics{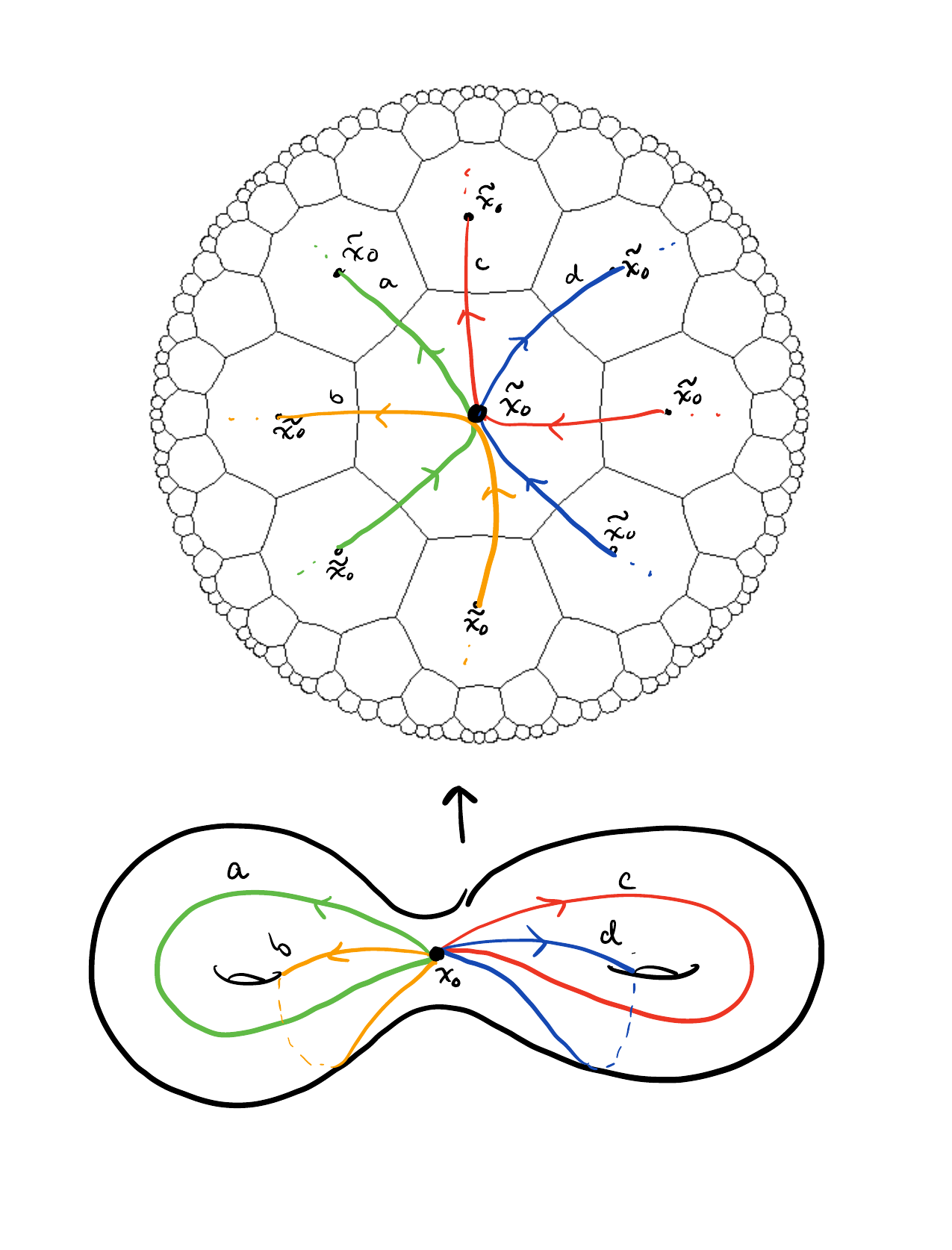}
}
\caption{Illustrating a surface of genus $2$, $\Sigma_2$, along with the $4$ generators of its fundamental group based at $x_0$. The generators are labelled with $a,b,c,d$, and the fundamental group has presentation $\langle a,b,c,d \mid [a,b][c,d]\rangle$. The surface lifts to its universal cover that is the Poincar\'e disk $\mathbb{H}^2$ and the quotient is given by an octogonal tiling. The dual curves to the tessellation give rise to a Cayley graph for $\pi_1(\Sigma_2, x_0)$. The tesselation is taken from \url{http://aleph0.clarku.edu/~djoyce/poincare/tilings.html}.}
\label{fig: hyp-graph-surface}
\end{figure} 

The fundamental group $\pi(\Sigma_2, x_0)$ acts on $\mathbb{H}^2$ by translating points along the generators. The action is is cocompact since the fundamental domain is an octagon which generates the tessellation, and it acts properly discontinuously because this fundamental domain only shares its sides with $8$ other octagons, so only $8$ distinct translations fail to move it disjointly to itself. By the \v{S}varc-Milnor lemma, the Cayley-graph is quasi-isometric to $\mathbb{H}^2$ and therefore $\delta$-hyperbolic. 
	The example of $g = 2$ generalises to $g > 2$. 
\end{ex}

We can apply the same process as the previous example to hyperbolic manifold groups. 

\begin{ex}[Hyperbolic manifold groups]\label{ex: hyperbolic-manifold-groups}
	If $M$ is a closed hyperbolic $3$-manifold, then $M$ can be written as $\mathbb{H}^3/G$, where $G$ is a torsion-free discrete group of isometries on $\mathbb{H}^3$. Thus the fundamental group $\pi_1(M)$ is quasi-isometric to the hyperbolic $3$-space $\mathbb{H}^3$ and is $\delta$-hyperbolic. 
\end{ex}

Another salient feature of infinite hyperbolic groups is that their Cayley graphs are very tree-like in structure, and one may wonder if there exists an embedding of a free group inside these hyperbolic group. The answer is yes. Roughly speaking, two generic elements in a hyperbolic group generate a free group. 

We start by introducing the following lemma is the standard trick for proving the existence of a free group in geometric group theory. 

\begin{lem}[Ping-Pong Lemma]
	Suppose $a$ and $b$ generate a group $G$ that acts on a set $X$. If 
	\begin{enumerate}
		\item $X$ has disjoint nonempty subsets $X_a$ and $X_b$ and 
		\item $a^k(X_b) \subset X_a$ and $b^k(X_a) \subset X_b$ for all non-zero powers of $k$,
	\end{enumerate}
	then $G$ is isomorphic to a free group of rank $2$. 
\end{lem}
A great explainer of this can be found in \cite[Office Hour Five]{OfficeHours2017}. Basically, the idea is that if you have an action which keeps telescoping you into smaller and smaller disjoint subsets, then that action has to be free and not the identity. 	
	
Relatedly, the following is true.
\begin{prop}
	If $G$ is hyperbolic, then for every set of elements $\{g_1, \dots, g_r\}$, there exists an integer $n > 0$ such that $\{g_1^n, \dots, g_r^n\}$ generates a free subgroup of rank at most $r$. 
\end{prop}
See \cite{BridsonHaefliger1999}[Chapter III.$\Gamma$ Proposition 3.20] for a proof. Moreover, we remark that the embedded free group is quasigeodesically embedded, since the generators of the free group are words which are $(n,0)$-quasi-geodesics. 

%% file: chap/semi-wreath.tex
\chapter{Semi-direct products and wreath products}\label{chap: semi-wreath}

In this chapter, we will review semi-direct products and define wreath products. This is because in Chapter \ref{chap: closure-extension}, we will use semidirect products in three ways: as inner semidirect products, left outer semidirect products and right outer semidirect products. Since this mix of semidirect products can be confusing to the reader,\sidenote{I confirm it was confusing to me, the writer.} let us introduce semidirect products from a ``first-principles'' approach.\sidenote{Semidirect products are something that I had seen in theory in undergrad, yet they remained confusing to me through grad school. I had to relearn this several times to great forgetfulness so I figured writing down a first-principles approach for myself to derive all three types of semidirect products would help it stick! A particularly nice introduction I used to learn this is \url{https://aprovost.profweb.ca/math/semidirectnotes/semidirectnotes.pdf}.}  %

Building on this review, we will then introduce the particular case of wreath products and define a positive cone for the wreath product $\bZ \wr \bZ$, which will also be discuss in Chapter \ref{chap: closure-extension}.

\section{Review of semi-direct products}
\label{sec: semi-direct-products}
We begin with a quick discussion of semidirect product and lamplighter groups. The reader is welcome to skip ahead to the positive cone for $\bZ \wr \bZ$ given in Section \ref{sec: LO-lamplighter}.

\begin{defn}[inner semidirect product]\label{defn: inner-semidirect-product}
Let $G$ be a group and $N, Q$ be subgroups such that 
\begin{enumerate}
	\item $N$ is a normal subgroup,
	\item $NQ = G$,
	\item $N \cap Q = \{1\}$.
\end{enumerate}
We call $G$ an \emph{inner semidirect product}.
\end{defn}
The terminology for semidirect product refers to the fact that elements can be decomposed in an coordinate-wise manner, with a multiplication that is a little more complex than those of direct products. 

Suppose that $G$ is an inner semidirect product. Observe that any element $g \in G$ admits a unique decomposition $g = nq$ with $n \in N, q \in Q$. Indeed, suppose the decomposition is not unique and that $g = n_1 q_1 = n_2 q_2$. Then $q_1 q_2\inv = n_1\inv n_2$. Since $N \cap Q = \{1\}$, this means that $q_1 = q_2$ and $n_1 = n_2$.

We can then define a bijective map $\Psi$ which decomposes $G$ coordinate-wise, $\Psi: G \to N \times Q$,
$$\Psi(g) = \Psi(nq) = (n,q).$$

Suppose we would like to upgrade this map $\Psi$ to an isomorphism. That is, starting from $N \times Q$ we would like to recover a group structure isomorphic to that of $G$. To do so, we need to add a group structure to $N \times Q$. The necessarity structure reveals itself through multiplying elements in the inner semidirect product $G$. 

Let $g_1, g_2 \in G$. Then 
\begin{align*}
	g_1 g_2 &= n_1 q_1 n_2 q_2 \\
	& = n_1 q_1 n_2 (q_1\inv q_1)q_2 \\
	&= n_1 (\underbrace{q_1 n_2 q_1\inv}_{\in N}) q_1 q_2
\end{align*}

Therefore, the group multiplicative structure we want over $N \times Q$ is given by the $\star$ operation, defined as 
\begin{align*}
\Psi(g_1 g_2) &= \Psi(g_1)\star \Psi(g_2) \\
&= (n_1,q_1) \star (n_2,q_2) \\
&= (n_1 \cdot q_1 n_2 q_1\inv, q_1q_2).
\end{align*}

To capture what we found, we denote $N \rtimes_\varphi Q := (N \times Q, \star)$ to be the candidate group with the $\star$-multiplicative structure, captured by the map $\varphi$. Indeed, we define $\varphi: Q \to \Aut(N)$ to be the conjugation action 
$$\varphi_{q}(n) := \varphi(q)(n) =  q n q\inv$$ 

such that we have the pretty\sidenote{I like to think of $\varphi$ as giving the ``twist'' that makes the ``$\times$'' in the underlying set notation pass to the ``$\rtimes$'' in the group notation, as it twists a direct product into a \emph{semi}direct product, if that makes sense.} notation
$$(n_1,q_1) \star (n_2q_2) = (n_1 \cdot \varphi_{q_1}(n_2), q_1q_2).$$
We can verify that this $\star$-multiplication is associative by computing
\begin{align*}
	&((n_1, q_1) \star (n_2, q_2)) \star (n_3, q_3) \\
	&= (n_1, \varphi_{q_1}(n_2), q_1 q_2) \star (n_3, q_3) \\
	&= (n_1 \varphi_{q_1}(n_2) \varphi_{q_1 q_2}(n_3), q_1 q_2 q_2) \\
	&= (n_1 q_1 n_2 \underbracket{q_2 n_3 q_2\inv} q_1\inv, q_1 q_2 q_3) \\
	&= (n_1 \underbracket{q_1 n_2 \varphi_{q_2}(n_3) q_1\inv}, q_1 q_2 q_3) \\
	&= (n_1, q_1) \star (n_2\varphi_{q_2}(n_3), q_2 q_3) \\
	&= (n_1, q_1) \star ((n_2, q_2) \star (n_3, q_3)).
\end{align*}

The inverse of some element $(n,q)$ is again found from the inner semidirect product $G$. If $g = nq$, then 
\begin{align*}
	g\inv &= q\inv n\inv \\
	&= q\inv n\inv \cdot q q \inv \\
	\Psi(g\inv) &= (\varphi_{q\inv}(n\inv),q\inv). \\
	&:= (n,q)\inv = \Psi(g)\inv.
\end{align*}
Finally, we can check that the identity element in $N \rtimes_\varphi Q$ is indeed $(1,1)$. 
\begin{align*}
	&(1,1)\star(n,q) = (1\varphi_1(n), 1 \cdot q) = (n,q) \\
	&(n,q)\star(1,1) = (v\varphi_q(1), q \cdot 1) = (n,q).
\end{align*}

We call the $N \rtimes_\varphi Q$ group a \emph{left outer semidirect product}. The ``outer'' part of the terminology seems to capture going from group multiplication within the group $G$ to explicitly a Cartesian product with a multiplication operation tacked on $N \rtimes_\varphi Q$. The  ``left'' part of the terminology refers to the $\varphi$ action, as it conjugates from the left: $\varphi_q(n) = qnq\inv$. Since the multiplicative structure is captured by the map $\varphi$ in $N \rtimes_\varphi Q$, we now omit the $\star$ notation in the sequel.

To summarize, we write down a formal definition for $N \rtimes_\varphi Q$. 

\begin{defn}[left outer semidirect product]\label{defn: left-outer-semidirect-product}
Let $N,Q$ be groups and let $\varphi:Q\to\Aut(N)$ be a homomorphism
(a left action by automorphisms).  The left \emph{outer
semidirect product} $N\rtimes_{\varphi} Q$ is the set $N\times Q$
with multiplication
\[
(n_1,q_1)(n_2,q_2)
= \bigl(n_1\,\varphi(q_1)(n_2),\,q_1q_2\bigr),
\]
identity $(e_N,e_Q)$, and inverse
\[
(n,q)^{-1}=\bigl(\varphi(q^{-1})(n^{-1}),\,q^{-1}\bigr).
\] 
\end{defn}

The right outer semidirect product $Q \ltimes_\varphi N$ is defined similarly as the left. However, the multiplicative operation will come out differently as writing $g \in G$ in the form $g = qn$ leads to a different outcome when two elements are multiplied together. This is straightforward to derive from the inner semidirect product. 

First we observe that for an inner semiproduct $G$ with subgroups $N,Q$ as in Definition \ref{defn: inner-semidirect-product}, and $\varphi_q(n) := qnq\inv$, since $G= NQ$, we have $G = G\inv = Q\inv N\inv = QN$. Therefore, any $g \in G$ admits a unique decomposition $g = qn$. Let $g_1, g_2 \in G$ such that $g_1 = q_1 n_1, g_2 = q_2 n_2$. Then, 
\begin{align*}
g_1 g_2 &= q_1 n_1 q_2 n_2 \\
&= q_1 (q_2 q_2\inv) n_1 q_2 n_2 \\
&= q_1 q_2 \underbrace{(q_2\inv n_1 q_2)}_{\in N} n_2
\end{align*}

Therefore, the correct group operation for $Q \ltimes_\varphi N$ is 

\begin{align*}
	(q_1, n_1)(q_2, n_2) &= (q_1 q_2, q_2\inv n_1 q_2 n_2) \\
	&= (q_1 q_2, \varphi_{q_2\inv}(n_1) \cdot n_2).
\end{align*}

Notice here that the $\varphi$ automorphism is given by $q_2\inv$ instead of $q_1$ as in the the left outer semidirect product. In particular, the $q_2$-action is on the right instead of left, which is why $Q \ltimes_\varphi N$ is called the \emph{right outer semidirect product}. 

The inverse element in $Q \ltimes_\varphi N$ can be found similarly. If $g = qn$, then
\begin{align*}
	g\inv &= n\inv q\inv \\
	&= \underbrace{q\inv}_{\in Q} \underbrace{q \cdot n\inv q\inv}_{\in N} \\
	&= (q\inv \cdot \varphi_q(n\inv))
\end{align*}
Therefore, 
$(q,n)\inv = (q\inv, \varphi_q(n\inv))$ in $Q \ltimes_\varphi N$. 

We omit the complete proof that the right outer semidirect product $Q \ltimes_\varphi N$ is indeed a group, as it is too similar to the proof of the left outer semidirect product. We move on instead to stating its definition formally. 

\begin{defn}[right outer semidirect product]\label{defn: right-outer-semidirect-product}
Let $N$ and $Q$ be groups and let $Q$ act on $N$. Let $\varphi: Q \to \Aut(N)$ be the left conjugation action of $Q$ on $N$. The \emph{left outer semidirect product} $N \rtimes_\varphi Q$ is the group of all pairs $(n,q)$ with $n \in N$ and $q \in Q$ with the group operation
$$(q_1, n_1)(q_2, n_2) = (q_1 q_2, \varphi_{q_2\inv}(n_1) \cdot n_2),$$ 
identity $(e_Q, e_N)$ and inverse
$$(q,n)\inv = (q\inv \cdot \varphi_q(n\inv)).$$
\end{defn}

Our work so far leads naturally to the following well-known theorem, which we state without proof because we have already (mostly) proved it! 

\begin{thm}
Let $G$ be an inner semidirect product with normal subgroup $N$ and subgroup $Q$. Let $\varphi: Q \to \Aut(N)$ be $\varphi_q(n) = qnq\inv$. Let $N \rtimes_\varphi Q$ be the left outer semidirect product of $N$ and $Q$ and $Q \ltimes_{\varphi} N$ be the right outer semidirect product of $Q$ and $N$. Then, 
$$G \cong N \rtimes_\varphi Q \cong Q \ltimes_{\varphi} N.$$
\end{thm}

\subsection{Left-orderability of semi-direct products}
 
\begin{cor}\label{cor: LO-semi-direct-prod} A semidirect product is left-orderable if and only if its factors are left-orderable.
\end{cor}
\begin{proof}
	The semidirect product $N \rtimes_\phi Q$ is an extension of $Q$ by $N$, By Lemma \ref{lem: LO-clos-ext}, if $Q$ and $N$ are left-orderable, so is their extension. 
\end{proof}

\subsection{Wreath products}\label{sec: wreath-products}
We give a definition of wreath products using the left outer semidirect product convention. 

\begin{defn}[wreath product]\label{defn: wreath-product}
	Let $N$ be a group and $Q$ be a group acting on $N$. The \emph{(restricted) wreath product} $N\wr Q$ of groups $N$ and $Q$, is the left outer semidirect product of ${\bf N}:=\oplus_{\omega \in Q} N$ by $Q$, ${\bf N} \rtimes_\varphi Q$, where the conjugation action of $Q$ on ${\bf N}$ is given by the left-multiplication action on the indexes of the copies of $N$.
That is, if ${\bf n}= (n_\omega)_{\omega \in Q}\in {\bf N}$ and $q\in Q$, we have $\varphi_q({\bf{n}}). = q{\bf n}q^{-1}= (n_{q\omega})_{\omega\in G}$.
\end{defn}

Let us elaborate with examples.
\subsection{Lamplighter groups}\label{sec: lamplighter}

\begin{ex}[$\bZ_2 \wr \bZ$]
Perhaps the most well-known example of a wreath product is the lamplighter group, which is given by the wreath product $G = \bZ_2 \wr \bZ$.

One way of understanding the group is by viewing the underlying set $(\oplus_{i \in \bZ} \bZ_2) \times \bZ$ as lamps (the $\bZ_2$ factors) laid out on an infinite line indexed by $i \in \bZ$  with all but finitely many being in the state ``on'' and the rest being ``off'' (represented by the direct sum $\oplus_{i \in \bZ} \bZ_2$), and a lamplighter who is positioned at any one of the lamp (represented by the product with $\bZ$). For example, Figure \ref{fig: lamplighter-identity} illustrates the group element $({\bf 0}, 0)$. 

\begin{figure}[h]{
\includegraphics{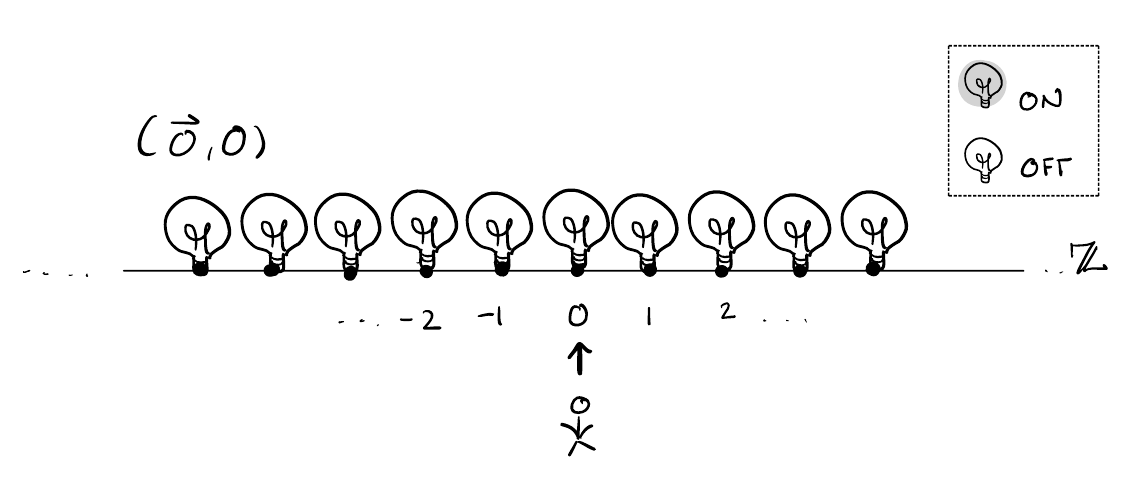}
}
\caption{The group element $({\bf 0}, 0)$ is represented in the form of an integer row of lamps all in the state ``off'' state, plus the lampligher at position 0.}
\label{fig: lamplighter-identity}
\end{figure}

Note that this group has generating set 
$T = ({\bf 0}, 1)$ and $A = (\delta_0, 0)$, where $\delta_0 = (\dots, 0, 1, 0, \dots)$, where the $1$ is at the index corresponding to $i = 0$. 

Indeed, $T$ corresponds to moving the lamplighter one position to the right, and $A$ corresponds to lighting a lamp at position $0$. It is straightforward to see how every configuration of a finite number of lamps being lit can be obtained using only these two generators.

Now, let us use the definition of wreath product to deduce the multiplicative behaviour of the group. Let $g, h \in G$ such that 
\begin{align*}
	gh = ({\bf n},q)({\bf m},p) &:= ({\bf n} \varphi_q({\bf m}), q + p) \\
	&=((n_i + m_{i + q})_{i \in \bZ}, q + p)
\end{align*}

This can be interpreted as applying instructions $g$ followed by instructions $h$, as illustrated in Figure \ref{fig: lamplighter-mult}. Indeed, starting from the identity state where all the lamps are off and the lamplighter at position zero, applying instructing $g$ gives us that the lamps are now in state ${\bf n}$ and the lamplighter at position $q$. Then, picking up from position $q$, we apply the state-changes ${\bf m}$, but with the origin shifted to $q$. We then move the lamplighter $p$ positions relative to $q$. 
\end{ex}

\begin{figure}[h]{
\includegraphics{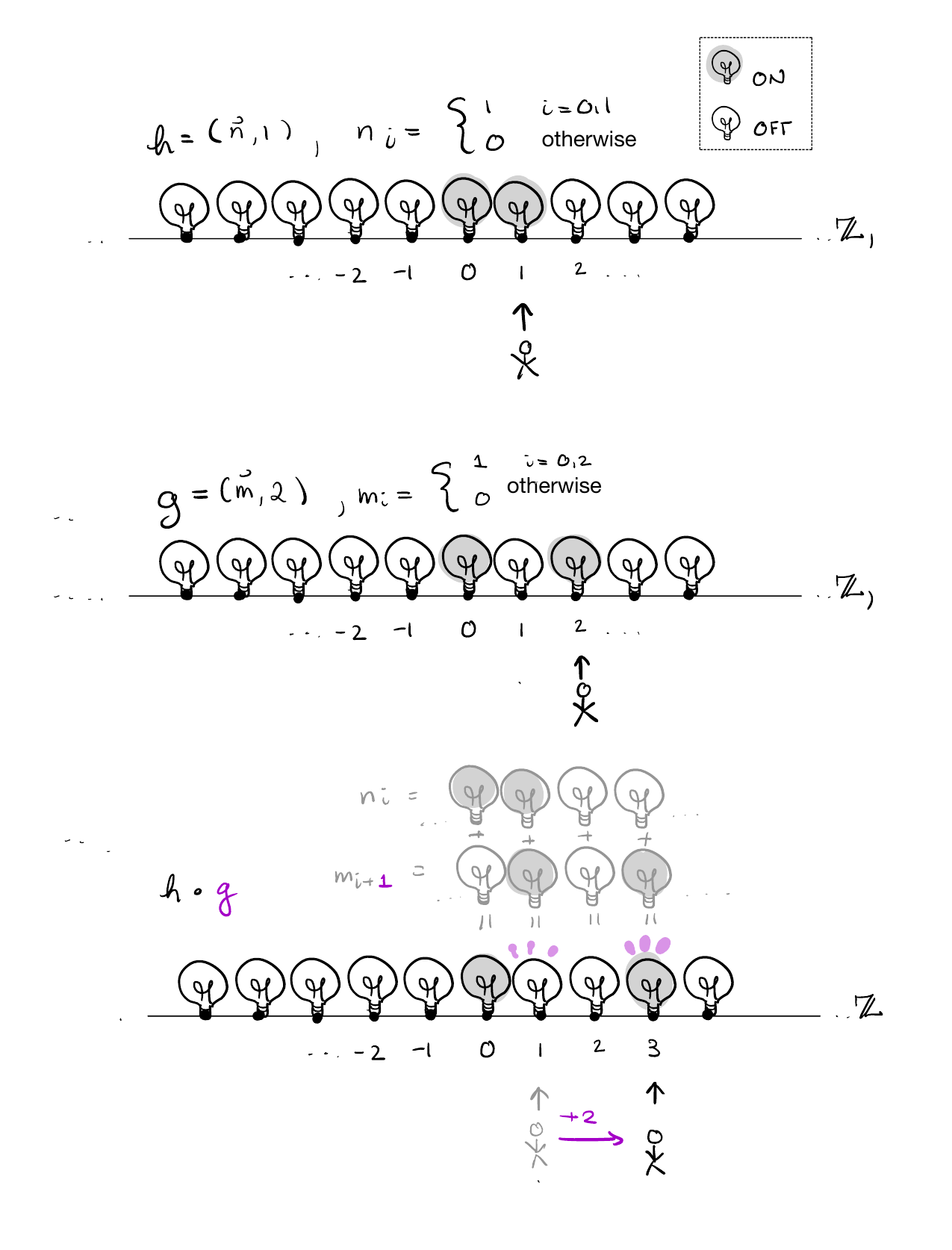}
}
\caption{We illustrate both $h$ and $g$ as applied to the identity element $({\bf 0}, 0)$. Then we illustrate what happens if $hg$ were applied to the identity. In grey, we have illustrated how ${\bf n}$ and ${\bf m}$ combine under the semidirect product structure. In purple we have highlighted the changes that happen to $h$ when multiplied with $g$.}
\label{fig: lamplighter-mult}
\end{figure}
 
Since we are interested in left-orderability, let us now replace the $\bZ_2$ factor in our previous group with the torsion-free group $\bZ$ and look at the left-orderability of $\bZ \wr \bZ = (\oplus_{i \in \bZ} \bZ) \rtimes_\phi \bZ$. We already know this group is left-orderable by Corollary \ref{cor: LO-semi-direct-prod}, since it is a semidirect product of groups which are left-orderable. 

\subsection{Constructing a left-order on the $\bZ \wr \bZ$}\label{sec: LO-lamplighter}
\begin{ex}[$\bZ \wr \bZ$]
\label{ex: LO-bZ-wr-bZ}
Let us construct a left-order for $\bZ \wr \bZ = (\oplus_{i \in \bZ} \bZ) \rtimes_\varphi \bZ$ from our definitions. From Lemma \ref{lem: P-clos-ext-N-leads}, if a group $G$ is of the form $N \rtimes Q$, where both $N$ and $Q$ are left-orderable, and $N$ is closed under $Q$-conjugation under $G$, then $G$ is left-orderable. 

Here, $N = \oplus_{i \in \bZ} \bZ$ where we can a positive cone for $N$ as 
$$P_N = \{ (n_i)_{i \in \bZ} \mid n_{i_0} > 0, \text{ where } i_0 := \min i \in \bZ \text{ such that } n_i \not= 0\}.$$
 This minimum is well-defined since by definition of the direct sum $\oplus$, the number of indices where $n_i \not= 0$ is finite. Moreover, $Q = \bZ$ and a positive cone $P_Q$ is straightforwardly given by $$P_Q = \{q \in \bZ \mid q > 0 \}.$$ Putting it all together, a positive cone for $\bZ \wr \bZ$ is 
$$P = \{( {\bf n}, q) \mid {\bf n} \in P_N \text{ or } {\bf n} = {\bf 0}, q \in P_Q\}.$$
\end{ex}

%% file: chap/Reidemeister-Schreier-method.tex
\chapter{The Reidemeister-Schreier method}\label{chap: RS}
In this chapter, we will introduce the Reidemeister-Schreier method, whose statement and ideas will feature heavily in Chapters \ref{chap: closure-finite-index} and \ref{chap: fg}. 

\section{Introduction}
The Reidemeister-Schreier method is an algorithm which takes as input a finitely generated group $G$ with a presentation $\langle X \mid R \rangle$ and a finite-index subgroup $H$ such that we know the right cosets $H \backslash G$ and a particular set of representatives called a Schreier transversal, and outputs a finite presentation for $H$. 

For the formal statement of the method, we refer to the formulation of \cite[Chapter II, Section 4.]{LyndonSchupp2001}. However, the ideas behind this method are presented in a more natural way in the older \cite[Chapter 2]{MagnusKarrassSolitar1966}, which we suggest for reference. In this chapter, we will first give the reader our own understanding of the method, which is informed by contemporary expository work on the this topic, such as \cite{Casey2017} and \cite[Section 2.5]{Knudsen2018}, then prove the statement in Section \ref{chap: rs, sec: pf}.

\section{The statement}

\begin{defn}[Schreier transversal]\label{defn: Schreier-transversal} Let $F$ be a free group, and $\tilde H$ be a subgroup of $F$. A \emph{Schreier transversal} $\tilde T$ of $\tilde H$ is a subset of $F$ such that for distinct $t \in \tilde T$, the cosets $\tilde Ht$ are distinct, $\bigcup_{t \in \tilde T} \tilde Ht = F$, and such that each initial segment (or prefix) of an element of $\tilde T$ belongs to $\tilde T$. \cite{LyndonSchupp2001}
\end{defn}

\begin{thm}[Reidemeister-Schreier method]\label{thm: RS}
Let $G = F/N$, where $F$ is free with basis $X$ and $N$ is the normal closure of the relator set $R$. Let $\phi: F \to F/N$ be the natural map of $F$ onto $G$. Let $H$ be a subgroup of $G$ with $\tilde H$ as the inverse image under $\phi$, and let $\tilde T$ be a Schreier transversal for $\tilde H$ in $F$. For $w$ in $F$, we define $\bar w$ by the condition that 
$$\tilde Hw = \tilde H\bar w, \quad \bar w \in \tilde T.$$
For $t \in \tilde T$, $x \in X$, we define 
$$\gamma(t,x) = tx(\ovl{tx})\inv, \quad \gamma(t,x\inv) = tx\inv(\ovl{tx\inv})\inv = \gamma(tx\inv,x)\inv.$$

Define a one-to-one correspondence between $\gamma(t,x)^*$ and $\gamma(t,x)$. Then $H$ has presentation $\langle Y \mid S \rangle$ where $Y = \{\gamma(t,x)^* : t \in \tilde T, x \in X, \gamma(t,x)^*\not= 1\}$. 

Let $F'$ be the free group generated with basis $Y$. Define $\tau: F \to F'$ as follows. If $w = y_1 \dots y_\ell$, 
$$\tau(w) = \gamma(1,y_1)^* \dots \gamma(\ovl{y_1 \dots y_{i-1}}, y_i)^* \dots \gamma(\ovl{y_1 \dots y_{\ell-1}}, y_\ell)^*.$$
Then $S = \{\tau(trt\inv) : t \in T, r \in R\}$. \cite{LyndonSchupp2001}
\end{thm}

Although the statement is a priori quite formal, and the computations used to apply the algorithm is fairly mechanical, the geometric ideas underpinning the method are actually quite beautiful in that way that feels characteristic of geometric group theory. Since this may not be obvious from the statement for a first time reader, let us explore these ideas with the help of an example and illustrations of that example.\sidenote{Full disclosure, the Reidemeister-Schreier Method took me one very long week to understand. It was one of those statements whose formalism had me banging my head against the wall until I found the right example to focus on, and then opened my mind to how pleasing the ideas behind it are. In this chapter, I hope the present the insight I gained working on that example so it is both easier to grasp ad more enjoyable to learn for the reader than it was for me.} 

\section{An illustrated example from a topological perspective}
The following example was taken from Master's project of Levi Casey \cite{Casey2017}, which we illustrate and present from a topological point of view similar to that of \cite{Knudsen2018}. 

Let our group $G := D_4 = \langle a, b \mid a^4, b^2, aba = b \rangle$ be the dihedral group, and $H := V = \langle x, y \mid x^2, y^2, xyx=y\rangle$ be the Klein group\sidenote{Not to be confused with the Klein bottle group.}. On the left of Figure \ref{fig: rs-D4-V}, we illustrate the Cayley graph of $\Gamma(G, X)$, highlight in yellow the elements belonging to the subgroup $H$, and highlight in purple the elements belonging to the coset $Ha$. On the right side of Figure \ref{fig: rs-D4-V}, we illustrate what happens to the Cayley graph $\Gamma(G, X)$ if we shrink the cosets $\{He, Ha\}$ into vertices. Such a graph is denoted $\Gamma(H\backslash G, X)$ and is called a \emph{Shreier graph}, and it visually encodes our knowledge of $H$ in terms of right cosets and chosen representatives.

\begin{figure}[h]{
\includegraphics{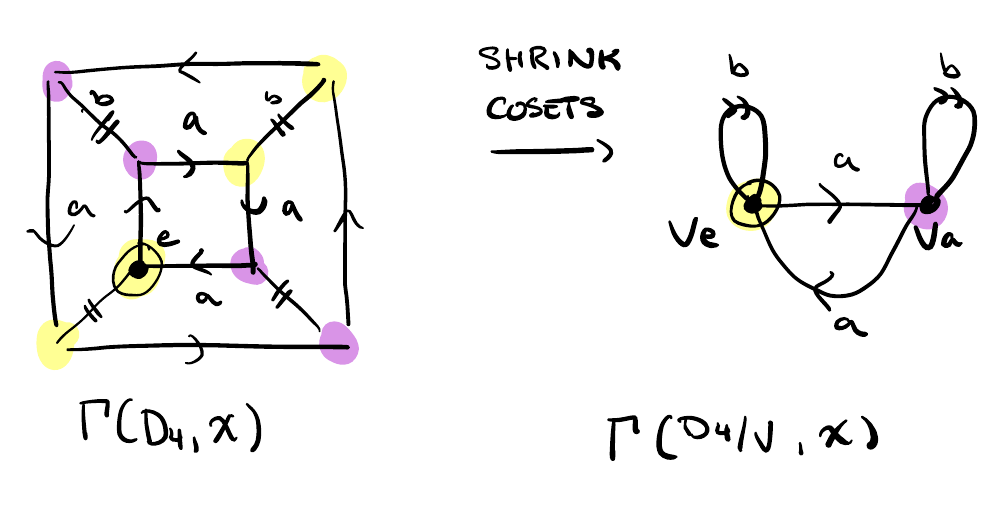}
}
\caption{On the left, a Cayley graph for $D_4$. We identify the elements which belong to $V$ in yellow, and the elements which belong to $Va$ in purple. On the right, we illustrate what happens if we shrink both cosets $Ve$ and $Va$ to a vertex.
}
\label{fig: rs-D4-V}
\end{figure}

Now suppose that, instead of knowing the presentation for $H$ we would like to recover it from $H \backslash G$.
\sidenote{One scenario in which this could happen naturally is for example if $H$ is the kernel of the map $\varphi: G \to \mathbb{Z}/2 \mathbb{Z}$ , $a \mapsto 1, b \mapsto 0$ for which we would like to construct a presentation. This will be precisely the use case in Chapter \ref{chap: fg}.} The first thing we want to do is to figure out a generating set for $H$, then write the relations of $H$ in terms of that generating set. 

Observe that the generating set of $H$, on the left side of Figure \ref{fig: rs-D4-V}, must be elements which start at the identity and end at $H$. Therefore, they must form paths in $\Gamma(G, X)$ starting and ending at vertices highlighted in yellow. From the Shreier graph on the right side of Figure \ref{fig: rs-D4-V}, this corresponds to paths from $He$ to $He$ (or in our case $Ve$ to $Ve$), and therefore, to elements of the fundamental group based at the identity coset, $\pi_1(\Gamma(H \backslash G),He)$. The generators of this fundamental group are then given by choosing a minimal spanning tree (or \emph{transversal}) such as $T = \{1, a\}$, and appending an adjacent edge in the set of edges $E(\Gamma(H \backslash G), X) - E(T)$ such that the resulting path is a loop (thus, every other composite loop would be a composition of such elements, proving we do have a generating set). Figure \ref{fig: rs-D4-V-gens} illustrates this idea, along with the associated computations. A generating set for $V = H$ can therefore be given by $\phi(Y) = \{\phi(x) := a^2, \phi(y) := b, \phi(x) := aba\inv\}$ (by abuse of notation, we identify the generators of $G$ with the basis $X$ of $F$). 
\begin{figure}[h]{
\includegraphics{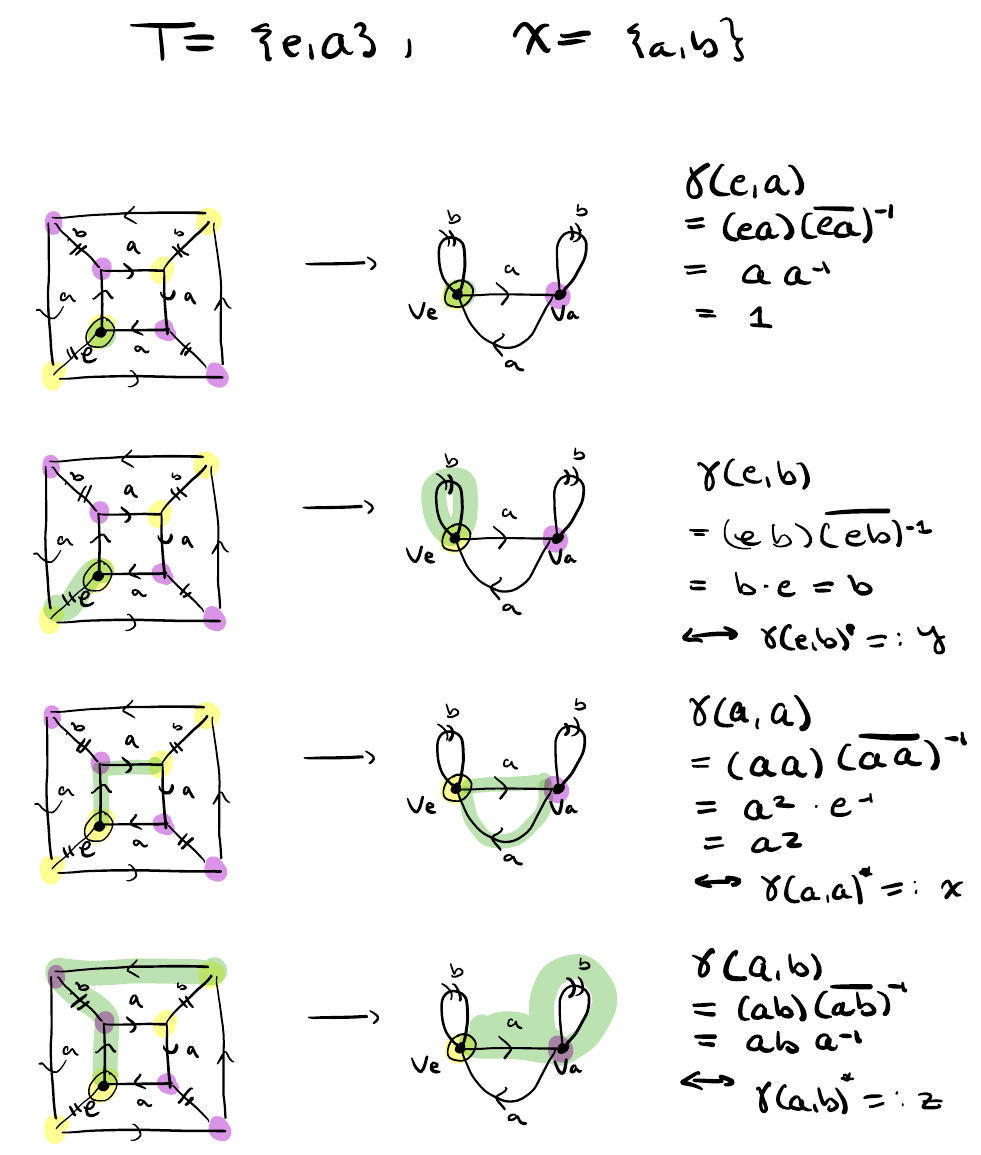}
}
\caption{Illustrating finding a generating set for $H$ in the Schreier graph, along with its correspondence in the Cayley graph.}
\label{fig: rs-D4-V-gens}
\end{figure}

Now, in the formal statement of Reidemeister-Schreier as written in the statement of Theorem \ref{thm: RS}, the computations to find the $\gamma(t,x)$'s as generators of $H$ are made over the lift of $H$, $\phi\inv(H) = \tilde H$ and not in the group itself. Note that the difference between $\phi(Y)$ and $Y$ is that if $\phi$ does not map generator-to-generator, then it must be that some $t \in T, x \in X$, we have that $\gamma(t,x) = n \in N$, which becomes the identity $H$. In that case, such a generator can be removed once we have computed the full presentation for $H$. The point is that working in the lift rather than group itself does not fundamentally change the idea of how the generators are found. 

The main advantage of working in the lift $\tilde H$ is that allows us to deal with the relations concretely rather than implicitly. Algebraically, this is made clear by realising $G = F/N$, and writing $R$ as elements of $F$ from which we deduce $S$, and thus allowing us to realise $H$ as $H = \langle Y \mid S\rangle$. From a topological perspective, we can view $\Gamma(G,X)$ as a CW complex with the $1$-dimensional complex being the universal cover of $G$ (or the Cayley graph of $F$), and its $2$-cells being the normal closure of the relations $R$. If $F'$ is the free group generated by $Y$, then, finding the presentation for $H$ becomes finding the correct $2$-cells to glue to the Cayley graph of $F'$ to obtain $\Gamma(H, Y)$. 
Figure \ref{fig: rs-D4-V-algtop} illustrates these ideas at a high level. We refer to \cite{Hatcher2002} for more details on these ideas. 

\begin{figure}[h]{
\includegraphics{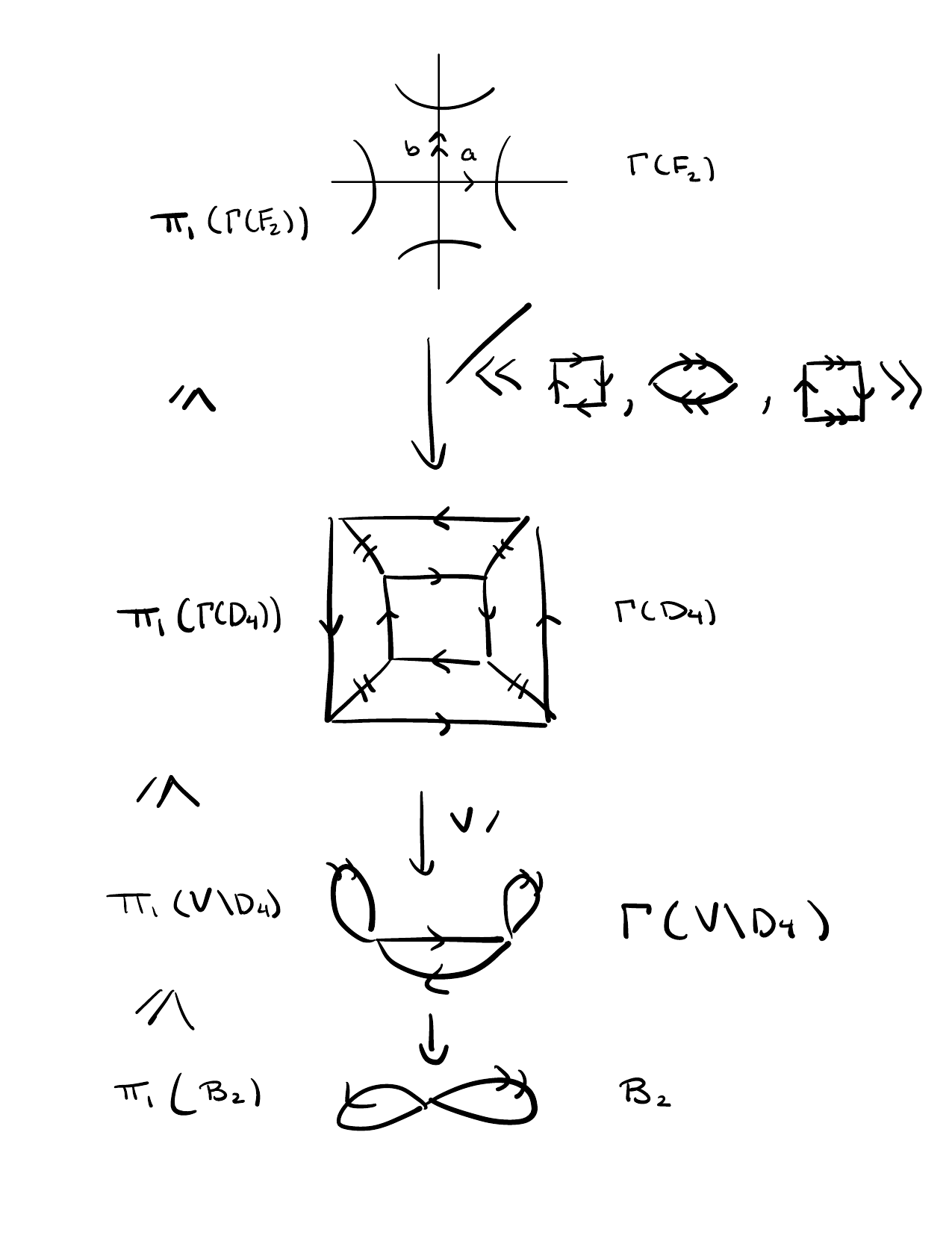}
}
\caption{We illustrate the underlying ideas of the Reidemeister-Schreier method from an algebraic topology perspective. At the top we start with $\Gamma(F)$, the Cayley graph for the free group $F$ for which $G$ is a subgroup, which is a universal cover and whose fundamental group is trivial. Then, by gluing the $2$-cell relations found in the presentation of $G$ to $\Gamma(F)$, we obtain the $\Gamma(G)$, where $G = F / \langle \langle R \rangle \rangle$. Since we have only added possible ways to induce new loops by gluing $2$-cells to $\Gamma(F)$, it follows that $\pi_1(F) \leq \pi_1(\Gamma(G))$. By shrinking $\Gamma(G)$ by its $H$-cosets, we obtain $\Gamma(G/H)$, we are once again folding the loops found in $\pi(\Gamma(G))$, giving us $\pi_1(\Gamma(G)) \leq \pi_1(\Gamma(G/H)$, which can again be folded into $\pi_1(B_{|X|})=F$, the fundamental group of the bouquet of $|X|$ flowers, which is simply $F$. This gives us the natural observation that all the graphs above $B_{|X|}$ are covers for it.
}
\label{fig: rs-D4-V-algtop}
\end{figure}

In order to pass from the relations of $G$ to the relations of $H$, we use $\tau$ as defined in Theorem \ref{thm: RS}. Indeed, the rewriting map $\tau: F \to F'$ rewrites each segment of a word from the alphabet $X$ to $Y$ as the generating loops of the Schreier graph $\Gamma(H \backslash G, X)$. Figures \ref{fig: rs-D4-V-tau-1}, \ref{fig: rs-D4-V-tau-2}, and \ref{fig: rs-D4-V-tau-3} illustrate the rewriting by $\tau$ for every relation and their conjugates by transversals, along their algebraic computations. The reason why we want the conjugates of the relations in $R$ by their transversals is algebraic and will be made clear in Section \ref{chap: rs, sec: pf}. We note that topologically, it is quite neat that the representative two-cells of the $2$-complex to realise $\Gamma(H,Y)$ are exactly the conjugated copies $trt\inv$ of the two-cells given by $r \in R$ such that they are based minimal spanning tree $T \ni t$ of the Schreier graph $\Gamma(H \backslash G, X)$. 

\begin{figure}[h]{
\includegraphics{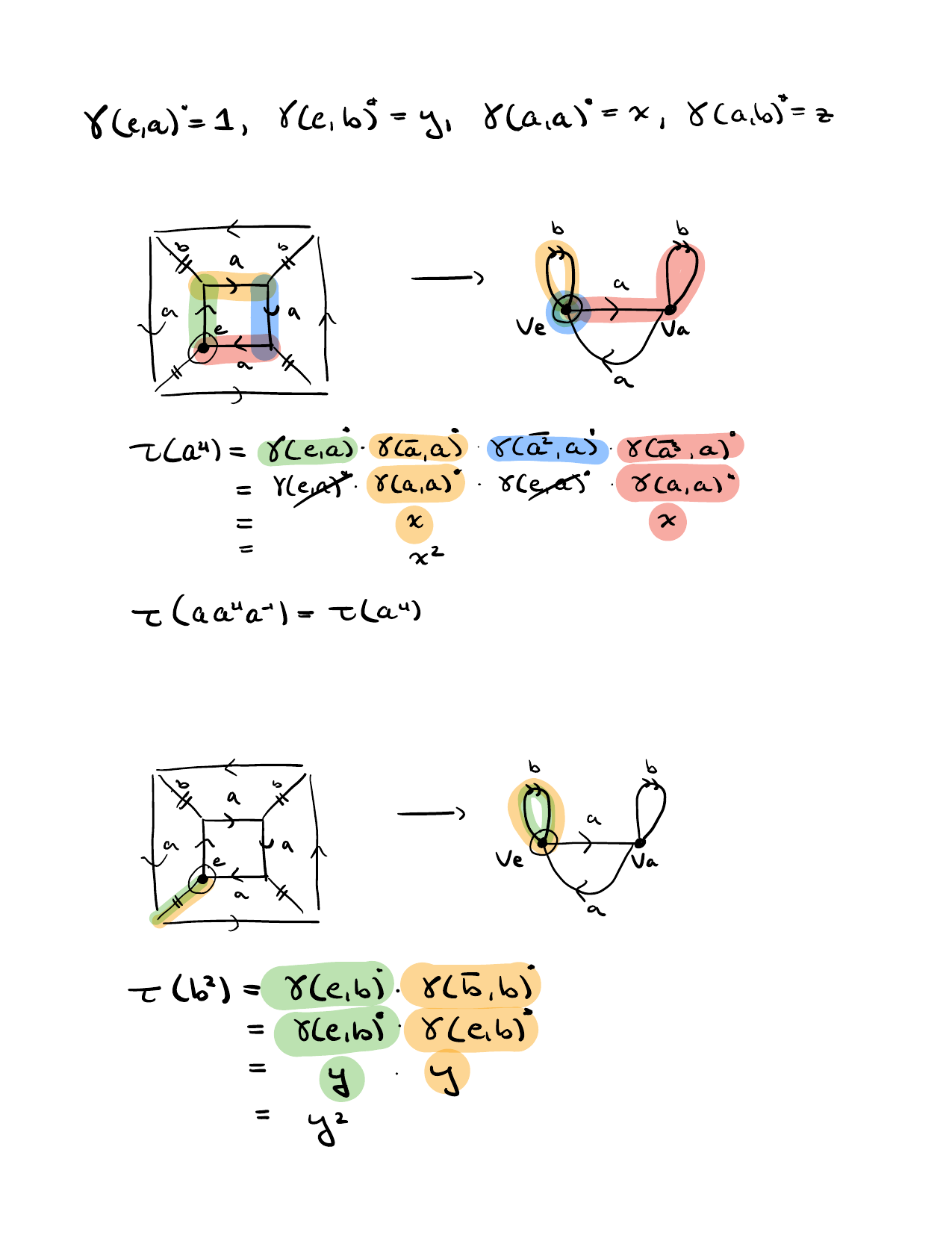}
}
\caption{Computing the relations of $H$ using the rewriting map $\tau$, part 1 of 3. Each initial segment of a relation in $\Gamma(G,X)$ is transformed into a generating loop in $\Gamma(H \backslash G, X)$. 
}
\label{fig: rs-D4-V-tau-1}
\end{figure}

\begin{figure}[h]{
\includegraphics{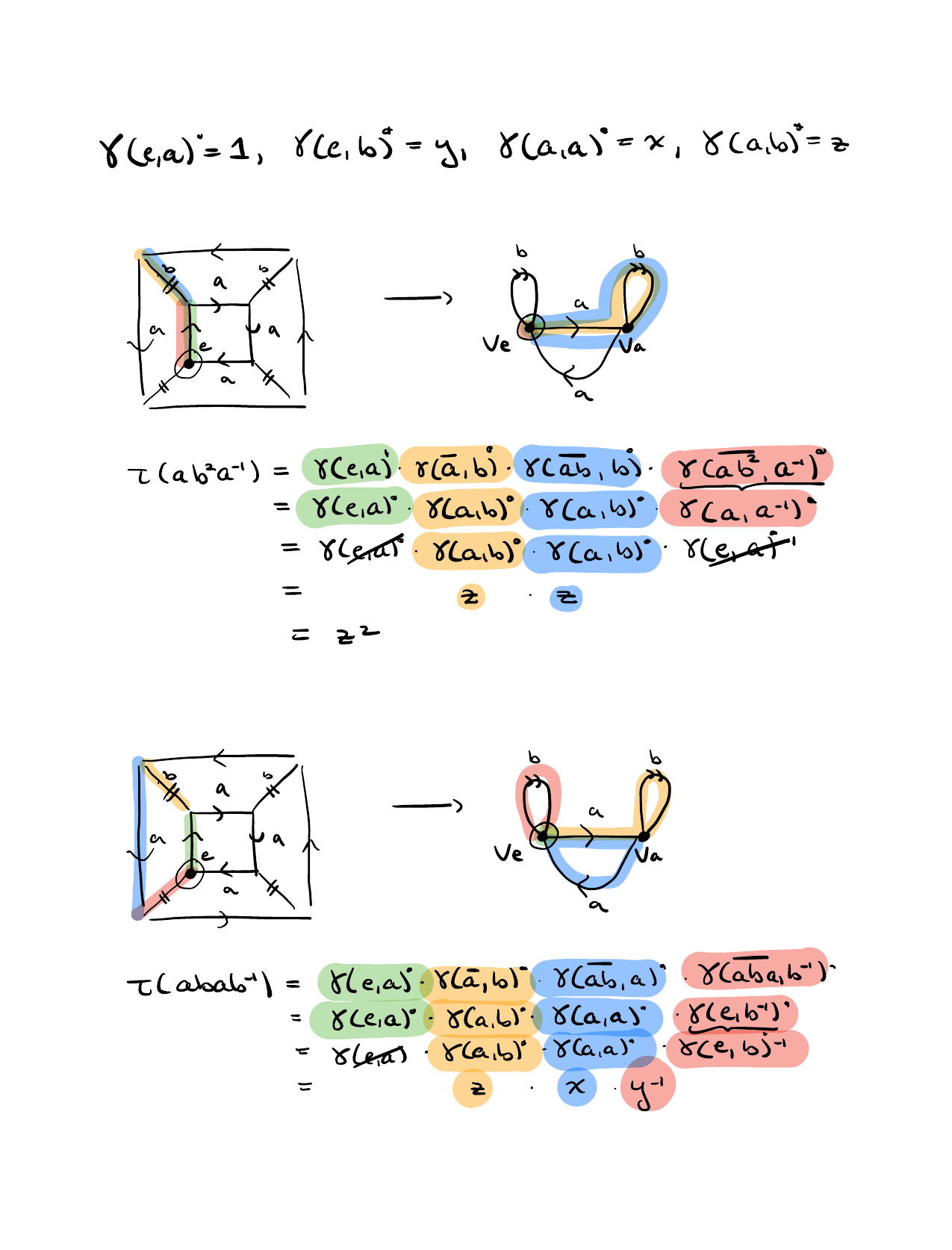}
}
\caption{Computing the relations of $H$ using the rewriting map $\tau$, part 2 of 3. Each initial segment of a relation in $\Gamma(G,X)$ is transformed into a generating loop in $\Gamma(H \backslash G, X)$.
}
\label{fig: rs-D4-V-tau-2}
\end{figure}

\begin{figure}[h]{
\includegraphics{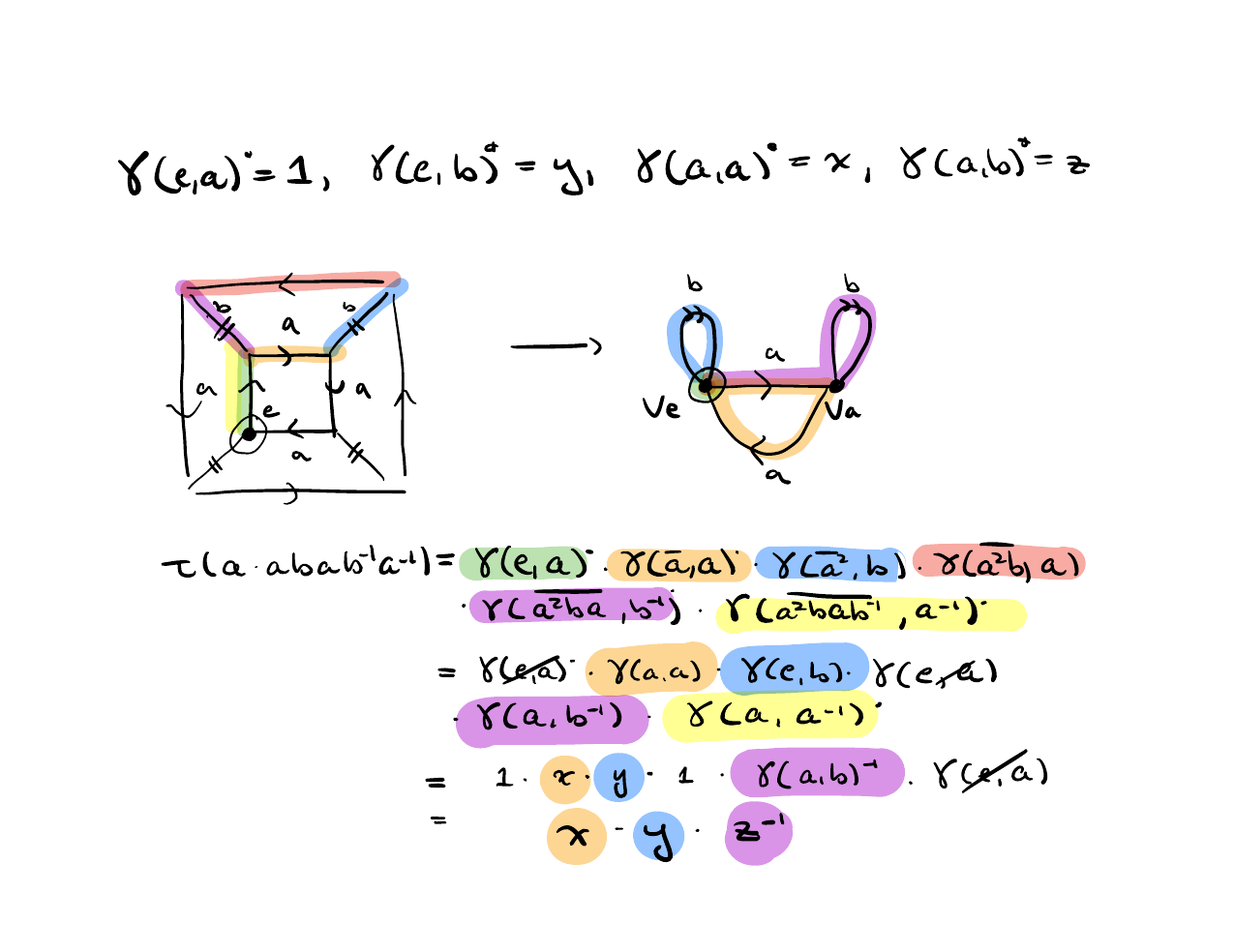}
}
\caption{Computing the relations of $H$ using the rewriting map $\tau$, part 3 of 3. Each initial segment of a relation in $\Gamma(G,X)$ is transformed into a generating loop in $\Gamma(H \backslash G, X)$.                                                                                    
}
\label{fig: rs-D4-V-tau-3}
\end{figure}

Finally, let us check the presentation we have obtained for $H$ and verify that it matches what we had at the beginning. A priori, we have $H = \langle x, y, z \mid x^2, y^2, z^2, zxy\inv, xyz\inv \rangle$ which is not the presentation we started with. However, since every generator is of order $2$, we get that $x = x\inv, y = y\inv, z = z\inv$. The last two relations then becomes $z = xy$, which gives us that $z^2 = xyxy = 1 \iff xyx = y$. We can now also eliminate $z$ from the generating set. Putting it all together, we get that $H = \langle x,y \mid x^2, y^2, xyx=y \rangle$, which is the presentation we started with at the beginning. 

Now that we have explored the visual ideas of this method, let us prove the statements formally. 
\section{An algebraic proof}\label{chap: rs, sec: pf}

An algebraic proof of the Reidemeister-Schreier method is actually straightforward. As a tradeoff, the many topological ideas we have explored above will be hidden in the calculations. Here, we are mainly using \cite{MagnusKarrassSolitar1966} as a reference, but adapting the formalism to the way it was written in \cite{LyndonSchupp2001}.

\begin{proof}[Proof of Theorem \ref{thm: RS}]
First note that each element of the form $\gamma(t,x)$ for $t \in \tilde T$ and $x \in X$ belongs to $\tilde H$. Indeed, since $\gamma(t,x) = (tx)(\ovl{tx})\inv$, and $tx \in Ht'$ for some $t'$, and $t' = \ovl{tx}$ by definition, $\gamma(t,x) \in \tilde H$. To show that $Y$ generates $\tilde H$, let $w = x_1 \dots x_n$ be some word in $F$ belonging to $\tilde H$. Then $\ovl{w} = 1$, and, by inserting words equating to the identity between each letters of $w$ and $\ovl{w}\inv = 1$ at the end of $w$,  we get the following equality in $F$:

\begin{align*}
	w &= x_1 (\ovl{x_1}\inv \ovl{x_1}) x_2 (\ovl{x_1 x_2}\inv \ovl{x_1 x_2}) \dots (\ovl{x_1 \dots x_{n-2}}\inv \ovl{x_1 \dots x_{n-2}}) x_{n-1} (\ovl{x_1 \dots x_{n-1}}\inv \ovl{x_1 \dots x_{n-1}}) x_n \ovl{w}\\
	&= (x_1 \ovl{x_1}\inv) (\ovl{x_1} x_2 \ovl{x_1 x_2}\inv) (\ovl{x_1 x_2} \dots \ovl{x_1 \dots x_{n-2}}\inv \ovl{x_1 \dots x_{n-2}} x_{n-1} \ovl{x_1 \dots x_{n-1}}\inv) (\ovl{x_1 \dots x_{n-1}} x_n \ovl{w}\inv) \\
	&= \gamma(e,x_1) \gamma(\ovl{x_1}, x_2)\dots \gamma(\ovl{x_1 \dots x_{n-2}}, x_{n-1}) \gamma(\ovl{x_1 \dots x_{n-1}}, x_n).
\end{align*}
Therefore, $w$ in the $Y$ alphabet is $\tau(w)$ by using the one-to-one correspondence between $\gamma(t,x)$ and $\gamma(t,x)^*$, which serves to formally distinguish between the alphabets of $X$ and $Y$. This shows that $Y$ generates $\tilde H$ as claimed. 

As for the relations, recall the notation $G = \langle X \mid R \rangle$ means that the entire normal closure $N = \langle \langle R \rangle \rangle$ of $G$ gives the identity. %
Thus, to find the relations for $H$, we must rewrite the normal closure of $N$ under the generators $Y$. Let $w$ be any freely reduced word in $F$. Since $\phi(wrw\inv) = 1 \in H$, $wrw\inv \in \tilde H$ and therefore, the rewriting can be  given by the map $\tau$. Thus, the set $M = \{\tau(wrw\inv) \mid r \in R, w \in F\}$ give us the desired quotient map $F' \to F'/M = H$. 

To obtain the presentation given in the theorem statement, we make the following observation. Since $w$ can be rewritten as $w = ut$, where $u \in \tilde H$ and $t \in \tilde T$, we get that $(ut)r(ut)\inv = u (t r t\inv) u\inv$. Moreover, $\tau(u (t r t\inv) u\inv) = \tau(u)\tau(trt\inv) (\tau(u))\inv$ so the relations in $H$ are given precisely the normal closure of $\{\tau(trt\inv) \mid t \in \tilde T, r \in R\}$. 
\end{proof}

%% file: chap/transducers.tex
\chapter{Transducers}\label{chap: transducers}
We introduce transducers here, which are equivalent to finite state automata but whose underlying ideas lend themselves better to the results of  Chapter \ref{chap: cross-Z}. In that chapter, we will be using transducers to prove how the group $\mathbb{Z}$ can function as a stack for certain one-counter positive cones of $G$, resulting in a regular positive cone for $G \times \mathbb{Z}$. 

The main reference for this chapter is a textbook by Berstel \cite{Berstel2009}.

We note that apart from the definition of transducers as finite state automata (Section \ref{sec: trans-fsa}), the rest of the chapter is optional. We will indicate a suggested skipping point at the appropriate place. 

\section{Transducers as finite state automata}\label{sec: trans-fsa}
A (rational) transducer is a finite state automaton which we view as
outputting a finite number of symbols for each input symbol. We clarify what we mean by this formally. The following definition is essentially a restatement of Nivat's Theorem proving the equivalence between transducers and finite state automaton, which we will prove at the end of the chapter (at \ref{thm: trans-Nivat}). 

\begin{defn}\label{defn: trans-fsa}
A \emph{(rational) transducer} $\bT$ is a finite state automaton in the sense of Definition \ref{defn: fsa} given by the quintuple $\bA = (S,X \sqcup Y,\delta, s_0, A)$, where $\bA$ is not assumed to be deterministic. We write the corresponding transducer $\bT$ as a refactored sextuple $\bT =(S,X,Y,\delta', s_0, A)$ where we call $S$ the set of states, $X$ the input alphabet, and $Y$ the output alphabet, with $X$ and $Y$ assumed to be disjoint from one another (for simplicity). The element $s_0\in S$ is called the initial state and $A\subseteq S$ is called the set of accept states. 

Let $\cP(S)$ denote the power set of $S$. Recall that for a non-deterministic finite state automaton, the transition function $\delta$ is given by $$\delta: S \times (X \sqcup Y) \to \cP(S).$$

In the transducer view, the function $\delta'$ is a refactored version of $\delta$ given by 
\begin{align*}
	& \delta': S \times X \to \cP(S \times Y) \\
	& \delta'(s_1,x) \ni (s_2, y) \iff \delta(s_1, x/y) \ni s_2.
\end{align*}

where the `/' character serving as a visual partition between the disjoint alphabets of $X$ and $Y$. We view $\delta'$ as encoding that while $\bT$ is in state $s_1$ and reads $x$, it goes into state $s_2$ and outputs $y$.

Note that for simplicity, we may write $\delta': S \times X \to S \times Y$, leading the power set notation out. Moreover, we can recursively extend $\delta'$ to $\delta': S \times X^* \to S \times Y^*$ in the natural way. 
\end{defn}

\begin{figure}[h]
\centering
{
\includegraphics[width = \textwidth]{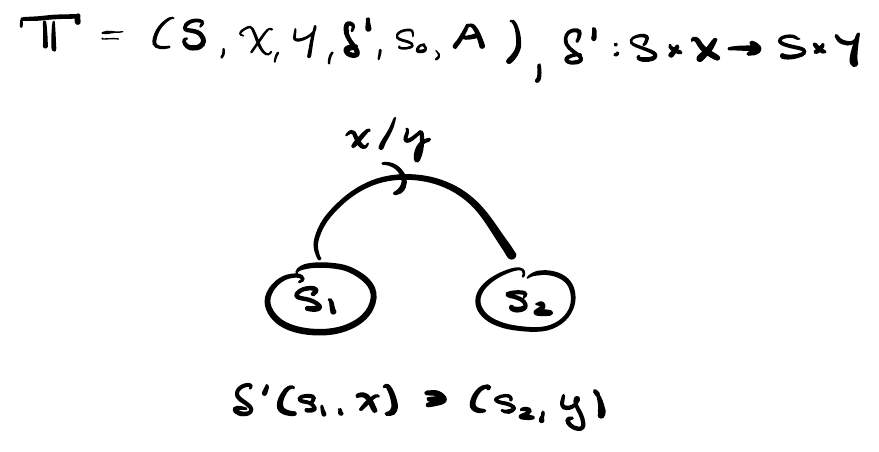}
}
\caption{A graphical representation of a transducer $\bT$ in terms of the graphical representation of a finite state automaton. The states $s, s'$ are given by circles, while the directed edge labelled $x/y$ takes $s$ to $s'$ given transition $\delta(s,x) \ni (s',y)$. 
}
\label{fig: trans-aut}
\end{figure}

\begin{rmk}\label{rmk: graph transducer}
We will also use the graph representation of a transducer $\bT$. It is similar to the graph representation of the finite state automaton associated with $\bT$, with a few changes. 

If $\bT =(S,X,Y,\delta', s_0,A)$ is a rational transducer, then its associated graph $\cG(T)$ has vertex set $S$. There is a directed edge labelled $x/y$ from $s_1\in S$ to $s_2\in S$ whenever $(s_2,y)\in \delta'(s_1,x)$ for some $x\in X$ and $y\in Y$. This is illustrated in Figure \ref{fig: trans-aut}.
\end{rmk}

The reader may ask at this point why bother with transducers, which are essentially a special case of finite state automata. For the purpose of the work presented in this thesis, the point of introducing transducers is to be able to work with AFLs using the frameworks of recognisable sets and rational families in order to exploit certain closure properties which will become relevant in Chapter \ref{chap: cross-Z}. For the purpose of following the results however, the above definition may be enough if they are willing the above definition, and hence to accept Nivat's Theorem, as a blackbox. The rest of this chapter builds up to this the statement and proof of this theorem. Along the way, it explores the relevant machinery with the hope Chapter \ref{chap: cross-Z} will be more rewarding to the reader if they choose to go along with our digression. From here on, we invite the reader to skip ahead at their leisure. 

\section{Transducers through rational sets and preimages}
In the sequel, we will explore how transducers are at the intersection of rational sets and recognisable languages, which are languages obtained by taking preimages of morphisms into finite monoids. We will show the closure properties of both class of sets and see how this can be used for the proof of Nivat's theorem.

Let us first step back from the finite state automata definition of a transducer, and view a transducer more abstractly as an automaton realising a rational transduction. We give the relevant definitions below.   

\subsection{Transducers as rational transductions}

\begin{defn}[Transduction]
	A \emph{transduction} $\tau$ from $X^* \to Y^*$ is a function $X^* \to \cP(Y^*)$ but written $\tau: X^* \to Y^*$, where 
	\begin{align*}
		\text{domain}(\tau) &= \{x \in X^* \mid \tau(x) \not= \emptyset\} \\
		\text{image}(\tau) &= \{y \in Y^* \mid \exists x \in X^* : y \in \tau(x)\}.
	\end{align*}
	We can extend the mapping to power sets, $\tau: \cP(A^*) \to \cP(B^*)$ by defining
	$$\tau(\mathcal{X}) := \bigcup_{x \in X} \tau(x).$$
\end{defn}

\begin{defn}[Relation]
	Let $X$ and $Y$ be alphabets. A \emph{relation} is an element $r \in X^* \times Y^*$.  
\end{defn}

\begin{defn}[Graph of a transduction]
	We can define a \emph{graph} of a transduction by 
	$$R = \{(x,y) \in X^* \times Y^* \mid y \in \tau(x)\}.$$
	
	Similarly, a relation defines a transduction $\tau: X^* \to Y^*$, 
	$$\tau(x) = \{y \in Y^* \mid (x,y) \in R\}.$$
\end{defn}

Thus, the study of transductions is the study of relations. A rational transducer is a transducer defined by a \emph{rational} relation, that is, a relation that belongs to a rational family, as defined below. 

\begin{defn}\label{defn: trans-rat}
	Let $M$ be a monoid. The family of \emph{rational} subsets $\Rat(M)$ is the least family $\cR$ satisfying
	\begin{enumerate}
		\item $\emptyset \in \cR$,
		\item $\forall m \in M, \quad \{m\} \in \cR$, 
		\item if $S, T \in \cR$, then $ST \in \cR$ and $S \cup T \in \cR$, 
		\item if $S \in \cR$, then $S^+ = \bigcup_{n\geq 1}S^n \in \cR$
	\end{enumerate} 
\end{defn}

\begin{rmk}\label{rmk: trans-rat-reg}
	Observe that for any finite alphabet $X$, $M = X^*$, $\Rat(M)$ is the set of rational expressions over $X$, as defined in Chapter \ref{chap: fsa}. 
\end{rmk}

Thus, $\Rat(M)$ is a generalisation of rational expressions over a general monoid $M$. 

\begin{defn}
	A \emph{rational relation} is an element of $\Rat(X^* \times Y^*)$ for some alphabets $X,Y$. 
\end{defn}

Putting everything together, we get the following definition of a rational transduction. 

\begin{defn}
	A transduction $\tau$ is \emph{rational} when it is defined by a rational relation $R \in \Rat(X^* \times Y^*)$. That is, if $\tau: X^* \to Y^*$, 
	$$\tau(x) = \{y \in Y^* \mid (x,y) \in R\}.$$ 
\end{defn}

We end this section by stating some properties of rational sets under morphisms, which will be heavily used in the proof of Nivat's Theorem (\ref{thm: trans-Nivat}). 

\begin{prop}[Closure properties of rational sets]\label{prop: trans-rat-closure}
	Let $M,M'$ be monoids, $\phi: M \to M'$ be a monoid morphism, and $\Rat(M), \Rat(M')$ be the families of rational subsets over $M$ and $M'$ respectively. The following closure properties hold. 
	\begin{enumerate}
		\item Rationality is closed under morphisms. That is, if $R \in \Rat(M)$, then $\phi(R) \in \Rat(M')$.
		\item If a monoid morphism is surjective from $M \to M'$, then it is surjective when passing to the rational families from $\Rat(M) \to \Rat(M')$ as well. That is, if $\phi$ is surjective, then for all $R' \in \Rat(M')$, there exists an $R \in \Rat(M)$ such that $\phi(R) = R'$. 
	\end{enumerate} 
\end{prop}
\begin{proof}
	To prove Statement 1, let $\cR$ be the family of subsets $R \subseteq M$ such that $\phi(R) \in \Rat(M')$. We will show that $\cR \supseteq \Rat(M)$, satisfying the requirement of the statement. 
	
	First note that by definition $\emptyset \in \cR$ since $\phi(\emptyset) = \emptyset \in \Rat(M')$, and similarly that $\{m\} \in \cR$ for all $m \in M$ since $\phi(m) = m' \in \Rat(M')$. Next, note that by properties of morphisms, for any $S, T \subseteq M$, we have $\phi(S \cup T) = \phi(S) \cup \phi(T)$, $\phi(ST) = \phi(S)\phi(T)$, $\phi(S^+) = (\phi(S))^+$, which are all in $\Rat(M')$. This implies that for any $S,T \in \cR$, $S \cup T, ST, S^+ \in \cR$, so $\cR$ satisfies the conditions of $\Rat(M)$ in Definition \ref{defn: trans-rat}. Thus, $\cR \supseteq \Rat(M)$, proving the first statement. 
	
	To prove Statement 2, let $\cS$ be the family $R' \subseteq M'$ such that $R' = \phi(R)$ for some $R \in \Rat(M)$. We will similarly show that $\cS \supseteq \Rat(M')$. Since $\phi$ is surjective, there exists a $\{m'\} \in \cS$ for all $m' \in M'$. It is clear that $\emptyset \in \cS$. As in the previous paragraph, we have that $\cS$ is closed under union, products, and the Kleene plus operation. Thus, $\cS \supseteq \Rat(M')$. 
	 \end{proof}
	
\subsection{Regular languages as preimages of morphisms onto finite monoids}
We now introduce a result that allows us to view regular languages as finite state automata as preimages of morphisms onto finite monoids. This shift in view is what will permit us to define rational transductions as finite state automata. 

\begin{defn}[Alphabetic morphism]
	A morphism $\alpha: X^* \to Y^*$ is \emph{alphabetic} if $\alpha(X) \subseteq Y \cup \{1_Y\}$, that is, it sends each letter of $X$ to another letter in $Y$ or to the monoid identity $1_Y$.
\end{defn}

\begin{prop}\label{prop: trans-morphism-fsa}
	Let $X$ be an alphabet, and $L \subseteq X^*$. Then $L$ is a regular language if and only if there exists a finite monoid $N$, a set $P \subseteq N$ and an alphabetic morphism $\alpha: X^* \to N$ such that $\alpha\inv(P) = L$. 
\end{prop}
\begin{proof}
	($\implies$) Let $\bA = (S, X, \delta, s_0, A)$ be a (deterministic) finite state automaton accepting $L$. Let $N = S^S$ be the set of all functions from $S \to S$, and observe that $N$ satisfies the properties of a monoid. Let $$P = \{f \in N: f(s_0) \in A\},$$
	that is, $P$ is the set of functions mapping the start state to an accept state. We can realise each word $w \in X^*$ as a function taking states to their transitioned state via $w$, that is $\tilde w(s) := \delta(s,w)$, and define $\alpha: X^* \to N$ as $$\alpha(w) := \tilde w.$$ Then, $\alpha$ satisfies the conditions of a monoid morphism and it is straightforward to conclude that by construction we must have $\alpha\inv(P) = L$, as $L$ is the collection of words taking the start state to final states in this framework. 
	
	($\impliedby$) Using $N, P, \alpha: X^* \to N$ as in the statement, we define a finite state automaton $\cA = (N, X, \delta, 1_M, P)$ where 
	$$\delta(n, x) = n \cdot \alpha(x)$$ where $\cdot$ is the monoid operation on $N$ as illustrated in Figure \ref{fig: trans-aut-pf}. 
	Then, the preimage $\alpha\inv(P)$ is precisely the words taking $1_N$ to $P$. 
\end{proof}

\begin{figure}[h]
\centering
{
\includegraphics[width = \textwidth]{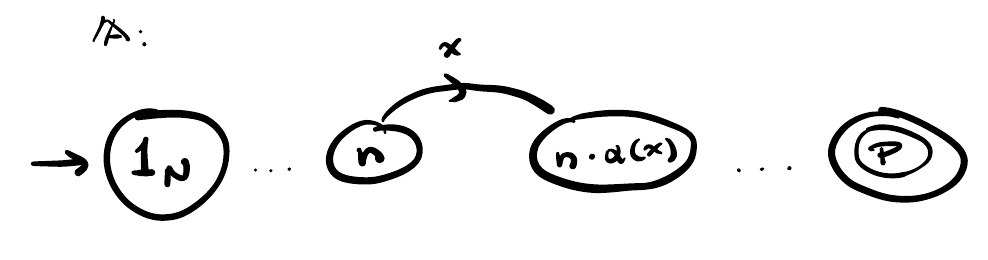}
}
\caption{Partial picture of $\bA$ automaton of Proposition \ref{prop: trans-morphism-fsa}. The finite states are given by the elements of the finite monoid $N$, and the accepted words $w$ are exactly those whose corresponding map $\tilde{w}$ that map the starting state $1_N$ to $P$.
}
\label{fig: trans-aut-pf}
\end{figure}

In the previous equivalence between finite state automata and preimage of alphabetic morphisms onto finite monoids, we have just seen that the finite codomain replaces the finite memories of the machines. Thus, one way we can generalise regular languages under this framework is by keeping the requirement of a finite monoid codomain, but dropping the requirements that the morphism be alphabetic and that the monoid in the domain be free and finitely generated. 

Via this preimage definition, we obtain the complexity class of recognisable languages, with neat closure properties. 

\begin{defn}\label{defn: trans-rec}
	A subset $R \subseteq M$ is called \emph{recognisable} if there exists a monoid $M$, a finite monoid $N$ and a monoid morphism $\phi: M \to N$ such that for some $P \subseteq N$, $R = \phi\inv(P)$. We denote the class of recognisable sets over $M$ as $\Rec(M)$. 
\end{defn}
\begin{rmk}\label{rmk: trans-rec-surj}
	In the setting of Definition \ref{defn: trans-rec}, we can always assume that $\phi$ is surjective. Indeed, suppose that it is not. Observe that it is always true that since $R = \phi\inv(P)$, $\phi(R) = P \cap \phi(M)$. Let $\psi:= M \to \phi(M)$ be the natural restriction of $\phi$ on the codomain $\phi(M)$, and let $Q := P \cap \phi(M)$. Then $\phi(M) \subseteq N$ is a finite monoid since $M$ is a monoid, $\phi$ is a monoid morphism and $N$ is finite, making $\psi$ also a monoid morphism onto a finite monoid. Moreover, $R = \psi\inv(Q)$ where $Q \subseteq \phi(M)$ and $\psi$ is surjective by construction. Thus, we can always replace $\phi$ with the surjective map $\psi$ without changing $R$. 
\end{rmk}

\begin{prop}[Closure properties of recognisable languages]\label{prop: trans-rec-closure}
	Let $M$ be a monoid and $\Rec(M)$ be the class of recognisable language sets over $M$. The following closure properties hold. 
	\begin{enumerate}
		\item $\Rec(M)$ is closed under union, intersection and complementation. That is, if $S, T \in \Rec(M)$, then $S \cup T, S \cap T, M - S \in \Rec(M)$. 
		\item $\Rec(M)$ is closed under differences. That is, if $S, T \in \Rec(M)$, then $S - T \in \Rec(M)$. 		
		\item $\Rec(M)$ is closed under inverse morphisms. That is, if $M$ and $M'$ are monoids, $\phi: M \to M'$ a morphism and $R' \in \Rec(M')$, then $R = \phi\inv(R') \in \Rec(M)$. 
		\end{enumerate}
\end{prop}
\begin{proof}
	(1) We will first show closure under intersection and complementation, then deduce closure by union using de Morgan's law. Note that by Remark \ref{rmk: trans-rec-surj}, we can always assume the map of the definition of a recognisable set in \ref{defn: trans-rec} is surjective. 
	
	Let's start with complementation. Let $S \in \Rec(M)$, let $N$ be a finite monoid, and let $\psi: M \to N$ be a surjective morphism with $P \subseteq N$ such that $S = \psi\inv(P)$. Then, $M - S = \psi\inv(N - P)$ as that is everything in the domain that does not map to $P$. Thus, $M - S \in \Rec(M)$. 
	
	For intersection, let $S, \psi, N$ be as in the above paragraph, and let $T \in \Rec(M), T = {\psi'}^{-1}(Q)$, where $\psi'$ is a surjective morphism from $M$ onto a finite monoid $N'$ with $Q \subseteq N'$. Let $N'' = N \times N'$ be the product monoid, and define $\gamma: M \to N''$ by $\gamma(m) := (\psi(m), \psi'(m))$ for all $m \in M$. Then, $\gamma$ is a morphism with the property that $\gamma(m) \in P \times Q$ if and only if $\psi(m) \in P$ and $\psi'(m) \in Q$, and thus if and only if $m \in \psi\inv(P) \cap {\psi'}^{-1}(Q)$. This means that $S \cap T = \gamma\inv(P \times Q)$, and since $N''$ is finite, $S \cap T \in \Rec(M)$. 
	
	Finally, closure under union is obtained by applying de Morgan's law. That is, $S \cup T = M - ((M - S) \cap (M - T)) \in \Rec(M)$ because its individual terms are in $\Rec(M)$. 
	
	(2) As a corollary to (1), $S - T = S - (S \cap T) \in \Rec(M)$. 
	
	(3) Let $\psi: M' \to N$ now be a surjective morphism onto a finite monoid $N$, let $P \subseteq N$ such that $R' = \psi\inv(P)$. Then $\phi\inv(R') = \phi\inv(\psi\inv(P)) = (\phi \circ \psi)\inv(P)$, and thus $R = \phi\inv(R') = \Rec(M)$. 
\end{proof}

\subsection{Rational transductions and regular languages}
The reader may have observed that regular languages satisfy both the requirements of a recognisable class of languages and a rational class of sets. This is just a restatement of Kleene's theorem from Chapter \ref{chap: fsa}! %
 \begin{thm}[Kleene]
 	For a finite alphabet $X$, $\Rec(X^*) = \Rat(X^*) = \Reg(X^*)$, where $\Reg(X^*)$ is the family of all regular languages over $X^*$. 
 \end{thm}
 \begin{proof}
 	By Proposition \ref{prop: trans-morphism-fsa}, $\Rec(X^*) = \Reg(X^*)$ when viewed as the set of languages over $X$ accepted by a finite state automaton. By Remark \ref{rmk: trans-rat-reg}, $\Rat(X^*) = \Reg(X^*)$ when viewed as the set of regular expressions over $X^*$. 
 \end{proof}
	
When passing to from finitely generated free monoids to monoids underlying relations, Kleene's theorem breaks down. 

\begin{prop}\label{prop: trans-rat-rec-relations}
	Let $X,Y$ be alphabets. Then $\Rec(X^* \times Y^*) \subset \Rat(X^* \times Y^*)$ where the inclusion is strict. 
\end{prop}

To show this, we will make use of 1964 proposition by McKnight. 
\begin{prop}[McKnight]\label{prop: trans-McKnight}
	Let $M$ be a finitely generated monoid. Then $\Rec(M) \subseteq \Rat(M)$. 
\end{prop}
\begin{proof}
	Since $M$ is finitely generated, there exists a finite alphabet $X$ and a surjective morphism $\phi: X^* \to M$ such that $\phi(X^*) = M$. Let $R \in \Rec(M)$. Then $\phi\inv(R) \in \Rec(X^*)$ by closure properties of recognisable languages (Proposition \ref{prop: trans-rec-closure}). By Kleene's Theorem, $\Rec(X^*) = \Rat(X^*)$ and thus $\phi\inv(R) \in \Rat(X^*)$. Now, $R = \phi(\phi\inv(R)) \in \Rat(M)$ by closure properties of rational sets (Proposition \ref{prop: trans-rat-closure}). 
\end{proof}

\begin{proof}[Proof of Proposition \ref{prop: trans-rat-rec-relations}]
	Since $X^* \times Y^*$ is finitely generated by $\{(x,1), (1,y) \mid x \in X, y \in Y\}$, $\Rec(X^* \times Y^*) \subseteq \Rat(X^* \times Y^*)$ by Proposition \ref{prop: trans-McKnight}.  
	
	To show the strictness of the inclusion, we construct an element of $\Rat(X^* \times Y^*)$ not in $\Rec(X^* \times Y^*)$.  Suppose that $C = (a,b)^* = \{(a^n, b^n) \mid n \geq 0 \}$. Then, $C \in \Rat(X^* \times Y^*)$ as $(a,1)(1,b) = (a,b) \in \Rat(X^* \times Y^*) \implies (a,b)^* \in \Rat(X^* \times Y^*)$. 
	
	Assume for contradiction that $C \in \Rec(X^* \times Y^*)$. Let $Z = \{\bar a, \bar b\}$ be the alphabet with letters $\bar a, \bar b$, and let $\gamma: Z^* \to X^* \times Y^*$, where $\bar a = (a, 1), \bar b = (1,  b)$. Let $\#_z : Z^* \to \mathbb{N}$, $\#(w)_z$ count the instances of a letter $z \in Z$'s in a word $w \in Z^*$. Then, $\gamma\inv (C) = \gamma\inv ((a,b)^*)  = \{w \in Z^* \mid \#(w)_{\bar a} = \#(w)_{\bar b}\} \in \Rec(Z^*)$ because $\Rec$ is closed under inverse morphism. However, $\gamma\inv(C)$ is not a regular language as seen in Chapter \ref{chap: fsa}. 
	This is our desired contradiction.
\end{proof}

This is where it is relevant to introduce Nivat's Theorem. The following result uses Kleene's Theorem to relate rational relations and regular languages in view of the previous proposition (\ref{prop: trans-rat-rec-relations}).

\begin{thm}[Nivat's theorem]\label{thm: trans-Nivat}\sidenote{Note that we have changed and shortened the statement slightly from the one found in Berstel to something more directly useful to the thesis.}
	Let $X$ and $Y$ be disjoint alphabets, and $R \subseteq (X \times Y)^*$. The following are equivalent. 
	\begin{enumerate}
		\item $R \in \Rat(X^* \times Y^*)$.
		\item There exists an alphabet $Z$ and two alphabetic morphisms $\alpha: Z^* \to X^*, \beta: Z^* \to Y^*$ and a regular language $L \subseteq Z^*$ such that $$R = \{(\alpha(w), \beta(w)) \mid w \in L\}.$$
		\item Let $\pi_X: (X \sqcup Y)^* \to X^*$ and $\pi_Y: (X \cup Y)^* \to Y^*$ be the projections from $(X \cup Y)^*$ to $X^*$ and $Y^*$ respectively. There exists a regular language $L \subseteq (X \sqcup Y)^*$ such that $$R = \{(\pi_X(w), \pi_Y(w)) \mid w \in L\}.$$
		\end{enumerate}
\end{thm}

\begin{proof}
	 (2 $\implies$ 1): Let $\gamma: Z^* \to X^* \times Y^*$ be a morphism, such that $\gamma(z) = (\alpha(z), \beta(z))$. Then, $\gamma(L) = R$ by assumption. We know that $L \in \Rec(X^*) = \Rat(X^*)$, and therefore $\gamma(L) = R \in \Rat(X^* \times Y^*)$ by Statement 1 of Proposition \ref{prop: trans-rat-closure}. 
	 	 
	 (1 $\implies$ 3): Define $\pi: (X \cup Y)^* \to X^* \times Y^*$ by $\pi(w) = (\pi_X(z), \pi_Y(z))$. Then, $\pi$ is surjective, which implies that if $R \in \Rat(X^* \times Y^*)$, then there exists a language $L \in \Rat((X \sqcup Y)^*) = \Rec((X \sqcup Y)^*)$ such that $\pi(L) = R$ by Statement 2 of Proposition \ref{prop: trans-rat-closure}. 
	 
	 (3 $\implies$ 1): Conversely, $\pi(L) = R \in \Rat(X^* \times Y^*)$ by Statement 1 of Proposition \ref{prop: trans-rat-closure}. 
	 	 
	 (1 $\implies$ 2): If $R \in \Rat(X^* \times Y^*)$, then $\pi$ as above is the alphabetic morphism we are looking for by the (1 $\implies$ 3) implication. 
\end{proof}

In other words, while the rational relations defining a rational transducers are not regular themselves, there is an overlying regular language from which the relation is an image pair from two different morphisms. Statement 3 of Theorem \ref{thm: trans-Nivat} is precisely what we are using for the finite state automaton view of a transducer in Definition \ref{defn: trans-fsa} we saw at the beginning of the chapter. Indeed, the automaton of the definition realises a regular language $L \subseteq (X \sqcup Y)^*$, which is the overlying regular language that is spelled by the edges with the $x/y$, with $x \in X^*, y \in Y^*$. (The `$/$' character is added for clarity when passing to the finite state automaton of Definition \ref{defn: trans-fsa}, but observe that it is not necessary as the alphabets of $X$ and $Y$ are already assumed to be disjoint in this paradigm.) The transition function $\delta'$ is doing the transduction by $\tau$, with the words in $X^*$ as inputs and the words in $Y^*$ as output. If we write $\delta': S \times X^* \to S \times Y^*$, then $\tau: X^* \to Y^*$ is the natural associated restriction of $\delta'$. 

The following corollary makes the relation of $\tau$ with $L$ precise in light of Nivat's theorem. 

\begin{cor}
	Theorem \ref{thm: trans-Nivat} in terms of rational transducer translates as the following. For $X$ and $Y$ disjoint alphabets, the following are equivalent. 
	\begin{enumerate}
		\item $\tau: X^* \to Y^*$ is a rational transducer. 
		\item There exists an alphabet $Z$ and two alphabetic morphisms $\alpha: Z^* \to X^*, \beta: Z^* \to Y^*$ and a regular language $L \subseteq Z^*$ such that for all $x \in X^*$, $$\tau(x) = \beta(\alpha\inv(x) \cap L).$$
		\item There exists a regular language $L \subseteq (X \cup Y)^*$ such that for all $x \in X^*$, $$\tau(x) = \pi_B(\pi_A\inv(x) \cap L).$$
	\end{enumerate}
\end{cor}
\begin{proof}
It is clear that Statement 1 of this corollary is equivalent to Statement 1 of Theorem \ref{thm: trans-Nivat}. Suppose that $L \subseteq Z^*$ is a regular language defining a rational relation $R$ via two morphisms $\phi:Z^* \to X^*$, $\psi: Z^* \to Y^*$, $$R = \{(\phi(w), \psi(w) \mid w \in L\}.$$

	Using the equivalence of Statement 2 of Theorem \ref{thm: trans-Nivat} with Statement 1 of this corollary, notice that the associated rational transducer $\tau: X^* \to Y^*$ is given by
	$$\tau(x) = \{y \in Y^* \mid \exists w \in L : x = \phi(w), y = \psi(w)\}.$$
	Denote $A := \{y \in Y^* \mid \exists w \in L : x = \phi(w), y = \psi(w)\}$ and $$B := \{y \in Y^* \mid y \in \psi(\phi\inv(x) \cap L\} = \{y \in Y^* \mid \psi\inv(y) \subseteq \phi\inv(x) \cap L\}.$$
	We claim that $A = B$. For the $B \subseteq A$ direction, let $y \in B$. Then, for any $w \in \psi\inv(y) \subseteq \phi\inv(x) \cap L$, that is, $w$ such that $\psi(w) = y$, we have that $w \in L$ and $\phi(w) = w$, proving the first claim. For the $A \subseteq B$ direction, suppose that $y \in A$. Then there exists $w \in L$ is such that $\phi(w) = x$, $\psi(w) = y$. Then, $w \in \psi\inv(y)$ and $w \in \phi\inv(x) \cap L$, and thus $\psi\inv(y) \subseteq \phi\inv(x) \cap L$ for all $y \in A$, proving the second claim. 
	
	Thus, $$\tau(x) = \psi(\phi\inv(x) \cap L).$$
	
	To see the equivalence of Statement 2 and Statement 3 of Theorem \ref{thm: trans-Nivat}, to Statement 2 and 3 of this corollary, we can simply replace $\phi, \psi$ by $\alpha, \beta$ and $\pi_X, \pi_Y$ respectively.
	
\end{proof}

We have now gone full circle to beginning of the chapter. Although this digression was not strictly necessary, we that what was explored here gives the reader a good overview of the ideas behind transducers, which will be relevant again in Chapter \ref{chap: cross-Z}.

%% file: chap/closure-finite-index.tex
\chapter{Closure under subgroups}\label{chap: closure-finite-index}

This chapter focuses on the closure properties of positive cones complexity under subgroups. 

\section{Main results}
We propose a criterion for preserving the regularity of a formal language representation when passing from groups to subgroups, called language-convexity (see Section \ref{sec: lang-conv}). We use language-convexity to show that the regularity of a positive cone language in a left-orderable group passes to its finite index subgroups and to finitely generated subgroups $H$ in an overgroup $G = H \rtimes \bZ$ which admit a regular lexicographic left-order led by $\bZ$. 

\begin{thm}[See Section \ref{sec: lang-conv-fi}]\label{thm: closure-fi-reg}
	Let $G$ be a finitely generated group with a regular positive cone. If $H$
is a finite index subgroup, then $H$ also admits a regular positive cone.
\end{thm}

\begin{thm}[See Section \ref{sec: lang-conv-order-convex}]\label{thm: lang-conv-lex-clos}
	Let $G = H \rtimes \bZ$ be a finitely generated group with a regular lexicographic positive cone led by $\bZ$. Then, $H$ admits a regular positive cone. 
\end{thm}

Moreover, we use the same criterion to show that there exists no left-order on a finitely generated acylindrically hyperbolic group such that the corresponding positive cone is represented by a quasi-geodesic regular language. Calegari showed in 2003 that no fundamental group of a hyperbolic $3$-manifold has a regular geodesic positive cone \cite{Calegari2003}. In 2017 Hermiller and \Sunik showed that no free products admits a regular positive cone \cite{HermillerSunic2017NoPC}. We have also seen a related result by Alonso, Antol\'in, Brum and Rivas in Section \ref{sec: fsa-reg-P-geom}: Lemma \ref{lem: regular-implies-coarsely-connected} says that positive cones of non-abelian free groups are not coarsely connected.

\begin{thm}[See Section \ref{sec: acyl-hyp-grp}]\label{thm: acyl}
	A quasi-geodesic positive cone language of a finitely generated acylindrically hyperbolic group cannot be regular.
\end{thm}

Our theorem fully generalises that of Calegari, and is a partial generalisation of Hermiller and \Sunik's result as acylindrically hyperbolic groups contain free products, but our result adds a quasigeodesic restriction on the statement about the positive cone language. In this light, the result Hermiller and \v{S}uni\'{c} and Theorem \ref{thm: acyl} can be interpreted as weaker analogous statements on more general versions of free groups and hyperbolic groups respectively. 

Since it is known that the lower bound of being 1-counter (the lowest complexity for a context-free language) is attained for some orders on free groups by a 2013 result of \Sunik \cite{Sunic2013}, our bound is the best possible within the Chomsky hierarchy. In 2006, Farb posed the question of whether mapping class groups of closed, oriented surfaces of genus greater or equal to two, with either zero or one puncture, have a finite index subgroup which is left-orderable \cite[5, Problem 6.3]{Farb2006}. Our theorem makes partial progress on Farb’s question by finding a lower bound on the formal language complexity for positive cones of acylindrically hyperbolic groups which are represented by languages containing only quasi-geodesic words. (Note that finite index subgroups of acylindrically hyperbolic groups are acylindrically hyperbolic \cite{Osin2016}.)
  
 The findings shown in this chapter are a longer form exposition of the results in \cite[Section 3,4 and 6]{Su2020} and \cite[Section 4.1]{AntolinRivasSu2021}. However, Proposition \ref{prop: lang-conv-easy} is an improvement on the work that was done to arrive to Theorem 1.1 of the paper (Theorem \ref{thm: closure-fi-reg} here) in the sense that we have simplified the construction of the resulting regular positive cone languages. The work was improved when writing a more complete exposition of the proof for this thesis led to this simplification. A more complete exposition of the original work that led to this result can be found in Chapter \ref{chap: old-lang-convex} of the Appendix. 
 
 We invite the reader to consult our companion chapter, Chapter \ref{chap: hyperbolic}, on hyperbolic groups (which acylindrically hyperbolic groups generalise) as needed. 
 
 \section{Motivation}

Recall the Klein bottle group $K_2 = \langle bab = a\rangle$, its positive cone $P = \langle a, b \rangle^+$ and its subgroup $\bZ^2 = \langle a^2, b \rangle$ of index two seen in Example \ref{ex: LO-K2} of Chapter \ref{chap: LO}. The positive cone $P$ belongs to a particularly simple class of regular positive cones which are given by positive cones that are finitely generated semigroup. If $P = \langle x_1 \dots x_n \rangle^+$, then a positive cone language is given by the regular expression $$L = (x_1 | \dots | x_n)^+.$$

As seen in Example \ref{ex: LO-Zsq-no-isolated-point},
$\bZ^2$ has uncountably many positive cones, none of which are finitely generated. Since $\bZ \leq \bZ^2 \leq K_2$ and both $\bZ$ and $K_2$ have finitely generated positive cones whereas $\bZ^2$ has none, this leads us to conclude that finitely generated positive cones are not stable in the sense that they are neither inherited by finite index subgroups nor extensions. 

A natural question is: can we generalise the finitely generated positive cone property such that its generalisation is inherited by both finite index subgroups and extensions? Given the focus of this thesis, the answer is of course yes, using formal language closure properties. 

We have already seen in Example \ref{ex: prelim-K2} that by writing the elements of $K_2$ and $\bZ^2$ lexicographically as $a^m b^n$, our finitely generated positive cone $P$ can be written as 
$$P = \{a^m b^n \mid m > 0 \text{ or } m = 0, n> 0\},$$
which can be represented by the regular expression 
$$L = a^+(b \mid b\inv)^* \cup b^+.$$
Restricting the language to elements of $\bZ^2$, we get 
$$L' = (a^2)^+(b \mid b\inv)^* \cup b^+.$$

This indicates that unlike the property of being finitely generated, regular positive cones might be able to be passed down to finite-index subgroups. We will prove that this is indeed the case, and offer a construction to obtain such regular positive cones generalising from the example above. 

\subsection{Why a naive approach fails}
We briefly pause here to make a contrast between working at the level of group elements and at the level of languages representing group elements. Suppose that a positive cone $P$ of a group $G$ has a regular language $L$ over an alphabet $X$, and that a subgroup $H$ can be represented by some language $M$ over $X$. Why isn't the an inherited language for the positive cone $P \cap H$ simply $L \cap M$? We can easily disprove this notion with the help of our Klein bottle group and $\bZ^2$ example. 
 
\begin{ex}\label{ex: Klein-Zsq-reg-lang-nonex}
Suppose we have $K_2$ with positive cone $P = \langle a, b \rangle^+$ as in Example \ref{ex: LO-K2}, and $\bZ^2 = \langle a^2, b \rangle$ as a subgroup of index $2$ in $K_2$. Then, $P \cap K_2$ is a positive cone for $\bZ^2$. However, at the level of languages if we realize both the positive cone $P$ as a regular language $L_P = (a|b)^+$, and $\bZ^2$ as a regular language $L_{\bZ^2} = (a^2 | a^{-2} | b | b\inv)^*$, then it is clear that $$\pi(L_P \cap L_{\bZ^2}) \not= \pi(L_P) \cap \pi(\bZ^2) = P \cap \bZ^2$$ since the element $a^2b\inv$ is not in $\pi(L_P \cap L_\bZ^2)$. This is illustrated in Figure \ref{fig: lang-conv-obv-does-not-work}. 

Note that this is consistent with what we have said in the section above, as if $\langle a^2, b\rangle^+$ were a positive cone for $\mathbb{Z}^2$, then $\bZ^2$ would have finitely generated positive cones. 

	\begin{figure}[h]{\includegraphics{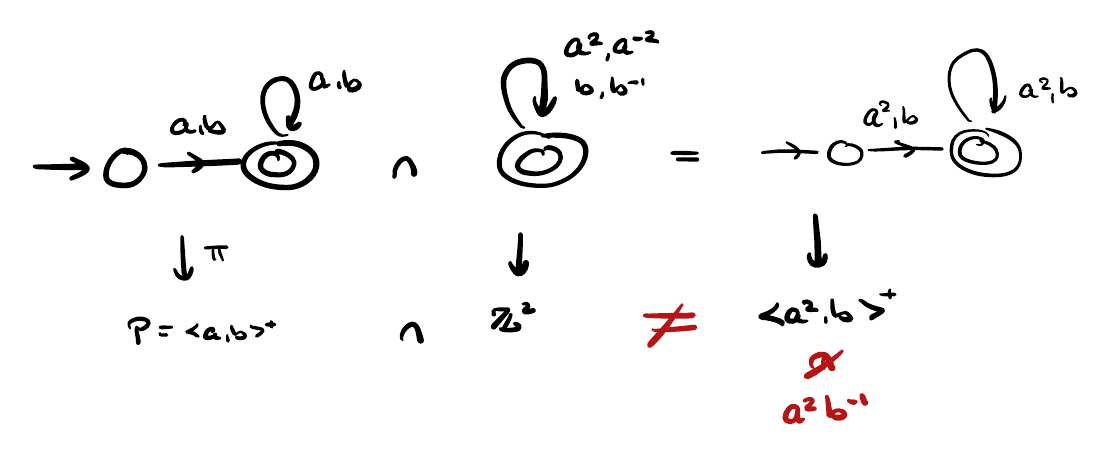}}
	\caption{
	Top: the finite state automata accepting the regular languages $L_P = (a|b)^+$, $L_{\bZ^2} (a^2 | a^{-2} | b | b^{-1})^*$ and their intersection $(a^2 | b)^+$. 
	
	Bottom: the image of these languages under evaluation map $\pi$. It is clear that intersection is not distributive with $\pi$ as $a^2b\inv \in P$, but there is no word $w \in (a^2 | b)^+$ such that $\pi(w) = a^2 b\inv$. 
	}
	\label{fig: lang-conv-obv-does-not-work}
	\end{figure}
\end{ex}

In other words, we are interested in a property that will allow a subgroup $H$ to inherit a positive cone at the language level, which behaves differently than at the set level. We will define a property that is key to this construction in the next section.

 \section{Language-convexity}\label{sec: lang-conv}
We refer to Chapter \ref{chap: hyperbolic} for a review on coarseness. Here, we recall a notion of coarse convexity, called \emph{quasiconvexity} and apply it to languages. 

\begin{defn}[Quasiconvexity]\label{defn: quasiconvex}
	A subspace $C$ of a geodesic metric space $S$ is said to be \emph{quasiconvex} if there exists a constant $R > 0$ such that for all $x, y \in C$, each geodesic joining $x$ to $y$ is contained in an $R$-neighbourhood of $C$. 
\end{defn}

In other words, a space being quasiconvex means that the space fails to be convex by a fixed known quantity $R$. We are interested in studying when $S = \Gamma$ is a Cayley graph, when $C = H \leq G$ i.e. the collection of vertices in $\Gamma$ formed by a subgroup $H$, which we may refer to as a \emph{quasiconvex subgroup}. Quasiconvex subgroups have many nice properties (see \cite[Chapter III]{BridsonHaefliger1999}). Restricted to groups and Cayley graphs, we can adapt the definitions of geodesic and quasiconvexity so that it lends itself better to our desired generalisation. 

\begin{defn}[Quasiconvexity in the Cayley graph]\label{defn: lang-conv-alt-quasiconvex}
Let $L$ the language of all geodesics of $\Gamma$ with alphabet $X$. A subset $H \subseteq G$ is \emph{quasiconvex} if there exists an $R \geq 0$ such that for each $w \in L$ with $\bar w \in H$, the induced path $p_w$ lies within distance $R$ from $H$ in the Cayley graph.
\end{defn}

Notice that in our first definition of quasiconvexity (Definition \ref{defn: quasiconvex}), we discuss metric spaces and subspaces. In the case that $\Gamma$ is a Cayley graph and $C = H$, the two definitions coincide. To generalise on the quasiconvex property, we open $L$ in Definition 
\ref{defn: lang-conv-alt-quasiconvex} to be any language as opposed to the language of geodesics.

\begin{defn}[Language convexity]\label{lconvex}\label{defn: lang-convex}
Let $L$ be a language over $X$ such that $\pi(L) \subseteq G$. A subset $H \subseteq G$ is \emph{language-convex with respect to $L$} or \emph{$L$-convex} if there exists an $R \geq 0$ such that for each $w \in L$ with $\pi(w) \in H$, the induced path $p_w$ lies within distance $R$ from $H$ in the Cayley graph. That is, $d_\Gamma(\bar w_i, H) \leq R$ for each $i \in \{1, \dots, |w|_X\}$. \cite{Su2020}
\end{defn}

Formally, we state the problem as follows. Given a left-orderable group $G$ with positive cone $P$, an $L$-convex subgroup $H \leq G$ and a regular language $L$ such that $\pi(L) = P$, construct a regular language $L_H$ such that $\pi(L_H) = P \cap H$. 

The $L$-convexity property allows us to construct the following finite state automaton which we will use to restrict the words in $L$ to the ones evaluating to $H$. The key is that the convexity parameter $R$ allows us to keep only a finite number of states in memory.

\begin{prop}\label{prop: lang-conv-easy}
Let $G$ be a group, and $\pi: X^* \to G$ be an evaluation map. Let $H$ be a subgroup and $R > 0$. Let $\bA = (S, X, \tau, A, s_0)$ be a finite state automaton with the states $S$ given by right cosets of $H$ with representatives no longer than $R$ and a failstate $\rho$. That is 
$$S = \{H\bar v \mid v \in X^*, |v|_X \leq R \} \cup \{\rho\},$$ alphabet $X$, transition function 
$$\tau(Hg, x) = 
\begin{cases}	
	& Hg\bar x, \quad Hg\bar x =_G H\bar v, \quad \exists v \in X^*, |v|_X \leq R, \\
	& \rho, \quad \text{otherwise}
\end{cases}$$
accept states $A = \{H\}$ and start state $s_0 = H$, as (vaguely) illustrated in Figure \ref{fig: fsa-lang-conv}.

	\begin{figure}[h]{\includegraphics{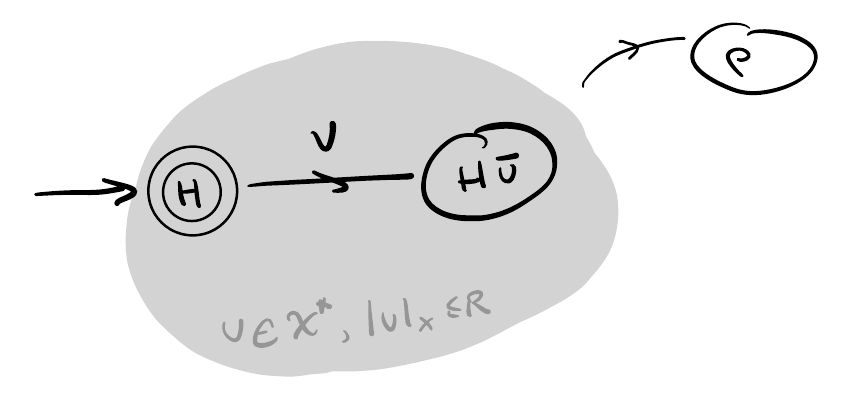}}
	\caption{
	A finite state automaton with states corresponding to the cosets of $H$, up to representatives with length $\leq R$ (in grey). If an input word is such that it can only be in a coset with geodesic representative $\geq R$, then the word goes to the fail state $\rho$. This permits us to remember all the words which are contained in a ball of radius $R$ from $H$. The finite state automaton only accepts words which represent elements of $H$. 
	}
	\label{fig: fsa-lang-conv}
	\end{figure}

Then $\bA$ accepts the language $$L_H = \{w \in X^* \mid \bar w \in H \text{ and } d_\Gamma(\bar w_i, H) \leq R \quad \forall i \in \{1, \dots, |w|_X\} \}.$$
\end{prop}
\begin{proof}
	We first claim by induction on the length of $w$ that $$\tau(s_0, w) = H\bar w$$ if and only if $d(\bar w_i, H) \leq R \quad \forall i \in \{1, \dots, |w|_X\}$, otherwise $$\tau(s_0, w) = \rho.$$
	
	If $|w| = 0$, then $\tau(s_0, \varepsilon) = H = H \bar \varepsilon$, and $\bar w \in H$, so $d(\bar w, H) = 0$.  Assuming the hypothesis is true for $|w| = n$, let $x \in X$. 
	
	There are two cases depending on the value of $\tau(s_0, wx)$. For the first case, start by supposing that $\tau(s_0, wx) = H\bar w \bar x$. Then, $H \bar w \bar x = H\bar v$ for some $v \in X^*, |v|_X \leq R$, and $\exists h \in H$ such that $\bar w \bar x = h\bar v \iff \bar w \bar x \bar v\inv = h$. Then, $d(\bar w \bar x, H) \leq d(\bar w \bar x, \bar w \bar x \bar v \inv) \leq |v|_X \leq R$. Now suppose that $d(\bar w \bar x, H) \leq R$. Then, there exists $v \in X^*$ with $|v|_X \leq R$ such that $\bar w \bar x \bar v = h \iff \bar w \bar x = h \bar v\inv$ for some $h \in H$. This implies that $H \bar w \bar x = H \bar v\inv$. Moreover, since $\tau(s_0,wx) = \tau(\tau(s_0, w), x) \not= \rho$, we have that $\tau(s_0, w) = H\bar w$. But this is true if and only if $d(w_j, H) \leq R$ for $j \in \{1, \dots, |w|_X\}$ from the hypothesis, completing the proof of the first case. 
	
	For the second case, suppose that $\tau(s_0, wx) = \rho$. Suppose first that $\tau(s_0, w) = H\bar w$. Then it must mean that there is no $v \in X^*$ with $|v|_X \leq R$ such that $H\bar w \bar x = H\bar v$. Therefore, $d(\bar w \bar x, H) > R$ for otherwise there would exists such a $v$ as shown previously. In the case that $\tau(s_0,w) = \rho$, by hypothesis this means that there exists $i \in \{1,\dots, |w|\}$ such that $d(w_i, H) > R$. This finishes the proof of the claim.
	
	Putting it all together, since the only accept state is the coset corresponding to $H$, we have that $\tau(s_0, w) \in A \iff \tau(s_0, w) = H \iff \bar w \in H$ and $d(w_i, H) \leq R \quad \forall i \in \{1, \dots, |w|\}$ by the previous claim. This finishes the proof.
	\end{proof}
	
	\begin{rmk}
		For a group $G$ with generating set $X$ with subgroup $H$, we remark the similarity between the automaton described in the proof of Proposition \ref{prop: lang-conv-easy} as a Schreier graph $\Gamma(H\backslash G, X)$ as defined in Chapter \ref{chap: RS} as the language-convexity parameter $R \to \infty$. 
		\end{rmk}
		
\begin{cor}\label{cor: lang-convex-closure}
	Let $G$ be a group with evaluation map $\pi: X^* \to G$, and $L$ be a regular language such that $\pi(L) = P$. Let $H$ be an $L$-convex subgroup in $G$ with parameter $R$ and $L_H$ be the language described in the statement of Proposition \ref{prop: lang-conv-easy}. Then, $L \cap L_H$ is a regular language evaluating to $P \cap H$ in $G$. 
\end{cor}

\begin{proof}
	Let $w \in L \cap L_H$. It is clear that $\bar w \in P \cap H$, showing that $\pi(L \cap L_H) \subseteq P \cap H$. 
	
	Now, let $h \in P \cap H$. Then, there exists $w_h \in L$ such that $\bar w_h = h$. Moreover, since $H$ is $L$-convex, $d(\bar w_j, H) \leq R$ for each $j \in \{1, \dots, |w|\}$, fulfilling the requirement that $w \in L_H$. This shows that $\pi(L \cap L_H) = P \cap H$. 
\end{proof}

We now have a good candidate for a regular positive cone language for an $L$-convex subgroup $H$. The last step is to convert the alphabet of the language $L \cap L_H$ as in the statement of Corollary \ref{cor: lang-convex-closure} to an alphabet $Y$ such that $Y$ generates $H$.

\begin{lem}\label{lem: lang-conv-fg}%
	Let $G, H, L, R$ and $L_H$ be as in the statement of Corollary \ref{cor: lang-convex-closure}. Then, there exists a map $\tau: (L \cap L_H) \to Y^*$, where  
	$$Y = \{y \in L_H : |y|_X \leq 2R + 1\},$$
	such that, for every $w \in L \cap L_H$, with $w = x_1 \dots x_n$, $w =_G \tau(w)$, where 
	$$\tau(w) = y_1 \dots y_n, \qquad y_i = u_{i-1}\inv x_i u_i,$$
	with $y_i \in Y$, $u_i \in X^*$ such that $|u_i|_X \leq R$ for $1 \leq i \leq n$ and $u_0, u_n$ are the empty word.
\end{lem}
\begin{proof}
Let $w \in L \cap L_H$, and suppose $w = x_1,\dots, x_n$, and write $w_i = x_1 \dots x_i$ for $1 \leq i \leq n$. By $L$-convexity of $H$, we may choose for $i = 1, \dots n$, a word $u_i$ of length at most $R$ so that $\pi( w_i) \pi( u_i) = h_i \in H$. Define $$y_i := u_{i-1}\inv x_i u_i.$$ By construction, $y_i \in L_H$. Indeed, $\pi(y_i) = \pi(w_{i-1})\inv\pi(w_i) = h_{i-1}\inv h_i$ and $d(\pi(w_i), H) \leq R$ for every $1 \leq i \leq n$. Moreover, for $1 \leq i \leq n$, $|y_i|_X \leq 2R + 1$, meaning that $y_i \in Y$. This is illustrated in Figure \ref{fig: fishing}.
	
	\begin{figure}[h]{\includegraphics{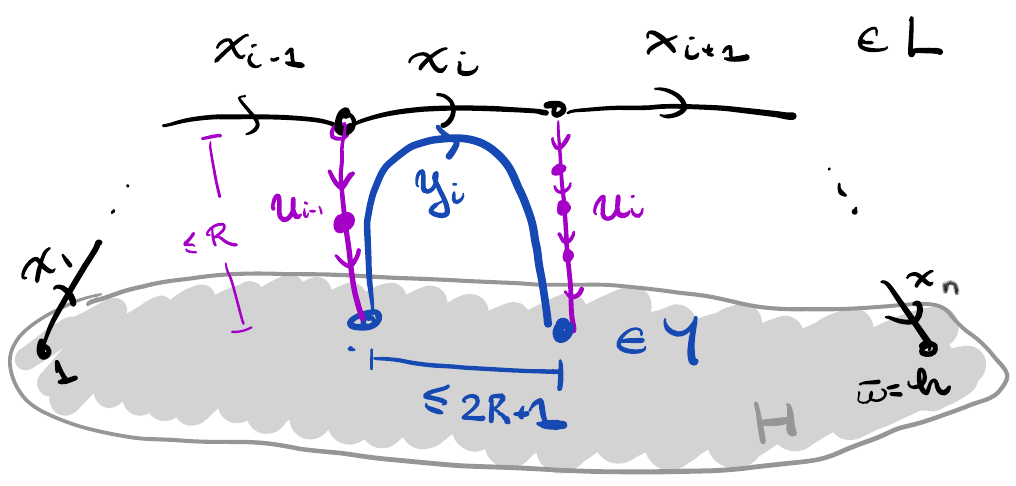}}
	\caption{We illustrate a word $w = x_1 \dots w_n$ in the Cayley graph creating a path $p_w$ going from the identity element to $\pi(w) = h$ (in black). At each prefix $w_i$, there exists by $L$-convexity of $H$ a word $u_i$ of length $\leq R$ with respect to $X$ such that $\pi(w_i u_i) = h_i \in H$ (in purple). From this we define words $y_i = u_{i-1} x_i u_i$ which represents generators of $H$ (in blue).
	}
	\label{fig: fishing}
	\end{figure}
	
	Define $\tau(w) := y_1 \dots y_n$. Then, by construction $\pi(w) = \pi(\tau(w))$. 
\end{proof}

Putting everything together, we get the following theorem.

\begin{thm}\label{thm: lang-convex-closure}
Let $G$ be finitely generated by $X$. Let $H\leqslant G$. 
Suppose that $L\subseteq X^*$ is a regular language that evaluates to a positive cone $P$ of $H$ and $H$ is $L$-convex with parameter $R$. There exists 
\begin{enumerate}
	\item a finite subset $Y$ of $H$ such that $\langle Y \rangle = H$,
	\item a regular language $L'\subseteq X^*$ that evaluates to $P\cap H$ equipped with a map $\phi\colon Y \to X^*$ that extends to a map $\phi:Y^*\to X^*$ such that for every $h\in H\cap P$ there is $w_h\in L'$ such that  $w\in  \mathrm{Im}(\phi)$.
\end{enumerate}

In particular, $\phi^{-1}(L')\subseteq Y^*$ is a regular positive cone language for $H$.
\end{thm}
\begin{proof}
Let $L_H$ be as in Proposition \ref{prop: lang-conv-easy}. Let $\tau$ and $Y$ be as in Lemma \ref{lem: lang-conv-fg}. 

To find a suitable generating set for $H$, observe that since $H = \langle P \cap H \rangle = \langle \pi(\tau(L \cap L_H)) \rangle \subseteq \langle \pi(Y^*) \rangle = \langle Y \rangle$, and $\pi(Y) \subseteq H$, we have that $H = \langle Y \rangle$. 

To construct our desired regular language $L'$, let $\mathbb{A}$ be a finite state automaton over the alphabet $X$ that accepts $L$. We view $\mathbb{A}$ as a finite $X$-labelled directed graph. 
We construct now a new automaton $\widetilde{\mathbb{A}}$ by taking each vertex of $\mathbb{A}$ and adding cycles reading trivial words in $X^*$ of length at most $2R$.  None of the vertices on these cycles are accepting states.

We let $L'$ be the language accepted by  $\widetilde{\mathbb{A}}$. We claim that $\pi(L')=\pi(L)$. Indeed,  clearly $L\subseteq L'$ so $\pi(L)\subseteq \pi(L')$. On the other hand, let $w\in L'$. Then $w=l_0x_1l_1x_2l_2\dots x_nl_n$, where $x_1x_2\dots x_n$ is the label of some path in $\mathbb{A}$ starting at the start state and finishing at an accepting state of $\mathbb{A}$ and each $l_i$ labels the path of one of the loops in  $\widetilde{\mathbb{A}}$ added to $\mathbb{A}$. We see that by construction $\pi(w)=\pi(x_1\dots x_n)$ and thus $\pi(L')\subseteq \pi(L)$.

Take $\phi: Y\to X^*$ as monoid homomorphism sending $Y$ to the words over $X$ in $L_H$ they represent. Let $h\in P \cap H = \pi(L \cap L_H)$. We know from Lemma \ref{lem: lang-conv-fg} that there exists $w_h = (x_1 u_1) (u_1^{-1}x_2 u_2) \dots (u_{n-1}^{-1} x_n) \in \tau(L \cap L_H)$ such that $\pi(w_h) = h$ and $|u_i| \leq R$ for every $1 \leq i \leq n$. That is, $w_h$ is a word in $Y^*$ for which
  $$w_h=x_1 (u_1 u_1^{-1}) x_2 (u_2 u_2^{-1}) x_3 (u_3 u_3^{-1}) \dots (u_{n-1} u_{n-1}^{-1}) x_n \in L'.$$ 

Thus, for every $h \in H \cap P$ there exists a word in $w_h \in \phi\inv(L') \subseteq Y^*$ such that $\pi(\phi(w_h)) = h$. Since $\phi\inv$ is an inverse homomorphism and $L'$ is a regular language, we have that $\phi\inv(L')$ satisfies the definition of a regular PCL for $H$. 
\end{proof}

\begin{rmk}
	Observe how $L \cap L_H$ in Proposition \ref{prop: lang-conv-easy} may be ``prettier'' in practice than $\phi\inv(L')$ in Theorem \ref{thm: lang-convex-closure} due to the change in generating set.  This is most concretely observed in Example \ref{ex: lang-conv-K2} where, if the generating set $Y$ could not be simplified, the PCL for $\bZ^2$ looks more complicated compared to the one indicated by $L \cap L_H$. An alternative proof of Theorem \ref{thm: lang-conv-lex-clos} is given in Chapter \ref{chap: old-lang-convex} where the starting step is to convert into the $Y$-generating set, leading to an opaque (in comparison) positive language for $L$-convex subgroups, even for seemingly simple examples such that the one for $G = K_2$ and $H = \bZ^2$. 
	
	An alternative definition of a positive cone language which skips the limitation of having to change the generating set as in Definition \ref{def: PCL} is to define, for a subgroup $H$ of a finitely generated group $G = \langle X \mid R \rangle$, a positive cone language $L \subseteq X^*$ for $H$ if $\pi(L)$ is a positive cone for $H$. 
\end{rmk}
	
\subsection{Application to finite index subgroups}\label{sec: lang-conv-fi}

Theorem \ref{thm: lang-convex-closure} proves Theorem \ref{thm: closure-fi-reg} because finite index subgroups inherit positive cone regularity from their overgroups, since finite index subgroups are always $L$-convex over any language $L$. 

\begin{lem}\label{lem: lang-conv-finite-index}
Let $L \subseteq X^*$ be a language and $H$ an $L$-convex subgroup of a finitely generated group $G$. If $K$ is a finite index subgroup of $H$, then $K$ is also $L$-convex in $G$. In particular, if $G = H$ and $K$ is a finite-index subgroup of $G$, then $K$ is $L$-convex for any language $L$. \cite{Su2020}
\end{lem}
\begin{proof}
We will denote by $N_R(K)$ a neighbourhood of radius $M$ around $K$.  

Let $R$ be the $L$-convexity constant of $H$. Let $C = \{h_1, \dots, h_n\}$ be a list of coset representatives of $K$ in $H$. Let $R' = \max_{h_i \in C} |h_i|$.  If $h \in H$, then there exists an $i \in \{1, \dots, n\}$ such that $h = kh_i$. Then $d(h,k) = |h_i\inv k\inv k| = |h_i| \leq R'$. This shows that $H \subseteq N_{R'}(K)$.  

Then for all $w \in L$ with $\bar w \in K$, we have that $p_w \subseteq N_R(H)$ by $L$-convexity of $H$. Moreover, $N_R(H) \subseteq N_R(N_{R'}(K)) = N_{R + R'}(K)$. Therefore $K$ is $L$-convex with $L$-convexity parameter $(R + R')$. 

Figure \ref{fig: finite-index} illustrates the special case when $G=H$. 

\begin{figure}[h]{\includegraphics{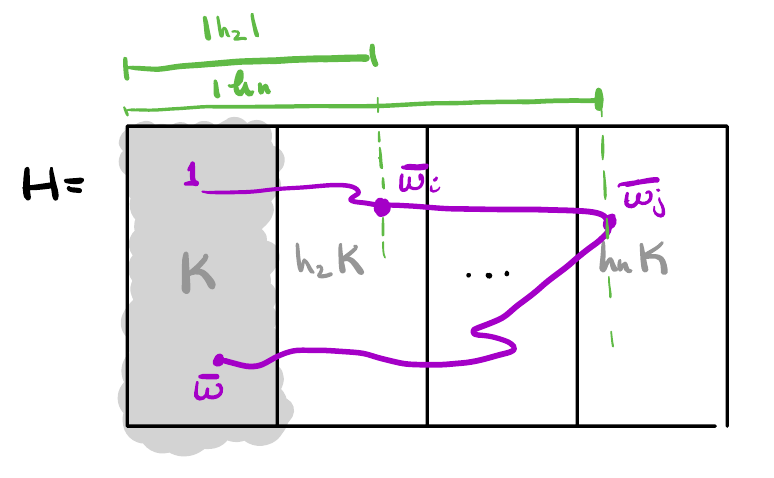}}
\caption{For the $G = H$ case, we represent $H$ as a rectangle, with $K$ partitioning $H$ into $n$ cosets represented by $K, h_2K, \dots, h_n K$ (in grey). The word $w$ starts at the identity and ends in the coset $K$ (in purple). Along the way, the prefixes $w_1, \dots w_n$ represent elements in $H$ and are as such only $\max_{h_i \in C} |h_i|$ far from $K$ (in green).}
\label{fig: finite-index}
\end{figure}
\end{proof}

\begin{rmk}\label{rmk: lang-conv-easy-finite-index}
Note that for $H$ a finite index subgroup, the automaton of Proposition \ref{prop: lang-conv-easy} can be simplified as follows. The states are given by $\{H=Hg_1, Hg_2 \dots, Hg_n\}$ where $\{g_1, \dots, g_n\}$ are coset representatives. There is no fail state because every element of $G$ falls within one of the cosets, and thus, there is a finite bound away from $H$ which holds for every word. The transition function becomes simply
$$\tau(Hg, x) = Hg\bar x.$$

The automaton accepts the language $$L_H = \{w \in X^* \mid \bar w \in H \}$$ since every word evaluates to an element at a finite distance from $H$. 
\end{rmk}

\begin{ex}\label{ex: lang-conv-K2}
	 Let us go back to the $G = K_2 = \langle a, b \mid a\inv ba = b\inv \rangle, P = \langle a, b \rangle^+, H = \langle a^2,b \rangle \cong \bZ^2$ case. Then, if $\{1_G, a\}$ is our set of coset representatives, the automaton accepting $$L_H = \{w \in X^* \mid \#a's \text{ is even}\}$$ is given by Figure \ref{fig: fsa-finite-index}. 
	\begin{figure}[h]{\includegraphics{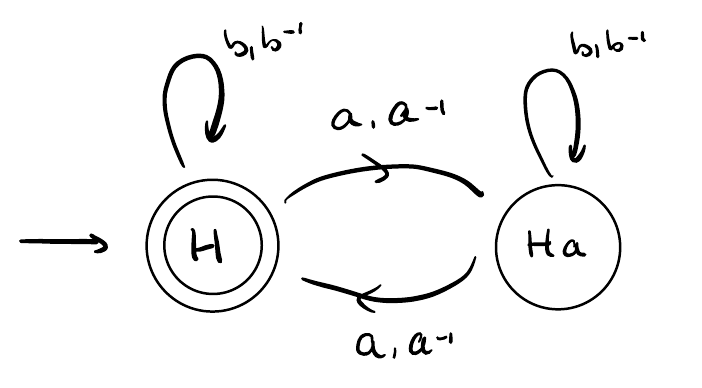}}
	\caption{
	The finite state automaton of Remark \ref{rmk: lang-conv-easy-finite-index} when $G = K_2$ and $H = \bZ^2$ as in Example \ref{ex: lang-conv-K2}. This is reminiscent of the automaton accepting only binary strings with an odd number of zeroes of Example \ref{ex: fsa-even-odd}, our first non-trivial example of a finite state automaton.}
	\label{fig: fsa-finite-index}
	\end{figure}
	 By varying the starting positive cone language $L$, we get different positive cone languages for $H$. 
	 
	 Suppose $L_1 = (a \mid b)^+$. Then $$L_1 \cap L_H = \{w \in L_1 \mid \#a's \text{ is even}\}.$$ 	 
	 
	 Suppose instead that we write the elements of $K_2$ lexicographically, such that the positive cone language is given by 
	 $$L_2 = a^+(b|b\inv)^* \cup b^+.$$ 
	 Then $$L_2 \cap L_H = \{w \in L_2 \mid \#a's \text{ is even}\} = (a^2)^+ (b|b\inv)^* \cup b^+,$$
	 which is the language we had at the beginning. 
	 
	 If one were to write these languages in terms of generators of $H$, the generators would be in terms of words evaluating to $H$ of length $\leq 3$ in $X$. Listing all the non-trivial generators yields
	 $$Y = \{b,b\inv, a^2, a^{-2}, aba, aba\inv, ab\inv a, ab\inv a\inv, a\inv ba, a\inv b a\inv, a\inv b\inv a\inv \}.$$
	 
	 By using the relation on $G = K_2$ we may reduce this generating set to
	 	$$Y' = \{b,b\inv, a^2, a^{-2}\}.$$

\end{ex}

\begin{rmk}
	The major weakness of Proposition \ref{prop: lang-conv-easy} from a computational perspective is that the transition function heavily relies on the evaluation map to decide whether a word belongs in a certain coset. In the example above, it was easy as this simply involves keeping track of the parity of the number of $a$'s in a word. However, depending on the group, as determining the coset becomes more complex, this could create a bottleneck in terms of computation.
\end{rmk}

\subsection{Application to order-convex subgroups}\label{sec: lang-conv-order-convex}

Recall the discussion on relative orders and convexity in Section \ref{sec: rel-ord}, and in particular Lemma \ref{lem: ord-convex-rel-cone}. We can conclude that if $H$ is $\prec$-convex, then for every  $g_1\prec g_2$,  $g_1,g_2\in G$ one has that  $g_1h_1\preceq g_2 h_2$ for all $h_1,h_2\in H$. Moreover, if $H$ is $\prec$-convex, then  $\prec$ induces a left-order on the coset space $G/H$. 

The proposition below says that for finitely generated subgroups of right semi-direct products with $\bZ$ where the left-order is lexicographic with leading factor $\bZ$ are $L$-convex by $\prec$-convexity, where $L$ gives represents the positive cone language that is lexicographic. It is a more precise restatement of Theorem \ref{thm: lang-conv-lex-clos}.

\begin{prop}\label{prop: convex implies L-convex}
Let $G=H \rtimes \bZ$ be finitely generated by $(X,\pi)$. 
Let $\prec$ be a lexicographical  $\Reg$-left-order on $G$ led by $\bZ$, with $L\subseteq X^*$, a regular positive cone language.
If $H$ is finitely generated, then $H$ is language-convex with respect to $L$.

In particular, the restriction of $\prec$ to $H$ is a $\Reg$-left-order.
\end{prop}

Given word $w\in X^*$, and $x\in X$, we use $\sharp_x (w)$ to denote the number of times the letter $x$ appears in the word $w$. Before proving the Proposition \ref{prop: convex implies L-convex} we need to describe all regular languages over the alphabet $\{t,t^{-1}\}$ mapping onto a  positive cone of $\bZ$.

Given a word $w=x_1\dots x_n\in \{t,t^{-1}\}^*$, with $x_i\in \{t,t^{-1}\}$ we define a function $f_w\colon \{0,1,\dots, n\}\to \bZ$ by $f_w(i)=\sharp_t(x_1\dots x_i)- \sharp_{t^{-1}}(x_1\dots x_i)$.

\begin{lem}\label{lem: coarsely non-decreasing}
Let $L \subseteq \{w\in \{t,t^{-1}\}^*\mid \sharp_t(w)- \sharp_{t^{-1}}(w)\geq 0\}$ be a regular language.
Then every $f_w$ is coarsely non-decreasing in the following sense: 
there is a constant $K\geq 0$ such that for all $w\in L$ and for all $i, j\in \{0,1\dots, \ell(w)\}$, if $j>i$ then  $f_w(j)>f_w(i)-K$.
\end{lem}
\begin{proof}
Let $\bM = (S, X=\{t,t^{-1}\}, \delta, s_0,\cA)$ be an automaton accepting $L$. 
We will consider $\bM$ to be without $\epsilon$-moves and the image of transition function $\delta$ being singletons (that is, we make the FSA deterministic). Moreover, we will think of this automaton as a directed graph.
Every edge has a label from $X$, and from every vertex there is at most one outgoing edge with label $x\in X$.
Hence, every $w\in L$ labels a unique a path in $\bM$ starting at the initial state $s_0$.
 
Every $w \in L$ can be decomposed  as $w = xyz$, where $y$ is a (possibly trivial) loop in $\bM$. 
Then, $w \in L$ implies that $xy^nz \in L$ for any $n\in \{0,1,2,\dots \}$. 
In other words, for each word $w$ accepted in the automaton for $L$, we may remove or insert words $y$ representing loops in $\bM$  and still get an accepted word $xy^nz$.

Let $w \in L$. Let $g(w) := \#_{t}(w) - \#_{t^{-1}}(w)$. Write $w = xyz$ where $y$ is a (possibly) trivial loop. 
Observe that $ g(y) \geq 0$, for otherwise $xy^iz \in L$ for any $i$, as we may pick $i$ large enough such that $g(xy^iz) < 0$, contradicting our assumption about $L$.
In other words, for any loop $y$ in $w=xyz$, 
\begin{equation}
\label{eq: y loop}
 g(xz)\leq g(xyz)
\end{equation}
and $xz \in L$ since we have only removed a loop $y$. 
 
Let $n$ be the number of states of $\bM$. 
For any subword $u$ of  $w \in L$, decompose $u = x_1 y_1 x_2 y_2 \dots x_{k-1} y_{k-1} x_{k}$, where the $y_i$ are loops and the length of the word $x_1x_2\dots x_k$ is minimal. 
We allow the subwords $x_i$ to be empty.
Viewing $u$ as a subpath of $w$ in $\bM$, we construct a new path whose label is $u' = x_1 \dots x_k$ which consists of removing all  loops from $u$.
 In particular, by the pigeonhole principle we have that $\ell(u') \leq n$ as otherwise the path associated to $u'$ would go through the same vertex in $\bM$ twice. Thus, we have a factorization of $u$ with a loop such that removing it produces a word shorter than $u'$. 
By \eqref{eq: y loop} we have that $g(u)\geq g(u')$, and as $\ell(u')\leq n$, this implies that 
\begin{equation}
\label{eq: g(u)}
g(u) \geq -n \text{ for all subword $u$ of $w\in L$,} 
\end{equation}
 since each transition can only contribute one $t$ or $t^{-1}$ and $\ell(u')\leq n$.

Let $1\leq i<j\leq \ell(w)$. 
We need to show that there is a constant $K\geq 0$ such that $f_w(j)-f_w(i)\geq -K$.
But $f_w(j)-f_w(i)=g(u)$ for $u$ equal to the subword of $w$ consisting of taking the prefix of $w$ of length $j$ and removing from it the prefix of length $i$.
Now the result follows from \eqref{eq: g(u)} and taking $K=n$.
\end{proof}

Now we can prove Proposition \ref{prop: convex implies L-convex}.
\begin{proof}[Proof of Proposition \ref{prop: convex implies L-convex}]
Let $(X,\pi_H)$ be a generating set for $H$ and $\{t,t^{-1}\}$ a generating set for $\bZ$.
We combine them to make  $(X'=X\sqcup \{t,t^{-1}\}, \pi)$ a generating set for $G$ with evaluation map $\pi$.
Let $L\subseteq (X')^*$ be a regular language such that $\pi(L)$ is a lexicographic positive cone with the quotient being the leading factor.

Let $\phi \colon (X')^*\to \{t,t^{-1}\}^*$ consisting on deleting the letters of $X$.
This is monoid morphism, and hence $\phi(L)$ is regular.
Since $L$ is the language of a lexicographic order, $\phi(L)$ is contained in  $\{w\in \{t,t^{-1}\}^*\mid \sharp_t(w)- \sharp_{t^{-1}}(w)\geq 0\}$.
By Lemma \ref{lem: coarsely non-decreasing}, we get that there is a $K\geq 0$ such that  $f_{\phi(w)}(j)>f_{\phi(w)}(i)-K$ for all $w\in L$ and for all $j>i$.

To see that $H$ is language-convex with respect to $L$, let $w\in L$, with $\pi(w)\in H$.
Then, we get that $\sharp_t \phi(w)- \sharp_{t^{-1}}\phi(w)=0$.
Then, $0=f_{\phi(w)}(\ell(w))\geq \max_{i} f_{\phi(w)}(i)- K$, so $\max_{i} f_{\phi(w)}(i) \leq K$.
Also, $\min_{i} f_{\phi(w)}(i)> f_{\phi(w)}(0)-K=-K$.
It follows that for every prefix $u$ of $w$, $|\sharp_t \phi(u)- \sharp_{t^{-1}}\phi(u)|\leq K$.
Therefore $$\dist_G(\pi(u), \pi( ut^{-\sharp_t \phi(u)+ \sharp_{t^{-1}}\phi(u)}))\leq K.$$ Observe that $\pi( ut^{-\sharp_t \phi(u)+ \sharp_{t^{-1}}\phi(u)})\in H$ since the exponent ${-\sharp_t \phi(u)+ \sharp_{t^{-1}}\phi(u)}$ cancels the $t$'s in $u$. This shows that $H$ is language-convex.

By Proposition \ref{prop: lang-conv-easy}, we get that the restriction of the left-order to $H$ is regular.
\end{proof}

\subsection{Application to acylindrically hyperbolic groups}
\label{acyl-hyp-grp}
\label{sec: acyl-hyp-grp}

Here we assume that the reader is familiar with hyperbolic groups and notions of quasi-geodicity, and invite them to read Chapter \ref{chap: hyperbolic} in case of the contrary. Roughly speaking, an acylindrically hyperbolic group generalises non-elementary hyperbolic groups.\sidenote{Non-elementary hyperbolic groups are hyperbolic groups which are not virtually cyclic, i.e. they are infinite and do not contain $\bZ$ as as a finite index subgroup} They do so by contracting away the non-hyperbolic behavior into subgroups that are known as \emph{hyperbolically embedded}, as the global behavior with contraction is in some sense well-behaved. This technique builds on the notion of relatively hyperbolic groups\sidenote{A class of group more general than hyperbolic groups, which include the fundamental groups of complete noncompact hyperbolic manifolds of finite volume.} and generalises them as well. Some other examples of acylindrically hyperbolic groups are infinite mapping class group of closed, oriented surfaces, groups of outer automorphism of free groups (excluding $\bZ$), free products, and RAAGs which are not cyclic and are directly indecomposable. 

The definitions for acylindrically hyperbolic groups and hyperbolically embedded subgroups are rather technical and not too helpful without proper context\sidenote{(in my opinion)} so we omit them here. We suggest \cite{Osin2006} as a reference for relatively hyperbolic groups, and \cite{Osin2016} for a reference on acylindrically hyperbolic groups.\sidenote{Full disclosure: At the time of the thesis writing, I am not overly familiar with the two topics, so take what I say about it with a grain of salt.}

Let us state the theorems of Calegari and Hermiller and Sunic we will generalise formally. 

\begin{defn}[Geodesic and quasi-geodesic positive cone language]\label{def-qgpc} Let $(G,\pi)$ be finitely generated. Let $P$ be any positive cone for $G$. We say that $L$ is a \emph{geodesic positive cone language} (resp \emph{quasi-geodesic positive cone language} if it satisfies the following two conditions.

\begin{enumerate} 
	\item Under the evaluation map $\pi: X^* \to G$ we have that $\pi(L) = P$.
	\item Every $w \in L$ is a geodesic word (resp. there exists some constants $\lambda$ and $\epsilon$ with $\lambda \geq 1$ and $\epsilon \geq 0$ for which every word $w \in L$ is a $(\lambda,\epsilon)$-quasi-geodesic word). 
\end{enumerate}
\end{defn}

\begin{thm}
Let $M$ be a closed, compact, connected hyperbolic $3$-manifold, and $G = \pi(M)$. Then $G$ does not admit a regular geodesic positive cone. \cite{Calegari2003}.
\end{thm}
The proof of this result relies on the fact that every such hyperbolic $3$-manifold $M$ contains a quasigeodesically embedded copy of $F_2$ the free group on two elements.\sidenote{That is, if $G = \langle X \rangle$ and $F_2 \leq G$, then the words over $X$ representing elements of $F_2$ would be quasigeodesic words.} Calegari then shows that the dynamical action of $F_2$ by order-preserving homeomorphisms prevents any left-ordering on $G$ from admitting regular positive cones. 

The generality of Proposition \ref{prop: lang-conv-easy} can be used to fully generalise Calegari's result to acylindrically hyperbolic groups, as they contain the hyperbolic $3$-manifold groups.\sidenote{From Example \ref{ex: hyperbolic-manifold-groups}, these groups are $\delta$-hyperbolic for some $\delta > 0$. It was not clear to me whether Calegari's proof also works in the case that the manifold is not closed or compact, but in that case, the group would be relatively hyperbolic, and thus, also contained in acylindrically hyperbolic groups which are also a generalization of the class of non-elementary hyperbolic groups.} 

On the other hand, we have the following theorem by Hermiller and Sunic, whose proof relies on a similar use of the Pigeonhole Principle as that of the Pumping Lemma (see \ref{lem: fsa-pumping}). 

\begin{thm}\label{thm: herm-sunic}
Let $A, B$ be two non-trivial, finitely generated, left-orderable groups. Let $G = A * B$. Then $G$ does not admit a regular positive cone. \cite{HermillerSunic2017NoPC}
\end{thm}

In particular, this theorem states that $F_2$ cannot have a regular positive cone. It is possible to show the result of Calegari by eliminating the dynamical part, and instead using the abstraction of $L$-convexity paired with Theorem \ref{thm: herm-sunic}. Suppose that $[p,q]$ form a geodesic path in the Cayley graph $\Gamma$ from $1$ to $g \in F_2 \leq G$. The $(\lambda, \varepsilon)$-quasigeodesic embedding of $F_2$ in a hyperbolic $3$-manifold group $G$ means that there is a $(\lambda, \varepsilon)$-quasigeodesic path $\gamma$ from $1$ to $g$ which contained in the subgraph induced by $F_2$. By the Morse Lemma (see Chapter \ref{chap: hyperbolic}), this means that it is at most $R(\lambda, \epsilon)$ away from $[p,q]$ in $\Gamma$. 
Now, given a $(\lambda', \varepsilon')$-quasigoedesic language $L$, and $w \in L$ such that $\bar w \in F_2$. The word $w$ induces a path $\gamma'$ that is at most $R'(\lambda', \varepsilon')$ away from $[p,q]$ in $\Gamma$, again by the Morse Lemma. Thus, $\gamma'$ is contained in a ball of radius $R + R'$ from $F_2$, so $F_2$ is $L$-convex in $G$ with parameter $R + R'$. This is illustrated in Figure \ref{fig: lang-conv-Morse}
	 
	\begin{figure}[h]{\includegraphics{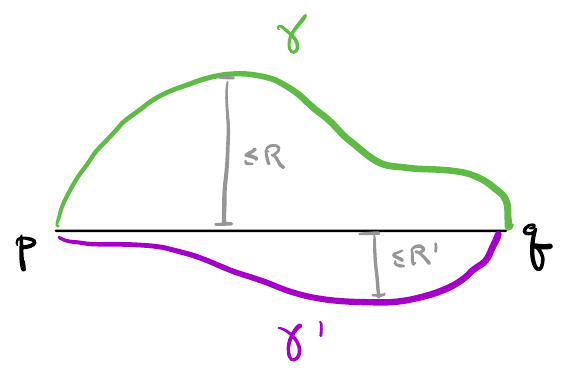}}
	\caption{
	A geodesic $[p,q]$ form a path between $1$ and $g \in F_2 \leq G$. The $(\lambda, \varepsilon)$ quasi-geodesic $\gamma$ is at most $R$ away from $[p,q]$ in Hausdorff distance (in green). Similarly, a $(\lambda', \varepsilon')$-quasigeodesic word in $L$ induces a path that is at most $R'$ away from $[p,q]$ (in purple). Therefore, every such $\gamma$ is at most $R + R'$ away from $\gamma'$. 
	}
	\label{fig: lang-conv-Morse}
	\end{figure}
	
This means that if $L$ was a regular positive cone language then the inherited positive cone of $F_2$ would also be regular, contradicting Theorem \ref{thm: herm-sunic}. Using this new $L$-convexity framework, we adapt Calegari's idea to generalise his result the much wider class of acylindrically hyperbolic groups as these groups have the right analogous properties to hyperbolic $3$-manifold groups. 

In other words, Theorem \ref{thm: acyl} says that if $G$ is a finitely generated, acylindrically hyperbolic group which admits a quasi-geodesic positive cone language $L$, then $L$ cannot be accepted by any finite state automaton. 

To prove the theorem, we first use the following lemma concerning the existence of a hyperbolically embedded subgroup (see \cite[Section 2.1]{DahmaniGuirardelOsin2011} for a rather long definition). It is not necessary to know the definition of hyperbolically embedded to follow the next results. The idea is that being hyperbolically embedded in an acylindrically hyperbolic group will be used here as being analogous to being quasigeodesically embedded in a hyperbolic $3$-manifold group. 

\begin{lem}\label{F_2-embed} If $G$ is a torsion-free acylindrically hyperbolic group, then there exists a hyperbolically embedded subgroup $H$ of $G$ that is isomorphic to $F_2$, the free group of two elements.
\end{lem}
\begin{proof}

Osin proved in \cite[Theorem 1.2]{Osin2016} that $G$ being acylindrically hyperbolic is equivalent to containing a proper infinite hyperbolically embedded subgroup. All we need for this proof is the result of Dahmani, Guirardel and Osin in \cite[Section 6.2]{DahmaniGuirardelOsin2011} which is dependent on the existence of a proper infinite hyperbolically embedded subgroup in $G$. The result states that if $G$ contains a proper infinite hyperbolically embedded subgroup, then for any $n \in \nat$ there exists a subgroup $H \leq G$ such that $H$ is hyperbolically embedded in $G$ and $H \cong F_n \times N$, where $F_n$ is a free group of rank $n$ and $N$ is the maximal finite normal subgroup of $G$. In particular, since $G$ is torsion-free, $N$ is trivial and there exists a hyperbolically embedded subgroup $H \leq G$ such that $H \cong F_2 \times \{1\} \cong F_2$.
\end{proof}

Next, we will need to analogue of the Morse Lemma for acylindrically hyperbolic groups, aptly named the Morse property. 

\begin{defn}[Morse property]
A subspace $\mathcal{Y}$ of a metric space $\mathcal{X}$ is said to be \emph{Morse} if for every $\lambda \geq 1$ and $\epsilon \geq 0$,  there exists a non-negative constant $R$ depending on $\lambda$ and $\epsilon$ with the property that all $(\lambda,\epsilon)$-quasi-geodesics in $\mathcal{X}$ whose endpoints are in $\mathcal{Y}$ are contained in the neighbourhood of radius $R$ around $\mathcal{Y}$. 
\end{defn}

This property will be used to show the following lemma. 

\begin{lem}\label{H-L} If $H$ is a hyperbolically embedded subgroup of an acylindrically hyperbolic group $G$, then $H$ is language-convex with respect to every quasi-geodesic language $L$.
\end{lem} 
\begin{proof}
Our lemma is largely a consequence of Sisto's theorem in \cite[Theorem 2]{Sisto2013}, which says the following. Let $G$ be a finitely generated group and let $H$ be a finitely generated subgroup that is hyperbolically embedded. Let $\Gamma$ be the Cayley graph of $G$ with respect to the finite generating set $X$ such that $L \subseteq X^*$. The embedding of $H$ in $\Gamma$ has the Morse property.

Thus, there exists an $R = R(\lambda, \epsilon)$ such that for every $(\lambda, \epsilon)$-quasi-geodesic word $u$ with the property that $\bar u \in H$, the induced path lies within $R$ of the embedding of $H$. In particular, this shows that $H$ is language-convex with respect to $L$.
\end{proof}

\begin{cor}\label{cor: acyl-contra} Let $G$ be a finitely generated acylindrically hyperbolic group with positive cone $P$. If there exists a regular quasi-geodesic positive cone language $L$ representing $P$, then there exists a regular positive cone language for $F_2$. 
\end{cor}
\begin{proof}
By Lemma \ref{F_2-embed}, we may assume there exists a hyperbolically embedded subgroup $H$ which is isomorphic to $F_2$. The subgroup $H$ is language-convex with respect to $L$ by Lemma \ref{H-L}, which means by Theorem \ref{thm: lang-convex-closure} that $H \cap P$ is a regular positive cone for $H \cong F_2$.
\end{proof} 

The main theorem of this section then follows easily. 

\begin{proof}[Proof of Theorem \ref{thm: acyl}]
Hermiller and \v{S}uni\'c's theorem (Theorem \ref{thm: herm-sunic}) states that there is no regular language representing a positive cone of $F_2$, contradicting the assumption of Corollary \ref{cor: acyl-contra}. 
\end{proof}

%% file: chap/closure-extension.tex
\chapter{Closure under extensions}\label{chap: closure-extension}

In this chapter, we will show various ways in which positive cone complexity is closed under taking extensions, as well as construct various examples. 

\section{Main results}

The results in this chapter will all be of the following type. 

\begin{thm}[See Section \ref{sec: clos-ext-quotient-leads}]\label{thm: clos-ext-reg}
	The class of finitely generated groups admitting regular positive cones is closed over passing to extensions.
\end{thm}

This result is fairly straightforward to prove, but perhaps more interesting are the explicit constructions of regular positive cones for the following classes of extension groups.

In particular, we prove the following. 

\begin{thm}[See Section \ref{sec: virtually-polycyclic}]
\label{thm: clos-ext-virt-polycyclic-reg}
	Left-orderable virtually polycyclic groups admit regular positive cones.
\end{thm}

\begin{thm}[See Section \ref{sec: BS}]\label{thm: clos-ext-BS-reg-1C}
	 For all $q \in \bZ$, the solvable Baumslag-Solitar group $\BS(1, q)$ admits a one-counter left-order. Moreover, $\BS(1, q)$ admits regular left-orders if and only if $q \geq -1$. \end{thm}

For $q\neq 0$ the solvable Baumslag-Solitar groups are defined by the presentation
$$\BS(1,q)= \langle a, b \,|\, aba^{-1} = b^{q}\rangle,$$
and for $q = 0$, $\BS(1, 0)\cong \mathbb{Z}$ which is consistent as $aba\inv = 1 \implies ab = a \implies b = 1$.  

A well-known fact (which we will nonetheless show later in the proof of Theorem \ref{thm: BS(1,q)-isomorphism}) is that these groups admit the subnormal series 
$$\{1\} \triangleleft \bZ[1/q] \triangleleft \BS(1,q).$$

In a sense, solvable Baumslag-Solitar groups are close to polycyclic groups. However, the result about regularity of  left-orders on polycyclic groups cannot be promoted to the case of all solvable groups by a result to Darbinyan \cite{Darbinyan2020}. In Section \ref{sec: BS}, we will give an answer for when a solvable Baumslag-Solitar groups $\BS(1,q)$, $q\in \bZ$ admits a regular positive cone. 

\begin{thm}[Section \ref{sec: clos-ext-wreath}]\label{thm: clos-ext-wreath-reg}
	The class of finitely generated groups admitting regular positive cones is closed under passing to wreath products.
\end{thm}

\begin{thm}[See Section \ref{sec: all left-orders are regular}]\label{thm: only-reg-poly-Z}
	A group only admits regular left-orders if and only if it is Tararin poly-$\bZ$.
\end{thm}

The results above are taken from \cite[Section 2, 3 and 4]{AntolinRivasSu2021}, where some details have been filled in for the sake of better exposition. Through doing so, Theorem \ref{thm: clos-ext-BS-reg-1C} has been improved through Proposition \ref{prop: BS-reg-conj} to show that all regular left-orders of $\BS(1,q)$ are automorphic. 

We invite the reader to consult Chapter \ref{chap: semi-wreath} as needed for a refresher on semi-direct products and wreath products, as we will use their structures quite extensively in this chapter. 

\section{Relative positive cones} 
Recall the relative positive cones of Section \ref{sec: rel-ord}. The following result will come in handy later on in the chapter. 
\begin{lem}
\label{lem: lang-from-Prel}
Let $G$ be a group and $H$ a subgroup.
Let $P_\text{rel}$ be a positive cone relative to $H$ and $P_H$ a positive cone for $H$. Then $P=P_\text{rel}\cup P_H$ is a positive cone for $G$.

Moreover, if $\cC$ is a class of languages closed under union,  $P_\text{rel}$ is  $\cC$-positive cone relative to $H$ and $P_H$ is a $\cC$-positive  cone for $H$, then $P$ is a $\cC$-positive cone. \cite{AntolinRivasSu2021}
\end{lem}
\begin{proof}
We have that $G=P_\text{rel}\sqcup H\sqcup P_\text{rel}\inv=(P_\text{rel}\sqcup P_H) \sqcup \{1\} \sqcup (P_H^{-1}\sqcup P_\text{rel}\inv)=P\sqcup \{1\} \sqcup P^{-1}.$
To see that $P$ is a semigroup, note that we have $P_\text{rel} P_\text{rel} \subseteq P_\text{rel}$ and $P_H P_H \subseteq P_H$ by assumption, and that we observed previously that $P_\text{rel} H \subseteq P_\text{rel}$ and $H P_\text{rel} \subseteq P_\text{rel}$, therefore $(P_\text{rel} \cup P_H)(P_\text{rel} \cup P_H) \subseteq (P_\text{rel} \cup P_H)$.

Finally, let $L_\text{rel}$ and $L_H$ be the $\cC$-positive cone languages of $P_\text{rel}$ and $P_H$ respectively. Then $L = L_\text{rel} \cup L_H$ is a $\cC$-positive cone language for $P$.  
\end{proof}

\section{Lexicographic left-orders where the quotient leads}
\label{sec: clos-ext-quotient-leads}

We have already seen in Lemma \ref{lem: quotient-leads} that lexicographic order on $N\times Q$ where $Q$ leads always induces a left-order on the underlying group $G$. We will now show that the positive cone of this lexicographic order, given in Lemma \ref{lem: LO-clos-ext}, also gives us a relative positive cone to which we can apply Lemma \ref{lem: lang-from-Prel} to obtain a positive cone for the extension which inherit the language complexity of those the positive cones from the quotient and kernel.

\begin{prop}\label{prop: lang-quotient-leads}
Let $\mathcal{C}$ be a class of languages closed under unions
and inverse homomorphisms.
Let $N$ and $Q$ be finitely generated groups and $G$ an extension of $Q$ by $N$.
Let $P_N$ and $P_Q$ be $\cC$-positive cones for $N$ and $Q$ respectively.
Then $P_{\prec_{lex}}$ constructed as in Lemma \ref{lem: quotient-leads} is a $\cC$-positive cone for $G$.
\end{prop}

\begin{proof}
We want to show that $f\inv(P_Q)$ is a $\cC$-positive cone relative to $N$, then use Lemma \ref{lem: lang-from-Prel} to show that $P_G = f\inv(P_Q) \cup P_N$ is a $\cC$-positive cone. 

Begin by fixing finite generating sets  $(X,\pi_N)$ and $(Y,\pi_Q)$ for $N$ and $Q$.
Let $P_N$ and $P_Q$ be  $\cC$-left-positive cones for $N$ and $Q$,
 and let $L_N\subseteq X^*$ and $L_Q\subseteq Y^*$ be in $\cC$ such that $\pi_N(L_N)=P_N$ and $\pi_Q(L_Q)=P_Q$.

Denote by $f$ the epimorphism of $G$ onto $Q$, i.e. $f\colon G\to Q$.
We can define a generating set  $(X\sqcup Y, \pi)$ for $G$ such that $\pi_N(x)=\pi(x)$ for $x\in X$ and $\pi_Q(y)=f(\pi(y))$ for $y\in Y$. That is, we do not explicitly define $\pi$ on $Y$ but only require that it coincides with $\pi_Q$ under $f$. At least one such choice exists for each $\pi(y)$, $y \in Y$ as $Q$ is a quotient of $G$.

Let $\rho: (X\sqcup Y)^*\to Y^*$ be the monoid morphism  that is the identity on $Y$ and sends elements of $X$ to the empty word, and $\widetilde{L_Q} := \rho\inv(L_Q)$. 
Note that $f(\pi(\widetilde{L_Q}))=P_Q$ by definition and hence $\pi(\widetilde{L_Q})\subseteq f^{-1}(P_Q)$.

To see that $\pi(\widetilde{L_Q})\supseteq f^{-1}(P_Q)$, let $g\in G$ such that $f(g)\in P_Q$. 
Then, there exists $w\in L_Q$ such that $\pi_Q(w)=f(\pi(w)) = f(g)$,   by definition of $\pi$. 
Since $\rho$ crushes only the generators of $N$ to the identity, there is $\tilde{w}\in \widetilde{L_Q} = \rho\inv(L_Q)$ such that $\rho(\tilde{w})=w$ and therefore
if $\tilde{g}=\pi(\tilde{w})$ we get that $f(g)=f(\tilde{g})$ and hence $g(\tilde{g})^{-1}\in N$.
There is $u\in X^*$ such that $\pi(u)=\pi_N(u)=g(\tilde{g})^{-1}$. 
Note that $\rho(u\tilde{w})=w$, therefore  $u\tilde{w}\in \widetilde{L_Q}$ and $\pi(u\tilde{w})=g(\tilde{g})^{-1}\tilde{g}=g$.

As $\cC$ is closed under inverse homomomorphism and $\widetilde{L_Q}\in \cC$, we get that $f^{-1}(P_Q)$ is $\cC$-positive cone relative to $N$.
By Lemma \ref{lem: lang-from-Prel}, $P_G = f^{-1}(P_Q)\cup P_N$ is a $\cC$-positive cone.
\end{proof}

Note that in the above Proposition \ref{prop: lang-quotient-leads}, $\cC$ can be any class of full AFL, such as regular, one-counter, or context-free languages. Thus, Theorem \ref{thm: clos-ext-reg} is simply the case where $\cC = \Reg$. 

In the case for the lexicographic order where the kernel leads, there is not always an induced left-order on the underlying group $G$. In Section \ref{sec: kernel}, we will find conditions on left-orders  on $N$ and $Q$ and the structure of $G$ so that the lexicographic order where $N$ leads induces a left-order on the group $G$. 

For now, let us look at some examples of lexicographic orders where the quotient leads. 

\subsection{Virtually polycyclic groups}\label{sec: virtually-polycyclic}

We use Proposition \ref{prop: lang-quotient-leads} to show that left-orderable virtually polycyclic groups admit $\Reg$-left-orders.

Let us recall some definitions.

\begin{defn}
Let $G$ be a group. 
A {\it subnormal series for $G$} is a increasing sequence of proper subgroups of $G$
$$\{1\} = G_0 \triangleleft G_1 \triangleleft \dots G_n = G$$
such that $G_{i}$ is normal in $G_{i+1}$ for $0 \leq i < n$. 
The quotients $G_i/G_{i-1}$ are called {\it factors}.
\end{defn}

\begin{defn}
A group $G$ is {\it polycyclic} (resp. {\it poly-$\bZ$})  if there is a finite subnormal series for $G$ with  cyclic (resp. infinite cyclic)  factors.
\end{defn}

As a consequence of Proposition \ref{prop: lang-quotient-leads} we have the following.

\begin{cor}\label{cor: poly-Z-reg-orders}
Poly-$\bZ$ groups have $\Reg$-left-orders.
\end{cor}
\begin{proof}
The corollary follows by induction on the length of the subnormal series.
The base case is $\mathbb{Z}$ and Example \ref{ex: P-Z} shows that $\bZ$ has regular positive cones. Proposition \ref{prop: lang-quotient-leads} allows the inductive argument. 
\end{proof}

Finally, we use a theorem of Morris \cite{Morris2006} to show the following proposition, for which Theorem \ref{thm: clos-ext-virt-polycyclic-reg} is a Corollary. 

\begin{prop}\label{prop: virtually polycyclic}
Left-orderable virtually polycyclic groups are poly-$\bZ$. In particular, they have $\Reg$-left-orders. 
\end{prop}

\begin{proof}
Recall that if $G$ is polycyclic, the number of  $\bZ$-factors in a subnormal series of cyclic factors is well-defined and called the {\it Hirsch length} and denoted $h(G)$.
For a  virtually polycyclic group $G$, we can define $h(G)$ to be the Hirsch length of any finite index polycyclic subgroup, and it is still well-defined (see for example \cite{Segal1983}). 
Moreover, if $H$ is a normal subgroup in $G$ then both $H$ and $G/H$ are virtually polycyclic and $h(G) = h(H) + h(G/H)$.

Since the Hirsch length of a virtually polycyclic group is well-defined, we can argue by induction on the Hirsch length that a virtually polycyclic left-orderable group is poly-$\bZ$.

If $h(G)=0$, then $G$ left-orderable and finite, so $G$ is trivial. 

Suppose now that $h(G)>0$. Morris' theorem \cite{Morris2006} says that finitely generated, left-orderable amenable groups surjects onto $\bZ$. 
Combining Morris' theorem with the well-known fact that all virtually solvable groups are amenable (see for example \cite{Garrido} for an introduction to amenable groups), and thus the fact that all polycylic groups are amenable, there is a normal kernel subgroup $N\unlhd G$ such that $G/N\cong \bZ$. Since $N$ is virtually polycyclic, left-orderable and $h(N)=h(G)-1$, we get by hypothesis that $N$ is poly-$\bZ$, and so is $G$.
\end{proof}

\subsection{Solvable Baumslag-Solitar groups}\label{sec: BS}\label{sec: ext-BS}

We start by showing the isomorphism $$\BS(1,q) \cong \bZ \ltimes_\varphi \bZ[1/q].$$ That is, let $\bZ[1/q]$ be the additive group of numbers $wq^{-s}$ with $w,s \in \bZ$. From this, we define a semidirect product $\bZ \ltimes_\varphi \bZ[1/q]$ with $\bZ$-action given by $\varphi: \bZ \to \Aut(\bZ[1/q])$ which denote by $\varphi_m := \varphi(m)$. Then $$\varphi_m(wq^s) := q^{-m} \cdot wq^s.$$

\begin{thm}\label{thm: BS(1,q)-isomorphism}
For $q\neq 0$, the elements of $G = \BS(1,q)$ can be written in the normal form $a^n(a^{-t}b^r a^t)$ with $t\geq 0, n,r \in \bZ$. Furthermore, the map $\Psi: G \to \bZ \ltimes_\varphi \bZ[1/q]$ given by $$\Psi(a^n \cdot a^{-t}b^r a^t) = (n, rq^{-t})$$ is an isomorphism map. 
\end{thm}
\begin{proof}
Playing with the relation of $\BS(1,q)$ given in its presentation above, it is straightforward to get the following rewriting rules\sidenote{If the reader does not know rewriting rules those are, replace the $\to$ arrow for the equal sign. (There is no significant difference here, I just wanted use more precise terminology.)}
$$ab \to b^qa, \quad ab\inv \to b^{-q} a, \quad ba\inv \to a\inv b^q, \quad b\inv a\inv \to a\inv b^{-q}.$$

It is then straightforward to observe that any word over $\{a,b\}$ may be rewritten as an expression of the form $a^{-\gamma}b^{r}a^{t}$ with $\gamma,t \in \bZ_{\geq 0}$ and $r \in \bZ$. To make our proof easier, we transform that expression into an expression of the form $a^{n}(a^{-t} b^r a^t)$ with $t \in \bZ_{\geq 0}$ and $n,r \in \bZ$. 

Now, let us define a map over the normal closure of the generator $b$ in the group, $\psi: \langle b \rangle^{G} \to \bZ[1/q]$ as 
$$\psi(a^{-t}b^ra^t) = rq^{-t}.$$
We claim that $\Psi: \BS(1,q) \to \bZ \ltimes_\varphi \bZ[1/q]$, 
$$\Psi(a^n \cdot a^{-t} b^r a^t) = (n, \psi(a^{-t}b^r a^t)) = (n, rq^{-t})$$
is an isomorphism map.

Let us first show that $\Psi$ is a homomorphism. We start by writing 
$$a = a^1 \cdot a^0 b^0 a^0, \qquad b = a^0 \cdot a^0 b^1 a^0, \qquad b^q = a^0 \cdot a^0 b^q a^0$$
and deduce that
$$\Psi(a) = (1,0), \qquad \Psi(b) = (0,1), \qquad \Psi(a)\inv = (1,0)\inv = (-1,0)$$ 
Since 
\begin{align*}
	\Psi(a) \Psi(b) \Psi(a)\inv &= (1,0)(0,1)(-1,0) \\
	&= (1+0, \varphi_0(1) + 0) (-1,0)\\
	&=(1,1) (-1,0) \\
	&=(1-1, \varphi_{-1}(1) + 0) \\
	&=(0,q) = \Psi(b)^q
\end{align*}
We conclude that $\Psi$ is well-defined and thus a homomorphism by universal property. 

The map $\Psi$ is surjective since $\Psi(a^n \cdot a^{-t} b^r a^t) = (n,rq^{-t})$ spans all of $\bZ \times \bZ[1/q]$ for $n,r \in \bZ$ and $t \in \bZ_{\geq 0}$. 

To show injectivity, assume that $\Psi(a^n \cdot a^{-t}b^r a^t) = \Psi(a^m \cdot a^{-s}b^w a^s$ for $s,t \geq 0$. Then, $(n,rq^{-t}) = (m, wq^{-s})$, which means that $n=m$ and $rq^{-t} = wq^{-s}$. 

Now observe that 
\begin{align*}
	&rq^{-t} = wq^{-s} \\
	&\iff rq^s = wq^t \\
	&\implies b^{rq^s} = b^{wq^t} \\
	&\iff a^s b^r a^{-s} = a^t b^w a^{-t} \qquad \text{ since $aba\inv = b^q$ } \\
	&\iff a^{-t}a^s b^r a^{-s} a^t = b^w \\
	&\iff a^{-t} b^r a^t = a^{-s} b^r a^s  
\end{align*}
Therefore, $a^n \cdot a^{-t}b^r a^t = a^m \cdot a^{-s}b^w a^s$. 

This completes the proof that $\Psi$ is an isomorphism. 
\end{proof}

We will show the following statements, which is a more precise version of Theorem \ref{thm: clos-ext-BS-reg-1C}. 

\begin{thm} The following statements are true for the Baumslag-Solitar groups $$BS(1,q)\cong \bZ \ltimes \bZ[1/q].$$
\begin{enumerate}
\item For $q \not = 0$, all $\BS(1,q)$ admit four left-orders where the quotient leads, given by $\cP = \{P_1, P_2, P_3, P_4\}$ where 
$$P_1=\{(n, rq^{-t})\in \bZ \ltimes \bZ[1/q] \mid n>0 \text{ or } (n=0 \text{ and } rq^{-t}>0)\},$$
$$P_2=\{(n, rq^{-t})\in \bZ \ltimes \bZ[1/q] \mid n<0 \text{ or } (n=0 \text{ and } rq^{-t}>0)\},$$
$P_3 := P_1^{-1}$ and $P_4 := P_2^{-1}$. Moreover,
	\item for $q = \{-1,0,1\}$, $\BS(1,q)$ is poly-$\bZ$ and $\cP$ induces the lexicographical left-orders of Corollary \ref{cor: poly-Z-reg-orders}, which are regular.
 	\item  For $q \not\in \{-1,0,1\}$, $\cP$ induces one-counter left-orders.
	\item For $q \leq -2$, all the left-orders of $\BS(1,q)$ are induced by $\cP$, and thus are one-counter.
	\item For $q \geq 2$, $\BS(1,q)$ all the left-orders are either induced by $\cP$ (and thus are one-counter) or induced by affine actions on $\bR$ such that for each $\epsilon \in \bR$, there exists an associated positive cone $Q_\epsilon$. Moreover, if $\epsilon \in \bZ[1/q]$, then $Q_\epsilon$ is regular. 
\end{enumerate}  
\end{thm}

We will prove each of the five statements in order. 

The first one is simply a corollary of Theorem \ref{thm: BS(1,q)-isomorphism}. 
\begin{cor}[Statement 1]\label{cor: BS-four-pos-cones}
For $q \not = 0$, there are only four lexicographic left-orders on $BS(1,q)\cong \bZ \ltimes \bZ[1/q]$ where the quotient leads, given by
$$P_1=\{(n, rq^{-t}) \in \bZ \ltimes \bZ[1/q] \mid n>0 \text{ or } (n=0 \text{ and } rq^{-t} > 0) \},$$
$$P_2=\{(n, rq^{-t}) \in \bZ \ltimes \bZ[1/q] \mid n<0 \text{ or } (n=0 \text{ and } rq^{-t}>0)\},$$
$P_3 := P_1^{-1}$ and $P_4 := P_2^{-1}$. 
\end{cor}
\begin{proof}
The groups $\mathbb{Z}$ and $\mathbb{Z}[1/q]$ admit only two left-orders each. The positive cones above correspond to the lexicographic left-orders constructed in Chapter 1, Lemma \ref{lem: LO-clos-ext}. 
\end{proof}

Let us now address Statement 2, which we state directly as a proposition. 
\begin{prop}[Statement 2]
	For $q = \{-1,0,1\}$, $\BS(1,q)$ is poly-$\bZ$ and $\cP$ induces the lexicographical left-orders of Corollary \ref{cor: poly-Z-reg-orders}, which are regular.
\end{prop}
\begin{proof}
	For $q = 0$, the relation is $aba\inv = 1 \iff ab = a \iff b = 1$, so $\BS(1,0) \cong \bZ$ is trivially poly-$\bZ$ its two left-orders are trivially lexicographic.

	First for $q = -1$, the relation is $aba\inv = b\inv \iff ab = b\inv a \iff bab = a$, which gives us that $\BS(1,-1) \cong K_2$, the Klein bottle group we have studied in the introduction. Recall that the poly-$\bZ$ left-orders of Corollary \ref{cor: poly-Z-reg-orders} we found in Example \ref{ex: LO-K2} correspond to the semigroups $\langle a^\pm, b^\pm \rangle^+$. Let us check that they correspond to the left-orders given by Corollary \ref{cor: BS-four-pos-cones}. We will only do the $\langle a,b \rangle^+ = P_1$ case as the others are similar. Since the set $\{a,b\}$ generates a positive cone as a semigroup on its own, it is sufficient to check that $a, b \in P_1$. To do so, we observe that $a = a^1 \cdot  a^0 b^0 a^0$, so we are in the $n = 1 > 0$ case which implies $a \in P_1$. Similarly, $b = a^0 \cdot a^0 b^1 a^0$ putting us in the $n=0, r/q^s = 1 > 0$ case which implies $b \in P_1$. This shows that $\langle a,b \rangle^+ \subseteq P_1$, and thus that $\langle a, b \rangle^+ = P_1$ by maximality argument (Lemma \ref{lem: lo-max-subset}).

	For $q = 1$, the relation is $aba\inv = b \iff ab = ba$, so $\BS(1,1) \cong \bZ^2$. In this case, the element $a^n a^{-t} b^r a^t = a^n b^r$, and the orders induced by $\cP$ are the lexicographic ones we found in the introduction. 
\end{proof}

Moving on to Statement 3, for $q \not\in \{-1, 0, 1\}$, let us show that the positive cones of $\cP$ are indeed one-counter.
\begin{prop}[Statement 3]\label{prop: BS(1,-q)}
For $q \not\in \{-1,0,1\}$ and each  $i=1,2,3,4$, there is a one-counter language over $\{a, a^{-1} ,b, b^{-1}\}^*$ evaluating onto the positive cone $P_i$ of $BS(1,q)$.
\end{prop}
\begin{proof}
Let us start with the case $q > 1$. Using Theorem \ref{thm: BS(1,q)-isomorphism}, it is straightforward to see that the languages 
\begin{align*}	
	&L_1 = \{a^n (a^{-t}b^ra^t) \mid t \geq 0, [n > 0, \text{ or } [n=0, r > 0]]\}, \\
	&L_2 = \{a^n (a^{-t}b^ra^t) \mid t \geq 0, [n < 0 \text{ or } [n=0, r > 0]]\}
\end{align*}
evaluate to $P_1$ and $P_2$ respectively. We want to show that these languages are one-counter. 

Consider the language 
$$M = \{a^{-t} b^r a^t \mid t \geq 0, r \in \bZ\}.$$ 
This language is accepted by the one-counter automaton illustrated in Figure \ref{fig: lquot} (we leave the proof to the reader\sidenote{with sincere apologies}) and is therefore one-counter. 

\begin{figure}[ht]
\begin{center}
\resizebox{\textwidth}{!}{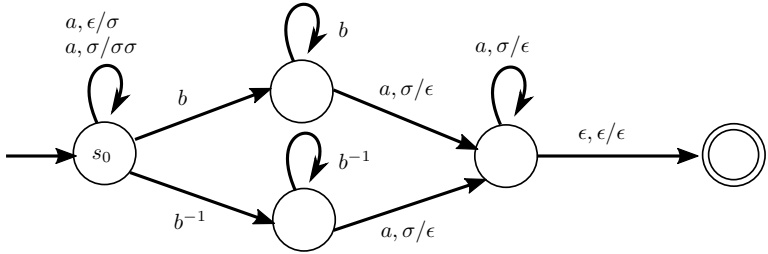}
\end{center}
\caption{Pushdown automaton accepting $M$. Here $\{\sigma\}$ is the  stack alphabet (i.e the counter symbol). The stack is used to count $t$, the number $a$'s at beginning of the string so that the automaton can decide whether that $t$ number is matched with the number of $a\inv$'s at the end of the string.}
\label{fig: lquot}
\end{figure}

Now, observe that 
\begin{align*}
	&M_1 := a^+ \cdot M = \{a^n(a^{-t}b^r a^t) \mid t \geq 0, n > 0\}, \\
	&M_1' := (a\inv)^+ \cdot M = \{a^n(a^{-t}b^r a^t) \mid t \geq 0, n < 0\},
\end{align*}
and 
$$M_2 := (a\inv)^* b^+ (a)^* \cap M = \{a^{-t}b^r a^t \mid t \geq 0,  r > 0\},$$
and that $L_1 = M_1 \cup M_2$, $L_2 = M_1' \cup M_2$. 
By closure properties of one-counter languages, $L_1, L_2$ are one-counter. 

Now let $\cdot \inv: X^* \to X^*$ be the inverse map. Notice that 
$$M\inv = \{a^{-t}b^ra^{t} \mid t \geq 0, r \in \bZ\} = M.$$

Thus, $$M_1\inv = M\inv (a\inv)^+ = M (a\inv)^+,$$
$$M_1^{-1'} = M\inv a^+ = M a^+$$
and
$$M_2\inv = (a\inv)^* (b\inv)^+ a^* \cap M\inv = (a\inv)^* (b\inv)^+ a^* \cap M$$
are all one-counter due to $M$ being $\onecounter$. This, $L_1\inv = M_1\inv \cup M_2\inv$ and $L_2\inv = M_1^{-1'} \cup M_2\inv$ evaluating to $P_3$ and $P_4$ respectively are also one one-counter. 

The case for $q < 1$ is similar. First observe that the languages 
\begin{align*}
	&L_1' = \{a^n(a^{-t}b^r a^t) \mid t \geq 0, [n > 0 \text{ or } [n = 0, [t \text{ even }, r > 0] \text{ or } [t \text{ odd }, r < 0]]]\}, \\
	&L_2' = \{a^n(a^{-t}b^r a^t) \mid t \geq 0, [n < 0 \text{ or } [n = 0, [t \text{ even }, r > 0] \text{ or } [t \text{ odd }, r < 0]]]\}
\end{align*}
evaluate to $P_1$ and $P_2$ respectively. Then construct the languages
\begin{align*}
	&N_1 := \{a^{-2}\}^*\{b\}^+\{a^2\}^* \cap M = \{a^{-t}b^ra^t \mid t \geq 0,  t \text{ even }, r > 0\}, \\
	&N_2 := \{a\inv\}\{a^{-2}\}^*\{b\inv\}^+\{a\}\{a^2\}^* \cap M = \{a^{-t}b^ra^t \mid t \geq 0,  t \text{ odd }, r < 0\},
\end{align*}
and observe that $L_1' = M_1 \cup N_1 \cup N_2$, $L_2' = M_1' \cup N_1 \cup N_2$, and $L_3' := (L_1')\inv, L_4' := (L_2')\inv$ evaluate to $P_3, P_4$ respectively. Similarly to the above, these languages are all one-counter by closure properties.
\end{proof}

We are going to now show now that the previous proposition is optimal, in the sense that positive cones $\cP$ cannot be regular. To show this, we need to state some results.

Recall first Lemma \ref{lem: regular-implies-coarsely-connected} that says that a regular positive cone is always coarsely connected. The next result makes use of the Bieri-Neumann-Strebel invariant or BNS invariant for short (see these notes from Strebel \cite{Strebel2013} for an introduction).\sidenote{Full disclosure: I do not know much about the subject and thus cannot evaluate if this is a good resource.} For the purposes of this thesis, it is enough to know that given a finitely generated infinite group $G$, the BNS invariant $\Sigma^1(G)$ is a geometric invariant of discrete groups that is given by the subset of a sphere called the \emph{character sphere} $S(G)$ of $G$. This sphere $S(G)$ is given by the set of equivalence classes $[\phi]$ of non-zero homomorphisms $\phi: G \to \mathbb{R}$ with respect to the equivalence relation $\sim$ where $\phi_1 \sim \phi_2$ when there exists some positive real number $r$ such that $\phi_1 = r\phi_2$. 

They key results from BNS theory we will use are as follows.
\begin{thm}\label{thm: BNS-belongs}
Given a finitely generated group $G$, a non-trivial homomorphism $\phi\colon G\to \bR$ belongs to the BNS invariant $\Sigma^1(G)$ if and only if  $\phi^{-1}((0,\infty))$ is coarsely connected. \cite{BieriNeumannStriebel1987}
\end{thm}

\begin{thm} \label{thm: BNS-fg}
Given a finitely generated group $G$, a non-trivial homomorphism $\phi\colon G\to \bR$, the kernel of $\phi$ is finitely generated if and only if both $\phi$ and $-\phi $ belong to $\Sigma^1(G)$. \cite{BieriNeumannStriebel1987} 
\end{thm}

Using these two theorems, we obtain the following.

\begin{lem}\label{lem: BNS}
Suppose that $G$ is a finitely generated extension of $\bZ$ by $N$. 
If $N$ is not finitely generated, then no lexicographic order on $G$ where $\bZ$ leads is regular.
\end{lem}
\begin{proof}
Suppose that $f\colon G \to \bZ$ is a non-trivial homomorphism with kernel $N$. 
By Lemma \ref{lem: quotient-leads} and without loss of generality, there is a lexicographic order where $\bZ$ leads has a positive cone of the form $P_G\coloneqq f^{-1}(\bZ_{\geq 1})\cup P_N$ where $P_N$ is a positive cone for $N$. 

Pick a group generator $x$ such that $f(x) > 0$ (if $f(x) < 0$, then pick $x\inv$ instead of $x$). We claim that since $f(P_N) = 0$, $f^{-1}(\bZ_{\geq 1})$ is coarsely connected if and only if $f^{-1}(\bZ_{\geq 1})\cup P_N$ is coarsely connected. Indeed, for the $(\implies)$ direction suppose that $f^{-1}(\bZ_{\geq 1})$ is coarsely connected with coefficient $R$. Then for any $g \in P_N$, we have that $f(gx) = f(g) + f(x) = f(x) > 0$, meaning that $P_Nx \subseteq f\inv(\bZ_\geq 1)$. Suppose that we have two elements $g, h \in f\inv(\bZ_{\geq 1}) \cup P_N$. Since $f(gx) = f(g) + f(x) > f(g) \geq 0$, we have that $gx, hx \in f\inv(\bZ_{\geq 1})$. Then $gx, hx$ can be connected by subpaths of length $\leq R$, so $g,h$ can be connected by subpaths of length $\leq R + 2$. This is illustrated in Figure \ref{fig: BNS-cc-1}

\begin{figure}[h]{\includegraphics{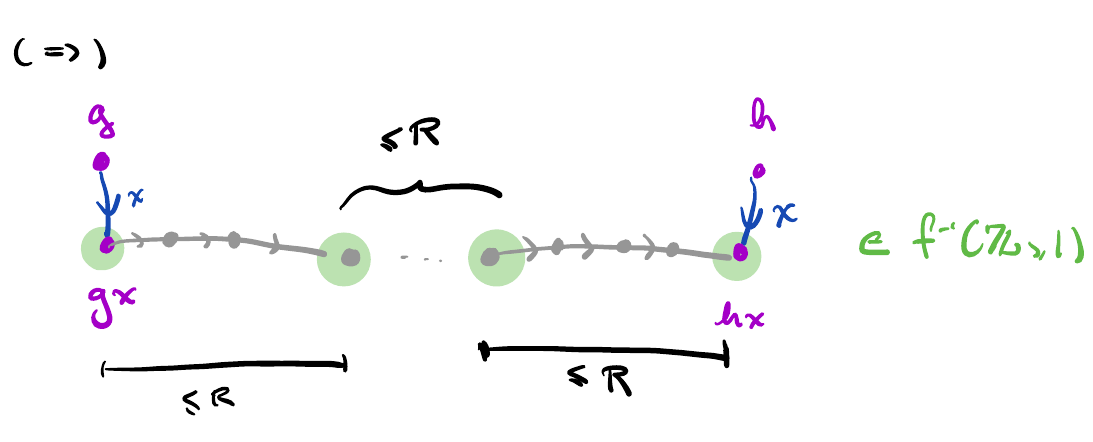}}
\caption{Illustrating the $\implies$ direction of the proof. The elements $g,h$ can be connected by subpaths of length $\leq R$ with endpoints in $f\inv(\bZ_{\geq 1})$ (highlighted in green) via $gx, hx$.}
\label{fig: BNS-cc-1}
\end{figure}	

Similarly for the $(\impliedby)$ direction, if $g,h \in f\inv(\bZ_{\geq 1})$, $g, h \in f\inv(\bZ_{\geq 1}) \cup P_N$ so they can be connected by subpaths with endpoints in $f\inv(\bZ_{\geq 1}) \cup P_N$ of length $\leq R$. Let $g', h'$ be the endpoints of one such a subpath $s$. We can extend those endpoints by $x$ such that $g'x, h'x$ are the new endpoints of $s$, which we call $s'$. Since the endpoints of $s'$ lie in $f\inv(\bZ_{\geq 1})$, we have connected $g, h$ in $f\inv(\bZ_{\geq 1})$ by subpaths of length $\leq R + 2$. This is illustrated in Figure \ref{fig: BNS-cc-2}.

\begin{figure}[h]{\includegraphics{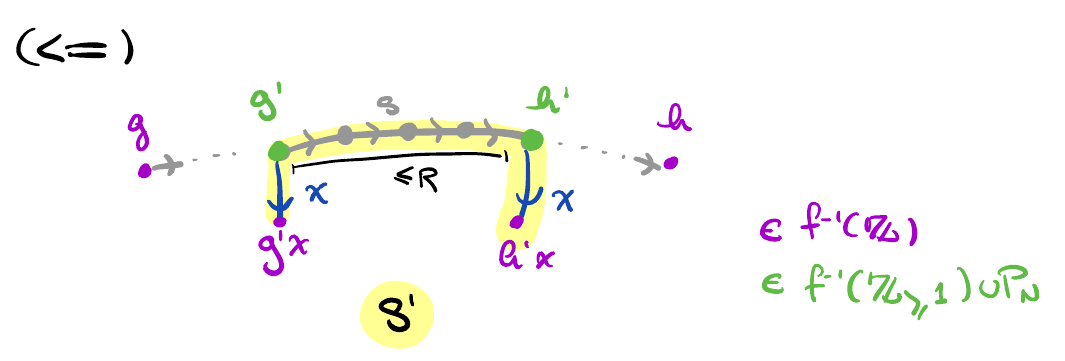}}
\caption{Illustrating the $\impliedby$ direction of the proof. The elements $g,h$ can be connected by subpaths of length $\leq R$ with endpoints in $f\inv(\bZ){\geq 1}) \cup P_N$ (in green). For each such pair of endpoints $g', h'$, we right multiply those endpoints by $x$ such that $g'x, h'x$ are paths of length $\leq R + 2$ with endpoints in $f\inv(\bZ_{\geq 1})$ (in purple). Such a resulting elongated subpath $s'$ is highlighted in yellow.}
\label{fig: BNS-cc-2}
\end{figure}	

We are done with the proof of the if and only if statement.

Since $f$ satisfies the conditions of Theorem \ref{thm: BNS-belongs}, $f \in \Sigma^1(G)$ if and only if $f^{-1}(\bZ_{\geq 1})$ is coarsely connected, that is, if and only if $f^{-1}(\bZ_{\geq 1}) \cup P_N$ is coarsely connected by the previous argument. Similarly, $-f\in \Sigma^1(G)$ if and only if $f\inv(\bZ_{\leq -1})\cup P_N^{-1}$ is coarsely connected.

Finally, assume that $N$ is not finitely generated and $P_G$ is a regular positive cone. From Theorem \ref{thm: BNS-fg}, either $P_G=f^{-1}(\bZ_{\geq 1})\cup P_N$ or $P_G\inv=f\inv(\bZ_{\leq -1}) \cup P_ N\inv$ is not coarsely connected. By Lemma \ref{lem: regular-implies-coarsely-connected} if the left-order given by $P_G$ is regular, then both cones $P_G$ and $P_G\inv$ are coarsely connected. We have arrived at a contradiction.
\end{proof}

\begin{cor}[Statement 3]\label{cor: BNS BS}
For each, $q\notin \{-1, 0, 1\}$ the positive cones $\cP$ of $\BS(1,q)$ are not regular. 
\end{cor}
\begin{proof}
It follows from Lemma \ref{lem: BNS} and the fact that $\bZ[1/q]$ is not finitely generated, that $P_1,P_2, P_3$ and $P_4$ cannot be regular.
\end{proof}

For the Statement 4 concerning the case of $q \leq -2$ having only $\cP$ as positive cones, we will use a result of Tararin. The statement of the result is however taken from \cite{BaClay2021} and \cite{ClayMannRivas2018}. 
 
 Recall Definition \ref{defn: Tararin-group} of a Tararin group. 
 
\begin{thm}
	Let $G$ be a left-orderable group. Then $G$ admits finitely many left-orderings if and only if $G$ is a Tararin group. 
	
	Moreover, if $G_i$, $i=1,\dots,n$ denotes the terms in the subnormal series, then $G$ has exactly $2^n$ left-orderings given by two orderings for each of the $G_i/G_{i-1}$ free-abelian factors of rank $1$. \cite{Tararin1991}.
\end{thm}

The lemma below shows that the result of Tararin completely characterizes the left-orders on $\BS(1,q)$ for $q \leq -1$.

\begin{lem}
	For $q \leq -1$, $\BS(1,q)$ is a Tararin group. 
\end{lem}
\begin{proof}
	Since $\BS(1,q) \cong \bZ \ltimes \bZ[1/q]$, for all solvable Baumslag-Solitar group we have the subnormal series 
	$$\BS(1,q) \triangleleft \bZ[1/q] \triangleleft \{1\}$$
	with $\BS(1,q)/\bZ[1/q] \cong \bZ$, making the series rational. 
	
	Moreover, $\BS(1,q)$ is not bi-orderable for $q \leq -1$. Indeed, suppose $\prec$ is a bi-order and without loss of generality assume that $1 \prec b$. Then, multiplying on both sides by $b\inv$ and conjugating by $a$, we get the following contradiction.
	$$1 \prec b \iff b\inv \prec 1 \iff 1 \prec aba\inv = b^{-q} \preceq b\inv.$$
\end{proof}

\begin{cor}[Statement 4]
	For $q \leq -2$, all the left-orders of $\BS(1,q)$ are induced by $\cP$, and thus are one-counter.
\end{cor}
\begin{proof}
	Since for $q \leq -1$, $\BS(1,q)$ only has $2^2 = 4$ left-orders, the lexicographic left-orders given by $\cP$ lists all the left-orders. By Statement 3, $\cP$ is one-counter for $q \leq -2$. 
\end{proof}

Finally, let us show Statement 5. Let $q \geq 2 $ and consider the map on $\rho: \BS(1,q) \to \Homeo^+(\bR)$ defined on the generators of $\BS(1,q)$ as follows $$\rho(a)(x) = qx, \quad \rho(b)(x) = x + 1.$$ Note that $q > 0$ is the condition for $\rho(a)$ to be orientation-preserving. We observe that by writing our group elements in normal form, the map sends every element of as follows
$$\rho(a^n a^{-m} b^k a^m)(x) = \rho(a^n a^{-m} b^k)(q^m x) = \rho(a^n a^{-m})(q^m x + k) = q^n(x + \frac{k}{q^m}).$$ 
Since $q > 0$, we remark that $\rho: \BS(1,q) \to \Homeo+(\bR)$. 

This map $\rho$ allows us to define the uncountable left-orderings on $\BS(1,q)$ as introduced by Smirnov \cite{Smirnov1966} (but really cited from \cite{Rivas2010}\sidenote{As the original is in Russian and I could not read it, I cited a source that cites it.}), given by 

$$Q_\epsilon = \begin{cases}
	\{g \in \BS(1,q) \mid \rho(g)(\epsilon) > \epsilon\}, & \epsilon \in \bR - \bZ[1/q] \\
	 \{g \in \BS(1,q) \mid \rho(g)(\epsilon) > \epsilon\} \cup \Stab^+(\epsilon) & \epsilon \in \bZ[1/q] \end{cases},$$	
where $\Stab^+(\epsilon)$ is defined as follows. It is known (but we will also prove it later in this section) that $\Stab(\epsilon) \cong \bZ$ for $\epsilon \in \bZ[1/q]$. Let $\Psi: \bZ \to \Stab(\epsilon)$, and let $z$ be a generator for $\bZ$. Then either $\Stab^+(\epsilon)$ is equal to either  $\Psi(\langle z \rangle^+)$ or $\Psi(\langle z\inv \rangle^+)$.

The proof is also straightforward, so we include it for accessibility. 

\begin{thm}\label{thm: BS-uncountable-P}
	Let $\epsilon \in \bR $. The set $\{Q_\epsilon \mid \epsilon \in \bR\}$ forms a collection of uncountably many positive cones for $\BS(1,q), \quad q \geq 2$. 
\end{thm}
\begin{proof}
	Let $\epsilon, \lambda \in \bR - \bZ[1/q]$ such that $\epsilon < \lambda$. It is sufficient to show that $Q_\epsilon, Q_\lambda$ are distinct since $\bR -  \bZ[1/q] \supset \bR - \bQ$ is uncountable. That is, $\exists g$ such that $g \in Q_\epsilon, g \not\in Q_\lambda$ and therefore $Q_\epsilon \not= Q_\lambda$. To do that, we will show $g(\epsilon) > \epsilon$ and $g(\lambda) < \lambda$. Without loss of generality, let $g = a^n a^{-m} b^k a^m$ where $m \in \bZ_\geq 0, n,k \in \bZ$. Then, 
	\begin{align*}
	& g(\epsilon) = q^n(\epsilon + \frac{k}{q^m}) > \epsilon \\	
	\iff & q^n \epsilon + q^n\frac{k}{q^m} > \epsilon \\
	\iff & q^n \frac{k}{q^m} > \epsilon - q^n \epsilon \\
	\iff & q^n\frac{k}{q^m} > \epsilon(1-q^n) \\
	\iff & \frac{k}{q^m} > \frac{\epsilon(1-q^n)}{q^n}
	\end{align*}
Similarly, we get $g(\lambda) < \lambda \iff \frac{k}{q^m} < \frac{\lambda(1-q^n)}{q^n}$. That is, we need to select $n, k,m$ such that 
$$\frac{\epsilon(1-q^n)}{q^n} < \frac{k}{q^m} < \frac{\lambda(1-q^n)}{q^n}.$$
By fixing $n < 0$, is it straightforward to see there always exists $k,m$ satisfying the inequality, and thus an element $g$ which differentiates $Q_\epsilon$ and $Q_\lambda$. 

\end{proof}

By a result of Rivas, the left-orders arising from $Q_\epsilon, \epsilon \in \bR$ and $\cP$ are the only left-orders for $\BS(1,2)$. Moreover, according to the paper the proof does not depend on $q$ and can be extended for all $q$ by changing the mapping $\rho(a) = 2x$ to $\rho(a) = qx$ as we have done. We write this conclusion formally. 

\begin{thm}
	For $q \geq 2$, the positive cones of $\BS(1,q)$ are given by $\cP \cup \{Q_\epsilon \mid \epsilon \in \bR\}$. \cite{Rivas2010}
\end{thm}

Finally, to show that $Q_\epsilon$ is a regular positive cone for $\epsilon \in \bZ[1/q]$ we start by constructing a regular left-order for the positive cone $Q_0$. 

\begin{prop}\label{prop: BS-reg-Pcone-Q0}
	$Q_0$ is a regular positive cone for $\BS(1,q)$, $q \geq 2$ which can be given by $$Q_0= \{a^n\mid n>0\} \cup \{ a^n(a^{-m}b^ka^{m})\mid k > 0, m \geq 0, n\in \bZ \}.$$ 
\end{prop}
\begin{proof}
Observe that since $\rho(g) = \rho(a^n a^{-m} b^k a^m)(x) = q^n(x + \frac{k}{q^m})$, every element with $k > 0$ belongs to $Q_0$. Moreover, if $g(0) = 0$, then $g = a^n$. Therefore, $\{a^n\mid n>0\} = \Stab^+(0)$, we have
$$Q_0=\{a^n\mid n>0\}\cup \{ a^n(a^{-m}b^ka^{m})\mid k> 0, m \geq 0,                                  n\in \bZ\}.$$ 

The positive cone $Q_0$ can be represented as a regular expression 
$$L_0 = a^+ \cup (a | a\inv)^*b^+a^*.$$
\end{proof}

\begin{prop}\label{prop: BS-reg-conj}
$Q_\epsilon$ with $\epsilon \in \bZ[1/q]$ is a regular positive cone for $\BS(1,q)$, $q \geq 2$, and is automorphic to $Q_0$ under the conjugation 
$$Q_\epsilon = hQ_0h\inv$$
where $\epsilon = \frac{r}{q^s}$ and $h = a^{-s}b^r a^s$.
\end{prop}
\begin{proof}
	Let $\epsilon = \frac{r}{q^s} \in \bZ[1/q]$, and observe that if $h = a^{-s}b^r a^s$, then $h(0) = q^0(0 + \frac{r}{q^s}) = \frac{r}{q^s}$. We claim that $hQ_0h\inv = Q_{rq^{-s}}$. 

Indeed, suppose that $g_0 \in Q_0$, and first assume that $g_0(0) > 0$. Then $hg_0h\inv(rq^{-s}) = h(g_0(0)) = g_0(0) + rq^{-s} > rq^{-s}$, so $hg_0h\inv \in Q_{rq^{-s}}$. Similarly, if $g_0(0) = 0$, then $hg_0h\inv(rq^{-s}) = hg_0(0) = h(0) = rq^{-s}$. Therefore, $hQ_0h\inv \subseteq Q_{rq^{-s}}$. Since conjugation by $h$ is an automorphism, $hQ_0h\inv$ is a positive cone by Lemma \ref{lem: lo-clos-isom}, so it must be that $hQ_0h\inv = Q_{rq^{-s}}$ by Lemma \ref{lem: lo-max-subset}, where $\Stab^+(rq^{-s}):= h\Stab^+(0)h\inv$.

To construct a regular language for $Q_{rq^{-s}}$, let $w_h = a^{-s}b^r a^s$ as a word, and $w_{h\inv} = a^{-s} b^{-r} a^s$. Let $L_0$ be a regular language for $Q_0$. Then $L_{rq^{-s}} := w_h L_0 w_{h\inv}$ is our desired regular language. 
\end{proof}

\begin{rmk}
	Recall that in Chapter \ref{chap: LO}, we saw in Lemma \ref{lem: LO-dyn-top-conjugate} that any two dynamical realisations are topologically conjugate. The above result seems reminiscent of such a phenomenon. 
\end{rmk}

\begin{cor}
For $q \geq 2$, $\BS(1,q)$ all the left-orders are either induced by $\cP$ (and thus are one-counter) or induced by affine actions on $\bR$ such that for each $\epsilon \in \bR$, there exists an associated positive cone $Q_\epsilon$. Moreover, if $\epsilon \in \bZ[1/q]$, then $Q_\epsilon$ is regular. 	
\end{cor}

We close the section on $\BS(1,q)$ with the following observation. 

\begin{obs}[Language complexity is positive cone dependent]	
Whereas the complexity of a positive cone is stable under changing finite generating sets, extensions and taking language-convex subgroups, the complexity of a positive cone language is dependent on the positive cone in question. Indeed, for $\BS(1,q)$ with $q \geq 2$, we have seen that the group admit both regular and one-counter positive cones. 
\end{obs}

\section{Lexicographic left-orders  where the kernel leads}\label{sec: kernel}

Recall the discussion of left-orders where the kernel leads in Chapter \ref{chap: LO}, Section \ref{sec: LO-clos-ext}. Lemma \ref{lem: semidirect} will be helpful to construct regular left-orders on extensions when the kernel is not finitely generated. Our main example are wreath products, however we have already observed this phenomenon in Baumslag-Solitar groups.

\begin{ex}[Lexicographic left-orders on $\BS(1,q)$ where the kernel leads]
For $q>0$ we have already seen that $\BS(1,q)$ has regular orders, constructed through an affine action on the real line.
Viewing $\BS(1,q)=\langle a,b \, \mid aba^{-1}=b^{q}\rangle \cong \bZ[1/q]\rtimes \bZ$, the positive cone $Q_0$ of Proposition \ref{prop: BS-reg-Pcone-Q0} is lexicographic where the factor $\bZ[1/q]$ leads. Indeed, using the isomorphism of Theorem \ref{thm: BS(1,q)-isomorphism}, $Q_0=  \{ a^n(a^{-m}b^ka^{m})\mid k > 0, m \geq 0, n\in \bZ \} \cup \{a^n\mid n>0\} $ in $\BS(1,q)$ corresponds to the set $\{(kq^{-m}, n) \mid k > 0, m \geq 0, n \in \bZ\} \cup \{(0, n) \mid n > 0\} $ in $\bZ[1/q] \rtimes \bZ$, which can be associated with the set $P_{\bZ[1/q]} \bZ \cup P_\bZ = P_NQ \cup P_Q$. 

Note that the $\bZ$-action on $\bZ[1/q]$ as a left outer semi-direct product is given by $\varphi_r(kq^{-m}) = q^r\cdot  kq^{-m}$ for any $r \in \bZ$ (our previous exposition used the right semi-direct product action). Meaning, $\varphi_r(P_{\bZ[1/q]}) = P_{\bZ[1/q]}$ since positive numbers go to positive numbers. From an inner semi-direct product perspective, this gives us $a^r P_{\bZ[1/q]} a^{-r} = P_{\bZ[1/q]}$ (with the embedding of $P_{\bZ[1/q]}$ in $\BS(1,q)$ being implicit here), so the condition $pP_Np\inv = P_N$ for $p \in Q$ for $G$-invariance of the left-order as stated above is satisfied.
\end{ex}

\subsection{Wreath products}\label{sec: clos-ext-wreath}
We begin with an example that will illustrate the construction we develop in this section.
Recall from Section \ref{sec: wreath-products} that $\bZ \wr \bZ = (\oplus_{i \in \bZ} \bZ) \rtimes_\phi \bZ$ by definition, where the multiplication of two elements $({\bf m},p)$ and $({\bf m},q)$ is given by 
$$({\bf n}, q)({\bf m},p)=((n_i + m_{i + q})_{i \in \bZ}, q + p).$$
We found in Example \ref{ex: LO-bZ-wr-bZ} that this group is left-orderable by a positive cone 
\begin{equation}\label{eqn: P-wreath}
	P = \{( {\bf n}, q) \mid {\bf n} \in P_N \text{ or } {\bf n} = {\bf 0}, q \in P_Q\},
\end{equation}

where $P_N$ and $P_Q$ are positive cones for $\oplus_{i\in \bZ} \bZ$ and $\bZ$ respectively. We would like to describe this positive cone in terms of a regular language. 

We start by stating following well-known theorem in the field of lamplighter groups (see for example \cite{Saint-Criq2021}). Note however that they employ the right semidirect product convention for the lamplighter group, whereas we employ the left.)

\begin{thm}
Let $R = \bZ[X, X\inv]$, the ring of polynomial in the formal variables $X$ and $X\inv$ whose coefficients are in $\bZ$. Let $\cM \leq \GL(2, R)$, such that
$$\cM = \bigg\{ \begin{pmatrix}
	 1 & 0 \\
	 P & X^k 
	 \end{pmatrix} \bigg| P \in \bZ, k \in \bZ \bigg\}.$$
Then, there is an isomorphism $\Phi: \bZ \wr \bZ \to \cM$ given by 
$$({\bf n}, q) \mapsto 
\begin{pmatrix}
	1 & 0 \\
	\sum_{i \in \bZ} n_i X^i & X^{-q}
\end{pmatrix}.$$
\end{thm}
\begin{proof}
First note that $\Phi$ is well-defined since ${\bf n}$ is finitely supported. The bijectivity of $\Phi$ is straightforward to see. For the homomorphism part, observe that 
\begin{align*}
\Phi(({\bf n}, q)) \Phi(({\bf m},p)) 
&= \begin{pmatrix}
	1 & 0 \\
	\sum_{i \in \bZ} n_i X^i & X^{-q}
\end{pmatrix}
\begin{pmatrix}
	1 & 0 \\
	\sum_{i \in \bZ} m_i X^i & X^{-p}
\end{pmatrix} \\
&= \begin{pmatrix}
	1 & 0   \\ 
	\sum_{i \in \bZ} n_i X^i + \sum_{i \in \bZ} m_i X^{i-q}  & X^{-(q+p)}
\end{pmatrix} \\
&=\begin{pmatrix}
	1 & 0   \\ 
	\sum_{i \in \bZ} (n_i + m_{i+q})X^i  & X^{-(q+p)}
\end{pmatrix} \\
&= \Phi(((n_i + m_{i+q})_{i \in \bZ}, q + p)).
\end{align*}
\end{proof}

We are ready to prove the following result. 

\begin{prop} \label{prop: lang-Z-wr-Z}
The group $\bZ\wr \bZ$ has $\Reg$-left-orders.
\end{prop}
\begin{proof}
Let $\rho = \rho(X)$ stand for a polynomial in $X$. An element of $\bZ\wr \bZ$ can be uniquely identified with the tuple $$(\rho = n_{i_0} X^{i_0} + n_{i_1} X^{i_1} + \dots n_{i_k} X^{i_k}, X^q)$$ with $i_0 >i_1 >i_2 > \dots > i_k \in \bZ$ and $q\in \bZ$. 
We call $n_{i_0}$ the leading coefficient of $\rho$, and write $\mathsf{leadcoef}(\rho):=n_{i_0}$. Let $$P_N = \{(n_i)_{i \in \bZ} \mid n_{i_0} > 0\}$$ and $$P_Q = \{q \in \bZ \mid q > 0\}$$ and $P$ be as in Equation \ref{eqn: P-wreath}. 

It is easy to check that $$\Phi(P)=\{(\rho,X^q) \mid  \mathsf{leadcoef}(\rho)>0 \}\cup \{(0,X^q) \mid q>0\},$$ up to identifying $(\rho, X^q)$ with $\begin{pmatrix} 1 & 0 \\ \rho & X^{-q} \end{pmatrix}$. 

Let $T = ({\bf 0}, 1)$ and $A = (\delta_0, 0)$, where $\delta_0 = (\dots, 0, 1, 0, \dots)$, where the $1$ is at index $0$ in the $\oplus_{i \in \bZ} \bZ$ factor of the wreath product, be the generators of $\bZ \wr \bZ$ as in Section \ref{sec: LO-lamplighter}. Define 
$$t := \Phi(T) = 
\begin{pmatrix}
	1 & 0 \\
	0 & X\inv 
\end{pmatrix}, \quad 
a := \Phi(A) = 
\begin{pmatrix}
	1 & 0 \\
	1 & 1
\end{pmatrix}.$$
Since $T,A$ were generators of $\bZ \wr \bZ$, we have that $t,a$ are the generators of $\bZ[X, X\inv]$. 

In particular, we can check that for all $r, s, v \in \bZ$,
$$t^r = \begin{pmatrix}
1 & 0 \\
0 & X^{-r}
\end{pmatrix}, \qquad 
a^s = \begin{pmatrix}
1 & 0 \\
s & 1
\end{pmatrix}, \qquad 
t^r a^s t^{-r} = \begin{pmatrix}
1 & 0 \\
sX^{-r} & 1
\end{pmatrix}, \qquad 
$$
Moreover, we have the following additive property
$$t^{r_1} a^s_1 t^{-r_1} \cdot t^{r_2} a^s_2 t^{-r_2} = \begin{pmatrix}
1 & 0 \\
s_1X^{-r_1} + s_2X^{-r_2} & 1
\end{pmatrix}.$$
By induction, this gives us the formula
$$\prod_{i=i_0}^{i_k} t^ia^{n_i}t^{-i} = \begin{pmatrix}
 1 & 0 \\
 \sum_{i=i_0}^{i_k} n_i X^{-i} & 1	
\end{pmatrix}.
$$
Finally, we remark that 
$$\prod_{i=i_0}^{i_k} t^ia^{n_i}t^{-i} \cdot t^{\ell} = \begin{pmatrix}
1 & 0 \\
 \sum_{i=i_0}^{i_k} n_i X^{-i} &  X^{-\ell}	
\end{pmatrix}.$$
Let $(p, X^\ell)$ be a tuple in $\Phi(P)$. Then $(p, X^\ell) = (n_{i_0} X^{i_0} + n_{i_1} X^{i_1} + \dots n_{i_k} X^{i_k}, X^\ell)$, which means it is given by the matrix product $t^{{i_0}} a^{n_{i_0}} t^{{i_1}-{i_0}} a^{n_{i_1}} t^{{i_2}-{i_1}} \cdots a^{n_{i_k}}t^{{-i_k}} t^\ell$, where $i_0 >i_1 >i_2> \dots >i_k$, where the $t_{i_j}\cdot t_{-i_{j+1}}$ factors have been simplified, and $i_1-i_0, i_2-i_1, \dots, i_{k}-i_{k-1}$ are all negative by assumption on the $i_j$'s. 
This gives us a language for $\Phi(P)$ as  
$$L = \{t^{n}a^m t^{n_1} a^{m_1} t^{n_2} a^{m_2}\dots a^{m_k} t^\ell \mid m>0,  n_i<0, m_i\in \bZ, k \geq 0, \ell\in \bZ\} \cup \{t^q \mid q>0\}.$$
This language is recognized by the finite state automaton of Figure \ref{fig: wreath-fsa} and is therefore regular.
\begin{figure}[ht] 
\begin{center}
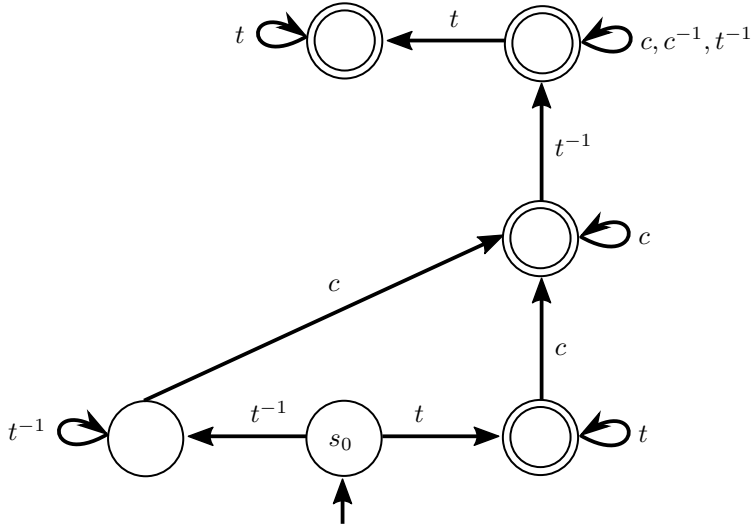
\end{center}
\caption{Finite state automaton accepting a positive cone language for $\bZ \wr \bZ$.}
\label{fig: wreath-fsa}
\end{figure}
\end{proof}

\begin{rmk}
	A presentation for $\bZ \wr \bZ$ is given by $\langle T, A \mid [T^i A T^{-i}, T^j A T^{-j}] \rangle$, where $T$ and $A$ can be identified with the matrices $t$ and $a$ respectively, as we found above. The commutators found in the relations can be understood as follows: a lamplighter first advancing $i$ to the right, changing the state of a lamp, then going back $i$ to the left and second then going $j$ to the right, changing the state of a lamp, then going back $j$ to the left gives the same result as doing the second step before the first step. (See for example \cite{Saint-Criq2021} for a deeper treatment on the subject.) 
\end{rmk}

The strategy used in the proof of Proposition \ref{prop: lang-Z-wr-Z} also works in the more general case of wreath products $N \wr Q$. Let $\prec_N$ and $\prec_Q$ be left-orders on $N$ and $Q$ respectively, with corresponding positive cones $P_N$ and $P_Q$.
We can construct a lexicographic order $\prec_{{\bf N}}$ on ${\bf N }= \oplus_{q\in Q} N$ as follows.  Given ${\bf n}=(n_q)_{q\in Q},{\bf n}'= (n'_q)_{q\in Q}\in {\bf N}$ we put ${\bf n}\prec {\bf n'}$ if ${\bf n}\neq {\bf n'}$ and for $q'=\max_{\prec_Q} \{q\in Q \mid n_{q}\neq n'_{q}\}$ we have that $n_{q'}\prec_N n'_{q'}$. Meaning that we take the maximal $Q$-index for which the ${\bf N}$-entries differ to compare ${\bf N}$-elements. This is a generalisation of comparing leading coefficients to left-order polynomials.

\begin{lem}\label{lem: lex on wreath}
The lexicographic order on $G = N \wr Q = {\bf N} \rtimes Q$ extending $\prec_{ \bf N}$ and $\prec_Q$ is an $G$-left-invariant lexicographic order with leading factor ${\bf N}$.
\end{lem}

\begin{proof}
Since $P_{\bf N} = \{ {\bf n} \succ {\bf 1_{\bf N}} \}$ by definition, we have that $$P_{\bf N} = \{ (n_q)_{q\in Q} \in {\bf N} \mid n_{q'}\in P_N \text{ where } q'=\max \{q\in Q \mid n_q\neq 1_N\}\}.$$

Our goal is to show that $P_{\bf N}$ is $Q$-invariant under conjugation, and thus satisfies the hypothesis of Lemma \ref{lem: semidirect}. To do so, we implicitly pass from the definition of wreath product as an outer semidirect product to an inner one. 

Given ${\bf n}\in P_{\bf N}$ and $q'=\max \{q\in Q \mid n_q\neq 1_N\}$, and given $p\in Q$ we set ${\bf n}' = p {\bf n} p^{-1}$. Observe that ${\bf n}' = \phi_{p}({\bf n}) = (n_{pq})$, since we are in the inner semidirect product view of $G$. Therefore $pq' = \max_{\prec_Q} \{q\in Q \mid n'_p = n_{pq} \neq 1_N\}$ since $\prec_Q$ is left $Q$-invariant and we have simply left-shifted all the $Q$-indices of ${\bf n}$ by $p$ when passing to ${\bf n'}$. 

We matches the indices of ${\bf n'}$ and ${\bf n'}$ to obtain that $n'_{pq'}=n_{q'} \in P_N$ and thus ${\bf n}' \in P_{\bf N}$.
We have showed that $p P_{\bf N } p^{-1} \subseteq P_{\bf N}$ for all $p \in Q$, which implies  $p P_{\bf N} p^{-1} =P_{\bf N}$ for all $p \in Q$ by symmetry. The proof then follows from Lemma \ref{lem: semidirect}. 
\end{proof}

We are now ready to tackle obtaining a positive cone language of complexity $\cC$ when $N,Q$ both have $\cC$-positive cones for the general case of the wreath product.

Let $X$ and $Y$ be generating sets of $N$ and $Q$ respectively. 
The set $X\cup Y$ generates  $N\wr Q$ since the $q$-th copy of $N$ in $\oplus_{q\in Q}N$ is identified with $qNq^{-1}$ and thus it can be generated by $qXq^{-1}$ (the conjugates of $X$ by $q$) and each element of $qXq^{-1}$ can be expressed in terms of $X$ and $Y$. 

An element $({\bf n}=\{n_q\}_{q\in Q}, p)\in N\wr Q$ can be written as 
$(\prod_{q\in Q} q n_q q^{-1}) p$, and we can use the $\prec_Q$-order  to write this element uniquely as
$$(q_1 n_1 q_1^{-1})(q_2 n_2 q_2^{-1})\cdots (q_n n_m q_m^{-1}) p$$
with the property that $q_1\succ_Q q_2 \succ_Q q_3 \succ_Q \dots \succ_Q q_m$ since the order is total. 

Thus, with this unique way of writing the elements of $N\wr Q$, a lexicographic positive cone of $N\wr Q$ is 
\begin{equation} \label{wr-pos}
P=\left\lbrace(q_1 n_1 q_1^{-1})(q_2 n_2 q_2^{-1})\cdots (q_m n_m q_m^{-1}) p \middle|  \begin{array}{c} q_1\succ_Q q_2 \succ_Q \dots \succ_Q q_m\\
(n_1 \in P_N \text{ and }  p\in Q) \text{ or } (m=0 \text{ and } p\in P_Q) \end{array}\right\rbrace.
\end{equation}

We will now define a positive cone language for wreath products of groups in terms of the positive (and negative) cone languages for $N$ and $Q$. 
\begin{prop} \label{prop: wr-lang}
Let $L_N\subseteq X^*$ and $L_Q, M_Q\subseteq Y^*$ be languages such that $\pi_N(L_N)=P_N$ and $\pi_Q(L_Q)=P_Q$ are positive cones for $N$ and $Q$ respectively, and  $\pi_Q(M_Q) = P_Q\inv$. 
Then, the  language
\begin{equation} \label{eq: pos-wr-AFL}
L \coloneqq Y^* L_N M_Q (X^*M_Q)^* Y^* \cup L_Q
\end{equation}
evaluates onto the positive cone $P$ of Equation \eqref{wr-pos}.  
\end{prop}
\begin{proof}
First observe that
\begin{equation} \label{eq: pos-wr}
L  =
 \left\lbrace v u_1 w_1 u_2 w_2 \dots u_m w_{m} z \middle| \begin{array}{c} v, z \in Y^*, \\ u_1 \in L_N \text{ or } (m=0, v= \varepsilon \text{ and } z \in L_Q), \\ u_i\in X^*, w_i \in M_Q\\
\end{array}\right\rbrace.
\end{equation}

Let $P$ be the positive cone described in \eqref{wr-pos}.

Let us first prove that $P \subseteq \pi(L)$. 
Let $g\in P$, and assume   that $g= (q_1 n_1 q_1^{-1})(q_2 n_2 q_2^{-1})\cdots (q_m n_m q_m^{-1}) p$ with $q_i,n_i,p$ as in \eqref{wr-pos}.
Let $z\in Y^*$ such that $\pi(z)=p$.
If $m=0$, then   $g=p \in P_Q$. 
We can assume that $z \in L_Q$.
If $m > 0$,  there is $v \in Y^*$ such that $\pi(v) = q_1$, $u_1 \in L_N$ such that $n_1 \in P_N$, and $w_i \in M_Q$ such that $\pi(w_i) = q_i\inv q_{i+1} \in P_Q\inv$, $u_i \in X^*$, such that $\pi(u_i) = n_i$ for $2 \leq i \leq m$. 
We see that $g= \pi(vu_1w_2u_2\dots u_mw_mz).$

To prove that $\pi(L) \subseteq P$, let ${\bf w}=v u_1 w_1 u_2 w_2 \dots u_m w_{m} z\in L$ as in the description in \eqref{eq: pos-wr}.
If $m=0$, then ${\bf w}= z$ and $z\in L_Q$. Thus $\pi({\bf w})\in P$.
If $m>0$, let $q_1=\pi(v)$ and for $i>1$, $q_i=q_{i-1}\pi(w_i)$, thus $\pi(w_i)=q_{i-1}^{-1}q_i$. For $i=1,\dots, m$ let $n_i=\pi(u_i)$ and $p=\pi(z)$.
Therefore $\pi({\bf w})=(q_1u_1q_1^{-1})(q_2u_2q_2^{-1})\cdots (q_mu_mq_m^{-1})p$.
Note that since $w_i\in M_Q$, we have that $q_i\prec_Q q_{i-1}$. It follows that $\pi({\bf w}) \in P$. 
\end{proof}

We now state a generalization of Proposition \ref{prop: lang-Z-wr-Z}/ 

\begin{prop}\label{prop:wreath product}
Let $\cC$ be a full AFL closed under reversal.
Let $N$ and $Q$ be finitely generated groups.
Suppose that $N$ and $Q$ have a $\cC$-left-order represented by $L_N$ and $L_Q$ respectively.
Then, the  wreath product  of $N\wr Q$ admits a $\cC$-left-order.

In particular admitting $\Reg$-left-orders is closed under wreath products.
\end{prop}
\begin{proof}
Assume that $N$ is generated by a finite set $X$ and $Q$ is generated by a finite set $Y$. 
We will construct a language over the generating set $X\sqcup Y$.
Let $M_Q = L_Q\inv$ be the  negative cone language associated to $L_Q$ obtained by reversal and sending each letter $x \mapsto x\inv$. 
Let $L$ be the language of Equation \eqref{eq: pos-wr-AFL}. 
By Proposition \ref{prop: wr-lang}, $L$ is a positive cone language for $N\wr Q$.
Since a class full AFL is  closed by concatenation, concatenation closure and union, we see that $L$ is in $\cC$.
\end{proof}

It is clear that Theorem \ref{thm: clos-ext-wreath-reg} follows as a corollary where $\cC = \Reg$. 

\section{Groups where all positive cones are regular}
\label{sec: all left-orders are regular}

In this section we classify the groups that only admit $\Reg$-left-orders.  

Observe that the set of finite state automata is countable and thus a left-orderable group can have at most a countable number of regular left-orders. The following 2001 result of Linnell implies that if all the left-orders are regular, then there should be finitely many of them. 

\begin{thm}
If a group admits infinitely many left-orders, then it admits uncountably many. \cite{Linnell2001} 
\end{thm}

The case when a group admits finitely many left-orders was classified by Tararin \cite{Tararin1991} (see also \cite{KopytovMedvedev1996, DeroinNavasRivas2016}). Recall that a torsion-free abelian group has {\it  rank 1} if for any two non-identity elements $a$ and $b$ there is a non-trivial relation between them over the integers: $na+mb=0$.  Torsion-free abelian groups of rank $1$ are, up to isomorphism, subgroups of $\mathbb{Q}$.

\begin{defn}[Tararin group]\label{defn: Tararin-group}
A group $G$ is a \emph{Tararin group} if it admits a finite rational series
\[
G=G_0 \triangleright G_1 \triangleright \cdots \triangleright G_n=\{1\}
\]
such that each factor $G_i/G_{i+1}$ is torsion\textendash free abelian of rank $1$, and for every $i=0,\dots,n-2$ the two\textendash step quotient $G_i/G_{i+2}$ is not bi\textendash orderable.
\end{defn}

A group admits finitely many left-orders if and only if it is a Tararin group (see \cite[Theorem 2.2.13]{DeroinNavasRivas2016}).
Moreover, let 
\[
G=G_0 \triangleright G_1 \triangleright \cdots \triangleright G_n=\{1\}
\]
be its (unique) Tararin series. For any left-order $\prec$ on $G$, the proper $\prec$-convex subgroups are exactly $G_1,\dots,G_n$. 
Consequently $|\mathrm{LO}(G)|=2^{\,n}$, and each left-order is completely determined by choosing the positivity of a nontrivial element in each rank-one factor $G_i/G_{i+1}$ for $i=0,\dots,n-1$.
More concretely, every left-order on $G$ is lexicographic with respect to the extensions
\[
G_{i+1}\hookrightarrow G_i \twoheadrightarrow G_i/G_{i+1},
\]
with the quotient leading, for each $i=0,\dots,n-1$.

\begin{lem}\label{lem: fg-Tararin-reg}
Suppose that $G$ is a finitely generated Tararin group with all the left-orders being regular. 
Then $G$ is poly-$\bZ$.
\end{lem}
\begin{proof}
Let $$G = G_0 \unrhd G_1 \unrhd \dots  \unrhd G_n = \{1\}$$ 

be the unique subnormal series of $G$ where  all the factors are torsion-free abelian groups of rank 1.
We have to show that all factors are cyclic.
We argue by induction on the length of the series. 
If the length is $0$, $G\cong \{1\}$.

Suppose that the length is $>0$.
Since $G$ is finitely generated, we get that $G_0/G_1$ is a finitely generated subgroup of $\bQ$ and hence $G_0/G_1\cong \bZ$.
Thus, $G=G_1\rtimes \bZ$.

Suppose that $G_1$ is finitely generated. 
Then by Proposition \ref{prop: convex implies L-convex} all the induced left-orders on $G_1$ are regular.
Since $G$ is a Tararin group,  all left-orders on $G_1$ are restrictions of left-orders in $G$.
Therefore, all left-orders in $G_1$ are regular, and by induction $G_1$ is poly-$\bZ$ and so is  $G$.

The remaining case is that $G_1$ is not finitely generated. 
Then, it follows from Lemma \ref{lem: BNS} that the lexicographic orders cannot be regular.
\end{proof}

Conversely we have the following. 

\begin{lem}\label{lem: poly-Z-Tararin-reg}
All left-orders on a poly-$\bZ$ Tararin group are regular.
\end{lem}
\begin{proof}
The proof is by induction on the Hirsch length. If $h(G)=0$, then $G\cong \{1\}$ and the lemma holds.

Now assume that $h(G)>0$.
 Let $G_1\unlhd G$ such that $G/G_1$ is infinite cyclic. 
Since $G$ is a Tararin group $G_1$ is $\prec$-convex in $G$ for any left-order $\prec$. 
Thus, any order on $G$ is a lexicographic order associated to an extension of $\bZ$ by $G_1$ in which the quotient group is the leading lexicographic factor. 
Since $G_1$ is a poly-$\bZ$ Tararin group with $h(G_1)<h(G)$, we get by induction 
that all the left-orders on $G_1$  are regular.
Recall from Example \ref{ex: P-Z} that the two left-orders that $\bZ$ admits are regular.
Therefore,  all the left-orders of $G$ are lexicographic extensions of a regular order on $G_1$ and a regular order on $\bZ$ and by Lemma \ref{lem: LO-clos-ext} all left-orders of $G$ are regular. 
\end{proof}

\begin{rmk}
Although we will not use this, it is worth pointing out that a group $G$ is poly-$\bZ$ Tararin if and only if there exists a unique subnormal series $$G = G_0 \unrhd G_1 \unrhd \dots  \unrhd G_n = \{1\}$$ 
such that for all $i$,  $G_i/G_{i+1}\cong \bZ$ and $G_i/G_{i+2}\cong K$ where $K=\langle a,b \mid aba^{-1}=b^{-1}\rangle$ is the Klein bottle group.
\end{rmk}

To prove Theorem \ref{thm: only-reg-poly-Z} that a group  only admits regular left-orders if and only if it is Tararin poly-$\bZ$, observe that by the previous discussion a group that admits only regular left-orders must admit only a countable number of left-orders and therefore it must be a Tararin group.

We now need to show that a Tararin group only admits regular left-orders if and only if it is poly-$\bZ$. That result follows from Lemma \ref{lem: fg-Tararin-reg} and Lemma \ref{lem: poly-Z-Tararin-reg}.

%% file: figs/lquot.pdf_tex
%% Creator: Inkscape 1.0.1 (c497b03c, 2020-09-10), www.inkscape.org
%% Accompanies image file 'lquot.pdf' (pdf, eps, ps)
\begingroup%
  \makeatletter%
  \providecommand\color[2][]{%
    \errmessage{(Inkscape) Color is used for the text in Inkscape, but the package 'color.sty' is not loaded}%
    \renewcommand\color[2][]{}%
  }%
  \providecommand\transparent[1]{%
    \errmessage{(Inkscape) Transparency is used (non-zero) for the text in Inkscape, but the package 'transparent.sty' is not loaded}%
    \renewcommand\transparent[1]{}%
  }%
  \providecommand\rotatebox[2]{#2}%
  \newcommand*\fsize{\dimexpr\f@size pt\relax}%
  \newcommand*\lineheight[1]{\fontsize{\fsize}{#1\fsize}\selectfont}%
  \ifx\svgwidth\undefined%
    \setlength{\unitlength}{374.10259577bp}%
    \ifx\svgscale\undefined%
      \relax%
    \else%
      \setlength{\unitlength}{\unitlength * \real{\svgscale}}%
    \fi%
  \else%
    \setlength{\unitlength}{\svgwidth}%
  \fi%
  \global\let\svgwidth\undefined%
  \global\let\svgscale\undefined%
  \makeatother%
  \begin{picture}(1,0.30219167)%
    \lineheight{1}%
    \setlength\tabcolsep{0pt}%
    \put(0,0){\includegraphics[width=\unitlength,page=1]{lquot.pdf}}%
    \put(0.1038419,0.11235805){\makebox(0,0)[lt]{\lineheight{1.25}\smash{\begin{tabular}[t]{l}$s_0$\end{tabular}}}}%
    \put(0.07195887,0.27131422){\makebox(0,0)[lt]{\lineheight{1.25}\smash{\begin{tabular}[t]{l}$a, \epsilon/\sigma$\end{tabular}}}}%
    \put(0.07152408,0.24124226){\makebox(0,0)[lt]{\lineheight{1.25}\smash{\begin{tabular}[t]{l}$a, \sigma/\sigma \sigma$\end{tabular}}}}%
    \put(0.56868002,0.24199963){\makebox(0,0)[lt]{\lineheight{1.25}\smash{\begin{tabular}[t]{l}$a, \sigma/\epsilon$\end{tabular}}}}%
    \put(0.45240176,0.18987489){\makebox(0,0)[lt]{\lineheight{1.25}\smash{\begin{tabular}[t]{l}$a, \sigma/\epsilon$\end{tabular}}}}%
    \put(0.45417634,0.01795173){\makebox(0,0)[lt]{\lineheight{1.25}\smash{\begin{tabular}[t]{l}$a, \sigma/\epsilon$\end{tabular}}}}%
    \put(0.69351569,0.13718611){\makebox(0,0)[lt]{\lineheight{1.25}\smash{\begin{tabular}[t]{l}$\epsilon, \epsilon/\epsilon$\end{tabular}}}}%
    \put(0.40296433,0.26056056){\makebox(0,0)[lt]{\lineheight{1.25}\smash{\begin{tabular}[t]{l}$b$\end{tabular}}}}%
    \put(0.20814402,0.17813371){\makebox(0,0)[lt]{\lineheight{1.25}\smash{\begin{tabular}[t]{l}$b$\end{tabular}}}}%
    \put(0.2034343,0.02754372){\makebox(0,0)[lt]{\lineheight{1.25}\smash{\begin{tabular}[t]{l}$b^{-1}$\end{tabular}}}}%
    \put(0.4030907,0.10012969){\makebox(0,0)[lt]{\lineheight{1.25}\smash{\begin{tabular}[t]{l}$b^{-1}$\end{tabular}}}}%
  \end{picture}%
\endgroup%

%% file: figs/wreath-prod-fsa.pdf_tex
%% Creator: Inkscape 1.0.1 (c497b03c, 2020-09-10), www.inkscape.org
%% Accompanies image file 'wreath-prod-fsa.pdf' (pdf, eps, ps)
\begingroup%
  \makeatletter%
  \providecommand\color[2][]{%
    \errmessage{(Inkscape) Color is used for the text in Inkscape, but the package 'color.sty' is not loaded}%
    \renewcommand\color[2][]{}%
  }%
  \providecommand\transparent[1]{%
    \errmessage{(Inkscape) Transparency is used (non-zero) for the text in Inkscape, but the package 'transparent.sty' is not loaded}%
    \renewcommand\transparent[1]{}%
  }%
  \providecommand\rotatebox[2]{#2}%
  \newcommand*\fsize{\dimexpr\f@size pt\relax}%
  \newcommand*\lineheight[1]{\fontsize{\fsize}{#1\fsize}\selectfont}%
  \ifx\svgwidth\undefined%
    \setlength{\unitlength}{333.18373552bp}%
    \ifx\svgscale\undefined%
      \relax%
    \else%
      \setlength{\unitlength}{\unitlength * \real{\svgscale}}%
    \fi%
  \else%
    \setlength{\unitlength}{\svgwidth}%
  \fi%
  \global\let\svgwidth\undefined%
  \global\let\svgscale\undefined%
  \makeatother%
  \begin{picture}(1,0.59022188)%
    \lineheight{1}%
    \setlength\tabcolsep{0pt}%
    \put(0,0){\includegraphics[width=\unitlength,page=1]{wreath-prod-fsa.pdf}}%
    \put(0.39196824,0.08778941){\color[rgb]{0,0,0}\makebox(0,0)[lt]{\lineheight{1.25}\smash{\begin{tabular}[t]{l}$s_0$\end{tabular}}}}%
    \put(0.48908291,0.11469446){\color[rgb]{0,0,0}\makebox(0,0)[lt]{\lineheight{1.25}\smash{\begin{tabular}[t]{l}$t$\end{tabular}}}}%
    \put(0.74244645,0.09218596){\color[rgb]{0,0,0}\makebox(0,0)[lt]{\lineheight{1.25}\smash{\begin{tabular}[t]{l}$t$\end{tabular}}}}%
    \put(0.30349952,0.11469446){\color[rgb]{0,0,0}\makebox(0,0)[lt]{\lineheight{1.25}\smash{\begin{tabular}[t]{l}$t^{-1}$\end{tabular}}}}%
    \put(0.02969694,0.0944013){\color[rgb]{0,0,0}\makebox(0,0)[lt]{\lineheight{1.25}\smash{\begin{tabular}[t]{l}$t^{-1}$\end{tabular}}}}%
    \put(0,0){\includegraphics[width=\unitlength,page=2]{wreath-prod-fsa.pdf}}%
    \put(0.74244645,0.31728694){\color[rgb]{0,0,0}\makebox(0,0)[lt]{\lineheight{1.25}\smash{\begin{tabular}[t]{l}$c$\end{tabular}}}}%
    \put(0.64790393,0.19123024){\color[rgb]{0,0,0}\makebox(0,0)[lt]{\lineheight{1.25}\smash{\begin{tabular}[t]{l}$c$\end{tabular}}}}%
    \put(0,0){\includegraphics[width=\unitlength,page=3]{wreath-prod-fsa.pdf}}%
    \put(0.39128891,0.26326225){\color[rgb]{0,0,0}\makebox(0,0)[lt]{\lineheight{1.25}\smash{\begin{tabular}[t]{l}$c$\end{tabular}}}}%
    \put(0,0){\includegraphics[width=\unitlength,page=4]{wreath-prod-fsa.pdf}}%
    \put(0.64790393,0.41633122){\color[rgb]{0,0,0}\makebox(0,0)[lt]{\lineheight{1.25}\smash{\begin{tabular}[t]{l}$t^{-1}$\end{tabular}}}}%
    \put(0,0){\includegraphics[width=\unitlength,page=5]{wreath-prod-fsa.pdf}}%
    \put(0.74469724,0.53788563){\color[rgb]{0,0,0}\makebox(0,0)[lt]{\lineheight{1.25}\smash{\begin{tabular}[t]{l}$c, c^{-1}, t^{-1}$\end{tabular}}}}%
    \put(0,0){\includegraphics[width=\unitlength,page=6]{wreath-prod-fsa.pdf}}%
    \put(0.52860044,0.56264544){\color[rgb]{0,0,0}\makebox(0,0)[lt]{\lineheight{1.25}\smash{\begin{tabular}[t]{l}$t$\end{tabular}}}}%
    \put(0,0){\includegraphics[width=\unitlength,page=7]{wreath-prod-fsa.pdf}}%
    \put(0.28631207,0.54460313){\color[rgb]{0,0,0}\makebox(0,0)[lt]{\lineheight{1.25}\smash{\begin{tabular}[t]{l}$t$\end{tabular}}}}%
  \end{picture}%
\endgroup%

%% file: chap/fg-positive-cones.tex
\chapter{Families of groups with finitely generated positive cones}\label{chap: fg}

In this chapter, we construct infinite families of groups which admit finitely generated positive cones. 

\section{Main results}
Although most of the research presented in this thesis focuses on formal languages and positive cones, finitely generated positive cones are desirable as they are easier to understand from a human perspective.  Furthermore, finitely generated positive cones are topologically distinct from regular positive cones as they induce an isolated point on the space of left-orders.\sidenote{Recall that regular positive cones do not do this, for example $\mathbb{Z}^2$ has many regular positive cones but no isolated left-order.}

\begin{defn}[Rank] The rank of a finitely generated semigroup (resp. group) is the smallest size of a generating set needed to generate the semigroup (resp. group). 
\end{defn}

Unfortunately, not many examples of positive cones of finite rank are known. A 2011 paper of Navas \cite{Navas2011} constructs an infinite family of groups given by the presentation $\Gamma_n = \langle a, b \mid ba^nba^{-1} \rangle$ for $n \geq 1$, which have positive cones of rank $2$. This remarkable family contains the Klein bottle group as $\Gamma_1$ and the braid group $B_3$ as $\Gamma_2$. In that same paper, Navas then poses the following problem: for every $k \geq 3$, find an infinite family of groups which admit a positive cone of rank $k$. In the first half of this chapter, we solve the problem by looking into finite-index subgroups of $\Gamma_n$. 

\begin{thm}[See Section \ref{sec: fg-family}]\label{thm: k-gen-P}
	Let $k \geq 3$ be an integer, and $m = k-1$. Let $n = m-1 + mt$, where $t$ is a non-negative integer. Let $\Gamma_n = \langle a,b \mid ba^nb = a \rangle$. Then $a \mapsto 1, b \mapsto 1$ extends to a surjective homomorphism $\vphi_{n,m}: \Gamma_n \to \mathbb{Z}/m\mathbb{Z}$. The family of groups $$\{H_{n,m} := \ker \vphi_{n,m} \mid t \text{ odd}\}$$ is an infinite family of groups with positive cones of rank $k$. 
\end{thm}

The discovery of this particular pattern of groups was first inferred using a computation in GAP (which we discuss in Section \ref{sec: fg-P-GAP}), then proven using the Reidemeister-Schreier method (Section \ref{sec: fg-P-formal-proof}). 

In the second half of this chapter, we apply the methods used to show Theorem \ref{thm: k-gen-P} to the special case of $F_2 \times \mathbb{Z}$ and its finite index subgroups $F_n \times \mathbb{Z}$. 

\begin{thm}[See Section \ref{sec: fg-F_nxZ}]\label{thm: FnxZ-fg-P}
	Let $F_n$ be a free group of rank $n \geq 2$. Then $F_n \times \mathbb{Z}$ has a finitely generated positive cone if and only if $n$ is even. 
\end{thm}

A 2018 result of Clay, Mann and Rivas that shows that $F_2 \times \mathbb{Z}$ has an isolated left-order \cite{ClayMannRivas2018}. A 2019 result of Malicet, Mann, Rivas and Triestino then expands on this result by showing that $F_n \times \mathbb{Z}$ has an isolated left-order if and only if $n$ is even \cite{MalicetMannRivasTriestino2019}. Our theorem is a strenghtening of these results as a finitely generated positive cone implies an isolated order. 

The result for $F_2 \times \mathbb{Z}$ follows from a special case of the proof of Theorem \ref{thm: k-gen-P}, as $F_2 \times \mathbb{Z}$ is a subgroup of $\Gamma_2$. A finitely generated positive cone for $F_2 \times \mathbb{Z}$ was found using the parameters $m=6$, $n=2$ and the map $\varphi_{2,6}: \Gamma_2 \to \mathbb{Z}/6\mathbb{Z}$ sending $a \mapsto 4, b \mapsto 1$. Since these parameters do not fit the restriction of Theorem \ref{thm: k-gen-P}, the positive cone is not of rank $7$. However, by following the statements leading to the proof of that theorem, we can deduce that the rank is bounded by $7$, thus proving the claim.

From then, we show the result for $n > 2$ by realising $F_n \times \mathbb{Z} \leq F_2 \times \mathbb{Z}$ as a subgroup of finite index, using immersions from $F_n$ to $F_2$. Then, we show that an appropriate choice of transversal and generating set gives us yet again another Reidemeister-Schreier generating set which also works as a finite generating set for the positive cone of $F_n \times \mathbb{Z}$ as a subgroup of $F_2 \times \mathbb{Z}$ if and only if $n$ is even. Combining this with the result of \cite{MalicetMannRivasTriestino2019} gives us Theorem \ref{thm: FnxZ-fg-P}.

The results shown here are originally from \cite[Section 5]{Su2020}, where we have significantly extended the exposition. However, the case for $n = 2$ in Theorem \ref{thm: FnxZ-fg-P} has been generalised to $n$ even in this thesis. 

We encourage the reader to consult Chapter \ref{chap: RS} as needed for background on the Reidemeister-Schreier method. We will use the method and refer to the ideas behind it heavily in this chapter. 

\section{Background on $\Gamma_n$}\label{sec: fg-Gamma_n}
Let $$\Gamma_n = \langle a,b \mid ba^nb = a \rangle$$ and consider the family of groups 
$$\{\Gamma_n: \quad n \geq 1 \}.$$
The remarkable fact about this family of groups is the following. 

\begin{thm}
	For $n \geq 1$, the semigroup $P_n = \langle a, b \rangle^+$ is a positive cone for $\Gamma_n$. \cite{Navas2011}
\end{thm} 

We have discussed some of the history this result in the introduction in Chapter \ref{chap: LO} Example \ref{ex: LO-B3}, and we encourage the reader to refer back to it. 

Algebraically, one way to understand Navas' result is to observe that $\Gamma_n$ behaves like a group of fractions. The original proof is quite short and provides insight as to how this family of group behaves. 
\begin{prop}
	Let $n \geq 1$, and $\Delta := a^{n+1}$. Every element $g \in \Gamma_n$ may be written in the form $g = u \Delta^\ell$ for some non-negative word $u$ and $\ell \in \mathbb{Z}$. That is, either $u$ is trivial or $u \in \langle a, b \rangle^+$. \cite{Navas2011}
\end{prop}
\begin{proof}\label{prop: fg-positive-form}
	Notice that $\Delta$ belongs to the center of $\Gamma_n$, since 
	$$a\Delta = a^{n+2} = \Delta a$$ and 
	$$b\Delta = ba^{n+1} = (ab\inv) a = a(b\inv a) = a(a^nb) = a^{n+1}b = \Delta b.$$ 
	Moreover, the relation gives us that $$a\inv = a^n \Delta\inv$$ and $$b\inv = \Delta\inv a^n b a^n.$$ By rewriting every negative letter of $g$ as above and commuting the $\Delta$ elements to the right of the word, we obtain the desired form. 
\end{proof}
Furthermore, by using $ba^n b = a$ we can check that we obtain the following normal forms for the elements of $\Gamma_n$. 

\begin{prop}\label{prop: fg-normal-form}
	Let $g \in \Gamma_n$. Then $g$ has a normal form given by $g = u \Delta^\ell$, where 
	$$u = b^{n_0}a^{m_0}b^{n_1}\dots b^{n_{k-1}} a^{m_k} b^{n_k}$$
	such that 
	\begin{enumerate}
		\item $n_i > 0$ for $0 < i < k, n_0 \geq 0$, $n_k \geq 0$, 
		\item $m_i \in \{1, \dots, n-1\}$ for $0 < i < l$, 
		\item $m_0$ lies in $\{1, \dots, n-1\}$ (resp. $\{1, \dots, n\}$ if $n_0 > 0$ (resp. $n_0 = 0$); similarly, $m_k$ lies in $\{1, \dots, n-1\}$ (resp. $\{1, \dots, n\}$ if $n_k > 0$ (resp. $n_k = 0$), 
		\item $\ell \in \mathbb{Z}$. 
	\end{enumerate} \cite{Navas2011} %
\end{prop}

\begin{rmk}\label{rmk: fg-normalize-alg}
	Proposition \ref{prop: fg-normal-form} can be summarised as the following algorithm. 
	
	Let $g \in \Gamma_n$ and let $w$ be a word representing $g$. While $w$ is not in normal form, do the following. 
	
	\begin{enumerate}
		\item Rewrite the inverse generators as $a\inv \to a^n \Delta\inv$ and $b\inv \to a^n b a^n \Delta\inv$.
		\item Apply the relation and rewrite $ba^n b \to a$. 
		\item Collect the $\Delta$'s by commuting them to the end of the word. 
	\end{enumerate}
	
	In the course of doing research to produce the results in this chapter, we have indeed coded this algorithm in GAP. It is available in the Chapter \ref{chap: fg-code} of the Appendix. 
\end{rmk}

Geometrically, Navas' result can be summarised as stating that $\{a,b\}$ span a half-plane positive cone decomposition of the Cayley graph for every group in the family $\Gamma_n$. In particular, for $n=1$, the Cayley graph of $\Gamma_1$ is given in Figure \ref{fig: P-K2}, where $\Gamma_1 = K_2$ and the positive elements have been highlighted (in green). For $n=2$, Figure \ref{fig: fg-n1-n2-diff-P} illustrates the necessity for a higher-dimensional Cayley graph when passing from $n = 1$ to $n = 2$, while preserving the property that ``right-hand-side'' of the Cayley graph forms the positive cone for such a family of groups. Navas aptly describes this new structure as ``essentially a product of $\mathbb{Z}^2$ by a dyadic rooted tree'' \cite{Navas2011}.

\begin{figure}[h]
\centering
{
\includegraphics[width = \textwidth]{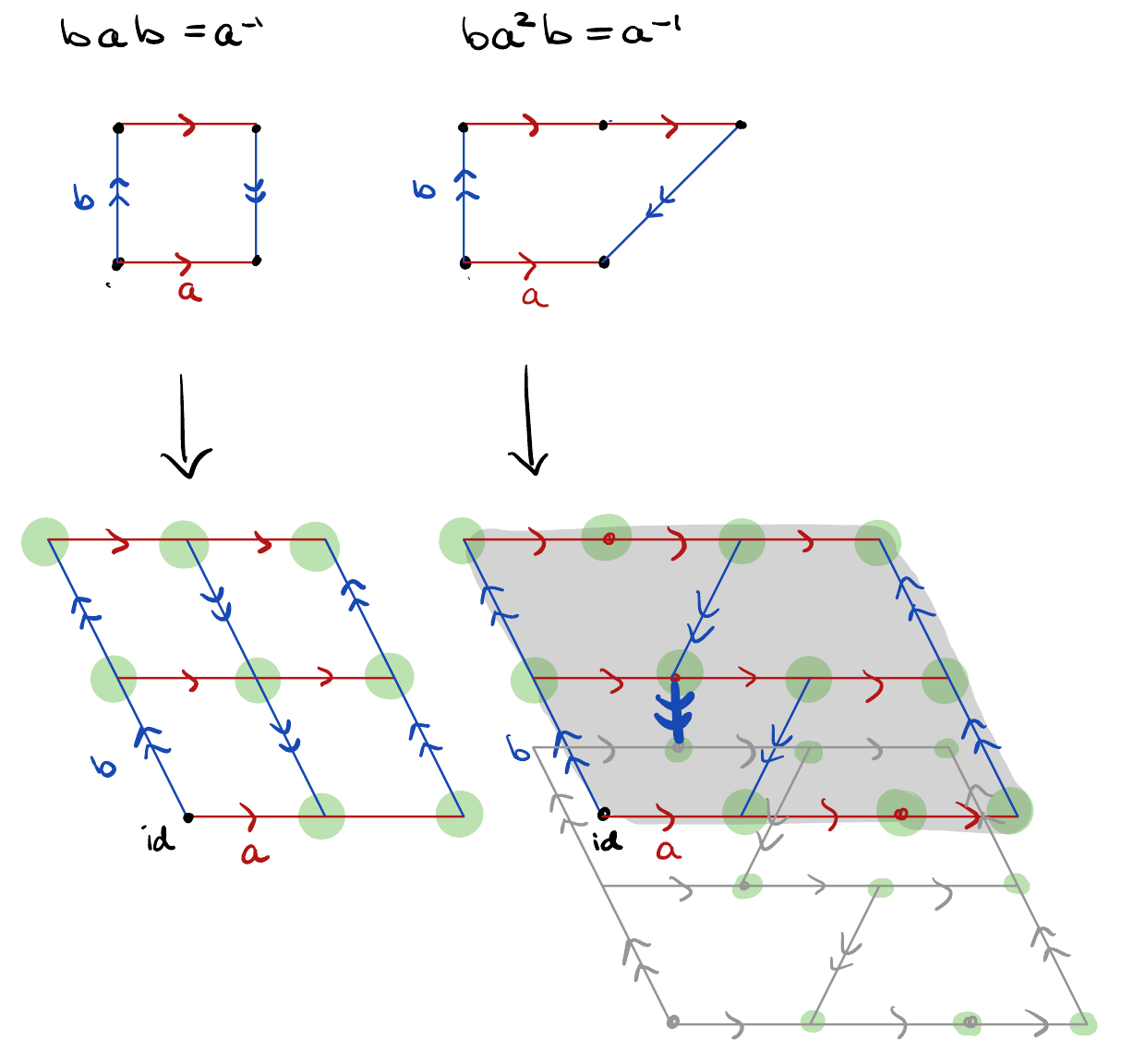}
}
\caption{At the top, two cells representing the relations of $bab = a\inv$, $ba^2b = a\inv$ of $\Gamma_1$ and $\Gamma_2$ respectively. At the bottom, a portion of the induced Cayley graphs, with emphasis on the induced sheet that $\Gamma_2$ inherits due to $n > 1$. The positive elements obtained by following the direction of the arrows are highlighted in both cases (in green) and roughly span the right half-plane(s). 
}
\label{fig: fg-n1-n2-diff-P}
\end{figure}

The following lemma can be interpreted as capturing the importance of the $a$-axis in the positivity of an element of $\Gamma_n$. It will be key in the proof of both our main theorems for this chapter. 

\begin{lem}\label{lem: in-P}
	For all integer $n \geq 1$, and any $p,q \in \mathbb{Z}$, the element $b^{p}ab^q$ belongs to $P_n = \langle a, b \rangle^+$.
\end{lem}
\begin{proof} We divide the proof into four cases.
	\begin{description}
		\item[Case 1: $p \geq 0, q \geq 0$.] 	The element $b^{p}ab^q$ is a positive word.	

		\item[Case 2: $p \leq 0, q \geq 0$.] 	Let $p = -s$ such that $s \geq 0$. Then $b^{p}ab^q = b^{-s} a \cdot b^q$, and it suffices to show that $b^{-s} a \in P_n$. 
			
			We claim by induction on $s$ that $b^{-s}a = a(a^{n-1}b)^s$. If $s=0$, then $$b^0a = a = a(a^{n-1}b)^0.$$
	
			As for the $s \imp s+1$ case, 
			\begin{align*}
				b^{-(s+1)}a &= b^{-s} b^{-1} a \\
				&= b^{-s} (a a^{-1}) b^{-1} a \\ 
				&= (b^{-s} a) a^{-1} (b^{-1} a), \qquad b^{-1}a = a^n b \\
				&= a(a^{n-1}b)^s \cdot a^{-1} a^{n} b \\
				&= a(a^{n-1}b)^s \cdot a^{n-1} b \\
				&= a(a^{n-1}b)^{s+1}. 
			\end{align*}
			This shows that $b^{-s}a \in P_n$ for any $s \geq 0$.
		
		\item[Case 3: $p  \geq 0, q \leq 0$.] Let $q = -t$ such that $t \geq 0$. Then $b^{p}ab^q = b^p \cdot a b^{-t}$ and it suffices to show that $a b^{-t} \in P_n$. 
				
				We claim by induction that $ab^{-t} = (ba^{n-1})^t a$. If $t = 0$, then 
				$$ab^0 = a = (ba^{n-1})^0 a.$$
				
				As for the $t \implies t + 1$ case, 
				\begin{align*}
					ab^{-(t+1)} &= ab^{-t} b^{-1} \\
					&= (ba^{n-1})^t a \cdot b^{-1} \\
					&= (ba^{n-1})^t (a b^{-1}), \qquad ab^{-1} = ba^n \\
					&= (ba^{n-1})^t \cdot ba^n \\
					&= (ba^{n-1})^t \cdot ba^{n-1} \cdot a \\
					&= (ba^{n-1})^{t+1} a. 
				\end{align*}
		
			This shows that $ab^{-t} \in P_n$ for any $t \geq 0$. 
	
	\item[Case 4: $p < 0, q < 0$.] Let $p = -s, q = -t$ such that $s,t > 0$. 
		Then, using the induction results of the two previous cases we have
		\begin{align*}
			b^{p}ab^q &= (b^{-s} a) b^{-t} \\ 
			&= a(a^{n-1}b)^s b^{-t} \\
			&= a(a^{n-1}b)^{s-1} \cdot a^{n-1}b \cdot  b^{-t} \\
			&= a(a^{n-1}b)^{s-1} \cdot (a^{n-1} b^{-t+1})
		\end{align*}
		Suppose first that $n = 1$. Then, 
		\begin{align*}
			b^{p}ab^q &= a(a^{0}b)^{s-1} \cdot (a^{0} b^{-t+1}) \\
			&= ab^{s-1} \cdot b^{-(t-1)} \\
			&= ab^{s-t}
		\end{align*}
		There are two subcases. 
		
		If $s - t \geq 0$, then $ab^{s-t}$ is a positive word. Otherwise, if $s - t < 0$ then using the relation $ab^{-1} = ba$, we have that $b^{-(s-t)}a$ is a positive word. 
		
		Let us now assume that $n \geq 2$. Then,
		\begin{align*}
			b^{p}ab^q &= a(a^{n-1}b)^{s-1} \cdot a^{n-2} \cdot a b^{-(t-1)} \\
			&= a(a^{n-1}b)^{s-1} \cdot a^{n-2} \cdot (ba^{n-1})^{t-1} a.
		\end{align*}

		Since we assumed that $n \geq 2$, $s, t \geq 1$, the obtained word is positive.
	\end{description}
	
	We have now shown that for all $p, q \in \mathbb{Z}$, $b^p a b^q$ can be rewritten as a positive word and therefore that $b^p a b^q \in P_n$ as claimed.
\end{proof} 

\section{Constructing subgroups with finitely generated positive cones}
\label{sec: fg-family}
In this first half of the chapter, we will look at how to ``fish'' for subgroups with finitely generated positive cones of higher rank within $\Gamma_n$. 
\subsection{Dividing Cayley graphs along kernels}

\begin{figure}[h]
\centering
{
\includegraphics[width = \textwidth]{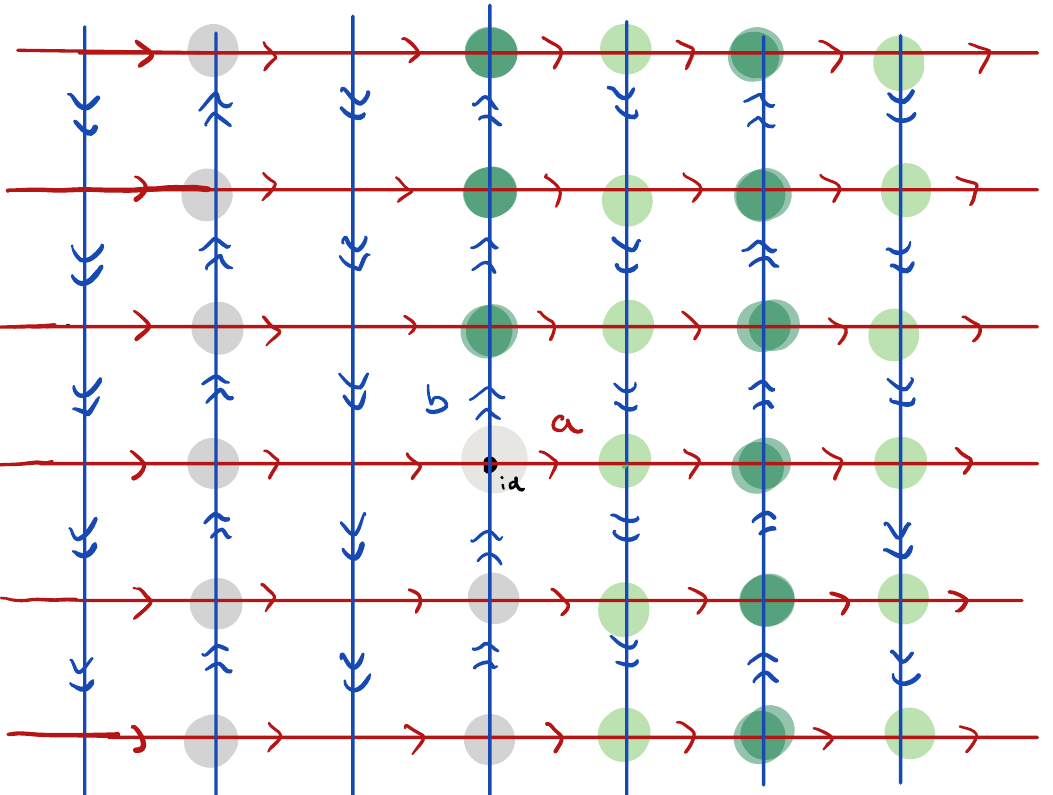}
}
\caption{The Cayley graph of $K_2$ under presentation $\langle a, b \mid bab = a\rangle$. The group has $\bZ^2$ as subgroup of index $2$, and the elements belonging to $\bZ^2$ are marked by the darker dots (in grey and dark green). For $K_2$, the direction of the $b$-arrows are alternating every other column, whereas they stay in the same direction for $\bZ^2$. The alternating behaviour provides the ``twist'' that is needed the semigroup $\langle a,b \rangle^+$ to span a positive cone (in green).  
}
\label{fig: P-K2-Zsq}
\end{figure}

We have seen in Chapter \ref{chap: LO} that $\mathbb{Z}^2$ is a subgroup of index $2$ in $\Gamma_1$ which does not inherit a finitely generated positive cone.
Looking closely at the Cayley graph of $\Gamma_1$ and its $\mathbb{Z}^2$ embedding (see Figure \ref{fig: P-K2-Zsq}), this seems to be due to losing the ``twist'' action of $a$ over $b$ when passing to $\bZ^2$. Indeed, we observe that $\langle a,b \rangle^+$ can roughly span the right half-plane of the Cayley graph (the positive cone) thanks to the alternating direction of the $b$-axis, whereas this is not possible to do with $\langle a^2,b \rangle^+$) due to the $b$-axis always pointing in the same direction. 

Formally, the relation in $\Gamma_1$ of $1 = baba\inv \iff aba\inv = b\inv$ gives $\Gamma_1$ an inner semidirect product structure, $\Gamma_1 \isom \langle a \rangle \ltimes_\theta \langle b \rangle$, with $\theta_a(b) = aba\inv = b\inv$. On the other hand $\bZ^2$ has a Cartesian product structure given by $\mathbb{Z}^2 \cong \langle a^2 \rangle \times \langle b \rangle$, due to the action $\theta_{a^2}(b) = a^2 b a^{-2} = ab\inv a\inv = b$. Thus, the $a$-axis generator $a^2$ has no twist action in $\mathbb{Z}^2$, as expected. 

The solution then appears to be to take a finite index subgroup along the $b$-axis of the Cayley graph instead of the $a$-axis which yields $\bZ^2$. For example, the subgroup $H = \langle a,b^2\rangle$ of index $2$ inherits a finitely generator positive cone $\langle a,b^2 \rangle^+$ from $P_1$ that is of rank $2$ (see Figure \ref{fig: P-K2-Zsq}).

Mathematically, both $\mathbb{Z}^2$ and $H$ can be realised as kernels of maps of the form $\vphi: \Gamma_1 \to \bZ/2 \bZ$, where $\vphi(baba\inv) \equiv 2\vphi(b) + 0\vphi(a) \equiv 0 \mod 2$. We observe that $\bZ^2 = \ker \vphi$ if $\vphi(a) = 1, \vphi(b) = 0$, and $H = \ker \vphi$ for $\vphi(a) = 0, \vphi(b) = 1$. Thus, we can formalise dividing by either axes as taking the kernel of $\vphi$ and sending either $a \mapsto 1$ or $b \mapsto 1$.

Now, since the first goal of this chapter is to find finite index subgroups with positive cones of higher rank, we can generalise looking at groups along the $b$ axis for $\Gamma_n$ with $n > 1$, and along $m$, the parameter of the map $\vphi: \Gamma_n \to \bZ/m \bZ$ (in the case of $\Gamma_1$, this would boil down to taking an element every $m$ point along $b$-axis the Cayley graph instead of every other point like we did for $H$ in the previous paragraph). For clarity, we will denote by $\vphi_{n,m} := \vphi$ and $H_{n,m} := H = \ker \vphi, \quad \vphi(b) = 1$ and drop the subindices when the context is clear. 

The kernel equation for $\vphi$ becomes $\vphi_{n,m} = \vphi(ba^nba\inv) \equiv 2\vphi(b) + (n-1)\vphi(a) \equiv 0 \mod m$. Notice that the $\Gamma_1$ case is essentially degenerate, since there are no restriction on $\vphi(a)$ in the kernel equation, but that the coefficient of $\vphi(a)$ becomes non-zero as soon as $n > 1$. 

Visually, what does would a subgroup $H_{n,m}$ look like for $\Gamma_n$ with $m > 1, n > 1$? For $n > 1$, the relation $ba^nb = a$ induces ``sheets'' in the Cayley graph (see Figure \ref{fig: fg-n1-n2-diff-P}), making it a higher dimensional case with more convoluted geometry which has the hope of giving us positive cones that are finitely generated with more than $2$ generators. 

Since the geometry is convoluted to explore visually at higher dimensions, let us explore it numerically in the next section. The formal statements and proofs are in Section \ref{sec: fg-P-formal-proof}, which the reader is welcome to skip ahead to.

\subsection{Numerical experiments on $\Gamma_n$ using GAP}\label{sec: fg-P-GAP}
The numerical experiment is set up as follows. Take the kernel equation
$\vphi(ba^nba\inv) \equiv 2\vphi(b) + (n-1)\vphi(a) \equiv 0 \mod m, \quad \vphi(b) = 1$
which essentially becomes 
$$(n-1)\mu \equiv -2 \mod m, \quad \mu := \vphi(a)$$
and, when it is possible, find a solution for $\mu$. From modular arithmetic, we know there always exists a solution when $\gcd(n-1,m) = 1$ since it is possible to compute $\frac{1}{n-1}$, and, in particular, $\frac{-2}{n-1}$. 

We use GAP to compute the kernel $\ker \vphi$ for fixed $n,m$. It is not necessarily easy to deduce a priori the rank of a positive cone for this subgroup, but we can use the following fact to get a lower bound on the rank of the positive cone of $\ker \vphi$. 

\begin{lem}
\label{rk}
\label{lem: fg-rk}
	Let $\rk(S)$ denote the rank of a semigroup $S$ (resp. group). Then, if $G$ is a group with a positive cone $P \subset G$ that is finitely generated by $X \subset G$ and $\Ab(G)$ is the abelianization of $G$, we have 
	$$|X| \geq \rk(P) \geq \rk(G) \geq \Ab(G).$$
\end{lem}
\begin{proof}
	Let $X = \{x_1, \dots, x_n\}$ and $P = \langle X \rangle^+$. Since $G = P \sqcup P\inv \sqcup \{1_G\}$, $G$ is generated by $X$. If $\Phi: G \to \Ab(G)$ is the map from $G$ to its abelianization, then $\Phi(X)$ generates $\Ab(G)$. Thus, we have 
	$$|X| \geq \rk(P) \geq \rk(G) \geq \rk(\Ab(G)),$$
	as claimed. 
\end{proof}

Bounding the rank of the positive cone by the rank of the abelianization is a good way to obtain preliminary information, as the abelianization of a group and its rank are fast to compute.  

The following code runs in GAP 4.10.0. For those already familiar with GAP, we have tried to make the code self-explanatory via the comments which are preceded by the \texttt{\#} symbol (in green). 

\begin{lstlisting}[label={code: fg-experiment}, language = GAP, breaklines]
# Input: integers n,m. Output: if n != 1 and m > 1 and gcd(n-1,m) = 1 then returns kernel of phi , else return false. 
ker_phi := function(n,m) 
	local F, a, b, Gamma_n, x, C, is_relatively_prime, mu, phi, ker_phi;

	# Construct gamma_n with generators a and b. 
	F := FreeGroup("a", "b");
	a := F.1;; b := F.2;;
	Gamma_n := F/[b*a^n*b*a^-1];
	a := Gamma_n.1;; b := Gamma_n.2;;

	# Construct cyclic group c of order m with generator x. 
	F := FreeGroup("x");
	x := F.1;;
	C := F/[x^m];
	x := C.1;;

	# If n != 1 and m > 1 and gcd(n-1,m) = 1 then returns kernel of phi , else return false. 
	is_relatively_prime := GcdInt(n-1,m) = 1;
	if is_relatively_prime and (not (n = 1)) and (m > 1) then 
		mu := -2/(n-1) mod m; 
		phi := GroupHomomorphismByImages(Gamma_n,C,[a,b],[x^mu,x]);
		ker_phi := Kernel(phi);
		return ker_phi;
	else
		return false;
	fi;
end; 

# Input: a group of fp type. Output: the abelianized group. 
abelianization := function(H)
	local iso, im_H, D, ab_H;
	iso := IsomorphismFpGroup(H); # GAP must deduce a finite presentation for G, otherwise it refuses to do the computation. 
	im_H := Image(iso);
	D := DerivedSubgroup(im_H); 
	ab_H := im_H/D; 
	return ab_H;
end;

# Input: abelian group. Output: rank of abelian group. 
rank := function(ab_H)
	local X, rk;
	X := GeneratorsOfGroup(ab_H);
	rk := Size(X);
	return rk;
end; 

# Input [n list] x [m list]. Output: nxm matrix with abelian ranks. 
compute_abelian_ranks := function(n_range, m_range)
	local n,m, H, rk, data, row;
	data := [];
	for n in n_range do	
		row := [];
		for m in m_range do 
			H := ker_phi(n,m);
			if not (H = false) then
				rk := rank(abelianization(H));
			else
				rk := -1; #If it is not possible to compute the abelianisation, then return -1 as a way of showing that. This is so all the outputs can be neatly stored in an integer matrix. 
			fi;
			Add(row, rk);
		od;
		Add(data, row);
	od;
	Display(data);
	return data;;
end; 	

# Input [n list] x m. Output list of list with structure description for H_{n,m} where n varies. 
compute_abelian_kernels := function(n_range, m)
	local data, n, row, H, ab_H, struct_ab_H;
	data := [];
	for n in n_range do
		row := [n];
		H := ker_phi(n,m);
		if not (H = false) then
			ab_H := abelianization(H);
			struct_ab_H := StructureDescription(ab_H);
		else
			struct_ab_H := "FAIL"; 
		fi;
		Add(row, struct_ab_H);
		Print(row, "\n");
		Add(data, row);
	od;
	return data;
end;
\end{lstlisting}

We show an example for $n=2, m=3$.\sidenote{Note that in accordance to our setup, \texttt{ker\_phi} only returns a group if $\gcd(n-1,m) =1$ and $n \not= 1$ and $m > 1$. To that end, we have set up \texttt{ker\_phi} to return \texttt{false} if any of these conditions are not met, as specified in our code. Attempting to use \texttt{abelianization} and \texttt{rank} on that input will result in an error.}

\begin{lstlisting}[language=bash]
gap> n := 2;; m :=3 ;;
gap> H := ker_phi(n,m);
Group(<fp, no generators known>)
gap> AbH := abelianization(H); 
Group([ f1*f2*f3^-1, f3, f2^-1*f3 ])
gap> rank(AbH);
3
\end{lstlisting}

While an individual data points $(n,m)$ does not give much information, the magic of doing this computation via GAP is that we can easily query over many such pairs $(n,m)$ using the \texttt{compute\_abelian\_ranks} function. In that case, there will be certain pairs $(n,m)$ for which we cannot compute the abelian rank per our setup. To resolve this, we set the value of the rank to be $-1$. This is done for practicality, so that the value can still be entered in an integer matrix to be displayed via GAP's \texttt{Display} function (on line 64, inside the \texttt{compute\_abelian\_rank} function). 

For example, this is how we would compute a table of all the values of $\rk(H_{n,m})$ for $n \in \{1,\dots, 20\}, m \in \{1,\dots,10\}$.

\begin{lstlisting}[language=bash]
gap> n_range := [1..20];; m_range := [1..10];;
gap> data := compute_abelian_ranks(n_range, m_range);; 
\end{lstlisting}

\begin{verbatim}
[ [  -1,  -1,  -1,  -1,  -1,  -1,  -1,  -1,  -1,  -1 ],
  [  -1,   2,   3,   2,   2,   3,   2,   2,   3,   2 ],
  [  -1,  -1,   2,  -1,   2,  -1,   2,  -1,   2,  -1 ],
  [  -1,   2,  -1,   3,   5,  -1,   2,   2,  -1,   5 ],
  [  -1,  -1,   4,  -1,   2,  -1,   2,  -1,   4,  -1 ],
  [  -1,   2,   2,   2,  -1,   3,   7,   2,   2,  -1 ],
  [  -1,  -1,  -1,  -1,   3,  -1,   2,  -1,  -1,  -1 ],
  [  -1,   2,   4,   3,   2,   4,  -1,   3,   9,   3 ],
  [  -1,  -1,   2,  -1,   6,  -1,   2,  -1,   2,  -1 ],
  [  -1,   2,  -1,   3,   3,  -1,   2,   2,  -1,   3 ],
  [  -1,  -1,   4,  -1,  -1,  -1,   3,  -1,   4,  -1 ],
  [  -1,   2,   2,   3,   2,   3,   3,   3,   2,   3 ],
  [  -1,  -1,  -1,  -1,   2,  -1,   8,  -1,  -1,  -1 ],
  [  -1,   2,   4,   3,   6,   4,   2,   3,   4,   6 ],
  [  -1,  -1,   2,  -1,   3,  -1,  -1,  -1,   3,  -1 ],
  [  -1,   2,  -1,   3,  -1,  -1,   3,   4,  -1,  -1 ],
  [  -1,  -1,   4,  -1,   3,  -1,   2,  -1,  10,  -1 ],
  [  -1,   2,   2,   3,   2,   3,   3,   3,   2,   5 ],
  [  -1,  -1,  -1,  -1,   6,  -1,   3,  -1,  -1,  -1 ],
  [  -1,   2,   4,   3,   3,   4,   8,   3,   4,   4 ] ]
\end{verbatim}

The data, stored in a matrix $M$, is as follows. 
$$M_{n,m} := \begin{cases}
	\rk(\Ab(H_{n,m})) & \gcd(n-1,m) = 1 \\
	-1 & \gcd(n-1,m) \not= 1 
\end{cases}, \quad n \in \{1,\dots, 20\}, m \in \{1,\dots,10\}.$$

In the 3rd, 5th, 7th, and 9th columns corresponding to $m = 3, 5, 7$ and $9$ respectively we can observe that there are locally maximal abelian ranks of $m+1$ that seem to repeat cyclically at $n$-intervals linearly dependent on $m$. Further computation using a larger ranges of $n$ and $m$ seem to corroborate this. 

Another straightforward-but-informative computation we can do is to probe what the structure of these abelianised kernels look like. Using the \texttt{compute\_abelian\_kernels} function, we may view what the abelian kernels $H_{n,m}$ look for a fixed $m$.  For example, this is how we would get the data for $m=3$. 

\begin{lstlisting}[language=bash]
gap> n_range := [1..20];; m :=3;;
gap> data := compute_abelian_kernels(n_range, 3);; 
\end{lstlisting}

\begin{verbatim}
[ 1, "FAIL" ]
[ 2, "C0 x C2 x C2" ]
[ 3, "C0 x C2" ]
[ 4, "FAIL" ]
[ 5, "C0 x C2 x C2 x C2" ]
[ 6, "C0" ]
[ 7, "FAIL" ]
[ 8, "C0 x C2 x C2" ]
[ 9, "C0 x C2" ]
[ 10, "FAIL" ]
[ 11, "C0 x C2 x C2 x C2" ]
[ 12, "C0" ]
[ 13, "FAIL" ]
[ 14, "C0 x C2 x C2" ]
[ 15, "C0 x C2" ]
[ 16, "FAIL" ]
[ 17, "C0 x C2 x C2 x C2" ]
[ 18, "C0" ]
[ 19, "FAIL" ]
[ 20, "C0 x C2 x C2" ]
\end{verbatim}

The first entry is of each list is the value of $n$ for $\Ab(H_{n,3})$, and the second entry is the structure description of $\Ab(H_{n,3})$. The value \texttt{Ck} corresponds in GAP to a cyclic group of order $k$ (for example, \texttt{C0} is trivial) when $\gcd(n-1,m) = 1$, whereas the "FAIL" value correspond to the case where $\gcd(n-1,m) \not= 1$.  

This pattern seems fairly cyclic. Upon further computation, other values of $m$ also yields $\Ab(H_{n,m})$ as copies of cyclic groups of order $2$. This gives us the indication that we may be able to find a general formal statement describing this behaviour that is fairly straightforward.

Now all that remains to do is to do the computation of obtaining these kernel subgroups formally. Luckily, the Reidemeister-Schreier method is the exact tool we need to do this. 

\subsection{Proving the numerical observations formally}\label{sec: fg-P-formal-proof}
The proof of the following proposition will rely on the Reidemeister-Schreier method, which we covered in Chapter \ref{chap: RS}. 

\begin{prop}\label{prop: fg-pos-cone-gen}
	Let $n,m$ and $\mu$ be such that $(n-1)\mu \equiv -2 \mod m$. Let $\vphi: \Gamma_n \to \mathbb{Z}/m \mathbb{Z}$ be a homomorphism such that $\vphi(a) = \mu$ and $\vphi(b) = 1$. Let $H = H_{n,m} := \ker \vphi$ and let $P = \langle a , b \rangle^+$ be a positive cone for $\Gamma_n$. Then $H \cap P$ admits the finite generating set $Y$, where $$Y := \{b^{-s}ab^{s+(m-\mu)}\}_{s=0}^{\mu-1} \cup \{b^{-s}ab^{s-\mu}\}_{s=\mu}^{m-1} \union \{b^m\}.$$ 
	\cite{Su2020}
\end{prop} 

\begin{proof}
	We will start by showing that $Y$ is a generating set for $H$ using the Reidemeister-Schreier method. Let $\phi: F_2 \to \Gamma_n$ be the canonical map from a free group on two elements onto $\Gamma_n$. The set $T = \{b^0, b^{-1}, \dots, b^{-(m-1)}\}$ is a Schreier transversal  for $\tilde H := \phi \inv (H)$, since the restriction to $T$ of $\vphi \phi: F_2 \to \mathbb{Z}/m\mathbb{Z}$ is bijective. Let $\gamma, \gamma^*$ and $\overline{\cdot}$ be defined as in the statement of the Reidemeister-Schreier method (Theorem \ref{thm: RS}). Then, a generating set for $H$ is given by $Y = \{ \gamma(t,x)^*: t \in T$, $x \in \{a,b\}, \gamma(t,x)^*\not=1 \}$. Identifying $\gamma(t,x)$ with $\gamma(t,x)^*$,

	\begin{align*}
		\gamma(b^{-s},a) = b^{-s}a (\overline{b^{-s}a})^{-1}
		&= \begin{cases} 
			b^{-s}ab^{-s + \mu} & \text{ if } 0 \leq \mu \leq s  \\
			b^{-s}ab^{m-(-s + \mu)} & \text{ if } s + 1 \leq \mu \leq m-1
		\end{cases} \\
	\gamma(b^{-s},b) = b^{-s}b (\overline{b^{-s}b})^{-1}
	&= \begin{cases} 
		1 & \text{ if } 1 \leq s \leq m-1 \\
		b^m & \text{ if } s = 0. \end{cases} 
	\end{align*}

We will now show that $\langle Y \rangle^+ \subseteq H \cap P$.
We have already shown that $Y$ generates $H$, and by Lemma \ref{lem: in-P} that $Y \subseteq P$. Thus $Y \subseteq H \cap P$. Since $H \cap P$ is a semigroup, $\langle Y \rangle^+ \subseteq H \cap P$. 

To show that $H \cap P \subseteq \langle Y \rangle^+$, let $\pi: \{a,a\inv,b,b\inv\}^* \to \Gamma_n$ be the standard evaluation map. We will show that for every word $w \in \langle a,b \rangle^+$ whose image $\pi(w)$ is in $H$, there is a corresponding word $v \in \langle Y \rangle^+$ such that $\pi( v) = \pi(w)$. 

Write $w = x_1 \dots x_{\ell}$ and let $w_i = x_1 \dots x_i$. Recall the rewriting map $\tau: F_2 \to F(Y)$ from the Reidemeister-Schreier method (Theorem \ref{thm: RS}), where $F(Y)$ stands for the free group with basis $Y$. Since $\pi( w) \in H$, $\tau(w)$ is well-defined and $\pi(\tau(w)) = \pi( w)$ by construction of $\tau$. Furthermore,
$$\tau(w) = \prod_{i=1}^\ell \gamma(\ovl{w_{i-1}},x_i)^*.$$
Since $w \in \langle a,b \rangle^+$, $x_i \in \{a,b\}$ for $1 \leq i \leq \ell$. Thus $\gamma(\ovl{w_{i-1}}, x_i)^* \in Y$ for $1 \leq i \leq \ell$ and $\tau(w) \in \langle Y \rangle^+$. 

This shows that $H \cap P = \langle Y \rangle^+$.

\end{proof}

\begin{cor}
Let $n,m$ and $\vphi$ be such that $(n-1)\vphi(a) \equiv -2 \mod m$. Let $H = H_{n,m} = \ker \vphi$ as in Proposition \ref{prop: fg-pos-cone-gen}, and let $P = P_n := \langle a , b \rangle^+$. Then the rank of $H \cap P$ is at most $m+1$. 
\end{cor} 
\begin{proof}
	Take $Y$ from Proposition \ref{prop: fg-pos-cone-gen}. Since it is a generating set for $H \cap P$, we have that $\rk(H \cap P) \leq |Y| = m+1$ by Lemma \ref{lem: fg-rk}. 
\end{proof} 

So far, we have shown that for certain integers pairs $n \geq 2$ and $m \geq 2$, there exists a homomorphism $\vphi: \Gamma_n \to \mathbb{Z}/m\mathbb{Z}$ which creates a subgroup $H := \ker \vphi$ of index $m$ which admits a positive cone with at most $m+1$ generators. In the sequel, we will show that for every fixed $m$, it is possible to pick an infinite family of $\Gamma_n$'s satisfying a certain criterion on $n$ such that the positive cone of the subgroups $H_{n,m}$ has a minimal number of generators that is exactly $m+1$. To aid our proof, we will use the following lemma.

\begin{prop}\label{H-pres}
Let $n = m-1 + mt$, where $t$ is a non-negative integer. Let $\Gamma_n = \langle a,b \mid ba^nb = a \rangle$. Then $a \mapsto 1, b \mapsto 1$ extends to a surjective homomorphism $\vphi: \Gamma_n \to \mathbb{Z}/m\mathbb{Z}$. The group $H = H_{n,m} := \ker \vphi$ is a subgroup of $\Gamma_n$ of index $m$ admitting a presentation on $m+1$ generators and $m$ relators, $H_{n,m} = \langle x_0, \dots, x_m \mid S_{n,m} \rangle$ where 
$$S_{n,m} = \{ x_i x_m^{t+1} x_i \mid i = 0, \dots, m-2 \} \cup \{x_{m-1}x_m^t x_{m-1}x_m\inv\}.$$ 
Furthermore, we may embed the generators of $H_{n,m}$ into $\Gamma_n$ by sending $x_i \mapsto a^iba^{-i-1}$ for $i = 0, \dots, m-2$, $x_{m-1} \mapsto a^{m-1}b$, and $ x_m \mapsto a^m$. 
\end{prop}

\begin{proof}
	Let $\phi: F_2 \to \Gamma_n$ be the canonical map sending reduced words to group elements. We will be using the Reidemeister-Schreier method again to derive a presentation for $H$, this time with choice of transversal $T = \{1, a, a^2, \dots, a^{m-1}\}$ as a special case of Proposition \ref{prop: fg-pos-cone-gen} for $\mu = 1$.\sidenote{We sincerely thank the anonymous peer reviewer of the original paper this result was featured in for suggesting this transversal to significantly simplify our computations.} Our transversal $T$ is a Schreier transversal since the restriction to $T$ of $\vphi \phi: F_2 \to \mathbb{Z}/m\mathbb{Z}$ is also bijective as now $a \mapsto 1$. Recall the functions $\gamma^*, \gamma$ and $\tau$ from the statement of the Reidemeister-Schreier method (Theorem \ref{thm: RS}). We know that $H$ is generated by $\{\gamma(t,x)^* :  t \in T, x \in \{a,b\}, \gamma(t,x)^*  \not= 1\}$.

	Now,
	\begin{align}\label{eq: gamma(a^s: a)}
		\gamma(a^s, a)^* &= \begin{cases} 
		1 & 0 \leq s \leq m-2 \\
		a^m & s = m-1
		\end{cases} 
		\\
		\gamma(a^s, b)^* &= \begin{cases}
		a^s b a^{-(s+1)} & 0 \leq s \leq m-2 \\
		a^{m-1} b & s = m-1.
		\end{cases}
	\end{align}

	Therefore, by identifying $x_s := \gamma(a^s,b)$ for $s = 0, \dots, m-2$, $x_{m-1} := \gamma(a^{m-1}, b)$  and $x_m := \gamma(a^{m-1},a)$, the set $\{x_0, \dots, x_m\}$ generates $H$.
	
	To compute the relators of $H$, recall that for a word $w := y_1\dots y_\ell$ with prefixes $w_i := y_1\dots y_i$, the rewriting function $\tau$ send $w$ to $\tau(w) = \prod_{i=1}^{\ell} \gamma(\overline{w_{i-1}}, y_{i})$, where we use the convention that the factors of the product are multiplied on the right-hand side. We separate the computation of the relators $\tau(w)$ in two cases. Note that in both cases, the main trick we will use is that $\overline{a^{m + i}} = a^i$ for $0 \leq i \leq m-1$. 
		
	\begin{description}
	   \item[Case 1: $w = a^sba^{n}ba^{-(s+1)}$, where $s = 0, \dots, m-2$.] Then,
		
		\begin{align*}
			\tau(a^sba^{n}ba^{-(s+1)}) = \underbrace{\left(\prod_{i=0}^{s-1} \gamma(a^i, a) \right)}_{=1} \gamma(a^s, b) \left( \prod_{i=s+1}^{s+n} \gamma(\overline{a^i}, a) \right)  \gamma(\overline{a^{s + n + 1}},b) \underbrace{\left( \prod_{i=0}^{s}\gamma(\overline{a^{s+n+2-i}}, a\inv) \right)}_{=1}.
		\end{align*}

		We have already established from Equation \ref{eq: gamma(a^s: a)} that the first factor $\prod_{i=0}^{s-1} \gamma(a^i, a) = 1$. We can similarly eliminate the last factor by substituting $n$ with $m-1+mt$, as $\ovl{a^{s+n+2-i}} = \ovl{a^{s+(m-1 + mt)+2-i}} = \ovl{a^{s+1-i} \cdot a^{m + mt}} = \ovl{a^{s+1-i}}$. This gives us that $\gamma(\overline{a^{s+n+2-i}}, a\inv) = \gamma(\ovl{a^{s+1-i}}, a\inv) = (\overline{a^{s+1-i}} \cdot a\inv)(\overline{a^{s+1-i} \cdot a\inv})\inv = a^{s+1-i} a\inv \cdot a^{s-i} = 1$.
	
		 We are left with 
		
		\begin{align*}
			\tau(a^sba^{n}ba^{-(s+1)}) &= \gamma(a^s, b) \left( \prod_{i=s+1}^{n+s} \gamma(\overline{a^i}, a) \right)  \gamma(\overline{a^{s +m(t+1)}},b) \\
			&= x_s \left( \prod_{i=s+1}^{m-1+mt+s} \gamma(\overline{a^i}, a) \right) x_s .
		\end{align*}
				
		In order to simplify the middle factor of the right-hand side, we claim by induction on $t$ that 
		$\prod_{i=s+1}^{m-1+mt+s} \gamma(\overline{a^i}, a) = x_m^{t+1}.$
		The base case $t=0$ gives us
		\begin{align*}
			\prod_{i=s+1}^{m-1+s} \gamma(\overline{a^i}, a) &=\underbrace{\gamma(a^{s+1},a) \gamma(a^{s+2}, a) \dots \gamma(a^{m-2},a)}_{=1} \underbrace{\gamma(a^{m-1},a)}_{=x_m} \underbrace{\gamma(\overline{a^m}, a) \dots \gamma(\overline{a^{m-1+s}},a)}_{=1} \\
			&= x_m.
			\end{align*} 
			Assuming the hypothesis, 
			\begin{align*}
			\prod_{i=s+1}^{m-1+mt+s} \gamma(\overline{a^i}, a) &= \left(\prod_{i=s+1}^{m-1+m(t-1)+s} \gamma(\overline{a^i}, a)\right)\left(\prod_{i=mt+s}^{m-1+mt+s} \gamma(\overline{a^i}, a)\right) \\
			&= x_m^{t} \left(\prod_{i=mt+s}^{m-1+mt+s}\gamma(\overline{a^i}, a)\right) \\
			&= x_m^{t} \left(\prod_{i=s}^{m-1+s} \gamma(\overline{a^i}, a)\right) \\
			&= x_m^{t} \underbrace{\gamma(a^s,a)}_{=1} \left( \underbrace{\prod_{i = s+1}^{m-1+s}\gamma(\overline{a^i},a)}_{\text{base case}}\right) \\
			&= x_m^{t+1}.
		\end{align*} 
		
		Therefore, $\tau(a^sba^nba^{-(s+1)}) = x_s x_m^{t+1} x_s$ for $0 \leq s \leq m-2$. 
		
	\item[Case 2: $w = a^{m-1}ba^nba^{-m}$.] Then, 
		$$\tau(a^{m-1}ba^nba^{-m})=
		\underbrace{\left(\prod_{i=0}^{m-2}\gamma(a^i, a)\right)}_{=1} 
		\underbrace{\gamma(a^{m-1},b)}_{x_{m-1}} 
		\left(\prod_{i=m}^{m+n-1}(\gamma(\overline{a^i},a)\right)
		\underbrace{\gamma(\ovl{a^{m+n}},b)}_{x_{m-1}}
		\underbrace{\left(\prod_{i=0}^{m-1}\gamma(\ovl{a^{m+n+1-i}},a\inv)\right)}_{=x_m\inv}$$ 
		Indeed, again replacing $n$ by $m-1+mt$, we can simplify the last factor as follows.
		$$\prod_{i=0}^{m-1}\gamma(\ovl{a^{m+n+1-i}},a\inv) = \prod_{i=0}^{m-1} \gamma(\ovl{a^{m-i}},a\inv) = \udb{\gamma({a^{m-0}},a\inv)}_{=x_m\inv} \udb{\prod_{i=1}^{m-1}\gamma({a^{m-i}},a\inv)}_{=1} = x_m\inv.$$
		
		We are left with 
		$$\tau(a^{m-1}ba^nba^{-m})=x_{m-1} \left(\prod_{i=m}^{m+n-1}(\gamma(\overline{a^i},a)\right) x_{m-1} \cdot x_m\inv.$$ 
	
	We claim by induction on $t$ that $\prod_{i=m}^{m+n-1}(\gamma(\overline{a^i},a)) = x_m^t.$ Note that  $m + n -1 = m + (m-1 + mt) -1 = m(t+2) - 2$. The base case $t=0$ implies that $n=m-1$. Substituting in the upper index of the product, 
	$$\prod_{i=m}^{m+n-1}(\gamma(\overline{a^i},a)) = \prod_{i=m}^{2m-2}(\gamma(\overline{a^i},a)) = \prod_{i=0}^{m-2}(\gamma(\overline{a^i},a)) =1.$$
	Assuming the hypothesis, and using that
	\begin{align*}
	\prod_{i=m}^{m(t+2)-2}(\gamma(\overline{a^i},a)) &= \left(\prod_{i=m}^{m(t+1)-2}(\gamma(\overline{a^i},a))\right) \left(\prod_{i=m(t+1)-1}^{m(t+2)-2}(\gamma(\overline{a^i},a))\right)\\
	&= x_m^{t-1}\left(\prod_{i=m-1}^{2m-2}(\gamma(\overline{a^i},a))\right) \\
	&=x_m^{t-1} \cdot \gamma(a^{m-1},a)\left(\prod_{i=m}^{2m-2}(\gamma(\overline{a^i},a))\right) \\
	&= x_m^{t-1} \cdot x_m \udb{\left(\prod_{i=0}^{m-2}(\gamma(\overline{a^i},a))\right)}_{=1}.
	\end{align*}
	
	\end{description}
	Therefore $\tau(a^{m-1}ba^nba^{-m}) = x_{m-1}x_m^t x_{m-1} x_m\inv$. This finishes the proof for the presentation of $H$.
\end{proof}

\begin{cor}\label{Ab-H}
Let $m,n,t$ and $H$ be as defined in Proposition \ref{H-pres}. If $t$ is an odd integer, then the abelianization of $H$, $\Ab(H)$, is isomorphic to $(\mathbb{Z}/2\mathbb{Z})^m \times \mathbb{Z}$.
\end{cor}
\begin{proof}
	Take the presentation of $H$ as given in the statement of Proposition \ref{H-pres}, and make a natural identification $x_i \mapsto y_i$, for $i=0, \dots, m$ from the generators of $H$ to the generators of the abelianization of $H$, which we denote $\Ab(H)$. Then, $\Ab(H)$ has a presentation 
	$$\Ab(H) = \langle y_0, \dots, y_m \mid 2y_i + (t+1)y_m, 2y_{m-1} + (t-1)y_m, \quad 0 \leq i \leq m-2 \rangle.$$ 
	Observe that $\Ab(H)$ is a $\mathbb{Z}$-module with basis $Y$. 
	Assume $t$ is odd, and define $z_i := y_i + \frac{t+1}{2}y_m$ for $i = 0, \dots, m-2$, $z_{m-1} := y_{m-1} + \frac{t-1}{2}y_m$, and $Z := \{z_0, \dots, z_{m-1}, y_m\}$. The map $Y \to Z$ as defined in the previous sentence is clearly an invertible linear transformation, and thus $Z$ is also a basis for $\Ab(H)$. 

Thus, keeping the relations the same, we may rewrite the presentation of $H_{{\text{ab}}}$ as 
$$\Ab(H) = \langle z_0, \dots, z_{m-1}, y_m \mid 2z_0, \dots, 2z_{m-1} \rangle.$$
From this, we can clearly see that $\Ab(H)$ isomorphic to $(\mathbb{Z}/2\mathbb{Z})^m \times \mathbb{Z}$. 
\end{proof}

\begin{cor}\label{cor: fg-P-ab-rank}
Let $m,n,t$ and $H$ be defined as in Proposition \ref{H-pres}. If $t$ is an odd integer, then $H$ has a positive cone of rank $m+1$
\end{cor}
\begin{proof}
By Corollary \ref{Ab-H}, the subgroup $H$ has an abelianization of rank $m+1$. By Proposition \ref{prop: fg-pos-cone-gen}, $H$ admits a positive cone $P$ with $m+1$ generators. By Lemma \ref{rk}, $m+1$ is the rank of $P$. 
\end{proof}

Putting everything together, we get the following theorem. 
\begin{thm} Let $\Gamma_n = \langle a, b \mid ba^nba^{-1} \rangle$. For every integer $m \geq 2$, and integer $n \geq 2 $ of the form $n = m-1 + mt$ for some odd integer $t$, the subgroup $H_{n,m} := \ker \vphi_{n,m}$ where $\vphi_{n,m}: \Gamma_n \to \mathbb{Z}/m\mathbb{Z}$, $a \mapsto 1, b \mapsto 1$ admits a positive cone of rank $m+1$. \cite{Su2020}
\end{thm}
\begin{proof}
	The Corollary \ref{cor: fg-P-ab-rank} shows that for any $m+1 \geq 3$, the subgroup $H \leq \Gamma_n$ as defined in Proposition \ref{H-pres} has a positive cone of rank $m+1$ as long as $n$ is of the form $n = m-1 + mt$ with odd integer $t \geq 1$. 
\end{proof}
Therefore, the family $$\{H_{n,m} \mid n = m-1 + mt, \quad t \geq 1 \text{ and is odd}\}$$ is an infinite family of groups with $m+1$-generated positive cones. This shows Theorem \ref{thm: k-gen-P}. 

\begin{rmk}[Comparison with $\BS(1,n)$] 
	It is worth comparing our work with $\Gamma_n = \langle a,b \mid ba^n b = a \rangle$ to the solvable Baumslag-Solitar groups $\BS(1,n) = \langle a, b \mid b\inv a^n b = a \rangle$ we explored in Section \ref{sec: BS}.  Notably, $\BS(1, -1) = \langle a,b \mid b\inv a\inv b \cong a\rangle = K_2 \cong \Gamma_1$ is the Klein bottle group, and thus admits a finitely generated positive cone given by $S = \langle a, b\rangle^+$ as well. We illustrate why $S$ fails to be a positive cone for $|n| > 1$ in the $\BS(1,n)$ case. 
	
	Recall that we have established in Section \ref{sec: BS} that $P = P_B A \cup P_A$ where $A = \langle a \rangle$, $B = \langle b \rangle$ and $P_A = \langle a \rangle^+$, $P_B = \langle b \rangle^+$. 
	is a positive cone for $\BS(1,n)$ whenever $|n| > 1$. Observe that $S \subseteq P$. 
	
	For the case of $n > 1$, the relation is similar to that of $\mathbb{Z}^2$ except that the relation is $a^n b = ba$ instead of $ab = ba$. The elements of $S$ behave similar as in $\mathbb{Z}^2$ in the sense that they approximately span the upper-right quarter-plane of each ``sheet'' of the Cayley graph, but fails to span $P$ as illustrated for $\BS(1,2)$ in Figure \ref{fig: fg-BS(1,n)}. Note that again similarly to $\mathbb{Z}^2$, $S$ can also be contained in some of the positive cones parametrised by rational numbers such as $Q_0$ of Proposition \ref{prop: BS-reg-Pcone-Q0}.  
	
	\begin{figure}[h]
	\centering
	{
	\includegraphics[width = \textwidth]{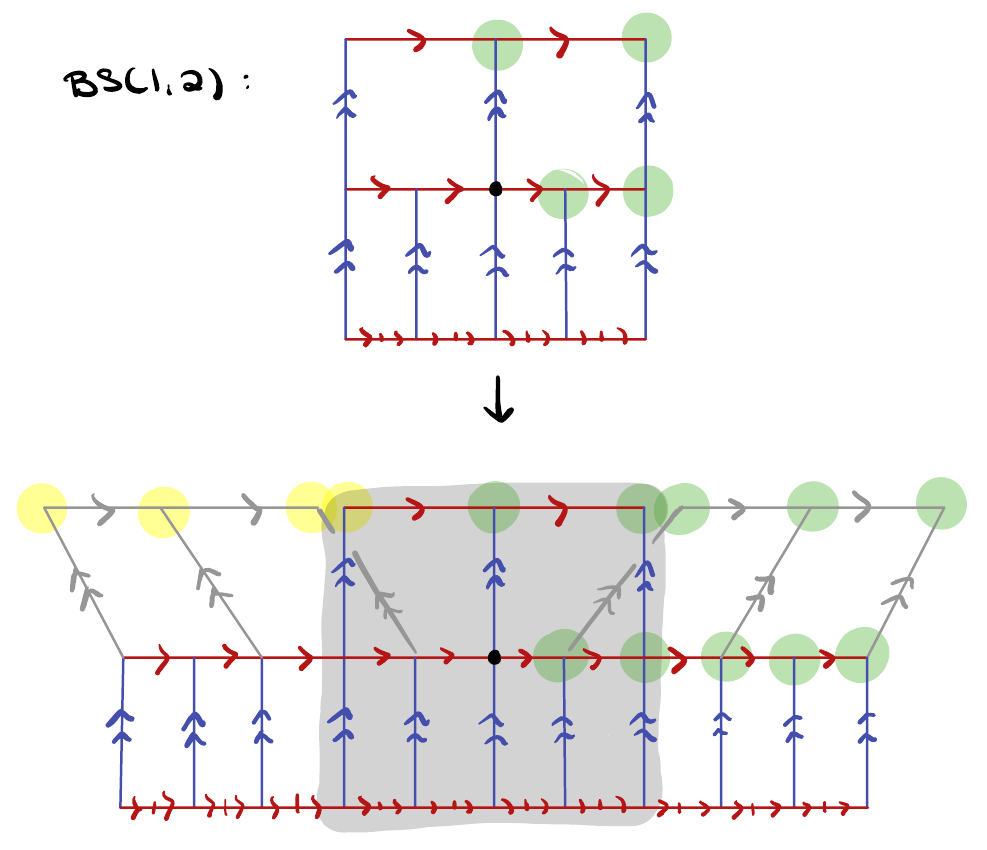}
	}
	\caption{A Cayley graph for $\BS(1,2)$, with highlighted elements given by the lexicographic positive cone $P$. The elements in green correspond to the elements spanned by $S = \langle a,b \rangle^+$, whereas the yellow ones are the ones in $P - S$.}
	\label{fig: fg-BS(1,n)}
	\end{figure} 
	
	The case of $n < 1$ looks a little more promising a priori, as the relation $a^{-2}b = ba$ suggests a twist of the $a$-axis in the Cayley graph instead of the twist in the $b$-axis as in $\Gamma_n$ (for the case of $\BS(1,-1)$, the $a$- and $b$-axes are interchangeable, hence $\BS(1,-1) \cong \Gamma_1$). Indeed, $S$ now does approximately span the upper half-plane of each ``sheet'' of the Cayley graph. However, the dyadic tree structure caused by the relation takes over now that the action of $b$ is conjugation, and the semigroup $S$ misses some elements of the upper part of the Cayley graph that are contained in $P$, as illustrated in Figure \ref{fig: fg-BS(1,-n)}. We already know that there are only four possible one-counter positive cones for $\BS(1,n)$ in this case, which are determined by which of $a^\pm$ and $b^\pm$ are included. Thus, $P$ is the only positive cone containing $S$ in this case. 
	
	\begin{figure}[h]
	\centering
	{
	\includegraphics[width = \textwidth]{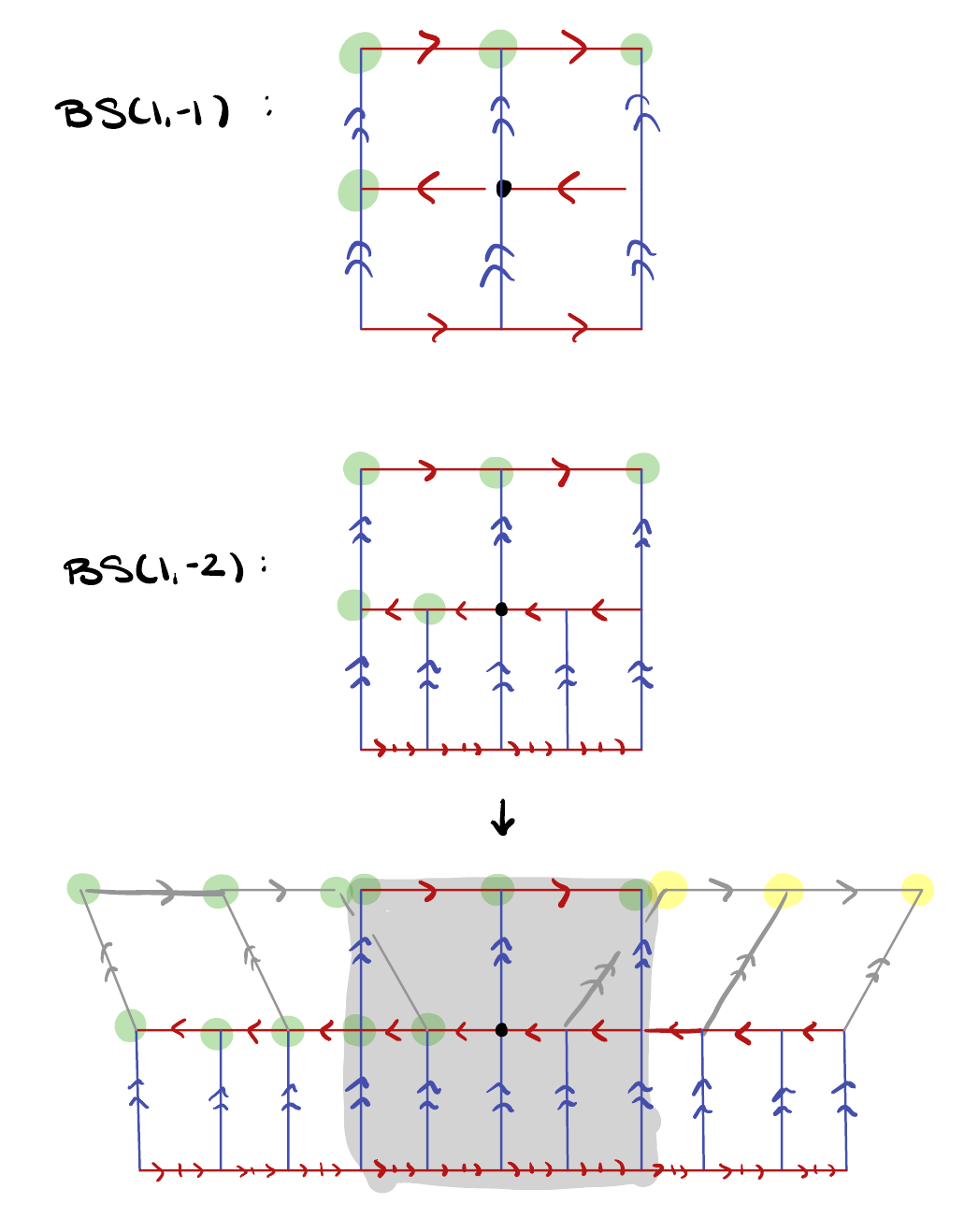}
	}
	\caption{A Cayley graph for $\BS(1,-2)$, with highlighted elements given by the lexicographic positive cone $P$. The elements in green correspond to the elements spanned by $S = \langle a,b \rangle^+$, whereas the yellow ones are the ones in $P - S$. A Cayley graph for $\BS(1,-1)$ is offered for comparison.}
	\label{fig: fg-BS(1,-n)}
	\end{figure} 
\end{rmk}

\section{$F_n \times \mathbb{Z}$ has a finitely generated positive cone for even $n$}\label{sec: fg-F_nxZ}
We begin by proving the $F_2 \times \mathbb{Z}$ case, then generalise to $n > 2$. 
\subsection{The $F_2 \times \mathbb{Z}$ case}
Let $n=2, m = 6$, and recall the statement of Proposition \ref{prop: fg-pos-cone-gen}. A possible solution for the equation 
$$2 + (n-1)\vphi(a) \equiv 0 \mod 6$$ 
is $\vphi(a) = 4$. Let $H := \ker \vphi$. Then $H$ admits a positive cone generated by 
\begin{equation}\label{eqn: fg-Y-gen-F2xZ}
	Y := \{ab^2, b^{-1}ab^3, b^{-2}ab^4, b^{-3}ab^5, b^{-4}a, b^{-5}ab, b^6\}, \qquad |Y| = 7.
\end{equation}

We will illustrate how we found via GAP that $H \cong F_2 \times \mathbb{Z}$. The reader is free to skip to the mathematical proof at Lemma \ref{lem: fg-H-F2xZ}. 

We first set the ambient environment for our computation by running the following code. 
\begin{lstlisting}[language = GAP, breaklines]
# Set the ambient group to Gamma_2. 
F := FreeGroup("a", "b");
a := F.1;; b := F.2;;
Gamma_2 := F/[b*a^2*b*a^-1];
a := Gamma_2.1;; b := Gamma_2.2;;

Y := [ a*b^2, b^-1*a*b^3, b^-2*a*b^4, b^-3*a*b^5, b^-4*a, b^-5*a*b, b^6 ];
H := Subgroup(Gamma_2, Y);
\end{lstlisting}

We then prompt GAP to make an isomorphism from $H$ to a finitely presented group $G$. 

\begin{lstlisting}[language=bash]
gap> iso := IsomorphismFpGroup(H);
[ <[ [ 1, 1 ] ]|a^3>, <[ [ 1, -1, 3, -1 ] ]|a^-2*b^-2*a>, 
  <[ [ 2, 1, 3, 1, 1, 1 ] ]|a*b^2*a^-1*b^2*a^2> ] -> [ F1, F2, F3 ]
gap> G := Image(iso); 
<fp group of size infinity on the generators [ F1, F2, F3 ]>
gap> RelatorsOfFpGroup(G);
[ F2*F1*F2^-1*F1^-1, F3*F1^-1*F3^-1*F1 ]
\end{lstlisting}

which gives us that a presentation is given by $$\langle F1, F2, F3 \mid [F2, F1], [F3, F1\inv] \rangle.$$ This is already clearly isomorphic to $F_2 \times \bZ$, but let us simplify the generating set of $G$ to the following ``prettier'' set of positive words
\begin{align*}
&x:= F3*F2 = ab^2 a\inv b^2 a^2 \cdot a^{-2} b^{-2} a = ab^2, \\
&y := (F2*F1\inv)\inv = (a^{-2} b^{-2} a \cdot a^{-3})\inv = a^2 b^2 a^2, \\
&z := F1 = a^3.
\end{align*}

We do so by means of Tietze transformations. 
\begin{align*}
	&\langle F1, F2, F3 \mid [F2, F1], [F3, F1\inv] \rangle \\
	&= \langle F1, F2, F3, x, y, z \mid x = F3*F2, y = (F2*F1\inv)\inv, \underline{z = F1}, [F2, F1], [F3, F1\inv] \rangle \\
	&= \langle F2, F3, x, y, z \mid x = F3*F2, y = (F2*z\inv)\inv, [F2, z], [F3, z\inv] \rangle \\
	&= \langle F2, F3, x, y, z \mid x = F3*F2, \underline{F2 = y\inv z}, [F2, z], [F3, z\inv] \rangle \\
	&= \langle F3, x, y, z \mid x = F3 * y\inv z, [y\inv z, z], [F3, z\inv] \rangle \\
	&= \langle F3, x, y, z \mid \underline{F3 = xz\inv y}, [y\inv z, z], [F3, z\inv] \rangle \\
	&= \langle x, y, z \mid [y\inv z, z], [xz\inv y, z\inv] \rangle \\
	&= \langle x, y, z \mid [x,z], [y,z] \rangle 
\end{align*}

Then, we have that $$H = \langle ab^2, a^2b^2a^2, a^3 \rangle \cong G = \langle x, y, z \mid [x,z], [y,z] \rangle,$$ where the isomorphism map is given by $\psi: H \to F_2 \times \mathbb{Z}$, where $\psi(ab^2) = x, \psi(a^2b^2a^2) = y, \psi(a^3) = z$ as we wanted. 
 
After reassigning some variables in GAP, we can make use of this map to find the representation of the $Y$ generators in terms of generators of $G$. 
 
\begin{lstlisting}[language=bash]
gap> x := a*b^2;; y := a^2*b^2*a^2;; z := a^3;;
gap> H := Subgroup(Gamma_2, [x,y,z]);
Group([ a*b^2, a^2*b^2*a^2, a^3 ])
gap> iso := IsomorphismFpGroupByGenerators(H, [x,y,z]);
[ a*b^2, a^2*b^2*a^2, a^3 ] -> [ F1, F2, F3 ]
gap> Apply(Y, y -> Image(iso, y));
gap> Y;
[ F1, F2*F3^-1*F1, F1^-1*F2*F3^-1*F1, F1^-1*F3^2*F2^-1*F1^-1*F3^-1*F2*F3^-1*F1, 
  F1^-1*F3*F2^-1*F3, F1^-1*F3*F2^-1*F3*F1, F3*F2^-1*F1^-1*F3^-1*F2*F3^-1*F1 ]
\end{lstlisting}
where $F1 = x, F2 = y, F3 = z$ under our new isomorphism. 

Simplifying slighlty, we get
\begin{equation}\label{eq: fg-Y}
	\psi(Y) = \{x, yxz^{-1}, x^{-1}yxz^{-1}, x^{-1}y^{-1}x^{-1}yx, x^{-1}y\inv z^2, x^{-1}y^{-1}xz^2, y^{-1}x^{-1}yxz^{-1} \}.
\end{equation} 
\subsection{Proving by hand that $H \cong F_2 \times \mathbb{Z}$}

If the previous computer-aided proof is not satisfying to the reader\sidenote{The proof was not satisfying to the author.}, we invite the reader to follow the proof of the following lemma for a more human understanding.\sidenote{Note that the statement of the lemma was derived after using the computer to obtain answers, and thus should be seen as a verification proof.}

\begin{lem}\label{lem: fg-H-F2xZ}
	Let $H$ be the subgroup of $\Gamma_2 = \langle a,b \mid ba^2b = a$ generated by the set $X = \{x, y, z \mid x := ab^2, y := a^2b^2a^2, z := a^3\}$. Then $H \cong F_2 \times \mathbb{Z}$ and $H = \ker \vphi$, where $\vphi: \Gamma_2 \to \mathbb{Z}/6\mathbb{Z}$, $a \mapsto 4, b \mapsto 1$. 
\end{lem}

A possible way to prove this lemma is of course by computing a Reidemeister-Schreier presentation by hand, and then compute a reduced presentation using Tietze transformations which includes the wanted the generator. However, this would be essentially no better than what the computer has already done\sidenote{See the GAP documentation on Finitely Presented Groups, \url{https://docs.gap-system.org/doc/ref/chap47.html}.}
and would not satisfy the goal of providing a more conceptual proof to the statement. Instead, we will send $\Gamma_2$ to $\PSL(2, \mathbb{Z})$  (as done in \cite[Section 3]{Navas2011}) and show that the image of $H$ corresponds to $F_2$ by using the Ping-Pong lemma (introduced in Chapter \ref{chap: hyperbolic}, see \cite[Chapter 5]{OfficeHours2017} for more details), 
then take the lift to show $H \cong F_2 \times \mathbb{Z}$. 
	
\begin{proof}
	Observe that for $f = a, h = b\inv a$, the presentation for $\Gamma_2 = \langle a,b \mid ba^2b = a\rangle$ becomes $\Gamma_2 = \langle f,h \mid f^3 = h^2\rangle$. Let $\phi: \Gamma_2 \to \Gamma_2/\langle \langle a^3 \rangle \rangle$ be the map taking the quotient by the center $\langle a^3\rangle$. Then $\phi(\Gamma_2) = \langle f,h \mid f^3 = h^2 = 1\rangle = \mathbb{Z}/3 \mathbb{Z} * \mathbb{Z}/2 \mathbb{Z}$, which is known to be isomorphic to $\PSL(2,\mathbb{Z})$ \cite{Alperin1993}.
	For any $g \in \Gamma_2$, let us denote $\bar g := \phi(g)$. 
	We can verify that the representation given by 
	$$\bar a = \begin{bmatrix}
		0 & 1 \\
		-1 & 1
	\end{bmatrix}, \qquad 
	\bar b = \begin{bmatrix}
		1 & 0 \\
		1 & 1
	\end{bmatrix}$$
	is indeed a representation, since $\phi(ba^2b) = \bar{a}$. %
	
	Then, 
	$$\bar x = \begin{bmatrix}
		2 & 1 \\
		1 & 1 
	\end{bmatrix}, \quad 
	\bar x\inv = \begin{bmatrix}
		1 & -1 \\
		-1 & 2
	\end{bmatrix}$$
	and 
	$$\bar y = \begin{bmatrix}
		-2 & 1 \\
		1 & -1
	\end{bmatrix}, \quad 
	\bar y\inv = \begin{bmatrix}
		 1 & 1 \\
		 1 & 2
	\end{bmatrix}.$$
	Let us define an action of $\PSL(2, \mathbb{Z})$ on the extended complex plane 
	$\hat {\mathbb{C}} := \mathbb{C} \cup \{\infty\},$ 
	$$\begin{bmatrix}
		a & b \\ c & d \end{bmatrix}(w) := \frac{aw + b}{cw + d}.$$ 
			
	Then, solving for the fixed points of $\bar x, \bar y$, we get 
	$$\bar x(w_0) = \frac{2w_0 + 1}{w_0 + 1} \implies w_0 = \frac{1 \pm \sqrt 5}{2}$$ and 
	$$\bar y(w_0) = \frac{-2w_0 + 1}{w_0 -1} \implies w_0 = \frac{-1 \pm \sqrt 5}{2}.$$
	We compute the derivatives at the fixed points to identify the attractors and repellers and obtain the following. 
	$$\bar x'(w) = \frac{1}{(1+w)^2}, \quad |\bar x'((1 + \sqrt 5)/2)| < 1 \quad |\bar x'((1-\sqrt 5)/2)| > 1$$
	$$\bar y'(w) = \frac{1}{(z-1)^2}, \quad |\bar y'((-1 + \sqrt 5)/2)| > 1,\quad \bar y'((-1 - \sqrt 5)/2)| < 1$$
	Let us denote by 
	\begin{align*}
		& x^- := (1+\sqrt 5)/2, & x^+ := (1 - \sqrt 5)/2 \\
		& y^- := (-1 - \sqrt 5)/2, & y^+ := (-1 + \sqrt 5)/2
	\end{align*} 
	the attractors and repellers of $\bar x$ and $\bar y$ respectively (and thus the repellers and attractors of $\bar x\inv, \bar y\inv$ respectively). With the exceptions of its fixed points, $\bar x$ moves every point towards $x^+$ and push every point away from $x^-$, and similarly for $\bar y$. %
	
	Let $S^1$ be the circle representing $\hat{\mathbb{R}}$ in $\hat{\mathbb{C}}$ after applying the conformal mapping $w \mapsto \frac{w - i}{1-iw}$ so that $i$ resides in the center of the circle, as illustrated in Figure \ref{fig: fg-Ping-Pong}. 
	
	\begin{figure}[h]
		\centering
		{
		\includegraphics[width = \textwidth]{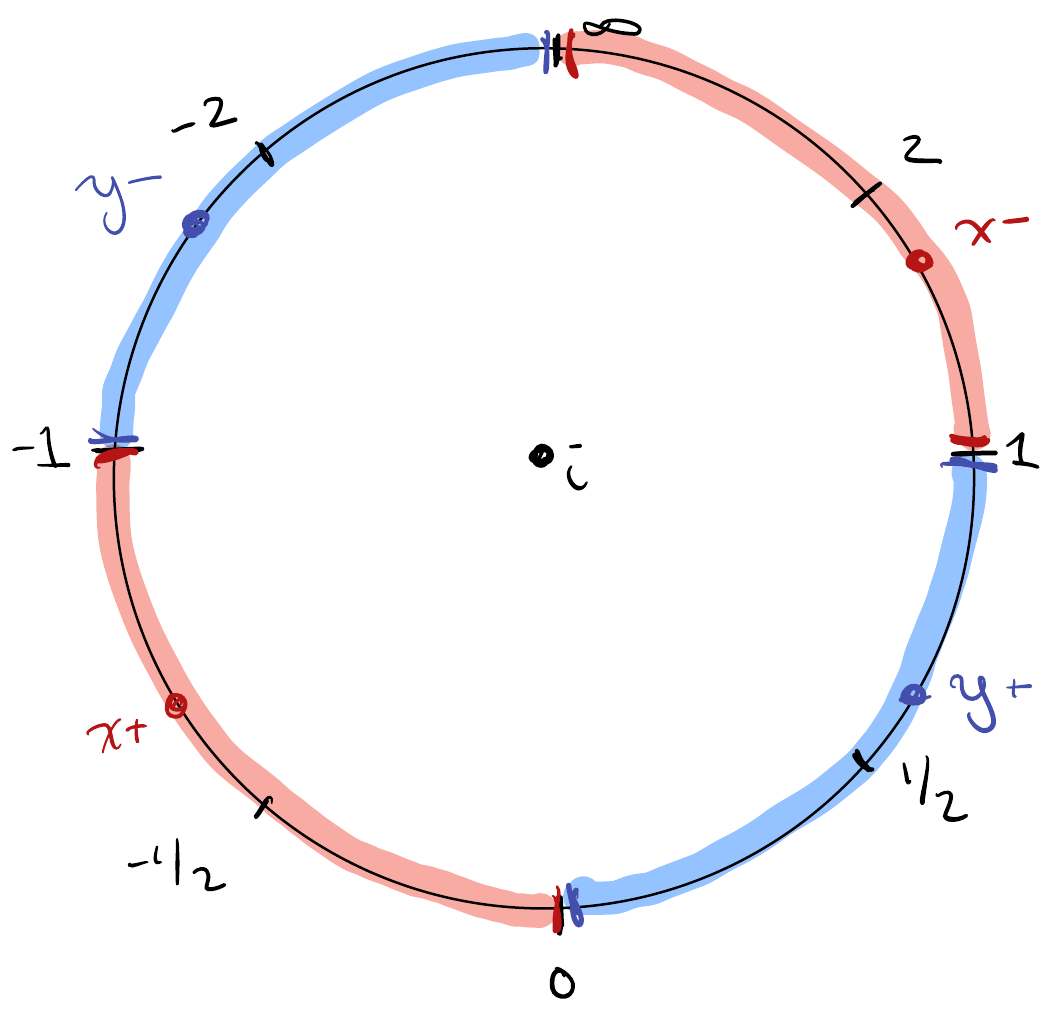}
		}
		\caption{A circle in $\hat{\mathbb{C}}$ representing $\hat{\mathbb{R}}$. The attractors and repellers of $\bar x$ (in red) and $\bar y$ (in blue) are identified, and the Ping-Pong sets $X^- \cup X^+$ and $Y^- \cup Y^+$ around them are highlighted (in red and blue respectively). The metric is hyperbolic and the identified coordinates were chosen as equidistant Farey coordinates. 
		}
		\label{fig: fg-Ping-Pong}
	\end{figure}
	
	Let $X^-, X^+, Y^-, Y^+ \subset S^1$ be neighbourhoods around the attractors and repellers points $x^-, x^+, y^-, y^+$ respectively. To simplify our computations, we choose the neighbourhoods to be separated by equidistant Farey coordinates.\sidenote{Coordinate pairs $p/q, r/s \in \mathbb{Q}$ such that $|ps-rq| = 1$, i.e. the action of $\PSL_2(\mathbb{Z})$ sends a neighbourhood separated by Farey coordinates to another neighbourhood separated Farey coordinate.} Denoting the endpoints of intervals clockwise, $ X^- := (\infty, 1), X^+ := (0, -1), Y^- := (-1, \infty), Y^+ := (1, 0)$. We claim that $X^- \cup X^+$, $Y^- \cup Y^+$ are our ping-pong sets, as illustrated in Figure \ref{fig: fg-Ping-Pong}.  
	
	To prove our claim, we compute where $\bar x, \bar x\inv, \bar y, \bar y\inv$ send each endpoint of the neighbourhood, and use the continuity of these conformal maps on $\hat{\mathbb{C}}$, which preserve their continuity when restricted to the closed set $S^1$, to deduce they map segments of $S^1$ to other segments of $S^1$. Indeed, since $S^1$ is the closure of $\mathbb{R}$, it follows simply looking at the definitions of these maps that they map $\hat{\mathbb{R}} \to \hat{\mathbb{R}}$. 
	
	 Thus, the result of these computations below (illustrated in Figure \ref{fig: fg-Ping-Pong-2}) proves the condition of the Ping-Pong lemma are satisfied. 

	\begin{align*}
		&\bar x(\infty) = 2, &\bar x(-1) = \infty,
			 &\implies \bar x(Y^-) = (\infty, 2) \subset X^- \\
		&\bar x(0) = 1, &\bar x(1) = 3/2,  
			&\implies \bar x(Y^+) = (3/2,1) \subset X^- \\
		\\
		&\bar x\inv(-1) = -2/3, &\bar x\inv(\infty) = -1, 
			&\implies \bar x\inv(Y^-) = (-2/3, -1) \subset X^+ \\
		&\bar x\inv(1) = 0, &\bar x\inv(0) = -1/2, 
			&\implies \bar x\inv(Y^+) = (0, -1/2) \subset X^+ \\
		\\
		&\bar y(\infty) = -2, &\bar y(1) = \infty, 
		 &\implies \bar y(X^-) = (-2, \infty) \subset Y^- \\
		&\bar y(0) = -1, &\bar y(-1) = -3/2, 
		 &\implies \bar y(X^+) = (-1, -3/2) \subset Y^- \\
		 \\
		&\bar y\inv(\infty) = 1, &\bar y\inv(1) = 2/3, 
		 &\implies \bar y\inv(X^-) = (1,2/3) \subset Y^+ \\
		&\bar y\inv(0) = 1/2, &\bar y\inv(-1) = 0, 
		 &\implies \bar y\inv(X^+) = (1/2,0) \subset Y^+
	\end{align*}
	
	\begin{figure}[h]
		\centering
		{
		\includegraphics[width = \textwidth]{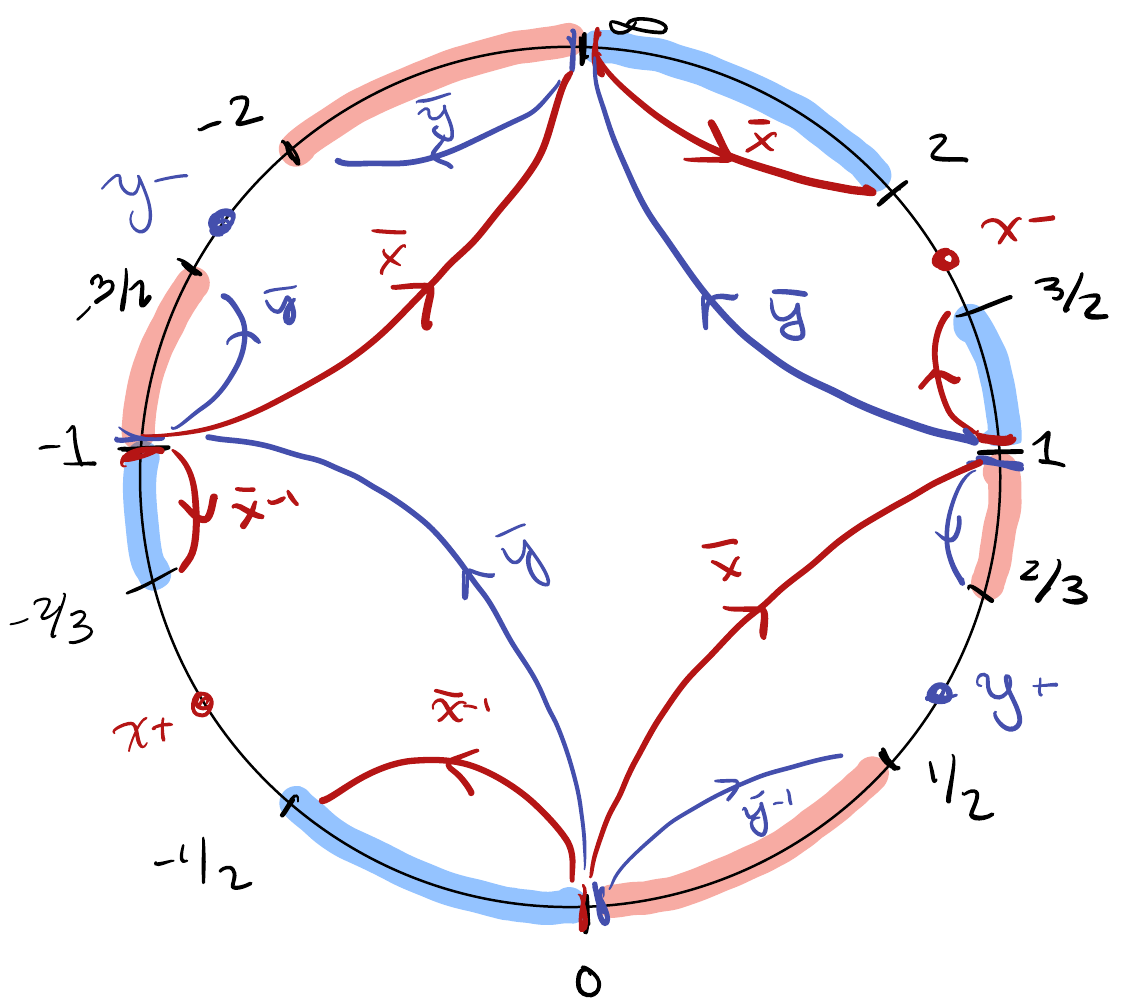}
		}
		\caption{A circle in $\hat{\mathbb{C}}$ representing $\hat{\mathbb{R}}$. The arrows illustrated the image of the $\bar x$ and $\bar y$ maps (in red and blue respectively). The image of the Ping-Pong sets $X^- \cup X^+$ and $Y- \cup Y^+$ are highlighted (in red and blue respectively). 
		}
		\label{fig: fg-Ping-Pong-2}
	\end{figure}
	
	Since $\phi(H) \cong F_2 = H / \langle \langle a^3 \rangle \rangle$ is the quotient of $H$ by the center, it follows that $H \cong F_2 \times \mathbb{Z}$.
	
	To show that $H = \ker \vphi$, we first observe that $\varphi(g) = 0 \mod 6$ for every $g \in X$. 
	Hence, $H \leq \ker \vphi$. To show that $H \geq \ker \vphi$, we can verify the equalities stated in Equation \ref{eq: fg-Y} by converting both sides into normal form using the algorithm proposed in Remark \ref{rmk: fg-normalize-alg}, computing either by hand or using the coded version available in the Chapter \ref{chap: fg-code} of the Appendix.\sidenote{Perhaps this is cheating, but one has to ask what is the real difference between applying a known algorithm by hand or by computer.} 
		
	We write down the normal forms for verification purposes using the preferred method by the reader. 
	
	\begin{align*}
		&ab^2 = ab^2 =  x\\
		&b\inv ab^3 = a^2 b^4 = yxz\inv \\ 
		&b^{-2} a b^4 = a(ab)^2b^4 = x^{-1}yxz^{-1} \\
		&b^{-3}a{b^5} = a(ab)^3b^5 = x^{-1}y^{-1}x^{-1}yx \\
		&b^{-4}a = a(ab)^4 = x^{-1}y\inv z^2 \\
		&b^{-5}ab = a(ab)^5b = x^{-1}y^{-1}xz^2 \\
		&b^6 = b^6 = y^{-1}x^{-1}yxz^{-1}
	\end{align*}
\end{proof}
\
\begin{cor}
	Let $F_2 \times \mathbb{Z} = \langle x, y, z \mid [x,z],[y,z]\rangle$. Then it admits a finitely generated positive cone given by $$P = \langle x, yxz^{-1}, x^{-1}yxz^{-1}, x^{-1}y^{-1}x^{-1}yx, x^{-1}y\inv z^2, x^{-1}y^{-1}xz^2, y^{-1}x^{-1}yxz^{-1} \rangle^+.$$
\end{cor}

\subsection{Generalising the result to $F_n \times \mathbb{Z}$}
\label{sec: Fn x Z}
We have just showed that $F_2 \times \mathbb{Z}$ has a finitely generated positive cone by using the overgroup $\Gamma_2$ with the finitely generated positive cone $P_2 = \langle a, b \rangle^+$, and finding a presentation using the Reidemeister-Schreier method and a particular transversal such that the generating set of the finite index subgroup consists only elements that are in $P_2$.  

The roadmap to proving that $F_n \times \mathbb{Z}$ for $n$ even has a finitely generated positive cone will be as follows: the overgroup will now be $F_2 \times \mathbb{Z}$, with generating set given by $Y$, and we will attempt to find a presentation for $F_n \times \mathbb{Z}$ as a finite index subgroup of $F_2 \times \mathbb{Z}$ under immersion with a particular choice of transversal such that every generator lies in $P_2$. To that end, we make the first key observation. 
\begin{lem}\label{lem: fg-Fn-pos-gen-set}
	Suppose that $T$ is a transversal generated by $t := b^6 \in Y$ and $\gamma(t,x)$ be as defined as in the statement of the Reidemeister-Schreier method (Theorem \ref{thm: RS}). Then, for any $x \in Y$, the element $\gamma(t,x) \in P_2 \cup \{1\}$. 
\end{lem}
\begin{proof}
	We have $$\gamma(t,x) := (\ovl{tx})(tx)\inv = b^p a b^q \text{ or } 1$$ for some $p,q \in \mathbb{Z}$. For the case where $\gamma(t,x) \not=1$, it is positive by Lemma \ref{lem: in-P}.
\end{proof}

It remains to show that there is an initial segment generated by $b^6$ which forms a transversal for $H \cong F_2 \times \mathbb{Z}$. Recall from Equation \ref{eq: fg-Y} that $b^6 \mapsto y\inv x\inv yx z\inv$ under this isomorphism. We reduce the problem showing that the restriction of $b^6$ in $F_2$, which we denote $g := y^{-1} x^{-1} y x$, acts transitively on the cosets of $F_n$. 

We want to show that $g$ induces a Schreier transversal for $F_n$ in $F_2$ if and only if $n$ is even. To do so, we will need the language of immersions. 

\subsection{Playing with immersions of $F_n$ in $F_2$}

\begin{figure}[h]
\centering
{
\includegraphics[width = \textwidth]{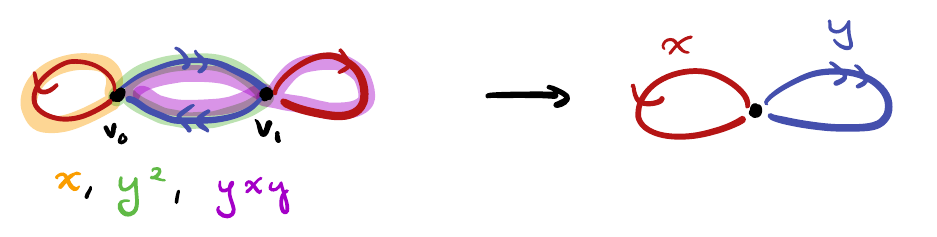}
}
\caption{An immersion from $F_3$ to $F_2$. We identify $F_3$ with the fundamental group of the $3$-rose $R_3$ based at $v_0$, $\pi_1(R_3, v_0)$, and $F_2$ with the fundamental group of the $2$-rose $R_2$. The map (or marking) sending the generators of $F_3$, which are the loops of $R_3$ based at $v_0$ (given here by the orange, green and purple highlights), to elements of the fundamental group of $R_2$ is given by the colored arrows with labels $x,y$. As such, the generators of $F_3$ correspond to the elements $x, y^2, yx^2y$ in $F_2$. Furthermore, each basepoint corresponds to a coset of $F_3$, and hence, $F_3$ is a subgroup of index $2$ in $F_2$. 
}
\label{fig: fg-F3-im}
\setfloatalignment{b}
\end{figure}

It is a well-known fact that free groups $F_n$ can always be embedded as a finite index subgroup of $F_2$ (see for example \cite{OfficeHours2017})\sidenote{That is the reference that I used to learn this fact which I did not know before the writing of this thesis.}. Indeed, consider the $2$-rose of the right-hand side of Figure \ref{fig: fg-F3-im}. We can realise $F_3$ as a subgroup of index $2$ by marking the loops of a $3$-rose $R_3$, as in the left-hand side of Figure \ref{fig: fg-F3-im}, such that the marking is locally injective (i.e. for every vertex, there are two incident outgoing $x$- and $y$-edges, and two incident incoming $x$- and $y$-edges), and such a map is called an \emph{immersion}. It is clear that the fundamental group of $R_3$ is $F_3$. By locally injective property, the marking tells us the injective map from $F_3 \to F_2$. Indeed, after choosing the basepoint $v_0$, the generators of $F_3$ can be identified as corresponding to $x, y^2, yxy$ in $F_2$, as indicated in Figure \ref{fig: fg-F3-im} (see \cite{OfficeHours2017} for more details). Then, $F_3$ is a subgroup of index $2$ in $F_2$, corresponding to the number of\ vertices of the immersion since each basepoint corresponds to a coset of $F_3$.

We can make a similar immersion for $F_4 \leq F_2$, as given by Figure \ref{fig: fg-F4-im}. 

\begin{figure}[h]
\centering
{
\includegraphics[width = \textwidth]{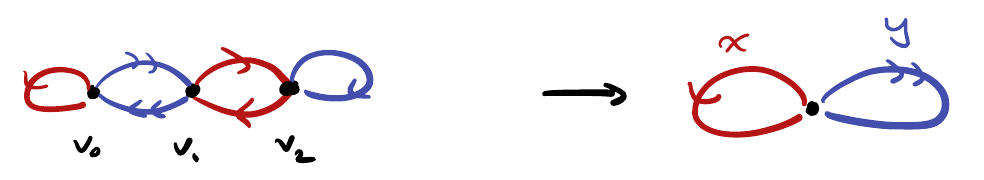}
}
\caption{An immersion from $F_4$ to $F_2$. 
}
\label{fig: fg-F4-im}
\end{figure}

In general, for every $n \geq 2$, there is an immersion $\iota_n: F_n \to F_2$ given by Figure \ref{fig: fg-Fn-im}, where we pay special attention to the fact that the parity of $n$ determines the incident edges of $v_{n-2}$ in each immersion. 

\begin{figure}[h]
\centering
{
\includegraphics[width = \textwidth]{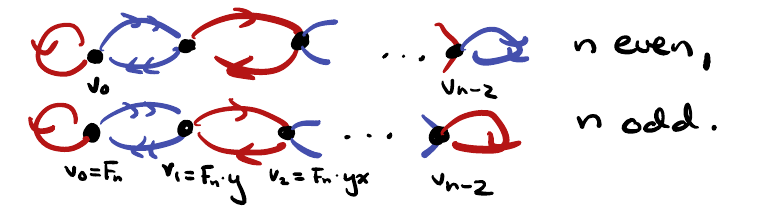}
}
\caption{An immersion from $F_n$ to $F_2$. The rightmost edge incident to $v_{n-2}$ has a different marking than the leftmost edge incident to $v_0$ when $n$ is even (top), and the same one when $n$ is odd (bottom). This will have important implications for whether $F_n \times \bZ$ has a finitely generated positive cone. 
}
\label{fig: fg-Fn-im}
\end{figure}

Observe that the vertices of the immersions $\iota_n$ can also be identified with right cosets of $F_n$. Indeed, as illustrated in Figure \ref{fig: fg-Fn-im}
 labelling the leftmost vertex as $v_0$ and identifying this vertex with the identity coset $F_n \cdot 1$, we can observe that every vertex $v_i$ on the left can be given by $F_n \cdot w_i$ where $w_i$ is a word in $F_2$ inducing a path taking $v_0$ to $v_i$.  More precisely, by following the markings of the immersion, we can identify 
 $$v_0 = F_n \cdot 1, \quad v_1 = F_n \cdot y, \quad v_2 = F_n \cdot yx, \dots, \quad v_{n-2} = F_n \cdot (yxyx)^{\lfloor(n-2)/4 \rfloor} \cdot w_n,$$
 where 
 $$w_n := \begin{cases}
 	1 & n-2 \equiv 0 \mod 4 \\
 	y & n-2 \equiv 1 \mod 4 \\
 	yx & n-2 \equiv 2 \mod 4 \\
 	yxy & n-2 \equiv 3 \mod 4.
 \end{cases}$$

It is then straightforward to observe that if $V_n = \{v_i\}_{i=0}^{n=2}$ is the set of vertices of the immersion $\iota_n$, then for each such immersion there exists a right $F_2$ action on $F_n$ which if given by $v_i \star g = v_i \cdot g$ for any $g \in F_2$. 

\begin{prop}\label{prop: fg-Fn-g-orbit}
	If $g := y\inv x\inv yx$, and $\langle g \rangle$ is the cyclic subgroup generated by $g$, then the orbit $v_0 \star \langle g \rangle = V_n$ if and only if $n$ is even. 
\end{prop}
\begin{proof}
	We begin by showing examples for $n = 2, 3, 4, 5$, then work on the general cases for $n > 5$, for which there are four. Note that the general cases actually also work for $n = 2,3,4$ with minor modifications, but we include them separately for clarity. 
		
	\begin{description}
	   \item[Ex: $n=2$.] Since $V_2 = \{v_0\}$ there is nothing to show. 
	   
	   \item[Ex: $n=3$.] By following the markings of Figure \ref{fig: fg-F3-im} $g \in F_3$, $v_0 \cdot g = F_3 \cdot g = F_3$, and $v_0 \star \langle g \rangle = \{v_0\} \not= V_3$. 
	   
	   \item[Ex: $n=4$.] By following the markings of Figure \ref{fig: fg-F4-im} 
	   we see that $v_0 \cdot g = F_4 \cdot g = F_4 \cdot y = v_1$, and $v_1 \cdot g = F_4 \cdot y g = F_4 \cdot yx = v_2$. Therefore, $v_0 \star \langle g \rangle = V_4$. 
	   
	   \item[Ex: $n=5$.] By following the markings of Figure \ref{fig: fg-Fn-im}, we see that $v_0 \cdot g = v_3$, and $v_3 \cdot g = v_0$, therefore $v_0 \star \langle g \rangle = \{v_0, v_3\} \not= V_5$. 
	\end{description}
	
	\begin{figure}[h]
		\centering
		{
		\includegraphics[width = \textwidth]{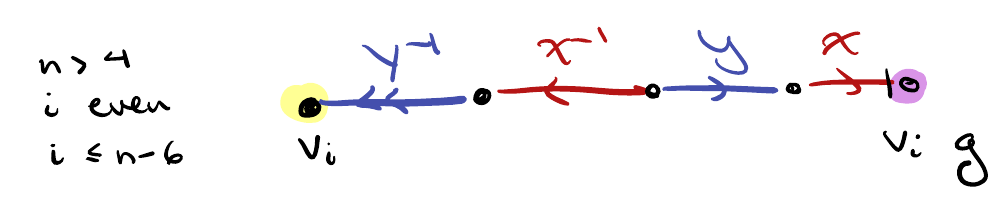}
		}
		\caption{The induced path from $v_i$ to $v_i g$ when $i$ is even and $i \leq n-6$ for $n > 4$. The start point is highlighted in yellow and endpoint in purple.  
		}
		\label{fig: fg-Fn-even}
	\end{figure}
	
	We now start working on the $n>5$ cases, for which there are two key observations due to the number of vertices of each immersion being $n-1 > 4$. First, notice that for $n>5$, $g$ always induces from $v_0$ a path which does not go through any loops, as illustrated in Figure \ref{fig: fg-Fn-even}. Since the vertices are denoted starting from $0$ to $n-2$, it is straightforward by looking at the arrows associated to even and odds vertices in the immersions to generalise this as the following observation: 
		\begin{equation}\label{eq: fg-Fn-even}
 			\text{for }i \text{ even and } i \leq n-6, \quad v_i \star g = v_{i+4}. 
 		\end{equation}

	\begin{figure}[h]
		\centering
		{
		\includegraphics[width = \textwidth]{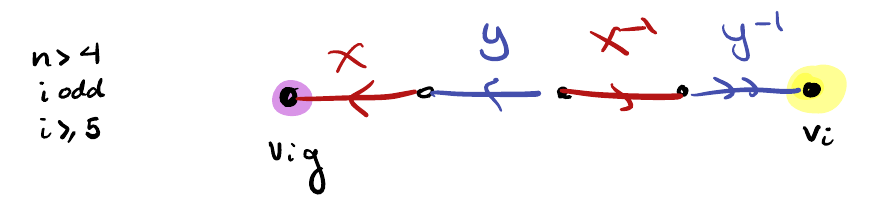}
		}
		\caption{The induced path from $v_i$ to $v_i g$ when $i$ is odd and $i \geq 5$ for $n > 4$. The start point is highlighted in yellow and endpoint in purple.  
		}
		\label{fig: fg-Fn-odd}
	\end{figure}
	
	By similar reasoning,  
		\begin{equation}\label{eq: fg-Fn-odd}
			\text{for }i \text{ odd and } i \geq 5, \quad v_i \star g = v_{i-4}. 
	 	\end{equation}
 	as illustrated in Figure \ref{fig: fg-Fn-odd}.
 	 	
 	The cases which are not included in these two above observations occur when the induced path from $v_i$ to $v_i \cdot g$ goes through an endpoint loop of the immersion, either from $v_0$ to itself or $v_{n-2}$ to itself. This is precisely what we will look at more closely in the casework. 
 	 	
	\begin{description}
	   \item[Case 1: $n-2 \equiv 0 \mod 4$.] Let $n-2 = 4m$. Starting at $v_0$, and looking at the path going through $\{v_0 \cdot g, v_0 \cdot g^2, v_0 \cdot g^3, \dots \}$, we deduce from Equation \ref{eq: fg-Fn-even} that $\{v_0, v_4, \dots, v_{4m} \} \subseteq v_0 \star \langle g \rangle$. At the right endpoint, we observe that $v_{4m} \cdot g =  v_{4(m-1) + 1}$ as illustrated in Figure \ref{fig: fg-Fn-mod0-1}. Then, since $i = 4(m-1) + 1$ is odd, we have by Equation \ref{eq: fg-Fn-odd} that $\{v_{4(m-1) + 1}, v_{4(m-2) + 1}, \dots, v_1\} \subseteq v_0 \star \langle g \rangle$.    		\begin{figure}[h]
				\centering
				{
				\includegraphics[width = \textwidth]{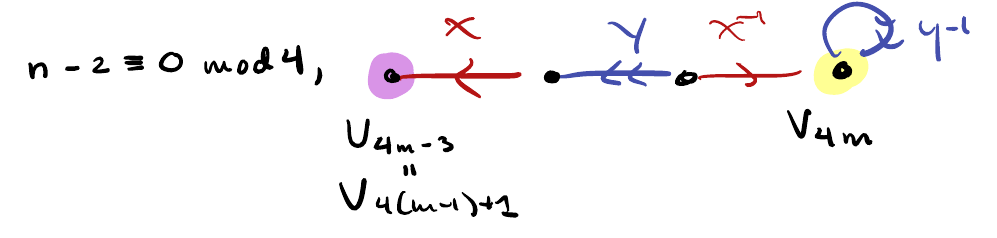}
				}
				\caption{The path induced by $g$ from $v_{4m}$ to $v_{4(m-1) -1}$ when $n-2 \mod 4 \equiv 0$. The start point is highlighted in yellow and endpoint in purple.  
				}
				\label{fig: fg-Fn-mod0-1}
			\end{figure}
	   
		   Back to the left endpoint, we have $v_1 \cdot g = v_2$ as illustrated in Figure \ref{fig: fg-Fn-mod0-2}. Therefore, $\{v_2, v_{2 + 4}, \dots, v_{4(m-1) + 2} \} \subseteq v_0 \star \langle g \rangle$ by Equation \ref{eq: fg-Fn-even}.  		
			   \begin{figure}[h]
					\centering
					{
					\includegraphics[width = \textwidth]{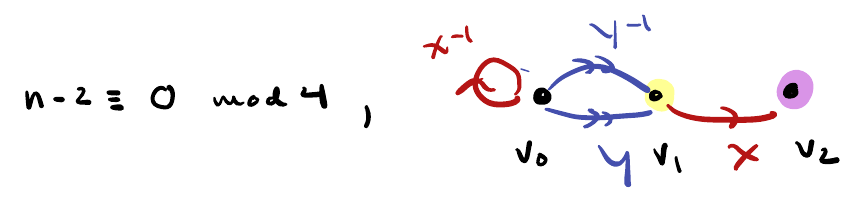}
					}
					\caption{The path induced by $g$ from $v_{1}$ to $v_{2}$ when $n-2 \mod 4 \equiv 0$. The start point is highlighted in yellow and endpoint in purple.  
					}
					\label{fig: fg-Fn-mod0-2}
				\end{figure}
		   
		   Finally, back at the right endpoint, we have that $v_{4(m-1) + 2} \cdot g = v_{4(m-1) + 3}$ as illustrated in Figure \ref{fig: fg-Fn-mod0-3}. Therefore, $\{v_{4(m-1) + 3}, v_{4(m-2) + 3}, \dots, v_3\} \subseteq v_0 \star \langle g \rangle$ by Equation \ref{eq: fg-Fn-odd}. 
			  \begin{figure}[h]
					\centering
					{
					\includegraphics[width = \textwidth]{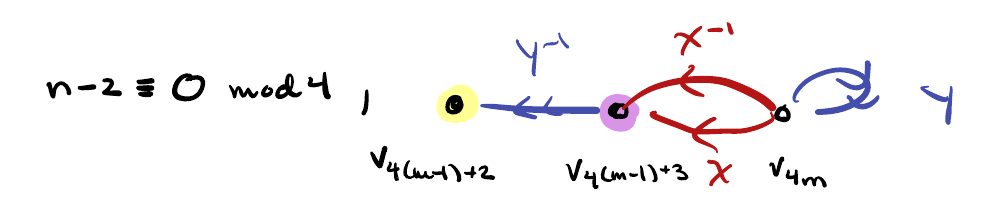}
					}
					\caption{The path induced by $g$ from $v_{4(m-1)+2}$ to $v_{4(m-1) + 3}$ when $n-2 \mod 4 \equiv 0$. The start point is highlighted in yellow and endpoint in purple.  
					}
					\label{fig: fg-Fn-mod0-3}
				\end{figure}
		   We have now shown that $V_n \subseteq v_0 \star \langle g \rangle$ and therefore that $v_0 \star \langle g \rangle = V_n$ for this case. 
		
		 \item[Case 2: $n-2 \equiv 1 \mod 4$.] Let $n-2 = 4m+1$. We know that $\{v_{4j}\}_{j=0}^m \subseteq v_0 \star \langle g \rangle$ by Equation \ref{eq: fg-Fn-even}. At the right endpoint, we have that $v_{4m} \cdot g = v_{4(m-1) + 3}$ as illustrated by Figure \ref{fig: fg-Fn-mod1-1}. 
			  \begin{figure}[h]
					\centering
					{
					\includegraphics[width = \textwidth]{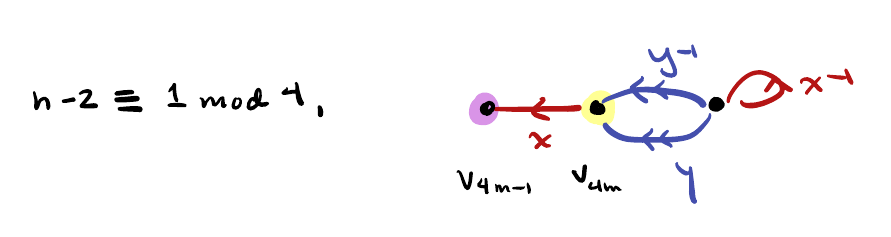}
					}
					\caption{The path induced by $g$ from $v_{4m-1}$ to $v_{4m}$ when $n-2 \equiv 1 \mod 4$. The start point is highlighted in yellow and endpoint in purple.  
					}
					\label{fig: fg-Fn-mod1-1}
				\end{figure}
		 
		 This means that $\{v_{3 + 4j}\}_{j=0}^{m-1} \subseteq v_0 \star \langle g \rangle$ by Equation \ref{eq: fg-Fn-odd}. At the left endpoint, we have however that $v_3 \cdot g = v_0$, as illustrated by Figure \ref{fig: fg-Fn-mod1-2}, thus closing the orbit. 
			  \begin{figure}[h]
					\centering
					{
					\includegraphics[width = \textwidth]{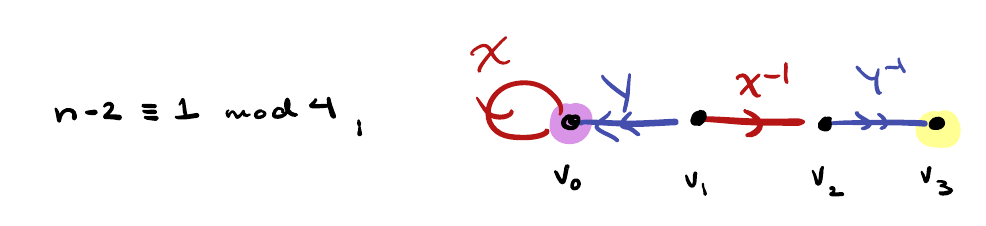}
					}
					\caption{The path induced by $g$ from $v_{3}$ to $v_{0}$ when $n-2 \equiv 1 \mod 4$. The start point is highlighted in yellow and endpoint in purple.  
					}
					\label{fig: fg-Fn-mod1-2}
				\end{figure}
		 
		 Therefore, $v_0 \star g = \{v_{4j}\}_{j=0}^{m} \cup \{v_{4j + 3}\}_{j=0}^{m-1} \not = V_n$ for this case. 
		 
		  \item[Case 3: $n-2 \equiv 2 \mod 4$.] Let $n-2 = 4m + 2$. As before, we know that $\{v_{4j}\}_{j=0}^m \subseteq v_0 \star \langle g \rangle$ by Equation \ref{eq: fg-Fn-even}. At the right endpoint, we have that $v_{4m} \cdot g = v_{4m+1}$, as illustrated Figure \ref{fig: fg-Fn-mod2-1}. Therefore, $\{v_{4j + 1}\}_{j=0}^m \subseteq v_0 \star \langle g \rangle$ by Equation \ref{eq: fg-Fn-odd}. 
			  \begin{figure}[h]
					\centering
					{
					\includegraphics[width = \textwidth]{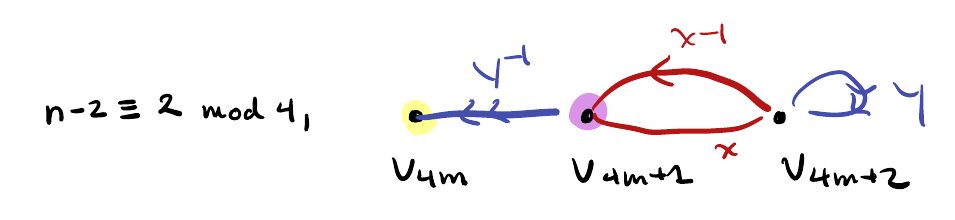}
					}
					\caption{The path induced by $g$ from $v_{4m}$ to $v_{4m+1}$ when $n-2 \equiv 2 \mod 4$. The start point is highlighted in yellow and endpoint in purple.  
					}
					\label{fig: fg-Fn-mod2-1}
				\end{figure}		  
				
			Back to the left endpoint, we have that $v_1 \cdot g = v_2$, as illustrated in Figure \ref{fig: fg-Fn-mod2-2}. Therefore, $\{v_{4j+2}\}_{j=0}^m \subseteq v_0 \star \langle g \rangle$ by Equation \ref{eq: fg-Fn-even}. 
				  \begin{figure}[h]
						\centering
						{
						\includegraphics[width = \textwidth]{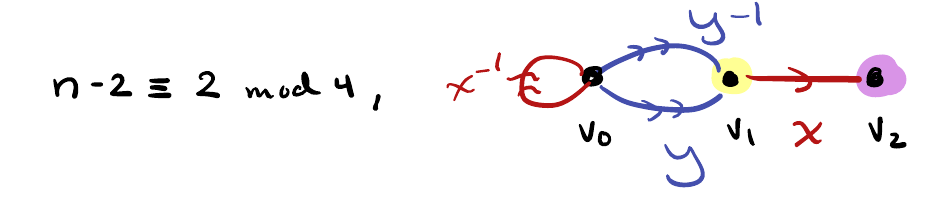}
						}
						\caption{The path induced by $g$ from $v_{1}$ to $v_{2}$ when $n-2 \equiv 2 \mod 4$. The start point is highlighted in yellow and endpoint in purple.  
						}
						\label{fig: fg-Fn-mod2-2}
					\end{figure}
					
			Back at the right endpoint again, we have that $v_{4m+2} \cdot g = v_{4(m-1) + 3}$ as illustrated in Figure \ref{fig: fg-Fn-mod2-3}. Therefore, $\{v_{4j+3}\}_{j=0}^{m-1} \subseteq v_0 \star \langle g \rangle$ by Equation \ref{eq: fg-Fn-odd}. 
				  \begin{figure}[h]
						\centering
						{
						\includegraphics[width = \textwidth]{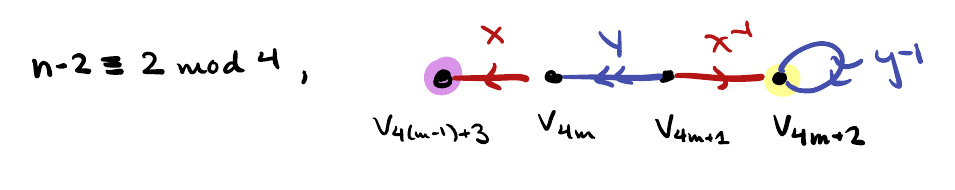}
						}
						\caption{The path induced by $g$ from $v_{4m+2}$ to $v_{4(m-1) + 3}$ when $n-2 \equiv 2 \mod 4$. The start point is highlighted in yellow and endpoint in purple.  
						}
						\label{fig: fg-Fn-mod2-3}
					\end{figure}
			
			We have shown that $v_0 \star \langle g \rangle = V_n$ for this case.
			
		\item[Case 4: $n-2 \equiv 3 \mod 4$.] Let $n-2 = 4m + 3$. As before, we know that $\{v_{4j}\}_{j=0}^m \subseteq v_0 \star \langle g \rangle$ by Equation \ref{eq: fg-Fn-even}. At the right endpoint, we have $v_{4m} \cdot g = v_{4m+3}$, as illustrated by Figure \ref{fig: fg-Fn-mod3-1}. 
		  \begin{figure}[h]
				\centering
				{
				\includegraphics[width = \textwidth]{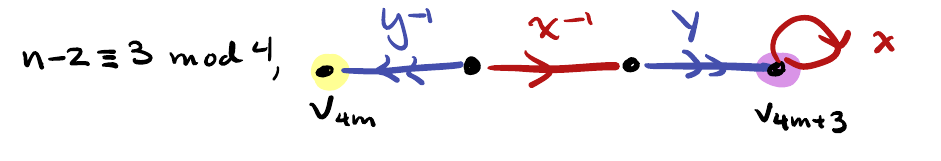}
				}
				\caption{The path induced by $g$ from $v_{4m}$ to $v_{4m + 3}$ when $n-2 \equiv 3 \mod 4$. The start point is highlighted in yellow and endpoint in purple.  
				}
				\label{fig: fg-Fn-mod3-1}
			\end{figure} 
		
		This gives us that $\{v_{4j + 3}\}_{j=0}^m \subseteq v_0 \star \langle g \rangle$ by Equation \ref{eq: fg-Fn-odd}. Going back to the left endpoint, we have that $v_3 \cdot g = v_0$ as illustrated in Figure \ref{fig: fg-Fn-mod3-2}, closing the orbit. 
		  \begin{figure}[h]
				\centering
				{
				\includegraphics[width = \textwidth]{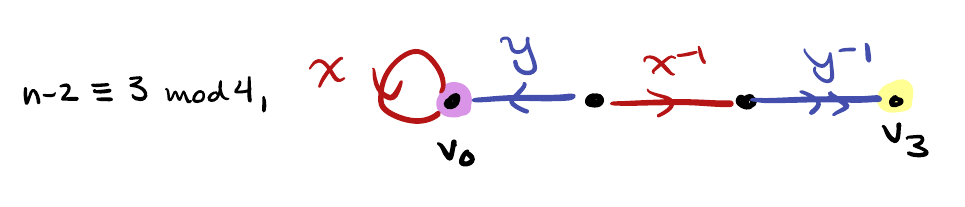}
				}
				\caption{The path induced by $g$ from $v_{0}$ to $v_{3}$ when $n-2 \equiv 3 \mod 4$. The start point is highlighted in yellow and endpoint in purple.  
				}
				\label{fig: fg-Fn-mod3-2}
			\end{figure} 
		
		Therefore, $v_0 \star \langle g \rangle = \{v_{4j}\}_{j=0}^{m} \cup \{v_{4j + 3}\}_{j=0}^m \not= V_n$.  
	\end{description}
	
	From the cases, we have shown that $v_0 \star \langle g \rangle = V_n$ if and only if $n - 2 \equiv 0 \text{ or } 2 \mod 4$. This means that $n$ must be even in those cases, proving the statement. 
\end{proof}

\begin{prop}
	Let $n$ be even. Then $F_{n} \times \mathbb{Z}$ has a finitely generated positive cone of rank bounded by $|Y| \cdot (n-1) = 7(n-1)$, where $Y$ given by Equation \ref{eqn: fg-Y-gen-F2xZ}. 
\end{prop}
\begin{proof}
	We have shown in Proposition \ref{prop: fg-Fn-g-orbit} that $g$ acts transitively on the right cosets of $F_n$. Let $\psi: H \to F_2 \times \mathbb{Z}$ be the isomorphism map of Equation \ref{eq: fg-Y}, and $h := \psi(b^6) = y\inv x\inv y x z\inv$. Let $F = \langle x, y, z\rangle$ be the lift of $F_2 \times \mathbb{Z}$ such that $\phi(F) = F_2 \times \mathbb{Z}$ where $\phi: F \to F_2 \times \bZ$ is the natural quotient by the normal subgroup $\langle \langle [x,y],[y,z] \rangle \rangle$. Let $\tilde{h} := \phi\inv(h)$. Since $h =  y\inv x\inv y x z\inv = (g, 1) \cdot (1, z\inv)$ and the action of the $(1,z\inv)$ factor is the identity on the $F_n \times \mathbb{Z}$ cosets, it follows that $h$ acts transitively on the $n-1$ $F_n \times \mathbb{Z}$ cosets in $F_2 \times \mathbb{Z}$. Therefore, $\tilde{T} = \{1, \tilde{h}, \tilde{h}^2, \dots, \tilde{h}^{n-2}\}$ is a Schreier transversal for $\phi\inv(F_n \times \mathbb{Z})$ in $\phi\inv(F_2 \times \mathbb{Z})$. 	
	
	As a result, we can use the Reidemeister-Schreier method (Theorem \ref{thm: RS}) to obtain $Z = \{\gamma(t,x)^* \mid t \in \tilde{T}, x \in \psi(Y), \gamma(t,x)^* \not= 1\}$ as a generating set for $F_n \times \mathbb{Z}$. 
	
	To determine the positivity of each element of $Z$, we check the positivity of each element of $\psi\inv(Z)$ which lives in the overgroup $\Gamma_2 = \langle a, b \mid ba^nb = a\rangle$ for which we know that every positive element can be written as an element of the semigroup $\langle a, b \rangle^+$. We have already established at the beginning of this section that $\psi\inv(h) = b^6$. Thus,  $\psi\inv(Z) = \{\gamma(t,x) \mid t \in \{1, b^6, \dots, b^{6 (n-1)}\}, x \in Y\}$. 

	By Lemma \ref{lem: fg-Fn-pos-gen-set}, we know that $\psi\inv(Z) \subseteq P_2 \cup \{1\}$, and since the identity is excluded from $Z$, $\psi\inv(Z) \subseteq P_2$. We claim that $Z$ is a generating set for $F_n \times \mathbb{Z} \cap P$, and will prove this the same as we did in the proof of Proposition \ref{prop: fg-pos-cone-gen}. 
	
	We first show that $\langle Z \rangle^+ \subseteq F_n \times \mathbb{Z} \cap P$, which is straightforward by construction. Indeed, every element of $Z$ is in $F_n \times \mathbb{Z}$ by definition of $\gamma(t,x)$ and since every element element of $Z$ is in $P$, it is clear that $Z \subseteq P$. 
	
	To show that $(F_n \times \mathbb{Z}) \cap P \subseteq \langle Z \rangle^+$, let $\pi: \psi(Y)^* \to F_n \times \bZ$ be the standard evaluation map. Take a word $w = x_1 \dots x_\ell$ with prefixed $w_i = x_1, \dots, x_i$ such that $\pi(w) \in (F_n \times \mathbb{Z}) \cap P$. Since $\pi(w) \in F_n \times \mathbb{Z}$, the map $\tau: F \to F(Z)$ is well-defined and $\pi(w) = \pi(\tau(w))$. 

	Furthermore,
	$$\tau(w) = \prod_{i=1}^\ell \gamma(\ovl{w_{i-1}},x_i)^*.$$
	Since $w \in \langle \psi(Y) \rangle^+$, $x_i \in \psi(Y)$ for $1 \leq i \leq \ell$. Thus $\gamma(\ovl{w_{i-1}}, x_i)^* \in Z$ for $1 \leq i \leq \ell$ and $\tau(w) \in \langle Z \rangle^+$.
	
	Finally, by construction $|Z| \leq |Y| \cdot |T| = 7(n-1)$. 
 \end{proof}
 
 The 2019 paper of Malicet, Mann, Rivas and Triestino \cite{MalicetMannRivasTriestino2019} states the following. 
 
 \begin{thm}
 	The group $F_n \times \mathbb{Z}$ has an isolated order if and only if $n$ is even.
 \end{thm}

Since $F_n \times \mathbb{Z}$ does not have an isolated order if $n$ is odd, and as a finitely generated positive cone implies an isolated order, it follows that $F_n \times \mathbb{Z}$ cannot have a finitely generated positive cone if $n$ is odd, hence showing Theorem \ref{thm: FnxZ-fg-P} as a corollary of this section. 

\begin{rmk}
	The proof of Proposition \ref{prop: fg-Fn-g-orbit} gives us some insight as to why $F_n \times \mathbb{Z}$ falls short of having a finitely generated positive cone when $n$ is odd. The geometry of $\Gamma_n$ is such that moving along the $b$-axis does not affect positivity unless an element is on the identity axis. Due to the parity of the markings of the $F_n$ immersion, the subgroups $\langle b \rangle \leq\Gamma_2$ acts transitively on the right cosets if and only if $n$ is even. 
\end{rmk}

%% file: chap/crossing-w-z.tex
\chapter{Attaching a stack via direct product with the integers}\label{chap: cross-Z}

In this chapter, we will show how $\bZ$ can be used as a stack in the direct product $G \times \bZ$ for certain groups $G$ which admit one-counter positive cones. 

\section{Main results}
It was found in 2012 by C. Rivas that the free products of left-orderable groups have no isolated left-orders \cite{Rivas2011}, and through a 2017 result of Hermiller and Sunic \cite{HermillerSunic2017NoPC}, it is known that free products do not admit regular positive cones. As such, one might assume that taking such groups and crossing them with $\bZ$ would not significantly alter the properties of their positive cones. That is not the case.

We now know that something inherent about positive cones change when  left-orderable groups are crossed with $\mathbb{Z}$. As discussed in Chapter \ref{chap: fg}, groups of the form $F_{2n} \times \mathbb{Z}$ have both isolated and non-isolated orders (see \cite{ClayMannRivas2018}, \cite{MalicetMannRivasTriestino2019}), and at least some of the isolated orders are finitely generated (see Chapter \ref{chap: fg} results), meaning that crossing a left-orderable group with $\mathbb{Z}$ does not only change the topological properties of its positive cones but computational ones as well. 

In this chapter, we present a more general result showing how the positive cone complexity of $G$ can change when crossed with $\bZ$. Dicks and \u{S}uni\'{c} found that if a group $G$ acts on a tree with trivial edge-stabilizer, such as free products, then it admit a quasi-morphism $\tau: G \to \mathbb{Z}$ which induce a left-order on $G$ \cite{DicksSunic2020}. This work was further generalised by Antol\'in, Dicks and \u{S}uni\'{c} to a further class of groups with non-trivial edge-stablizers, such as amalgamated free products \cite{AntolinDicksSunik2018}.

Adapting this under the lens of formal languages, we obtain that amalgamated free products of left-orderable groups with regular positive cones admit a one-counter positive cone.

\begin{thm}[See Section \ref{sec: crossz-transducer}]
\label{thm: crossz-amalg-1C}
Let $G$ be an amalgamated free product of groups admitting regular positive cones. Then $G$ admits a one-counter positive cone.
\end{thm}
 
When the amalgamation is trivial, this result is optimal in the view of a result of Hermiller and \u{S}uni\'{c} \cite{HermillerSunic2017NoPC} that says that free products do not admit regular positive cones.  Moreover, J. Alonso, Y. Antol\'in,  J. Brum and C. Rivas proved that certain free products with amalgamation also do not admit regular positive cones \cite[Theorem 1.6]{AlonsoAntolinBrumRivas2022}. 

What's more, the formal languages framework allows us to interpret crossing with $\mathbb{Z}$ as ``attaching a stack'' to our group in the sense $G \times \bZ$ has a positive cone of lowered complexity. 

\begin{thm}[See Section \ref{sec: crossz-stack-embedding}]
\label{thm: crossz-crossz-regular}
	Let $G$ be an amalgamated free product of groups admitting regular positive cones. Then $G \times \bZ$ admits a regular positive cone.
\end{thm}

We will see from the proof of this theorem (\ref{thm: crossz-aut-tau-prime}) that this interpretation can be made quite literal. 

Finally, we will look at some applications of Theorem \ref{thm: crossz-crossz-regular}. Most importantly, we get the following. 

\begin{cor}[See Section \ref{sec: crossz-applications}]
\label{cor: crossz-free-by-ZxZ}
	Let $G$ be a $F_n$-by$-\bZ$ group where where $F_n$ is a free group of rank $n > 1$. Then, $G$ admits a one-counter positive cone, and $G \times \bZ$ admits a regular positive cone.
\end{cor}

Using a 1999 result of Hermiller and Meier \cite{HermillerMeier1999}, we also get the following. 

\begin{cor}[See Section \ref{sec: crossz-applications}]\label{cor: crossz-artin-trees}
	Artin groups whose defining graphs are trees admit one-counter positive cones, and when crossed with $\bZ$, admit regular positive cones 
\end{cor}

Finally, we will use a result of \cite{Droms1982} to show the following. 

\begin{cor}[See Section \ref{sec: crossz-applications}]\label{cor: crossz-artin-droms}
Let $G$ be a right-angled Artin group based on a connected graph with no induced subgraph isomorphic to $C_4$ (the cycle with $4$ edges) or $L_3$ (the line with 3 edges).
Then $G$ has regular left-orders.
\end{cor}

The chapter will be divided as follows: we will first introduce the concept of a map $\tau: G \to \bZ$ being an ordering quasi-morphism. Then, we will show that if $G$ admits a $\tau$-transducer, then it has a one-counter positive cone. In particular, we will build an example of such a transducer when $G$ is an amalgamated free product with the ordering quasi-morphism given to us by \cite{AntolinDicksSunik2018}. Then, we will show that if a group $G$ admits a $\tau$-transducer such that $\tau$ is odd, then $G \times \bZ$ has a regular positive cone. 

Since the constructions build on top of each other, we will use the running example of $G = F_2$ to illustrate our ideas. Although we already know that $F_2 \times \bZ$ admits a finitely generated positive cone from Chapter \ref{chap: fg}, we think this example is still worth looking at as it does not depend on the embedded geometry of $F_2 \times \bZ$ into $\Gamma_n$ as in the previous chapter, and as it generalises naturally to amalgamated free products. 

The results discussed in this chapter are a longer form of what can be found in \cite[Section 5]{AntolinRivasSu2021}, where the exposition has been significantly improved. 

We encourage the reader to familiarise themselves with transducers as needed through Chapter \ref{chap: transducers}, as they will be heavily used in this chapter. 

\section{Ordering quasi-morphisms}
A {\it quasi-morphism} $\phi\colon G \to \bR$ is a function that is at bounded distance from a group homomorphism, i.e. there is a constant $D$ such that $|\phi(g)+\phi(h)-\phi(gh)|\leq D$. 

\begin{defn}\label{defn: crossz-ordering-quasimorphism}
	Let $G$ be a group. Then $\tau\colon G \to \mathbb{Z}$ is an \emph{ordering quasi-morphism} if it satisfies the following properties for all $g,h \in G$. 
	\begin{enumerate}
		\item $C=\{g\in G \mid \tau(g) = 0\}$ is a subgroup of $G$,
		\item $\tau(g) = -\tau(g\inv)$, 
		\item $\tau(g) + \tau(h) + \tau((gh)\inv) \leq 1$.
	\end{enumerate} 
	We call the subgroup $C$ the \emph{kernel} of $\tau$.
\end{defn}

The following is a slight generalization of \cite[Proposition 15]{Sunic2013} (see also \cite{DicksSunic2020}).
\begin{lem}\label{lem: alt-poscone}
Let $G$ be a group, $\tau: G \to \mathbb{Z}$ be an ordering quasi-morphism. Let $P = \{g \in G \mid \tau(g) > 0\}$. Then $P$ is a positive cone relative to $C$.\sidenote{See Chapter \ref{chap: LO} for relative positive cones.} 
\end{lem}
\begin{proof}
First observe that $P$ is disjoint from $P\inv$. Indeed, suppose that there exists $g \in P \cap P\inv$. Then, $g\inv \in P\inv \cap P$, and we get 
$$0 < \tau(g\inv) = -\tau(g) < 0,$$ 
a contradiction. 

Furthermore, since $\tau(g) = 0$ implies that $g \in C$, we have that $$G = P \sqcup P\inv \sqcup C.$$ Next, we check that $P$ is a semigroup. Let $g,h \in P$. Then $\tau(g)\geq 1$ and $\tau(h)\geq 1$. Also, we have that 
$$\tau(g) + \tau(h) + \tau((gh)\inv) - 1 \leq 0,$$
therefore, adding $\tau(gh)$ on both sides, 
\begin{align*}
\tau(gh) &\geq \tau(gh) + \tau(g) + \tau(h) + \tau((gh)\inv) - 1 \\
&= \tau(g) + \tau(h) -1 \geq 1.
\end{align*}
Thus $gh \in P$. 
\end{proof}

The following is an extra property for $\tau$ that will be extremely useful in this chapter. 
\begin{defn}
	We call an ordering quasi-morphism $\tau: G \to \bZ$ \emph{odd} if for all $g \in G$, $\tau(g)$ is either an odd integer or equal to zero. 
\end{defn}

Below we present an important construction of an odd ordering quasimorphism on amalgamated free products, from the 2018 work of Antol\'in, Dicks and \Sunik and the 2020 work of Dicks and \Sunik. %

	\begin{thm}[Ordering quasi-morphism on free products]\label{thm: crossz-quasimorphism-free-prod}\label{ex: cross-z-quasimorph-free-prod}
	Let  $\prec_I$ be a total order on the index set $I$. If $G=*_{i\in I} G_i$ is a free product of left-ordered groups $(G_i,\prec_i)$, then one can define an ordering quasi-morphism on $G$ as follows.	
	\begin{itemize}
		\item Given $g\in G$, the {\it normal form} of $g$ is an expression $g=g_1g_2\dots g_n$, where for $1 \leq i \leq n$ each $g_i$ (called {\it syllable}) is in some $G_{i_j}\setminus \{1\}$ and two consecutive syllables $g_i,g_{i+1}$ lie in different free factors. The normal form of an element $g\in G$ is unique. 
		\item A syllable $g_i$ is \emph{positive} if $1 \prec_{G_{i_j}} g_i$ and \emph{negative} if $1 \succ_{G_{i_j}} g_i$, i.e. if $g_i$ is positive (resp. negative) in the factor of $G$ it belongs to. 
		\item An {\it index jump} (resp. \emph{index drop}) in a normal form $g_1\dots g_n$ is a pair of  consecutive syllables $g_i g_{i+1}$ such that $G_{i_j} \prec_I G_{i_{j+1}}$ (resp. $G_{i_j} \succ_I G_{i_{j+1}}$). 
	\end{itemize}
	
	Let $\sharp$ stand in for ``number of''. We can define $\tau: G \to \mathbb{Z}$, 
	\begin{align*}
		\tau(g )&:= \sharp(\text{positive syllables in $g$})- \sharp(\text{negative syllables in $g$})\\
		&+ \sharp(\text{index jumps in $g$})- \sharp(\text{index drops in $g$}).
	\end{align*}
	
	Then, $\tau$ is an ordering quasi-morphism on $G$ with trivial kernel. \cite[Proposition 4.2]{DicksSunic2020}
\end{thm}

\begin{cor}[Ordering quasi-morphism on amalgamated free products]
	\label{cor: crossz-quasimorphism-amalg} 
	Let $I$ be ordering by $<_I$ and $G$ be the free product of $G_i$, $i\in I$ amalgamated along a common subgroup $C\leqslant G_i$ for all $i$.  
	For each $i$, assume that $P_i$ is a positive cone relative to $C$ (i.e. $G_i=P_i\sqcup C \sqcup P_i^{-1}$ and $P_i$ is a sub-semigroup of $G_i$).
	Then, any element $g\in G$ can be written (uniquely fixing transversals)  as $g=g_1g_2\dots g_n c$ with $c\in C$ and $g_j\in P_{i_j}\cup P_{i_j}^{-1}$, $1 \leq j \leq n$.
	Let $\tau: *_{i \in I}G_i \to \bZ$ be the ordering quasimorphism on the free product as in Example \ref{ex: cross-z-quasimorph-free-prod}. 	  Define
	\begin{align*}
	\tau'(g )&:=\tau(g_1\dots g_k).
	\end{align*}
	Then, $\tau'$ is an ordering quasi-morphism on $G$ with kernel $C$. \cite{AntolinDicksSunik2018}
\end{cor}
\begin{proof}
	Since by fixing transversals, any element $g\in G$ can be written uniquely as $g=g_1g_2\dots g_n c$ with $c\in C$ and $g_j\in P_{i_j}\cup P_{i_j}^{-1}$, $\tau'$ is well-defined and $\tau'(g) = 0 \iff g \in C$ by construction. 
\end{proof}

We worked out an example when $G = F_n$ in Example \ref{ex: LO-Fn-ordqm} in Chapter \ref{chap: LO} of the introduction. We encourage the reader to go back and refer to it. 

\section{Transducers computing ordering quasi-morphisms}
\label{sec: crossz-transducer}
	The previous results (\ref{ex: cross-z-quasimorph-free-prod}, \ref{cor: crossz-quasimorphism-amalg}) showcases the computational potential of ordering quasi-morphisms well. If $\tau$ is an ordering quasi-morphism, then $P_\tau := \{g \in G \mid \tau(g) > 0\}$ is one-counter if there exists a language $L$ with $\pi(L) = P_\tau$ such that for any $w = x_1 \dots x_n$ and prefixes $w_i = x_1 \dots x_i$, if we can keep track $\tau(\pi(w_i))$ via a stack, which is more or less how we computed $\tau$ as in Corollary \ref{cor: crossz-quasimorphism-amalg}. 
	
	We will formalise this intuition in this section using the machinery of transducers, which we introduced in Chapter \ref{chap: transducers}. 

Let $\bT = (S, X, Y, \delta, s_0, A)$ be a rational transducer in the sense of Chapter \ref{chap: transducers}. For $u\in X^*$, define $$\bT(u)\coloneqq\{v\in Y^* \mid (a,v) \in \delta(s_0,u) \text{ and } a\in A \}$$ as the output languages. That is, recall that $\delta: S \times X \to S \times Y$ is the transition function doing the transduction. That is $\delta(s, x) \ni (s', y)$ means going from state $s$ to $s'$ when reading $x$ and outputting the symbol $y$. We want to view $\bT$ as a machine taking inputs in $X^*$ and output in $Y^*$. We will write $\bT$ as the function $\delta\big| X^* \to Y^*$ restricted to its transduction. That is, by abuse of notation, $\bT$ as a set is the sextuple which defines a rational transducer, and $\bT$ as a function is the transduction associated with the transducer $\bT$. 

\begin{defn}
	For a language $L \subseteq X^*$, we define {\it the image} of $L\subseteq X^*$ under $\bT$ as  $$\bT(L)\coloneqq\{\bT(u)\mid u\in L\}.$$
	For $L\subseteq Y^*$, define the {\it inverse image} of $L$ under $\bT$ as the set $$\bT^{-1}(L)\coloneqq\{u\in X^* \mid (a,v)\in \delta(s_0, u) \text{ with } a\in \cA \text{ and }v\in L\}.$$
	Essentially, the image $\bT(L)$ is the output language of the transduction over $L$, and the inverse image $\bT\inv(L)$ is the input language of the transduction over $L$. 
\end{defn}

\begin{prop}\label{prop: closure transducers}
Full AFL classes  are closed under the image and inverse image of rational transducers.
\end{prop}
\begin{proof}
First observe that if $\bT$ is a rational transduction, then so is $\bT\inv$. Indeed, a transduction $\bT: X^* \to Y^*$ is rational if and only if $\bT(x) = \{y \in Y^* \mid (x, y) \in R\}$ where $R$ is a rational set. Let $\phi: X \times Y \to Y \times X$ be the map inverting coordinates, $\phi(x,y) = (y,x)$. By closure properties of rational sets under morphisms (Proposition \ref{prop: trans-rat-closure}) $\phi(R)$ is rational. Then, $\bT\inv: Y^* \to X^*$ is given by $\bT\inv(y) = \{x \in X^* \mid (y,x) \in \phi(R)\}$ and also satisfies the definition of a rational transducer. Thus, it suffices to show that full AFL classes are closed under the image of a rational transducer.

Now, Nivat's Theorem (\ref{thm: trans-Nivat}) says that  if $\bT$ is a rational transducer with input alphabet $X$,  output alphabet $Y$ and $L\subseteq X^*$ is a language, then there is a finite alphabet $Z$, a regular language $R \subseteq Z^*$ and homomorphisms $\phi\colon Z^*\to X^*$ and $\psi\colon Z^*\to Y^*$ such that $\bT(L)=\phi(\psi^{-1}(L)\cap  R)$. Thus, the image of a rational transducer can be written using operations which are closed under full AFL. This finishes the proof. 
\end{proof}

For the purposes of this chapter, we are interested in transducers which keep track of an ordering quasimorphism $\tau$. We formalise what this means below. 

\begin{defn}\label{defn: crossz-tau-transducer}
Let $G$ be a group finitely generated by $X$ with evaluation map $\pi_G: X^* \to G$. Let $\bZ$ be generated by $(T = \{t,t\inv \}, \pi_T)$. Let $\tau\colon G\to \bZ$ be an ordering quasi-morphism.

A {\it $\tau$-transducer} is a rational transducer $\bT$ with input alphabet $X$ and output alphabet $T := \{t, t\inv \}$ such that
\begin{enumerate}
\item $G =\pi_G(\bT^{-1}(T^*))$. That is, the input language of $\bT$ evaluates to $G$ in the sense that for every $g \in G$ there is a representative $w_g \in \bT^{-1}(T^*)$.
\item for every $w\in \bT^{-1}(T^*)$, $\tau(\pi_G(w))= \pi_T(\bT(w))$. That is, the exponent sum of the $\bT$ output under for $w$ is exactly equal to $\tau_G(\pi(w))$. 
\end{enumerate}
\end{defn}

Note that for a group $G$ and an ordering quasi-morphism $\tau$, a $\tau$-transducers does not describe a positive cone for $G$ on its own, as it lacks a counting mechanisms for the $t,t\inv$'s it outputs. To obtain such a mechanism, we will need to upgrade our apparatus to one-counter complexity. 

\begin{lem}\label{lem: crossz-1C-lang}
	The language of words which have more $t$ than $t\inv$ characters $$L=\{w\in \{t, t\inv \}^* \mid \sharp_t(w)>\sharp_{t^{-1}}(w)\}$$
	is in the class $\onecounter$ but not regular. 
\end{lem}
\begin{proof}
	Assume first for contradiction that $L$ is regular. Let $n$ be the pumping constant of the Pumping Lemma (\ref{lem: fsa-pumping}) for $L$. Clearly, $w = t^{-n} t^{n+1} \in L$, and so $w$ can be divided into substrings $w = xyz$ where $xy^kz \in L$ for any $k \geq 0$. However, since $|xy| \leq n$, $y = t^{-p}$ for some $p \geq 0$, and there exists some $N \geq 0$ such that $xy^N z \not\in L$. This is our desired contradiction. 
	
	To show that $L$ is one-counter, we refer to Example \ref{ex: pda-int-counter}.
\end{proof}

Naturally, this implies that a language positive cone for which we need to keep count of the $t,t\inv$'s is one-counter.

\begin{prop}\label{prop: transducer imply one-counter}
Let $G$ be a group finitely generated by $(X,\pi)$ and $\tau \colon G\to \bZ$ an ordering quasi-morphism with kernel $C$.
If a $\tau$-transducer exists, then $P_\tau=\{g\in G \mid \tau(g)>0\}$ is a $\onecounter$-positive cone relative to $C$.

In particular, if there is a $\onecounter$ language $L_C$ such that $\pi(L_C)=P_C$ is a positive cone for $C$, then $P_\tau\cup P_C$ is a $\onecounter$ positive cone for $G$.
\end{prop}

\begin{proof}
Let $\bT$ be a $\tau$-transducer. Observe that $P_\tau$ is a positive cone relative to $C$ by Lemma \ref{lem: alt-poscone}, and the language $L=\{w\in \{t^{-1},t\} \mid \sharp_t(w)-\sharp_{t^{-1}}(w)>0\}$ is a one-counter language by Lemma \ref{lem: crossz-1C-lang}. By Proposition \ref{prop: closure transducers}), we have that the class of one-counter languages is closed under inverse image of rational transducers, and thus that $\tilde L =\bT^{-1}(L)$ is a one-counter language.
By Definition \ref{defn: crossz-ordering-quasimorphism}, we get that $\pi (\tilde{L}) = P_\tau$.

Finally, suppose that $L_C$ is a $\onecounter$ language such that $\pi(L_C)$ is a positive cone for $C$. Since one-counter languages are closed under union,  we get that $\tilde{L}\cup L_C$ is a one-counter language representing $P_\tau\cup P_C$.
\end{proof}

Let $\tau$ be an ordering quasi-morphism on a free product constructed as in Corollary \ref{cor: crossz-quasimorphism-amalg}.
Our objective now is to construct $\tau$-transducer when the free factors of an amalgamated product over $C$ have regular positive cones relative to $C$. 
The automaton presented in the following lemma, which is a simple combination of a positive cone finite state automaton with its corresponding negative cone finite state automaton, will be the building block for the amalgamated product construction. 

\begin{lem}\label{lem: pm automaton}
Let $(X, \pi)$ be a finite generating set of $G$.
Suppose that $P$ is a $\Reg$-positive cone relative to $C\leqslant G$ with regular language $L$. Then, $L \cup L^{-1}$ is a regular language accepted by a finite state automaton without $\epsilon$-transitions $\bM$ as $\bM=(S,X, \delta, s_0,A)$ where the set of accepting states $A$  is a disjoint union $A = A^- \sqcup A^+$ such that the language $L^{-1}$ is accepted by $\bM^-=(S,X,\delta,s_0,A^-)$, and the language $L$ is accepted by $\bM^+=(S,X,\delta, s_0,A^+)$. 
\end{lem}

\begin{proof}
By assumption, there is a regular language $L\subseteq X^*$ that evaluates to $P$.
Moreover, by Lemma \ref{lem: negative cone in cC}, there is a regular language $L^{-1}\subseteq X^*$ that evaluates to $P^{-1}$.
We can always choose non-deterministic finite state automata $\bM^-$ and $\bM^+$ without $\epsilon$-moves accepting $L^{-1}$ and $L$ respectively (see Chapter 1). 
Viewing this automata as directed graphs, we obtain the desired automaton by identifying the start vertex of $\bM^-$ with the start vertex of $\bM^+$ and designing that vertex to be $s_0$, the start vertex of $\bM$. 
In Figure \ref{fig: mfsa} we see an example of this construction.
\end{proof}

\begin{figure}[ht]
\begin{center}
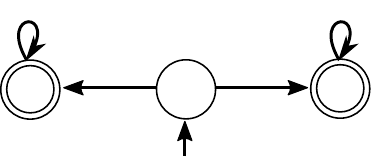
\end{center}
\caption{Example of the construction of Lemma \ref{lem: pm automaton} for $\bZ$, which gives an automaton that have states for recognizing the positive cone and states for recognizing the negative cone.}
\label{fig: mfsa}
\end{figure}

\begin{figure}[ht]
\begin{center}
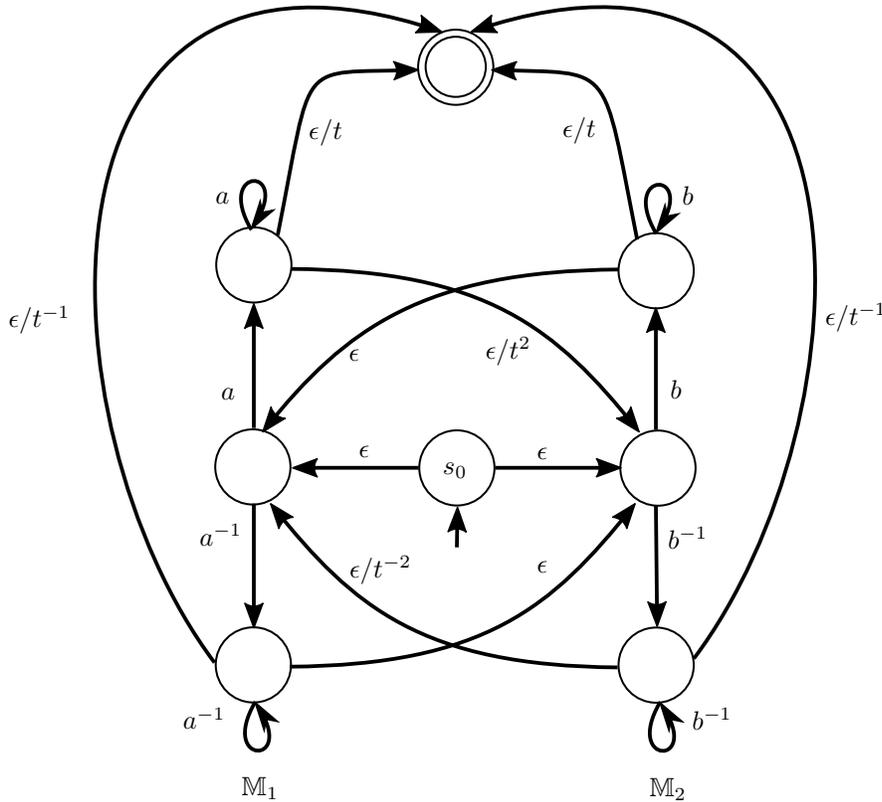
\end{center}
\caption{The $\tau$-transducer $\bT$ for $F_2 = \langle a \rangle * \langle b \rangle$ constructed following the proof of Proposition \ref{prop: crossz-transducer-amalg}. 
Observe that the three vertices on the left, and the three vertices of the right are copies of the automaton of Figure \ref{fig: mfsa}, labelled by $\bM_1$ and $\bM_2$ respectively. 
Observe also that $\bT^{-1}(\{t,t^{-1}\}^*)$ consists of all reduced words in $\{a,b,a^{-1},b^{-1}\}$. }
\label{fig: transducer-f2}
\end{figure}

\begin{ex}\label{ex: crossz-transducer-f2}
	Let $\tau$ be the quasimorphism of Theorem \ref{thm: crossz-quasimorphism-free-prod}. Figure \ref{fig: transducer-f2} illustrates how we construct a $\tau$-transducer $\bT$ for $F_2 = \langle a \rangle * \langle b \rangle$ with alphabet $X = A \cup B$ where $A = \{a,a\inv\}, B = \{b,b\inv\}$ and positive cone languages $a^+$ and $b^+$ respectively. 
	\begin{enumerate}
		\item [I.] $\bT$ has its own start state $s_0$ and a new final state $f$, which is the only accept state of $\bT$. There is an $\epsilon$-transition from $s_0$ to $f$, ensuring that the empty word is accepted. 
		\item [II.] The three vertices on the left (resp. right) labelled by $\bM_1$ (resp. $\bM_2)$ are from the construction in Lemma \ref{lem: pm automaton}, with starting group $G_1 = \langle a \rangle$ with alphabet $A = \{a,a\inv\}$ (resp. $G_2 =  \langle b \rangle$ with alphabet $B = \{b, b\inv\}$).
		\item [III.] Taking the union of $\bM_1$ and $\bM_2$ and adding some $\epsilon$-transitions allow us to jump from $\bM_1$ to $\bM_2$, capturing how an input word $w = w_1 \dots w_n \in X^*$ representing $g = g_1 \dots g_n$ where $g_j \in G_i$ or would jump from the $A$ and $B$ alphabets. More specifically
		\begin{enumerate}
			\item [i)] To represent that $w \not= \epsilon$ can start from either the $A$ or the $B$ alphabet, we add an $\epsilon$-move from $s_0$ to the start states of $\bM_1$ and $\bM_2$. 
			\item [ii)] 
				 \begin{itemize} 
					\item To represent the transition between $w_j, w_{j+1}$, we add an $\epsilon$-move from each accepting state of $\bM_1$ to the start state of $\bM_2$ and vice versa.
					\item We output symbols $t$ or $t\inv$ following $\tau$ of Theorem \ref{thm: crossz-quasimorphism-free-prod}. That is, we output $t^2$ (resp. $t^{-2})$ when our automaton passes from $\bM_1^+$ to $\bM_2$ (resp. $\bM_2^-$ to $\bM_1$) as it corresponds to an index jump (resp. index drop) in $w$ passing from the alphabet $A$ to $B$ on a positive syllable (resp. $B$ to $A$ on a negative syllable). When our automaton passes from $\bM_1^-$ to $\bM_2$ (resp. $\bM_2^+$ to $\bM_1$), we do not output anything as this corresponds to a negative syllable with an index jump (resp. a positive syllable with index rise).
 					\end{itemize}
			\item [iii)] 
				\begin{itemize} 
					\item	We add an $\epsilon$-move from each accepting state of $\bM_1$ and $\bM_2$ to the accepting state $f$ in case the automaton finishes reading $w$. 
					\item We output $t$ (resp. $t\inv)$ if it was in $\bM_1^+$ or $\bM_2^+$ (resp, $\bM_1^-$, $\bM_2^-$), corresponding to the last factor of $w$ being a positive (resp. negative) syllable. 
				\end{itemize}
		\end{enumerate}
	\end{enumerate}
	
	Observe that the input language $\bT\inv(T^*)$ is precisely the set of reduced words in $X^*$ (including the empty word), as the individual positive and negative cone languages $a^+, b^+, (a^{-1})^+, (b^{-1})^+$ were themselves reduced. Since every element $g \in F_2$ has a reduced representative, this satisfies Condition 1 of the definition. Moreover, the output words are given satisfy $\tau(\pi_{F_2}(w)) = \pi_{Z}(\bT(w))$ by construction, satisfying Condition 2 of the definition. 
\end{ex}

We now generalise and prove Example \ref{ex: crossz-transducer-f2} via the following proposition, which lets us construct a $\tau$-transducer on amalgamated free products, where $\tau$ is as in Corollary \ref{cor: crossz-quasimorphism-amalg}. 

\begin{prop}\label{prop: crossz-transducer-amalg}
	Let $G = \mathop * \limits_{i \in I} {}_C G_i$ denote the free product of the $G_i$ amalgamated over $C \leq G_i, i \in I$, where 
	\begin{enumerate}
		\item $I$ is a finite set,
		\item each $G_i, i \in I$ is a left-orderable group finitely generated by $(X_i,\pi_i)$ with positive cone $P_i$ associated regular positive cone languages $L_i^+ \subseteq X^*$, with inverse positive cone languages $L_i^-$. 
		\item $C$ is finitely generated by $(Y,\pi_C)$.
	\end{enumerate}

	Let $\tau\colon G\to \bZ$ be the ordering quasi-morphism for $G$ with kernel $C$ of Corollary \ref{cor: crossz-quasimorphism-amalg}. Let $(X=Y\sqcup (\bigsqcup_{i \in I} X_i),\pi)$ be a generating set for $G$  where for $x\in X_i$, $\pi(x):=\pi_i(x)$ and for $y\in Y$, $\pi(y):=\pi_C(y)$. 
	
	We claim that $\bT$ as given below is a $\tau$-transducer accepting the language  
	$$L = \{w = w_{i_1} \dots w_{i_m} v \in X^*\mid i_j\in I,  i_j \not= i_{j+1},  w_{i_j} \in L^+_{i_j} \sqcup L^-_{i_j}, v \in Y^*\}.$$
	
	\begin{enumerate}
		 \item[I.] 
 			Create a start $s_0$ and a final state $f$ for $\bT$. Add an $\epsilon$-transition from $s_0$ to $f$ in order to accept the empty word. 
  		 \item[II.] 	Add in the union of the finite state automata $\bM_i := (S_i,X_i, \delta_i, s_{i},A_i)$ for $i=1,\dots, n$ and $A_i = A_i^-\sqcup A_i^+$ as in Lemma \ref{lem: pm automaton}, such that the words accepted by $\bM_i$ on a state from  $A_i^{+}$ form a regular language $L_i^+$ evaluating onto $P_i$ and the words accepted by $\bM_i$ on  a state from  $A_i^{-}$ form a regular language $L_i^-$ evaluating onto $P_i^{-1}.$ 
		 \item [III.] To capture going from one factor of the amalgamation to another add the following $\epsilon$-moves (i) and transduction outputs (ii). 
			\begin{enumerate}
				\item 
					\begin{enumerate}
						\item Add an $\epsilon$-move from $s_0$ to every start state $s_i$ each $\bM_i$. 
						\item Do not output any words. 
					\end{enumerate}
				\item
					\begin{enumerate}
						\item Add an $\epsilon$-move from each accepting state $f_i \in A_i$ of $\bM_i$ to the start state $s_j$ of $\bM_j$ with $i\neq j$.
						\item Encode the contributions of index jumps and drops, and positive syllables and negative syllable as defined in Corollary \ref{cor: crossz-quasimorphism-amalg} as follows. Output $tt$ if they start on some $A_i^+$ and go to the start state of $\bM_j$ with $i<_Ij$. The $\epsilon$-moves of type II output $t^{-1}t^{-1}$ if they start on some $A_i^-$ and go to the start state of $\bM_j$ with $i>_Ij$. The other $\epsilon$-moves of type ii) do not output any word as their associated index rise and index drop cancel each other out. That is, they start either start on some $A_i^+$ and go to the start state of $\bM_j$ with $i >_I j$ or they start at some $A_i^-$ and go to the start state of $\bM_j$ with $i <_I j$. 
					\end{enumerate}
				\item
					\begin{enumerate}
						\item Add an $\epsilon$-move from each accepting states of $\bM_i$ to $f$, and state $f$ is the only accepting state of $\bM$.
						\item Output a $t$ if they start on some vertex of $A^+_i$ and output a $t^{-1}$ if they start on some vertex of $A^-_i$, thus keeping track of the last syllable. 
					\end{enumerate}
			\end{enumerate}
		\item[IV)] For each $y\in Y$, add loops on from $f$ to $f$ with label $y$, f
	\end{enumerate}
\end{prop}
\begin{proof}
	Let $\cL(\bT)$ be the language accepted by $\bT$. We want to show that $\cL = L$. Let $S$ be the set of states for $\bT$, and let $\delta: S \times X^* \to S \times Y^*$ be its transition function.  For simplicity of writing, we will assume that $\delta: S \times X^* \to S$, that is, we will not write the transduction for this part of the proof. 
	
	We will first show that $L \subseteq \cL(\bT)$. Let $w = w_{i_1} \dots w_{i_m} v \in L$. If $m=0$, then $w = v$, and $\delta(s_0, v) \supseteq \delta(f, v) \ni f$ by construction. 
	
	If $m > 0$, then write $w = w'v$, where $w' = w_{i_1} \dots w_{i_m}$. We first claim by induction on $w'$ that $\delta(s_0, w') \in A_{i_m}$. Suppose $m=1$. Then, $\delta(s_0, w') = \delta(s_0, w_{i_1}) \supseteq  \delta(s_{i_1}, w_{i_1}) \ni f_{i_j} \in A_{i_j}$ via $\epsilon$-transitions. Now, assume the hypothesis that $\delta(s_0,  w_{i_1} \dots w_{i_{m}}) \in A_{i_m}$. Then, if $w' =  w_{i_1} \dots w_{i_{m}} w_{i_{m+1}}$, $\delta(s_0, w') = \delta(f_{i_m}, w_{i_{m+1}})$ for some $f_{i_m} \in A_{i_m}$, we have that $\delta(f_{i_m}, w_{i_{m+1}}) \supseteq \delta(s_{i_{m+1}}, w_{i_{m+1}}) \supseteq  A_{i_{m+1}}$ again via $\epsilon$-transitions. This finishes the claim.
	
	Since there is an $\epsilon$-transition from each $f_i \in A_i$ to $f$, we have that $\delta(s_0, w) = \delta(s_0, w) \supseteq \delta(f_{i_m},v) \supseteq \delta(f,v) \ni f$. 
	
	Next, we want to show that $\cL(\bT) \subseteq L$. Let $w \in \cL(\bT)$. Then, $w$ is of the form $w = w'v$ for some $v \in Y^*$ by construction of $\delta$, where we assume that $w' \in X^*$. If $w' = \epsilon$, then $w \in L$ so assume that $w' \not= \epsilon$. As $X$ is a collection of disjoint alphabets $X_{i_j}$, we can write $w' = w_{i_j} \dots w_{i_m}$ such that each $w_{i_j} \in X_{i}$. By construction on $\bT$, we have that $\delta(s_0, w_{i_j-1}) \ni s_{i_j}$ for each $1 \leq j \leq m$, and that $\delta(s_{i_j}, w_{i_j}) \in f_{i_j})$, meaning that each $w_{i_j} \in L_{i_j}^+ \sqcup L_{i_j}^-$ by assumption on $\bM_{i_j}$. This shows what $w \in L$, and finishes the proof that $\cL(\bT) = L$. 
	
	Now, to show that $\bT$ is a $\tau$-transducer, we must show that it satisfies Condition $1$ and $2$ of Definition \ref{defn: crossz-tau-transducer}. 
	
	For Condition 1, notice that since $\cL(\bT) = L$, and $\pi(L) = G$, we are done. 
	
	For Condition 2, let $w = w_{i_1} \dots w_{i_m} v \in L$ again. As $C$ is in the kernel of $\tau$, $\tau(\pi(w_{i_1} \dots w_{i_m} v)) = \tau(\pi(w_{i_1} \dots w_{i_m})$, so we may assume that $w = w_{i_1} \dots w_{i_m}$. 	Assume the induction hypothesis such that $w = w'w_{i_{m+1}}$ and $w' = w_{i_1} \dots w_{i_m}$ such that $\pi_T(\bT(w') = \tau(\pi_G(w'))$.
	
	Observe that from the definition of $\tau$, we may deduce that $$\tau(\pi(w)) = \tau(w') - \text{sign}(w_{i_m}) + \Delta_{i_{m+1}}$$ where 
	$$\text{sign}(w_{i_m}) = \begin{cases}
		1 & \pi(w_{i_m}) \in L^+_{i_m} \\
		-1 & \pi(w_{i_m}) \in L^{-1}_{i_m}
	\end{cases}$$

	$$\Delta_{i_{m+1}} = \begin{cases}
		2 & i_{m} < i_{m+1}, \quad w_{i_{m+1}} \in L_{m+1}^+ \\
		0 & i_{m} < i_{m+1}, \quad w_{i_{m+1}} \in L_{m+1}^- \\
		0 & i_{m} > i_{m+1}, \quad w_{i_{m+1}} \in L_{m+1}^+ \\
		-2 & i_{m} > i_{m+1}, \quad w_{i_{m+1}} \in L_{m+1}^+.
	\end{cases}$$
	
	That is, by definition of $\tau$, if we take $\tau(\pi(w'))$ and we substract the last syllable contribution of $w'$ given by $\text{sign}(\pi(w_{i_m}))$ and add it back along the index rise or index fall when going from $i_m$ to $i_{m+1}$ as given by $\Delta_{i_{m+1}}$, then we get $\tau(\pi(w))$. 
	
	Now by definition of $\bT$, starting from $f_{i_m}$ (where $\bT$ would have left off after reading $w'$ if it did not jump to the final state) and reading $w_{i_{m+1}}$, $\bT$ outputs
	$$\begin{cases}
		tt & i_{m} < i_{m+1}, \quad w_{i_{m+1}} \in L_{m+1}^+ \\
		\epsilon & i_{m} < i_{m+1}, \quad w_{i_{m+1}} \in L_{m+1}^- \\
		\epsilon & i_{m} > i_{m+1}, \quad w_{i_{m+1}} \in L_{m+1}^+ \\
		t\inv t\inv & i_{m} > i_{m+1}, \quad w_{i_{m+1}} \in L_{m+1}^+.
	\end{cases}$$
	Thus, $\tau(\pi(w)) = \tau(\pi(w')) - \text{sign}(w_{i_m}) + \Delta_{i_{m+1}} = \pi_T(\bT(w))$. 
	
This completes the proof. 
\end{proof}

\begin{cor}
	Let $\bT$ be a $\tau$-transducer as in Proposition \ref{prop: crossz-transducer-amalg}. 
	The language  $$\bT^{-1}(\{w\in \{t,t^{-1}\}^*: \sharp_t(w)-\sharp_{t^{-1}}(w)>0\})$$ evaluates onto a positive cone for $G$ relative to $C$ and it is equal to $$L = \{w = w_{i_1} \dots w_{i_m} v \in X^*\mid i_t\in I,  i_j \not= i_{j+1},  w_{i_j} \in L^+_{i_j} \cup L^-_{i_j}, v \in Y^*, \tau(\pi(w)) > 0 \}.$$
\end{cor}

The following is a more precise restatement of Theorem \ref{thm: crossz-amalg-1C} in the introduction. 
\begin{cor}\label{cor: crossz-amalg-1C}
	Let $G_1,G_2, \dots G_n$ be finitely generated groups with a common subgroup $C$ such that each $G_i$ admits a $\Reg$-positive cone relative to $C$ and $C$ admits a $\Reg$-positive cone.
Let $G$ be the free product of the $G_i$'s amalgamated over $C$. Then $G$ admits a one-counter positive cone. 
\end{cor}
\begin{proof}
From Propositions \ref{prop: transducer imply one-counter}, $L$ of Proposition \ref{prop: crossz-transducer-amalg} is one-counter. 
\end{proof}

A simpler version of this corollary is the following. 

\begin{cor}\label{cor: A*BxZ-1C}
Let $A, B$ be groups admitting a $\Reg$-left-orders. 
Then $(A*B)$ admits a $\onecounter$ left-order.
\end{cor}

\begin{ex}[$\BS(1,m;1,n)$]\label{ex: crossz-BS(1,m;1,n)]}
	To introduce an interesting example of an amalgamated free product allowing a $\onecounter$ positive cone, let $n,m$ be two positive integers and consider the group
	 \begin{align*}
		 \BS(1,m;1,n)& \coloneqq \langle a, b, c \mid aba^{-1} = b^m,\, aca^{-1} = c^n \rangle \\
		 & \cong \langle a, b \mid aba^{-1} = b^m \rangle *_{\langle a \rangle} \langle a, c \mid  aca^{-1}= c^n \rangle \\
		 &\cong \BS(1,m) *_{\langle a \rangle} \BS(1, n). 
	 \end{align*}
	 
	Actions of $BS(1,m;1,n)$ on the closed interval $[0,1]$ have been studied in \cite{BonattiMonteverdeNavasRivas2017}, where it is showed that $BS(1,m;1,n)$ has no faithful action on $[0,1]$ by diffeomorphisms.	
\end{ex}

\section{Stack embedding theorem}
\label{sec: crossz-stack-embedding}
The main result of this section is a construction of a regular left-order on $G\times \bZ$ where $G$ admits an odd ordering quasi-morphism $\tau: G \to \bZ$ with kernel $C$. Our construction relies on two main ideas, which we will then formalise. 

\begin{enumerate}
	\item (Preserving $\tau$-oddness): We make a new ordering quasi-morphism $\tau': G \times \bZ$ from $\tau: G \to \bZ$ such that the output of $\tau'$ is still an odd integer for all inputs except for the identity if that was the case for $\tau$. This is accomplished in Proposition \ref{prop: crossz-tau-crossz}. 
	\item (Balanced language): For a positive cone defined by $\tau'$, the associated positive cone language uses the $\bZ$-factor in $G \times \bZ$ as a stack in place of the outputs of $\bT$. We define what this means formally in Definition \ref{defn: crossz-balanced-lang}. 
\end{enumerate}

Let us start by exploring the first idea formally. 

\begin{prop}[Preserving $\tau$-oddness]\label{prop: crossz-tau-crossz}
Let $G$ be a group, $C$ be a subgroup of $G$, and $\tau: G \to \bZ$ be an odd ordering quasi-morphism with kernel $C$. Let $\tau': G \times \bZ \to \bZ$ be defined as $\tau'((g,n)) := \tau(g) + 2n$. 
Then $$ P = \{(g,n) \in G \times \bZ \mid \tau'(g,n) > 0 \}$$
is a positive cone relative to $C\times\{0\}$ for $G \times \bZ$.
\end{prop}
\begin{proof}
We show that $\tau'$ satisfies the conditions of an ordering quasi-morphism as in Definition \ref{defn: crossz-ordering-quasimorphism}.
	\begin{enumerate}
		\item Since $\tau(g)$ is odd for every $g\in G-C$, we get that $\tau'((g,n))=0$ if and only if $g\in C$ and $n=0$.
		\item Observe that $-\tau'((g,n))=-\tau(g)-2n= \tau(g^{-1})-2n =\tau'((g^{-1},-n))$ for all $(g,n)\in G\times \bZ$.
		\item Finally, observe that $\tau'((g,n))+\tau'((h,m))-\tau'((g^{-1}h^{-1},-n-m))=\tau(g)+\tau(h)-\tau( g^{-1}h^{-1})\leq 1$ for all $(g,n),(h,m)\in G\times \bZ$.
	\end{enumerate}
\end{proof}

Let us now define more precisely what we mean by using the $\bZ$-factor as a stack. 

\begin{defn}[Balanced language]\label{defn: crossz-balanced-lang}
	Let $G$ be a group with generating set $(X, \pi_G)$ admitting an ordering quasi-morphism $\tau$. Let $\tau': G \times \bZ \to \bZ$, $\tau'((g,n)) = \tau(g) + 2n$ defines a positive cone $P$ relative to $C \times \{0\}$ as in Proposition \ref{prop: crossz-tau-crossz}. Let $\bZ$ have generating set $Z = \{z,z\inv\}$ with $\pi_\bZ(z) = 1$. Let $\pi: (X \sqcup Z)^* \to G \times \bZ$ be the evaluation map of $G \times \bZ$. Let $\rho_X : (X \sqcup Z)^* \to X^*$, $\rho_Z : (X \sqcup Z^*) \to Z^*$ be the projection maps mapping to the alphabets $X$ and $Z$ respectively and sending every other character to the empty word $\epsilon$. 
	
	A language $B \subseteq (X \sqcup Z)^*$ is \emph{balanced} for an ordering quasi-morphism $\tau': G \times \bZ$ if for all $g \in G$ with $g \not= 1_G$, there exists $w \in B$ such that $\pi_G(\rho_X(w)) = g$, and $\tau'(\pi(w)) = 1$. %
\end{defn}

Essentially, for a word $w_g \in X^*$ representing $g \in G$, its corresponding balanced word version $w \in B$ will also represent $g$, but with additional $z,z\inv$'s inserted such that its $\tau'$-evaluation is always $1$. From a balanced language for $\tau'$, it is straightforward to construct a positive cone language for $G \times \bZ$. 

\begin{ex}[$G = \bZ$]
	For demonstrative purposes, let us construct a balanced language for $G = F_1 = \langle x \rangle \cong \bZ$. Then, we can define an ordering quasimorphism $\tau$ as
	$$\tau(x^m) = \begin{cases}
		1 & m > 0 \\
		0 & m = 0 \\
		-1 & m < 0
	\end{cases}.$$
	We can verify that $$B = (x^+ \mid (x^{-1})^+ z)$$ is our desired language. 
	
	Indeed, for all $g \in \bZ$, with we have that we can either write $g$ in normal form as $g = x^m$ with $m > 0$ or $m < 0$. 
	
	If $m > 0$, then $w = x^m$ satisfies $\pi_G(\rho_X(w)) = x^m$ and $\tau'(\pi(w)) = \tau'(x^m,0) = 1$. 
	
	If $m < 0$, then $w = x^m z$ satisfies $\pi_G(\rho_X(w)) = x^m$ and $\tau'(\pi(w)) = \tau'(x^m, 1) = -1 + 2 = 1$. 
\end{ex}

\begin{prop}\label{prop: crossz-balanced-P}
	Let $B \subseteq (X \sqcup Z)^*$ be a balanced language for $G$ with $\tau'$. Then, $L := \{wz^m \mid w \in B, m \geq 0\}$ is a positive cone language for $P$. 
\end{prop}
\begin{proof}
	We claim that $L$ is a positive cone language for $P$. 
	
	Let us first show that $\pi(L) \subseteq P$. Let $wz^m$ with $w \in B$ and $m \geq 0$. Then $\tau'(wz^m) = \tau'(w) + 2m = 1 + 2m > 0$ for $m \geq 0$. 
	
	To show that $P \subseteq \pi(L)$, let $(g,n) \in P$. Then, $\tau'(g,n) = \tau(g) + 2n > 0 \iff n > -\frac{\tau(g)}{2}$. If $g = 1_G$, $\tau(g) = 0$ and $n > 0$, so $z^n \in \pi(L)$. Else if $g \not= 1_G$, then by definition of $B$, there exists $w \in B$ with $\pi(\rho_X(w)) = g$ and $\tau'(w) = 1$. Set $\pi(wz^m) = (g,n)$. For this to be true, we must solve $n = \pi_\bZ(w) + m$ for the $\bZ$-coordinate. Clearly,  $m = n - \pi_\bZ(\rho_Z(w))$ is a solution, and now we must show that $m \geq 0$ to have $wz^m \in L$. Observe that since 
	
	$$1 = \tau'(w) = \tau(g) + 2\pi_\bZ(\rho_Z(w)) \iff \pi_\bZ(\rho_Z(w)) = \frac{1 - \tau(g)}{2},$$
	
	we have that
	
	\begin{align*}
		m &= n - \pi_\bZ(\rho_Z(w)) \\
		&= n - \frac{1 - \tau(g)}{2} \\
		&> -\frac{\tau(g)}{2} - \frac{1 - \tau(g)}{2} \\
		&= -\frac{1}{2} \\
		&\geq 0
	\end{align*}
	since $m$ is assumed to be an integer. 	
	This completes the proof that $\pi(L) = P$.
\end{proof}

We claim that a balanced language can be implemented over a finite state automaton $\bM$ by modifying a $\tau$-transducer $\bT$ with the following steps. 
	\begin{enumerate}
			\item Let $g \in G$, then there exists $w_g \in \bT\inv(T^*)$ such that $\pi(w_g) = g$. 
			\item Since $\bZ$ commutes with $G$, we can insert $z,z\inv$ wherever we want into $w_g$ to create $w'$ such that $0 \leq 2\pi_\bZ(\rho_Z(w')) + \pi_T(\bT(w)) \leq 1$ at each syllable $1 \leq j \leq n$. Then, we can define $\mu(w') \in \{0,1\}$ such that $w'z^\mu$ is a balanced word.  
			\item In particular, when $\bT$ outputs $t$, we want to output something negative in $\bZ$. However, since each $z, z\inv$-value is multiplied by $2$ in the $\tau'$-equation, we accomplish this by outputting $z\inv,z$ for every \emph{two} $t$ (resp. $t\inv$). That is, whenever there is a $x/u$ arrow in $\bT$ with $u \in T^*$ and $x \in X$, we want to insert some compensating $xv$ arrow with $v \in Z^*$. 
 			 \item To keep track of the parity of outputted $t,t\inv$'s, we  take the set of states we had in $\bT$ previously, and take their product with $\ddagger = \{0,1\}$. The bit $\ddagger$ remembers the difference in offset, such that $\tau'(w) = \tau(\pi_G(\rho_X(w)) + \tau(\pi_G(\rho_Z(w)) \in \ddagger$. 
 		\end{enumerate}
 		
To make it easier to follow how we implement the Balancing Property, we start by illustrating our running example with $F_2 \times \mathbb{Z}$, then generalise our insights to $G \times \bZ$ as before.

\begin{ex}[$F_2 \times \mathbb{Z}$]\label{ex: crossz-reg-f2xZ}
	The transducer $\bT$ of Example \ref{ex: crossz-transducer-f2} modified with Idea 2 is shown in Figure \ref{fig: crossz-reg-f2xZ}. 
	
	\begin{figure}
		\includegraphics[width = \textwidth]{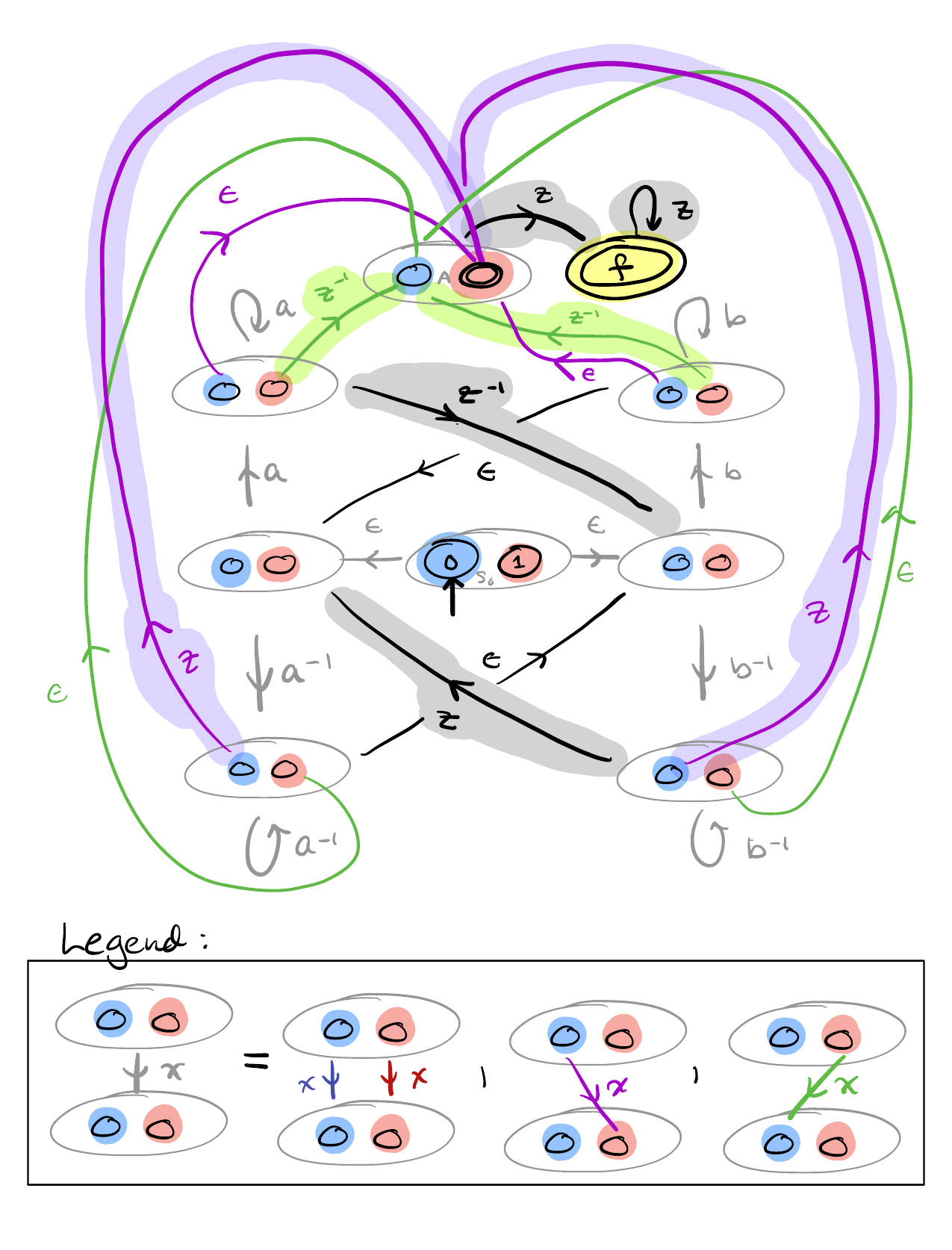}
		\caption{
		The finite state automaton accepting a regular positive cone language for $F_2 \times \bZ$ as given in Example \ref{ex: crossz-reg-f2xZ}. It is colour-coded for clarity as per the legend at the bottom. 
		
		Highlighted in blue: all the states $(s, 0)$ where $s \in S$ (on the left of a red state). 
		
		Highlighted in red: all the states $(s,1)$ where $s \in S$ (on the right of a blue state). 
		
		Highlighted in yellow: the new final state $f$. 
		
		Black or grey arrow: represents two arrows going from $(s,\dagger)$ to $(s', \dagger)$ where $\dagger \in \{0,1\}$ and $s,s' \in S$ ($\dagger$ is fixed in this case). 
		
		Purple arrow: an arrow going from a blue state to a red state. That is, an arrow going from $(s, 0)$ to $(s', 1)$ for $s,s' \in S$. 
		
		Green arrow: an arrow going from a red state to a blue state. That is, an arrow going from $(s,1)$ to $(s',0)$ for $s,s' \in S$. 
		
		Highlighted arrow: an arrow labeled with $z$ or $z\inv$ as transition, for emphasis. 
		}
		\label{fig: crossz-reg-f2xZ}
	\end{figure}

	Let $S$ be the set of states of $\bT$, and let $\ddagger = \{0,1\}$ We have modified the new set of states to be $S \times \ddagger$. Let $\dagger \in \ddagger$. 
	
	The transitions not involving $t$ or $t\inv$-outputs are kept essentially the same in the sense that if there exists a transition from $s$ to $s'$ with label $x \in X$, then it is replaced by two transitions $(s, \dagger)$ to $(s', \dagger)$ with label $x$.
	
	To take into account that $t$ (resp $t\inv$) counts for $\frac{1}{2} z$ (resp. $\frac{1}{2}z\inv$) in terms of $\tau$-evaluation when respecting the balancing property, we replace each $t^2$ (resp $t^{-2}$) output from $s,$ to $s'$ with a $z\inv$-labelled (resp. $z$-labelled) transition from $(s,\dagger)$ to $(s', \dagger)$, where $\dagger \in \{0, +1\}$. Each output $t$ (resp $t\inv$) from $s$ to $s'$ is replaced by a transition from either $(s,0)$ to $(s',+1)$ with an $\epsilon$-label or $(s,+1)$ to $(s',0)$ with a $z\inv$-transition (resp $z$-transition). 
\end{ex}

Next, we state and prove the general version of the example we just saw.

\begin{thm}\label{thm: crossz-aut-tau-prime}[Stack Embedding Theorem]
	Let $G$ be a finitely generated group with odd ordering quasi-morphism $\tau$ which has kernel $C$. Let $(X,\pi_G)$ be a finite generating set for $G$, and extend it to a generating set $(X \sqcup Z,\pi)$ of $G\times \bZ$,  where the elements of $Z = \{z,z\inv\}$ evaluate to $1$ and $-1$ on $\bZ$ respectively. Let  $\bT =(\cS,X, T=\{t,t^{-1}\}, \delta_{\bT}, s_0, \cA)$ be a $\tau$-transducer. Define  $\tau'\colon G\times \bZ\to \bZ$, 
	$$\tau'((g,n)) := \tau(g)+2n,$$ and let 
	$$P=\{(g,n)\mid \tau'((g,n))>0\}.$$
	
	Then, the finite state automaton $\bM$ defined below accepts a language $L = \{wz^m \mid w \in B, m \geq 0\}$ representing $P$, where $B$ is a balanced language for $\tau'$ as in Definition \ref{defn: crossz-balanced-lang}. 
		
	$\bM$ is defined as the following. 
	\begin{itemize}
		\item Its alphabet is $X\sqcup Z$.
		\item  Its set of states is given by $S_\bM = \left( S\times \{0,1\} \right) \cup \{f\}$ where $S$ are the states of $\bT$ and $f$ is a new final state.
		\item Its initial state is $(s_0,0)$.
		\item Its accepting states are $A_\bM = \{(\alpha,1) \mid \alpha \in A\}\cup \{f\}$, where $A$ is the set of accept state of $\bT$, and $f$ is a newly added final state. 
		\item Its transition function $\delta = \delta_\bM$ has two kinds of transitions.  
			\begin{itemize} 
				\item Transitions creating the balanced language. Let $\pi_T$ denote the evaluation map $\pi_T: T^*=\{t,t^{-1}\}^* \to \bZ$ with $t \mapsto 1$. 
					\begin{itemize}	
						\item For each $\dagger \in \{0, 1\}$, $s\in S$ and $x\in X\cup \{\epsilon\}$, and each pair $s,s'$ such that $\delta_\bT(s, x) \ni (s',u)$, we define a corresponding transition in $\delta_\bM((s,\dagger),xv) \ni (s', \dagger')$. Such that			\begin{itemize}
								\item $\dagger' := \dagger + \pi_T(u) \mod 2$, \quad $\dagger' \in \{0,1\}$,
								\item $\kappa := \frac{1}{2}( (\dagger' - \dagger) - \pi_T(u))$,
								\item $v := z^{\kappa}, \quad \kappa \not= 0$,
								\item $v := \epsilon, \quad \kappa = 0$. 
							\end{itemize}
						\end{itemize}
							(In particular, when $u = \epsilon$, $\dagger' = \dagger$ and $v = \epsilon$. When $u \not=\epsilon$, note that $v$ is designed to counterbalance the output $u$ of the $\tau$-transduction using the finite memory given by $(\dagger'-\dagger)$ to ensure that $\kappa$ is always integer-valued despite its $\frac{1}{2}$ factor which comes from $v$ being valued double that of $u$ in $\tau'$. %
				\item Transitions leading to the additional final state $f$. 
					\begin{itemize}
						\item $\delta_\bM ((\alpha,\dagger), z)= f$ for all $\dagger\in \{0,1\}$ and $\alpha \in A$ (i.e if $|\tau'(\pi(w))| \geq 1$, then $wz$ is accepted by $\bM$). 						
						\item $\delta_\bM(f,z)=f$ (i.e. if a word is accepted, we can keep reading $z$'s).  
					\end{itemize}
			\end{itemize}
	\end{itemize}
\end{thm}

\begin{proof}
Define $B$ as the set of words $w = w'z^\mu$ where $w' \in (X\sqcup Z)^*$ such that $\delta_M((s_0,0),w') \ni (\alpha,\dagger)$ for some $\alpha \in A$, the accept states of $\bT$, $\dagger \in \{0,1\}$, and 
$$\mu = \begin{cases}
	0 & \dagger = 1 \\
	1 & \dagger = 0
\end{cases}.$$ 

We first claim by induction on $\ell$ that $\tau'(w) = \dagger = \dagger_\ell$ for all $\ell \geq 0$, then show that this implies that $B$ is balanced. 

As each label of each edge of $\bM$ is of the form $xv$ with $x\in X\cup\{\epsilon\}$ and $v\in \{z,z^{-1}\}^*$, we can write 
 $w$ as  $x_1v_1x_2v_2\dots x_\ell v_\ell$ with prefixes $w_i' = x_1v_1x_2v_2\dots x_i v_i$ such that $x_i\in X\cup \{\epsilon\}$, $v_i\in Z^*$, and the sequence $(x_iv_i)_i$ is the sequence of the labels of the edges in the path in $\bM$ defined by $w'$, and $(s_i, \dagger_i) \in \delta((s_0, 0), w_i')$ for all $0 \leq i \leq \ell$. 

For $\ell = 0$, $w' = \epsilon$, and $\dagger_\ell = \dagger_0 = 0$ as we are in the start state. Since $\tau'(\epsilon) = 0$, the condition is satisfied. 

Next, assume the induction hypothesis that $\delta_\bM((s_0,0), x_1v_1\dots x_\ell v_\ell) \ni (s_{\ell},\dagger_\ell)$ and $\tau'(\pi(x_1v_1\dots x_\ell v_\ell)) = \dagger_\ell$. Assume that $v_\ell = z^{\kappa_\ell}$, and recall that $\tau(x_1 \dots x_{\ell+1}) = \pi_T(u_1 \dots u_{\ell+1})$ by definition of $\tau$-transducer. Now, we are dealing with the transition $\delta_\bM((s_\ell,\dagger_\ell), x_{\ell+1}v_{\ell+1}) \ni (s_{\ell+1},\dagger_{\ell+1})$, which is constructed from $\delta_\bT(s_\ell, x_{\ell+1}) \ni (s_{\ell+1}, u_{\ell+1})$. Recall that $\pi_T(u_{\ell+1}) = -2\kappa_{\ell+1} + (\dagger_{\ell+1} - \dagger_{\ell})$ by construction. 

We have 
\begin{align*}
\tau'(\pi(w')) & = \tau(\pi(x_1x_2\dots x_\ell x_{\ell +1})) + 2 \pi_\bZ(v_1\dots v_{\ell+1}) \\
& =\tau(\pi(x_1x_2\dots x_\ell)) + \pi_T(u_{\ell+1})  + 2 \pi_\bZ(v_1\dots v_{\ell}) + 2\pi_\bZ( v_{\ell+1})\\
& =\underbrace{\tau(\pi(x_1 x_2\dots x_\ell)) + 2 \pi_\bZ(v_1\dots v_{\ell})}_{\tau'(x_1 v_1 \dots x_\ell v\ell)} + \pi_T(u_{\ell+1})   + 2\pi_\bZ( v_{\ell+1})\\
& = \dagger_\ell + \pi_T(u_{\ell+1})  + 2 \pi_\bZ( v_{\ell+1}) \\
& = \dagger_\ell + \pi_T(u_{\ell+1})  + 2 \cdot \frac{1}{2}(\dagger_{\ell+1} - \dagger_\ell - \pi_T(u_{\ell+1})) \\
& = \dagger_{\ell+1}.
\end{align*}

This finishes our induction claim. It is now straightforward to observe that $\tau'(w'z^\mu) = \dagger + \mu = 1$ by definition of $\mu$. Finally, since every $g \in G$ has a representative word $w_g = x_1 \dots x_n$ in $\bT$, it follows that we can set $\rho_X(w') = w_g$ to obtain the representative word in $B$.  

Recall that $L = \{wz^m \mid w \in B, m \geq 0\}$. We want to show that the accepted language of $\cL(\bM) = L$. Let us first show that $L \subseteq \cL(\bM)$.

Suppose that $w'z^\mu \in B$. Then, $\delta((s_0,0), w'z^\mu) \supseteq \delta((\alpha, \dagger), z^\mu)$. Assume first that $\dagger = 0$. Then, $\mu = 1$. By construction of $\delta$, $\delta((\alpha, 0), z) = f$. Finally, $\delta(f,z^m) = f$ for any $m \geq 0$. Therefore, $\delta((s_0,0), w'z^\mu z^m) = f$ for $m \geq 0$. 

Assume now that $\dagger = 1$. Then, $\mu = 0$, and $\delta((s_0,0), w') = (\alpha, 1)$. Since $(\alpha, 1)$ is an accept state, and $\delta((s_0,0), w'z^0 z^m)) = f$ for $m > 0$, it follows that $w'z^0 z^m$ is accepted for $m \geq 0$. 

Let us now show that $\cL(\bM) \subseteq L$. There are two accept states in $\bM$, given by $(\alpha,1), \alpha \in A$ and $f$. Assume first that $\omega$ is such that $\delta((s_0, 0), \omega) = (\alpha, 1)$. By our induction proof above, we must have that $\tau'(\omega) = \dagger = 1$, and therefore that $\omega = w'z^0 \in B \subseteq L$. 

Assume now that $\omega$ is such that $\delta((s_0, 0), \omega) = f$. Then by construction of $\delta$ it must be that $\omega$ ends in $z$, and we can separate $\omega = w'z^{m'}$ where $\delta((s_0,0), w') = (\alpha,\dagger)$ for some $m' > 0$. We can thus rewrite $\omega = w'z^\mu z^m$ with $m \geq 0$, such that $\tau'(w'z^\mu) = \dagger + \mu = 1$ by assigning $\mu = 1$ when $\dagger = 0$ and $\mu = 0$ when $\dagger = 1$ as before. Then, $\omega = wz^m, m \geq 0$ where $w \in B$ as desired. 

By Proposition \ref{prop: crossz-balanced-P}, this means that $\cL(\bM) = L$ is a positive cone language for $P$. This completes the proof.
\end{proof}

As a corollary of the above, we obtain a more precise restatement of Theorem \ref{thm: crossz-crossz-regular} in the introduction. 

\begin{cor}\label{cor: crossz-gxz-regular}
Let $G_1,G_2, \dots G_n$ be finitely generated groups with a common subgroup $C$ such that each $G_i$ admits a $\Reg$-positive cone relative to $C$ and $C$ admits a $\Reg$-positive cone.
Let $G$ be the free product of the $G_i$'s amalgamated over $C$. Then $G\times \bZ$ admits a $\Reg$-left-order.
\end{cor}
\begin{proof}
By Proposition \ref{prop: crossz-tau-crossz}, $G$ admits a $\tau$-transducer where $\tau$ is an odd ordering quasi-morphism with kernel $C$.
Now by Proposition \ref{thm: crossz-aut-tau-prime}, $G\times \bZ$ has $\Reg$-positive relative to $C\times \{0\}$.
Finally, since $C\cong C\times \{0\}$ has $\Reg$-positive cones, we get from Lemma \ref{lem: lang-from-Prel} that $G\times \bZ$ has a regular positive cone.
\end{proof}

A simpler version of this corollary is the following. 

\begin{cor}\label{cor: A*BxZ-reg}
Let $A, B$ be groups admitting $\Reg$-left-orders. 
Then $(A*B)\times \bZ$ admits $\Reg$-left-orders.
\end{cor}

From the previous theorem and Example \ref{ex: crossz-BS(1,m;1,n)]} we have the following. 
\begin{cor}
For $n,m\geq 1$, the group  $BS(1,m;1,n)\times \bZ$ has a $\Reg$-positive cone. 
\end{cor}

\section{Applications to group extensions}
\label{sec: crossz-applications}

The following is a more precise version of Corollary \ref{cor: crossz-free-by-ZxZ} seen in the introduction. 

\begin{cor}\label{cor: free-by-ZxZ}
Suppose that $G$ is a $F_n$-by-$\bZ$ group, where $F_n$ is a free group of rank $n > 1$. Then, no lexicographic left-order $\prec$ on $G$ where $\bZ$ leads is regular.

However, there is a lexicographic left-order  on $G$ where $\bZ$ leads  that  is one-counter and extends to regular left-order on $G\times \bZ$.
\end{cor}
\begin{proof}
From Proposition \ref{prop: convex implies L-convex}, if $G$ admits a regular lexicographic left-order where $\bZ$ leads, then there will be a regular order on a finitely generated free group, contradicting  the theorem of Hermiller and \u{S}uni\'{c} \cite{HermillerSunic2017NoPC} that says that non-abelian free groups do not admit regular positive cones.

On the other hand, suppose that $f\colon G\to \bZ$ is a surjective homomorphism with kernel $F_n$, a free group of rank $n$. By Theorem \ref{thm: crossz-quasimorphism-free-prod}, there exists an odd ordering-quasimorphism  $\tau: F_n \to \bZ$ with trivial kernel, as well as an associated $\tau$-transducer (see Example \ref{ex: crossz-transducer-f2}, which can be generatlized from $F_2$ to $F_n$ in a natural way). 
By Theorem \ref{thm: crossz-amalg-1C}, $F_n$ admits a one-counter positive cone, and by Proposition \ref{prop: lang-quotient-leads}, since $\bZ$ has regular orders, $G$ and $G$ is $F_n$-by-$\bZ$, $G$ has a one-counter positive cone.

By Theorem \ref{thm: crossz-aut-tau-prime}, we see that $F_n\times \bZ$ has a regular positive cone.
Note that $F_n\times \bZ$ is the kernel of $\tilde{f}\colon G\times \bZ\to \bZ$ given by $(g,n)\mapsto f(g)$.
As $\bZ$ has regular positive cones, Proposition \ref{prop: lang-quotient-leads} guaranties that $G\times \bZ$ has regular positive that is lexicographic with leading factor the quotient, when viewing $G\times \bZ$ as a ($F_n\times \bZ$)-by-$\bZ$ extension.
The restriction of this order on $G$ is still lexicographic.
\end{proof}

Some families of groups that are known to be (finitely-generated free)-by-cyclic are provided in the next corollary. 

We first encourage the reader to look back at the notion of Artin group in Definition \ref{def: LO-Artin-groups} of Chapter \ref{chap: LO} if needed. 

\begin{proof}[Proof of Corollary \ref{cor: crossz-artin-trees}]
By a result of Hermiller and Meier \cite{HermillerMeier1999}, $A(\Gamma)$ admits a short exact sequence of the form $1 \to F_m \to A(\Gamma) \to \mathbb{Z} \to 1$ (and in fact  $m = \sum_{e_i \in \Gamma} (n_i - 1)$, where $n_i$ is the label of the edge $e_i$). Hence, $G$ falls into the hypothesis of Corollary \ref{cor: free-by-ZxZ} and the conclusion follows.
\end{proof}

\begin{proof}[Proof of Corollary \ref{cor: crossz-artin-droms}]
The proof is by induction on the size of the defining graph $\Gamma$.
If the graph has one vertex, then $G\cong \bZ$ and we know that $\bZ$ only has regular left-orders.
Now suppose that the defining graph has more than one vertex.
Droms \cite[Lemma]{Droms1982} observed that in that case there is a vertex connected to all other vertex of the graph $\Gamma$. 
That is $G\cong \bZ\times H$ where $H$ is right-angled Artin group based on a graph $\Gamma'$ that contains no induced subgraph isomorphic to $C_4$ or $L_3$. 
Note that $\Gamma'$ has fewer vertices that $\Gamma$.
Then, by induction, each connected component of $\Gamma'$ defines a subgroup of $H$ that has regular left-orders.
If $\Gamma'$ is connected, then $H$ has regular left-orders and so does $G$.
If $\Gamma'$ is not connected, then $H=A(\Gamma')$ is a free product of groups having regular left-orders. By Theorem \ref{thm: crossz-amalg-1C} and Theorem \ref{thm: crossz-crossz-regular} we get that $H\times \bZ$ has regular left-orders.
\end{proof}

\begin{rmk}
It follows from \cite[Theorem 4]{HermillerSunic2017NoPC}, that there is no positive cone $P$ on a right-angled Artin group defined over a graph of diameter $\geq 3$ such the set of  all geodesic words in the standard generating set that represent elements of $P$ form a regular language.

We remark that all the defining graphs of considered in  Corollary \ref{cor: crossz-artin-droms} have diameter at most two.
We also point out that the words of the regular left-orders constructed for $F_2\times \bZ$ do not induce (uniformly quasi)-geodesic paths.
\end{rmk}

%% file: figs/mfsa.pdf_tex
%% Creator: Inkscape 1.0.1 (c497b03c, 2020-09-10), www.inkscape.org
%% Accompanies image file 'mfsa.pdf' (pdf, eps, ps)
\begingroup%
  \makeatletter%
  \providecommand\color[2][]{%
    \errmessage{(Inkscape) Color is used for the text in Inkscape, but the package 'color.sty' is not loaded}%
    \renewcommand\color[2][]{}%
  }%
  \providecommand\transparent[1]{%
    \errmessage{(Inkscape) Transparency is used (non-zero) for the text in Inkscape, but the package 'transparent.sty' is not loaded}%
    \renewcommand\transparent[1]{}%
  }%
  \providecommand\rotatebox[2]{#2}%
  \newcommand*\fsize{\dimexpr\f@size pt\relax}%
  \newcommand*\lineheight[1]{\fontsize{\fsize}{#1\fsize}\selectfont}%
  \ifx\svgwidth\undefined%
    \setlength{\unitlength}{179.35838114bp}%
    \ifx\svgscale\undefined%
      \relax%
    \else%
      \setlength{\unitlength}{\unitlength * \real{\svgscale}}%
    \fi%
  \else%
    \setlength{\unitlength}{\svgwidth}%
  \fi%
  \global\let\svgwidth\undefined%
  \global\let\svgscale\undefined%
  \makeatother%
  \begin{picture}(1,0.42942464)%
    \lineheight{1}%
    \setlength\tabcolsep{0pt}%
    \put(0,0){\includegraphics[width=\unitlength,page=1]{mfsa.pdf}}%
    \put(0.46552185,0.17374266){\color[rgb]{0,0,0}\makebox(0,0)[lt]{\lineheight{1.25}\smash{\begin{tabular}[t]{l}$s_0$\end{tabular}}}}%
    \put(0.64592594,0.22372262){\color[rgb]{0,0,0}\makebox(0,0)[lt]{\lineheight{1.25}\smash{\begin{tabular}[t]{l}$a$\end{tabular}}}}%
    \put(0.89078028,0.39098549){\color[rgb]{0,0,0}\makebox(0,0)[lt]{\lineheight{1.25}\smash{\begin{tabular}[t]{l}$a$\end{tabular}}}}%
    \put(0.30117859,0.22372262){\color[rgb]{0,0,0}\makebox(0,0)[lt]{\lineheight{1.25}\smash{\begin{tabular}[t]{l}$a^{-1}$\end{tabular}}}}%
    \put(0.06853534,0.39510081){\color[rgb]{0,0,0}\makebox(0,0)[lt]{\lineheight{1.25}\smash{\begin{tabular}[t]{l}$a^{-1}$\end{tabular}}}}%
    \put(0.2835857,0.0085374){\color[rgb]{0,0,0}\makebox(0,0)[lt]{\lineheight{1.25}\smash{\begin{tabular}[t]{l}$\bM^{-}$\end{tabular}}}}%
    \put(0.61126803,0.00906529){\color[rgb]{0,0,0}\makebox(0,0)[lt]{\lineheight{1.25}\smash{\begin{tabular}[t]{l}$\bM^{+}$\end{tabular}}}}%
  \end{picture}%
\endgroup%

%% file: figs/transducer-f2.pdf_tex
%% Creator: Inkscape 1.0.1 (c497b03c, 2020-09-10), www.inkscape.org
%% Accompanies image file 'transducer-f2.pdf' (pdf, eps, ps)
\begingroup%
  \makeatletter%
  \providecommand\color[2][]{%
    \errmessage{(Inkscape) Color is used for the text in Inkscape, but the package 'color.sty' is not loaded}%
    \renewcommand\color[2][]{}%
  }%
  \providecommand\transparent[1]{%
    \errmessage{(Inkscape) Transparency is used (non-zero) for the text in Inkscape, but the package 'transparent.sty' is not loaded}%
    \renewcommand\transparent[1]{}%
  }%
  \providecommand\rotatebox[2]{#2}%
  \newcommand*\fsize{\dimexpr\f@size pt\relax}%
  \newcommand*\lineheight[1]{\fontsize{\fsize}{#1\fsize}\selectfont}%
  \ifx\svgwidth\undefined%
    \setlength{\unitlength}{380.88241457bp}%
    \ifx\svgscale\undefined%
      \relax%
    \else%
      \setlength{\unitlength}{\unitlength * \real{\svgscale}}%
    \fi%
  \else%
    \setlength{\unitlength}{\svgwidth}%
  \fi%
  \global\let\svgwidth\undefined%
  \global\let\svgscale\undefined%
  \makeatother%
  \begin{picture}(1,0.78778947)%
    \lineheight{1}%
    \setlength\tabcolsep{0pt}%
    \put(0,0){\includegraphics[width=\unitlength,page=1]{transducer-f2.pdf}}%
    \put(0.17119494,0.07026774){\color[rgb]{0,0,0}\makebox(0,0)[lt]{\lineheight{1.25}\smash{\begin{tabular}[t]{l}$a^{-1}$\end{tabular}}}}%
    \put(0,0){\includegraphics[width=\unitlength,page=2]{transducer-f2.pdf}}%
    \put(0.33535529,0.43470941){\color[rgb]{0,0,0}\makebox(0,0)[lt]{\lineheight{1.25}\smash{\begin{tabular}[t]{l}$\epsilon$\end{tabular}}}}%
    \put(0,0){\includegraphics[width=\unitlength,page=3]{transducer-f2.pdf}}%
    \put(0.4282545,0.32290451){\color[rgb]{0,0,0}\makebox(0,0)[lt]{\lineheight{1.25}\smash{\begin{tabular}[t]{l}$s_0$\end{tabular}}}}%
    \put(0.52121083,0.33783592){\color[rgb]{0,0,0}\makebox(0,0)[lt]{\lineheight{1.25}\smash{\begin{tabular}[t]{l}$\epsilon$\end{tabular}}}}%
    \put(0.52121083,0.22490126){\color[rgb]{0,0,0}\makebox(0,0)[lt]{\lineheight{1.25}\smash{\begin{tabular}[t]{l}$\epsilon$\end{tabular}}}}%
    \put(0.33583997,0.22039642){\color[rgb]{0,0,0}\makebox(0,0)[lt]{\lineheight{1.25}\smash{\begin{tabular}[t]{l}$\epsilon/t^{-2}$\end{tabular}}}}%
    \put(0.18762723,0.25181752){\color[rgb]{0,0,0}\makebox(0,0)[lt]{\lineheight{1.25}\smash{\begin{tabular}[t]{l}$a^{-1}$\end{tabular}}}}%
    \put(0.47052405,0.43682283){\color[rgb]{0,0,0}\makebox(0,0)[lt]{\lineheight{1.25}\smash{\begin{tabular}[t]{l}$\epsilon/t^2$\end{tabular}}}}%
    \put(0.29566805,0.65449273){\color[rgb]{0,0,0}\makebox(0,0)[lt]{\lineheight{1.25}\smash{\begin{tabular}[t]{l}$\epsilon/t$\end{tabular}}}}%
    \put(0.54654536,0.65621108){\color[rgb]{0,0,0}\makebox(0,0)[lt]{\lineheight{1.25}\smash{\begin{tabular}[t]{l}$\epsilon/t$\end{tabular}}}}%
    \put(0.65101147,0.24907173){\color[rgb]{0,0,0}\makebox(0,0)[lt]{\lineheight{1.25}\smash{\begin{tabular}[t]{l}$b^{-1}$\end{tabular}}}}%
    \put(0.67502807,0.07003426){\color[rgb]{0,0,0}\makebox(0,0)[lt]{\lineheight{1.25}\smash{\begin{tabular}[t]{l}$b^{-1}$\end{tabular}}}}%
    \put(0.20897101,0.39663791){\color[rgb]{0,0,0}\makebox(0,0)[lt]{\lineheight{1.25}\smash{\begin{tabular}[t]{l}$a$\end{tabular}}}}%
    \put(0.20342602,0.59046167){\color[rgb]{0,0,0}\makebox(0,0)[lt]{\lineheight{1.25}\smash{\begin{tabular}[t]{l}$a$\end{tabular}}}}%
    \put(0.66491828,0.58953643){\color[rgb]{0,0,0}\makebox(0,0)[lt]{\lineheight{1.25}\smash{\begin{tabular}[t]{l}$b$\end{tabular}}}}%
    \put(0.65312864,0.39794057){\color[rgb]{0,0,0}\makebox(0,0)[lt]{\lineheight{1.25}\smash{\begin{tabular}[t]{l}$b$\end{tabular}}}}%
    \put(-0.00158965,0.46819476){\color[rgb]{0,0,0}\makebox(0,0)[lt]{\lineheight{1.25}\smash{\begin{tabular}[t]{l}$\epsilon/t^{-1}$\end{tabular}}}}%
    \put(0.80637018,0.47066317){\color[rgb]{0,0,0}\makebox(0,0)[lt]{\lineheight{1.25}\smash{\begin{tabular}[t]{l}$\epsilon/t^{-1}$\end{tabular}}}}%
    \put(0,0){\includegraphics[width=\unitlength,page=4]{transducer-f2.pdf}}%
    \put(0.22935141,0.00443549){\color[rgb]{0,0,0}\makebox(0,0)[lt]{\lineheight{1.25}\smash{\begin{tabular}[t]{l}$\bM_1$\end{tabular}}}}%
    \put(0.63290598,0.00370234){\color[rgb]{0,0,0}\makebox(0,0)[lt]{\lineheight{1.25}\smash{\begin{tabular}[t]{l}$\bM_2$\end{tabular}}}}%
    \put(0.34509067,0.33916134){\color[rgb]{0,0,0}\makebox(0,0)[lt]{\lineheight{1.25}\smash{\begin{tabular}[t]{l}$\epsilon$\end{tabular}}}}%
  \end{picture}%
\endgroup%

%% file: chap/old-lang-convex.tex
\chapter{(Deprecated) Closure under language-convexity}\label{chap: old-lang-convex}

This chapter is dedicated to giving a more through exposition of the proof of Theorem 1.1 as originally stated in \cite{Su2020}. As stated in Chapter \ref{chap: closure-finite-index}, it has been deprecated as through the writing of this thesis an improvement has been found.\sidenote{As a result, there are rough edges from this chapter that have not been smoothed out for the interest of submitting the thesis on time. Apologies!}

We pick up from Chapter \ref{chap: closure-finite-index} Definition \ref{defn: lang-convex}. We will be referring back to the following lemma constantly through the chapter.\sidenote{Which is why I decided to give it a memorable name for practicality. I do not pretend this is a particularly novel result, as it is in fact very similar in both statement and proof to its analogous result for quasiconvex subgroups.} It is an older variant of Lemma \ref{lem: lang-conv-fg}. 

\begin{lem}\label{lem: lang-conv-fg-old}[Corresponding Words Lemma]
	Let $L \subseteq X^*$ with $\pi(L) \subseteq G$ and $H \leq G$. If $H$ is $L$-convex, then the subset $\pi(L) \cap H$ is finitely generated by 
	$$Y = \{y \in X^*: \bar y \in H, |y|_X \leq 2R + 1\}.$$ 
	
	In particular, for each $w \in L$ such that $\bar w \in H$, there corresponds a word $v$ such that $v \in Y^*$ and, viewing $v$ as a word in $X^*$, $\pi(w) = \pi(v)$. 
\end{lem}
\begin{proof}
Given $h \in \pi(L) \cap H$. We choose a word $w \in L$ such that $p_w$ is a path from $1$ to $h$ in $\Gamma$. Suppose that it is labelled $w = x_1,\dots, x_n$. By $L$-convexity of $H$, we may choose for $i = 1, \dots n$, a word $u_i$ of length at most $R$ so that $\bar w_i \bar u_i = h_i \in H$ ($u_0$ and $u_n$ are defined to be the empty word) and define $y_i := u_{i-1}\inv x_i u_i$ so that $|y_i|_X \leq 2R + 1$. Then for $i = 1, \dots, n$, $\bar y_i = h_{i-1} \bar w_{i-1} \bar x_i \bar w_i\inv h_i = h_{i-1}h_i \in H$. This is illustrated in Figure \ref{fig: fishing}.
	
	We have $v := y_1 \dots y_n$ and $h = \bar v$ so $\pi(L) \cap H$ is generated by a set $Y$ of such elements that lie in the ball of radius $2R + 1$ about the identity in $\Gamma$. It follows that $Y$ is a finite set, making $\pi(L) \cap H$ finitely generated as claimed. 
\end{proof}

Note that in the case where $L = X^*$, we recover Lemma \ref{lem: lang-conv-finite-index} as a corollary. Moreover, 

\begin{lem}
Let $Y$ be as in Lemma \ref{lem: lang-conv-fg-old}. Then $Y^*$ is a regular language. 	
\end{lem}
\begin{proof}
	Since $Y$ is a finite set composed a words in $X^*$, it is a regular language. Then $Y^*$ is also regular by Kleene closure of regular languages. 
\end{proof}

With this lemma, we are getting closer to our wanted result as $Y^*$ is regular and $\pi(L) \cap H \subseteq \pi(Y^*)$. However, $Y^*$ is not our desired regular language for $\pi(L) \cap H$ since it is not necessarily true that $\pi(Y^*) \subseteq \pi(L) \cap H$. Let us illustrate this with an example. 

\begin{ex}
Consider our example group $K_2$ with positive cone $P=\langle a, b\rangle^+$ and subgroup $H = \bZ^2$ of index $2$. A regular language for $P$ is $L = (a | b)^+$. The $L$-convexity parameter for $H$ is $1$, since every word $w \in L$ either has an even number of $a$'s (making $\bar w \in H$) or an odd number of $a$'s (making $\bar w \bar a \in H$). 

Consider the word $w = aba$. Using the algorithm outlined in the lemma above, we obtain $v = a^2 \cdot a\inv b a \cdot \varepsilon$, so $y = a\inv b a $ is a generator in $Y$. However, $\bar y = b\inv$, which is clearly not in $P$, so it cannot be true that $\pi(Y^*) \subseteq P \cap H$. 
\end{ex}

To restrict the image of $Y^*$ as in the Corresponding Word Lemma \ref{lem: lang-conv-fg-old} under evaluation map, we will intersect its automaton with a finite state automaton which accepts pairs of words which do not stray far from one another, and represent the same elements as the starting language $L$. 

\section{Pairs of Fellow-Travelling Words Form A Regular Language} \label{fellow-travel}
Let us formalise what we meant by ``pairs of words which do not stray far from one another''. 

\begin{defn}[Synchronous fellow travel]
	Let $(w,v) \in (X \times X)^*$. Let $w = x_1 \dots x_n$ and $v = y_1 \dots y_n$, and $w_i := x_1 \dots x_i$, $v_i := y_1 \dots y_i$ be the prefixes of $w$ and $v$, respectively. Let $M \geq 0$ be fixed. We say that $w$ and $v$ \emph{synchronously $M$-fellow travel} if $d(\bar w_i, \bar v_i) \leq M$ for $i = 0,\dots,n$. Figure \ref{fig: fellow-travel} illustrates this definition. 
	
\begin{figure}[h]{\includegraphics{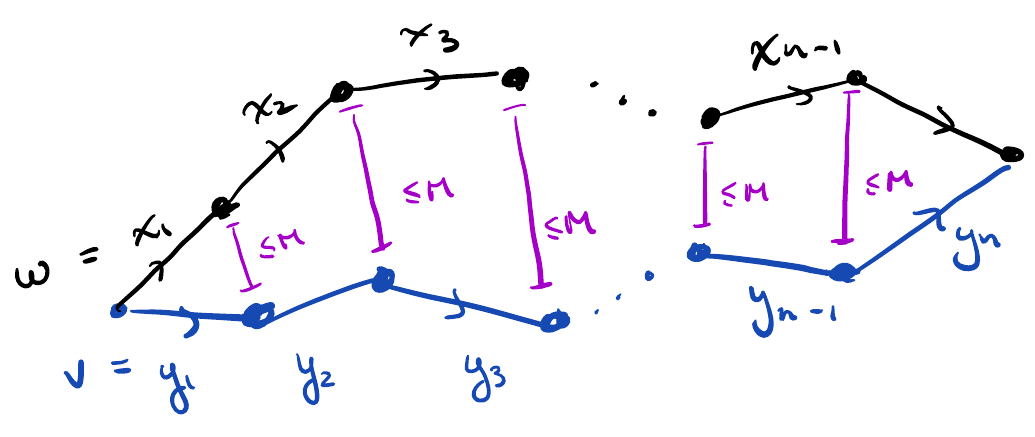}}
\caption{We illustrate a pair of words $(w,v)$, both in $X^*$ with the same length, as $M$-fellow travelling. On top, the word $w$ (in black), on the bottom, the word $v$ (in blue). The prefixes $\bar w_i$ and $\bar v_i$ are separated in the Cayley graph by paths of length $\leq M$. 
}
\label{fig: fellow-travel}
\end{figure}
\end{defn}

Given a fixed $M \geq 0$, pairs of $M$-fellow-travelling words can be recognised by a finite state automaton. The key idea of the automaton is that if $(w,v)$ $M$-fellow-travel, then remembering the difference between $w,v$ takes finite memory, since $\bar w\inv \bar v \in B_M$, a ball of radius $M$ about the identity. We formalise this. 

\begin{prop}[Fellow-Travelling Automaton]\label{prop: lang-convex-regular-ft}
Let $M \geq 0$, and $X$ be some finite alphabet. Define the finite-state automaton $\bA$ as follows. The automaton $\bA$ is the quintuple $(S, X \times X, \tau, A, s_0)$, where $B_M \subseteq V(\Gamma)$ is the set of group elements contained in a ball of radius $M$ around the identity, and $S := B_M \union \{\rho\}$, where $\rho$ denotes a fail state. Let $g \in B_M$, and define the transition function $\tau: S \times (X \times X)^* \to S$, as

\begin{align*}
& \tau(g, (x,y)) = \begin{cases} \bar{x}\inv g \bar{y} & \bar{x}\inv g \bar{y} \in B_M\\ 
\rho &  \bar{x}\inv g \bar{y} \notin B_M \end{cases},  \qquad g \in B_M\\
& \tau(\rho, (x,y)) = \rho \qquad\qquad\qquad\qquad\qquad\qquad \forall (x,y) \in (X \times X)^*. 
\end{align*}
Set $A = B_M$ to be the set of accepting states, and the initial state be $s_0 = 1_G$.
as illustrated in Figure \ref{fig: ft-fsa}.

\begin{figure}[h]{\includegraphics{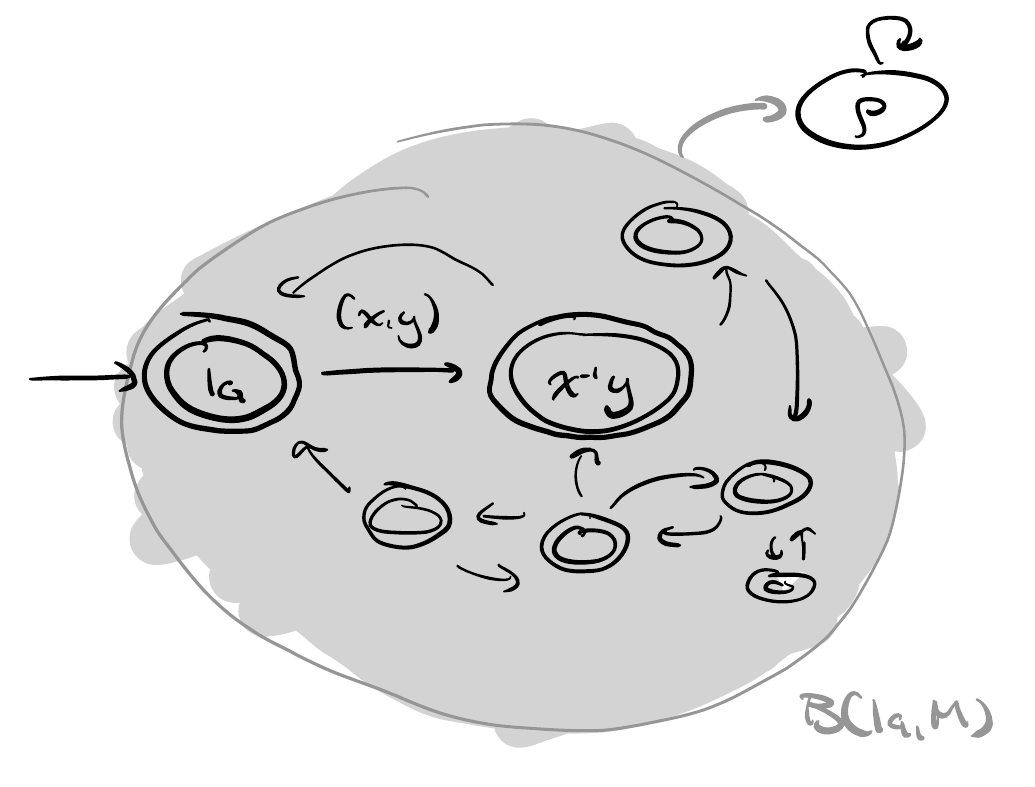}}
\caption{
A finite state automaton accepting the set of $M$-fellow-travelling word for some $M \geq 0$. The automaton starts at the identity and memorises the difference between the prefixes of the pair $(w,u)$, which is $\bar w_i\inv \bar u_i$ and determines whether it is in $B_M$ (in grey), a ball of radius $M$ around the identity. It accepts every input which has not been transitioned to the fail state at any point, as the fail state represent the difference being outside of $B_M$ at some point.  
}
\label{fig: ft-fsa}
\end{figure}

The the accepted language of $\bA$ is the language of pairs of words $(w,v) \in (X \times X)^*$ such that $u$ and $v$ synchronously $M$-fellow-travel $$\L_M := \{(w,v) \in (X \times X)^*  \mid d(\bar w_i, \bar v_i) \leq M, \quad i = 0, \dots, n\}$$
\end{prop} 
\begin{proof}
Let us first show that $\cL_M$ is accepted by the automaton by induction $|w|$. If $|w| = 0$, then $\tau(1_G), (\varepsilon, \varepsilon)) = 1_G$ which is accepted. Assume now that every pair of words in $\cL_M$ is accepted up to $|w| = n$. Let $(x,y) \in (X \times X)$ and $\tau(1_G,(w,u)) = g \in B_M$. Let $(wx,uy) \in \cL_M$ be a pair of words of length $n+1$. Then $\tau(1_G,(wx,uy)) = \tau(g,(x,y)) = \bar x\inv g \bar y \in B_M$ by assumption on $(wx,uv)$. 

To show that any accepted word is in $\cL_M$, we also induct on the length of the input $(w,v)$. If $|w| = 0$, then $\bar w\inv \bar v = 1_G$ which is accepted and indeed, $(w,v) \in \cL_M$. Assume any accepted word is in $\cL_M$ for length up to $n$, let $(x,y) \in X \times X$, and $(wx,vy)$ be some pair of words of length $n+1$. If $(w,v)$ do not $M$-fellow-travel, then it is rejected at some point and so is $(wx,vy)$. If $(w,v)$ is accepted, then $\tau(1_G,(wx,vy)) = \tau(g, (x,y))$ for some $g \in B_M$. Then, $\pi(wx)\inv \pi(vy) = x\inv \bar w \bar v \bar y = x\inv g y$ and $(wx,vy)$ is accepted if $x\inv g y \in B_M$, and thus $(wx,vy) \in \cL_M$, otherwise it is rejected. 
\end{proof}
\section{Fishing from the regular positive cone language}
Now, given our regular positive cone language $L$ the hope is that given $w \in L$ with $\bar w \in H$, we can recognise a corresponding word $v \in Y^*$ in the Corresponding Word Lemma \ref{lem: lang-conv-fg-old} via finite state automaton. The idea is to restrict $Y^*$ with another regular language. To do so, we use the automaton of Proposition \ref{prop: lang-convex-regular-ft}  we just found as a starting point, and make the necessary modifications. 

The first step is that we want to ensure our automaton accepts only when $\bar w = \bar v$. This is a straightforward fix as follows. 

\begin{cor}[Asynchronously Fellow-Travelling Automaton]
	Let $M \geq 0$. The language of synchronously $M$-fellow-travelling words such that $(u,v)$ spell the same element in $G$, $$\L_M := \{(u,v) \in (X \times X)^* : \bar u = \bar v \text{ and } d(\bar u_i, \bar v_i) \leq M, i = 1,\dots,n \}$$ is a regular language. 
\end{cor}
\begin{proof}[Sketch-proof]
	Modify the automaton of Proposition \ref{prop: lang-convex-regular-ft} so that its accept state is given by $A = \{1_G\}$. Then our finite state automaton accepts pairs words which $M$-fellow-travel, but whose final distance $d(\bar w, \bar v) = 0$. We leave the proof by induction as an exercise. 
\end{proof}

We will need more fixes. Indeed, suppose $w \in L$ and $v$ is its corresponding word in $H$ as in the Corresponding Word Lemma \ref{lem: lang-conv-fg-old}. Then $|w|_X = |v|_Y$, but clearly, $|w|_X \leq |v|_X$ and therefore $(w,v)$ do not fit the definition of synchronous fellow-travel. However, we would like to capture the notion of $w$ and $v$ asynchronously travelling, for example if each $x_i$, $w$ ``waits'' for $v$ to go through its $y_i$ syllable. To do so, we can insert waiting symbols in $w$, also known as a \emph{padding symbol}. 

\begin{defn}[Padding symbol]\label{defn: padding-symbol}
	Let $X$ be an alphabet with an evaluation map $\pi: X^* \to G$. A \emph{padding symbol}, usually denoted by $\$$ is a symbol which does not belong to $X$ which extends the alphabet of $X$ to $X^\$$, and such that the extended evaluation function $\pi': (X^\$)^* \to G$ is the same on $X$, $\pi'(x) = \pi(x), \quad x \in X$, and evaluates to the identity on the padding symbol, $\pi(\$) = 1_G$. 
\end{defn}
For the rest of the chapter, given a padded word $w^\$$, we may write its evaluation as $\bar w^\$ := \pi'(w^\$)$ in line with Definition \ref{defn: padding-symbol}.  

\begin{defn}[Padded version]\label{defn: padded-version}
Let $w \in X^*$. A \emph{padded version} $w^\$$ of $w$ is a word obtained from inserting padding symbols $\$$ between the characters of $w$. 
\end{defn}
Naturally, $\bar w^\$ = \bar w$. 

\begin{defn}[Asynchronous fellow travel]
	We say that $w$ and $v$ \emph{asynchronously} $M$-fellow-travel if there exists a pair of padded versions of $w$ and $v$ respectively $(w^\$, v^\$) \in \Y$ such that $(w^\$, v^\$)$ synchronously $M$-fellow-travel.
\end{defn}

Using these new definitions, we obtain the following desired lemma. 
  
\begin{lem}\label{lem: lang-conv-corr-words-async}
	Let $L$ be a language representing $P \subseteq G$ and $H \leq G$ be an $L$-convex subgroup with parameter $R$. Let $w \in L$ with $\bar w \in H$ and $v \in Y^*$ be its corresponding word as in the statement of Lemma \ref{lem: lang-conv-fg-old}. Then $w$ and $v$ asynchronously $M$-fellow-travel for $M := 3R + 1$. 
\end{lem}
\begin{proof}
	For $w = x_1 \dots x_n$ and $v = y_1 \dots y_n$ as in the Corresponding Word Lemma \ref{lem: lang-conv-fg-old}, observe that each prefix $\bar w_i$ is at distance at most $R$ from each prefix $\overline{v_i}$ for $i = 1, \dots, n$ by construction of $v$. The subpath connecting $\overline {v_{i-1}}$ to $\overline {v_{i}}$ is given by $y_i$ for $i = 1,\dots\,n$, (where $v_0$ is the empty word). Therefore, any vertex in such a subpath is at distance at most $2R + 1$ from $\overline {v_i}$ and hence at most $3R + 1$ from $\bar w_i$ by triangle inequality. Let $w^\$$ be a padded version of $w$ which has $|y_i|_X - 1$  padding symbols $\$$ added after each $x_i$ for $i = 1,\dots,n$, $w^\$:= x_1 \$^{|y_1| - 1} x_2 \dots x_{n-1} \$^{|y_{n-1}| -1} x_n$. Then, $(w^\$,v)$ synchronously $M$-fellow-travel as words in $X^*$. This is illustrated in Figure \ref{fig: fishing-aft}.
 
\begin{figure}[h]{\includegraphics{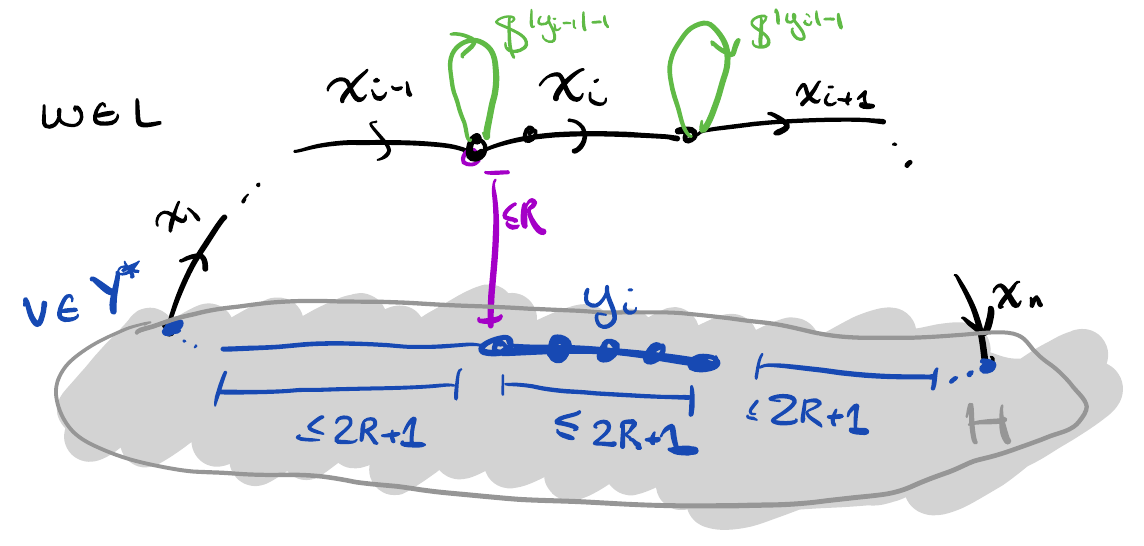}}
\caption{
We illustrate a word $w \in L$ (in black), the padding symbol $\$$ (in green) with exponents $|y_1 - 1$ added to to make $w^\$$ synchronous $M$-fellow-travel with corresponding word $v = y_1 \dots y_n \in Y^*$ (in blue). Each element $y_i$ has length $|y_i| \leq 2R + 1$ with respect to generating set $X$, and each $\bar w_i$ is at most $R$ from $\bar v_i$. By triangle inequality, this means that $w^\$$ and $\phi(v)$ (or $v$ viewed as a word in $X^*$) synchronously $M$-fellow travel.}
\label{fig: fishing-aft}
\end{figure}

\end{proof}

Moreover, asynchronously $M$-fellow-travelling words are also accepted by a finite state automaton, with very little extra work. 

\begin{cor}\label{cor: lang-convex-regular-aft}
Let $M \geq 0$. The language of pairs of words $(u,v) \in (\X \times \X)^*$ such that $u$ and $v$ synchronously $M$-fellow-travel and represent the same element in $G$, $$\L_M := \{(u,v) \in \Y : \bar u = \bar v \text{ and } d(\bar u_i, \bar v_i) \leq M, i = 1,\dots,n \}$$ is a regular language. 
\end{cor}
\begin{proof}
Let $Z$ be the arbitrary alphabet of the statement of Fellow-Travelling Automaton (Proposition \ref{prop: lang-convex-regular-ft}), and $X$ our current alphabet. We set $Z = X \cup \{\$\}$ with $\bar \$= 1_G$, so that the padding symbol is treated as any other letter in the alphabet on input, but evaluated as the identity, and modifying the automaton to have accept state $A = 1_G$, we obtain the desired result.
\end{proof}

Moreover, given a fixed regular language $L$ and words $w \in L$, we may extend $L$ to accept arbitrarily padded versions of $w$ while keeping the language regular. We denote the resulting padded language by $L^\$$.

\begin{defn}[Padded language]\label{padded-lang}
Let $L$ be the regular language accepted by the finite state automaton $\bA = (S, X, \tau, A, s_0)$. Fix $\$$ as a padding symbol. The \emph{padded language} $L^\$$ of $L$ is the language accepted by the automaton $\mathcal{A^\$} = (S, X, \tau^\$, A, s_0)$, where $\tau^\$$ is the function from $S \times (X \union \{\$\})$ to $S$ defined as
$$\tau^\$(s,x) := \begin{cases} \tau(s,x) & s \in S, x \in X \\ 
s & s \in S, x = \$. \end{cases}$$
\end{defn}

For the above definition to make sense, we need to show the following is true. 
\begin{lem}
	The language $L^\$$ consists of all padded versions of the words in $L$.
\end{lem}
\begin{proof}
	Let $w \in L$ and $w^\$$ be a padded version of $w$. We will show that $\tau^\$(s_0, w^\$) = \tau(s_0, w)$ on the number of padding symbols of $w^\$$. 
	Suppose that there are zero padding symbols. Then $w^\$ = w$ and $\tau^\$(s_0, w^\$) = \tau(s_0, w)$.
		
	Suppose the statement is true for up to $m$ padding symbols and that $w^\$$ has $m+1$ padding symbols. Let $w = x_1 \dots x_n$ and $w'$ be the word formed by $w^\$$ with the last padding symbol removed. Without loss of generality, we may write $w^\$ = w'_j \$ x_{j+1} \dots x_n$. Then, 
	\begin{align*}
		\tau^\$(s_0, w^\$) &= \tau^\$(\tau^\$(s_0, w'_j), \$x_{j+1}\dots x_n) \\
		&= \tau^\$(\tau(s_0, w_j), \$ x_{j+1} \dots x_n) \\
		&= \tau^\$(\tau(s_0, w_j), x_{j+1} \dots x_n) \\
		&= \tau(\tau(s_0, w_j), x_{j+1} \dots x_n) \\
		&= \tau(s_0, w).
	\end{align*}
	Since $\tau^\$(s_0, w^\$) = \tau(s_0, w)$, it is clear that $w^\$ \in L^\$ \iff w \in L$. 
\end{proof}

Using everything we have presented above, we finally construct the language we will intersect with $Y^*$ in the Corresponding Word Lemma \ref{lem: lang-conv-fg-old} so that we only pick words in $Y^*$ which represent positive cone elements given by $L$. 

\begin{lem}\label{L-tilde} Let $L \subseteq X^*$ be a regular language, let $M\geq 0$, and let $\mathcal{L}_M$ be the language of synchronously $M$-fellow-travelling pairs of words in $\Y$ such that each pair spell the same element in $G$, as defined in the Asynchronously Fellow-Travelling Automaton (Corollary \ref{cor: lang-convex-regular-aft}). Then, $$\tilde L_M := \{v \in \Z 
: \exists w \in L^\$ \text{ such that } (w,v) \in \mathcal{L}_M\},$$ is a regular language and $\pi(\tilde L_M) = \pi(L)$.
\end{lem}
To prove this lemma, we will make use of predicate calculus. A proof of the well-known results below can be found in the background of this thesis or in \cite[Section 1.4]{Epstein1992} and involve constructing the appropriate automata.
\begin{thm}[Predicate calculus]\label{prop-cal}
Given regular languages $L_1$ and  $L_2$ over the same finite alphabet $Y$, the following languages are also regular.

\begin{itemize}
\item $L_1 \times L_2$ where  $L_1 \times L_2 = \{(w,v)\in (Z \times Z)^* : w\in L_1, v\in L_2\}$.
\item $L_1 \intersect L_2$. 
\end{itemize}
Moreover, if $L_3$ is a regular language over a product of finite alphabets $Z_1 \times \dots \times Z_n$ and $\text{Proj}_i : (Z_1 \times \dots \times Z_n)^* \to Z_i^*$ is the projection map on the $i$th coordinate,  then $\text{Proj}_i(L_3)$ is a regular language.
\end{thm}

\begin{proof}[Proof of Lemma \ref{L-tilde}] We will be using predicate calculus (Theorem \ref{prop-cal}) several times. Let $$L' := \{(w,v) \in \Y : w \in \padL \text{ and } v \in \Z\}.$$ Observe that $L' = L^\$ \times (\X)^*$ so it is regular. Set $$L_M'' :=  \{(w,v) \in \Y : w \in \padL \text{ and } (w,v) 
\in \mathcal{L}_M\}.$$ Since $L_M'' = \L_M \intersect L'$,  $L_M''$ is also regular. Set $$
\tilde L_M := \{v \in \Z 
: \exists w \in \padL \text{ such that } (w,v) \in \mathcal{L}_M\} $$
and observe that $\tilde L_M = \text{Proj}_2(L_M'')$, so it is regular. Finally, we observe that 
\begin{align*}
\pi(L) 
&= \pi(\text{Proj}_1(L')), & \qquad \pi(L) = \pi(L^\$) \\
&= \pi(\text{Proj}_1(L_M'')), & \qquad (w,v) \in L' \imp (w,v) \in L_M'' \\
&= \pi(\text{Proj}_2(L_M'')), & \qquad (w,v) \in L''_M \imp \pi(u) = \pi(v)\\
&= \pi(\tilde L_M). \end{align*} \end{proof}

\section{Language-convexity result}
We are ready to put everything we have worked towards together! Note that the statement below is a slightly modified from the one in \cite{Su2020}. 

\begin{thm}\label{thm: lang-conv-thm}
	Let $X$ be a finite set containing its own inverses, $X = X^{-1}$. Set $G = \langle X \rangle$. Let $L$ be a regular language, and let $P = \pi(L)$ where $\pi$ is the evaluation map onto $G$. Let $H$ be a subgroup of $G$ that is $L$-convex, and $Y$ be a generating set for $H \cap P$ as in the Corresponding Word Lemma \ref{lem: lang-conv-fg-old}. Let $\tilde L_M$ be the regular language given by Lemma \ref{L-tilde}. Fix $M := 3R + 1$. Then $$L_H = Y^* \intersect \tilde L_M$$ is a regular language for $P \cap H$. 
\end{thm}
\begin{proof}
Since $Y^*$ and $\tilde L_M$ are both regular, so is $L_H$ by closure properties of regular languages. We will argue that $L_H$ is a language representing $H \intersect P$. 

We first show $\pi(L_H) \supseteq H \intersect P$. Let $h \in P \cap H$ and $w \in L$ such that $\bar w \in H$. Let $v \in Y^*$ be the word corresponding to $w$ as in the Corresponding Word Lemma \ref{lem: lang-conv-fg-old}. By construction, $\bar v = \bar w \in P \cap H$. We have already shown that $v \in \tilde L_M$ for $M = 3R + 1$ in Lemma \ref{lem: lang-conv-corr-words-async} and thus $\pi(L_H) \supseteq H \intersect P$. 

We conclude by proving that $\pi(L_H) \subseteq H \intersect P$. If $v \in L_H$, then $\bar v \in Y^*$. Since $\pi(Y^*) \subseteq H$, we have that $\bar v \in H$. Moreover, $\bar v \in \pi(\tilde L_M) = \pi(L) = P$, so we obtain that $\pi(L_H) \subseteq H \intersect P$.  
\end{proof}

We obtain the following cleaner statement as a corollary. 
\begin{cor}
Let $G$ be a finitely generated group with a regular positive cone. If $H$ is a finite index subgroup, then $H$ also admits a regular positive cone. \cite{Su2020}
\end{cor}

\section{Interpreting the result}

Unfortunately, the language constructed in the statement of Theorem \ref{thm: lang-conv-thm} does not lend itself well in practice to obtaining a concrete regular language, and should be interpreted as a complexity bound on positive cones of subgroups. In some sense, the language $L_H$ is too big to be very useful as there are too many representative of the same element in $\pi(L) \cap H$.  

This is because by padding $L$, $L_H$ as in Theorem \ref{thm: lang-conv-thm} accepts many words representing the same elements which have ``meandering additions'' (denoted by $u$ in the lemma below) which asynchronously fellow-travel with words in $L^\$$ as follows. 

\begin{rmk}
	Let $w \in L$ and $v$ such that $(w,v) \in \cL_M$ as in the Fellow Travelling Automaton (Proposition \ref{prop: lang-convex-regular-ft}). Then, if $d(\bar w, \bar u) = r$, and $|u|_X \leq M - r$, then $(w^{\$^{|u|}}, vu)$ also synchronously $M$-fellow-travel. 
\end{rmk}
\begin{proof}
	This is since $d(\bar w, \pi(vu)) \leq d(\bar w, \bar v) + d(\bar v, \pi(vu)) \leq r + |u|_X$. 
\end{proof}

\begin{ex}\label{ex: lang-conv-thm-Klein-Zsq}
	We illustrate this with an example word. Let us go back to the $G = K_2, P = \langle a, b \rangle^+, H = \bZ^2$ case. Suppose $w = aba \in L$. Then, $W^\$ = \{\$\}^* a \{\$\}^* b \{\$\}^* a \{\$\}^* \subseteq L^\$$ and we need to consider words in $Y^*$ as in the Corresponding Word Lemma \ref{lem: lang-conv-fg-old} with at most distance $3R + 1 = 4$ from the regular expression $W^\$$ such as $a^2 a\inv b a a\inv a$, $a^2 b\inv$, $a^2 b^2 b^{-2} a^{-2} \cdot a^2 b\inv$ are all in $L_H$ as in Theorem \ref{thm: lang-conv-thm}. 
	
	Notice here that while it is easy to observe that $\underbrace{a^2 b^2 b^{-2} a^{-2}}_u \cdot a^2 b\inv$ should have $u$ removed to obtain a somewhat optimal language, it is not so obvious to decide between $a^2 a\inv b a a\inv a$, and $a^2 b\inv$ despite the latter being shorter, as the first one more clearly follows the fellow-travel pattern that was intended when constructing $L_H$. 
				
	A computation in GAP trying to get an explicit regular language for $P \cap H$ yielded a very large regular language which did not finish computing in a reasonable timeframe using a normal computer.
	
	In any case, we know that $L_H$ does not look like the language $\cL(\bA)$ accepted by the automaton $\bA$ for $P \cap \bZ^2$ given in Figure \ref{fig: lang-conv-obv-does-not-work}. Indeed, every word accepted by the automaton representing a positive cone language for $P \cap \bZ^2$ is given in shortlex form. However, given $h = a^{2m} b^{-n}$ for some $m, n \in \bN$, a word in $L$ is given by $w = b^n a^{2m}$. Since shortlex is a normal form for $\bZ^2$, $v = a^{2m} b^{-n} \in \cL(\bA)$ is the only word accepted by $\bA$ such that $\bar w = \bar v$. However, by construction for any $M \geq 0$, we can also pick $m,n$ sufficiently large that $w,v$ do not asynchronously $M$-fellow-travel. 
		
	One way to interpret this result is that the shortlex normal form used above requires additional knowledge of the geometry of the group, whereas Theorem \ref{thm: lang-conv-thm} is quite general and only uses the starting language and the fact that subgroup is of finite index.	
	
\end{ex}

In light of the previous example, it is at least possible to bound the number of states of the automaton in Theorem \ref{thm: lang-conv-thm} as follows.

\begin{cor}
Let $G, P, X, L, H$ and $L_H$ be as in Theorem \ref{thm: lang-conv-thm}. Let $|\mathcal{A(\cL)}|$ represent the number of states in an automaton $\bA$ accepting $\cL$ and $\gamma_G, \gamma_H$ be the growth functions of $G$ and $H$ respectively. Then,

$$|\bA(L_H)| \leq (2R + 1) \cdot |\bA(L)| \cdot \gamma_H(2R + 1) \cdot (\gamma_G(3R + 1) + 1) .$$
\end{cor}
\begin{proof}
Let $\tilde L$ is given by Lemma \ref{L-tilde}, and $Y$ be as in the Corresponding Word Lemma \ref{lem: lang-conv-fg-old}. Recall that $L_H = Y^* \cap \tilde L_M$ with $M=3R + 1$. It is straightforward to show that $|\bA(Y^*)| \leq (2R + 1)$. 

The states of an automaton accepting an intersection of two regular languages are given by the product of the states of the two automaton accepting each of the languages. Therefore, we have $$|\bA(L_H)| \leq (2R + 1) \cdot \gamma_H(2R + 1) \cdot |\bA(\tilde L)|.$$

It remains to bound $|\bA(\tilde L)|$. Set $M = 3R + 1$. Set $\mathcal{L}_M$ to be the language in the Asynchronously Fellow-Travelling Automaton (Corollary \ref{cor: lang-convex-regular-aft}. It follows from the Fellow-Travelling Automaton (Proposition \ref{prop: lang-convex-regular-ft} that $\bA (\mathcal{L}_M)$ has at most $\gamma_G(3R + 1) + 1$ states. Finally, set $\tilde L, L', L''$ and $L^\$$ to be as in Lemma \ref{L-tilde}. Recall from the proof of Lemma \ref{L-tilde} that $L' = L^\$ \times (\X)^*, L'' = \L_M \intersect L'$ and $\tilde L = \text{Proj}_2(L'')$. Since projection doesn't increase the number of states, $|\bA(\tilde L)| \leq |\bA(L'')|$. By the previous remark about intersection, $|\bA(L'')| \leq |\bA(L')| \cdot |\bA(\mathcal{L}_M)|$. One can construct an automaton for $L' = L^\$ \times (\X)^*$ by taking the product of automata for $L^\$$ and $(\X)^*$. It is clear that $|\bA((\X)^*)| = 1$ and Definition \ref{padded-lang}, $|\bA(L^\$)| = |\bA(L)|$. We conclude that $|\bA(L')| = |\bA(L)|$. Putting it all together, the corollary follows. 
\end{proof}

Note that if in the Corresponding Word Lemma \ref{lem: lang-conv-fg-old}, we take $Y = \{\bar h \in H \mid |y|_X \leq 2R + 1\}$, i.e. we choose a geodesic representative word for each element of $H$, then we have $|\bA(Y^*)| \leq 2R + 1 \cdot \gamma_H(2R + 1)$ as an additional cost for the states, since we must remember each representative geodesic in $B_{2R + 1}$ which we what we had in \cite{Su2020}

However, it turns out that this additional cost in terms of states is not particularly useful in terms of constructing a language since this does not change Example \ref{ex: lang-conv-thm-Klein-Zsq} and requires more knowledge of the group.

One effective way to restrict the resulting language $L_H$ is by restricting it with some kind of normal form. In some sense, having a normal form makes it easy to compute the corresponding language if one is familiar with the group since it is easy to know what to expect. %

This was ultimately unsatisfying, and led me to work on a more constructive and simpler proof of the theorem, as found in Chapter \ref{chap: closure-finite-index}.

%% file: chap/fg-positive-cones-code.tex
\chapter{Computing normal forms in $\Gamma_n$}\label{chap: fg-code}
We reproduce the code used to write elements of $\Gamma_2$ in normal form as stated in Proposition \ref{prop: fg-normal-form}.

\begin{lstlisting}[language = GAP, breaklines]
# Set up the environment for freely reduced words. 
F := FreeGroup("a", "b", "delta");
a := F.1;; b := F.2;; delta := F.3; 
n := 2;

# word in {a,a^-1, b, b^-1,delta, delta^-1} -> normalized word. 
normalize := function(w)
	local len_syl, last_syl_num, last_syl_pow, u, delta_pow, lw, norm_w;
	len_syl := NumberSyllables(w);
	if len_syl = 0 then 
		return w;
	else 
		last_syl_num := GeneratorSyllable(w, len_syl);
		last_syl_pow := ExponentSyllable(w, len_syl);
		if last_syl_num = 3 then 
			u := Subword(w, 1, len_syl-1);
			delta_pow := last_syl_pow;
		else 
			u := w;
			delta_pow := 0;
		fi;
		lw := [u, delta_pow]; 

		while not is_normal_list(lw) do
			lw := replace_inverses(lw);
			lw := apply_relation(lw); 
			lw := collect_deltas(lw); 
		od;
	fi; 
	u := lw[1];
	delta_pow := lw[2];
	norm_w := u*delta^delta_pow;
	return norm_w;
end; 

# [u, delta_pow] -> [u, delta_pow]
replace_inverses := function(lw)
	local lu, delta_pow, len_syl, new_u, new_delta_pow, i, gen, pow, u, abs_pow, syl;
	u := lw[1];
	delta_pow := lw[2];
	len_syl := NumberSyllables(u); 

	new_u := a^0; 
	new_delta_pow := delta_pow;

	for i in [1..len_syl] do
		gen := GeneratorSyllable(u, i);
		pow := ExponentSyllable(u, i);

		if gen = 1 then
			if pow < 0 then 
				abs_pow := -pow; 
				syl := (a^n)^abs_pow;
				new_delta_pow := new_delta_pow - abs_pow; 
			else
				syl := a^pow;
			fi;

		elif gen = 2 then 
			if pow < 0 then 
				abs_pow := -pow; 
				syl := (a^n*b*a^n)^abs_pow;
				new_delta_pow := new_delta_pow - abs_pow; 
			else
				syl := b^pow;
			fi;
		else 
			Error("Generator is not a or b in replace_inverses.");
		fi;
		new_u := new_u*syl; 
	od; 
	return [new_u, new_delta_pow];
end;

# [u, delta_pow] -> [u, delta_pow]
apply_relation := function(lw)
	local u, delta_pow, new_u, new_delta_pow;
	u := lw[1];
	delta_pow := lw[2];
	new_u := u;
	new_delta_pow := delta_pow; 

	while not(SubstitutedWord(new_u, b*a^n*b, 1, a) = fail) do
		new_u := SubstitutedWord(new_u, b*a^n*b, 1, a);
	od;
	return [new_u, new_delta_pow];
end;

# [u, delta_pow] -> [u, delta_pow]
collect_deltas := function(lw)
	local u, delta_pow, len_syl, new_u, new_delta_pow, i, gen, pow, syl, factor;
	u := lw[1];
	delta_pow := lw[2];
	len_syl := NumberSyllables(u); 

	new_u := a^0; 
	new_delta_pow := delta_pow;

	for i in [1..len_syl] do
		gen := GeneratorSyllable(u, i);
		pow := ExponentSyllable(u, i);

		if gen = 1 then 
			if pow > n then 
				factor := Int(pow / (n+1));
				new_delta_pow := new_delta_pow + factor;
				pow := pow - (factor * (n+1)); 
			fi;
			syl := a^pow; 

		elif gen = 2 then 
			syl := b^pow; 

		else
			Error("Generator is not a or b in collect_deltas.");
		fi; 
		new_u := new_u*syl;
	od;
	return [new_u, new_delta_pow];
end; 

# word in {a,a^-1, b, b^-1,delta, delta^-1} -> boolean
is_normal := function(w)
	local len_syl, last_syl_num, delta_pow, u, lw;
	len_syl := NumberSyllables(w);
	if len_syl = 0 then 
		return true;
	else 
		last_syl_num := GeneratorSyllable(w, len_syl);
		if last_syl_num = 3 then 
			delta_pow := ExponentSyllable(w, len_syl);
			u := SubSyllables(w, 1, len_syl-1); 
			lw := [u, delta_pow];
			return is_normal_list(lw);
		else 
			u := w;
			lw := [u, 0];
			return is_normal_list(lw);
		fi; 
	fi; 
end;

# [u, delta_pow] -> boolean
is_normal_list := function(lw)
	local u, len_syl, i, gen, pow;
	u := lw[1];
	len_syl := NumberSyllables(u); 
	for i in [1..len_syl] do
		gen := GeneratorSyllable(u, i);
		pow := ExponentSyllable(u, i);

		# check for word positivity
		if pow < 0 then 
			return false;

		# check that the deltas are collected
		elif gen = 3 and not(i = len_syl) then 
			return false;

		# check if b*a^n*b -> a has been done. 
		elif gen = 1 and pow > n-1 then  
			if not(i = 1) and not(i = len_syl) then 
				#Print("flagged");
				return false;
			fi; 
		fi; 
	od;
	return true;
end; 
\end{lstlisting}

An example use case normalising $w = b\inv a b^3$ in $n=2$ would be the following. 
\begin{lstlisting}[language=bash]
gap> w := b^-1*a*b^3;;
gap> normalize(w);
a^2*b^4
\end{lstlisting}

If one wishes to verify an entire list at once, such as the generating set $Y$ of Section \ref{sec: fg-F_nxZ} for the proof of Lemma \ref{lem: fg-H-F2xZ},
then they can use the following prompt. 

\begin{lstlisting}[language=bash, breaklines]
gap> Y := [ a*b^2, b^-1*a*b^3, b^-2*a*b^4, b^-3*a*b^5, b^-4*a, b^-5*a*b, b^6 ];;
gap> Apply(Y, y -> normalize(y));
gap> Y; 
[ a*b^2, a^2*b^4, a*(a*b)^2*b^4, a*(a*b)^3*b^5, a*(a*b)^4, a*(a*b)^5*b, b^6 ]
gap> psi_Y := [x, y*x*z^-1, x^-1*y*x*z^-1, x^-1*y^-1*x^-1*y*x,x^-1*y^-1*z^2 ,x^-1*y^-1*x*z^2, y^-1*x^-1*y*x*z^-1];;
gap> Apply(psi_Y, y -> normalize(y));
gap> psi_Y;
[ a*b^2, a^2*b^4, a*(a*b)^2*b^4, a*(a*b)^3*b^5, a*(a*b)^4, a*(a*b)^5*b, b^6 ]
gap> Y = psi_Y;
true
\end{lstlisting}

%% file: back/conclusions.tex
\begin{fullwidth} %
\thispagestyle{empty} %
\setlength{\parindent}{0pt} %

\begin{abstract}[Conclusiones]
Esta tesis muestra que la interacción, relativamente poco explorada, entre los lenguajes formales y la ordenabilidad por la izquierda en la teoría de grupos puede resultar fructífera. Sus teoremas principales demuestran que la estructura algebraica de un grupo se ve reflejada en la complejidad computacional de los órdenes que admite y viceversa (véase la Parte \ref{part: results}). Además, la tesis aclara cómo se preserva la complejidad de los órdenes por la izquierda bajo diversas construcciones de teoría de grupos, y ofrece nuevos ejemplos de órdenes tanto regulares como finitamente generados. Estos resultados abren el camino a futuras investigaciones en este campo.
\end{abstract}

\vspace{5cm}

\begin{abstract}[Conclusions]
This thesis shows that the relatively under-explored interplay between formal languages and left-orderability in group theory can be fruitful. The main theorems demonstrate that the algebraic structure of a group reflects on the computational complexity of the orders it admits and vice versa (see Part \ref{part: results}). Furthermore, the thesis clarifies how the complexity of left-orders is preserved under various group-theoretical constructions and provides new examples of both regular and finitely generated left-orders. These results pave the way for further research in the field.
\end{abstract}

\end{fullwidth}

%% file: bibliography/bib-page.tex
\printbibliography